\documentclass[reqno]{amsart}
\usepackage{hyperref}
\usepackage{fullpage}
\usepackage{tikz}
\usetikzlibrary{arrows}
\usepackage{newverbs}
\usepackage{fancyvrb}

\newif\ifshowkeys
\showkeysfalse

\ifshowkeys
\newcommand{\lbl}[1]{\label{#1}\textup{[\texttt{#1}]}\par}
\else
\newcommand{\lbl}{\label}
\fi

\input xypic
\newdir{ >}{{}*!/-9pt/\dir{>}}

\definecolor{olivegreen}{cmyk}{0.64,0,0.95,0.40}
\definecolor{rawsienna}{cmyk}{0,0.72,1,0.45}
\definecolor{darkred}{rgb}{0.5,0,0}
\definecolor{darkgreen}{rgb}{0,0.5,0}

\definecolor{maplegrey}{rgb}{0.1,0.1,0.1}
\definecolor{maplecyan}{rgb}{0,1,1}
\definecolor{maplegreen}{rgb}{0.04,0.98,0.04}
\definecolor{maplemagenta}{rgb}{1,0.02,1}
\definecolor{mapleblue}{rgb}{0,0,1}
\definecolor{maplered}{rgb}{1,0,0}
\definecolor{maplepurple}{rgb}{0.5,0,0.5}

\newcommand{\Aut}	{\operatorname{Aut}}
\newcommand{\Map}	{\operatorname{Map}}
\newcommand{\Hol}	{\operatorname{Hol}}
\newcommand{\Hom}	{\operatorname{Hom}}
\newcommand{\Inn}	{\operatorname{Inn}}
\newcommand{\Net}	{\operatorname{Net}}
\newcommand{\Out}	{\operatorname{Out}}

\newcommand{\alg}	{\operatorname{alg}}
\newcommand{\arctanh}	{\operatorname{arctanh}}

\newcommand{\hyp}	{\operatorname{hyp}}
\newcommand{\stab}	{\operatorname{stab}}
\newcommand{\spec}	{\operatorname{spec}}

\newcommand{\Dl}        {\Delta}
\newcommand{\Gm}        {\Gamma}
\newcommand{\Tht}       {\Theta}
\newcommand{\Lm}        {\Lambda}
\newcommand{\Sg}        {\Sigma}
\newcommand{\Om}        {\Omega}

\newcommand{\al}        {\alpha}
\newcommand{\bt}        {\beta}
\newcommand{\gm}        {\gamma}
\newcommand{\dl}        {\delta}
\newcommand{\ep}        {\epsilon}
\newcommand{\zt}        {\zeta}

\newcommand{\tht}       {\theta}
\newcommand{\kp}        {\kappa}
\newcommand{\lm}        {\lambda}
\newcommand{\om}        {\omega}
\newcommand{\sg}        {\sigma}

\newcommand{\xla}       {\xleftarrow}
\newcommand{\xra}       {\xrightarrow}

\newcommand{\N}         {{\mathbb{N}}}
\newcommand{\Z}         {{\mathbb{Z}}}
\newcommand{\Q}         {{\mathbb{Q}}}
\newcommand{\R}         {{\mathbb{R}}}
\newcommand{\C}         {{\mathbb{C}}}
\renewcommand{\H}       {{\mathbb{H}}}

\newcommand{\ov}[1]     {\overline{#1}}
\newcommand{\un}[1]     {\underline{#1}}
\newcommand{\ip}[1]     {\langle #1\rangle}
\newcommand{\st}        {\;|\;}
\newcommand{\tm}        {\times}
\newcommand{\sm}        {\setminus}
\newcommand{\bbm}       {\left[\begin{matrix}}
\newcommand{\ebm}       {\end{matrix}\right]}

\newcommand{\sse}       {\subseteq}
\newcommand{\half}      {\tfrac{1}{2}}
\newcommand{\ppi}       {\tfrac{\pi}{2}}
\newcommand{\qpi}       {\tfrac{\pi}{4}}
\newcommand{\rt}        {\sqrt{2}}
\newcommand{\ot}        {\otimes}

\newcommand{\dhyp}      {d_{\text{hyp}}}

\newcommand{\tA}	{\widetilde{A}}
\newcommand{\tC}	{\widetilde{C}}

\newcommand{\tP}	{\widetilde{P}}

\newcommand{\tY}	{\widetilde{Y}}

\newcommand{\tc}	{\widetilde{c}}
\newcommand{\tf}	{\widetilde{f}}
\newcommand{\tj}	{\widetilde{\jmath}}
\newcommand{\mt}	{\widetilde{m}}

\newcommand{\tu}	{\widetilde{u}}
\newcommand{\ty}	{\widetilde{y}}
\newcommand{\tpi}	{\widetilde{\pi}}
\newcommand{\tPi}	{\widetilde{\Pi}}
\newcommand{\tPhi}	{\widetilde{\Phi}}

\newcommand{\CL}        {\mathcal{L}}
\newcommand{\CO}        {\mathcal{O}}
\newcommand{\CQ}        {\mathcal{Q}}
\newcommand{\CX}        {\mathcal{X}}

\newcommand{\hp}        {\widehat{p}}

\newcommand{\pp}        {\hphantom{+}}

\renewcommand{\ss}{\scriptstyle}
\renewcommand{\:}{\colon}

\newtheorem{theorem}{Theorem}[subsection]

\newtheorem{lemma}[theorem]{Lemma}
\newtheorem{proposition}[theorem]{Proposition}
\newtheorem{corollary}[theorem]{Corollary}
\theoremstyle{definition}
\newtheorem{remark}[theorem]{Remark}

\newtheorem{definition}[theorem]{Definition}

\newtheorem{method}[theorem]{Method}

\DefineVerbatimEnvironment{checks}{Verbatim}{formatcom=\color{blue}}
\DefineVerbatimEnvironment{mcodeblock}{Verbatim}{formatcom=\color{darkred}}
\newverbcommand{\mcode}{\color{darkred}}{}
\newverbcommand{\fname}{\color{darkgreen}}{}

\begin{document}
\title{Uniformization of embedded surfaces}
\author{N.~P.~Strickland}
\date{\today}
\bibliographystyle{abbrv}

\maketitle
\tableofcontents

\section{Introduction}
\lbl{sec-intro}

Let $X\subset S^3$ be a smoothly embedded closed surface of genus
$g>1$.  As we will explain, $X$ automatically has a very rich and
rigid geometric structure.  Indeed, $X$ inherits a Riemannian metric
from $S^3$, and after specifying some conventions we also obtain a
well-defined orientation.  Now consider a point $x\in X$, and let
$T_xX$ denote the corresponding tangent space.  Let $J_x\:T_xX\to
T_xX$ be the anticlockwise rotation through $\pi/2$ (which is
meaningful given the metric and orientation).  This depends smoothly
on $x$ and satisfies $J_x^2=-1$, so it gives an almost complex
structure on $X$.  It has been known since the early twentieth century
that any almost complex structure on a manifold of real dimension two
integrates to give a genuine complex structure.  Thus, $X$ can be
regarded as a compact Riemann surface.  It is known that any compact
Riemann surface can be regarded as a projective algebraic variety over
$\C$, and also as a branched cover of the Riemann sphere.
Alternatively, as we have assumed that the genus is larger than one,
the universal cover of $X$ is conformally equivalent to the open unit
disc $\Dl$.  This means that $X$ is conformally equivalent to the
quotient $\Dl/\Pi$ for some Fuchsian group $\Pi$.

To the best of our knowledge, the literature contains no examples
where a significant fraction of this structure can be made explicit.
This monograph is a partially successful attempt to provide such
an example, involving the surface
\[ EX^* = \{x\in S^3\st
              (3x_3^2-2)x_4+\rt(x_1^2-x_2^2)x_3=0
         \},
\]
with weaker results for a one-parameter family of surfaces in which
$EX^*$ appears.  To display $EX^*$ visually, we apply the
stereographic projection map $s\:S^3\to\R^3\cup\{\infty\}$ defined as
follows:
\begin{align*}
 s(x) &=
  \left(\frac{x_1}{1-x_4},\frac{x_2}{1-x_4},\frac{x_3}{1-x_4}\right)
  \\
 s^{-1}(u) &=
  \left(\frac{2u_1}{\|u\|^2+1},
        \frac{2u_2}{\|u\|^2+1},
        \frac{2u_3}{\|u\|^2+1},
        \frac{\|u\|^2-1}{\|u\|^2+1}\right).
\end{align*}

The image $s(EX^*)$ looks like this:
\[ \includegraphics[scale=0.4,angle=90,clip=true,
                    trim=10cm 8cm 10cm 8cm]{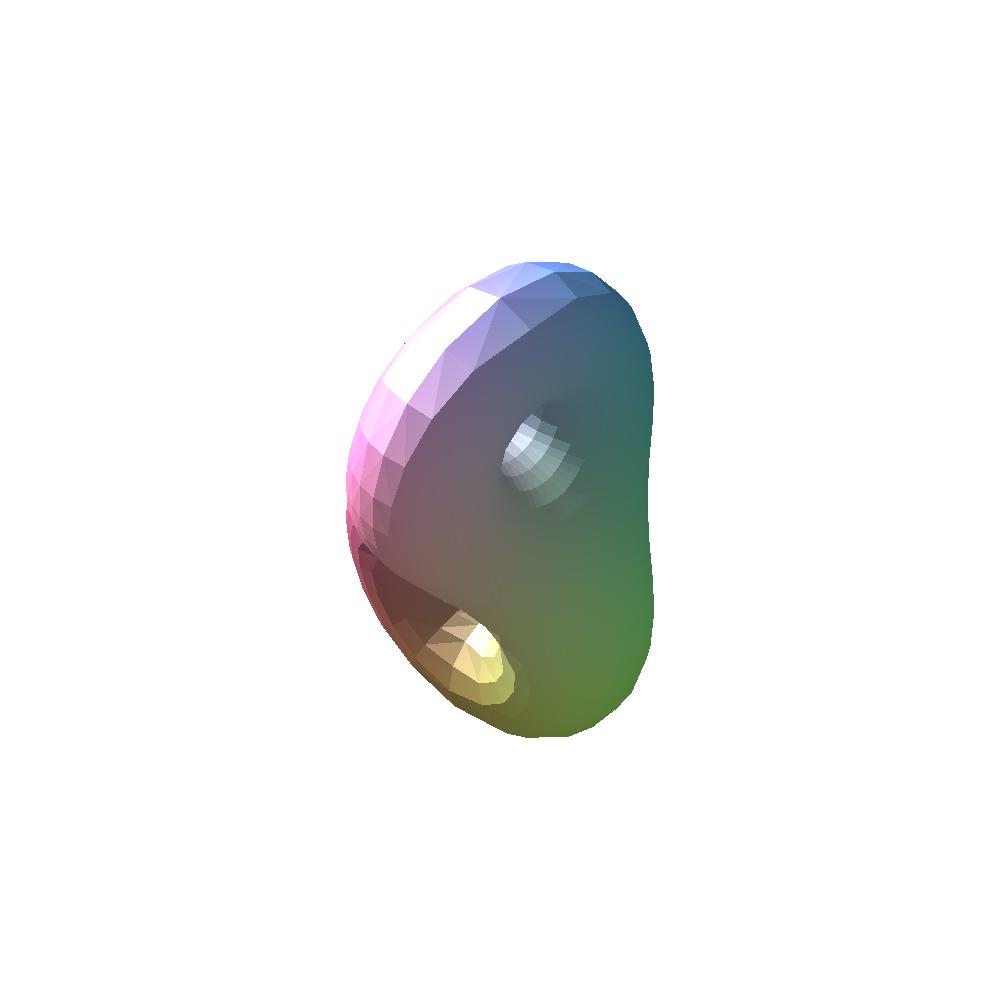}
\]

Our work is organised around the following
definitions:
\begin{definition}\lbl{defn-G}
 Let $G$ be the group of order $16$ generated by $\lm$, $\mu$ and
 $\nu$ subject to relations
 \[ \lm^4=\mu^2=\nu^2=(\mu\nu)^2=(\lm\mu)^2=(\lm\nu)^2=1, \]
 so
 \[ G=\{\lm^i\mu^j\nu^k\st 0\leq i<4,\;0\leq j,k<2\}. \]
 We use the following notation for subgroups:
 \begin{align*}
  D_8 &= \ip{\lm,\mu} &
  C_4 &= \ip{\lm} \\
  C_2 &= \ip{\lm^2} &
  C'_2 &= \ip{\mu\nu}.
 \end{align*}
\end{definition}

\begin{definition}\lbl{defn-V-star}
 We write $V^*$ for the set $\{0,\dotsc,13\}$ equipped with the action
 of $G$ by the following permutations:
 \begin{align*}
  \lm &\mapsto (2\;3\;4\;5)\;(6\;7\;8\;9)\;(10\;11)\;(12\;13) \\
  \mu &\mapsto (0\;1)\;(3\;5)\;(6\;9)\;(7\;8)\;(10\;12)\;(11\;13) \\
  \nu &\mapsto (3\;5)\;(6\;9)\;(7\;8).
 \end{align*}
\end{definition}

\begin{remark}\lbl{rem-V-star}
 The orbits in $V^*$ are
 \begin{align*}
  \{0,1\}         & \simeq G/\ip{\lm,\nu} \\
  \{2,3,4,5\}     & \simeq G/\ip{\mu,\nu} \\
  \{6,7,8,9\}     & \simeq G/\ip{\lm\mu,\lm\nu} \\
  \{10,11,12,13\} & \simeq G/\ip{\lm^2,\nu}.
 \end{align*}
 The action can be displayed as follows:
 \begin{center}
  \begin{tikzpicture}[scale=1.7,auto,shorten >= 1pt]
   \node (v0) at ( 0, 1) [circle,draw] {$\ss 0$};
   \node (v1) at ( 0, 0) [circle,draw] {$\ss 1$};
   \node (v2) at ( 1, 1) [circle,draw] {$\ss 2$};
   \node (v3) at ( 2, 1) [circle,draw] {$\ss 3$};
   \node (v4) at ( 2, 0) [circle,draw] {$\ss 4$};
   \node (v5) at ( 1, 0) [circle,draw] {$\ss 5$};
   \node (v6) at ( 3, 1) [circle,draw] {$\ss 6$};
   \node (v7) at ( 4, 1) [circle,draw] {$\ss 7$};
   \node (v8) at ( 4, 0) [circle,draw] {$\ss 8$};
   \node (v9) at ( 3, 0) [circle,draw] {$\ss 9$};
   \node (va) at ( 5, 1) [circle,draw] {$\ss {10}$};
   \node (vb) at ( 5, 0) [circle,draw] {$\ss {11}$};
   \node (vc) at ( 6, 1) [circle,draw] {$\ss {12}$};
   \node (vd) at ( 6, 0) [circle,draw] {$\ss {13}$};
   \draw[->,red,thick] (v2) to (v3);
   \draw[->,red,thick] (v3) to (v4);
   \draw[->,red,thick] (v4) to (v5);
   \draw[->,red,thick] (v5) to (v2);
   \draw[->,red,thick] (v6) to (v7);
   \draw[->,red,thick] (v7) to (v8);
   \draw[->,red,thick] (v8) to (v9);
   \draw[->,red,thick] (v9) to (v6);
   \draw[<->,red,thick] (va) to (vb);
   \draw[<->,red,thick] (vc) to (vd);
   \draw[<->,olivegreen,thick,dotted] (v3) to (v5);
   \draw[<->,olivegreen,thick,bend left,dotted] (v6) to (v9);
   \draw[<->,olivegreen,thick,bend right,dotted] (v7) to (v8);
   \draw[<->,blue,thick,dashed] (v0) to (v1);
   \draw[<->,blue,thick,dashed] (va) to (vc);
   \draw[<->,blue,thick,dashed] (vb) to (vd);
  \end{tikzpicture}
 \end{center}
 The solid red arrows show the action of $\lm$, the dotted green
 arrows show the action of $\nu$, and the dashed blue arrows show the
 action of $\mu\nu$.
\end{remark}

\begin{definition}\lbl{defn-precromulent}
 A \emph{precromulent surface} is a compact Riemann surface $X$ of genus
 $2$ with an action of $G$ such that
 \begin{itemize}
  \item[(a)] The elements of $D_8$ act conformally, and the elements
   of $G\sm D_8$ act anticonformally.
  \item[(b)] The set $V=\{v\in X\st \stab_{D_8}(v)\neq 1\}$ is
   isomorphic to $V^*$ as a $G$-set.
 \end{itemize}
 A \emph{precromulent labelling} of $X$ is a specific choice of
 isomorphism $V^*\simeq V$, or equivalently, a listing of the points
 in $V$ as $v_0,\dotsc,v_{13}$ such that $G$ permutes these points in
 accordance with the permutations listed in
 Definition~\ref{defn-V-star}.  A \emph{cromulent labelling} is a
 precromulent labelling such that
 \begin{itemize}
  \item[(c)] $\lm$ acts on the tangent space $T_{v_0}X$ as
   multiplication by $i$.
  \item[(d)] In the set $X'=\{x\in X\st\stab_G(x)=1\}$, there is a
   connected component $F'$ whose closure contains
   $\{v_0,v_3,v_6,v_{11}\}$.
 \end{itemize}
 We will show in Proposition~\ref{prop-labellings} that every
 precromulent surface has precisely two cromulent labellings, which
 are exchanged by the action of $\lm^2$.  A \emph{cromulent surface}
 is a precromulent surface with a choice of cromulent labelling.

 A \emph{(precromulent) isomorphism} between precromulent surfaces
 will mean a $G$-equivariant conformal isomorphism.  A
 \emph{(cromulent) isomorphism} between cromulent surfaces
 will mean a $G$-equivariant conformal isomorphism that is compatible
 with the specified labellings.
\end{definition}

Now fix a parameter $a\in (0,1)$.  We put
\[ EX(a) = \{x\in\R^4\st \|x\|=1,\;
              ((a^{-2}+1)x_3^2-2)x_4+a^{-1}(x_1^2-x_2^2)x_3=0\},
\]
and observe that $EX^*=EX(1/\rt)$.  In Section~\ref{sec-E} we
will give $EX(a)$ a $G$-action and labelling making it a cromulent
surface.  We call these surfaces the \emph{embedded family}.  Although
our central problem is to study uniformizations of the surfaces
$EX(a)$, we will also discuss many other features of their geometry
and topology.  In particular, we will give an alternative definition
which is much more geometric but takes longer to state.  Two special
features of the case $a=1/\rt$ are as follows:
\begin{itemize}
 \item[(a)] By a \emph{great circle} we mean the intersection of $S^3$
  with a two-dimensional vector subspace of $\R^4$.  For all $a$, the
  fixed set of the element $\nu\in G$ is the disjoint union of three
  curves, each of which is diffeomorphic to $S^1$.  If
  $a=1/\rt$ (but for no other value) then one of those components
  is a great circle.
 \item[(b)] One can show that the complexified variety
  \[ CEX(a) = \{x\in\C^4\st \sum_ix_i^2=1,\;
               ((a^{-2}+1)x_3^2-2)x_4+a^{-1}(x_1^2-x_2^2)x_3=0\}
  \]
  is smooth for all $a\neq 1/\rt$, but $CEX(1/\rt)$ is
  singular at the eight points in the $G$-orbit of
  $(i\rt,0,\rt,1)$.
\end{itemize}

Next, put
\[ PX_0(a) = \{(w,z)\in\C^2\st w^2=z^5-(a^2+a^{-2})z^3+z\}. \]
This is an affine hyperelliptic curve.  By well-known methods we can
construct a compact Riemann surface $PX(a)$ which is the union of
$PX_0(a)$ with a single extra point.  In
Section~\ref{sec-P} we will give this a $G$-action and
labelling making it a cromulent surface.  We call these surfaces the
\emph{projective family}.

Finally, let $\Pi$ be the abstract group generated by symbols $\bt_k$
(for $k\in\Z/8$) subject to the following relations:
\begin{align*}
 \bt_{k+4} &= \bt_k^{-1} \\
 \bt_0\bt_1\bt_2\bt_3\bt_4\bt_5\bt_6\bt_7 &= 1.
\end{align*}
In Section~\ref{sec-H} we will give a free action of $\Pi$ on
the unit disc $\Dl=\{z\in\C\st |z|<1\}$, depending on a parameter
$a\in(0,1)$.  We will show that the orbit space $HX(a)$ for this
action is a compact Riemann surface of genus two, and give it a
cromulent structure.  We call these surfaces the \emph{hyperbolic family}.

In Theorem~\ref{thm-classify-cromulent},
Corollary~\ref{cor-cromulent-iso} and Theorem~\ref{thm-H-universal}
we will show that
\begin{itemize}
 \item For any two cromulent surfaces, there is at most one
  isomorphism between them.
 \item For any cromulent surface $X$ there is a unique $a\in (0,1)$
  such that $X\simeq PX(a)$, and there is a unique $b\in (0,1)$ such
  that $X\simeq HX(b)$.
\end{itemize}
In other words, the projective family and the hyperbolic family are
both universal.  We conjecture that the embedded family is also
universal, but we have not proved this.

As a consequence of universality, for every $a\in (0,1)$ there is a
unique $b\in (0,1)$ such that $PX(a)\simeq HX(b)$.  In
Section~\ref{sec-P-H} we will develop two different methods for
computing $a$ as a function of $b$ or \emph{vice versa}, and for
computing the corresponding cromulent isomorphism.  One method
involves a rich theory based on Fuchsian differential equations and
the Schwarzian derivative; the other is less illuminating, but in some
respects more efficient and direct.  The graph of $a$ against $b$ is
as follows.
\begin{center}
 \begin{tikzpicture}[scale=4]
  \draw[black,->] (-0.05,0) -- (1.05,0);
  \draw[black,->] (0,-0.05) -- (0,1.05);
  \draw[black] (1,-0.05) -- (1,0);
  \draw[black] (-0.05,1) -- (0,1);
  \draw ( 0.00,-0.05) node[anchor=north] {$0$};
  \draw ( 1.00,-0.05) node[anchor=north] {$1$};
  \draw (-0.05, 0.00) node[anchor=east ] {$0$};
  \draw (-0.05, 1.00) node[anchor=east ] {$1$};
  \draw ( 1.05, 0.00) node[anchor=west ] {$b$};
  \draw ( 0.00, 1.05) node[anchor=south] {$a$};
 \draw[red] plot[smooth] coordinates{ (0.000,1.000) (0.060,1.000) (0.080,1.000) (0.100,1.000) (0.120,1.000) (0.140,1.000) (0.160,1.000) (0.180,1.000) (0.200,1.000) (0.220,1.000) (0.240,0.999) (0.260,0.997) (0.280,0.994) (0.300,0.990) (0.320,0.983) (0.340,0.974) (0.360,0.961) (0.380,0.944) (0.400,0.923) (0.420,0.898) (0.440,0.869) (0.460,0.835) (0.480,0.798) (0.500,0.757) (0.520,0.713) (0.540,0.667) (0.560,0.620) (0.580,0.571) (0.600,0.521) (0.620,0.472) (0.640,0.423) (0.660,0.375) (0.680,0.328) (0.700,0.283) (0.720,0.241) (0.740,0.201) (0.760,0.163) (0.780,0.129) (0.800,0.099) (0.820,0.073) (0.840,0.050) (0.860,0.032) (0.880,0.019) (0.900,0.009) (0.920,0.004) (0.940,0.001) (1.000,0.000) };
  \fill[black] (0.801,0.098) circle(0.015);
 \end{tikzpicture}
\end{center}

We have conducted a fairly extensive heuristic
search for closed-form relationships between the above graph and
various other functions that we know to be relevant, but without
success.

Finally, we want to find $a$ and $b$ such that
$EX^*\simeq PX(a)\simeq HX(b)$.  Our best estimates are
$a\simeq 0.0983562$ and $b\simeq 0.8005319$, corresponding to the
marked point on the above graph.  We have some reason to hope that all
the quoted digits are accurate, but we have not performed a rigorous
error analysis.  In Section~\ref{sec-classify-roothalf} we will
explain the methods used to calculate $b$ (and then $a$ is calculated
from $b$ as described previously).  The first step is to find the
unique smooth function $f$ on $EX^*$ such that $e^{2f}$ times the
standard metric has curvature equal to $-1$.  We can then find the
lengths of certain curves with respect to this rescaled metric, and
the value of $b$ can be determined from these lengths.

\subsection{Maple code}

To carry out the work described above, we need to check a very large
number of reasonably complex formulae and combinatorial facts, and we
also need to perform extensive numerical calculations.  Most of the
formulae could individually be checked by hand with sufficient effort.
However, the number and size of the formulae are so large that
computer assistance is required for the project as a whole.  We have
used Maple for this.  The code and documentation are distributed
alongside this monograph, and there is an overview of the structure in
Section~\ref{sec-maple}.  This monograph contains many lines like
this:
\begin{checks}
 group_check.mpl: check_group_properties(), check_character_table()
\end{checks}
This indicates that some set of claims that have recently been made in
the text can be checked by executing the functions
\mcode+check_group_properties()+ and \mcode+check_character_table()+,
which are defined in the file \fname+group_check.mpl+.  These functions
are set up so that they will print their own names, then they will run
silently unless they detect any errors.  One can set the global
variable \mcode+assert_verbosely+ to \mcode+true+, and then the checking
functions will print additional information about the individual
claims being checked.  One can check the complete set of claims for
the whole monograph by reading the file \fname+check_all.mpl+.  While
this does not quite reach the level of rigour provided by formal proof
assistants such as Isabelle, it is a major step in that direction.

The worksheet \fname+text_check.mw+ also provides another means to check
the consistency of the text with the Maple code.  (Some fragments of
\LaTeX code were generated automatically by Maple to ensure
correctness, but technical problems with precise control of formatting
dissuaded us from using this approach more extensively.)

One can also repeat all the numerical calculations by following the
instructions in Section~\ref{sec-build}.

The most convenient place to view and download the code, documentation
and other associated files is the page
\url{https://neilstrickland.github.io/genus2/}.  A snapshot will also
be placed on the arxiv, as a set of ancilliary files.

\section{General theory of precromulent surfaces}
\lbl{sec-general}

\subsection{Representations of \texorpdfstring{$G$}{G}}
\lbl{sec-representations}

We first discuss the representation theory of $G$, which will be useful
for organising various algebraic calculations later.  We assume that
the reader is familiar with the basic ideas of representation theory,
which are discussed in~\cite{se:lrf}, for example.

\begin{proposition}\lbl{prop-characters}
 The centre of $G$ is $\{1,\lm^2,\mu\nu,\lm^2\mu\nu\}$, and the
 commutator subgroup is $\{1,\lm^2\}$.  The character table is as
 follows:
 \[ \renewcommand{\arraystretch}{1.5}
    \begin{array}{|c|c|c|c|c|c|c|c|c|c|c|} \hline
                  & \chi_0 & \chi_1 & \chi_2 & \chi_3 & \chi_4 & \chi_5 & \chi_6 & \chi_7 & \chi_8 & \chi_9 \\ \hline
                1 &  1     &  1     &  1     &  1     &  1     &  1     &  1     &  1     &  2     &  2     \\ \hline
            \lm^2 &  1     &  1     &  1     &  1     &  1     &  1     &  1     &  1     & -2     & -2     \\ \hline
           \mu\nu &  1     &  1     & -1     & -1     &  1     &  1     & -1     & -1     &  2     & -2     \\ \hline
      \lm^2\mu\nu &  1     &  1     & -1     & -1     &  1     &  1     & -1     & -1     & -2     &  2     \\ \hline
      \lm^{\pm 1} &  1     & -1     &  1     & -1     & -1     &  1     & -1     &  1     &  0     &  0     \\ \hline
     \mu,\lm^2\mu &  1     &  1     & -1     & -1     & -1     & -1     &  1     &  1     &  0     &  0     \\ \hline
   \lm^{\pm 1}\mu &  1     & -1     & -1     &  1     &  1     & -1     & -1     &  1     &  0     &  0     \\ \hline
     \nu,\lm^2\nu &  1     &  1     &  1     &  1     & -1     & -1     & -1     & -1     &  0     &  0     \\ \hline
   \lm^{\pm 1}\nu &  1     & -1     &  1     & -1     &  1     & -1     &  1     & -1     &  0     &  0     \\ \hline
\lm^{\pm 1}\mu\nu &  1     & -1     & -1     &  1     & -1     &  1     &  1     & -1     &  0     &  0     \\ \hline
   \end{array}
 \]
\end{proposition}
\begin{proof}
 The commutator of $\lm$ and $\mu$ is $\lm^2$, and it is clear from
 the form of the defining relations that $\{1,\lm^2\}$ is normal and
 that $G/\{1,\lm^2\}$ is elementary abelian.  It follows that the
 commutator subgroup is precisely $\{1,\lm^2\}$.  It is a
 straightforward calculation that the elements $1$, $\lm^2$, $\mu\nu$
 and $\lm^2\mu\nu$ are central, but that no other element commutes
 with $\lm$.  It follows that the centre is as claimed.  We now see
 that if $\al$ is a non-central element then the corresponding
 conjugacy class is just $\{\al,\lm^2\al\}$.  This means that there
 are ten conjugacy classes, as listed in the left hand column.  The
 characters of degree one are the same as the homomorphisms from the
 abelianization $G/\{1,\lm^2\}$ to $S^1$.  As $G/\{1,\lm^2\}$ is
 elementary abelian of order $8$, it is easy to check that
 $\chi_0,\dotsc,\chi_7$ is a complete list of such characters.  There
 are two different retractions of $G$ onto $D_8$, one sending $\mu\nu$
 to the identity, and the other sending $\mu\nu$ to $\lm^2$.  There is
 a standard action of $D_8$ as the isometries of a square in $\R^2$,
 and by pulling this back along the two projections we get two
 two-dimensional representations of $G$, with characters $\chi_8$ and
 $\chi_9$.  These are irreducible, because in each case the sum of the
 squares of the character values is equal to the group order.  We now
 have ten different irreducible representations, which matches the
 number of conjugacy classes, so the list is complete.
 \begin{checks}
  group_check.mpl: check_group_properties(), check_character_table()
 \end{checks}
\end{proof}

\begin{remark}
 Maple notation for the elements of $G$ is as follows:
 \begin{align*}
   1      &= 1  & \lm       &= L   & \lm^2       &= LL   & \lm^3       &= LLL  \\
   \mu    &= M  & \lm\mu    &= LM  & \lm^2\mu    &= LLM  & \lm^3\mu    &= LLLM \\
   \nu    &= N  & \lm\nu    &= LN  & \lm^2\nu    &= LLN  & \lm^3\nu    &= LLLN \\
   \mu\nu &= MN & \lm\mu\nu &= LMN & \lm^2\mu\nu &= LLMN & \lm^3\mu\nu &= LLLMN
 \end{align*}
 To make this work reliably, the code in the file \fname+group.mpl+
 protects the symbols \mcode+L+, \mcode+N+, \mcode+LLMN+ and so on, so
 they cannot be assigned values.  The function \mcode+G_mult+ computes
 the group operation, so \mcode+G_mult(M,L)+ returns \mcode+LLLM+, for
 example.  The functions \mcode+G_inv+ and \mcode+G_conj+ compute
 inverses and conjugates.  To retrieve $\chi_8(\lm^2)$ (for example),
 one can enter \mcode+character[8][LL]+.  The variable \mcode+G16+
 contains the list of all elements of $G$.  All of this is set up by
 the file \fname+group.mpl+.

 Note that this discussion of the contents of \fname+group.mpl+ is
 incomplete, as will be the case with similar comments throughout this
 monograph.  For full information, the reader should consult the code
 itself, and the comments therein.  The full set of files for this
 project contains a \fname+doc+ directory.  The file \fname+defs.html+
 in that directory is an index of all defined symbols, with links to
 the relevant lines in in the files of Maple code.
\end{remark}

\subsection{Automorphisms of \texorpdfstring{$V^*$}{V*}}
\lbl{sec-aut-V}

As we stated in the introduction, every precromulent surface has
precisely two cromulent labellings.  In order to prove this, we will
need to understand the automorphisms of the $G$-set $V^*$, and it is
convenient to treat that question now.

\begin{proposition}\lbl{prop-aut-V}
 $\Aut(V^*)$ is isomorphic to $C_2^5$, with the following generators:
 \begin{align*}
  \phi_0 &= (0\;1) \\
  \phi_1 &= (2\;4)(3\;5) \\
  \phi_2 &= (6\;8)(7\;9) \\
  \phi_3 &= (10\;11)(12\;13) \\
  \phi_4 &= (10\;12)(11\;13).
 \end{align*}
\end{proposition}
Readers may find it helpful to consider the picture in
Remark~\ref{rem-V-star} when reading the argument below.  The
permutation $\phi_i$ is represented in Maple as \mcode+aut_V_phi[i]+.
\begin{proof}
 First, it is straightforward to check directly that the above
 permutations commute with $\lm$, $\mu$ and $\nu$ (so they define
 automorphisms of $V^*$).  It is also easy to see that they are
 commuting involutions and that they generate a group $A$ isomorphic
 to $C_2^5$.

 Now consider an arbitrary permutation $\phi$ that commutes with
 $\lm$, $\mu$ and $\nu$; we must show that $\phi\in A$.  As $\phi$
 commutes with $G$, we must have $\stab_G(\phi(i))=\stab_G(i)$ for all
 $i$.  The stabilisers are as follows:
 \begin{align*}
  \stab_G(0) = \stab_G(1) &= \ip{\lm,\nu} \\
  \stab_G(2) = \stab_G(4) &= \ip{\mu,\nu} \\
  \stab_G(3) = \stab_G(5) &= \ip{\lm^2\mu,\lm^2\nu} \\
  \stab_G(6) = \stab_G(8) &= \ip{\lm\mu,\lm\nu} \\
  \stab_G(7) = \stab_G(9) &= \ip{\lm^{-1}\mu,\lm^{-1}\nu} \\
  \stab_G(10) = \dotsb = \stab_G(v_{13}) &= \ip{\lm^2,\nu}.
 \end{align*}
 It follows that $\phi$ must preserve each of the following sets:
 \[ \{0,1\},\;\{2,4\},\;\{3,5\},\;\{6,8\},\;\{7,9\},\;
      \{10,11,12,13\}.
 \]
 The restriction of $\phi$ to $\{2,3,4,5\}$ must commute with the
 restrictions of $\lm$ and $\mu$, which are $(2\;3\;4\;5)$ and
 $(3\;5)$.  It follows easily that the restriction of $\phi$ is
 $\phi_1=(2\;4)(3\;5)$ or the identity.  A similar argument shows that
 the restriction of $\phi$ to $\{6,7,8,9\}$ is $\phi_2=(6\;8)(7\;9)$
 or the identity, and the restriction to $\{10,11,12,13\}$ must be a
 transposition pair or the identity.  Here the possible transposition
 pairs are $\phi_3=(10\;11)(12\;13)$ and $\phi_4=(10\;12)(11\;13)$ and
 $\phi_3\phi_4=(10\;13)(11\;12)$.  The claim follows easily.
 \begin{checks}
  group_check.mpl: check_aut_V()
 \end{checks}
\end{proof}

\subsection{Quotients}
\lbl{sec-quotients}

Let $X$ be a precromulent surface.  It is standard that for any finite
group $H$ of conformal automorphisms, the quotient $X/H$ always has a
canonical structure as a compact connected Riemann surface such that
the projection $X\to X/H$ is a branched cover.  We will need to
understand the genus of $X/H$, which is determined by its Euler
characteristic, which is given by the following result:

\begin{lemma}\lbl{lem-chi-quotient}
 For any subgroup $H\leq D_8$ we have $\chi(X/H)=|V/H|-16/|H|$.
\end{lemma}
\begin{proof}
 We can write $X=A\cup B$, where $A$ is a union of small discs
 around the points of $V$, and $B$ is the closure of the complement of
 $A$.  This means that the set $C=A\cap B$ is a disjoint union of
 circles, so $\chi(C)=0$, so $\chi(A)+\chi(B)=\chi(X)$.  Similarly,
 $C/H$ is again a union of circles, so $\chi(C/H)=0$, so
 $\chi(A/H)+\chi(B/H)=\chi(X/H)$.  Now $A$ and $A/H$ are homotopy
 equivalent to $V$ and $V/H$, so $\chi(A)=14$ and $\chi(A/H)=|V/H|$.
 As $X$ has genus $g=2$ we have $\chi(X)=2-2g=-2$, so
 $\chi(B)=-2-14=-16$.  Next, note that the action of $H$ on $B$ is
 free.  Thus, if we choose a finite regular cell
 structure on $B/H$, then the preimage in $B$ of each cell in $B/H$
 will be a disjoint union of $|H|$ cells.  Using this we see that
 $\chi(B/H)=\chi(B)/|H|=-16/|H|$, so $\chi(X/H)=|V/H|-16/|H|$.
\end{proof}

Recall that we use the following notation for subgroups of $D_8$
\[
  D_8 = \ip{\lm,\mu} \hspace{5em}
  C_4 = \ip{\lm} \hspace{5em}
  C_2 = \ip{\lm^2}.
\]

\begin{corollary}\lbl{cor-quotient-types}
 The surfaces $X/C_2$, $X/C_4$ and $X/D_8$ are all conformally
 equivalent to $\C_\infty$.  However, $X/\ip{\mu}$ and $X/\ip{\lm\mu}$
 are elliptic curves.
\end{corollary}
\begin{proof}
 By the classification of compact connected Riemann surfaces, it will
 suffice to show that $\chi(X/C_2)=\chi(X/C_4)=\chi(X/D_8)=2$ and
 $\chi(X/\ip{\mu})=\chi(X/\ip{\lm\mu})=0$.  If $H\leq D_8$ is
 generated by a single element $\sg$, then $|V/H|$ is just the number
 of cycles (including $1$-cycles) in the permutation corresponding to
 $\sg$.  This gives everything in the following table except for the
 case $H=D_8$, which is easily handled in an \emph{ad-hoc} way.
 \[ \begin{array}{|c|c|c|c|c|}
     \hline
      H           & |H| & \sg    & |V/H| & \chi(X/H) \\ \hline
      C_2         & 2   & \lm^2  & 10    & 2         \\ \hline
      C_4         & 4   & \lm    & 6     & 2         \\ \hline
      \ip{\mu}    & 2   & \mu    & 8     & 0         \\ \hline
      \ip{\lm\mu} & 2   & \lm\mu & 8     & 0         \\ \hline
      D_8         & 8   &        & 4     & 2         \\ \hline
    \end{array}
 \]
\end{proof}

\begin{remark}\lbl{rem-conj-quot}
 Recall that the action of an element $g\in G$ gives an isomorphism
 $X/H\to X/gHg^{-1}$.  In particular, the action of $\lm$ gives
 isomorphisms $X/\ip{\mu}\to X/\ip{\lm^2\mu}$ and
 $X/\ip{\lm\mu}\to X/\ip{\lm^3\mu}$.  Because of this, we will mostly
 restrict attention to $X/\ip{\mu}$ and $X/\ip{\lm\mu}$, and ignore
 $X/\ip{\lm^2\mu}$ and $X/\ip{\lm^3\mu}$.
\end{remark}

\begin{remark}\lbl{rem-smooth-branch}
 All of the above relies on the standard fact that if $Z$ is a Riemann
 surface and $H$ is a finite group of holomorphic automorphisms, then
 $Z/H$ has a natural structure as a Riemann surface, and in particular
 has a smooth structure.  We offer some remarks about this, some of
 which will be needed later.

 More precisely, the claim is that this structure makes $Z/H$ into a
 coequaliser for the action in the analytic category: if $U$ is an
 $H$-invariant open subset of $Z$, and $f\:U\to W$ is an $H$-invariant
 analytic function to another Riemann surface $W$, then $U/H$ is open
 in $Z/H$, and there is a unique analytic function $g\:U/H\to W$ such
 that the composite $U\to U/H\xra{g}W$ is $f$.  Note here that
 coequalisers are automatically unique up to unique isomorphism.
 Thus, it does not matter if we make some arbitrary choices in the
 process of constructing a coequaliser; the result will be independent
 of those choices.

 The proof of the claim is local on $Z$.  Given $z\in Z$, put
 \begin{align*}
  C_0 &= \{\al\in H\st \al(z)=z\} \\
  C_1 &= \{\al\in H\st \al=1 \text{ on some neighbourhood of } z\} \\
  C &= C_0/C_1.
 \end{align*}
 Then each element $\al\in C$ must act
 on $T^*_zZ$ as multiplication by some scalar $\chi(\al)\in\C^\tm$;
 this defines a homomorphism $\chi\:C\to\C^\tm$.  By power series
 methods, one can check that $\chi$ must be injective, and thus that
 $C$ must be cyclic, of order $n$ say.  Now choose a local parameter
 $f_0$ with $f_0(z)=0$, and put
 \[ f(w) = |C|^{-1}\sum_{\al\in C}\chi(\al)^{-1}f_0(\al(w)). \]
 This is the same as $f_0$ to first order, so it is again a local
 parameter, and it satisfies $f(\al(w))=\chi(\al)\,f(w)$.  Using this,
 we reduce to the case where the group $\mu_n$ of $n$'th roots of
 unity acts on $\C$ by multiplication.  Here, the map
 $\sg_n\:z\mapsto z^n$ is easily seen to be a coequaliser.

 Note, however, that the map $\sg_n\:\C\to\C$ is not a coequaliser in
 the smooth category (provided that $n>1$).  Indeed, the function
 $f(z)=|z|^2$ is smooth and $\mu_n$-invariant.  There is a unique map
 $g\:\C\to\R$ with $f=g\circ\sg_n$, namely $g(w)=|w|^{2/n}$.  However,
 $g$ is not smooth.  Because of this, if we start with a smooth
 surface $Z$ and an orientation-preserving action of a finite group
 $H$, there is no obvious way to obtain a smooth structure on $Z/H$.
 Given $z$ and $C$ as above, we can choose a chart $\phi$ at $z$ on
 which $C$ acts by rotation, and using this we obtain a chart
 $\ov{\phi}$ on the quotient.  However, if $\psi$ is another local
 chart at $z$ on which $C$ acts by rotation, then
 $\ov{\psi}^{-1}\circ\ov{\phi}$ need not be smooth.

 We can always obtain a smooth structure on $Z/H$ by choosing a smooth
 invariant Riemannian metric, using this to give $Z$ a conformal
 structure, and then taking a major detour through the analytic
 category as above.  However, the result will depend on the choice of
 metric, and we do not know any way to shortcut the detour.
\end{remark}

\subsection{Curve systems}
\lbl{sec-curve-systems}

In this section, we define what we mean by a \emph{curve system} on a
precromulent surface.  Later we will exhibit curve systems for the
projective family, the hyperbolic family and the embedded family.  We
will also show that the projective family is universal, so in fact
every precromulent surface has a curve system.

\begin{definition}\lbl{defn-precromulent-C}
 Let $X$ be a labelled precromulent surface.  For any $\gm\in G$ we
 put $X^\gm=\{x\in X\st \gm(x)=x\}$.  We then put
 \begin{align*}
  C_0 &= \text{ the component of $v_2$ in } X^{\mu\nu} \\
  C_1 &= \text{ the component of $v_0$ in } X^{\lm\nu} \\
  C_2 &= \text{ the component of $v_0$ in } X^{\lm^3\nu} \\
  C_3 &= \text{ the component of $v_{11}$ in } X^{\lm^2\nu} \\
  C_4 &= \text{ the component of $v_{10}$ in } X^\nu \\
  C_5 &= \text{ the component of $v_0$ in } X^\nu \\
  C_6 &= \text{ the component of $v_0$ in } X^{\lm^2\nu} \\
  C_7 &= \text{ the component of $v_1$ in } X^\nu \\
  C_8 &= \text{ the component of $v_1$ in } X^{\lm^2\nu}.
 \end{align*}
\end{definition}
\begin{remark}\lbl{rem-Ck-circle}
 The elements $\mu\nu$ and $\lm^k\nu$ act on $X$ as antiholomorphic
 involutions.  A standard result, which we will recall as
 Corollary~\ref{cor-fixed-circles}, shows that the fixed set of an
 antiholomorphic involution on a compact Riemann surface is always
 diffeomorphic to a finite disjoint union of circles.  Thus, each of
 the sets $C_k$ above is a circle.  If $\al$ and $\bt$ are distinct
 antiholomorphic involutions in $G$ then $X^\al\cap X^\bt$ is fixed by
 the holomorphic element $\al\bt$ and so is contained in the finite
 set $V$.  Thus, for example, we have $C_0\cap C_1\sse V$.  On the
 other hand, $C_4$ and $C_5$ are two components in $X^\nu$, so they
 are either equal or disjoint.  In fact, we will see later that they
 are always disjoint, but this will require some further theory.
 More generally, $C_4$, $C_5$ and $C_7$ are disjoint, and $C_3$, $C_6$
 and $C_8$ are disjoint.
\end{remark}
\begin{remark}
 The antiholomorphic involution that fixes $C_k$ is represented in
 Maple as \mcode+c_involution[k]+.  For example, \mcode+c_involution[6]+
 evaluates to \mcode+LLN+, which is our Maple notation for $\lm^2\nu$.
\end{remark}

\begin{definition}\lbl{defn-curve-system}
 Let $X$ be a labelled precromulent surface.  A \emph{curve system} on
 $X$ is a system of maps $c_k\:\R\to X$ (for $0\leq k\leq 8$) such
 that:
 \begin{itemize}
  \item[(a)] Each $c_k$ is real-analytic and $2\pi$-periodic and induces an
   embedding $\R/2\pi\Z\to X$.
  \item[(b)] The vertices $v_0,\dotsc,v_{13}$ occur as values of the maps
   $c_0,\dotsc,c_8$, as follows:
   \[ \begin{array}{|c|c|c|c|c|c|c|c|c|c|c|c|c|c|c|}
   \hline
   &0&1&2&3&4&5&6&7&8&9&10&11&12&13\\ \hline
    0&&&0&\ppi&\pi&-\ppi&\tfrac{\pi}{4}&\tfrac{3\pi}{4}&-\tfrac{3\pi}{4}&-\tfrac{\pi}{4}&&&&\\ \hline
    1&0&\pi&&&&&\ppi&&-\ppi&&&&&\\ \hline
    2&0&\pi&&&&&&\ppi&&-\ppi&&&&\\ \hline
    3&&&&\ppi&&-\ppi&&&&&&0&&\pi\\ \hline
    4&&&-\ppi&&\ppi&&&&&&0&&\pi&\\ \hline
    5&0&&&&&&&&&&&\pi&&\\ \hline
    6&0&&&&&&&&&&\pi&&&\\ \hline
    7&&0&&&&&&&&&&&&\pi\\ \hline
    8&&0&&&&&&&&&&&\pi&\\ \hline
   \end{array} \]
   In more detail, if the above table has an angle $\tht$ in column
   $j$ of row $i$, then $c_i(\tht)=v_j$, but if column $j$ of row $i$
   is empty, then $v_j\not\in c_i(\R)$.
  \item[(c)] The group $G$ acts on the curves $c_k$ as follows:
   \begin{align*}
    \lm(c_{ 0}(t)) &= c_{ 0}( t+\pi/2) &
    \mu(c_{ 0}(t)) &= c_{ 0}(-t)       &
    \nu(c_{ 0}(t)) &= c_{ 0}(-t) \\
    \lm(c_{ 1}(t)) &= c_{ 2}( t)       &
    \mu(c_{ 1}(t)) &= c_{ 2}( t + \pi) &
    \nu(c_{ 1}(t)) &= c_{ 2}(-t) \\
    \lm(c_{ 2}(t)) &= c_{ 1}(-t)       &
    \mu(c_{ 2}(t)) &= c_{ 1}( t + \pi) &
    \nu(c_{ 2}(t)) &= c_{ 1}(-t) \\
    \lm(c_{ 3}(t)) &= c_{ 4}( t)       &
    \mu(c_{ 3}(t)) &= c_{ 3}( t + \pi) &
    \nu(c_{ 3}(t)) &= c_{ 3}(-t) \\
    \lm(c_{ 4}(t)) &= c_{ 3}(-t)       &
    \mu(c_{ 4}(t)) &= c_{ 4}(-t - \pi) &
    \nu(c_{ 4}(t)) &= c_{ 4}( t) \\
    \lm(c_{ 5}(t)) &= c_{ 6}( t)       &
    \mu(c_{ 5}(t)) &= c_{ 7}( t)       &
    \nu(c_{ 5}(t)) &= c_{ 5}( t) \\
    \lm(c_{ 6}(t)) &= c_{ 5}(-t)       &
    \mu(c_{ 6}(t)) &= c_{ 8}(-t)       &
    \nu(c_{ 6}(t)) &= c_{ 6}(-t) \\
    \lm(c_{ 7}(t)) &= c_{ 8}( t)       &
    \mu(c_{ 7}(t)) &= c_{ 5}( t)       &
    \nu(c_{ 7}(t)) &= c_{ 7}( t) \\
    \lm(c_{ 8}(t)) &= c_{ 7}(-t)       &
    \mu(c_{ 8}(t)) &= c_{ 6}(-t)       &
    \nu(c_{ 8}(t)) &= c_{ 8}(-t) \\
   \end{align*}
 \end{itemize}
\end{definition}

\begin{remark}
 The details of axiom~(b) are represented in Maple in several
 different ways, which are useful for different purposes.  Consider,
 for example, the fact that $c_2(\pi/2)=v_7$ but $v_7$ does not lie on
 $C_3$.
 \begin{itemize}
  \item \mcode+v_on_c[7,2]+ is \mcode+Pi/2+, but \mcode+v_on_c[7,3]+ is
   \mcode+NULL+.
  \item \mcode+c_gen[2](Pi/2)+ evaluates to \mcode+v_gen[7]+.  Here
   \mcode+v_gen[7]+ is just a symbol, with no assigned value.  On the
   other hand, \mcode+c_gen[2](Pi/4)+ just evaluates to itself,
   corresponding to the fact that we have no axiom about the value of
   $c_2(\pi/4)$.
  \item \mcode+v_track[7]+ is a list of equations, one of which
   is the equation \mcode+2=Pi/2+.  There is no equation in the list
   with $3$ on the left hand side.
  \item \mcode+c_track[2]+ is a list of equations, one of which is the
   equation \mcode+7=Pi/2+.  On the other hand, \mcode+c_track[3]+ has no
   equation with $7$ on the left hand side.
 \end{itemize}
 The details of axiom~(c) are encoded in the table \mcode+act_c_data+,
 which is indexed by pairs \mcode+[g,i]+ with \mcode+g+ in $G$ and
 \mcode+i+ in $\{0,\dotsc,8\}$.  If \mcode+act_c_data[g,i]+ evaluates to
 \mcode+[j,m,a]+ then the corresponding axiom is $g.c_i(t)=c_j(mt+a)$.
\end{remark}

\begin{remark}
 Suppose we have a curve system $(c_k)_{k=0}^8$ and a strictly
 increasing analytic diffeomorphism $u\:\R\to\R$ with $u(-t)=-u(t)$
 and $u(t+\pi/4)=u(t)+\pi/4$; then the maps $c_k\circ u$ give another
 curve system.  Thus, curve systems are not unique.  However, they are
 unique up to a kind of reparametrisation slightly more general than
 that described above; we will not spell out the details.
\end{remark}

\begin{proposition}\lbl{prop-curve-system}
 Let $(c_k)_{k=0}^8$ be a curve system on a labelled precromulent
 surface $X$.  Then:
 \begin{itemize}
  \item[(1)] For each $k$ the map $c_k$ gives a diffeomorphism
   $\R/2\pi\Z\to C_k$.
  \item[(2)] The sets $C_4$, $C_5$ and $C_7$ are disjoint.
  \item[(3)] The sets $C_3$, $C_6$ and $C_8$ are disjoint.
  \item[(4)] For all $i\neq j$ we have $C_i\cap C_j\sse V$ (so a
   precise list of elements of $C_i\cap C_j$ can be read off from
   axiom~(b)).
 \end{itemize}
\end{proposition}
\begin{proof}
 Axiom~(c) gives $\mu\nu(c_0(t))=\mu(c_0(-t))=c_0(t)$, so
 $c_0(\R)\sse X^{\mu\nu}$.  Moreover, $c_0(\R)$ is connected and
 contains $c_0(0)$, which is $v_2$ by axiom~(b).  This proves that
 $c_0(\R)\sse C_0$.  Axiom~(a) tells us that $c_0$ gives a smooth
 embedding $\R/2\pi\Z\to C_0$, but $C_0$ is diffeomorphic to a circle
 by Remark~\ref{rem-Ck-circle}, and any smooth embedding of a circle
 in a circle is necessarily a diffeomorphism.  The same line of
 argument shows that $c_k$ induces a diffeomorphism $\R/2\pi\Z\to C_k$
 for all $k$.  Next, axiom~(b) tells us that $v_0\not\in c_4(\R)=C_4$,
 so $C_5$ is a component of $X^\nu$ which is different from $C_4$ and
 therefore disjoint from $C_4$.  The same line of argument shows that
 $C_4$, $C_5$ and $C_7$ are disjoint, and also that $C_3$, $C_6$ and
 $C_8$ are disjoint.  Now consider an intersection $C_i\cap C_j$ that
 is not covered by~(b) or~(c).  We then find that $C_i\sse X^\gm$ and
 $C_j\sse X^\dl$ for some antiholomorphic involutions $\gm,\dl\in G$
 with $\gm\neq\dl$, so $\gm\dl$ is a nontrivial element of $D_8$.  Any
 element of $C_i\cap C_j$ is fixed by $\gm\dl$, and so lies in $V$ by
 the definition of $V$.
\end{proof}

\begin{proposition}\lbl{prop-empty-boxes}
 Suppose we have a system of maps $c_k\:\R\to X$ such that axioms~(a)
 and~(c) are satisfied.  Suppose also that
 \begin{itemize}
  \item[(p)] The part of axiom~(b) corresponding to the nonempty boxes
   in the table is satisfied.
  \item[(q)] The sets $c_3(\R)$, $c_6(\R)$ and $c_8(\R)$ are disjoint.
 \end{itemize}
 Then the maps $c_k$ give a curve system.
\end{proposition}
\begin{proof}
 First note that part~(a) of Proposition~\ref{prop-curve-system} used
 only axioms that we are still assuming here, so we again have
 $C_k=c_k(\R)$ for all $k$.  We also see from axiom~(c) that
 $\lm(C_3)=C_4$ and $\lm(C_6)=C_5$ and $\lm(C_8)=C_7$, so $C_4$, $C_5$
 and $C_7$ are also disjoint.

 Next, we can redraw the table in axiom~(b) as follows:
 \[ \begin{array}{|c|c|c|c|c|c|c|c|c|c|c|c|c|c|c|}
   \hline
     &0&1  &2   &3   &4   &5    &6&7&8&9&10&11&12&13\\ \hline
    0&A&A  &0   &\ppi&\pi &-\ppi&\tfrac{\pi}{4}&\tfrac{3\pi}{4}&-\tfrac{3\pi}{4}&-\tfrac{\pi}{4}&A&A&A&A\\ \hline
    1&0&\pi&A   &A   &A   &A    &\ppi&A&-\ppi&A&A&A&A&A\\ \hline
    2&0&\pi&A   &A   &A   &A    &A&\ppi&A&-\ppi&A&A&A&A\\ \hline
    3&B&B  &A   &\ppi&A   &-\ppi&A&A&A&A&B&0&B&\pi\\ \hline
    4&B&B  &\ppi&A   &\ppi&A    &A&A&A&A&0&B&\pi&B\\ \hline
    5&0&B  &B   &A   &B   &A    &A&A&A&A&B&\pi&B&B\\ \hline
    6&0&B  &A   &B   &A   &B    &A&A&A&A&\pi&B&B&B\\ \hline
    7&B&0  &B   &A   &B   &A    &A&A&A&A&B&B&B&\pi\\ \hline
    8&B&0  &A   &B   &A   &B    &A&A&A&A&B&B&\pi&B\\ \hline
 \end{array} \]
 All the boxes that were blank in the original table have been marked
 $A$ or $B$.  Consider for example column $5$, corresponding to
 $v_5$.  It follows from the definition of a precromulent labelling
 that the stabiliser group of $v_5$ is
 $\{1,\mu\nu,\lm^2\mu,\lm^2\nu\}$, so in particular $v_5$ is not fixed
 by $\nu$, $\lm\nu$ or $\lm^3\nu$, so it cannot lie in $c_i(\R)$ for
 $i\in\{1,2,4,5,7\}$.  This accounts for all the boxes in column $5$
 marked $A$.  We also have $-\pi/2$ in row $3$, indicating that
 $v_5=c_3(-\pi/2)\in C_3$.  As $C_3$, $C_6$ and $C_8$ are disjoint, we
 see that $v_5\not\in c_6(\R)$ and $v_5\not\in c_8(\R)$, which
 accounts for the remaining two boxes in column $5$ marked $B$.  The
 same line of argument works for all the other columns.
\end{proof}

\begin{definition}\lbl{defn-std-isotropy}
 We say that $X$ has \emph{standard isotropy} if
 \begin{align*}
  X^{\mu\nu} &= C_0 \\
  X^{\lm\nu} &= C_1 \\
  X^{\lm^3\nu} &= C_2 \\
  X^{\nu} &= C_4\amalg C_5 \amalg C_7 \\
  X^{\lm^2\nu} &= C_3\amalg C_6 \amalg C_8 \\
  X^{\lm^2\mu\nu} &= \emptyset.
 \end{align*}
\end{definition}
We will show later that every cromulent surface has standard
isotropy.

\subsection{Holomorphic curve systems}
\lbl{sec-holomorphic-curves}

For any curve system, it turns out that each map $c_k\:\R\to X$ can be
extended to give a holomorphic map defined on a suitable neighbourhood
of $\R$ in $\C$.  We will start by developing the relevant theory in a
slightly more abstract setting.

\begin{proposition}\lbl{prop-chart}
 Let $X$ be a Riemann surface, and let $c\:\R\to X$ be a real analytic
 map with $c'(0)\neq 0$.  Then there is a unique germ of an analytic
 map $\phi\:\C\to X$ with $\phi(t)=c(t)$ for small real values of
 $t$.  Similarly, there is a unique local conformal
 parameter $z$ at $c(0)$ such that $z(c(t))=t$ for small $t\in\R$.
 Moreover, if $\tau$ is an antiholomorphic involution on $X$ with
 $\tau(c(t))=c(t)$ for all $t$, then $z(\tau(u))=\ov{z(u)}$.
\end{proposition}
\begin{proof}
 The claim is local on $X$, so we may assume that $X=\C$ and
 $c(0)=0$.  As $c$ is real analytic, there are coefficients $a_k\in\C$
 such that $c(t)=\sum_ka_kt^k$, with the sum being absolutely
 convergent for small real values of $t$.  It follows in a standard
 way that the sum is still absolutely convergent for small complex
 values of $t$, and this gives us a germ of a complex analytic map
 $\phi\:\C\to X$ extending $c$.  This is unique, by the Identity
 Principle.  We also have $\phi'(0)=c'(0)\neq 0$, so $\phi$ is locally
 invertible near $0$, and the inverse is the unique local parameter
 $z$ such that $z(c(t))=t$.

 Now suppose that $\tau$ is an antiholomorphic involution on $X$ with
 $\tau(c(t))=c(t)$ for all $t$.  Then the map
 $u\mapsto\ov{z(\tau(u))}$ has the defining property of $z$, and so is
 the same as $z$, as claimed.
\end{proof}

Given $c\:\R\to X$, we can apply the above proposition to $c(t_0+t)$
for various different values of $t_0$, and then patch the results
together.  To organise this construction, we introduce the
following definitions:
\begin{definition}\lbl{defn-band}
 Consider a point $z=x+iy\in\C$.  We let $\CQ(z)$ denote the set of
 pairs $(U_0,\tc_0)$, where $U_0$ is a convex open subset of $\C$
 containing $x$ and $z$, and $\tc_0\:U_0\to X$ is a holomorphic map with
 $\tc_0|_{U_0\cap\R}=c|_{U_0\cap\R}$.  We then put
 $V=\{z\st\CQ(z)\neq\emptyset\}$.  If $z\in V$ then we choose any
 $(U_0,\tc_0)\in\CQ(z)$ and put $\tc(z)=\tc_0(z)$; a straightforward
 argument with the identity principle shows that this is independent
 of the choice of $(U_0,\tc_0)$.
\end{definition}

\begin{proposition}\lbl{prop-band-chart}
 The set $V$ is open in $\C$ and contains $\R$, and it is closed under
 conjugation.  The map $\tc\:V\to X$ is holomorphic, and satisfies
 $\tc|_{V\cap\R}=c$ and $\tc(\ov{z})=\tau(\tc(z))$.  Moreover, if
 $c(t+2\pi)=c(t)$ for all $t$, then we also have $V+2\pi=V$ and
 $\tc(t+2\pi)=\tc(t)$.
\end{proposition}
\begin{proof}
 Straightforward.
\end{proof}

\begin{corollary}\lbl{cor-band-charts}
 Suppose that $X$ is a cromulent surface, and $(c_k)_{k=0}^8$ is a
 curve system.  Then each map $c_k\:\R\to X$ has a canonical
 holomorphic extension $\tc_k\:V_k\to X$, where $V_k$ is a
 $2\pi$-periodic open neighbourhood of $\R$ in $\C$. \qed
\end{corollary}

\subsection{Fundamental domains}
\lbl{sec-fundamental}

\begin{definition}\lbl{defn-fundamental}
 Let $X$ be a compact Hausdorff space, and let $H$ be a finite group acting
 continuously on $X$.  Let $F$ be a closed subset of $X$.
 \begin{itemize}
  \item[(a)] We say that $F$ is a \emph{fundamental domain} for $H$ if
   $X=\bigcup_{\gm\in H}\gm(F)$, and $\text{int}(F)\cap\gm(F)=\emptyset$
   for all $\gm\in H\sm\{1\}$.
  \item[(b)] We say that $F$ is a \emph{retractive fundamental domain} if,
   in addition, there is a continuous map $r\:X\to F$ such that
   \begin{itemize}
    \item[(i)] $r(x)=x$ for all $x\in F$ (so $r$ is a retraction).
    \item[(ii)] $r(\gm(x))=r(x)$ for all $x\in X$ and $\gm\in H$.
   \end{itemize}
 \end{itemize}
\end{definition}

\begin{proposition}\lbl{prop-fundamental}
 Let $F$ be a retractive fundamental domain for $H$, with retraction $r$.
 Then
 \begin{itemize}
  \item[(a)] For all $x\in X$, the point $r(x)$ lies in the same
   $H$-orbit as $x$.
  \item[(b)] For all $\gm\in H$ we have
   $F\cap\gm(F)=\{x\in F\st \gm(x)=x\}$.
  \item[(c)] There is a canonical homeomorphism $X/H\simeq F$.
  \item[(d)] There is a canonical homeomorphism
   $X\simeq(G\tm F)/\sim$, where $(\gm_0,x_0)\sim(\gm_1,x_1)$ iff
   $x_0=x_1$ and $\gm_1^{-1}\gm_0\in\stab_G(x_0)$.
 \end{itemize}
\end{proposition}
\begin{proof}
 \begin{itemize}
  \item[(a)] As $F$ is a fundamental domain, we have $x=\gm(y)$ for
   some $y\in F$ and $\gm\in H$.  This gives $r(x)=r(\gm(y))$, but we
   can use axioms~(ii) and~(i) to see that $r(\gm(y))=r(y)=y$, so
   $r(x)=y$.  Thus, $x$ and $r(x)$ lie in the same $H$-orbit.
  \item[(b)] Now suppose that $x\in F\cap\gm(F)$, so $x=\gm(y)$ for
   some $y\in F$.  We now have $x=r(x)=r(\gm(y))=r(y)=y$, so
   $x=\gm(x)$ as required.
  \item[(c)] We have an inclusion $j\:F\to X$ and a projection
   $p\:X\to X/H$.  We will show that $pj$ is a homeomorphism.

   As $r$ is continuous with $r(\gm(x))=r(x)$ for all $x$ and $\gm$,
   we see that there is a unique map $\ov{r}\:X/H\to F$ with
   $\ov{r}p=r$, and that this is continuous.  As $r$ is a retraction
   we have $\ov{r}pj=rj=1$.  Next, as $x$ is in the same orbit as
   $r(x)$, we have $p(x)=pjr(x)=pj\ov{r}p(x)$.  As $p$ is surjective
   it follows that $pj\ov{r}=1$.  This proves that $\ov{r}$ is an
   inverse for $pj$, as required.
  \item[(d)] We have a continuous map $m\:G\tm F\to X$ given by
   $m(\gm,x)=\gm(x)$.  As $X=\bigcup_\gm\gm(F)$ we see that $m$ is
   surjective.  The source and target are compact Hausdorff spaces, so
   $m$ is automatically a quotient map.  If
   $m(\gm_0,x_0)=m(\gm_1,x_1)$ then the element $\gm=\gm_1^{-1}\gm_0$
   has $\gm(x_0)=x_1$.  Applying $r$ to this gives $x_0=x_1$, and it
   follows that $\gm\in\stab_G(x_0)$.  It follows that $m$ induces a
   homeomorphism $(G\tm F)/\sim\to X$, as claimed.
 \end{itemize}
\end{proof}

Now let $X$ be a labelled precromulent surface with a curve system.
By the axioms for a curve system, we have
\begin{align*}
 v_0 &= c_1(0) = c_5(0) &
 v_3 &= c_0(\pi/2) = c_3(\pi/2) \\
 v_6 &= c_0(\pi/4) = c_1(\pi/2) &
 v_{11} &= c_3(0) = c_5(\pi).
\end{align*}
This means that the set
\[ DF_{16} =
    c_0([\tfrac{\pi}{4},\ppi]) \cup
    c_1([0,\ppi]) \cup
    c_3([0,\ppi]) \cup
    c_5([0,\pi])
\]
fits together as follows:
\begin{center}
 \begin{tikzpicture}[scale=5]
  \begin{scope}
   \draw[blue]       (1.00,0.00) -- (1.00,1.00);
   \draw[blue,->]    (1.00,0.00) -- (1.00,0.50);
   \draw[green]      (1.00,0.00) -- (0.00,0.00);
   \draw[green,->]   (1.00,0.00) -- (0.50,0.00);
   \draw[magenta]    (1.00,1.00) -- (0.00,1.00);
   \draw[magenta,->] (1.00,1.00) -- (0.50,1.00);
   \draw[cyan]       (0.00,0.00) -- (0.00,1.00);
   \draw[cyan,->]    (0.00,0.00) -- (0.00,0.50);
   \fill (1.00,0.00) circle(0.005);
   \fill (1.00,1.00) circle(0.005);
   \fill (0.00,0.00) circle(0.005);
   \fill (0.00,1.00) circle(0.005);
   \draw (1.00,0.00) node[anchor=north] {$\ss v_0$};
   \draw (1.00,1.00) node[anchor=south] {$\ss v_{11}$};
   \draw (0.00,0.00) node[anchor=north] {$\ss v_{6}$};
   \draw (0.00,1.00) node[anchor=south] {$\ss v_3$};
   \draw (1.00,0.50) node[anchor=west]  {$\ss c_5([0,\pi])$};
   \draw (0.00,0.50) node[anchor=east]  {$\ss c_0([\pi/4,\pi/2])$};
   \draw (0.50,0.00) node[anchor=north] {$\ss c_1([0,\pi/2])$};
   \draw (0.50,1.00) node[anchor=south] {$\ss c_3([0,\pi/2])$};
  \end{scope}
 \end{tikzpicture}
\end{center}
(Using Proposition~\ref{prop-curve-system}, we see that the four
boundary arcs cannot have any additional intersection points.)
\begin{remark}
 Information about the above picture is stored as a table in Maple in
 the global variable \mcode+F16_curve_limits+, which is defined in
 \fname+cromulent.mpl+.  For example, \mcode+F16_curve_limits[1]+ is
 the range \mcode+0..Pi/2+, whereas \mcode+F16_curve_limits[2]+ is
 \mcode+NULL+ (because none of the sides of $DF_{16}$ lies along
 $C_2$).  Note that Maple does not display the full structure of
 tables by default; to see all entries in \mcode+F16_curve_limits+, one
 needs to enter \mcode+eval(F16_curve_limits)+, not just
 \mcode+F16_curve_limits+.
\end{remark}
\begin{remark}
 The colours in the above diagram will be used systematically
 throughout this monograph: the curve $c_0$ is cyan, the curves $c_1$
 and $c_2$ are green, the curves $c_3$ and $c_4$ are magenta, and the
 curves $c_5$ to $c_8$ are blue.  The colour of $c_k$ is represented
 in Maple as \mcode+c_colour[k]+.  Readers who have trouble
 distinguishing these colours can try changing the definitions of
 \mcode+c_colour[k]+ in the file \fname+cromulent.mpl+ and
 regenerating the diagrams using the functions in various files called
 \fname+plots.mpl+ appearing in several different directories.
 However, colour should not be strictly necessary for any of the
 diagrams. 
\end{remark}

\begin{lemma}\lbl{lem-F-stabilisers}
 The stabilisers of points in $DF_{16}$ are as follows:
 \begin{align*}
  \stab_G(v_0) &= \ip{\lm,\nu} &
  \stab_G(v_3) &= \ip{\lm^2\mu,\lm^2\nu} \\
  \stab_G(v_6) &= \ip{\lm\mu,\lm\nu} &
  \stab_G(v_{11}) &= \ip{\lm^2,\nu}
 \end{align*}
 \begin{align*}
  \stab_G(c_0(t)) &= \{1,\mu\nu\}   & \text{ for } & \pi/4<t<\pi/2 \\
  \stab_G(c_1(t)) &= \{1,\lm\nu\}   & \text{ for } & 0<t<\pi/2 \\
  \stab_G(c_3(t)) &= \{1,\lm^2\nu\} & \text{ for } & 0<t<\pi/2 \\
  \stab_G(c_5(t)) &= \{1,\nu\}      & \text{ for } & 0<t<\pi.
 \end{align*}
\end{lemma}
\begin{proof}
 The stabilisers for the points $v_i$ are determined by
 Definition~\ref{defn-V-star}.  Next, axiom~(c) in
 Definition~\ref{defn-curve-system} tells us that
 $\stab_G(c_0(t))\supseteq\{1,\mu\nu\}$ for all $t$.  Moreover, if
 $\pi/4<t<\pi/2$ then axiom~(b) tells us that $c_0(t)\neq v_i$ for all
 $i$, so $\stab_G(c_0(t))\cap D_8=\{1\}$.  It is easy to see that any
 subgroup strictly larger than $\{1,\mu\nu\}$ has nontrivial
 intersection with $D_8$, so we must have
 $\stab_G(c_0(t))=\{1,\mu\nu\}$ as claimed.  The same line of argument
 works for $c_1$, $c_3$ and $c_5$.
\end{proof}

\begin{definition}\lbl{defn-standard-F}
 A \emph{standard fundamental domain} is a subset $F_{16}\sse X$ that
 is a retractive fundamental domain for $G$ and is homeomorphic to a
 square and has boundary $DF_{16}$.
\end{definition}

\begin{remark}\lbl{rem-standard-F}
 We will see later that every cromulent surface has a unique standard
 fundamental domain.  Conversely, suppose that $X$ is a labelled
 precromulent surface with a given curve system and a standard
 fundamental domain, and that $\lm_*=i\:T_{v_0}X\to T_{v_0}X$.  Then
 the interior of the standard fundamental domain has the property
 specified in Definition~\ref{defn-precromulent}(d), which proves that
 $X$ is actually cromulent.
\end{remark}

\begin{remark}
 If $F_{16}$ is a standard fundamental domain, the
 Proposition~\ref{prop-fundamental}(d) allows us to identify $X$ with
 $(G\tm F_{16})/\sim$ for a certain equivalence relation $\sim$.  This
 relation depends only on the stabilisers of points in $DF_{16}$,
 which are given by Lemma~\ref{lem-F-stabilisers}.  We can also
 identify $F_{16}$ with $[0,1]^2$ and thus identify $X$ with a
 quotient of $G\tm [0,1]^2$.
\end{remark}

If we perform only some of the identifications given by the above
equivalence relation, we obtain the following space, which we call
$\Net_0$.  It is clearly homeomorphic to a square.
\begin{center}
 \begin{tikzpicture}[scale=.5]
  \draw[  green] (  0,  0) -- (  6,  6);
  \draw[  green] (  0,  0) -- ( -6,  6);
  \draw[  green] (  0,  0) -- ( -6, -6);
  \draw[  green] (  0,  0) -- (  6, -6);
  \draw[   blue] (  0,  0) -- (  0,  8);
  \draw[   blue] (  0,  0) -- (  0, -8);
  \draw[   blue] (  0,  0) -- (  8,  0);
  \draw[   blue] (  0,  0) -- ( -8,  0);
  \draw[   blue] ( 12, 12) -- (  4, 12);
  \draw[   blue] ( 12, 12) -- ( 12,  4);
  \draw[  green] ( 12, 12) -- (  6,  6);
  \draw[  green] (-12, 12) -- ( -6,  6);
  \draw[  green] (-12,-12) -- ( -6, -6);
  \draw[  green] ( 12,-12) -- (  6, -6);
  \draw[   cyan] (  4,  8) -- (  6,  6);
  \draw[magenta] (  4,  8) -- (  0,  8);
  \draw[magenta] (  4,  8) -- (  4, 12);
  \draw[   cyan] (  8,  4) -- (  6,  6);
  \draw[magenta] (  8,  4) -- (  8,  0);
  \draw[magenta] (  8,  4) -- ( 12,  4);
  \draw[   cyan] ( -4,  8) -- ( -6,  6);
  \draw[magenta] ( -4,  8) -- (  0,  8);
  \draw[magenta] ( -4,  8) -- ( -4, 12);
  \draw[   cyan] (  8, -4) -- (  6, -6);
  \draw[magenta] (  8, -4) -- (  8,  0);
  \draw[magenta] (  8, -4) -- ( 12, -4);
  \draw[   blue] (-12, 12) -- ( -4, 12);
  \draw[   blue] (-12, 12) -- (-12,  4);
  \draw[   blue] (-12,-12) -- ( -4,-12);
  \draw[   blue] (-12,-12) -- (-12, -4);
  \draw[   blue] ( 12,-12) -- (  4,-12);
  \draw[   blue] ( 12,-12) -- ( 12, -4);
  \draw[   cyan] (  4, -8) -- (  6, -6);
  \draw[magenta] (  4, -8) -- (  0, -8);
  \draw[magenta] (  4, -8) -- (  4,-12);
  \draw[   cyan] ( -8,  4) -- ( -6,  6);
  \draw[magenta] ( -8,  4) -- ( -8,  0);
  \draw[magenta] ( -8,  4) -- (-12,  4);
  \draw[   cyan] ( -4, -8) -- ( -6, -6);
  \draw[magenta] ( -4, -8) -- (  0, -8);
  \draw[magenta] ( -4, -8) -- ( -4,-12);
  \draw[   cyan] ( -8, -4) -- ( -6, -6);
  \draw[magenta] ( -8, -4) -- ( -8,  0);
  \draw[magenta] ( -8, -4) -- (-12, -4);
  \draw[-angle 90, cyan] (  4,  8) -- (  5,  7);
  \draw[-angle 90, cyan] (  6,  6) -- (  7,  5);
  \draw[-angle 90, cyan] (  8, -4) -- (  7, -5);
  \draw[-angle 90, cyan] (  6, -6) -- (  5, -7);
  \draw[-angle 90, cyan] ( -4, -8) -- ( -5, -7);
  \draw[-angle 90, cyan] ( -6, -6) -- ( -7, -5);
  \draw[-angle 90, cyan] ( -8,  4) -- ( -7,  5);
  \draw[-angle 90, cyan] ( -6,  6) -- ( -5,  7);
  \draw ( -5,  7) node[anchor=north west] {$\ss c_0$};
  \draw (  5,  7) node[anchor=north east] {$\ss c_0$};
  \draw ( -5, -7) node[anchor=south west] {$\ss c_0$};
  \draw (  5, -7) node[anchor=south east] {$\ss c_0$};
  \draw[-angle 90,green] (  3,  3) -- (  4,  4);
  \draw[-angle 90,green] (  7,  7) -- (  8,  8);
  \draw[-angle 90,green] ( -5, -5) -- ( -4, -4);
  \draw[-angle 90,green] ( -9, -9) -- ( -8, -8);
  \draw (  4,  4) node[anchor=north west] {$\ss c_1$};
  \draw[-angle 90,green] ( -3,  3) -- ( -4,  4);
  \draw[-angle 90,green] ( -7,  7) -- ( -8,  8);
  \draw[-angle 90,green] (  5, -5) -- (  4, -4);
  \draw[-angle 90,green] (  9, -9) -- (  8, -8);
  \draw ( -4,  4) node[anchor=north east] {$\ss c_2$};
  \draw[-angle 90,magenta] (-11, -4) -- (-10, -4);
  \draw[-angle 90,magenta] ( -8, -4) -- ( -8, -2);
  \draw[-angle 90,magenta] ( -8,  0) -- ( -8,  2);
  \draw[-angle 90,magenta] ( -9,  4) -- (-10,  4);
  \draw ( -8, -2) node[anchor=east] {$\ss c_3$};
  \draw[-angle 90,magenta] ( 11, -4) -- ( 10, -4);
  \draw[-angle 90,magenta] (  8, -4) -- (  8, -2);
  \draw[-angle 90,magenta] (  8,  0) -- (  8,  2);
  \draw[-angle 90,magenta] (  9,  4) -- ( 10,  4);
  \draw (  8, -2) node[anchor=west] {$\ss c_3$};
  \draw[-angle 90,magenta] (  4, 11) -- (  4, 10);
  \draw[-angle 90,magenta] (  4,  8) -- (  2,  8);
  \draw[-angle 90,magenta] (  0,  8) -- ( -2,  8);
  \draw[-angle 90,magenta] ( -4,  9) -- ( -4, 10);
  \draw ( -2,  8) node[anchor=south] {$\ss c_4$};
  \draw[-angle 90,magenta] (  4,-11) -- (  4,-10);
  \draw[-angle 90,magenta] (  4, -8) -- (  2, -8);
  \draw[-angle 90,magenta] (  0, -8) -- ( -2, -8);
  \draw[-angle 90,magenta] ( -4, -9) -- ( -4,-10);
  \draw ( -2, -8) node[anchor=north] {$\ss c_4$};
  \draw[-angle 90,blue] (  3,  0) -- (  4,  0);
  \draw[-angle 90,blue] ( -5,  0) -- ( -4,  0);
  \draw (  4,  0) node[anchor=north] {$\ss c_5$};
  \draw[-angle 90,blue] (  0,  3) -- (  0,  4);
  \draw[-angle 90,blue] (  0, -5) -- (  0, -4);
  \draw (  0,  4) node[anchor=east] {$\ss c_6$};
  \draw[-angle 90,blue] (-12,  7) -- (-12,  8);
  \draw[-angle 90,blue] (-12, -7) -- (-12, -8);
  \draw[-angle 90,blue] ( 12,  9) -- ( 12,  8);
  \draw[-angle 90,blue] ( 12, -9) -- ( 12, -8);
  \draw (-12,  8) node[anchor=east] {$\ss c_7$};
  \draw (-12, -8) node[anchor=east] {$\ss c_7$};
  \draw ( 12,  8) node[anchor=west] {$\ss c_7$};
  \draw ( 12, -8) node[anchor=west] {$\ss c_7$};
  \draw[-angle 90,blue] ( -9, 12) -- ( -8, 12);
  \draw[-angle 90,blue] (  9, 12) -- (  8, 12);
  \draw[-angle 90,blue] ( -7,-12) -- ( -8,-12);
  \draw[-angle 90,blue] (  7,-12) -- (  8,-12);
  \draw (  8,-12) node[anchor=north] {$\ss c_8$};
  \draw ( -8,-12) node[anchor=north] {$\ss c_8$};
  \draw (  8, 12) node[anchor=south] {$\ss c_8$};
  \draw ( -8, 12) node[anchor=south] {$\ss c_8$};
  \fill[black] (  0,  0) circle(0.10);
  \fill[black] ( -4, -8) circle(0.10);
  \fill[black] ( 12, 12) circle(0.10);
  \fill[black] (  4,  8) circle(0.10);
  \fill[black] (  8,  4) circle(0.10);
  \fill[black] ( -4,  8) circle(0.10);
  \fill[black] (  8, -4) circle(0.10);
  \fill[black] (  6,  6) circle(0.10);
  \fill[black] ( -6,  6) circle(0.10);
  \fill[black] (  6, -6) circle(0.10);
  \fill[black] ( -6, -6) circle(0.10);
  \fill[black] (  8,  0) circle(0.10);
  \fill[black] (  0,  8) circle(0.10);
  \fill[black] ( 12,  4) circle(0.10);
  \fill[black] (  4, 12) circle(0.10);
  \fill[black] ( -8,  4) circle(0.10);
  \fill[black] (  0, -8) circle(0.10);
  \fill[black] ( -8, -4) circle(0.10);
  \fill[black] (-12,  4) circle(0.10);
  \fill[black] ( -4,-12) circle(0.10);
  \fill[black] (-12, 12) circle(0.10);
  \fill[black] (-12,-12) circle(0.10);
  \fill[black] ( -8,  0) circle(0.10);
  \fill[black] (-12, -4) circle(0.10);
  \fill[black] (  4,-12) circle(0.10);
  \fill[black] ( 12,-12) circle(0.10);
  \fill[black] ( 12, -4) circle(0.10);
  \fill[black] (  4, -8) circle(0.10);
  \fill[black] ( -4, 12) circle(0.10);
  \draw (  0,  0) node[anchor=north]{$\ss v_{0}$};
  \draw ( -4, -8) node[anchor=north west]{$\ss v_{4.1}$};
  \draw ( 12, 12) node[anchor=south west]{$\ss v_{1}$};
  \draw (  4,  8) node[anchor=south east]{$\ss v_{2}$};
  \draw (  8,  4) node[anchor=north west]{$\ss v_{3}$};
  \draw ( -4,  8) node[anchor=south west]{$\ss v_{4}$};
  \draw (  8, -4) node[anchor=south west]{$\ss v_{5}$};
  \draw (  6,  6) node[anchor=north]{$\ss v_{6}$};
  \draw ( -6,  6) node[anchor=north]{$\ss v_{7}$};
  \draw (  6, -6) node[anchor=north]{$\ss v_{9}$};
  \draw ( -6, -6) node[anchor=north]{$\ss v_{8}$};
  \draw (  8,  0) node[anchor=west]{$\ss v_{11}$};
  \draw (  0,  8) node[anchor=south]{$\ss v_{10}$};
  \draw ( 12,  4) node[anchor=north west]{$\ss v_{13}$};
  \draw (  4, 12) node[anchor=south east]{$\ss v_{12}$};
  \draw ( -8,  4) node[anchor=north east]{$\ss v_{3.1}$};
  \draw (  0, -8) node[anchor=north]{$\ss v_{10.1}$};
  \draw ( -8, -4) node[anchor=south east]{$\ss v_{5.1}$};
  \draw (-12,  4) node[anchor=north east]{$\ss v_{13.3}$};
  \draw ( -4,-12) node[anchor=north west]{$\ss v_{12.1}$};
  \draw (-12, 12) node[anchor=south east]{$\ss v_{1.1}$};
  \draw (-12,-12) node[anchor=north east]{$\ss v_{1.2}$};
  \draw ( -8,  0) node[anchor=east]{$\ss v_{11.1}$};
  \draw (-12, -4) node[anchor=south east]{$\ss v_{13.1}$};
  \draw (  4,-12) node[anchor=north east]{$\ss v_{12.3}$};
  \draw ( 12,-12) node[anchor=north west]{$\ss v_{1.3}$};
  \draw ( 12, -4) node[anchor=south west]{$\ss v_{13.2}$};
  \draw (  4, -8) node[anchor=north east]{$\ss v_{2.1}$};
  \draw ( -4, 12) node[anchor=south west]{$\ss v_{12.2}$};
  \draw (6.00,2.00) node{$1$};
  \draw (-2.00,6.00) node{$\lm$};
  \draw (-6.00,-2.00) node{$\lm^2$};
  \draw (2.00,-6.00) node{$\lm^3$};
  \draw (10.00,-6.00) node{$\mu$};
  \draw (6.00,10.00) node{$\lm\mu$};
  \draw (-10.00,6.00) node{$\lm^2\mu$};
  \draw (-6.00,-10.00) node{$\lm^3\mu$};
  \draw (6.00,-2.00) node{$\nu$};
  \draw (2.00,6.00) node{$\lm\nu$};
  \draw (-6.00,2.00) node{$\lm^2\nu$};
  \draw (-2.00,-6.00) node{$\lm^3\nu$};
  \draw (10.00,6.00) node{$\mu\nu$};
  \draw (-6.00,10.00) node{$\lm\mu\nu$};
  \draw (-10.00,-6.00) node{$\lm^2\mu\nu$};
  \draw (6.00,-10.00) node{$\lm^3\mu\nu$};
 \end{tikzpicture}
\end{center}
\begin{checks}
 nets_check.mpl: check_nets()
\end{checks}
\begin{remark}
 There is code for dealing with nets in the file \fname+nets.mpl+.
 This uses the object oriented programming framework described in
 Section~\ref{sec-oo-maple}.  Information about $\Net_0$ is stored in
 the variable \mcode+net_0+, as an instance of the class \mcode+net+.
 This means that
 \begin{itemize}
  \item One can enter \mcode+net_0["v"][12.1]+ to retrieve the
   coordinates of the point $v_{12.1}$ in the above picture.
  \item One can enter \mcode+net_0["squares"][M]+ to retrieve the list
   \mcode+[9,1,13.2,5]+ corresponding to the vertices of the region
   marked $\mu$ (recall that $\mu$ is represented as \mcode+M+ in Maple).
  \item One can enter \mcode+net_0["plot"]+ to generate a picture of
   the net as a Maple plot structure.  Note that this is an example of
   a method rather than a property: it performs an operation rather
   than simply returning information that was previously stored.
  \item One can enter \mcode+net_0["check"]+ to perform various
   consistency checks on the combinatorial structure of the net.
 \end{itemize}
 There are also various other properties and methods.
\end{remark}

\begin{remark}\lbl{rem-match-net}
 Elsewhere we will consider a number of other constructions which give
 partial or global maps between cromulent surfaces and $\R^2$ or $\C$
 or $\C_\infty\simeq S^2$ or $\R^2/\Z^2$.  We will usually arrange the
 details of such maps so that they match up with $\Net_0$ as far as
 possible: $v_0$ will go to the origin, $c_5(t)$ will go to the
 positive $x$-axis for small $t>0$, $c_6(t)$ will go to the positive
 $y$-axis for small $t>0$, and so on.
\end{remark}

We can obtain the space $X$ by performing some additional
identifications on the boundary.  The points marked $v_{12}$,
$v_{12.1}$, $v_{12.2}$ and $v_{12.3}$ all map to $v_{12}$, and
similarly for the other points with fractional labels.

The above net inherits an orientation from $\R^2$, but it also
inherits an orientation from $X$, so we can ask whether these
orientations are the same.  To see that they are, recall that $\lm$
acts on the tangent space $T_{v_0}X$ as multiplication by $i$.  (This
was part of the definition of a cromulent labelling.)  On the other
hand, we have $\lm(c_1(t))=c_2(t)$, and from this we see that $\lm$
acts on the net near $v_0$ as an anticlockwise turn through $\pi/2$.
This implies that the orientations are compatible as claimed.

To explain the gluing conditions on the boundary in more detail, we
use the following, less cluttered version of the above diagram:
\begin{center}
 \begin{tikzpicture}[scale=0.3]
  \draw[blue]    (-12, 12) -- ( -4, 12);
  \draw[blue]    (  4, 12) -- ( 12, 12);
  \draw[blue]    (-12,-12) -- ( -4,-12);
  \draw[blue]    (  4,-12) -- ( 12,-12);
  \draw[blue]    ( 12,-12) -- ( 12, -4);
  \draw[blue]    ( 12,  4) -- ( 12, 12);
  \draw[blue]    (-12,-12) -- (-12, -4);
  \draw[blue]    (-12,  4) -- (-12, 12);
  \draw[magenta] ( -4, 12) -- ( -4,  8) -- (  4,  8) -- (  4, 12);
  \draw[magenta] ( -4,-12) -- ( -4, -8) -- (  4, -8) -- (  4,-12);
  \draw[magenta] ( 12, -4) -- (  8, -4) -- (  8,  4) -- ( 12,  4);
  \draw[magenta] (-12, -4) -- ( -8, -4) -- ( -8,  4) -- (-12,  4);
  \draw[-angle 90,blue] ( 12, 12) -- (  8, 12);
  \draw[-angle 90,blue] (-12, 12) -- ( -8, 12);
  \draw[-angle 90,blue] (  4,-12) -- (  8,-12);
  \draw[-angle 90,blue] ( -4,-12) -- ( -8,-12);
  \draw[-angle 90,blue] ( 12, 12) -- ( 12,  8);
  \draw[-angle 90,blue] ( 12,-12) -- ( 12, -8);
  \draw[-angle 90,blue] (-12,  4) -- (-12,  8);
  \draw[-angle 90,blue] (-12, -4) -- (-12, -8);
  \draw[-angle 90,magenta] (  4,  8) -- (  0,  8);
  \draw[-angle 90,magenta] (  4, -8) -- (  0, -8);
  \draw[-angle 90,magenta] ( -8, -4) -- ( -8,  0);
  \draw[-angle 90,magenta] (  8, -4) -- (  8,  0);
  \draw (  8,  0) node[anchor=east]  {$c_3$};
  \draw ( -8,  0) node[anchor=west]  {$c_3$};
  \draw (  0,  8) node[anchor=north] {$c_4$};
  \draw (  0, -8) node[anchor=south] {$c_4$};
  \draw ( 12, -8) node[anchor=west]  {$c_7^+$};
  \draw ( 12,  8) node[anchor=west]  {$c_7^+$};
  \draw (-12, -8) node[anchor=east]  {$c_7^-$};
  \draw (-12,  8) node[anchor=east]  {$c_7^-$};
  \draw ( -8, 12) node[anchor=south] {$c_8^+$};
  \draw (  8, 12) node[anchor=south] {$c_8^+$};
  \draw ( -8,-12) node[anchor=north] {$c_8^-$};
  \draw (  8,-12) node[anchor=north] {$c_8^-$};
 \end{tikzpicture}
\end{center}
The gluing rules are as follows:
\begin{itemize}
 \item The two edges marked $c_7^+$ are identified together.
 \item The two edges marked $c_7^-$ are identified together.
 \item The two edges marked $c_8^+$ are identified together.
 \item The two edges marked $c_8^-$ are identified together.
 \item The curve marked $c_3$ consisting of three edges at the left of
  the diagram is identified with the corresponding curve at the right
  of the diagram.
 \item The curve marked $c_4$ consisting of three edges at the top of
  the diagram is identified with the corresponding curve at the bottom
  of the diagram.
\end{itemize}
The edges $c_k^+$ (for $k\in\{7,8\}$) become the arcs
$c_k([0,\pi])\sse X$, and the edges $c_k^-$ become the arcs
$c_k([-\pi,0])$.


Here are three more ways we can perform partial gluing to get a net
for $X$; we will call them $\Net_1$, $\Net_2$ and $\Net_3$.
\begin{center}
 \begin{tikzpicture}[scale=1]
  \draw[  green] (  4,  4) -- (  3,  3);
  \draw[   blue] (  4,  4) -- (  0,  4);
  \draw[   blue] (  4,  4) -- (  4,  0);
  \draw[  green] (  2,  2) -- (  3,  3);
  \draw[   blue] (  2,  2) -- (  0,  2);
  \draw[   blue] (  2,  2) -- (  2,  0);
  \draw[   cyan] (  0,  3) -- ( -3,  3);
  \draw[   cyan] (  0,  3) -- (  3,  3);
  \draw[magenta] (  0,  3) -- (  0,  4);
  \draw[magenta] (  0,  3) -- (  0,  2);
  \draw[   cyan] ( -3,  0) -- ( -3,  3);
  \draw[   cyan] ( -3,  0) -- ( -3, -3);
  \draw[magenta] ( -3,  0) -- ( -4,  0);
  \draw[magenta] ( -3,  0) -- ( -2,  0);
  \draw[   cyan] (  0, -3) -- ( -3, -3);
  \draw[   cyan] (  0, -3) -- (  3, -3);
  \draw[magenta] (  0, -3) -- (  0, -4);
  \draw[magenta] (  0, -3) -- (  0, -2);
  \draw[   cyan] (  3,  0) -- (  3, -3);
  \draw[   cyan] (  3,  0) -- (  3,  3);
  \draw[magenta] (  3,  0) -- (  4,  0);
  \draw[magenta] (  3,  0) -- (  2,  0);
  \draw[  green] ( -4,  4) -- ( -3,  3);
  \draw[   blue] ( -4,  4) -- (  0,  4);
  \draw[   blue] ( -4,  4) -- ( -4,  0);
  \draw[  green] ( -4, -4) -- ( -3, -3);
  \draw[   blue] ( -4, -4) -- (  0, -4);
  \draw[   blue] ( -4, -4) -- ( -4,  0);
  \draw[  green] (  4, -4) -- (  3, -3);
  \draw[   blue] (  4, -4) -- (  4,  0);
  \draw[   blue] (  4, -4) -- (  0, -4);
  \draw[  green] ( -2,  2) -- ( -3,  3);
  \draw[   blue] ( -2,  2) -- (  0,  2);
  \draw[   blue] ( -2,  2) -- ( -2,  0);
  \draw[  green] ( -2, -2) -- ( -3, -3);
  \draw[   blue] ( -2, -2) -- (  0, -2);
  \draw[   blue] ( -2, -2) -- ( -2,  0);
  \draw[  green] (  2, -2) -- (  3, -3);
  \draw[   blue] (  2, -2) -- (  2,  0);
  \draw[   blue] (  2, -2) -- (  0, -2);
  \draw (  4,  4) node[anchor=south west]{$\ss 0$};
  \draw (  2,  2) node[anchor=north east]{$\ss 1$};
  \draw (  0,  3) node{$\ss 2$};
  \draw ( -3,  0) node{$\ss 3$};
  \draw (  0, -3) node{$\ss 4$};
  \draw (  3,  0) node{$\ss 5$};
  \draw ( -3,  3) node{$\ss 6$};
  \draw ( -3, -3) node{$\ss 7$};
  \draw (  3,  3) node{$\ss 9$};
  \draw (  3, -3) node{$\ss 8$};
  \draw (  4,  0) node[anchor=west]{$\ss 11$};
  \draw (  0,  4) node[anchor=south]{$\ss 10$};
  \draw (  2,  0) node[anchor=east]{$\ss 13$};
  \draw (  0,  2) node[anchor=north]{$\ss 12$};
  \draw (  0, -4) node[anchor=north]{$\ss 10.1$};
  \draw (  0, -2) node[anchor=south]{$\ss 12.1$};
  \draw ( -2,  2) node[anchor=north west]{$\ss 1.1$};
  \draw (  4, -4) node[anchor=north west]{$\ss 0.3$};
  \draw ( -2, -2) node[anchor=south west]{$\ss 1.2$};
  \draw ( -4,  0) node[anchor=east]{$\ss 11.1$};
  \draw ( -4,  4) node[anchor=south east]{$\ss 0.1$};
  \draw ( -2,  0) node[anchor=west]{$\ss 13.1$};
  \draw (  2, -2) node[anchor=south east]{$\ss 1.3$};
  \draw ( -4, -4) node[anchor=north east]{$\ss 0.2$};
  \draw (-3.50,1.75) node{$1$};
  \draw (-1.75,-3.50) node{$\lm$};
  \draw (3.50,-1.75) node{$\lm^2$};
  \draw (1.75,3.50) node{$\lm^3$};
  \draw (2.50,1.25) node{$\mu$};
  \draw (-1.25,2.50) node{$\lm\mu$};
  \draw (-2.50,-1.25) node{$\lm^2\mu$};
  \draw (1.25,-2.50) node{$\lm^3\mu$};
  \draw (3.50,1.75) node{$\nu$};
  \draw (-1.75,3.50) node{$\lm\nu$};
  \draw (-3.50,-1.75) node{$\lm^2\nu$};
  \draw (1.75,-3.50) node{$\lm^3\nu$};
  \draw (-2.50,1.25) node{$\mu\nu$};
  \draw (-1.25,-2.50) node{$\lm\mu\nu$};
  \draw (2.50,-1.25) node{$\lm^2\mu\nu$};
  \draw (1.25,2.50) node{$\lm^3\mu\nu$};
 \end{tikzpicture}
\end{center}

\vspace{2ex}

\begin{center}
 \begin{tikzpicture}[scale=.4]
  \draw[  green] ( -4,  0) -- (  0,  0);
  \draw[   blue] ( -4,  0) -- ( -4,  4);
  \draw[   blue] ( -4,  0) -- ( -4, -4);
  \draw[  green] ( -4,  0) -- ( -8,  1);
  \draw[  green] ( -4,  0) -- ( -8, -1);
  \draw[  green] (  4,  0) -- (  0,  0);
  \draw[  green] (  4,  0) -- (  8,  1);
  \draw[  green] (  4,  0) -- (  8, -1);
  \draw[   blue] (  4,  0) -- (  4,  4);
  \draw[   blue] (  4,  0) -- (  4, -4);
  \draw[   cyan] (  0,  4) -- (  0,  0);
  \draw[magenta] (  0,  4) -- ( -4,  4);
  \draw[magenta] (  0,  4) -- (  4,  4);
  \draw[   cyan] (  0,  4) -- (  0,  8);
  \draw[   cyan] (  0, -4) -- (  0,  0);
  \draw[magenta] (  0, -4) -- ( -4, -4);
  \draw[magenta] (  0, -4) -- (  4, -4);
  \draw[   cyan] (  0, -4) -- (  0, -8);
  \draw[   cyan] (  8,  4) -- (  8,  1);
  \draw[   cyan] (  8,  4) -- (  8,  8);
  \draw[magenta] (  8,  4) -- (  4,  4);
  \draw[   cyan] (  8, -4) -- (  8, -1);
  \draw[magenta] (  8, -4) -- (  4, -4);
  \draw[   cyan] (  8, -4) -- (  8, -8);
  \draw[   blue] ( -4,  8) -- ( -4,  4);
  \draw[  green] ( -4,  8) -- ( -8,  8);
  \draw[  green] ( -4,  8) -- (  0,  8);
  \draw[   blue] ( -4, -8) -- ( -4, -4);
  \draw[  green] ( -4, -8) -- (  0, -8);
  \draw[  green] ( -4, -8) -- ( -8, -8);
  \draw[  green] (  4,  8) -- (  8,  8);
  \draw[   blue] (  4,  8) -- (  4,  4);
  \draw[  green] (  4,  8) -- (  0,  8);
  \draw[   blue] (  4, -8) -- (  4, -4);
  \draw[  green] (  4, -8) -- (  0, -8);
  \draw[  green] (  4, -8) -- (  8, -8);
  \draw[magenta] ( -8,  4) -- ( -4,  4);
  \draw[   cyan] ( -8,  4) -- ( -8,  1);
  \draw[   cyan] ( -8,  4) -- ( -8,  8);
  \draw[magenta] ( -8, -4) -- ( -4, -4);
  \draw[   cyan] ( -8, -4) -- ( -8, -8);
  \draw[   cyan] ( -8, -4) -- ( -8, -1);
  \draw ( -4,  0) node[anchor=east]{$\ss 0$};
  \draw ( -8,  4) node[anchor=east]{$\ss 4.1$};
  \draw (  4,  0) node[anchor=west]{$\ss 1$};
  \draw (  0,  4) node{$\ss 2$};
  \draw (  0, -4) node{$\ss 3$};
  \draw (  8,  4) node[anchor=west]{$\ss 4$};
  \draw (  8, -4) node[anchor=west]{$\ss 5$};
  \draw (  0,  0) node{$\ss 6$};
  \draw (  8,  1) node[anchor=west]{$\ss 7$};
  \draw (  8, -1) node[anchor=west]{$\ss 9$};
  \draw (  8,  8) node[anchor=south west]{$\ss 8$};
  \draw ( -8,  1) node[anchor=east]{$\ss 7.1$};
  \draw ( -4, -4) node{$\ss 11$};
  \draw ( -4,  4) node{$\ss 10$};
  \draw (  4, -4) node{$\ss 13$};
  \draw (  4,  4) node{$\ss 12$};
  \draw ( -8,  8) node[anchor=south east]{$\ss 8.1$};
  \draw ( -8, -4) node[anchor=east]{$\ss 5.1$};
  \draw (  0, -8) node[anchor=north]{$\ss 7.2$};
  \draw ( -8, -8) node[anchor=north east]{$\ss 8.2$};
  \draw (  4,  8) node[anchor=south]{$\ss 1.1$};
  \draw (  8, -8) node[anchor=north west]{$\ss 8.3$};
  \draw (  4, -8) node[anchor=north]{$\ss 1.2$};
  \draw ( -4,  8) node[anchor=south]{$\ss 0.1$};
  \draw ( -8, -1) node[anchor=east]{$\ss 9.1$};
  \draw (  0,  8) node[anchor=south]{$\ss 9.2$};
  \draw ( -4, -8) node[anchor=north]{$\ss 0.2$};
  \draw (-2.00,-2.00) node{$1$};
  \draw (-6.00,2.25) node{$\lm$};
  \draw (-6.00,-6.00) node{$\lm^2$};
  \draw (-2.00,6.00) node{$\lm^3$};
  \draw (6.00,-2.25) node{$\mu$};
  \draw (2.00,2.00) node{$\lm\mu$};
  \draw (2.00,-6.00) node{$\lm^2\mu$};
  \draw (6.00,6.00) node{$\lm^3\mu$};
  \draw (-6.00,-2.25) node{$\nu$};
  \draw (-2.00,2.00) node{$\lm\nu$};
  \draw (-2.00,-6.00) node{$\lm^2\nu$};
  \draw (-6.00,6.00) node{$\lm^3\nu$};
  \draw (2.00,-2.00) node{$\mu\nu$};
  \draw (6.00,2.25) node{$\lm\mu\nu$};
  \draw (6.00,-6.00) node{$\lm^2\mu\nu$};
  \draw (2.00,6.00) node{$\lm^3\mu\nu$};
 \end{tikzpicture}
 \hspace{1em}
 \begin{tikzpicture}[scale=.4]
  \draw[  green] (  8,  4) -- (  4,  4);
  \draw[   blue] (  8,  4) -- (  8,  8);
  \draw[   blue] (  8,  4) -- (  8,  0);
  \draw[  green] (  0,  4) -- ( -4,  4);
  \draw[  green] (  0,  4) -- (  4,  4);
  \draw[   blue] (  0,  4) -- (  1,  8);
  \draw[   blue] (  0,  4) -- (  0,  0);
  \draw[   blue] (  0,  4) -- ( -1,  8);
  \draw[   cyan] (  4,  8) -- (  4,  4);
  \draw[magenta] (  4,  8) -- (  8,  8);
  \draw[magenta] (  4,  8) -- (  1,  8);
  \draw[   cyan] ( -4,  0) -- ( -4,  4);
  \draw[   cyan] ( -4,  0) -- ( -4, -4);
  \draw[magenta] ( -4,  0) -- (  0,  0);
  \draw[magenta] ( -4,  0) -- ( -8,  0);
  \draw[   cyan] (  4, -8) -- (  4, -4);
  \draw[magenta] (  4, -8) -- (  8, -8);
  \draw[magenta] (  4, -8) -- (  1, -8);
  \draw[   cyan] (  4,  0) -- (  4, -4);
  \draw[   cyan] (  4,  0) -- (  4,  4);
  \draw[magenta] (  4,  0) -- (  8,  0);
  \draw[magenta] (  4,  0) -- (  0,  0);
  \draw[  green] ( -8,  4) -- ( -4,  4);
  \draw[   blue] ( -8,  4) -- ( -8,  8);
  \draw[   blue] ( -8,  4) -- ( -8,  0);
  \draw[  green] ( -8, -4) -- ( -4, -4);
  \draw[   blue] ( -8, -4) -- ( -8, -8);
  \draw[   blue] ( -8, -4) -- ( -8,  0);
  \draw[  green] (  8, -4) -- (  4, -4);
  \draw[   blue] (  8, -4) -- (  8,  0);
  \draw[   blue] (  8, -4) -- (  8, -8);
  \draw[  green] (  0, -4) -- ( -4, -4);
  \draw[  green] (  0, -4) -- (  4, -4);
  \draw[   blue] (  0, -4) -- (  0,  0);
  \draw[   blue] (  0, -4) -- ( -1, -8);
  \draw[   blue] (  0, -4) -- (  1, -8);
  \draw[   cyan] ( -4,  8) -- ( -4,  4);
  \draw[magenta] ( -4,  8) -- ( -8,  8);
  \draw[magenta] ( -4,  8) -- ( -1,  8);
  \draw[   cyan] ( -4, -8) -- ( -4, -4);
  \draw[magenta] ( -4, -8) -- ( -8, -8);
  \draw[magenta] ( -4, -8) -- ( -1, -8);
  \draw (  8,  4) node[anchor=west]{$\ss 0$};
  \draw ( -4, -8) node[anchor=north]{$\ss 4.1$};
  \draw (  0,  4) node[anchor=south]{$\ss 1$};
  \draw (  4,  8) node[anchor=south]{$\ss 2$};
  \draw ( -4,  0) node{$\ss 3$};
  \draw (  4, -8) node[anchor=north]{$\ss 4$};
  \draw (  4,  0) node{$\ss 5$};
  \draw ( -4,  4) node{$\ss 6$};
  \draw ( -4, -4) node{$\ss 7$};
  \draw (  4,  4) node{$\ss 9$};
  \draw (  4, -4) node{$\ss 8$};
  \draw (  8,  0) node[anchor=west]{$\ss 11$};
  \draw (  8,  8) node[anchor=south west]{$\ss 10$};
  \draw (  0,  0) node{$\ss 13$};
  \draw (  1,  8) node[anchor=south]{$\ss 12$};
  \draw ( -8,  8) node[anchor=south east]{$\ss 10.1$};
  \draw (  8, -8) node[anchor=north west]{$\ss 10.3$};
  \draw ( -1,  8) node[anchor=south]{$\ss 12.1$};
  \draw ( -8, -8) node[anchor=north east]{$\ss 10.2$};
  \draw (  0, -4) node[anchor=north]{$\ss 1.1$};
  \draw (  8, -4) node[anchor=west]{$\ss 0.3$};
  \draw ( -8,  0) node[anchor=east]{$\ss 11.1$};
  \draw ( -8,  4) node[anchor=east]{$\ss 0.1$};
  \draw (  1, -8) node[anchor=north]{$\ss 12.3$};
  \draw ( -4,  8) node[anchor=south]{$\ss 2.1$};
  \draw ( -8, -4) node[anchor=east]{$\ss 0.2$};
  \draw ( -1, -8) node[anchor=north]{$\ss 12.2$};
  \draw (-6.00,2.00) node{$1$};
  \draw (-6.00,-6.00) node{$\lm$};
  \draw (6.00,-2.00) node{$\lm^2$};
  \draw (6.00,6.00) node{$\lm^3$};
  \draw (2.00,2.00) node{$\mu$};
  \draw (-2.25,6.00) node{$\lm\mu$};
  \draw (-2.00,-2.00) node{$\lm^2\mu$};
  \draw (2.25,-6.00) node{$\lm^3\mu$};
  \draw (6.00,2.00) node{$\nu$};
  \draw (-6.00,6.00) node{$\lm\nu$};
  \draw (-6.00,-2.00) node{$\lm^2\nu$};
  \draw (6.00,-6.00) node{$\lm^3\nu$};
  \draw (-2.00,2.00) node{$\mu\nu$};
  \draw (-2.25,-6.00) node{$\lm\mu\nu$};
  \draw (2.00,-2.00) node{$\lm^2\mu\nu$};
  \draw (2.25,6.00) node{$\lm^3\mu\nu$};
 \end{tikzpicture}
\end{center}
\begin{checks}
 nets_check.mpl: check_nets()
\end{checks}
In all of $\Net_0,\dotsc,\Net_3$, the labels $0$, $11$, $3$ and $6$ occur
in anticlockwise order around the region marked $1$.  This shows that
all the nets have orientation compatible with each other and thus also
compatible with the orientation of $X$.

Here is another way to assemble the pieces.  The left hand picture
(which we call $\Net_4^+$) consists of eight distorted copies of
$F_{16}$, and is homeomorphic to a disc with two holes, or a ``pair of
pants''.  The right hand picture ($\Net_4^-$) consists of the other eight
translates of $F_{16}$.  The surface can be obtained by gluing the two
pictures together along $C_3\amalg C_6\amalg C_8$: this is a ``pair of
pants decomposition''.
\begin{center}
 \begin{tikzpicture}[scale=0.8]
  \draw[magenta] (-4,-3) -- ( 4,-3) -- ( 4, 3) -- (-4, 3) -- (-4,-3);
  \draw[green]   (-3,-2) -- ( 3,-2) -- ( 3, 2) -- (-3, 2) -- (-3,-2);
  \draw[cyan]    ( 0,-3) -- ( 0, 3);
  \draw[blue]    (-4, 0) -- (-3, 0);
  \draw[blue]    ( 3, 0) -- ( 4, 0);
  \draw[magenta] (-1, 0) -- ( 1, 0);
  \filldraw[draw=blue,fill=gray!20]
   (-3, 0) -- (-2, 1) -- (-1, 0) -- (-2,-1) -- (-3, 0);
  \filldraw[draw=blue,fill=gray!20]
   ( 3, 0) -- ( 2, 1) -- ( 1, 0) -- ( 2,-1) -- ( 3, 0);
  \fill[black] (-4, 0) circle(0.04);
  \fill[black] (-3, 0) circle(0.04);
  \fill[black] (-1, 0) circle(0.04);
  \fill[black] ( 0, 0) circle(0.04);
  \fill[black] ( 1, 0) circle(0.04);
  \fill[black] ( 3, 0) circle(0.04);
  \fill[black] ( 4, 0) circle(0.04);
  \fill[black] ( 0,-3) circle(0.04);
  \fill[black] ( 0,-2) circle(0.04);
  \fill[black] ( 0, 2) circle(0.04);
  \fill[black] ( 0, 3) circle(0.04);
  \draw (-4, 0) node[anchor=east] {$\ss v_{13}$};
  \draw (-3, 0) node[anchor=west] {$\ss v_1$};
  \draw (-1, 0) node[anchor=east] {$\ss v_{12}$};
  \draw ( 0, 0) node[anchor=north east] {$\ss v_2$};
  \draw ( 1, 0) node[anchor=west] {$\ss v_{10}$};
  \draw ( 3, 0) node[anchor=east] {$\ss v_0$};
  \draw ( 4, 0) node[anchor=west] {$\ss v_{11}$};
  \draw ( 0,-3) node[anchor=north] {$\ss v_5$};
  \draw ( 0,-2) node[anchor=north east] {$\ss v_9$};
  \draw ( 0, 2) node[anchor=south east] {$\ss v_6$};
  \draw ( 0, 3) node[anchor=south] {$\ss v_3$};
  \draw ( 3.5, 2.5) node {$\ss 1$};
  \draw ( 3.5,-2.5) node {$\ss \nu$};
  \draw (-3.5, 2.5) node {$\ss \mu\nu$};
  \draw (-3.5,-2.5) node {$\ss \mu$};
  \draw ( 0.9, 1.2) node {$\ss \lm\nu$};
  \draw ( 0.9,-1.2) node {$\ss \lm^3$};
  \draw (-0.9, 1.2) node {$\ss \lm\mu$};
  \draw (-0.9,-1.2) node {$\ss \lm^3\mu\nu$};
  \draw[->,cyan]    ( 0.0, 2.4) -- ( 0.0, 2.5);
  \draw[->,cyan]    ( 0.0, 0.9) -- ( 0.0, 1.0);
  \draw[->,cyan]    ( 0.0,-1.1) -- ( 0.0,-1.0);
  \draw[->,cyan]    ( 0.0,-2.6) -- ( 0.0,-2.5);
  \draw ( 0.0, 0.9) node[anchor=west] {$\ss c_0$};
  \draw[->,green]    ( 1.6, 2.0) -- ( 1.5, 2.0);
  \draw[->,green]    (-1.4, 2.0) -- (-1.5, 2.0);
  \draw ( 1.5, 2.0) node[anchor=south] {$\ss c_1$};
  \draw (-1.5, 2.0) node[anchor=south] {$\ss c_1$};
  \draw[->,green]    ( 1.4,-2.0) -- ( 1.5,-2.0);
  \draw[->,green]    (-1.6,-2.0) -- (-1.5,-2.0);
  \draw ( 1.5,-2.0) node[anchor=north] {$\ss c_2$};
  \draw (-1.5,-2.0) node[anchor=north] {$\ss c_2$};
  \draw[->,magenta] ( 4.0, 1.4) -- ( 4.0, 1.5);
  \draw[->,magenta] ( 4.0,-1.6) -- ( 4.0,-1.5);
  \draw[->,magenta] (-4.0, 1.6) -- (-4.0, 1.5);
  \draw[->,magenta] (-4.0,-1.4) -- (-4.0,-1.5);
  \draw ( 4.0, 1.5) node[anchor=west] {$\ss c_3$};
  \draw ( 4.0,-1.5) node[anchor=west] {$\ss c_3$};
  \draw (-4.0, 1.5) node[anchor=east] {$\ss c_3$};
  \draw (-4.0,-1.5) node[anchor=east] {$\ss c_3$};
  \draw[->,blue]    ( 3.4, 0.0) -- ( 3.5, 0.0);
  \draw ( 3.5, 0.0) node[anchor=north] {$\ss c_5$};
  \draw[->,blue]    (-3.4, 0.0) -- (-3.5, 0.0);
  \draw (-3.5, 0.0) node[anchor=north] {$\ss c_7$};
  \draw[->,blue]    ( 2.6, 0.4) -- ( 2.5, 0.5);
  \draw[->,blue]    ( 1.6, 0.6) -- ( 1.5, 0.5);
  \draw[->,blue]    ( 1.4,-0.4) -- ( 1.5,-0.5);
  \draw[->,blue]    ( 2.4,-0.6) -- ( 2.5,-0.5);
  \draw ( 2.5, 0.5) node[anchor=south west] {$\ss c_6$};
  \draw[->,blue]    (-2.6, 0.4) -- (-2.5, 0.5);
  \draw[->,blue]    (-1.6, 0.6) -- (-1.5, 0.5);
  \draw[->,blue]    (-1.4,-0.4) -- (-1.5,-0.5);
  \draw[->,blue]    (-2.4,-0.6) -- (-2.5,-0.5);
  \draw (-2.5, 0.5) node[anchor=south east] {$\ss c_8$};
  \draw[->,magenta] (-0.6, 0.0) -- (-0.5, 0.0);
  \draw[->,magenta] ( 0.4, 0.0) -- ( 0.5, 0.0);
  \draw ( 0.5, 0.0) node[anchor=north] {$\ss c_4$};
 \end{tikzpicture}
 \qquad
 \begin{tikzpicture}[scale=0.8]
  \draw[magenta] (-4,-3) -- ( 4,-3) -- ( 4, 3) -- (-4, 3) -- (-4,-3);
  \draw[green]   (-3,-2) -- ( 3,-2) -- ( 3, 2) -- (-3, 2) -- (-3,-2);
  \draw[cyan]    ( 0,-3) -- ( 0, 3);
  \draw[blue]    (-4, 0) -- (-3, 0);
  \draw[blue]    ( 3, 0) -- ( 4, 0);
  \draw[magenta] (-1, 0) -- ( 1, 0);
  \filldraw[draw=blue,fill=gray!20]
   (-3, 0) -- (-2, 1) -- (-1, 0) -- (-2,-1) -- (-3, 0);
  \filldraw[draw=blue,fill=gray!20]
   ( 3, 0) -- ( 2, 1) -- ( 1, 0) -- ( 2,-1) -- ( 3, 0);
  \fill[black] (-4, 0) circle(0.04);
  \fill[black] (-3, 0) circle(0.04);
  \fill[black] (-1, 0) circle(0.04);
  \fill[black] ( 0, 0) circle(0.04);
  \fill[black] ( 1, 0) circle(0.04);
  \fill[black] ( 3, 0) circle(0.04);
  \fill[black] ( 4, 0) circle(0.04);
  \fill[black] ( 0,-3) circle(0.04);
  \fill[black] ( 0,-2) circle(0.04);
  \fill[black] ( 0, 2) circle(0.04);
  \fill[black] ( 0, 3) circle(0.04);
  \draw (-4, 0) node[anchor=east] {$\ss v_{13}$};
  \draw (-3, 0) node[anchor=west] {$\ss v_1$};
  \draw (-1, 0) node[anchor=east] {$\ss v_{12}$};
  \draw ( 0, 0) node[anchor=north east] {$\ss v_4$};
  \draw ( 1, 0) node[anchor=west] {$\ss v_{10}$};
  \draw ( 3, 0) node[anchor=east] {$\ss v_0$};
  \draw ( 4, 0) node[anchor=west] {$\ss v_{11}$};
  \draw ( 0,-3) node[anchor=north] {$\ss v_5$};
  \draw ( 0,-2) node[anchor=north east] {$\ss v_8$};
  \draw ( 0, 2) node[anchor=south east] {$\ss v_7$};
  \draw ( 0, 3) node[anchor=south] {$\ss v_3$};
  \draw ( 3.5, 2.5) node {$\ss \lm^2\nu$};
  \draw ( 3.5,-2.5) node {$\ss \lm^2$};
  \draw (-3.5, 2.5) node {$\ss \lm^2\mu$};
  \draw (-3.5,-2.5) node {$\ss \lm^2\mu\nu$};
  \draw ( 0.9, 1.2) node {$\ss \lm$};
  \draw ( 0.9,-1.2) node {$\ss \lm^3\nu$};
  \draw (-0.9, 1.2) node {$\ss \lm\mu\nu$};
  \draw (-0.9,-1.2) node {$\ss \lm^3\mu$};
  \draw[->,cyan]    ( 0.0, 2.6) -- ( 0.0, 2.5);
  \draw[->,cyan]    ( 0.0, 1.1) -- ( 0.0, 1.0);
  \draw[->,cyan]    ( 0.0,-0.9) -- ( 0.0,-1.0);
  \draw[->,cyan]    ( 0.0,-2.4) -- ( 0.0,-2.5);
  \draw ( 0.0, 0.9) node[anchor=west] {$\ss c_0$};
  \draw[->,green]    ( 1.6, 2.0) -- ( 1.5, 2.0);
  \draw[->,green]    (-1.4, 2.0) -- (-1.5, 2.0);
  \draw ( 1.5, 2.0) node[anchor=south] {$\ss c_2$};
  \draw (-1.5, 2.0) node[anchor=south] {$\ss c_2$};
  \draw[->,green]    ( 1.4,-2.0) -- ( 1.5,-2.0);
  \draw[->,green]    (-1.6,-2.0) -- (-1.5,-2.0);
  \draw ( 1.5,-2.0) node[anchor=north] {$\ss c_1$};
  \draw (-1.5,-2.0) node[anchor=north] {$\ss c_1$};
  \draw[->,magenta] ( 4.0, 1.4) -- ( 4.0, 1.5);
  \draw[->,magenta] ( 4.0,-1.6) -- ( 4.0,-1.5);
  \draw[->,magenta] (-4.0, 1.6) -- (-4.0, 1.5);
  \draw[->,magenta] (-4.0,-1.4) -- (-4.0,-1.5);
  \draw ( 4.0, 1.5) node[anchor=west] {$\ss c_3$};
  \draw ( 4.0,-1.5) node[anchor=west] {$\ss c_3$};
  \draw (-4.0, 1.5) node[anchor=east] {$\ss c_3$};
  \draw (-4.0,-1.5) node[anchor=east] {$\ss c_3$};
  \draw[->,blue]    ( 3.6, 0.0) -- ( 3.5, 0.0);
  \draw ( 3.5, 0.0) node[anchor=north] {$\ss c_5$};
  \draw[->,blue]    (-3.6, 0.0) -- (-3.5, 0.0);
  \draw (-3.5, 0.0) node[anchor=north] {$\ss c_7$};
  \draw[->,blue]    ( 2.6, 0.4) -- ( 2.5, 0.5);
  \draw[->,blue]    ( 1.6, 0.6) -- ( 1.5, 0.5);
  \draw[->,blue]    ( 1.4,-0.4) -- ( 1.5,-0.5);
  \draw[->,blue]    ( 2.4,-0.6) -- ( 2.5,-0.5);
  \draw ( 2.5, 0.5) node[anchor=south west] {$\ss c_6$};
  \draw[->,blue]    (-2.6, 0.4) -- (-2.5, 0.5);
  \draw[->,blue]    (-1.6, 0.6) -- (-1.5, 0.5);
  \draw[->,blue]    (-1.4,-0.4) -- (-1.5,-0.5);
  \draw[->,blue]    (-2.4,-0.6) -- (-2.5,-0.5);
  \draw (-2.5, 0.5) node[anchor=south east] {$\ss c_8$};
  \draw[->,magenta] (-0.4, 0.0) -- (-0.5, 0.0);
  \draw[->,magenta] ( 0.6, 0.0) -- ( 0.5, 0.0);
  \draw ( 0.5, 0.0) node[anchor=north] {$\ss c_4$};
 \end{tikzpicture}
\end{center}
In this case the orientation of $\Net_4^+$ is compatible with the
orientation of $X$, but the orientation of $\Net_4^-$ is reversed.

We now give another net which we call $\Net_5$.  Note that the central
octagon is the same as for $\Net_0$, but the outer pieces have been
rearranged.
\begin{center}
 \begin{tikzpicture}[scale=1]
  \draw[  green] (  3,  2) -- (  2,  2);
  \draw[   blue] (  3,  2) -- (  3,  5);
  \draw[  green] (  0,  0) -- (  2,  2);
  \draw[  green] (  0,  0) -- (  2, -2);
  \draw[  green] (  0,  0) -- ( -2, -2);
  \draw[  green] (  0,  0) -- ( -2,  2);
  \draw[   blue] (  0,  0) -- (  3,  0);
  \draw[   blue] (  0,  0) -- (  0,  3);
  \draw[   blue] (  0,  0) -- ( -3,  0);
  \draw[   blue] (  0,  0) -- (  0, -3);
  \draw[   cyan] (  3,  1) -- (  2,  2);
  \draw[magenta] (  3,  1) -- (  3,  0);
  \draw[   cyan] (  1, -3) -- (  2, -2);
  \draw[magenta] (  1, -3) -- (  0, -3);
  \draw[   cyan] ( -3, -1) -- ( -2, -2);
  \draw[magenta] ( -3, -1) -- ( -3,  0);
  \draw[   cyan] ( -1,  3) -- ( -2,  2);
  \draw[magenta] ( -1,  3) -- (  0,  3);
  \draw[   blue] (  2,  6) -- (  3,  5);
  \draw[  green] (  2,  6) -- (  1,  6);
  \draw[  green] ( -2,  3) -- ( -2,  2);
  \draw[   blue] ( -2,  3) -- ( -5,  3);
  \draw[  green] ( -6,  2) -- ( -6,  1);
  \draw[   blue] ( -6,  2) -- ( -5,  3);
  \draw[  green] ( -3, -2) -- ( -2, -2);
  \draw[   blue] ( -3, -2) -- ( -3, -5);
  \draw[  green] ( -2, -6) -- ( -1, -6);
  \draw[   blue] ( -2, -6) -- ( -3, -5);
  \draw[  green] (  2, -3) -- (  2, -2);
  \draw[   blue] (  2, -3) -- (  5, -3);
  \draw[   blue] (  6, -2) -- (  5, -3);
  \draw[  green] (  6, -2) -- (  6, -1);
  \draw[   cyan] ( -3,  1) -- ( -2,  2);
  \draw[   cyan] ( -3,  1) -- ( -6,  1);
  \draw[magenta] ( -3,  1) -- ( -5,  3);
  \draw[magenta] ( -3,  1) -- ( -3,  0);
  \draw[   cyan] (  1,  3) -- (  2,  2);
  \draw[magenta] (  1,  3) -- (  3,  5);
  \draw[magenta] (  1,  3) -- (  0,  3);
  \draw[   cyan] (  1,  3) -- (  1,  6);
  \draw[   cyan] (  3, -1) -- (  2, -2);
  \draw[magenta] (  3, -1) -- (  5, -3);
  \draw[magenta] (  3, -1) -- (  3,  0);
  \draw[   cyan] (  3, -1) -- (  6, -1);
  \draw[   cyan] ( -1, -3) -- ( -2, -2);
  \draw[   cyan] ( -1, -3) -- ( -1, -6);
  \draw[magenta] ( -1, -3) -- ( -3, -5);
  \draw[magenta] ( -1, -3) -- (  0, -3);
  \draw (  3,  2) node[anchor=west]{$\ss 0$};
  \draw (  0,  0) node[anchor=east]{$\ss 1$};
  \draw (  3,  1) node[anchor=west]{$\ss 2$};
  \draw (  1, -3) node[anchor=north]{$\ss 3$};
  \draw ( -3, -1) node[anchor=east]{$\ss 4$};
  \draw ( -5,  3) node[anchor=south]{$\ss 10.1$};
  \draw ( -1,  3) node[anchor=south]{$\ss 5$};
  \draw ( -6,  2) node[anchor=south east]{$\ss 0.3$};
  \draw (  2,  2) node[anchor=south]{$\ss 6$};
  \draw (  2, -2) node[anchor=west]{$\ss 7$};
  \draw ( -2,  2) node[anchor=east]{$\ss 9$};
  \draw ( -2, -2) node[anchor=north]{$\ss 8$};
  \draw (  3,  5) node[anchor=west]{$\ss 11$};
  \draw (  5, -3) node[anchor=north]{$\ss 10$};
  \draw (  0,  3) node[anchor=south]{$\ss 13$};
  \draw (  3,  0) node[anchor=west]{$\ss 12$};
  \draw ( -6,  1) node[anchor=north]{$\ss 6.1$};
  \draw ( -1, -3) node[anchor=north west]{$\ss 5.1$};
  \draw (  2, -3) node[anchor=north]{$\ss 0.6$};
  \draw ( -1, -6) node[anchor=west]{$\ss 9.1$};
  \draw (  3, -1) node[anchor=south west]{$\ss 4.1$};
  \draw ( -3, -2) node[anchor=east]{$\ss 0.4$};
  \draw (  1,  3) node[anchor=south east]{$\ss 3.1$};
  \draw (  6, -2) node[anchor=north west]{$\ss 0.7$};
  \draw (  0, -3) node[anchor=north]{$\ss 13.1$};
  \draw (  2,  6) node[anchor=south west]{$\ss 0.1$};
  \draw (  1,  6) node[anchor=east]{$\ss 7.1$};
  \draw ( -2,  3) node[anchor=south]{$\ss 0.2$};
  \draw ( -3,  1) node[anchor=north east]{$\ss 2.1$};
  \draw ( -2, -6) node[anchor=north east]{$\ss 0.5$};
  \draw ( -3,  0) node[anchor=east]{$\ss 12.1$};
  \draw (  6, -1) node[anchor=south]{$\ss 8.1$};
  \draw ( -3, -5) node[anchor=east]{$\ss 11.1$};
  \draw (2.25,3.00) node{$1$};
  \draw (3.00,-2.25) node{$\lm$};
  \draw (-2.25,-3.00) node{$\lm^2$};
  \draw (-3.00,2.25) node{$\lm^3$};
  \draw (-0.75,2.00) node{$\mu$};
  \draw (2.00,0.75) node{$\lm\mu$};
  \draw (0.75,-2.00) node{$\lm^2\mu$};
  \draw (-2.00,-0.75) node{$\lm^3\mu$};
  \draw (-1.75,-5.00) node{$\nu$};
  \draw (-5.00,1.75) node{$\lm\nu$};
  \draw (1.75,5.00) node{$\lm^2\nu$};
  \draw (5.00,-1.75) node{$\lm^3\nu$};
  \draw (0.75,2.00) node{$\mu\nu$};
  \draw (2.00,-0.75) node{$\lm\mu\nu$};
  \draw (-0.75,-2.00) node{$\lm^2\mu\nu$};
  \draw (-2.00,0.75) node{$\lm^3\mu\nu$};
 \end{tikzpicture}
\end{center}
The point about $\Net_5$ is that it allows us to read off a convenient
presentation of the fundamental group $\pi_1(X,v_0)$.

\begin{definition}
 We define $\bt_i\:[0,1]\to X$ for $i\in\Z/8$ by
 \begin{align*}
  \bt_0(t) &= c_5(2\pi t) \\
  \bt_1(t) &= \begin{cases}
               c_1(   -3t \pi  ) & 0   \leq t \leq 1/6 \\
               c_0((-1-3t)\pi/2) & 1/6 \leq t \leq 2/6 \\
               c_4(( 6t-1)\pi/2) & 2/6 \leq t \leq 4/6 \\
               c_0(( 3t-2)\pi/2) & 4/6 \leq t \leq 5/6 \\
               c_1(( 3-3t)\pi  ) & 5/6 \leq t \leq 1,
              \end{cases}
 \end{align*}
 then $\bt_{i+2j}(t)=\lm^j\bt_i(t)$ for $i\in\{0,1\}$ and
 $j\in\{1,2,3\}$.
\end{definition}

It is straightforward to check that $\bt_k(0)=\bt_k(1)=v_0$ for all
$k$, so $\bt_k$ represents an element of $\pi_1(X,v_0)$.

\begin{proposition}\lbl{prop-pi-one}
 $\pi_1(X,v_0)$ is generated by the elements $\bt_i$, subject only to
 the relations $\bt_i\bt_{i+4}=1$ and
 \[ \bt_0\bt_1\bt_2\bt_3\bt_4\bt_5\bt_6\bt_7 = 1. \]
\end{proposition}
\begin{proof}
 Inspection of the definitions, together with part~(c) of
 Definition~\ref{defn-curve-system}, shows that
 $\bt_{i+4}(t)=\bt_i(1-t)$ for all $i$, so $\bt_{i+4}$ is inverse to
 $\bt_i$ in $\pi_1(X,v_0)$.  The paths $\bt_i$ appear in $\Net_5$ as
 follows:
 \begin{center}
  \begin{tikzpicture}[scale=0.6]
   \draw[  blue,-<] (  3,  2) -- (  3,  4);
   \draw[  blue,-<] ( -2,  3) -- ( -4,  3);
   \draw[  blue,-<] ( -3, -2) -- ( -3, -4);
   \draw[  blue,-<] (  2, -3) -- (  4, -3);
   \draw[orange,-<] (  3, -1) -- (  3,  0);
   \draw[orange,-<] (  1,  3) -- (  0,  3);
   \draw[orange,-<] ( -3,  1) -- ( -3,  0);
   \draw[orange,-<] ( -1, -3) -- (  0, -3);
   \draw[ orange] (  3,  2) -- (  2,  2);
   \draw[   blue] (  3,  2) -- (  3,  5);
   \draw[ orange] (  3,  1) -- (  2,  2);
   \draw[ orange] (  3,  1) -- (  3,  0);
   \draw[ orange] (  1, -3) -- (  2, -2);
   \draw[ orange] (  1, -3) -- (  0, -3);
   \draw[ orange] ( -3, -1) -- ( -2, -2);
   \draw[ orange] ( -3, -1) -- ( -3,  0);
   \draw[ orange] ( -1,  3) -- ( -2,  2);
   \draw[ orange] ( -1,  3) -- (  0,  3);
   \draw[   blue] (  2,  6) -- (  3,  5);
   \draw[ orange] (  2,  6) -- (  1,  6);
   \draw[ orange] ( -2,  3) -- ( -2,  2);
   \draw[   blue] ( -2,  3) -- ( -5,  3);
   \draw[ orange] ( -6,  2) -- ( -6,  1);
   \draw[   blue] ( -6,  2) -- ( -5,  3);
   \draw[ orange] ( -3, -2) -- ( -2, -2);
   \draw[   blue] ( -3, -2) -- ( -3, -5);
   \draw[ orange] ( -2, -6) -- ( -1, -6);
   \draw[   blue] ( -2, -6) -- ( -3, -5);
   \draw[ orange] (  2, -3) -- (  2, -2);
   \draw[   blue] (  2, -3) -- (  5, -3);
   \draw[   blue] (  6, -2) -- (  5, -3);
   \draw[ orange] (  6, -2) -- (  6, -1);
   \draw[ orange] ( -3,  1) -- ( -6,  1);
   \draw[ orange] ( -3,  1) -- ( -3,  0);
   \draw[ orange] (  1,  3) -- (  0,  3);
   \draw[ orange] (  1,  3) -- (  1,  6);
   \draw[ orange] (  3, -1) -- (  3,  0);
   \draw[ orange] (  3, -1) -- (  6, -1);
   \draw[ orange] ( -1, -3) -- ( -1, -6);
   \draw[ orange] ( -1, -3) -- (  0, -3);
   \draw ( -3, -5) node[anchor=south west]{$\bt_0$};
   \draw ( -3, -1) node[anchor=south west]{$\bt_1$};
   \draw ( -5,  3) node[anchor=north west]{$\bt_2$};
   \draw ( -1,  3) node[anchor=north west]{$\bt_3$};
   \draw (  3,  5) node[anchor=north east]{$\bt_4$};
   \draw (  3,  1) node[anchor=north east]{$\bt_5$};
   \draw (  5, -3) node[anchor=south east]{$\bt_6$};
   \draw (  1, -3) node[anchor=south east]{$\bt_7$};
   \fill (  3,  2) circle(0.05);
   \fill ( -6,  2) circle(0.05);
   \fill (  2, -3) circle(0.05);
   \fill ( -3, -2) circle(0.05);
   \fill (  6, -2) circle(0.05);
   \fill (  2,  6) circle(0.05);
   \fill ( -2,  3) circle(0.05);
   \fill ( -2, -6) circle(0.05);
  \end{tikzpicture}
 \end{center}
 The surface $X$ can be recovered by gluing $\bt_i$ to the reverse of
 $\bt_{i+4}$ for all $i$, so we have a presentation of $X$ of the type
 used in the standard approach to the classification of surfaces,
 which gives the claimed presentation of the fundamental group.
\end{proof}

\begin{remark}\lbl{rem-groupoid}
 For some purposes it is more convenient to work with the fundamental
 groupoid $\Pi_1(X)$, or the full subgroupoid $\Gm\subset\Pi_1(X)$
 with objects $\{v_i\st 0\leq i<14\}$.  This has an action of $G$ by
 groupoid automorphisms.  Each side of $F_{16}$ gives a generator, and
 the interior of $F_{16}$ gives a relation.  One can check that the
 $G$-orbits of these generators and relations give an equivariant
 presentation for $\Gm$, fom which one can recover our earlier
 presentation of the group $\pi_1(X,v_0)=\Gm(v_0,v_0)$.  Details are
 in the file \fname+groupoid.mpl+. 
\end{remark}

Given a subgroup $H\leq G$, it is usually straightforward to find a
subset of $\Net_0$ that gives a fundamental domain for the action of
$H$, and thus to understand the topology of $X/H$.  In the case
$H\leq D_8$, this will be consistent with
Corollary~\ref{cor-quotient-types}.  We will do this explicitly for
the cases $H=\ip{\lm}$ and $H=\ip{\lm\mu}$, where $X/H$ is an elliptic
curve.  First, the inner octagon in $\Net_0$ is a fundamental domain
for $\ip{\mu}$.  We can redraw this octagon in a slightly distorted
form as follows:
\begin{center}
 \begin{tikzpicture}[scale=1]
  \draw[magenta] (-3,-3) -- (-3, 3);
  \draw[magenta] ( 3,-3) -- ( 3, 3);
  \draw[magenta] (-1, 3) -- ( 1, 3);
  \draw[magenta] (-1,-3) -- ( 1,-3);
  \draw[blue   ] (-3, 0) -- ( 3, 0);
  \draw[blue   ] ( 0,-3) -- ( 0, 3);
  \draw[green  ] (-2,-3) -- ( 2, 3);
  \draw[green  ] (-2, 3) -- ( 2,-3);
  \draw[cyan   ] (-3, 3) -- (-1, 3);
  \draw[cyan   ] ( 3, 3) -- ( 1, 3);
  \draw[cyan   ] (-3,-3) -- (-1,-3);
  \draw[cyan   ] ( 3,-3) -- ( 1,-3);
  \fill[black] (-3,-3) circle(0.03);
  \fill[black] (-2,-3) circle(0.03);
  \fill[black] (-1,-3) circle(0.03);
  \fill[black] ( 0,-3) circle(0.03);
  \fill[black] ( 1,-3) circle(0.03);
  \fill[black] ( 2,-3) circle(0.03);
  \fill[black] ( 3,-3) circle(0.03);
  \fill[black] (-3, 0) circle(0.03);
  \fill[black] ( 0, 0) circle(0.03);
  \fill[black] ( 3, 0) circle(0.03);
  \fill[black] (-3, 3) circle(0.03);
  \fill[black] (-2, 3) circle(0.03);
  \fill[black] (-1, 3) circle(0.03);
  \fill[black] ( 0, 3) circle(0.03);
  \fill[black] ( 1, 3) circle(0.03);
  \fill[black] ( 2, 3) circle(0.03);
  \fill[black] ( 3, 3) circle(0.03);
  \draw (-2.0, 1) node{$\ss\lm^2\nu$};
  \draw (-0.7, 2) node{$\ss\lm$};
  \draw ( 0.7, 2) node{$\ss\lm\nu$};
  \draw ( 2.0, 1) node{$\ss 1$};
  \draw (-2.0,-1) node{$\ss\lm^2$};
  \draw (-0.7,-2) node{$\ss\lm^3\nu$};
  \draw ( 0.7,-2) node{$\ss\lm^3$};
  \draw ( 2.0,-1) node{$\ss\nu$};
 \end{tikzpicture}
\end{center}
One can now check that $X/\ip{\mu}$ is obtained by gluing the left
edge of the square to the right edge, and the top to the bottom, which
gives a torus as expected.

For $X/\ip{\lm\mu}$, it is best to cut some corners off the inner
octagon as shown on the left below, and then rearrange the pieces as
shown on the right.
\begin{center}
 \begin{tikzpicture}[scale=0.6]
  \draw[blue   ] (-4, 0) -- ( 4, 0);
  \draw[blue   ] ( 0,-4) -- ( 0, 4);
  \draw[magenta] (-4,-2) -- (-4, 2);
  \draw[magenta] ( 4,-2) -- ( 4, 2);
  \draw[magenta] (-2,-4) -- ( 2,-4);
  \draw[magenta] (-2, 4) -- ( 2, 4);
  \draw[cyan   ] (-4, 2) -- (-2, 4);
  \draw[cyan   ] (-4,-2) -- (-2,-4);
  \draw[cyan   ] ( 4, 2) -- ( 2, 4);
  \draw[cyan   ] ( 4,-2) -- ( 2,-4);
  \draw[green  ] (-3,-3) -- ( 3, 3);
  \draw[green  ] (-3, 3) -- ( 3,-3);
  \draw[orange ] ( 0, 4) -- (-3, 3) -- (-4, 0);
  \draw[orange ] ( 4, 0) -- ( 3,-3) -- ( 0,-4);
  \fill[black  ] ( 4, 0) circle(0.05);
  \fill[black  ] ( 4, 2) circle(0.05);
  \fill[black  ] ( 3, 3) circle(0.05);
  \fill[black  ] ( 2, 4) circle(0.05);
  \fill[black  ] ( 0, 4) circle(0.05);
  \fill[black  ] (-2, 4) circle(0.05);
  \fill[black  ] (-3, 3) circle(0.05);
  \fill[black  ] (-4, 2) circle(0.05);
  \fill[black  ] (-4, 0) circle(0.05);
  \fill[black  ] (-4,-2) circle(0.05);
  \fill[black  ] (-3,-3) circle(0.05);
  \fill[black  ] (-2,-4) circle(0.05);
  \fill[black  ] ( 0,-4) circle(0.05);
  \fill[black  ] ( 2,-4) circle(0.05);
  \fill[black  ] ( 3,-3) circle(0.05);
  \fill[black  ] ( 4,-2) circle(0.05);
  \fill[black  ] ( 0, 0) circle(0.05);
  \draw ( 2, 1) node{$\ss 1$};
  \draw ( 1, 2) node{$\ss \lm\nu$};
  \draw (-1, 2) node{$\ss \lm$};
  \draw (-2, 1) node{$\ss \lm^2\nu$};
  \draw (-2,-1) node{$\ss \lm^2$};
  \draw (-1,-2) node{$\ss \lm^3\nu$};
  \draw ( 1,-2) node{$\ss \lm^3$};
  \draw ( 2,-1) node{$\ss \nu$};
 \end{tikzpicture}
 \hspace{8em}
 \begin{tikzpicture}[scale=0.8]
  \draw[orange ] (-3,-3) -- ( 3,-3) -- ( 3, 3) -- (-3, 3) -- (-3,-3);
  \draw[blue   ] ( 0,-3) -- ( 0, 3);
  \draw[blue   ] (-3, 0) -- ( 3, 0);
  \draw[green  ] (-1,-1) -- ( 1, 1);
  \draw[green  ] (-3, 3) -- ( 3,-3);
  \draw[cyan   ] (-3,-3) -- (-1,-1);
  \draw[cyan   ] ( 1, 1) -- ( 3, 3);
  \draw[magenta] ( 3, 0) -- ( 2, 2) -- ( 0, 3);
  \draw[magenta] (-3, 0) -- (-2,-2) -- ( 0,-3);
  \fill[black  ] ( 3,-3) circle(0.04);
  \fill[black  ] ( 3, 0) circle(0.04);
  \fill[black  ] ( 3, 3) circle(0.04);
  \fill[black  ] ( 0, 3) circle(0.04);
  \fill[black  ] (-3, 3) circle(0.04);
  \fill[black  ] (-3, 0) circle(0.04);
  \fill[black  ] (-3,-3) circle(0.04);
  \fill[black  ] ( 0,-3) circle(0.04);
  \fill[black  ] ( 2, 2) circle(0.04);
  \fill[black  ] (-2,-2) circle(0.04);
  \fill[black  ] ( 1, 1) circle(0.04);
  \fill[black  ] (-1,-1) circle(0.04);
  \fill[black  ] ( 0, 0) circle(0.04);
  \draw ( 1.6, 0.8) node{$\ss 1$};
  \draw ( 0.8, 1.6) node{$\ss \lm\nu$};
  \draw (-0.8, 1.6) node{$\ss \lm$};
  \draw (-1.6, 0.8) node{$\ss \lm^2\nu$};
  \draw (-1.6,-0.8) node{$\ss \lm^2$};
  \draw (-0.8,-1.6) node{$\ss \lm^3\nu$};
  \draw ( 0.8,-1.6) node{$\ss \lm^3$};
  \draw ( 1.6,-0.8) node{$\ss \nu$};
  \draw ( 2.6, 2.0) node{$\ss \lm^2\nu$};
  \draw ( 2.0, 2.6) node{$\ss \lm^3$};
  \draw (-2.6,-2.0) node{$\ss \nu$};
  \draw (-2.0,-2.6) node{$\ss \lm$};
 \end{tikzpicture}
\end{center}
One can again check that $X/\ip{\lm\mu}$ is obtained by gluing the left
edge of the square to the right edge, and the top to the bottom, which
gives a torus as expected.

\subsection{Homology}
\lbl{sec-homology}

We next consider the homology groups of a cromulent surface.  For any
compact Riemann surface of genus $2$, it is standard that
$H_0(X)\simeq H_2(X)\simeq\Z$ and $H_1(X)\simeq\Z^4$, and that all
other homology groups are zero.  Our main task is to give specific
generators for $H_2(X)$, and understand the action of $G$ in terms of
those generators.  One approach is to recall that $H_1(X)$ is just the
abelianization of $\pi_1(X,v_0)$; we see from
Proposition~\ref{prop-pi-one} that this is the free abelian group
generated by $\{\bt_0,\bt_1,\bt_2,\bt_3\}$.  However, we will use
different generators that interact with the group action in a more
convenient way.

\begin{proposition}\lbl{prop-homology}
 Let $X$ be a cromulent surface with a curve system.  Then there is an
 isomorphism $\psi\:H_1(X)\to\Z^4$, with the following effect on the
 homology classes of the curves $c_k$:
 \begin{align*}
  \psi(c_{ 0}) &= (\pp 0,\pp 0,\pp 0,\pp 0) \\
  \psi(c_{ 1}) &= (\pp 1,\pp 1,   -1,   -1) &
  \psi(c_{ 2}) &= (   -1,\pp 1,\pp 1,   -1) \\
  \psi(c_{ 3}) &= (\pp 0,\pp 1,\pp 0,   -1) &
  \psi(c_{ 4}) &= (   -1,\pp 0,\pp 1,\pp 0) \\
  \psi(c_{ 5}) &= (\pp 1,\pp 0,\pp 0,\pp 0) &
  \psi(c_{ 6}) &= (\pp 0,\pp 1,\pp 0,\pp 0) \\
  \psi(c_{ 7}) &= (\pp 0,\pp 0,\pp 1,\pp 0) &
  \psi(c_{ 8}) &= (\pp 0,\pp 0,\pp 0,\pp 1).
 \end{align*}
 This is equivariant with respect to the following action of $G$ on
 $\Z^4$:
 \begin{align*}
  \lm(n) &= (   -n_2,\pp n_1,   -n_4,\pp n_3) \\
  \mu(n) &= (\pp n_3,   -n_4,\pp n_1,   -n_2) \\
  \nu(n) &= (\pp n_1,   -n_2,\pp n_3,   -n_4).
 \end{align*}
 Moreover, the intersection product on $H_1(X)$ corresponds to the
 following bilinear form on $\Z^4$:
 \[ (n,m) = n_1m_2 - n_2m_1 - n_3m_4 + n_4m_3. \]
\end{proposition}
\begin{proof}
 We use the net $\Net_0$ for $X$ discussed in
 Section~\ref{sec-fundamental}.  This gives a CW structure on $X$ with
 a single $0$-cell, a single $2$-cell, and four $1$-cells
 corresponding to $c_3$, $c_4$, $c_7$ and $c_8$.  The attaching map
 for the $2$-cell is a product of commutators and so is homologically
 trivial.  It follows that the homology classes $[c_k]$ for
 $k\in\{3,4,7,8\}$ give a basis for $H_1(X)$.

 We next derive some relations.  Consider the following subsets of $\Net_0$:
 \begin{center}
  \begin{tikzpicture}[scale=.18]
   \begin{scope}
    \fill[gray!20] ( -4, -8) -- (  4, -8) -- (  8, -4) -- (  8,  4) --
                   (  4,  8) -- ( -4,  8) -- ( -8,  4) -- ( -8, -4) -- cycle;
    \draw (  0,-14) node {$r_0$};
    \draw[  green] (  0,  0) -- (  6,  6);
    \draw[  green] (  0,  0) -- ( -6,  6);
    \draw[  green] (  0,  0) -- ( -6, -6);
    \draw[  green] (  0,  0) -- (  6, -6);
    \draw[   blue] (  0,  0) -- (  0,  8);
    \draw[   blue] (  0,  0) -- (  0, -8);
    \draw[   blue] (  0,  0) -- (  8,  0);
    \draw[   blue] (  0,  0) -- ( -8,  0);
    \draw[   blue] ( 12, 12) -- (  4, 12);
    \draw[   blue] ( 12, 12) -- ( 12,  4);
    \draw[  green] ( 12, 12) -- (  6,  6);
    \draw[  green] (-12, 12) -- ( -6,  6);
    \draw[  green] (-12,-12) -- ( -6, -6);
    \draw[  green] ( 12,-12) -- (  6, -6);
    \draw[   cyan] (  4,  8) -- (  6,  6);
    \draw[magenta] (  4,  8) -- (  0,  8);
    \draw[magenta] (  4,  8) -- (  4, 12);
    \draw[   cyan] (  8,  4) -- (  6,  6);
    \draw[magenta] (  8,  4) -- (  8,  0);
    \draw[magenta] (  8,  4) -- ( 12,  4);
    \draw[   cyan] ( -4,  8) -- ( -6,  6);
    \draw[magenta] ( -4,  8) -- (  0,  8);
    \draw[magenta] ( -4,  8) -- ( -4, 12);
    \draw[   cyan] (  8, -4) -- (  6, -6);
    \draw[magenta] (  8, -4) -- (  8,  0);
    \draw[magenta] (  8, -4) -- ( 12, -4);
    \draw[   blue] (-12, 12) -- ( -4, 12);
    \draw[   blue] (-12, 12) -- (-12,  4);
    \draw[   blue] (-12,-12) -- ( -4,-12);
    \draw[   blue] (-12,-12) -- (-12, -4);
    \draw[   blue] ( 12,-12) -- (  4,-12);
    \draw[   blue] ( 12,-12) -- ( 12, -4);
    \draw[   cyan] (  4, -8) -- (  6, -6);
    \draw[magenta] (  4, -8) -- (  0, -8);
    \draw[magenta] (  4, -8) -- (  4,-12);
    \draw[   cyan] ( -8,  4) -- ( -6,  6);
    \draw[magenta] ( -8,  4) -- ( -8,  0);
    \draw[magenta] ( -8,  4) -- (-12,  4);
    \draw[   cyan] ( -4, -8) -- ( -6, -6);
    \draw[magenta] ( -4, -8) -- (  0, -8);
    \draw[magenta] ( -4, -8) -- ( -4,-12);
    \draw[   cyan] ( -8, -4) -- ( -6, -6);
    \draw[magenta] ( -8, -4) -- ( -8,  0);
    \draw[magenta] ( -8, -4) -- (-12, -4);
   \end{scope}
   \begin{scope}[xshift=30cm]
    \fill[gray!20] (-12,-12) -- ( -4,-12) -- ( -4, -8) -- (  4, -8) --
                   (  4,-12) -- ( 12,-12) -- ( 12, -4) -- (  8, -4) --
                   (  8,  4) -- ( 12,  4) -- ( 12, 12) -- cycle;
    \draw (  0,-14) node {$r_1$};
    \draw[  green] (  0,  0) -- (  6,  6);
    \draw[  green] (  0,  0) -- ( -6,  6);
    \draw[  green] (  0,  0) -- ( -6, -6);
    \draw[  green] (  0,  0) -- (  6, -6);
    \draw[   blue] (  0,  0) -- (  0,  8);
    \draw[   blue] (  0,  0) -- (  0, -8);
    \draw[   blue] (  0,  0) -- (  8,  0);
    \draw[   blue] (  0,  0) -- ( -8,  0);
    \draw[   blue] ( 12, 12) -- (  4, 12);
    \draw[   blue] ( 12, 12) -- ( 12,  4);
    \draw[  green] ( 12, 12) -- (  6,  6);
    \draw[  green] (-12, 12) -- ( -6,  6);
    \draw[  green] (-12,-12) -- ( -6, -6);
    \draw[  green] ( 12,-12) -- (  6, -6);
    \draw[   cyan] (  4,  8) -- (  6,  6);
    \draw[magenta] (  4,  8) -- (  0,  8);
    \draw[magenta] (  4,  8) -- (  4, 12);
    \draw[   cyan] (  8,  4) -- (  6,  6);
    \draw[magenta] (  8,  4) -- (  8,  0);
    \draw[magenta] (  8,  4) -- ( 12,  4);
    \draw[   cyan] ( -4,  8) -- ( -6,  6);
    \draw[magenta] ( -4,  8) -- (  0,  8);
    \draw[magenta] ( -4,  8) -- ( -4, 12);
    \draw[   cyan] (  8, -4) -- (  6, -6);
    \draw[magenta] (  8, -4) -- (  8,  0);
    \draw[magenta] (  8, -4) -- ( 12, -4);
    \draw[   blue] (-12, 12) -- ( -4, 12);
    \draw[   blue] (-12, 12) -- (-12,  4);
    \draw[   blue] (-12,-12) -- ( -4,-12);
    \draw[   blue] (-12,-12) -- (-12, -4);
    \draw[   blue] ( 12,-12) -- (  4,-12);
    \draw[   blue] ( 12,-12) -- ( 12, -4);
    \draw[   cyan] (  4, -8) -- (  6, -6);
    \draw[magenta] (  4, -8) -- (  0, -8);
    \draw[magenta] (  4, -8) -- (  4,-12);
    \draw[   cyan] ( -8,  4) -- ( -6,  6);
    \draw[magenta] ( -8,  4) -- ( -8,  0);
    \draw[magenta] ( -8,  4) -- (-12,  4);
    \draw[   cyan] ( -4, -8) -- ( -6, -6);
    \draw[magenta] ( -4, -8) -- (  0, -8);
    \draw[magenta] ( -4, -8) -- ( -4,-12);
    \draw[   cyan] ( -8, -4) -- ( -6, -6);
    \draw[magenta] ( -8, -4) -- ( -8,  0);
    \draw[magenta] ( -8, -4) -- (-12, -4);
   \end{scope}
   \begin{scope}[xshift=60cm]
    \fill[gray!20] ( -8,  0) -- (  8,  0) -- (  8, -4) -- ( 12, -4) --
                   ( 12,-12) -- (  4,-12) -- (  4, -8) -- ( -4, -8) --
                   ( -4,-12) -- (-12,-12) -- (-12, -4) -- ( -8, -4) -- cycle;
    \draw (  0,-14) node {$r_2$};
    \draw[  green] (  0,  0) -- (  6,  6);
    \draw[  green] (  0,  0) -- ( -6,  6);
    \draw[  green] (  0,  0) -- ( -6, -6);
    \draw[  green] (  0,  0) -- (  6, -6);
    \draw[   blue] (  0,  0) -- (  0,  8);
    \draw[   blue] (  0,  0) -- (  0, -8);
    \draw[   blue] (  0,  0) -- (  8,  0);
    \draw[   blue] (  0,  0) -- ( -8,  0);
    \draw[   blue] ( 12, 12) -- (  4, 12);
    \draw[   blue] ( 12, 12) -- ( 12,  4);
    \draw[  green] ( 12, 12) -- (  6,  6);
    \draw[  green] (-12, 12) -- ( -6,  6);
    \draw[  green] (-12,-12) -- ( -6, -6);
    \draw[  green] ( 12,-12) -- (  6, -6);
    \draw[   cyan] (  4,  8) -- (  6,  6);
    \draw[magenta] (  4,  8) -- (  0,  8);
    \draw[magenta] (  4,  8) -- (  4, 12);
    \draw[   cyan] (  8,  4) -- (  6,  6);
    \draw[magenta] (  8,  4) -- (  8,  0);
    \draw[magenta] (  8,  4) -- ( 12,  4);
    \draw[   cyan] ( -4,  8) -- ( -6,  6);
    \draw[magenta] ( -4,  8) -- (  0,  8);
    \draw[magenta] ( -4,  8) -- ( -4, 12);
    \draw[   cyan] (  8, -4) -- (  6, -6);
    \draw[magenta] (  8, -4) -- (  8,  0);
    \draw[magenta] (  8, -4) -- ( 12, -4);
    \draw[   blue] (-12, 12) -- ( -4, 12);
    \draw[   blue] (-12, 12) -- (-12,  4);
    \draw[   blue] (-12,-12) -- ( -4,-12);
    \draw[   blue] (-12,-12) -- (-12, -4);
    \draw[   blue] ( 12,-12) -- (  4,-12);
    \draw[   blue] ( 12,-12) -- ( 12, -4);
    \draw[   cyan] (  4, -8) -- (  6, -6);
    \draw[magenta] (  4, -8) -- (  0, -8);
    \draw[magenta] (  4, -8) -- (  4,-12);
    \draw[   cyan] ( -8,  4) -- ( -6,  6);
    \draw[magenta] ( -8,  4) -- ( -8,  0);
    \draw[magenta] ( -8,  4) -- (-12,  4);
    \draw[   cyan] ( -4, -8) -- ( -6, -6);
    \draw[magenta] ( -4, -8) -- (  0, -8);
    \draw[magenta] ( -4, -8) -- ( -4,-12);
    \draw[   cyan] ( -8, -4) -- ( -6, -6);
    \draw[magenta] ( -8, -4) -- ( -8,  0);
    \draw[magenta] ( -8, -4) -- (-12, -4);
   \end{scope}
  \end{tikzpicture}
 \end{center}
 The boundary of $r_0$ consists of four fragments of $c_0$ (which
 together make up the whole of $c_0$) together with a fragment of
 $c_3$ repeated twice with opposite orientations, and a fragment of
 $c_4$ repeated twice with opposite orientations.  From this we
 conclude that $[c_0]=0$ in $H_1(X)$.  Similarly, the boundary of
 $r_1$ consists of $c_1$, $c_3$ and $c_4$ together with mutually
 cancelling fragments of $c_7$ and $c_8$.  Here $c_1$ and $c_4$ run
 clockwise but $c_3$ runs anticlockwise.  We therefore have
 $[c_1]-[c_3]+[c_4]=0$.  Applying the same method to $r_2$ gives
 $[c_4]+[c_5]-[c_7]=0$.  Next, part of the definition of a curve
 system is that $\lm(c_1(t))=c_2(t)$ and $\lm(c_3(t))=c_4(t)$ and
 $\lm(c_4(t))=c_3(-t)$, which gives $\lm_*[c_1]=[c_2]$ and
 $\lm_*[c_3]=[c_4]$ and $\lm_*[c_4]=-[c_3]$.  We can therefore apply
 $\lm_*$ to the relation $[c_1]-[c_3]+[c_4]=0$ to get
 $[c_2]-[c_3]-[c_4]=0$.  Similarly, we can apply $\lm_*$ to the
 relation $[c_4]+[c_5]-[c_7]=0$ to get $-[c_3]+[c_6]-[c_8]=0$.  This
 is enough to show that $[c_5]$, $[c_6]$, $[c_7]$ and $[c_8]$ form an
 alternative basis for $H_1(X)$.  By writing everything in terms of
 this basis we get an isomorphism $\psi\:H_1(X)\to\Z^4$, and by
 inspecting the above relations we see that this is given by the
 claimed formulae.

 Next, it is part of the definition of cromulence that
 \begin{align*}
  \lm(c_{ 5}(t)) &= c_{ 6}( t)       &
  \mu(c_{ 5}(t)) &= c_{ 7}( t)       &
  \nu(c_{ 5}(t)) &= c_{ 5}( t) \\
  \lm(c_{ 6}(t)) &= c_{ 5}(-t)       &
  \mu(c_{ 6}(t)) &= c_{ 8}(-t)       &
  \nu(c_{ 6}(t)) &= c_{ 6}(-t) \\
  \lm(c_{ 7}(t)) &= c_{ 8}( t)       &
  \mu(c_{ 7}(t)) &= c_{ 5}( t)       &
  \nu(c_{ 7}(t)) &= c_{ 7}( t) \\
  \lm(c_{ 8}(t)) &= c_{ 7}(-t)       &
  \mu(c_{ 8}(t)) &= c_{ 6}(-t)       &
  \nu(c_{ 8}(t)) &= c_{ 8}(-t).
 \end{align*}
 The action in homology can be read off from this in an obvious way,
 and we find that it works as described in the statement of the
 Proposition.

 Finally, we need to analyse the intersection pairing.  From the
 definition of a curve system and associated discussion, we see that
 $C_5\cap C_6=\{v_0\}$ and $C_7\cap C_8=\{v_1\}$ and
 \[ C_5\cap C_7 = C_5\cap C_8 = C_6\cap C_7 = C_6\cap C_8 = \emptyset.
 \]
 This means that the corresponding products in homology are
 $[c_5]\cdot[c_6]=\pm 1$ and $[c_7]\cdot[c_8]=\pm 1$ and
 \[ [c_5]\cdot [c_7] = [c_5]\cdot [c_8] =
    [c_6]\cdot [c_7] = [c_6]\cdot [c_8] = 0.
 \]
 In the net, $v_0$ is the origin, $c_5$ runs to the right along the
 $x$-axis and $c_6$ runs upwards along the $y$-axis, so
 $[c_5]\cdot [c_6]=+1$.  Moreover, the map $\mu$ preserves orientation
 and has $\mu_*[c_5]=[c_7]$ and $\mu_*[c_6]=-[c_8]$; this means that
 $[c_7]\cdot[c_8]=-1$.  The claimed description of the intersection
 product follows easily.

 \begin{checks}
  homology_check.mpl: check_homology()
 \end{checks}
\end{proof}

\begin{remark}
 The vector $\psi(c_4)=(-1,0,1,0)$ (for example) is represented in
 Maple as \mcode+c_homology[4]+.  The action of $G$ on $\Z^4$ is
 represented by \mcode+act_Z4+, which is a table indexed by the
 elements of $G$, whose entries are functions.  For example, the
 expression \mcode+act_Z4[LMN]([1,2,3,4])+ evaluates to
 \mcode+[-4,3,-2,1]+, corresponding to the fact that
 $\lm\mu\nu(1,2,3,4)=(-4,3,-2,1)$.  Most other actions of $G$ on other
 sets are also represented in this way.
\end{remark}

We now relate the above description to the fundamental group.
\begin{lemma}
 The homology classes of the generators $\bt_i\in\pi_1(X,v_0)$ are as
 follows:
 \begin{align*}
  [\bt_0] &= (\pp 1,\pp 0,\pp 0,\pp 0) &
  [\bt_1] &= (   -1,   -1,\pp 1,\pp 0) \\
  [\bt_2] &= (\pp 0,\pp 1,\pp 0,\pp 0) &
  [\bt_3] &= (\pp 1,   -1,\pp 0,\pp 1) \\
  [\bt_4] &= (   -1,\pp 0,\pp 0,\pp 0) &
  [\bt_5] &= (\pp 1,\pp 1,   -1,\pp 0) \\
  [\bt_6] &= (\pp 0,   -1,\pp 0,\pp 0) &
  [\bt_7] &= (   -1,\pp 1,\pp 0,   -1)
 \end{align*}
\end{lemma}
\begin{proof}
 The claim for $\bt_0$ is immediate from the definitions, and the
 claims for $\bt_2$, $\bt_4$ and $\bt_6$ follow using the group
 action.  Now consider the left hand half of $\Net_5$:
 \begin{center}
  \begin{tikzpicture}[scale=0.6]
   \fill[gray!20]
    (  0, -3) -- (  0,  3) -- ( -1,  3) -- ( -2,  2) --
    ( -2,  3) -- ( -5,  3) -- ( -6,  2) -- ( -6,  1) --
    ( -3,  1) -- ( -3, -1) -- ( -2, -2) -- ( -3, -2) --
    ( -3, -5) -- ( -2, -6) -- ( -1, -6) -- ( -1, -3) -- cycle;
   \draw[  blue,-<] (  3,  2) -- (  3,  4);
   \draw[  blue,-<] ( -2,  3) -- ( -4,  3);
   \draw[  blue,-<] ( -3, -2) -- ( -3, -4);
   \draw[  blue,-<] (  2, -3) -- (  4, -3);
   \draw[orange,-<] (  3, -1) -- (  3,  0);
   \draw[orange,-<] (  1,  3) -- (  0,  3);
   \draw[orange,-<] ( -3,  1) -- ( -3,  0);
   \draw[orange,-<] ( -1, -3) -- (  0, -3);
   \draw[blue   ] (  0, -3) -- (  0,  3);
   \draw[blue,->] (  0, -3) -- (  0,  0);
   \draw (0,0) node[anchor=west] {$\ss c_7$};
   \draw[ orange] (  3,  2) -- (  2,  2);
   \draw[   blue] (  3,  2) -- (  3,  5);
   \draw[ orange] (  3,  1) -- (  2,  2);
   \draw[ orange] (  3,  1) -- (  3,  0);
   \draw[ orange] (  1, -3) -- (  2, -2);
   \draw[ orange] (  1, -3) -- (  0, -3);
   \draw[ orange] ( -3, -1) -- ( -2, -2);
   \draw[ orange] ( -3, -1) -- ( -3,  0);
   \draw[ orange] ( -1,  3) -- ( -2,  2);
   \draw[ orange] ( -1,  3) -- (  0,  3);
   \draw[   blue] (  2,  6) -- (  3,  5);
   \draw[ orange] (  2,  6) -- (  1,  6);
   \draw[ orange] ( -2,  3) -- ( -2,  2);
   \draw[   blue] ( -2,  3) -- ( -5,  3);
   \draw[ orange] ( -6,  2) -- ( -6,  1);
   \draw[   blue] ( -6,  2) -- ( -5,  3);
   \draw[ orange] ( -3, -2) -- ( -2, -2);
   \draw[   blue] ( -3, -2) -- ( -3, -5);
   \draw[ orange] ( -2, -6) -- ( -1, -6);
   \draw[   blue] ( -2, -6) -- ( -3, -5);
   \draw[ orange] (  2, -3) -- (  2, -2);
   \draw[   blue] (  2, -3) -- (  5, -3);
   \draw[   blue] (  6, -2) -- (  5, -3);
   \draw[ orange] (  6, -2) -- (  6, -1);
   \draw[ orange] ( -3,  1) -- ( -6,  1);
   \draw[ orange] ( -3,  1) -- ( -3,  0);
   \draw[ orange] (  1,  3) -- (  0,  3);
   \draw[ orange] (  1,  3) -- (  1,  6);
   \draw[ orange] (  3, -1) -- (  3,  0);
   \draw[ orange] (  3, -1) -- (  6, -1);
   \draw[ orange] ( -1, -3) -- ( -1, -6);
   \draw[ orange] ( -1, -3) -- (  0, -3);
   \draw ( -3, -5) node[anchor=east ]{$\bt_0$};
   \draw ( -3, -1) node[anchor=east ]{$\bt_1$};
   \draw ( -5,  3) node[anchor=south]{$\bt_2$};
   \draw ( -1,  3) node[anchor=south]{$\bt_3$};
   \draw (  1, -3) node[anchor=north]{$\bt_7$};
   \fill (  3,  2) circle(0.05);
   \fill ( -6,  2) circle(0.05);
   \fill (  2, -3) circle(0.05);
   \fill ( -3, -2) circle(0.05);
   \fill (  6, -2) circle(0.05);
   \fill (  2,  6) circle(0.05);
   \fill ( -2,  3) circle(0.05);
   \fill ( -2, -6) circle(0.05);
  \end{tikzpicture}
 \end{center}
 Define $u(t)=\bt_3(t/2)=\bt_7(1-t/2)$ for $0\leq t\leq 1$.  The
 boundary of the above region gives a relation
 \[ [\bt_0] + [\bt_1] + [\bt_2] + [u] - [c_7] - [u] = 0, \]
 so
 \[ [\bt_1] = [c_7] - [\bt_0] - [\bt_2] = (-1,-1,1,0) \]
 as claimed.  The remaining claims for $\bt_3$, $\bt_5$ and $\bt_7$
 now follow using the group action.
\end{proof}

For some purposes it is convenient to use a different basis.  We put
\begin{align*}
 u_{11} &= ( 1,0,\pp 1,0) & u_{12} &= (0, 1,0,\pp 1) \\
 u_{21} &= ( 1,0,-1,0) & u_{22} &= (0, 1,0,-1).
\end{align*}
These elements do not generate all of $\Z^4$, but only a subgroup of
index $4$.  However, they give a basis for $\Q^4$.  One can check that
\begin{align*}
 \lm(u_{11}) &= \pp u_{12} & \mu(u_{11}) &= \pp u_{11} & \nu(u_{11}) &= \pp u_{11} \\
 \lm(u_{12}) &=    -u_{11} & \mu(u_{12}) &=    -u_{12} & \nu(u_{12}) &=    -u_{12} \\
 \lm(u_{21}) &= \pp u_{22} & \mu(u_{21}) &=    -u_{21} & \nu(u_{21}) &= \pp u_{21} \\
 \lm(u_{22}) &=    -u_{21} & \mu(u_{22}) &= \pp u_{22} & \nu(u_{22}) &=    -u_{22}.
\end{align*}
This shows that $\{u_{11},u_{12}\}$ is a basis for a subrepresentation
$U_1$ with character $\chi_8$, whereas $\{u_{21},u_{22}\}$ is a basis
for a subrepresentation $U_2$ with character $\chi_9$.

We now consider the homology of certain quotient spaces $X/H$.  For
most subgroups $H\leq G$, the quotient $X/H$ is either a disc or a
sphere, and so the first homology group is trivial.  The interesting
cases are the elliptic curves $X/\ip{\mu}$ and $X/\ip{\lm\mu}$.  We
write $q_+\:X\to X/\ip{\mu}$ and $q_-\:X\to X/\ip{\lm\mu}$ for the
quotient maps.

\begin{proposition}\lbl{prop-elliptic-homology}
 There is a commutative diagram
 \[ \xymatrix{
  \Z^2 \ar[d]^\simeq_{\psi_+} &
  \Z^4 \ar@{->>}[l]_{\tht_+} \ar@{->>}[r]^{\tht_-} \ar[d]_\psi^\simeq &
  \Z^2 \ar[d]_\simeq^{\psi_-} \\
  H_1(X/\ip{\mu}) &
  H_1(X) \ar@{->>}[l]^{(q_+)_*} \ar@{->>}[r]_{(q_-)_*} &
  H_1(X/\ip{\lm\mu})
 } \]
 where
 \begin{align*}
  \tht_+(n) &= (n_1+n_3,\;n_2-n_4) \\
  \tht_-(n) &= (n_2+n_3,\;n_1+n_4).
 \end{align*}
\end{proposition}
\begin{proof}
 First, we can regard $H_1(X)$ as the abelianization of
 $\pi_1(X,v_2)$, and $H_1(X/\ip{\mu})$ as the abelianization of
 $\pi_1(X/\ip{\mu},q_+(v_2))$.  Because $q_+$ is a branched covering,
 any loop $u$ based at $q_+(v_2)$ can be lifted to give a path $\tu$
 in $X$ starting at $v_2$.  The endpoint $\tu(1)$ will lie over
 $q_+(v_2)$, but $\mu(v_2)=v_2$ so this ensures that $\tu(1)=v_2$.
 This means that $\tu$ is again a loop, and we deduce that the map
 \[ (q_+)_*\:\pi_1(X,v_2)\to \pi_1(X/\ip{\mu},q_+(v_2)) \]
 is surjective.  It follows that the map
 \[ (q_+)_* \: H_1(X) \to H_1(X/\ip{\mu}) \]
 is also surjective.  A similar argument (using the basepoint
 $v_6=\lm\mu(v_6)$) shows that $(q_-)_*$ is also surjective on $H_1$.
 Now recall that $\psi$ is equivariant for the following action:
 \begin{align*}
  \mu(n)    &= (n_3,-n_4,n_1,-n_2) \\
  \lm\mu(n) &= (n_4,n_3,n_2,n_1).
 \end{align*}
 It follows easily that $\tht_+$ is a coequaliser for the action of
 $\mu$, and $\tht_-$ is a coequaliser for the action of $\lm\mu$.
 This implies that there are unique maps $\psi_+$ and $\psi_-$ making
 the diagram commute.  As $\psi$ is an isomorphism and $(q_\pm)_*$ is
 surjective, we see that $\psi_\pm$ is also surjective.  However,
 $X/\ip{\mu}$ and $X/\ip{\lm\mu}$ are both homeomorphic to the torus,
 so in each case $H_1\simeq\Z^2$, and this implies that any surjective
 homomorphism $\Z^2\to H_1$ is automatically an isomorphism.
\end{proof}

We next discuss the Jacobian variety $JX$.  This is usually
constructed as an abelian variety using methods of algebraic geometry,
as we will recall in Section~\ref{sec-ellquot}.  However, in our
cromulent setting it is possible to construct $JX$ as a space using
only topological methods.

We first give some definitions related to coverings and subgroups of
the fundamental group.  These are standard, but we just want to pin
down some issues of naturality.
\begin{definition}
 Let $Y$ be a (locally tame) path-connected space with a basepoint
 $y_0$.  We define $\tY$ to be the space of paths in $Y$ starting at
 $y_0$, modulo the equivalence relation of homotopy relative to
 endpoints.  This has a natural basepoint $\ty_0$, which is the
 equivalence class of the constant path at $y_0$.  All this is clearly
 functorial for based maps.  Moreover, path join gives a free left
 action of $\pi_1(Y,y_0)$ on $\tY$.  Evaluation at $1$ gives a
 projection $\pi\:\tY\to Y$, which induces a homeomorphism
 $\tY/\pi_1(Y,y_0)\to Y$.  Given a subgroup $H\leq\pi_1(Y,y_0)$, we
 have a based covering map $\tY/H\to Y$, whose effect in $\pi_1$ is
 the inclusion $H\to\pi_1(Y,y_0)$.  Given a based homeomorphism
 $f\:Y_0\to Y_1$ with $f_*(H_0)=H_1$, we get an induced based
 homeomorphism $\tY_0/H_0\to\tY_1/H_1$.
\end{definition}

\begin{definition}\lbl{defn-JX}
 We put $J'X=(X/\ip{\mu})\tm(X/\ip{\lm\mu})$, and define
 $q'\:X\to J'X$ by $q'(x)=(q_+(x),q_-(x))$.  The resulting map
 \[ \tht' = (\Z^4 \xra{\psi} H_1(X)
                  \xra{q'_+} H_1(X/\ip{\mu})\oplus H_1(X/\ip{\lm\mu})
                  \xra{(\psi_+,\psi_-)^{-1}} \Z^4)
 \]
 is then given by
 \[ \tht'(n) = (\tht_+(n),\tht_-(n))
     = (n_1+n_3,n_2-n_4,n_1+n_4,n_2+n_3).
 \]
 One can check that the image of $\tht'$ is the subgroup
 \[ \Tht' = \{m\in\Z^4\st\sum_im_i = 0\pmod{2}\}, \]
 which has index two.  Note also that $J'X\simeq(S^1)^4$, so the natural
 map $\pi_1(J'X,q'(v_0))\to H_1(J'X)$ is an isomorphism.  Thus, $\Tht'$
 corresponds to a subgroup of index two in $\pi_1(J'X,q'(v_0))$, and
 so gives rise to a double cover of $J'X$, which we call $JX$.
 Because $\pi_1(JX)=\Tht'$, we see that $q'$ lifts to give a map
 $q\:X\to JX$.  By construction, this induces an isomorphism
 $H_1(X)\to H_1(JX)$.
\end{definition}

\begin{remark}
 If $Y$ and $Z$ are homotopy equivalent to $(S^1)^n$ and $(S^1)^m$,
 then standard methods of homotopy theory show that the natural map
 \[ H_1 \: [Y,Z] \to \Hom(H_1(Y),H_1(Z)) \]
 is bijective.  (It makes no difference here whether we work in a
 based context or an unbased context.)  Using this, we can produce a
 map $JX\tm JX\to JX$ which makes $JX$ into an abelian group up to
 homotopy, and we can produce a map $G\to[JX,JX]$ giving an action of
 $G$ on $JX$ up to homotopy.  Using analytic methods, we can improve
 this: $JX$ becomes an abelian variety, with an action of $G$ by
 automorphisms of abelian varieties.  However, we cannot
 build a topological group structure or a $G$-action by homeomorphisms
 without using analysis.  The element $\lm^2\in G$ normalises the
 subgroups $\ip{\mu}$ and $\ip{\lm\mu}$, and preserves the basepoint
 $v_0$, so this induces an involution on $J'X$ and on $JX$.  However,
 no other nontrivial element of $G$ shares these properties.
\end{remark}

\section{The projective family}
\lbl{sec-P}

In this section we construct a family of precromulent surfaces $PX(a)$
(for $a\in (0,1)$) as branched covers of the Riemann sphere
$\C_\infty$.  Later we will consider an arbitrary precromulent surface
$X$, and attempt to find an isomorphism $X\simeq PX(a)$ for some $a$.
The notion of a cromulent labelling will emerge naturally from this
analysis.

\subsection{Definitions}
\lbl{sec-P-defs}

\begin{definition}\lbl{defn-P}
 For any $a\in(0,1)$ we put $A=a^2+1/a^2\in(2,\infty)$ and define
 $r_a\:\C\to\C$ by
 \[ r_a(z) = z(z-a)(z+a)(z-1/a)(z+1/a)
           = z^5 - A z^3 + z.
 \]
 Next, we put
 \begin{align*}
  R_0(a) &= \C[z,w]/(w^2-r_a(z)) \\
  PX_0(a) &= \spec(R_0(a)) = \{(w,z)\in \C^2 \st w^2=r_a(z) \},
 \end{align*}
 so $PX_0(a)$ is a smooth affine hyperelliptic curve.  Unfortunately,
 if we just take the closure in $\C P^2$, the resulting curve is
 singular at infinity.  To get a nonsingular completion, we put
 \[ PX(a) = \{[z]\in\C P^4\st
              z_1^2-z_2z_3-z_4z_5+Az_3z_4=0,\;
              z_2z_4=z_3^2,\;z_2z_5=z_3z_4,\;z_3z_5=z_4^2
            \}.
 \]
 There is a map $j\:PX_0(a)\to PX(a)$ given by
 \[ j(w,z) = [w:1:z:z^2:z^3]. \]
 We define points $v_0,\dotsc,v_{13}\in PX(a)$ by
 \begin{align*}
  v_0 &= [0:1:0:0:0] = j(0,0) \\
  v_1 &= [0:0:0:0:1] \\
  v_2 &= j( -(a^{-1}-a),-1) &
  v_6 &= j\left(\pp \frac{1+i}{\rt}(a^{-1}+a),\pp i\right) &
  v_{10} &= j(0,-a) \\
  v_3 &= j(-i(a^{-1}-a),\pp 1) &
  v_7 &= j\left(-\frac{1-i}{\rt}(a^{-1}+a),-i\right) &
  v_{11} &= j(0,\pp a) \\
  v_4 &= j(\pp (a^{-1}-a),-1) &
  v_8 &= j\left(-\frac{1+i}{\rt}(a^{-1}+a),\pp i\right) &
  v_{12} &= j(0,-a^{-1}) \\
  v_5 &= j(\pp i(a^{-1}-a),\pp 1) &
  v_9 &= j\left(\pp \frac{1-i}{\rt}(a^{-1}+a),-i\right) &
  v_{13} &= j(0,\pp a^{-1}).
 \end{align*}
 We let $G$ act on $\C P^4$ and $PX(a)$ by
 \begin{align*}
  \lm[z] &= [iz_1:z_2:-z_3:z_4:-z_5] \\
  \mu[z] &= [z_1:z_5:z_4:z_3:z_2] \\
  \nu[z] &= [\ov{z}_1:\ov{z}_2:\ov{z}_3:\ov{z}_4:\ov{z}_5].
 \end{align*}
\end{definition}
\begin{checks}
 cromulent.mpl: check_precromulent("P")
 projective/PX_check.mpl: check_P_action()
 projective/PX_check.mpl: check_j_P()
\end{checks}
We will prove as Proposition~\ref{prop-P-precromulent} that this
gives a precromulent surface; then we will see in
Proposition~\ref{prop-P-fundamental} that it is actually cromulent.
\begin{remark}
 The parameters $a$ and $A$ are \mcode+a_P+ and \mcode+A_P+ in Maple.
 The polynomial $r_a(z)$ is \mcode+r_P(z)+.  Points in $\C^2$ are
 represented as lists of length two.  (Maple distinguishes between
 lists and vectors, and we have generally preferred to use lists, for
 technical reasons that we will not explore here.)  One can check
 whether a list lies in $PX_0(a)$ using the function
 \mcode+is_member_P_0(z)+.  Points in $\C P^4$ are represented by lists
 of length $5$, and one can check whether two lists are projectively
 equivalent using the function \mcode+is_equal_P(w,z)+.  One can also
 check whether a point lies in $PX(a)$ using the function
 \mcode+is_member_P(z)+.  The map $j\:PX_0(a)\to PX(a)$ is
 \mcode+j_P(z)+.  The points $v_i\in PX(a)$ are \mcode+v_P[i]+.  The action
 of $g\in G$ on $PX(a)$ is given by \mcode+act_P[g](z)+, and the
 corresponding action on $PX_0(a)$ is \mcode+act_P_0[g](z)+.  All this
 comes from the file \fname+projective/PX.mpl+.
\end{remark}
\begin{remark}\lbl{rem-a-Po}
 All the above functions treat \mcode+a_P+ as a symbol.  There is also
 a global variable \mcode+a_P0+ which holds a numerical value for
 \mcode+a_P+.  In fact, there are two such variables, called
 \mcode+a_P0+ and \mcode+a_P1+.  This is intended to cover the case
 where \mcode+a_P0+ is an exact expression (such as a rational number)
 and \mcode+a_P1+ is a floating point approximation to the same
 number.  However, we have usually taken \mcode+a_P0+ to be a floating
 point number, so that there is no distinction between \mcode+a_P0+ and
 \mcode+a_P1+.  These variables should be set using the function
 \mcode+set_a_P0+, defined in the file \fname+projective/PX0.mpl+.
 This will then set a large number of other variables by substituting
 \mcode+a_P0+ or \mcode+a_P1+ for \mcode+a_P+.  For example,
 \mcode+v_P0[i]+ and \mcode+v_P1[i]+ are obtained by applying these
 substitutions to \mcode+v_P[i]+.  When the file
 \fname+projective/PX0.mpl+ is loaded, it calls the function
 \mcode+set_a_P0+ to set \mcode+a_P0+ to a particular value
 (approximately $0.0984$), which is close to the value for which
 $EX^*$ is isomorphic to $PX(a)$.  However, one can call
 \mcode+set_a_P0+ again to change the value if  desired.
\end{remark}
\begin{remark}\lbl{rem-simplify-P}
 When working with an expression $m$ involving \mcode+a_P+, it is
 sometimes convenient to use the function \mcode+simplify_P(m)+
 (defined in \fname+projective/PX.mpl+).  This will try some
 substitutions like $\sqrt{1-a^2}=\sqrt{1-a}\sqrt{1+a}$ that are not
 always used by the default simplification functions.
\end{remark}

\begin{remark}\lbl{rem-r-coprime}
 As $a\in(0,1)$ we see that $r_a(z)$ has no repeated roots, so
 $r_a(z)$ and $r'_a(z)$ are coprime in $\R[z]$ or $\C[z]$.  More
 specifically, one can check by direct expansion that
 $m_0(z)r_a(z)+m_1(z)r'_a(z)=1$, where
 \begin{align*}
  m_0(z) &= \frac{(100-30A^2)z^3+(18A^3-70A)z}{4(A^2-4)} \\
  m_1(z) &= \frac{(6A^2-20)z^4-(6A^3-22A)z^2+(4A^2-16)}{4(A^2-4)}.
 \end{align*}
 (These are \mcode+r_P_cofactor0(z)+ and \mcode+r_P_cofactor1(z)+ in the
 Maple code.)
 \begin{checks}
  projective/PX_check.mpl: check_r_P_cofactors()
 \end{checks}
\end{remark}

\begin{lemma}\lbl{lem-R-domain}
 The ring $R_0(a)$ is an integral domain, and the set
 \[ B = \{z^iw^j\st i\in\N,\;j\in\{0,1\}\} \]
 is a basis for $R_0(a)$ over $\C$.
\end{lemma}
\begin{proof}
 From the description $R_0(a)=\C[z,w]/(w^2-r_a(z))$ it is clear that
 $\{1,w\}$ is a basis for $R_0(a)$ as a module over $\C[z]$.  It
 follows that $B$ is a basis for $R_0(a)$ over $\C$.  Next, for any
 element $f=p(z)+q(z)w\in R_0(a)$, we put
 \[ N(f) = (p(z)+q(z)w)(p(z)-q(z)w)=p(z)^2-q(z)^2r_a(z) \in\C[z]. \]
 It is easy to check that $N(fg)=N(f)N(g)$.  Moreover, by considering
 the highest power of $z$ that divides the various terms, we see that
 $N(f)\neq 0$ whenever $f\neq 0$.  Thus, if $g$ is also nonzero we
 have $N(fg)=N(f)N(g)$, which is nonzero because $\C[z]$ is a domain,
 so $fg\neq 0$ as required.
\end{proof}

\begin{lemma}\lbl{lem-P-differentials}
 The module $\Om^1(PX_0(a))$ of K\"ahler differentials is freely
 generated over $R_0(a)$ by the element
 \[ \om_0 = m_0(z)w\,dz + 2m_1(z)\,dw \]
 (where $m_0$ and $m_1$ are as in Remark~\ref{rem-r-coprime}).  In
 particular, we have
 \begin{align*}
  dz &= w\,\om_0 \\
  dw &= \half r'_a(z)\om_0.
 \end{align*}
\end{lemma}
\begin{proof}
 We will put $\Om^1=\Om^1(PX_0(a))$ for brevity.  Differentiating the
 equation $w^2=r_a(z)$ gives $2w\,dw=r'_a(z)\,dz$, and (essentially by
 definition) the module $\Om^1$ is generated by $dw$ and $dz$ subject
 only to this relation.  Now
 \begin{align*}
  w\,\om_0 &= m_0(z)w^2\,dz + 2m_1(z)w\,dw \\
           &= m_0(z)r_a(z)\,dz + m_1(z)r'_a(z)\,dz \\
           &= (m_0(z)r_a(z)+m_1(z)r'_a(z))\,dz = dz.
 \end{align*}
 A similar argument shows that $\half r'_a(z)\om_0=dw$, so both $dz$
 and $dw$ lie in the submodule of $\Om^1$ generated by $\om_0$.  This
 implies that $\om_0$ generates all of $\Om^1$.  Now note that our
 original presentation of $\Om^1$ implies that $\Om^1[w^{-1}]$ is
 freely generated over $R_0(a)[w^{-1}]$ by $dw$.  Thus, if
 $f\in R_0(a)$ satisfies $f\om_0=0$ then $f\Om^1=0$ so
 $f\Om^1[w^{-1}]=0$ so $f$ must map to zero in $R_0(a)[w^{-1}]$.
 However, as $R_0(a)$ is an integral domain we see that the map
 $R_0(a)\to R_0(a)[w^{-1}]$ is injective, so $f$ must be zero.  It
 follows that $\Om^1$ is \emph{freely} generated by $\om_0$, as
 claimed.
\end{proof}

\begin{lemma}\lbl{lem-nonzero}
 Consider a point $[z]\in PX(a)$.
 \begin{itemize}
  \item[(a)] If any of the coordinates $z_2,\dotsc,z_5$ is zero, then
   $[z]\in\{v_0,v_1\}$.
  \item[(b)] Either $z_2\neq 0$ or $z_5\neq 0$.
 \end{itemize}
\end{lemma}
\begin{proof}
 We put
 \begin{align*}
  r_0(z) &= z_1^2-z_2z_3-z_4z_5+Az_3z_4 \\
  r_1(z) &= z_2 z_4 - z_3^2 \\
  r_2(z) &= z_2 z_5 - z_3 z_4 \\
  r_3(z) &= z_3 z_5 - z_4^2,
 \end{align*}
 so that $PX(a)=\{[z]\st r_0(z)=\dotsb=r_3(z)=0\}$.
 \begin{itemize}
  \item[(a)] Suppose that $[z]\in PX(a)$ and $z_2z_3z_4z_5=0$.  Using
   $r_2$ we see that $(z_2z_5)^2$ and $(z_3z_4)^2$ are both equal to
   $z_2z_3z_4z_5$ and thus to zero, so $z_2z_5=z_3z_4=0$.  Thus, either
   $z_2=0$ or $z_5=0$; and either $z_3=0$ or $z_4=0$.  If $z_3=0$ then
   $r_3$ gives $z_4=0$; conversely, if $z_4=0$ then $r_1$ gives
   $z_3=0$.  We must therefore have $z_3=z_4=0$.  Substituting this in
   $r_0$ gives $z_1=0$.  In summary, at most one of the coordinates
   $z_i$ can be nonzero, and the nonzero coordinate must be $z_2$ or
   $z_5$.  It follows that $[z]\in\{v_0,v_1\}$ as required.
  \item[(b)] The claim is clear if $z_2z_3z_4z_5\neq 0$, and follows
   from~(a) if $z_2z_3z_4z_5=0$.
 \end{itemize}
\end{proof}

\begin{proposition}\lbl{prop-P-j}
 The map $j$ gives an isomorphism $PX_0(a)\simeq PX(a)\sm\{v_1\}$.
\end{proposition}
\begin{proof}
 The lemma shows that for $[z]\in PX(a)\sm\{v_1\}$ we have
 $z_2\neq 0$.  We can thus define $k\:PX(a)\sm\{v_1\}\to\C^2$ by
 $k[z]=(z_1/z_2,z_3/z_2)$.  It will be harmless to rescale $z$ so
 that $z_2=1$, and then the last three defining relations for $PX(a)$
 become $z_4=z_3^2$ and $z_5=z_4z_4$ and $z_3z_5=z_4^2$, so
 $(z_2,z_3,z_4,z_5)=(1,z_3,z_3^2,z_3^3)$.  If we use these to rewrite
 the first relation we get
 \[ z_1^2-z_3-z_3^5+(a^2+a^{-2})z_3^3 = 0, \]
 so $k[z]\in PX_0(a)$.  The equations $jk=1$ and $kj=1$ are now
 clear.
 \begin{checks}
  projective/PX_check.mpl: check_j_P()
 \end{checks}
\end{proof}

\begin{definition}\lbl{defn-P-cover}
 We put
 \begin{align*}
  U_0 &= PX(a)\sm\{v_1\} \\
  U_1 &= PX(a)\sm\{v_0\}=\mu(U_0)\simeq U_0 \\
  U_{01} &= U_0\cap U_1.
 \end{align*}
\end{definition}

\begin{remark}\lbl{rem-P-cover}
 We use $j$ to silently identify $U_0$ with $PX_0(a)$, which is the
 spectrum of the ring $R_0(a)=\C[z,w]/(w^2-r_a(z))$.  This in turn
 identifies $U_{01}$ with the spectrum of the ring
 \[ R'(a) = R_0(a)[z^{-1}] = \C[z^{\pm 1},w]/(w^2-r_a(z)). \]
 The set
 \[ B_0=\{z^i\st i\geq 0\} \amalg\{z^iw\st i\geq 0\} \]
 is a basis for the subring $R_0(a)\subset R'(a)$ over $\C$.  The
 set $R_1(a)=\mu(R_0(a))$ is also a subring of $R'(a)$, with basis
 \[ B_1=\{z^i\st i\leq 0\} \amalg\{z^iw\st i\leq -3\}. \]
\end{remark}

\begin{proposition}\lbl{prop-P-precromulent}
 The above definitions make $PX(a)$ into a precromulent surface.
\end{proposition}
\begin{proof}
 First, the standard Jacobian condition shows that $PX_0(a)$ is
 smooth, so Proposition~\ref{prop-P-j} shows that $PX(a)$ is smooth
 except possibly at $v_1$.  Next, straightforward calculation shows
 that the action of $G$ on $\C P^4$ preserves $PX(a)$, with $D_8$
 acting conformally and $G\sm D_8$ acting anticonformally, and
 \begin{align*}
  \lm(j(w,z)) &= j(iw,-z) \\
  \mu(j(w,z)) &= j(-w/z^3,1/z) \\
  \nu(j(w,z)) &= j(\ov{w},\ov{z}).
 \end{align*}
 In particular, we see that $\mu$ gives an isomorphism between
 $PX(a)\sm\{v_1\}$ and $PX(a)\sm\{v_0\}$, showing that $PX(a)$ is
 smooth everywhere.  It is clearly closed in $\C P^4$ and therefore
 compact.  It is standard that for any polynomial $r(z)$ of degree
 $2g+1$, the hyperelliptic curve $w^2=r(z)$ has genus $g$; in
 particular, $PX(a)$ has genus $2$.  A straightforward but lengthy
 check shows that $G$ permutes the points $v_i$ in accordance with
 Definition~\ref{defn-precromulent}(b).  The tangent space to
 $PX_0(a)$ at $(0,0)$ is $\C\oplus 0$, and $\lm$ acts on this is
 multiplication by $i$.  All that is left is to check that $D_8$ acts
 freely on $PX(a)\sm V$, where $V=\{v_0,\dotsc,v_{13}\}$.  The fixed
 points of $\lm^2$ on $PX_0(a)$ are pairs $(w,z)$ with
 $(-w,z)=(w,z)$ and $w^2=z(z^2-a^2)(z^2-a^{-2})$, which means that
 $w=0$ and $z\in\{0,a,-a,a^{-1},-a^{-1}\}$.  It follows that
 \[ PX(a)^{\langle\lm^2\rangle} =
     \{v_0,v_1,v_{10},v_{11},v_{12},v_{13}\} \sse V.
 \]
 All fixed points of $\lm$ or $\lm^3$ are also fixed by $\lm^2$ and so
 lie in $V$.  A similar analysis shows that the fixed points of $\mu$,
 $\lm\mu$, $\lm^2\mu$ and $\lm^3\mu$ also lie in $V$, as required.
\end{proof}

\begin{remark}\lbl{rem-P-quotient}
 We can define $p\:PX(a)\to\C_\infty$ by $p([z])=z_3/z_2$, so
 $pj(w,z)=z$.  Using $\lm^2j(w,z)=(-w,z)$, it is not hard to check
 that $p$ gives an isomorphism $PX(a)/\ip{\lm^2}\to\C_\infty$.
 Moreover, $p$ is equivariant if we use the following action on
 $\C_\infty$:
 \[ \lm(z) = -z \hspace{4em}
    \mu(z) = 1/z \hspace{4em}
    \nu(z) = \ov{z}.
 \]
 This action is represented by \mcode+act_C+ in Maple.  For example,
 \mcode+act_C[L](3)+ evaluates to \mcode+-3+.
\end{remark}

\subsection{The curve system}
\lbl{sec-P-curves}

\begin{definition}\lbl{defn-P-curves}
 We define maps $j'\:\C^3\sm\{0\}\to\C P^4$ by
 \[ j'(w,x,y) = [w:x^3:x^2y:xy^2:y^3]. \]
 Note that $j(w,z)=j'(w,1,z)$, and when $x\neq 0$ we have
 $j'(w,x,y)=j(w/x^3,y/x)$.  We then define $c_k\:\R\to PX(a)$ as follows:
 \begin{align*}
  c_0(t) &=
   j'(-\sqrt{(a^{-1}-a)^2+4\sin^2(2t)},\;e^{it},\;-e^{-it}) \\
  c_1(t) &=
   j'\left(\frac{1+i}{8\rt}\sin(t)
           \sqrt{16\cos(t)^2+(a+a^{-1})^2\sin(t)^4},\;
           \frac{1+\cos(t)}{2},\;\frac{1-\cos(t)}{2}i\right) \\
  c_2(t) &= \lm(c_1(t)) \\
  c_3(t) &=
   j'\left(-i\frac{a^{-1}-a}{8}
               \sin(t)\sqrt{(1+a)^4-(1-a)^4\cos(t)^2}
                      \sqrt{(1+a)^2-(1-a)^2\cos(t)^2},\right. \\
         & \left.\qquad\qquad\vphantom{\frac{a^{-1}-a}{8}}
           \frac{(1+a)+(1-a)\cos(t)}{2},\;
           \frac{(1+a)-(1-a)\cos(t)}{2}\right)  \\
  c_4(t) &= \lm(c_3(t)) \\
  c_5(t) &= j\left(\frac{\sin(t)}{8}\sqrt{2a(3-\cos(t))(4-a^4(1-\cos(t))^2)},\;
                    a \frac{1 - \cos(t)}{2}\right) \\
  c_6(t) &= \lm(c_5(t)) \\
  c_7(t) &= \mu(c_5(t)) \\
  c_8(t) &= \lm\mu(c_5(t)).
 \end{align*}
\end{definition}
Maple notation for $c_k(t)$ is \mcode+c_P[k](t)+.  The versions with a
numerical values for $a$ are \mcode+c_P0[k](t)+ and
\mcode+c_P1[k](t)+.  The map $j'$ is \mcode+jj_P+.

\begin{proposition}\lbl{prop-P-curves}
 The above maps give a curve system on $PX(a)$ (in the sense of
 Definition~\ref{defn-curve-system}).
\end{proposition}
\begin{proof}
 Combine Lemmas~\ref{lem-P-curves-bc} and~\ref{lem-P-curves-a}.
 \begin{checks}
  cromulent.mpl: check_precromulent("P")
 \end{checks}
\end{proof}

\begin{lemma}\lbl{lem-P-curves-bc}
 For $0\leq k\leq 8$, the map $c_k\:\R\to\C P^4$ is smooth, with
 $c_k(t+2\pi)=c_k(t)$, and the image is contained in $PX(a)$.
 Moreover, parts~(b) and~(c) of Definition~\ref{defn-curve-system} are
 satisfied.
\end{lemma}
\begin{proof}
 Direct calculation.
\end{proof}

\begin{lemma}\lbl{lem-P-pc}
 The composites $\R\xra{c_k}PX(a)\xra{p}\C_\infty$ and their images
 are as follows:
 \begin{align*}
  pc_0(t) &= - e^{-2it} &
  pc_0(\R) &= S^1 \\
  pc_1(t) &= \pp\frac{1-\cos(t)}{1+\cos(t)}i &
  pc_1(\R) &= [0,\infty] i \\
  pc_2(t) &= -\frac{1-\cos(t)}{1+\cos(t)}i &
  pc_2(\R) &= [-\infty,0] i \\
  pc_3(t) &= \pp\frac{(1+a)-(1-a)\cos(t)}{(1+a)+(1-a)\cos(t)} &
  pc_3(\R) &= [a,a^{-1}] \\
  pc_4(t) &= -\frac{(1+a)-(1-a)\cos(t)}{(1+a)+(1-a)\cos(t)} &
  pc_4(\R) &= [-a^{-1},-a] \\
  pc_5(t) &= \pp\frac{1-\cos(t)}{2}a &
  pc_5(\R) &= [0,a] \\
  pc_6(t) &= -\frac{1-\cos(t)}{2}a &
  pc_6(\R) &= [-a,0] \\
  pc_7(t) &= \pp\frac{2}{1-\cos(t)}a^{-1} &
  pc_7(\R) &= [a^{-1},\infty] \\
  pc_8(t) &= -\frac{2}{1-\cos(t)}a^{-1} &
  pc_8(\R) &= [-\infty,-a^{-1}] \\
 \end{align*}
\end{lemma}
\begin{proof}
 Direct calculation.
 \begin{checks}
  projective/PX_check.mpl: check_pc_P()
 \end{checks}
\end{proof}

\begin{lemma}\lbl{lem-jj-equal}
 Suppose that $j'(w_0,x,y)=j'(w_1,x,y)$ and that this point lies in
 $PX(a)$.  Then $w_0=w_1$.
\end{lemma}
\begin{proof}
 By assumption we have
 \[ [w_0:x^3:x^2y:xy^2:y^3]=[w_1:x^3:x^2y:xy^2:y^3]. \]
 It follows easily that $w_0=w_1$ unless $x=y=0$.  However, as this
 point lies in $PX(a)$, Lemma~\ref{lem-nonzero} tells us that $x$ and
 $y$ cannot both vanish.
\end{proof}

\begin{lemma}\lbl{lem-P-curves-a}
 For $0\leq k\leq 8$, the induced map $c_k\:\R/2\pi\Z\to PX(a)$ is a
 smooth embedding.
\end{lemma}
\begin{proof}
 Because of the group action, it will suffice to treat the cases
 $k=0,1,3,5$.

 Using Lemma~\ref{lem-P-pc} we see that $(pc_k)'(t)\neq 0$ except
 when $k>0$ and $t\in\pi\Z$.  It follows that $c_k'(t)\neq 0$ except
 possibly when $k>0$ and $t\in\pi\Z$.  Moreover, one can check that to
 first order in $\ep$ we have
 \begin{align*}
  c_1(\ep)     &\simeq \left[\frac{1+i}{2\rt}\ep:1:0:0:0\right] \\
  c_1(\ep+\pi) &\simeq \left[\frac{1-i}{2\rt}\ep:0:0:0:1\right].
 \end{align*}
 It follows easily that $c'_1(t)\neq 0$ for all $t$, so $c_1$ is at
 least an immersion.  Similar calculations show that $c_0,\dotsc,c_8$
 are all immersions.

 We now just need to show that $c_k\:\R/2\pi\Z\to PX(a)$ is
 injective.
 \begin{itemize}
  \item[(a)] Consider the case $k=0$.  If $u=c_0(t)$ we have
   $p(u)=-e^{-2it}$, and it follows that the quantity
   $a^{-2}+a^2-p(u)^2-p(u)^{-2}$ is equal to $a^{-2}+a^2-2\cos(4t)$
   and so is strictly positive.  It follows in turn that
   \[ -q(u)\,p(u)^{-2}\,(a^{-2}+a^2-p(u)^2-p(u)^{-2})^{-1/2} = e^{it}.
   \]
   From this it is clear that when $c_0(s)=c_0(t)$ we have
   $e^{is}=e^{it}$ and so $s-t\in 2\pi\Z$ as required.
  \item[(b)] Now suppose instead that $k\in\{1,3,5\}$ and
   $c_k(s)=c_k(t)$.  We then have $pc_k(s)=pc_k(t)$, and using
   Lemma~\ref{lem-P-pc} we can deduce that $\cos(s)=\cos(t)$.  Now, we
   can use the identities $\sin(2t)=\sin(t)\cos(t)$ and
   $\sin^2(t)=1-\cos^2(t)$ to rewrite $c_k(t)$ in the form 
   \[ c_k(t)=j'(u(\cos(t))\sin(t),\;v(\cos(t)),\;w(\cos(t))), \]
   for some functions $u$, $v$ and $w$.  As
   $\cos(s)=\cos(t)$ and $c_k(s)=c_k(t)$ we can use
   Lemma~\ref{lem-jj-equal} to see that
   \[ u(\cos(t))\sin(s) = u(\cos(t))\sin(t). \]
   Moreover, in each case one can check that $u(\cos(t))$ is never
   zero, so $\sin(s)=\sin(t)$.  We thus have $s-t\in 2\pi\Z$ again.
 \end{itemize}
\end{proof}

\begin{proposition}\lbl{prop-P-std-isotropy}
 $PX(a)$ has standard isotropy (as in
 Definition~\ref{defn-std-isotropy}).
\end{proposition}
\begin{proof}
 First, we put $D_k=p(C_k)\sse\C_\infty$; these sets are described by
 Lemma~\ref{lem-P-pc}.  Note that we have
 $\lm^2(c_0(t))=c_0(t+\pi)$, and $\lm^2(c_k(t))=c_k(-t)$ for $1\leq
 k\leq 8$.  It follows that for all $k$ we have $\lm^2(C_k)=C_k$.  In
 view of Remark~\ref{rem-P-quotient}, it follows that
 $C_k=p^{-1}(D_k)$.

 It is clear that the sets $D_4=[-a^{-1},-a]$, $D_5=[0,a]$ and
 $D_7=[a^{-1},\infty]$ are disjoint, and it follows that $C_4$, $C_5$
 and $C_7$ are disjoint.  Similarly, $C_3$, $C_6$ and $C_8$ are
 disjoint.

 Next, recall that the map $p\:PX(a)\to\C_\infty$ is equivariant with
 respect to the action described in Remark~\ref{rem-P-quotient}
 (given by $\lm(z)=-z$ and $\mu(z)=z^{-1}$ and $\nu(z)=\ov{z}$).  In
 particular, if $x\in PX(a)$ is fixed by an element $\al\in G$, then
 $p(x)$ is also fixed by $\al$.

 \begin{itemize}
  \item[(a)] Consider a point $x\in PX(a)$ with $\mu\nu(x)=x$.  Then
   $p(x)$ is also fixed by $\mu\nu$, which means that
   $p(x)=1/\ov{p(x)}$, so $p(x)\in S^1=D_0$, so $x\in
   p^{-1}(D_0)=C_0$.  We thus have $PX(a)^{\ip{\mu\nu}}=C_0$ as claimed.
  \item[(b)] Consider a point $x\in PX(a)$ with $\lm\nu(x)=x$.  If
   $x=v_0=c_1(0)$ or $x=v_1=c_1(\pi)$ then it is clear that
   $x\in C_1$.  Suppose instead that $x\not\in\{v_0,v_1\}$, so can be
   written as $j(w,z)$ with $z\neq 0$.  In general we have
   $\lm\nu(j(w,z))=j(i\ov{w},-\ov{z})$.  As $\lm\nu(x)=x$ we see that
   $w=i\ov{w}$ and $z=-\ov{z}$.  Put
   $\om=e^{i\pi/4}=\frac{1+i}{\rt}$, so $\om^2=i$ and
   $\om=i\ov{\om}$.  We find that $w=\om w_1$ and $z=iz_1$ for some
   $w_1,z_1\in\R$.  The equation $w^2=r_a(z)$ becomes
   $w_1^2=z_1(z_1^2+a^2)(z_1^2+a^{-2})$.  From this it is clear that
   $z_1>0$.  By elementary calculus, there is a unique $t\in(0,\pi)$
   with $z_1=(1-\cos(t))/(1+\cos(t))$.  For this $t$ we find that
   $p(c_1(t))=z$, and thus that $x$ is either $c_1(t)$ or
   $\lm^2(c_1(t))=c_1(-t)$.   Either way we have $x\in c_1(\R)=C_1$,
   so $PX(a)^{\ip{\lm\nu}}=C_1$ as claimed.
  \item[(c)] As $\lm^3\nu$ is conjugate to $\lm\nu$, we can use the
   group action to deduce that $PX(a)^{\ip{\lm^3\nu}}=C_2$.
  \item[(d)] Now consider a point $x\in PX(a)$ with $\nu(x)=x$.  If
   $x=v_1$ then $x\in C_7$.  Otherwise, we have $x=j(w,z)$ for some
   $(w,z)\in PX_0(a)$.  As $\nu(x)=x$ we see that $w$ and $z$ are
   real.  As $r_a(z)=w^2$ we have $r_a(z)\geq 0$.  Recall
   that the roots of $r_a(z)$, listed in increasing order, are
   $-a^{-1},-a,0,a$ and $a^{-1}$.  It follows that
   \[ p(x) = z \in [-a^{-1},-a] \amalg [0,a] \amalg [a^{-1},\infty] \\
           = D_4\amalg D_5 \amalg D_7,
   \]
   and thus that $x\in C_4\amalg C_5\amalg C_7$.
  \item[(e)] Consider instead a point $x\in PX(a)$ with
   $\lm^2\nu(x)=x$.  Then the point $y=\lm(x)$ satisfies $\nu(x)=x$
   and so lies in $C_4\amalg C_5\amalg C_7$.  However, one can check
   from Definition~\ref{defn-precromulent-C} that
   \begin{align*}
    \lm(C_3) &= C_4 & \lm(C_4) &= C_3 \\
    \lm(C_5) &= C_6 & \lm(C_6) &= C_5 \\
    \lm(C_7) &= C_8 & \lm(C_8) &= C_7,
   \end{align*}
   so $x\in C_3\amalg C_6\amalg C_8$.
  \item[(f)] Finally, consider a point $x\in PX(a)$ with
   $x=\lm^2\mu\nu(x)$.  Then $x$ cannot be equal to $v_1$, so
   $x=j(w,z)$ for some $w$ and $z$.  The equation $x=\lm^2\mu\nu(x)$
   gives $z=-1/\ov{z}$ and so $|z|^2=-1$, which is impossible.  Thus,
   there are no such points $x$.
 \end{itemize}
\end{proof}

This is a convenient place to record the following result, which will
be needed later.
\begin{lemma}\lbl{lem-right-angle}
 \begin{itemize}
  \item[(a)] The curves $C_0$ and $C_3$ cross at right angles at $v_3$
  \item[(b)] The curves $C_0$ and $C_1$ cross at right angles at $v_6$
  \item[(c)] The curves $C_3$ and $C_5$ cross at right angles at $v_{11}$.
 \end{itemize}
\end{lemma}
If we were willing to wait until we had proved
Lemma~\ref{lem-bt-al-fixed}, we could give a non-computational
argument based on that.  However, we will just calculate the relevant
derivatives instead.
\begin{proof}
 For~(a), recall that $v_3=c_0(\pi/2)=c_3(\pi/2)$, so we need to
 compare $c'_0(\pi/2)$ with $c'_2(\pi/2)$.  Because the map
 $p\:PX(a)\to\C_\infty$ is conformal, it will suffice to show that
 $(pc_0)'(\pi/2)$ and $(pc_3)'(\pi/2)$ are nonzero and that the ratio
 between them is purely imaginary.  This is easy to do using the
 formulae in Lemma~\ref{lem-P-pc}.  Specifically, we have
 $(pc_0)'(\pi/2)=-2i$ and $(pc_3)'(\pi/2)=2(1-a)/(1+a)$.  We can
 prove~(b) in the same way using $(pc_0)'(\pi/4)=2$ and
 $(pc_1)'(\pi/2)=2i$.  We need a slightly different method for
 $v_{11}$ because $p$ has derivative zero there.  We instead define a
 rational map $q\:PX(a)\to\C_\infty$ by $q(z)=z_1/z_2$, so
 $qj(w,z)=w$.  This is conformal and satisfies $q(v_{11})=0$, so it
 will suffice to show that $(qc_3)'(0)$ and $(qc_5)'(\pi)$ are
 nonzero, and that the ratio between them is purely imaginary.  A
 standard calculation from the definitions gives
 \begin{align*}
  (qc_3)'(0) &= -i(1-a^2)(1+a^2)^{1/2}/\rt \\
  (qc_5)'(\pi) &= -(1-a^2)^{1/2}(1+a^2)^{1/2}a^{1/2}/\rt
 \end{align*}
 as required.
\end{proof}

\subsection{Fundamental domains}
\lbl{sec-P-fundamental}

\begin{proposition}\lbl{prop-P-fundamental}
 If we put
 \begin{align*}
  PF'_{16}(a) &= \{(w,z)\in PX_0(a)\st
           \text{Re}(z),\text{Im}(z),\text{Re}(w)\geq 0,\;
           \text{Re}(w)\geq\text{Im}(w),\;|z|\leq 1\} \\
  PF_{16}(a) &= j(PF'_{16}(a))\subset PX(a),
 \end{align*}
 then $PF_{16}(a)$ is a standard fundamental domain for $PX(a)$ (as in
 Definition~\ref{defn-standard-F}).  Thus, $PX(a)$ is cromulent (by
 Remark~\ref{rem-standard-F}).
\end{proposition}
\begin{proof}
 For brevity, we will write $F'$ and $F$ for $PF'_{16}(a)$ and
 $PF_{16}(a)$.  Put
 \begin{align*}
  Z   &= \{x+iy\in\C\st x,y\geq 0,\;x^2+y^2\leq 1\} &
      &= \{r\,e^{i\tht}\st 0\leq r\leq 1,\;0\leq\tht\leq\pi/2\} \\
  W   &= \{x+iy\in\C\st x\geq 0,\; x\geq y\} &
      &= \{r\,e^{i\tht}\st 0\leq r,\; -\pi/2\leq\tht\leq\pi/4\} \\
  W^2 &= \{x+iy\in\C\st x\geq 0 \text{ or } y\leq 0\} &
      &= \{r\,e^{i\tht}\st 0\leq r,\; -\pi\leq\tht\leq\pi/2\}.
 \end{align*}
 We then have $F'=(W\tm Z)\cap PX_0(a)$.

 We now claim that $r_a(Z)\sse W^2$.  Indeed, it is clear that
 $\partial Z$ is a simple closed curve.  The image $r_a(\partial Z)$
 consists of the points $r_a(t)\in\R$ (for $0\leq t\leq 1$) and
 $r_a(it)\in i\R$ (for $0\leq t\leq 1$) and $r_a(e^{it})$ (for
 $0\leq t\leq\pi/2$).  Here
 \[ r_a(e^{it}) = - (4\sin(t)^2+(a^{-1}-a)^2) e^{3it}, \]
 so $\text{arg}(r_a(e^{it}))=3t-\pi\in[-\pi,\pi/2]$, so
 $r_a(e^{it})\in W^2$.  We now see that $r_a(\partial Z)$ is a simple
 closed curve in $W^2$.  The argument principle shows that $r_a(Z)$ is
 the interior of $r_a(\partial Z)$, and this is contained in $W^2$ as
 claimed.

 Now consider a point $v\in PX(a)$.  If $p(v)=x+iy\in\C$ then we put
 \begin{align*}
  s_0(v) &= \frac{|x| + i|y|}{\max(1,x^2+y^2)} \in Z \\
  s(v) &= j(\sqrt{r_a(s_0(v))},s_0(v)) \in F'.
 \end{align*}
 (Here $r_a(s_0(v))$ lies in $W^2$, and $\sqrt{r_a(s_0(v))}$ refers to
 the unique choice of square root that lies in $W$.)  For the
 exceptional case $v=v_1$, we put $s(v_1)=(0,0)$.  It is easy to see
 that $s$ is a retraction.  Using Remark~\ref{rem-P-quotient}, we
 see that $s_0(v)=s_0(v')$ iff $Gv=Gv'$, and also that
 $s_0(v)\in G.p(v)$.  After recalling that
 $\lm^2j(w,z)=j(-w,z)$, we deduce that $s(v)=s(v')$ iff $Gv=Gv'$, and
 also that $s(v)\in G.v$.  It follows that
 $PX(a)=\bigcup_{\gm\in G}\gm.F'$, with
 \[ F'\cap\gm F' = \{v\in F'\st \gm(v)=v\}. \]
 If $v$ lies in the interior of $F'$ then it is easy to see that
 $p(v)$ lies in the interior of $Z$, and thus that
 $\stab_G(v)\sse\stab_G(p(v))=\{1,\lm^2\}$.  On the other hand, for
 $v$ in the interior of $F'$ we also have $r_a(p(v))\neq 0$, so $v$ is
 not fixed by $\lm^2$, so $\stab_G(v)=1$.  We now see that
 $\text{int}(F')\cap\gm(F')=\emptyset$ for $\gm\neq 1$, so $F'$ is a
 retractive fundamental domain for $PX(a)$.

 Next, the formulae in Lemma~\ref{lem-P-pc} show that
 \begin{align*}
  \partial Z &= [0,a] \cup [a,1] \cup e^{[0,\pi/2]i} \cup [i,0] \\
   &= pc_5([0,\pi])\cup pc_3([0,\pi/2])\cup
      pc_0([\pi/4,\pi/2]) \cup pc_1([0,\pi/2]).
 \end{align*}
 From this we deduce that
 \[ \partial F' =
      c_5([0,\pi])\cup c_3([0,\pi/2])\cup
      c_0([\pi/4,\pi/2]) \cup c_1([0,\pi/2]) = DF_{16},
 \]
 so $F'$ is a standard fundamental domain.
\end{proof}

We can illustrate the surface $PX(a)$ as follows.  The picture on the
left shows the image under $p\:PX(a)\to\C_\infty$ of the fundamental
domain $F$, and the picture on the right shows the image under $q$.
(In both cases the origin is at $v_0$.)
\[
 \begin{tikzpicture}[scale=4]
  \draw[blue] (0,0) -- (0.600,0);
  \draw[magenta] (0.600,0) -- (1,0);
  \draw[cyan] (0,0) (1,0) arc(0:90:1);
  \draw[green] (0,1) -- (0,0);
  \fill[black] (0.600,0) circle(0.015);
  \fill[black] (0,0) circle(0.015);
  \fill[black] (1,0) circle(0.015);
  \fill[black] (0,1) circle(0.015);
  \draw (0.000,0.000) node[anchor=north] {$\ss v_0$};
  \draw (0.600,0.000) node[anchor=north] {$\ss v_{11}$};
  \draw (0.000,1.000) node[anchor=south] {$\ss v_6$};
  \draw (1.000,0.000) node[anchor=north] {$\ss v_3$};
 \end{tikzpicture}
 \hspace{10em}
 \begin{tikzpicture}[scale=2]
  \draw[magenta] (0,0) -- (0,-1.067);
  \draw[cyan] plot[smooth]
   coordinates{ (0.000,-1.067) (0.259,-1.081) (0.560,-1.098)
                (0.910,-1.065) (1.284,-0.933) (1.637,-0.678)
                (1.914,-0.303) (2.070,0.163) (2.074,0.674)
                (1.914,1.173) (1.603,1.603) };
  \draw[green] (1.603,1.603) -- (0,0);
  \draw[blue] (0,0) -- (0.470,0);
  \fill[black] (0.000,0.000) circle(0.03);
  \fill[black] (0.000,-1.067) circle(0.03);
  \fill[black] (1.603,1.603) circle(0.03);
  \draw (0.000,0.000) node[anchor=east] {$\ss v_0,v_{11}$};
  \draw (0.000,-1.067) node[anchor=east] {$\ss v_3$};
  \draw (1.603,1.603) node[anchor=south] {$\ss v_6$};
 \end{tikzpicture}
\]

We next consider differential forms on $PX_0(a)$ and $PX(a)$.
\begin{remark}\lbl{rem-anticonformal-forms}
 Holomorphic differential forms are clearly functorial for conformal
 isomorphisms.  In fact, they are also functorial for anticonformal
 isomorphisms.  Indeed, given an anticonformal map $\phi\:Z_0\to Z_1$
 of Riemann surfaces and a holomorphic function $f\in\CO(Z_1)$, we can
 define $\phi^\#(f)\in\CO(Z_0)$ by $\phi^\#(f)(z)=\ov{f(\phi(z))}$.
 It is not hard to see that there is a unique locally determined map
 $\phi^\#\:\Om^1(Z_1)\to\Om^1(Z_0)$ satisfying
 $\phi^\#(f\,dg)=\phi^\#(f)\,d\phi^\#(g)$ for all $f,g\in\CO(Z_1)$.
 We therefore have an action of $G$ on $\Om^1(PX(a))$.
\end{remark}

\begin{proposition}\lbl{prop-holomorphic-forms}\leavevmode
 \begin{itemize}
  \item[(a)] The differential form $\om_0\in\Om^1(U_0)$ (from
   Lemma~\ref{lem-P-differentials}) extends to give a holomorphic
   differential form on all of $PX(a)$ (which we also call $\om_0$).
  \item[(b)] The form $\om_1=\mu^*(\om_0)$ satisfies $\om_1=z\,\om_0$
   when restricted to $U_0$.
  \item[(c)] The set $\{\om_0,\om_1\}$ is a basis for $\Om^1(PX(a))$
   over $\C$.
  \item[(d)] The group $G$ acts on this space by
   \begin{align*}
    \lm^*(\om_0)  &= \pp i\om_0 &
    \mu^*(\om_0)  &= \om_1 &
    \nu^\#(\om_0) &= \om_0 \\
    \lm^*(\om_1)  &=    -i\om_1 &
    \mu^*(\om_1)  &= \om_0 &
    \nu^\#(\om_1) &= \om_1.
   \end{align*}
 \end{itemize}
\end{proposition}
\begin{proof}
 As $z\om_0\in\Om^1(U_0)$ and $\mu(U_1)=U_0$ we have a holomorphic
 form $\om'_0=\mu^*(z\,\om_0)\in\Om^1(U_1)$.  Recall that
 $w\om_0=dz$.  After restricting to $U_{01}$ we can apply $\mu^*$ to
 this equation, giving
 \[ -wz^{-3}\mu^*(\om_0) = d(z^{-1})
                         = -z^{-2}\,dz
                         = -z^{-2}w\om_0,
 \]
 which implies that $\mu^*(\om_0)=z\om_0$, and thus that
 $\mu^*(z\om_0)=\om_0$.  This implies that $\om_0$ and $\om'_0$ have
 the same restriction to $U_{01}$, so we can patch them together to
 give a holomorphic form on all of $U_0\cup U_1=PX(a)$.  Claims~(a)
 and~(b) are now clear, and~(d) is a straightforward calculation.
 This just leaves~(c).  Consider a holomorphic form
 $\al\in\Om^1(PX(a))$.  Lemma~\ref{lem-P-differentials} tells us
 that there is a unique function $f_0\in R_0(a)$ such that
 $\al=f_0\om_0$ on $U_0$.  By applying the same logic to $\mu^*(\al)$,
 and applying $\mu^*$ again, we see that there is also a unique
 function $f_1\in R_1(a)$ such that $\al=f_1\om_1$ on $U_1$.  On
 $U_{01}$ we now see that $f_0\om_0=\al=f_1\om_1=f_1z\om_0$, so
 $f_0=f_1z$.  Using the bases described in Remark~\ref{rem-P-cover}
 we see that $f_0\in R_0(a)\cap R_1(a)z=\C\{1,z\}$, and it follows
 that $\al\in\C\{\om_0,\om_1\}$.  Moreover, it is easy to see that
 $\om_1$ vanishes at $v_0$ but $\om_0$ does not, and the other way
 around at $v_1$.  This means that $\om_0$ and $\om_1$ are linearly
 independent, so they form a basis for $\Om^1(PX(a))$.
\end{proof}

\begin{remark}\lbl{rem-P-parameter}
 The coordinate $w$ is a local parameter on $PX_0(a)$ at the point
 $v_0=(0,0)$.  In terms of this parameter we have $\om_0=2\,dw+O(w^4)$
 and $\om_1=2w^2\,dw+O(w^6)$.
\end{remark}

\begin{definition}\lbl{defn-period-p}
 The \emph{periods} for $PX(a)$ are the numbers
 $p_{jk}(a)=\int_{c_j}\om_k\in\C$ (for $0\leq j\leq 8$ and
 $k\in\{0,1\}$).
\end{definition}

\subsection{Galois theory}
\lbl{sec-galois}

Let $PK(a)$ denote the field of rational functions on $PX_0(a)$ (which
is the same as the field of rational functions on $PX(a)$).  This can
be described as
\[ PK(a) = \C(z)[w]/(w^2-r_a(z)). \]
This field has an action of the group $D_8$, and for any subgroup 
$H\leq D_8$, we can identify the fixed field $PK(a)^H$ with the field
of rational functions on the quotient $PX(a)/H$, with its standard
structure as a Riemann surface.  The subgroups of $D_8$ can be
enumerated as follows:
\begin{center}
 \begin{tikzpicture}[scale=3.1]
  \def\ya{0.7}
  \def\yb{1.4}
  \def\yc{2.1}
  \def\xa{0.8}
  \def\xb{1.6}
  \def\Da{( 0.0, 0.0)}
  \def\Ba{(-\xa, \ya)}
  \def\Bb{( \xa, \ya)}
  \def\Ca{(   0, \ya)}
  \def\Aa{(-\xb, \yb)}
  \def\Ab{( \xa, \yb)}
  \def\Ac{(-\xa, \yb)}
  \def\Ad{( \xb, \yb)}
  \def\Za{(   0, \yb)}
  \def\Ta{(   0, \yc)}
  \begin{scope}[xshift=1cm]
   \draw \Da node{$H_{10}=D_8$};
   \draw \Ca node{$H_3=\ip{\lm}$};
   \draw \Ba node{$H_8=\ip{\lm^2,\mu}$};
   \draw \Bb node{$H_9=\ip{\lm^2,\lm\mu}$};
   \draw \Aa node{$H_6=\ip{\lm^2\mu}$};
   \draw \Ab node{$H_5=\ip{\lm\mu}$};
   \draw \Ac node{$H_4=\ip{\mu}$};
   \draw \Ad node{$H_7=\ip{\lm^3\mu}$};
   \draw \Za node{$H_2=\ip{\lm^2}$};
   \draw \Ta node{$H_1=1$};
   \draw[<-,shorten <=15pt,shorten >=15pt] \Da -- \Ba;
   \draw[<-,shorten <=15pt,shorten >=15pt] \Da -- \Bb;
   \draw[<-,shorten <=15pt,shorten >=15pt] \Da -- \Ca;
   \draw[<-,shorten <=15pt,shorten >=15pt] \Ca -- \Za;
   \draw[<-,shorten <=15pt,shorten >=15pt] \Ba -- \Aa;
   \draw[<-,shorten <=15pt,shorten >=15pt] \Ba -- \Ac;
   \draw[<-,shorten <=15pt,shorten >=15pt] \Ba -- \Za;
   \draw[<-,shorten <=15pt,shorten >=15pt] \Bb -- \Ab;
   \draw[<-,shorten <=15pt,shorten >=15pt] \Bb -- \Ad;
   \draw[<-,shorten <=15pt,shorten >=15pt] \Bb -- \Za;
   \draw[<-,shorten <=19pt,shorten >=15pt] \Aa -- \Ta;
   \draw[<-,shorten <=15pt,shorten >=15pt] \Ab -- \Ta;
   \draw[<-,shorten <=15pt,shorten >=15pt] \Ac -- \Ta;
   \draw[<-,shorten <=19pt,shorten >=15pt] \Ad -- \Ta;
   \draw[<-,shorten <=15pt,shorten >=15pt] \Za -- \Ta;
  \end{scope}
 \end{tikzpicture}
\end{center}

We will describe the fixed fields $L_i=PK(a)^{H_i}$ in terms of the
following elements:
\begin{align*}
 t_0 &= z   & u_0 &= w \\
 t_1 &= z^2 &
 u_1 &= \frac{2w(1-z)}{(1+z^2)^2} &
 v_1 &= w(1-1/z^3)/2 \\
 t_2 &= \frac{2z}{1+z^2} & 
 u_2 &= -\frac{1+i}{\rt}\,\frac{2w(i+z)}{(1-z^2)^2} &
 v_2 &= w(1-i/z^3)/2 \\
 t_3 &= \frac{2iz}{1-z^2} & 
 u_3 &= \frac{-2iw(1+z)}{(1+z^2)^2} &
 v_3 &= w(1+1/z^3)/2 \\
 t_4 &= \frac{2z^2}{1+z^4} & 
 u_4 &= -\frac{1-i}{\rt}\,\frac{2w(i-z)}{(1-z^2)^2} &
 v_4 &= w(1+i/z^3)/2.
\end{align*}

One can check that 
\begin{align*}
 L_1    &= \C(t_0)\{1,u_0\} \\
 L_2    &= \C(t_0) \\
 L_3    &= \C(t_1) \\
 L_4    &= \C(t_2)\{1,u_1\} = \C(t_2)\{1,v_1\} \\
 L_5    &= \C(t_3)\{1,u_2\} = \C(t_3)\{1,v_2\} \\
 L_6    &= \C(t_2)\{1,u_3\} = \C(t_2)\{1,v_3\} \\
 L_7    &= \C(t_3)\{1,u_4\} = \C(t_3)\{1,v_4\} \\
 L_8    &= \C(t_2) \\
 L_9    &= \C(t_3) \\
 L_{10} &= \C(t_4).
\end{align*}

Details for $L_4$ and $L_5$ will be given in
Section~\ref{sec-ellquot}.  The cases $L_6$ and $L_7$ can be recovered
from this, because $H_6$ and $H_7$ are conjugate to $H_4$ and $H_5$
respectively.  The other cases are relatively easy (and are easily
seen to be consistent with Corollary~\ref{cor-quotient-types}, which
gives the genera of the quotients $PX(a)/H_i$).  One can find further
information in the files \fname+parabolic/galois.mpl+ and
\fname+parabolic/PK_subfields.mpl+.  
\begin{checks}
 parabolic/galois_check.mpl: check_PK(), check_PK_subfields()
\end{checks}

\subsection{Elliptic quotients}
\lbl{sec-ellquot}

We next study the quotients $PX(a)/\ip{\mu}$ and $PX(a)/\ip{\lm\mu}$.

First note that there is no natural action of the full group $G$ on
$PX(a)/\ip{\mu}$.  Instead, there is an action of the centraliser of
$\mu$, which is $\ip{\lm^2,\mu,\nu}\simeq C_2^3$.  This action factors
through the quotient group $\ip{\lm^2,\mu,\nu}/\ip{\mu}\simeq C_2^2$.
Similarly, we have a natural action of the group
$\ip{\lm^2,\lm\mu,\mu\nu}/\ip{\lm\mu}\simeq C_2^2$ on $PX(a)/\ip{\lm\mu}$.

\begin{definition}\lbl{defn-ellquot}
 We put $b_{\pm}=(a^{-1}\pm a)/2$, and define affine curves
 $E_0^{\pm}(a)$ as follows.
 \begin{align*}
  q^+_a(x) &= 2x(x-1)\left(b_+^2x^2-1\right) \\
  q^-_a(x) &= 2x(x-1)\left(b_-^2x^2+1\right) \\
  E^+_0(a) &= \{(y,x)\in\C^2\st y^2=q_+(x)\} \\
  E^-_0(a) &= \{(y,x)\in\C^2\st y^2=q_-(x)\}.
 \end{align*}
 We can obtain smooth projective completions of these curves by taking
 the closures of their images under the map $j\:\C^2\to\C P^3$ given
 by
 \[ j(y,x) = [y:1:x:x^2]. \]
 The results are
 \begin{align*}
  E^+(a) &=
   \{[z]\st z_2z_4-z_3^2 = z_1^2 - 2(z_4-z_3)(b_+^2z_4-z_2)=0\} \\
  E^-(a) &=
   \{[z]\st z_2z_4-z_3^2 = z_1^2 - 2(z_4-z_3)(b_-^2z_4+z_2)=0\}.
 \end{align*}
 We define an action of the group $\ip{\lm^2,\mu,\nu}$ on $E^+(a)$ as
 follows:
 \begin{align*}
  \lm^2[z] &= [   -z_1:z_2:z_3:z_4] & \lm^2j(y,x) &= j(   -y,x) \\
    \mu[z] &= [\pp z_1:z_2:z_3:z_4] &  \mu j(y,x) &= j(\pp y,x) \\
    \nu[z] &= [\pp \ov{z_1}:\ov{z_2}:\ov{z_3}:\ov{z_4}] &
    \nu j(y,x) &= j(\pp\ov{y},\ov{x}).
 \end{align*}
 Similarly, we define an action of the group
 $\ip{\lm^2,\lm\mu,\mu\nu}$ on $E^-(a)$ as follows:
 \begin{align*}
  \lm^2[z]  &= [   -z_1:z_2:z_3:z_4] & \lm^2j(y,x) &= j(   -y,x) \\
  \lm\mu[z] &= [\pp z_1:z_2:z_3:z_4] & \lm\mu j(y,x) &= j(\pp y,x) \\
  \mu\nu[z] &= [\pp \ov{z_1}:\ov{z_2}:\ov{z_3}:\ov{z_4}] &
  \mu\nu j(y,x) &= j(\pp\ov{y},\ov{x}).
 \end{align*}
\end{definition}
\begin{checks}
 projective/ellquot_check.mpl: check_ellquot()
\end{checks}
\begin{remark}
 Code for all this is in \fname+ellquot.mpl+.  The polynomials $q^+_a$
 and $q^-_a$ are \mcode+q_Ep+ and \mcode+q_Em+.  Elements of
 $E^+_0(a)$ and $E^-_0(a)$ are represented as lists of length two,
 whereas elements of $E^+(a)$ and $E^-(a)$ are represented as lists of
 length four.  The functions \mcode+is_equal_Ep+ and
 \mcode+is_equal_Em+ (which are actually the same) can be used to test
 projective equality.  The function \mcode+is_member_Ep_0+ can be used
 to test whether a point lies in $E^+_0(a)$, and similarly for
 \mcode+is_member_Em_0+,  \mcode+is_member_Ep+ and
 \mcode+is_member_Em+.  The inclusion $E^+_0(a)\to E^+(a)$ and its
 inverse are \mcode+j_Ep+ and \mcode+j_inv_Ep+, and similarly for
 \mcode+j_Em+ and \mcode+j_inv_Em+.  The function \mcode+NF_Ep+ can be
 used to reduce a polynomial in $z_1,\dotsc,z_5$ to normal form modulo
 the Gr\"obner basis for the ideal that defines $E^+(a)$.  There is a
 similar function \mcode+NF_Em+ for $E^-(a)$.  Actions of $G$ are
 given by \mcode+act_Ep_0+, \mcode+act_Em_0+, \mcode+act_Ep+ and
 \mcode+act_Em+.
\end{remark}

\begin{proposition}\lbl{prop-P-Ep}
 There is a unique morphism $\phi^+\:PX(a)\to E^+(a)$ satisfying
 \[ \phi^+(j(w,z)) = j\left(\frac{2w(1-z)}{(1+z^2)^2},
                            \frac{2z}{1+z^2}\right)
 \]
 for all $(w,z)\in PX_0(a)$ with $z\neq\pm i$.  Moreover, this is
 equivariant with respect to $\ip{\lm^2,\mu,\nu}$, and it induces an
 isomorphism $PX(a)/\ip{\mu}\to E^+(a)$.
\end{proposition}
Two variants of this map are represented by \mcode+P_to_Ep+ and
\mcode+P_to_Ep_0+.
\begin{proof}
 First define $\psi\:\C^5\to\C^4$ by
 \[ \psi(z) = (2(z_2-z_3)z_1,\;z_2^2+2z_2z_4+z_3z_5,\;
               2z_2(z_3+z_5),\;4z_2z_4).
 \]
 This is homogeneous of degree two, so it induces a map
 $\ov{\psi}\:U\to\C P^3$, where $U=\{[z]\in\C P^5\st\psi(z)\neq 0\}$.
 Now put
 \[ V = \{j(w,z)\in j(PX_0(a))\st z\not\in\{0,i,-i\}\}, \]
 and note that this is open and dense in $PX(a)$ and is preserved by
 $G$.  It is straightforward to check that $V\sse j(PX_0(a))\sse U$,
 so we can define $\phi^+_0$ to be the restriction of $\ov{\psi}$ to
 $PX_0(a)$.  It follows easily from the definitions that for
 $j(w,z)\in V$ we have
 \[ \phi^+_0j(w,z) = j\left(\frac{2w(1-z)}{(1+z^2)^2},
                           \frac{2z}{1+z^2}\right),
 \]
 and that this lies in $E^+(a)$.  From this we also see that the
 restriction of $\phi^+_0$ to $V$ is equivariant, and in particular
 that $\phi^+_0=\phi^+_0\mu$ on $V$.  By continuity, we must have
 $\phi^+_0=\phi^+_0\mu$ on all of $PX_0(a)$.  We can thus patch
 together $\phi^+_0$ and $\phi^+_0\mu$ to get a morphism from
 $PX(a)=j(PX_0(a))\cup\mu j(PX_0(a))$ to $E^+(a)$.  The equivariance
 conditions are satisfied on the open dense subset $V$, so they are
 satisfied everywhere.

 Now consider a point $(y,x)\in E^+_0(a)$ with $x\not\in\{0,1\}$.  Let
 $u$ be a square root of $1-x^2$, and put
 \[ v_{\pm} = \left(
     \frac{\pm(2-x)u-(x+2)(x-1)}{x^3(x-1)}y,\;
     \frac{1\pm u}{x}
    \right).
 \]
 Straightforward algebra shows that $v_+,v_-\in PX_0(a)$ with
 $\mu(v_+)=v_-$, and that $(\phi^+)^{-1}\{(y,x)\}=\{v_+,v_-\}$.  It
 follows that the induced map $PX(a)/\ip{\mu}\to E^+(a)$ is
 generically bijective.  As the source and target are both smooth and
 complete algebraic curves, it follows that the map is an isomorphism,
 as claimed.
 \begin{checks}
  projective/ellquot_check.mpl: check_ellquot()
 \end{checks}
\end{proof}

\begin{proposition}\lbl{prop-P-Em}
 There is a unique map $\phi^-\:PX(a)\to E^-(a)$ satisfying
 \[ \phi^-(j(w,z)) = \left(-\frac{\rt(1+i)w(i+z)}{(1-z^2)^2},
                           \frac{2iz}{1-z^2}\right) \in E^0_+(a)
 \]
 for all $(w,z)\in PX_0(a)$ with $z\neq\pm 1$.  Moreover, this is
 equivariant with respect to $\ip{\lm^2,\lm\mu,\mu\nu}$, and it
 induces an isomorphism $PX(a)/\ip{\lm\mu}\to E^-(a)$.
\end{proposition}
Two variants of this map are represented by \mcode+P_to_Em+ and
\mcode+P_to_Em_0+.
\begin{proof}
 Similar to the previous proposition, using the formulae
 \[ \psi(z) = (\rt (1-i)z_1(z_2-iz_3),\;
               z_2^2-2z_2z_4+z_3z_5,\;
               2iz_2(z_3-z_5),\;
               -4z_2z_4)
 \]
 and
 \[ v_{\pm} = \left(
     \frac{1+i}{\rt}\;\frac{\pm(x-2)u-(x+2)(x-1)}{x^3(x-1)}y,\;
     \frac{\pm u-1}{x}i
    \right).
 \]
 \begin{checks}
  projective/ellquot_check.mpl: check_ellquot()
 \end{checks}
\end{proof}

\begin{definition}
 We will write $v_i^+=\phi^+(v_i)\in E^+(a)$, and similarly for
 $v_i^-$.
\end{definition}

\begin{remark}\lbl{rem-Ep-infinity}
 One can check that
 \begin{align*}
  v^+_6 = v^+_9 &= [\pp a+a^{-1}:0:0:-\rt] \\
  v^+_7 = v^+_8 &= [   -a-a^{-1}:0:0:-\rt],
 \end{align*}
 and these are the only points in $E^+(a)\sm j(E^+_0(a))$.  Similarly,
 we have
 \begin{align*}
  v^-_2 = v^-_3 &= [\pp a-a^{-1}:0:0:-\rt] \\
  v^-_4 = v^-_5 &= [   -a+a^{-1}:0:0:-\rt],
 \end{align*}
 and these are the only points in $E^-(a)\sm j(E^-_0(a))$.
\end{remark}

\begin{remark}\lbl{rem-elliptic-group}
 One can also check that
 \begin{align*}
  v^+_0 = v^+_1 &= j(0,0) = [0:1:0:0] \in E^+(a) \\
  v^-_0 = v^-_1 &= j(0,0) = [0:1:0:0] \in E^-(a).
 \end{align*}
 We use these points as the basepoints in $E^+(a)$ and $E^-(a)$.  As
 these are elliptic curves, each of them has a unique group structure
 for which the specified basepoint is the zero element.  We also find
 that
 \begin{align*}
  j^{-1}\phi^+c_5(t) &= (\sqrt{a}t,0) + O(t^2) \\
  j^{-1}\phi^+c_1(t) &= (e^{i\pi/4}t,0) + O(t^2) \\
  j^{-1}\phi^-c_5(t) &= (e^{-i\pi/4}\sqrt{a}t,0) + O(t^2) \\
  j^{-1}\phi^-c_1(t) &= (t,0) + O(t^2).
 \end{align*}
 Thus, if we use the map $z\mapsto j^{-1}(z)_1$ as a local coordinate
 at the basepoint $v^\pm_0$, then $E^+(a)$ looks like our
 standard picture $\Net_0$, but $E^-(a)$ is rotated clockwise by
 $\pi/4$.
 \begin{checks}
  projective/ellquot_check.mpl: check_ellquot_origin()
 \end{checks}
\end{remark}

\begin{definition}\lbl{defn-T-matrices}
 We define matrices $T_i^\pm$ as follows:
 {\tiny\begin{align*}
  T_1^+ &= \bbm
   b_+^2-1 & 0 & 0 & 0 \\
   0 & 1 & -2b_+^2 & b_+^4 \\
   0 & 1 & -b_+^2-1 & b_+^2 \\
   0 & 1 & -2 & 1
  \ebm &
  T_1^- &= \bbm
   b_-^2+1 & 0 & 0 & 0 \\
   0 & -1 & -2b_-^2 & -b_-^4 \\
   0 & -1 & -b_-^2+1 & b_-^2 \\
   0 & -1 & 2 & -1
  \ebm \\
  T_2^+ &= \bbm
   2(1-b_+) & 0 & 0 & 0 \\
   0 & b_+ & 2b_+(b_+-2) & b_+(b_+-2)^2 \\
   0 & 1 & -2 & -b_+(b_+-2) \\
   0 & b_+^{-1} & -2 & b_+
  \ebm &
  T_2^- &= \bbm
   2(1-b_-) & 0 & 0 & 0 \\
   0 & b_- & 2ib_-(2i-b_-) & -b_-(2i-b_-)^2 \\
   0 & i & -2i & -ib_-(2i-b_-) \\
   0 & -b_-^{-1} & -2i & b_-
  \ebm \\
  T_3^+ &= \bbm
   -2(1+b_+) & 0 & 0 & 0 \\
   0 & b_+ & -2b_+(b_++2) & b_+(b_++2)^2 \\
   0 & -1 & 2 & b_+(b_++2) \\
   0 & b_+^{-1} & 2 & b_+
  \ebm &
  T_3^- &= \bbm
   2(i+b_-) & 0 & 0 & 0 \\
   0 & -b_- & -2ib_-(b_-+2i) & b_-(b_-+2i)^2 \\
   0 & i & -2i & ib_-(b_-+2i) \\
   0 & b_-^{-1} & -2i & -b_-
  \ebm
 \end{align*}}
 One can check that in $PGL_4(\C)$ we have
 \[ (T_1^+)^2 = (T_2^+)^2 = (T_3^+)^2 = T_1^+T_2^+T_3^+ =
    (T_1^-)^2 = (T_2^-)^2 = (T_3^-)^2 = T_1^-T_2^-T_3^- = 1.
 \]
 One can also check that these matrices preserve the defining
 equations for $E^+(a)$ or $E^-(a)$ as appropriate, so we have
 holomorphic involutions $\tau_i^+\:E^+(a)\to E^+(a)$ and
 $\tau_i^-\:E^-(a)\to E^-(a)$ for $i=1,2,3$.  We also define
 $\tau_0^\pm$ to be the identity.
 \begin{checks}
  projective/ellquot_check.mpl: check_translations()
 \end{checks}
\end{definition}
The maps $\tau_i^+$ are \mcode+Ep_trans[i]+ (on $E^+_0(a)$) or
\mcode+Ep_0_trans[i]+ (on $E^+(a)$), and the maps $\tau_i^-$ are
\mcode+Em_trans[i]+ or \mcode+Em_0_trans[i]+.

Because $E^\pm(a)$ is an elliptic curve, it is standard that the line
bundle $\Om^1$ is trivial.  For an elliptic curve in Weierstrass form
$y^2=x^3+ax+b$, it is also standard that $dx/y$ is a generator for
$\Om^1$.  As our conventions are slightly different, it is not quite
so standard that the same formula remains valid, but we will now prove
that it is.

\begin{proposition}\lbl{prop-elliptic-forms}
 There is a unique differential form $\om^\pm$ on $E^\pm(a)$ such that
 $j^*(\om^\pm)=dx/y$ on $E_0^\pm(a)$.  Moreover, this is everywhere
 finite and nonzero, so it generates the module $\Om^1_{E^{\pm}(a)}$.
 Near the origin we have $j^*(\om^\pm)=(1+O(y^2))dy$.
\end{proposition}
\begin{proof}
 Differentiating the relation $y^2=q^+_a(x)$ gives $2y\,dy=(q^+_a)'(x)\,dx$
 and thus $dx/y=y\,dx/(q^+_a)(x)=2dy/(q^+_a)'(x)$.  Here $q_a^+(x)$ has no
 repeated roots, so there are no points where $q_a^+(x)$ and $(q_a^+)'(x)$
 both vanish.  It follows that $dx/y$ is finite and nonzero everywhere
 in $E^+_0(a)$.  Next, recall that
 $E^+(a)=j(E^+_0(a))\cup\{v_6^+,v_7^+\}$.  Calculation shows that
 $\tau^+_1(v^+_6)$ and $\tau^+_1(v^+_7)$ lie in $j(E^+_0(a))$, so
 $E^+(a)=j(E^+_0(a))\cup\tau^+_1j(E^+_0(a))$.  One can check from the
 definitions that
 \[ \tau^+_1j(y,x) = j\left(
     \frac{b_-^2y}{(b_+^2x-1)^2},
     \frac{x-1}{b_+^2x-1}
    \right),
 \]
 and thus that $(\tau^+_1)^*(x/dy)$ agrees with $x/dy$ on their common
 domain.  We can thus patch together $dx/y$ with $(\tau^+_1)^*(x/dy)$
 to get a form $\om^+$ which is finite and nonzero everywhere on
 $E^+(a)$, as required.  One can check that $(q_a^+)'(x)=2+O(x)$, and
 the relation $y^2=q^+_a(x)$ gives $x=O(y^2)$.  We have seen that
 $j^*(\om^+)=2dy/(q^+_a)'(x)$, so $j^*(\om^+)=(1+O(y^2))dy$ as claimed.

 The same method works for $E^-(a)$.
 \begin{checks}
  projective/ellquot_check.mpl: check_translations()
 \end{checks}
\end{proof}

\begin{remark}\lbl{rem-translate}
 For any $i$, the form $(\tau^+_i)^*(\om^+)$ must have the form
 $u\,\om^+$ for some function $u$ which is holomorphic everywhere on
 $E^+(a)$, and so is constant.  As $(\tau^+_i)^2=1$ we see that
 $u^2=1$, so $u=\pm 1$.  In the case $i=1$, we saw in the proof of the
 above proposition that $u=1$.  By the same method one can check that
 $u=1$ for $i=2,3$ as well.  This implies that all the maps $\tau^+_i$
 are actually translations with respect to the standard group
 structure on $E^+(a)$.  More specifically, the zero element is
 $o=v^+_0=v^+_1$, and one can check that
 \begin{align*}
  \tau^+_1(o) &= v^+_3 = v^+_5 \\
  \tau^+_2(o) &= v^+_{11} = v^+_{13} \\
  \tau^+_3(o) &= v^+_{10} = v^+_{12}.
 \end{align*}
 Thus, we have $\tau^+_1(p)=p+v^+_3$ and so on.

 The situation for $E^-(a)$ is similar, but with
 \begin{align*}
  \tau^-_1(o) &= v^-_7 = v^-_9 \\
  \tau^-_2(o) &= v^-_{11} = v^-_{13} \\
  \tau^-_3(o) &= v^-_{10} = v^-_{12}.
 \end{align*}
 \begin{checks}
  projective/ellquot_check.mpl: check_translations()
 \end{checks}
\end{remark}

\begin{proposition}\lbl{prop-Ep-Em}
 There are (unbranched) double covering maps
 \[ E^+(a)\xra{\pi^+} E^-(a) \xra{\pi^-} E^+(a) \]
 given generically by
 \begin{align*}
  \pi^+(j(y,x)) &=
    j\left(
      \frac{\rt y((1-x)^2+b_-^2x^2)}{((1-x)^2-b_-^2x^2)^2},\;
      \frac{2x(x-1)}{((1-x)^2-b_-^2x^2)}
     \right) \\
  \pi^-(j(y,x)) &=
    j\left(
      \frac{\rt y((1-x)^2-b_+^2x^2)}{((1-x)^2+b_+^2x^2)^2},\;
      \frac{2x(x-1)}{((1-x)^2+b_+^2x^2)}
     \right).
 \end{align*}
 (More precisely, the above formulae are valid for all points $(y,x)$
 where the denominators are nonzero.)  These are in fact surjective
 group homomorphisms, with
 \begin{align*}
  \ker(\pi^+) &= \{j(0,0),j(0,1)\} = \{v_0^+,v_3^+\} \\
  \ker(\pi^-) &= \{j(0,0),j(0,1)\} = \{v_0^-,v_7^-\}.
 \end{align*}
\end{proposition}
These maps are \mcode+Ep_to_Em+ and \mcode+Em_to_Ep+ (or
\mcode+Ep_0_to_Em_0+ and \mcode+Em_0_to_Ep_0+).
\begin{proof}
 We define $\tpi^+,\tpi^-\:\C^4\to\C^4$ by
 \begin{align*}
  \tpi^+(z)_1 &= \rt z_1(z_2-2z_3+b_+^2z_4) \\
  \tpi^+(z)_2 &= (2-b_+^2)^2z_4^2 +(z_2-2z_3+2(2-b_+^2)z_4)(z_2-2z_3) \\
  \tpi^+(z)_3 &= 2(z_2-b_+^2z_4)(z_4-z_3)+4(z_4^2-2z_3z_4+z_2z_4) \\
  \tpi^+(z)_4 &= 4(z_4-2z_3+z_2)z_4 \\
  \tpi^-(z)_1 &= \rt z_1(z_2-2z_3-b_-^2z_4) \\
  \tpi^-(z)_2 &= (2+b_-^2)^2z_4^2 +(z_2-2z_3+2(2+b_-^2)z_4)(z_2-2z_3) \\
  \tpi^-(z)_3 &= 2(z_2+b_-^2z_4)(z_4-z_3)+4(z_4^2-2z_3z_4+z_2z_4) \\
  \tpi^-(z)_4 &= 4(z_4-2z_3+z_2)z_4.
 \end{align*}
 Recall that
 \begin{align*}
  E^+(a) &= \{[z]\st \rho_0(z)=\rho_1(z)=0\} \\
  E^-(a) &= \{[z]\st \rho_0(z)=\rho_2(z)=0\}
 \end{align*}
 where
 \begin{align*}
  \rho_0(z) &= z_2z_4-z_3^2 \\
  \rho_1(z) &= z_1^2 - 2(z_4-z_3)(b_+^2z_4-z_2) \\
  \rho_2(z) &= z_1^2 - 2(z_4-z_3)(b_-^2z_4+z_2).
 \end{align*}
 We claim that there are no nonzero points in $\C^4$ where
 $\rho_0(z)=\rho_1(z)=0$ and $\tpi^+(z)=0$.  This can be proved by
 using Gr\"obner basis methods to prove that the ideals
 \[ I_k = (z_k-1,\rho_0(z),\rho_1(z),\tpi^+(z)_1,\dotsc,\tpi^+(z)_4) \]
 all contain $1$, or just by solving the equations in a more
 elementary way.  One can also use Gr\"obner bases to check that
 $\rho_i(\tpi^+(z))\in(\rho_0(z),\rho_1(z))$ for $i\in\{0,2\}$.  It
 follows that the rule $\pi^+([z])=[\tpi^+(z)]$ gives a well-defined
 morphism $\pi^+\:E^+(a)\to E^-(a)$.  Straightforward algebra shows
 that $\pi^+(j(y,x))$ is given by the stated formula whenever
 $(1-x)^2-b_-^2x^2\neq 0$.  In particular, we have
 $\pi^+(j(0,0))=j(0,0)$, so $\pi^+$ preserves basepoints.  It is a
 standard fact that any basepoint preserving morphism of elliptic
 curves is a group homomorphism, and in this context, any non-constant
 group homomorphism is a covering map, so we just need to identify
 $\ker(\pi^+)$.  If $(1-x)^2-b_-^2x^2\neq 0$ then the stated formula
 for $\pi^+(j(x,y))$ is valid, and we see that $\pi^+(j(x,y))=j(0,0)$
 iff $(x,y)\in\{(0,0),(0,1)\}$.  The exceptional points where
 $(1-x)^2-b_-^2x^2=0$ are as follows:
 \begin{align*}
  w_1 &= j(\pp 2b_-/(1+b_-)^2,\;1/(1+b_-)) &
  w_2 &= j(\pp 2b_-/(1-b_-)^2,\;1/(1-b_-)) \\
  w_3 &= j(  - 2b_-/(1+b_-)^2,\;1/(1+b_-)) &
  w_4 &= j(  - 2b_-/(1-b_-)^2,\;1/(1-b_-)).
 \end{align*}
 These satisfy $\pi^+(w_1)=\pi^+(w_2)=v_2^+\neq v_0^+$ and
 $\pi^+(w_3)=\pi^+(w_4)=v_4^+\neq v_0^+$, so they do not contribute to
 the kernel.  This just leaves the points in $E^+(a)$ that do not lie
 in the image of $j$, which are $v_6^+$ and $v_7^+$; direct
 calculation shows again that these are not in the kernel of $\pi^+$.
 This completes the proof that $\ker(\pi^+)=\{v_0^+,v_3^+\}$.  (As an
 alternative, we could reach the same conclusion by calculating
 Gr\"obner bases for each of the ideals
 \[ (z_i-1,\rho_0(z),\rho_1(z),\tpi^+(z)_1,\tpi^+(z)_3,\tpi^+(z)_4),
 \]
 or by showing that the degree of the relevant field extension is two
 and appealing to some more abstract arguments.)

 The proof for $\pi^-$ is essentially the same.
 \begin{checks}
  projective/ellquot_check.mpl: check_isogenies()
 \end{checks}
\end{proof}

\begin{definition}\lbl{defn-PJ}
 We note that Proposition~\ref{prop-Ep-Em} implies that the elements
 $v_3^+$ and $v_7^-$ have order two, so $(v_7^-,v_3^+)$ generates
 a subgroup $Z$ of order two in $(E^-(a)\tm E^+(a))$.  We define
 \[ PJ(a) = (E^-(a)\tm E^+(a))/Z. \]
 We also define $\tht^+\:PJ(a)\to E^+(a)$ and
 $\tht^-\:PJ(a)\to E^-(a)$ by
 \begin{align*}
  \tht^+((w^-,w^+)+Z) &= \pi^-(w^-) \\
  \tht^-((w^-,w^+)+Z) &= \pi^+(w^+).
 \end{align*}
\end{definition}

\begin{proposition}\lbl{prop-P-to-PJ}
 There is a unique morphism $\phi\:PX(a)\to PJ(a)$ such that
 $\phi(v_0)=o$ and $\tht^+\phi=\phi^+$ and $\tht^-\phi=\phi^-$.
\end{proposition}
\begin{proof}
 Suppose we have $(w,z)\in PX_0(a)$ and $u\in\C$ with $u^2=z$.  When the
 relevant denominators are nonzero, we then put
 \begin{align*}
  x_- &= \;\frac{w/u - (1-z)^2}{2(b_-^2z-\;(1-z)^2)} \\
  x_+ &=  i\frac{w/u + (i+z)^2}{2(b_+^2z+ i(i+z)^2)} \\
  y_- &= \frac{1}{\rt}\; \frac{u(1-z)(2b_+^2(1-z)^2-b_-^2(w/u+1+z^2))}{(b_-^2z-\;(1-z)^2)^2} \\
  y_+ &= \frac{1+i}{2}\; \frac{u(i+z)(2b_-^2(i+z)^2+b_+^2(w/u+1-z^2))}{(b_+^2z+i (i+z)^2)^2},
 \end{align*}
 and $\phi_0(w,z,u)=(y_-,x_-,y_+,x_+)$.  One can check that
 \begin{itemize}
  \item[(a)] $y_-^2=q^-_a(x_-)$ and $y_+^2=q^+_a(x_+)$, so
   $\phi_0(w,z,u)\in E^-_0(a)\tm E^+_0(a)$.
  \item[(b)] $\phi_0(w,z,-u)=(\tau_1^-\tm\tau_1^+)\phi_0(w,z,u)$.
  \item[(c)] $\pi^-(j(y^-,x^-))=\phi^+_0(w,z)$ and
   $\pi^+(j(y^+,x^+))=\phi^-_0(w,z)$.
 \end{itemize}
 It follows that the image of $\phi_0(w,z,u)$ in $PJ(a)$ is
 independent of the choice of $u$, so we can call it $\phi_0(w,z)$.
 This defines a rational map from $EX_0(a)$ to $PJ(a)$, but $EX_0(a)$
 is dense subset of the smooth curve $EX(a)$, and $PJ(a)$ is complete,
 so this extends uniquely to give a morphism $\phi\:PX(a)\to PJ(a)$.
 Point~(c) above shows that $\tht^\pm\phi=\phi^\pm$.
 \begin{checks}
  projective/ellquot_check.mpl: check_isogenies()
 \end{checks}
\end{proof}
The map $\phi$ is represented in Maple as \mcode+P_0_to_J_0+.

\begin{corollary}\lbl{cor-jacobian}
 $PJ(a)$ can be regarded as the Jacobian variety of $PX(a)$.
\end{corollary}
\begin{proof}
 The Jacobian variety $J$ can be characterised by the fact that it is
 an abelian variety equipped with a map $\dl\:PX(a)\to J$ of varieties
 such that the induced map $H_1PX(a)\to H_1J$ is an isomorphism.  We
 saw in Definition~\ref{defn-JX} that the map
 \[ (\phi^+,\phi^-)_*\: H_1PX(a) \to H_1(E^+(a)\tm E^-(a)) \]
 is injective, and has image of index two.  As the map
 \[ \tht=(\tht^+,\tht^-)\:PJ(a)\to E^+(a)\tm E^-(a) \]
 is a connected double covering, it also gives an index two subgroup
 of $\pi_1=H_1$.  As $(\phi^+,\phi^-)=\tht\phi$, we deduce that
 $\phi_*\:H_1PX(a)\to H_1PJ(a)$ is an isomorphism, as required.
\end{proof}

It is standard that any elliptic curve has an analytic parametrisation
via the Weierstrass $\wp$-function.  Details for the present case are
as follows.
\begin{definition}\lbl{defn-weierstrass-xi}
 We put
 \begin{align*}
  g_2^+ &= 4(\tfrac{1}{3}+b_+^2) &
  g_2^- &= 4(\tfrac{1}{3}-b_-^2) \\
  g_3^+ &= \tfrac{8}{3}(\tfrac{1}{9}-b_+^2) &
  g_3^- &= \tfrac{8}{3}(\tfrac{1}{9}+b_-^2) \\
  p_0^+(z) &= \wp(z/\rt;g_2^+,g_3^+) &
  p_0^-(z) &= \wp(iz/\rt;g_2^-,g_3^-) \\
  p_1^+(z) &= \wp'(z/\rt;g_2^+,g_3^+) &
  p_1^-(z) &= \wp'(iz/\rt;g_2^-,g_3^-)
 \end{align*}
 and then
 \begin{align*}
  \xi^+(z) &= j\left(-\frac{p_1^+(z)}{\rt(p_0^+(z)+1/3)^2},
                       \frac{1}{(p_0^+(z)+1/3)}\right)\in\C P^3 \\
  \xi^-(z) &= j\left(i\frac{p_1^-(z)}{\rt(p_0^-(z)+1/3)^2},
                       \frac{1}{(p_0^-(z)+1/3)}\right)\in\C P^3.
 \end{align*}
\end{definition}
In Maple, the parameters $g_i^\pm$ are
\mcode+Wg2p+, \mcode+Wg3p+, \mcode+Wg2m+ and \mcode+Wg3m+.  The maps
$\xi^\pm$ are \mcode+C_to_Ep_0+ and \mcode+C_to_Em_0+.

\begin{remark}\lbl{rem-xi-differential}
 It is a standard fact that $\wp(z)=z^{-2}+O(z^2)$, so
 $\wp'(z)=-2z^{-3}+O(z)$.  Using this, we find that
 $j^{-1}\xi^+(z)=(z,0)+O(z^2)$ and  $j^{-1}\xi^-(z)=(z,0)+O(z^2)$.
\end{remark}

\begin{proposition}\lbl{prop-xi-functions}
 There are lattices $\Lm^+,\Lm^-\subset\C$ such that $\xi^{\pm}$
 induces an isomorphism $\C/\Lm^{\pm}\to E^{\pm}(a)$.  Moreover, the
 forms $\om^\pm$ on $E^{\pm}(a)$ (from
 Proposition~\ref{prop-elliptic-forms}) satisfy
 $(\xi^\pm)^*(\om^\pm)=dz$.
\end{proposition}
\begin{proof}
 Given any $g_2,g_3$ we can define $f(x)=4x^3-g_2x-g_3$ and
 $F_0=\{(y,x)\in\C^2\st y^2=f(x)\}$, and we can then define $F$ to be
 the normalisation of $F_0$.  It is standard that
 \[ \wp'(z;g_2,g_3)^2 = f(\wp(z;g_2,g_3)), \]
 so we have a meromorphic function $z\mapsto(\wp'(z),\wp(z))$ from
 $\C$ to $F$.  It is also standard that this induces an isomorphism
 $\C/\Lm\to F$, for a suitable lattice $\Lm\subset\C$.  The first
 claim follows from this by a change of coordinates.

 Next, we have $(\xi^\pm)^*(\om^\pm)=u(z)\,dz$ for some function
 $u(z)$ which is holomorphic, nowhere zero, and periodic with respect
 to $\Lm^\pm$.  This forces $u(z)$ to be constant.
 Remark~\ref{rem-xi-differential}, together with the last part of
 Proposition~\ref{prop-elliptic-forms}, shows that $u=1$.
 \begin{checks}
  projective/ellquot_check.mpl: check_weierstrass()
 \end{checks}
\end{proof}

We now want to understand the lattices $\Lm^+$ and $\Lm^-$ in more
detail, which is essentially the same as calculating the periods
$p_{jk}(a)=\int_{c_j}\om_k$ as in Definition~\ref{defn-period-p}.

\begin{definition}
 We define the complete elliptic integral $K(k)$ (for $0<k<1$) by
 \[ K(k) = \int_0^1 \frac{dt}{\sqrt{1-t^2}\sqrt{1-k^2t^2}}. \]
 In this range the square roots are real and positive and there is no
 need for branch cuts.
\end{definition}

Our definition is the same as Maple's \mcode+EllipticK(k)+, but
slightly different conventions are used in some other sources.

\begin{definition}\lbl{defn-period-rs}
 We put
 \begin{align*}
  m_+ &= \frac{1 + a}{\sqrt{2(1+a^2)}} &
  m_- &= \frac{1 - a}{\sqrt{2(1+a^2)}} \\
  \al_+ &= 2b_+^{-1/2} K(m_+) &
  \al_- &= 2b_+^{-1/2} K(m_-).
 \end{align*}
 (Note here that the first factor in $\al_-$ involves $b_+$, not $b_-$.)
\end{definition}
In Maple, these are \mcode+mp_period+, \mcode+mm_period+,
\mcode+ap_period+ and \mcode+am_period+.  These are defined in the file
\fname+projective/picard_fuchs.mpl+.
%

\begin{remark}\lbl{rem-legendre}
 A mixture of theoretical arguments and numerical calculations makes
 it clear that we also have
 \begin{align*}
  \al_+ &= \pi P_{-1/4}(A/2) + 2Q_{-1/4}(A/2) \\
  \al_- &= \pi P_{-1/4}(A/2)
 \end{align*}
 where $A=a^{-2}+a^2$ as before, and $P$ and $Q$ are Legendre
 functions.  We do not have a complete proof, but we will mention some
 ingredients.  Consider the differential operator
 \[ \CL =
     198 \frac{\partial^2}{\partial A^2} +
     192 A \frac{\partial^3}{\partial A^3} +
     (32A^2 - 128) \frac{\partial^4}{\partial A^4}.
 \]
 One can check by direct calculation that
 \[ \CL((z^5-Az^3+z)^{-1/2}) \,dz =
      \frac{d}{dz}\left(
       \frac{33z^8-3Az^{10}-27z^{12}}{(z^5-Az^3+z)^{7/2}}
      \right) \,dz.
 \]
 The terms on the left can be interpreted as meromorphic differential
 forms on $PX(a)$, all of whose residues are zero.  This means that
 their integrals round a loop depend only on the homology class of the
 loop.  On the other hand, the integral of the right hand side around
 any loop is zero.  Using this, we see that the periods (when
 expressed as a function of $A$) are annihilated by $\CL$.  This is
 the Picard-Fuchs equation for the family $\{PX(a)\st a\in(0,1)\}$.
 Maple asserts that the annihilator of $\CL$ is spanned by $1$, $A$,
 $P_{-1/4}(A/2)$ and $Q_{-1/4}(A/2)$.  This could presumably be
 checked using hypergeometric series, but we have not attempted that.
 The specific coefficients in the stated expressions for $\al_+$ and
 $\al_-$ were obtained by graphical and numerical experimentation.
 \begin{checks}
  projective/picard_fuchs_check.mpl: check_picard-fuchs()
 \end{checks}
\end{remark}

\begin{proposition}\lbl{prop-periods}
 The periods are
 \begin{align*}
  p_{0,0} &=          0       & p_{0,1} &= 0          \\
  p_{1,0} &= (i+1)\al_-       & p_{1,1} &= (i-1)\al_- \\
  p_{2,0} &= (i-1)\al_-       & p_{2,1} &= (i+1)\al_- \\
  p_{3,0} &=     i\al_-       & p_{3,1} &=     i\al_- \\
  p_{4,0} &=     -\al_-       & p_{4,1} &=      \al_- \\
  p_{5,0} &=  (\al_++\al_-)/2 & p_{5,1} &=   (\al_+-\al_-)/2 \\
  p_{6,0} &= i(\al_++\al_-)/2 & p_{6,1} &= -i(\al_+-\al_-)/2 \\
  p_{7,0} &=  (\al_+-\al_-)/2 & p_{7,1} &=   (\al_++\al_-)/2 \\
  p_{8,0} &= i(\al_+-\al_-)/2 & p_{8,1} &= -i(\al_++\al_-)/2.
 \end{align*}
\end{proposition}
In Maple, $p_{ij}$ is \mcode+p_period[i,j]+.

The proof will be given after some preparatory lemmas.

\begin{lemma}\lbl{lem-int-q}
 \begin{align*}
  \int_0^{1/b_+}  \frac{dx}{\sqrt{ q_a^+(x)}} &= \al_+/2 \\
  \int_{-1/b_+}^0 \frac{dx}{\sqrt{-q_a^+(x)}} &= \al_-/2.
 \end{align*}
\end{lemma}
\begin{proof}
 For the first integral we use the substitution
 $x=(1-t^2)/(b_+-t^2)$, and for the second we use the
 substitution $x=b_+^{-1}(1-2/t^2)^{-1}$.  In both cases we get an
 extra factor of $-1$ in the integrand, which is cancelled by the fact
 that the limits are reversed, because $x$ is a decreasing function of
 $t$.
 \begin{checks}
  projective/picard_fuchs_check.mpl: check_period_integrals()
 \end{checks}
\end{proof}

\begin{lemma}\lbl{lem-elliptic-periods}
 \begin{align*}
  \int_{\phi^+\circ c_5}\frac{dx}{y} &= \al_+ \\
  \int_{\phi^+\circ c_6}\frac{dx}{y} &= i\al_-.
 \end{align*}
\end{lemma}
\begin{proof}
 From the definitions we have $j^{-1}(c_5(t))=(w(t),z(t))$, where
 $z(t)=a\sin(t/2)^2\in [0,a]$ and $w(t)$ is a positive multiple of
 $\sin(t)$.  It follows that $j^{-1}\phi^+(c_5(t))=(y(t),x(t))$, where
 $x(t)=2/(z(t)+z(t)^{-1})$, and $y(t)$ is again a positive multiple of
 $\sin(t)$.  On the other hand, we have $(y(t),x(t))\in E^+_0(a)$ so
 $y(t)^2=q_a^+(x(t))$, so for $0\leq t\leq \pi$ we must have
 $y(t)=\sqrt{q^+_a(x(t))}$, and for $\pi\leq t\leq 2\pi$ we must have
 $y(t)=-\sqrt{q^+_a(x(t))}$.   Moreover, as $t$ runs from $0$ to $\pi$
 we see that $z(t)$ increases from $0$ to $a$, and so $x(t)$ increases
 from $0$ to $1/b_+$.  On the other hand, as $t$ increases from $\pi$
 to $2\pi$ we see that $x(t)$ decreases from $1/b_+$ to $0$.  It
 follows that
 \[ \int_{\phi^+\circ c_5}\frac{dx}{y} =
    \int_{x=0}^{1/b_+}\frac{dx}{\sqrt{q_a^+(x)}} +
    \int_{x=1/b_+}^{0}\frac{dx}{-\sqrt{q_a^+(x)}} =
    2\int_{0}^{1/b_+}\frac{dx}{\sqrt{q_a^+(x)}} =
    \al_+.
 \]

 The second integral is similar.  We have
 $j^{-1}(c_6(t))=(i\,w(t),-z(t))$, where $z(t)=a\sin(t/2)^2\in [0,a]$ and
 $w(t)$ is a positive multiple of $\sin(t)$.  It follows that
 $j^{-1}\phi^+(c_5(t))=(y(t),x(t))$, where $x(t)=-2/(z(t)+z(t)^{-1})$,
 and $y(t)$ is again a positive multiple of $i\sin(t)$.  In this range
 $q_a^+(x(t))\leq 0$ so it is natural to consider
 $\sqrt{-q_a^+(x(t))}$.  As $(y(t),x(t))\in E^+_0(a)$, we must have
 $(y(t)/i)^2=-q_a^+(x(t))$, so for $0\leq t\leq \pi$ we must have
 $y(t)=i\sqrt{-q^+_a(x(t))}$, and for $\pi\leq t\leq 2\pi$ we must
 have $y(t)=-i\sqrt{-q^+_a(x(t))}$.  Moreover, as $t$ runs from $0$ to
 $\pi$ and then to $2\pi$, we see that $x(t)$ decreases from $0$ to
 $-1/b_+$, and then increases back to $0$ again.  It follows that
 \[ \int_{\phi^+\circ c_6}\frac{dx}{y} =
    \int_{0}^{-1/b_+}i\frac{dx}{\sqrt{-q_a^+(x)}} +
    \int_{-1/b_+}^{0}-i\frac{dx}{\sqrt{-q_a^+(x)}} =
    2i\int_{0}^{1/b_+}\frac{dx}{\sqrt{-q_a^+(x)}} =
    i\al_-.
 \]
 \begin{checks}
  projective/picard_fuchs_check.mpl: check_period_integrals()
 \end{checks}
\end{proof}

\begin{lemma}\lbl{lem-form-pullback}
 For the forms $\om^{\pm}$ on $E^{\pm}(a)$ and $\om_i$ on $PX(a)$ (as in
 Propositions~\ref{prop-holomorphic-forms}
 and~\ref{prop-elliptic-forms}), we have
 \begin{align*}
  (\phi^+)^*(\om^+) &= \om_0+\om_1 \\
  (\phi^-)^*(\om^+) &= \frac{1+i}{\rt}(i\om_0-\om_1).
 \end{align*}
\end{lemma}
\begin{proof}
 On the affine pieces $PX_0(a)$ and $E^+_0(a)$ we have $\om^+=dx/y$
 and
 \[ \phi^+(w,z) = \left(\frac{2w(1-z)}{(1+z^2)^2},
                        \frac{2z}{1+z^2}\right),
 \]
 so
 \begin{align*}
  (\phi^+)^*(\om^+) &=
     \frac{(1+z^2)^2}{2w(1-z)}d\left(\frac{2z}{1+z^2}\right) \\
  &= \frac{(1+z^2)^2}{2w(1-z)} \; \frac{2-2z^2}{(1+z^2)^2}dz
   = \frac{1+z}{w}dz.
 \end{align*}
 On the other hand, we have $\om_0=dz/w$ and $\om_1=z\,dz/w$, so
 $(\phi^+)^*(\om^+)=\om_0+\om_1$.  The argument for
 $(\phi^-)^*(\om^-)$ is similar.
 \begin{checks}
  projective/picard_fuchs_check.mpl: check_period_integrals()
 \end{checks}
\end{proof}

\begin{corollary}\lbl{cor-basic-periods}
 We have $p_{5,0}=(\al_++\al_-)/2$ and $p_{5,1}=(\al_+-\al_-)/2$.
\end{corollary}
\begin{proof}
 Lemma~\ref{lem-elliptic-periods} says that
 \[ \int_{c_5}(\phi^+)^*\left(\frac{dx}{y}\right) = \al_+. \]
 On the left hand side, Lemma~\ref{lem-form-pullback} says that the
 integrand is $\om_0+\om_1$, so the integral is $p_{5,0}+p_{5,1}$.
 Similarly, the second equation in Lemma~\ref{lem-elliptic-periods}
 becomes $p_{6,0}+p_{6,1}=i\al_-$.  On the other hand, we have seen that
 $\lm\circ c_5=c_6$ and $\lm^*\om_0=i\om_0$ and $\lm^*\om_1=-i\om_1$;
 it follows that $p_{6,0}+p_{6,1}=ip_{5,0}-ip_{5,1}$.  Linear algebra
 now gives $p_{5,0}=(\al_++\al_-)/2$ and $p_{5,1}=(\al_+-\al_-)/2$ as
 claimed.
\end{proof}

\begin{proof}[Proof of Proposition~\ref{prop-periods}]
 It is standard that there is a well-defined pairing
 $H_1(PX(a))\ot\Om^1(PX(a))\to\C$ such that $([\gm],\al)=\int_\gm\al$
 for all closed curves $\gm$ in $PX(a)$ and all $\om\in\Om^1(PX(a))$.
 In particular, we have $p_{i,j}=([c_i],\om_j)$.  Thus, relations in
 $H_1(PX(a))$ will give relations between periods.  Note also that for
 $g\in D_8$ we have
 $(g_*[c_i],\om_j)=([g\circ c_i],\om_j)=([c_i],g^*\om_j)$.
 Proposition~\ref{prop-holomorphic-forms} gives the action of $D_8$ on
 $\om_0$ and $\om_1$, whereas Definition~\ref{defn-curve-system}(c)
 and Proposition~\ref{prop-P-curves} give the action on the curves
 $c_i$.  In particular, we have
 \[ [c_6] = \lm_*[c_5] \hspace{4em}
    [c_7] = \mu_*[c_5] \hspace{4em}
    [c_8] = (\mu\lm)_*[c_5],
 \]
 whereas
 \[ \lm^*\om_0 = i\om_0 \hspace{4em}
    \lm^*\om_1 = -i\om_1 \hspace{4em}
    \mu^*\om_0 = \om_1  \hspace{4em}
    \mu^*\om_1 = \om_0.
 \]
 It therefore follows from Corollary~\ref{cor-basic-periods} that
 \begin{align*}
  p_{5,0} &=  (\al_++\al_-)/2 & p_{5,1} &=   (\al_+-\al_-)/2 \\
  p_{6,0} &= i(\al_++\al_-)/2 & p_{6,1} &= -i(\al_+-\al_-)/2 \\
  p_{7,0} &=  (\al_+-\al_-)/2 & p_{7,1} &=   (\al_++\al_-)/2 \\
  p_{8,0} &= i(\al_+-\al_-)/2 & p_{8,1} &= -i(\al_++\al_-)/2.
 \end{align*}
 Next, we see from Proposition~\ref{prop-homology} that
 \begin{align*}
  [c_0] &= 0 \\
  [c_1] &=  [c_5]+[c_6]-[c_7]-[c_8] \\
  [c_2] &= -[c_5]+[c_6]+[c_7]-[c_8] \\
  [c_3] &=        [c_6]      -[c_8] \\
  [c_4] &= -[c_5]      +[c_7].
 \end{align*}
 We can now apply the maps $(-,\om_0)$ and $(-,\om_1)$ to deduce that
 \begin{align*}
  p_{0,0} &=          0       & p_{0,1} &= 0          \\
  p_{1,0} &= (i+1)\al_-       & p_{1,1} &= (i-1)\al_- \\
  p_{2,0} &= (i-1)\al_-       & p_{2,1} &= (i+1)\al_- \\
  p_{3,0} &=     i\al_-       & p_{3,1} &=     i\al_- \\
  p_{4,0} &=     -\al_-       & p_{4,1} &=      \al_-.
 \end{align*}
 \begin{checks}
  projective/picard_fuchs_check.mpl: check_periods()
 \end{checks}
\end{proof}

\begin{proposition}
 The lattices $\Lm^+$ and $\Lm^-$ (from
 Proposition~\ref{prop-xi-functions}) are
 \begin{align*}
  \Lm^+ &= \{n\al_++m\al_-i \st n,m\in\Z\} \\
  \Lm^- &= \{(n\al_++m\al_-i)/\sqrt{2}\st n,m\in\Z,\; n=m\pmod{2}\}.
 \end{align*}
\end{proposition}
\begin{proof}
 We can identify $\Lm^+$ with
 $\{\int_\gm dz\st\gm\in\pi_1(\C/\Lm^+)\}$.  Now $\xi^+$ induces an
 isomorphism $\C/\Lm^+\to E^+(a)$, under which $dz$ corresponds to
 $\om^+$ (by the last part of Proposition~\ref{prop-xi-functions}).
 This means that $\Lm^+=\{\int_\gm\om^+\st\gm\in\pi_1(E^+(a))\}$.  On
 the other hand, we know from Proposition~\ref{prop-elliptic-homology}
 that the group $\pi_1(E^+(a))=H_1(E^+(a))$ is a quotient of
 $H_1(PX(a))$, so
 \[ \Lm^+=\{\int_\gm(\phi^+)^*(\om^+)\st\gm\in\pi_1(PX(a))\}. \]
 Using Lemma~\ref{lem-form-pullback} we now see that $\Lm^+$ is
 spanned by the numbers $p_{k0}+p_{k1}$, and by inspecting
 Proposition~\ref{prop-periods} we conclude that
 $\Lm^+=\{n\al_++m\al_-i\st n,m\in\Z\}$ as claimed.

 In the same way, using the relation
 $(\phi^-)^*(\om^-)=\frac{1+i}{\rt}(i\om_0-\om_1)$ we see that $\Lm^-$
 is generated by the numbers $\frac{1+i}{\rt}(ip_{k0}-p_{k1})$.  These
 numbers can again be read off from Proposition~\ref{prop-periods},
 giving
 \[ \Lm^- = \{(n\al_++m\al_-i)/\sqrt{2}\st n,m\in\Z,\; n=m\pmod{2}\}
 \]
 as claimed.
 \begin{checks}
  projective/ellquot_check.mpl: check_weierstrass()
 \end{checks}
\end{proof}

The closed curves $c_k(t)\in PX(a)$ can be mapped to $E^+(a)$ using
$\phi^+$ and then lifted via $\xi^+$ to give curves in $\C$ which are
usually not closed.  There are no closed formulae for these curves,
but the functions \mcode+c_TEp_approx[k]+ and \mcode+c_TEm_approx[k]+
give good approximations.  The set
\[ \{x+iy\st |x|\leq\al_+/2,|y|\leq\al_-/2\} \]
is a fundamental domain for the action of $\Lm^+$ on $\C$, and the
parts of the lifted curves lying in that domain can be illustrated as
follows:
\begin{center}
 \begin{tikzpicture}[scale=4]
 \draw[magenta] (0.859,-0.785) -- (0.859,0.785);
 \draw[magenta] (-0.859,-0.785) -- (-0.859,0.785);
 \draw[blue] (-0.859,0.000) -- (0.859,0.000);
 \draw[blue] (0.000,-0.785) -- (0.000,0.785);
 \draw[magenta] (-0.466,0.785) -- (0.466,0.785);
 \draw[magenta] (-0.466,-0.785) -- (0.466,-0.785);
 \draw[cyan] (-0.859,0.785) -- (-0.466,0.785);
 \draw[cyan] (0.859,0.785) -- (0.466,0.785);
 \draw[cyan] (-0.859,-0.785) -- (-0.466,-0.785);
 \draw[cyan] (0.859,-0.785) -- (0.466,-0.785);
 \draw[smooth,green]  (-0.581,-0.785) -- (-0.576,-0.736) -- (-0.560,-0.679) -- (-0.528,-0.608) -- (-0.473,-0.520) -- (-0.390,-0.412) -- (-0.280,-0.286) -- (-0.146,-0.147) -- (0.000,0.000) -- (0.146,0.147) -- (0.280,0.286) -- (0.390,0.412) -- (0.473,0.520) -- (0.528,0.608) -- (0.560,0.679) -- (0.576,0.736) -- (0.581,0.785) ;
 \draw[smooth,green]  (0.581,-0.785) -- (0.576,-0.736) -- (0.560,-0.679) -- (0.528,-0.608) -- (0.473,-0.520) -- (0.390,-0.412) -- (0.280,-0.286) -- (0.146,-0.147) -- (0.000,0.000) -- (-0.146,0.147) -- (-0.280,0.286) -- (-0.390,0.412) -- (-0.473,0.520) -- (-0.528,0.608) -- (-0.560,0.679) -- (-0.576,0.736) -- (-0.581,0.785) ;
 \fill (-0.859, 0.785) circle(0.011);
 \fill (-0.581, 0.785) circle(0.011);
 \fill (-0.466, 0.785) circle(0.011);
 \fill ( 0.000, 0.785) circle(0.011);
 \fill ( 0.466, 0.785) circle(0.011);
 \fill ( 0.581, 0.785) circle(0.011);
 \fill ( 0.859, 0.785) circle(0.011);
 \fill (-0.859,-0.785) circle(0.011);
 \fill (-0.581,-0.785) circle(0.011);
 \fill (-0.466,-0.785) circle(0.011);
 \fill ( 0.000,-0.785) circle(0.011);
 \fill ( 0.466,-0.785) circle(0.011);
 \fill ( 0.581,-0.785) circle(0.011);
 \fill ( 0.859,-0.785) circle(0.011);
 \fill (-0.859, 0.000) circle(0.011);
 \fill ( 0.000, 0.000) circle(0.011);
 \fill ( 0.859, 0.000) circle(0.011);
 \end{tikzpicture}
\end{center}
This is combinatorially equivalent to the net for $X/\ip{\mu}$ which
we exhibited in Section~\ref{sec-fundamental}.  We have used the value
$a=0.1$, which is roughly right for the embedded surface $EX^*$.  The
above domain is not square; the ratio
$\text{width}/\text{height}=\al_+/\al_-$ is approximately $1.1$.

The situation for $E^-(a)$ is a little more complicated.  The picture
below shows the domain
\[ \{x+iy\st |x|\leq\al_+/\rt,|y|\leq\al_-/\rt\}, \]
which covers $\C/\Lm^-$ twice.  The diagonal blue curves, which
represent $c_5$ and $c_6$, are close to being straight, but they are
not exactly straight.  The dashed orange lines enclose a fundamental
domain for $\Lm^-$.  This is a rhombus, but the angles are not
$\pi/2$.  It is combinatorially equivalent to the net for
$X/\ip{\lm\mu}$ which we exhibited in Section~\ref{sec-fundamental},
but is rotated through $\pi/4$ as well as being slightly
distorted.
\begin{center}
 \begin{tikzpicture}[scale=3]
 \draw[cyan]  (-0.394, 1.111) -- ( 0.394, 1.111);
 \draw[cyan]  (-0.394,-1.111) -- ( 0.394,-1.111);
 \draw[cyan]  (-1.215, 0.000) -- (-0.822, 0.000);
 \draw[cyan]  ( 1.215, 0.000) -- ( 0.822, 0.000);
 \draw[green] (-0.822, 0.000) -- ( 0.822, 0.000);
 \draw[green] ( 0.394, 1.111) -- ( 1.215, 1.111);
 \draw[green] (-0.394, 1.111) -- (-1.215, 1.111);
 \draw[green] ( 0.394,-1.111) -- ( 1.215,-1.111);
 \draw[green] (-0.394,-1.111) -- (-1.215,-1.111);
 \draw[green] (-1.215,-1.111) -- (-1.215, 1.111);
 \draw[green] ( 0.000,-1.111) -- ( 0.000, 1.111);
 \draw[green] ( 1.215,-1.111) -- ( 1.215, 1.111);
 \draw[smooth,magenta]  (0.278,-1.111) -- (0.292,-1.024) -- (0.344,-0.907) -- (0.453,-0.746) -- (0.608,-0.555) -- (0.762,-0.364) -- (0.871,-0.204) -- (0.924,-0.087) -- (0.937,0.000) -- (0.924,0.087) -- (0.871,0.204) -- (0.762,0.364) -- (0.608,0.555) -- (0.453,0.746) -- (0.344,0.907) -- (0.292,1.024) -- (0.278,1.111) ;
 \draw[smooth,magenta]  (-0.278,-1.111) -- (-0.292,-1.024) -- (-0.344,-0.907) -- (-0.453,-0.746) -- (-0.608,-0.555) -- (-0.762,-0.364) -- (-0.871,-0.204) -- (-0.924,-0.087) -- (-0.937,0.000) -- (-0.924,0.087) -- (-0.871,0.204) -- (-0.762,0.364) -- (-0.608,0.555) -- (-0.453,0.746) -- (-0.344,0.907) -- (-0.292,1.024) -- (-0.278,1.111) ;
 \draw[smooth,blue]  (-1.215,1.111) -- (-1.129,1.024) -- (-1.044,0.941) -- (-0.964,0.864) -- (-0.887,0.794) -- (-0.814,0.729) -- (-0.744,0.668) -- (-0.675,0.611) -- (-0.608,0.555) -- (-0.540,0.500) -- (-0.471,0.442) -- (-0.401,0.382) -- (-0.328,0.317) -- (-0.251,0.246) -- (-0.171,0.169) -- (-0.087,0.086) -- (0.000,0.000) -- (0.087,-0.086) -- (0.171,-0.169) -- (0.251,-0.246) -- (0.328,-0.317) -- (0.401,-0.382) -- (0.471,-0.442) -- (0.540,-0.500) -- (0.608,-0.555) -- (0.675,-0.611) -- (0.744,-0.668) -- (0.814,-0.729) -- (0.887,-0.794) -- (0.964,-0.864) -- (1.044,-0.941) -- (1.129,-1.024) -- (1.215,-1.111) ;
 \draw[smooth,blue]  (-1.215,-1.111) -- (-1.129,-1.024) -- (-1.044,-0.941) -- (-0.964,-0.864) -- (-0.887,-0.794) -- (-0.814,-0.729) -- (-0.744,-0.668) -- (-0.675,-0.611) -- (-0.608,-0.555) -- (-0.540,-0.500) -- (-0.471,-0.442) -- (-0.401,-0.382) -- (-0.328,-0.317) -- (-0.251,-0.246) -- (-0.171,-0.169) -- (-0.087,-0.086) -- (0.000,0.000) -- (0.087,0.086) -- (0.171,0.169) -- (0.251,0.246) -- (0.328,0.317) -- (0.401,0.382) -- (0.471,0.442) -- (0.540,0.500) -- (0.608,0.555) -- (0.675,0.611) -- (0.744,0.668) -- (0.814,0.729) -- (0.887,0.794) -- (0.964,0.864) -- (1.044,0.941) -- (1.129,1.024) -- (1.215,1.111) ;
 \fill (-1.215,-1.111) circle(0.014);
 \fill (-0.394,-1.111) circle(0.014);
 \fill (-0.278,-1.111) circle(0.014);
 \fill ( 0.000,-1.111) circle(0.014);
 \fill ( 0.278,-1.111) circle(0.014);
 \fill ( 0.394,-1.111) circle(0.014);
 \fill ( 1.215,-1.111) circle(0.014);
 \fill (-1.215, 1.111) circle(0.014);
 \fill (-0.394, 1.111) circle(0.014);
 \fill (-0.278, 1.111) circle(0.014);
 \fill ( 0.000, 1.111) circle(0.014);
 \fill ( 0.278, 1.111) circle(0.014);
 \fill ( 0.394, 1.111) circle(0.014);
 \fill ( 1.215, 1.111) circle(0.014);
 \fill (-0.608,-0.555) circle(0.014);
 \fill ( 0.608,-0.555) circle(0.014);
 \fill (-0.608, 0.555) circle(0.014);
 \fill ( 0.608, 0.555) circle(0.014);
 \fill (-1.215, 0.000) circle(0.014);
 \fill (-0.822, 0.000) circle(0.014);
 \fill (-0.937, 0.000) circle(0.014);
 \fill ( 0.000, 0.000) circle(0.014);
 \fill ( 0.937, 0.000) circle(0.014);
 \fill ( 0.822, 0.000) circle(0.014);
 \fill ( 1.215, 0.000) circle(0.014);
 \draw[orange,dashed] (1.215,0.000) -- (0.000,1.111) -- (-1.215,0.000) -- (0.000,-1.111) -- cycle;
 \end{tikzpicture}
\end{center}

\subsection{Some general theory of Riemann surfaces}
\lbl{sec-riemann}

In the next section, we will give a classification of (pre)cromulent
surfaces.  In the present section, we develop some more general theory
of Riemann surfaces, which will feed into that classification.  All of
it is essentially standard; we discuss it here in order to have a
convenient reference with a uniform approach.

\subsubsection{Involutions}

Throughout this section, $Z$ will be a compact connected Riemann
surface, with a conformal involution $\al\:Z\to Z$, and an
anticonformal involution $\bt\:Z\to Z$ that commutes with $\al$.  We
also assume that $\al$ has only isolated fixed points (which means
that the total number of fixed points is finite).  We will write $\Dl$
for the open unit disc in $\C$.

\begin{lemma}\lbl{lem-invariant-metric}
 $Z$ admits a smooth Riemannian metric that is compatible with the
 conformal structure and is invariant under the action of $\al$ and
 $\bt$.
\end{lemma}
\begin{proof}
 Any coordinate patch clearly admits a smooth conformal Riemannian
 metric, and one can use a partition of unity to combine such local
 metrics to give a local metric, say $\mu_0$.  Any smooth automorphism
 of $Z$ acts in an evident way on the set of metrics, and that set is
 convex, so we can define
 \[ \mu = (\mu_0 + \al^*\mu_0 + \bt^*\mu_0 + \al^*\bt^*\mu_0)/4. \]
 This is the required invariant metric.
\end{proof}

For the rest of this section we will assume that an invariant metric
has been chosen.

\begin{remark}\lbl{rem-std-param}
 Let $C$ be a closed connected one-dimensional smooth submanifold of
 $Z$.  Then $C$ is necessarily diffeomorphic to the circle.  Now fix a
 point $a\in C$, and a unit vector $v\in T_aC$.  It is then standard
 that there is a unique smooth map $c_1\:\R\to C$ with $c_1(0)=a$ and
 $c_1'(0)=v$ and $\|c'_1(t)\|=1$ for all $t$.  One can check that
 $c_1$ induces a diffeomorphism $\R/\Z d\to C$ for some $d>0$.  We put
 $c(t)=c_1(td/2\pi)$, so $c$ induces a diffeomorphism
 $\R/2\pi\Z\to C$, which we call a \emph{standard parametrisation}
 of $C$.  It depends on the choice of $a$ and $v$, but if we make
 different choices then the new standard parametrisation will be of
 the form $t\mapsto c(p+t)$ or $t\mapsto c(p-t)$ for some constant $p$.
\end{remark}

\begin{remark}\lbl{rem-circle-involution}
 Let $c\:\R/2\pi\Z\to C$ be as in the previous remark, let
 $\gm\:Z\to Z$ be an involution that preserves the metric, and suppose
 that $\gm(C)=C$.  Then $\gm(c(t))$ must have the form $c(p+t)$ or
 $c(p-t)$ for some constant $p$ (which is well-defined modulo
 $2\pi$).
 \begin{itemize}
  \item[(a)] If $\gm(c(t))=c(p+t)$ then the equation $\gm^2=1$ gives
   $2p=0\pmod{2\pi}$, so we can take $p=0$ or $p=\pi$.  If $p=0$ then
   of course $\gm|_C=1$.  If $p=\pi$ then $\gm$ acts freely on $C$, so
   $C/\ip{\gm}$ is again a circle.
  \item[(b)] If $\gm(c(t))=c(p-t)$ then we can define a new standard
   parametrisation by $c^*(t)=c(p/2+t)$, and this satisfies
   $\gm(c^*(t))=c^*(-t)$.  It follows that the points $a=c^*(0)$ and
   $b=c^*(\pi)$ are fixed by $\gm$, but that $\gm$ acts freely on
   $C\sm\{a,b\}$.  If we put $P=c^*([0,\pi])$ and $Q=c^*([-\pi,0])$
   then the evident maps
   \[ [0,\pi] \xra{c^*} P \xra{} C/\ip{\gm}
       \xla{} Q \xla{c^*} [-\pi,0]
   \]
   are homeomorphisms.  Moreover, we have $P\cup Q=C$ and
   $P\cap Q=\{a,b\}$.
 \end{itemize}
\end{remark}

\begin{definition}\lbl{defn-local-parameter}
 Suppose that $a\in Z$, that $U_0$ is an open neighbourhood of $a$,
 and that $f_0\:U_0\to\C$ is a holomorphic map.  We say that $f_0$ is
 a \emph{centred local parameter} at $a$ if $f_0(a)=0$, and that
 $df_0$ generates the cotangent space to $Z$ at $a$.  We say that the
 pair $(U_0,f_0)$ is \emph{normalised} if $f_0$ gives a conformal
 isomorphism from $U_0$ to the open unit disc $\Dl=\{z\in\C\st |z|<1\}$.
\end{definition}

\begin{remark}\lbl{rem-shrinking}
 Let $f_0\:U_0\to\C$ be a centred local parameter that need not be
 normalised.  The holomorphic inverse function theorem then guarantees
 that there is a smaller open neighbourhood $U$ with $a\in U\sse U_0$
 and a number $\ep>0$ such that $f_0$ restricts to give a conformal
 isomorphism $U\to\{z\in\C\st |z|<\ep\}$.  This means that the map
 $f=\ep^{-1}f_0|_U\:U\to\Dl$ is a conformal isomorphism, so $(U,f)$ is
 normalised.  The operation that converts $(U_0,f_0)$ to $(U,f)$ will
 be called \emph{shrinking}.
\end{remark}

\begin{lemma}\lbl{lem-al-fixed}
 Suppose that $a\in Z$ with $\al(a)=a$.  Then there is a normalised
 local parameter $f\:U\to\Dl$ at $a$ such that $\al(U)=U$ and
 $f(\al(u))=-f(u)$ for all $u\in U$.
\end{lemma}
\begin{proof}
 Choose any centred local parameter $f_0\:U_0\to\C$.  Put
 $U_1=U_0\cap\al(U_0)$ and $f_1=f_0|_{U_1}$; this still gives a
 centred local parameter.  We can expand $f_1(\al(u))$ as a power
 series $\sum_{k=1}^\infty a_kf_1(u)^k$.  Because $\al^2=1$, we have
 $a_1=\pm 1$.  We claim that $a_1$ cannot be equal to $1$.  Indeed, as
 $\al$ has isolated fixed points, we cannot have $f_1(\al(u))=f_1(u)$
 as a power series.  Thus, if $a_1=1$ then there must exist $k>1$ such
 that $a_i=0$ for $1<i<k$ and $a_k\neq 0$, so
 \[ f_1(\al(u)) = f_1(u) + a_kf_1(u)^k + O(f_1(u)^{k+1}). \]
 If we substitute this into itself and use $\al^2=1$ we get $2a_k=0$,
 which is a contradiction.  We must therefore have $a_1=-1$.

 Now put $f_2(u)=(f_1(u)-f_1(\al(u)))/2$, so $f_2\:U_1\to\C$ is
 holomorphic with $f_2(a)=0$ and $f_2(\al(u))=-f_2(u)$.  We also have
 $f_2(u)=f_1(u)+O(f_1(u)^2)$, and thus that $f_2(u)$ is again a
 centred local parameter.  We can therefore produce the required pair
 $(U,f)$ by shrinking.
\end{proof}

\begin{lemma}\lbl{lem-al-not-fixed}
 Suppose that $a\in Z$ with $\al(a)\neq a$.  Then there is a
 normalised local parameter $f\:U\to \Dl$ at $a$ such that
 $\al(\ov{U})\cap\ov{U}=\emptyset$.
\end{lemma}
\begin{proof}
 By standard arguments with compact Hausdorff spaces, we can choose
 open neighbourhoods $V$ of $a$ and $W$ of $\al(a)$ such  that
 $\ov{V}\cap\ov{W}=\emptyset$.  Put $U_0=V\cap\al(W)$, so $U_0$ is an
 open neighbourhood of $a$ with
 $\ov{U_0}\cap\al(\ov{U_0})=\emptyset$.  Now let $f\:U\to \Dl$ be any
 normalised local parameter at $a$ with $U\sse U_0$.
\end{proof}

\begin{lemma}\lbl{lem-bt-fixed}
 Suppose that $a\in Z$ with $\bt(a)=a$.  Then there is a normalised
 local parameter $f\:U\to \Dl$ at $a$ with $\bt(U)=U$ and
 $f(\bt(u))=\ov{f(u)}$ for all $u\in U$.
\end{lemma}
\begin{proof}
 Choose any centred local parameter $f_0\:U_0\to\C$.  Put
 $U_1=U_0\cap\bt(U_0)$ and $f_1=f_0|_{U_1}$; this still gives a
 centred local parameter.  The map $u\mapsto\ov{f_1(\bt(u))}$ is
 another centred local parameter on $U_1$, so we have
 $\ov{f_1(\bt(u))}=c\,f_1(u)+O(f_1(u)^2)$ for some $c\neq 0$.  Using
 $\bt^2=1$ we find that $\ov{c}\,c=1$, so $c=e^{2i\tht}$ for some
 $\tht\in\R$.  Now put
 \[ f_2(u) = (e^{i\tht} f_1(u) + e^{-i\tht}\ov{f_1(\bt(u))})/2
       = e^{i\tht} f_1(u) + O(f_1(u)^2).
 \]
 This is again a centred local parameter at $a$, and it satisfies
 $f_2(\bt(u))=\ov{f_2(u)}$.  Shrinking now gives the required pair
 $(U,f)$.
\end{proof}

\begin{corollary}\lbl{cor-fixed-circles}
 The fixed set $Z^{\ip{\bt}}$ is a closed submanifold of $Z$, and so
 is diffeomorphic to a finite disjoint union of circles.  The same
 applies to $Z^{\ip{\al\bt}}$.
\end{corollary}
\begin{proof}
 Any normalised local parameter $f\:U\to \Dl$ as in the lemma gives a
 diffeomorphism
 \[ U\cap Z^{\ip{\bt}}\to \Dl\cap\R = (-1,1), \]
 and it follows easily from this that $Z^{\ip{\bt}}$ is a closed
 submanifold of $Z$.  As $\al\bt$ is an equally good example of an
 anticonformal involution, we see that $Z^{\ip{\al\bt}}$ is also a
 closed submanifold.
\end{proof}

\begin{remark}\lbl{rem-preserved-circle}
 As $\al$ commutes with $\bt$, it preserves the set $Z^{\ip{\bt}}$.
 However, if $Z^{\ip{\bt}}$ has several components, then they need not
 be preserved individually.  If a certain component is preserved, then
 Remark~\ref{rem-circle-involution} will apply.
\end{remark}

\begin{lemma}\lbl{lem-bt-al-fixed}
 Suppose that $a\in Z$ satisfies $\al(a)=\bt(a)=a$.  Then there is a
 normalised local parameter $f\:U\to \Dl$ at $a$ such that
 $f(\al(u))=-f(u)$ and $f(\bt(u))=\ov{f(u)}$ for all $u\in U$.
 Moreover:
 \begin{itemize}
  \item[(a)] $f$ induces a conformal isomorphism
   $g\:U/\ip{\al}\to \Dl$ with $g([u])=f(u)^2$.
  \item[(b)] $f$ restricts to give a diffeomorphism from
   $U^{\ip{\bt}}$ to the real axis in $\Dl$.
  \item[(c)] Similarly, $f$ restricts to give a diffeomorphism from
   $U^{\ip{\al\bt}}$ to the imaginary axis in $\Dl$.
  \item[(d)] $g$ restricts to give homeomorphisms
   $U^{\ip{\bt}}/\ip{\al}\to [0,1)$ and
   $U^{\ip{\al\bt}}/\ip{\al}\to(-1,0]$.
 \end{itemize}
\end{lemma}
\begin{proof}
 We choose a local parameter $f_0$ with $f_0(\al(u))=-f(u)$ as in
 Lemma~\ref{lem-al-fixed}.  We then take this as the initial choice in
 the proof of Lemma~\ref{lem-bt-fixed}.  This gives a normalised local
 parameter $f$ with $f(\bt(u))=\ov{f(u)}$, and by inspecting the
 construction we see that the property $f(\al(u))=-f(u)$ is retained
 as well.  The additional properties~(a) to~(d) follow easily.
\end{proof}

\subsubsection{Branched coverings}

First, we give a formal definition:
\begin{definition}
 We put $\Dl=\{z\in\C\st |z|<1\}$, and $\Dl'=\Dl\sm\{0\}$.  We let
 $\pi\:\Dl\tm\{0,1\}\to \Dl$ denote the projection, and we let
 $\sg\:\Dl\to \Dl$ denote the squaring map.

 Let $f\:X\to Y$ be a holomorphic map between Riemann surfaces.  We
 say that $f$ is a \emph{branched double covering} if for each
 $y\in Y$ there is a diagram of one of the following types:
 \[ \xymatrix{
  \Dl\tm\{0,1\} \ar[r]^q \ar[d]_{\pi} &
  X \ar[d]^f &&
  \Dl \ar[r]^q \ar[d]_{\sg} &
  X \ar[d]^f \\
  \Dl \ar[r]_p &
  Y &&
  \Dl \ar[r]_p &
  Y
 } \]
 where
 \begin{itemize}
  \item[(a)] $p$ is a holomorphic chart with $p(0)=y$.
  \item[(b)] The square is a pullback, so $q$ gives a holomorphic
   isomorphism from $\Dl\tm\{0,1\}$ or $\Dl$ to $f^{-1}(p(\Dl))$.
 \end{itemize}
\end{definition}

\begin{remark}
 We note that in the left hand case we have $|p^{-1}\{y\}|=2$, whereas
 in the right hand case we have $|p^{-1}\{y\}|=1$, so the two cases are
 disjoint.  In the right hand case we say that $y$ is a \emph{branch
  point}, and we write $B(f)$ for the set of branch points.  We note
 that all points in $p(\Dl)\sm\{y\}$ have two preimages and so are not
 branch points; thus, the set $B(f)$ is discrete.
\end{remark}

\begin{lemma}\lbl{lem-removable}
 Let $X$ and $Y$ be Riemann surfaces such that $Y$ is isomorphic to
 $\Dl$ or $\Dl\amalg \Dl$, and let $a$ be a point in $X$.  Then any
 holomorphic map $f\:X\sm\{a\}\to Y$ has a unique holomorphic
 extension $f\:X\to Y$.
\end{lemma}
\begin{proof}
 By choosing a chart around $a$ we can reduce to the case where $X=\Dl$
 and $a=0$.  Now, even if $Y\simeq \Dl\amalg \Dl$ the image $f(\Dl')$ will
 be connected and therefore contained in one of the two copies of
 $\Dl$.  We can thus assume that $Y=\Dl$.  This case is just the standard
 theorem on removable singularities in complex analysis.
\end{proof}

\begin{proposition}\lbl{prop-un-branched}
 Let $Y$ be a Riemann surface, let $V$ be a discrete subset of $Y$,
 and put $Y'=Y\sm V$.  Let $\CX$ be the category of branched double
 coverings $f\:X\to Y$ with $B(f)\sse V$, and let $\CX'$ be the
 category of unbranched double coverings of $Y'$.  (In both cases, the
 morphisms are holomorphic isomorphisms covering the identity on $Y$
 or $Y'$.)  Let $R\:\CX\to\CX'$ be the evident restriction functor,
 given by $R(X\xra{f}Y)=(X'\xra{f'}Y')$ where $X'=f^{-1}(Y')$ and
 $f'=f|_{X'}$.   Then $R$ is an equivalence of categories.
\end{proposition}
\begin{proof}
 Suppose we have objects $(X_i\xra{f_i}Y)\in\CX$ for $i=0,1$ and an
 isomorphism $g'\:X'_0\to X'_1$ with $f'_1g_1=f'_0$.  We claim that
 there is a unique holomorphic extension $g\:X_0\to X_1$, and that
 this satisfies $f_1g=f_0$.  This can be checked locally on $Y$, so we
 can restrict attention to a small neighbourhood of a point in $V$ and
 therefore assume that $(X_1\xra{f_1}Y)$ is either
 $(\Dl\tm\{0,1\}\xra{\pi}\Dl)$ or $(\Dl\xra{\sg}\Dl)$.  In either case it is
 clear from Lemma~\ref{lem-removable} that $g'$ has a unique
 holomorphic extension $g\:X_0\to X_1$.  We can also apply the
 uniqueness clause in the same lemma to the map $f'_1g'=f'_0$; this
 gives $f_1g=f_0$.  We now see that $R$ is full and faithful.

 We now need to show that $R$ is essentially surjective.  Consider an
 unbranched covering $f'\:X'\to Y'$.  For each $v\in V$ we can choose
 a chart $p_v\:\Dl\to Y$ with $p_v(0)=v$.  As $V$ is discrete we may
 assume, after shrinking the charts if necessary, that the sets
 $p_v(\Dl)$ are disjoint.  In particular, this means that
 $p_v(\Dl)\cap V=\{v\}$, so $p_v(\Dl')\sse Y'$.  The pullback $p_v^*(X')$
 is an unbranched double cover of $\Dl'$, so it is isomorphic to
 $(\Dl'\tm\{0,1\}\xra{\pi'}\Dl')$ or to $(\Dl'\xra{\sg'}\Dl')$.  In either
 case there is an evident way to extend $p_v^*(X')$ to give a branched
 cover of all of $\Dl$, and this in turn extends $X'$ to give a branched
 cover of $p_v(\Dl)$.  These extensions can be patched together to give
 a branched cover $(X\xra{f}Y)$ extending the original unbranched
 cover, as required.
\end{proof}

\begin{definition}\lbl{defn-monodromy}
 Let $(X\xra{f}Y)$ be a branched double cover with $B(f)\sse V$,
 giving an unbranched cover $(X'\xra{f'}Y')$.  For any closed loop
 $u\:[0,1]\to Y'$, we define the \emph{monodromy} $\mu_X(u)\in\Z/2$ as
 follows.  The fibre $F=(f')^{-1}\{u(0)\}$ will have precisely two
 elements.  For any element $a$, there is a unique continuous lift
 $\tu\:[0,1]\to X'$ with $f\tu=u$ and $u(0)=a$.  We put
 $\sg(a)=\tu(1)\in F$.  This defines a permutation $\sg\:F\to F$; we
 put $\mu_X(u)=0$ if $\sg$ is the identity, and $\mu_X(u)=1$ if $\sg$
 is the transposition.  This depends only on the homotopy class of the
 loop $u$.

 Next, for $v\in V$ we let $\om_v$ denote a small loop in $Y'$ that
 winds once around $v$ and does not wind around any of the other
 points in $V$.  It is clear that $v$ is a branch point for $X$ iff
 $\mu_X(\om_v)=1$.
\end{definition}

\begin{lemma}\lbl{lem-homology-monodromy}
 There is a unique homomorphism $\ov{\mu}_X\:H_1(Y')\to\Z/2$ such that
 $\mu_X(u)=\ov{\mu}_X([u])$ for all loops $u$, where $[u]$
 denotes the homology class represented by $u$.
\end{lemma}
\begin{proof}
 We can reduce to the case where $Y'$ is connected, and choose a
 basepoint $b\in Y'$.  The Hurewicz map then
 gives an isomorphism $h_b\:\pi_1(Y',b)_{\text{ab}}\to H_1(Y')$, and
 it follows that there is a unique homomorphism
 $\ov{\mu}\:H_1(Y')\to\Z/2$ with $\mu(u)=\ov{\mu}(\ip{u})$ for all
 loops $u$ based at $b$.  If $u$ is a loop that is not based at $b$,
 we can choose a path $w$ from $b$ to $u(0)$, and let $u'$ be the loop
 given by $w$ followed by $u$ followed by the reverse of $w$, so $u'$
 is based at $b$.  We then have $[u']=[u]$ and $\mu(u)=\mu(u')$,
 so we still have $\mu(u)=\ov{\mu}([u])$.
\end{proof}

\begin{lemma}\lbl{lem-covering-classification}
 The map
 \[ [X\xra{f}Y]\mapsto\ov{\mu}_X \]
 gives a bijection from the set of isomorphism classes in $\CX$ (or
 $\CX'$) to $\Hom(H_1(Y'),\Z/2)$.
\end{lemma}
\begin{proof}
 We can again reduce to the case where $Y'$ is connected, choose a
 basepoint $b\in Y'$, and use the Hurewicz isomorphism
 $H_1(Y')=\pi_1(Y',b)_{\text{ab}}$.  The claim is then that unbranched
 double covers of $Y'$ are classified by homomorphisms
 $\pi_1(Y',b)\to\Z/2$, which is standard covering theory.
\end{proof}

\begin{definition}
 Let $f\:X\to Y$ be a branched double covering.  We define
 $\chi\:X\to X$ by
 \[ \chi(x) = \begin{cases}
                x' & \text{ if } f^{-1}\{f(x)\} = \{x,x'\}
                      \text{ with } x'\neq x \\
                x  & \text{ if } f^{-1}\{f(x)\} = \{x\}.
              \end{cases}
 \]
 This is easily seen to be holomorphic.
\end{definition}

\begin{proposition}\lbl{prop-branched-unique}
 Suppose that $Y$ is isomorphic to $\C_\infty$, and that $V\subset Y$
 is a finite subset of even size.  Then:
 \begin{itemize}
  \item[(a)] There is a branched covering $f\:X\to Y$ for which
   $B(f)=V$.
  \item[(b)] If $f_0\:X_0\to Y$ and $f_1\:X_1\to Y$ are as in~(a),
   then there is an isomorphism $g\:X_0\to X_1$ with $f_1g=f_0$, and
   we have $\chi_1g=g\chi_0$.
  \item[(c)] If $g,h\:X_0\to X_1$ are as in~(b), then either $h=g$
   or $h=\chi_1g$.
 \end{itemize}
\end{proposition}
\begin{proof}
 Recall that for each $v\in V$ we have a loop $\om_v$ and a homology
 class $[\om_v]\in H_1(Y')$.  A standard calculation using the
 Mayer-Vietoris sequence shows that $H_1(Y')$ is generated by these
 classes subject only to the relation $\sum_{v\in V}[\om_v]=0$.  As
 $|V|$ is even it follows that there is a unique homomorphism
 $\nu\:H_1(Y')\to\Z/2$ with $\nu([\om_v])=1$ for all $v\in V$.  By
 Lemma~\ref{lem-covering-classification}, there exist unbranched
 coverings of $Y'$ with monodromy $\nu$, and any two such are
 isomorphic.  By Proposition~\ref{prop-un-branched}, it follows that
 there exist branched coverings of $Y$ whose branch set is precisely
 $V$, and any two such are isomorphic.  This proves~(a) and~(b) except
 for the fact that $\chi_1g=g\chi_0$.  This fact and claim~(c) are
 standard covering theory and are left to the reader.
\end{proof}

\subsection{The projective family is universal}
\lbl{sec-P-universal}

We will prove the following two theorems.

\begin{theorem}\lbl{thm-classify-precromulent}
 Let $X$ be a precromulent surface.  Then there is a unique number
 $a\in(0,1)$ such that $X\simeq PX(a)$ as $G$-equivariant Riemann
 surfaces.  Moreover, there are precisely two isomorphisms
 $X\to PX(a)$, which are related by the action of $\lm^2$.
\end{theorem}

\begin{theorem}\lbl{thm-classify-cromulent}
 Let $X$ be a cromulent surface.  Then there is a unique number
 $a\in(0,1)$ such that $X\simeq PX(a)$ as $G$-equivariant Riemann
 surfaces.  Moreover, there is precisely one cromulent isomorphism
 $X\to PX(a)$.
\end{theorem}

The proofs will be given after some preliminary results.  First,
however, we record a consequence of
Theorem~\ref{thm-classify-cromulent}:
\begin{corollary}
 Let $X$ be a cromulent surface.  Then $X$ admits a curve system (as
 in Definition~\ref{defn-curve-system}) and has standard isotropy (as
 in Definition~\ref{defn-std-isotropy}).
\end{corollary}
\begin{proof}
 Definition~\ref{defn-P-curves} and
 Proposition~\ref{prop-P-std-isotropy} show that this holds for
 $PX(a)$, which is sufficient by
 Theorem~\ref{thm-classify-cromulent}.
\end{proof}

\begin{proposition}\lbl{prop-v-zero}
 There is a unique point $v_0\in X$ such that $\lm(v_0)=v_0$ and
 $\lm_*=i\:T_{v_0}X\to T_{v_0}X$.  Similarly, there is a unique point
 $v_1=\mu(v_0)\in X$ such that $\lm(v_1)=v_1$ and
 $\lm_*=-i\:T_{v_1}X\to T_{v_1}X$.
\end{proposition}
\begin{proof}
 In $V^*$ there are precisely two points that are fixed by $\lm$, and
 they are exchanged by $\mu$.  The same must therefore be true in $X$.
 Let $a$ be one of these points, so the other one is $b=\mu(a)$.  Note
 that the holomorphic involution $\lm^2$ fixes $a$ and $b$, so it acts
 as $-1$ on $T_aX$ and $T_bX$ by Lemma~\ref{lem-al-fixed}, so $\lm$
 acts as $\pm i$.  Now consider the commutative diagram on the left
 below, and the resulting commutative diagram on the right:
 \[ \xymatrix{
     X \ar[r]^\mu \ar[d]_\lm &
     X \ar[d]^{\lm^{-1}} & &
     T_{a}X \ar[r]^{\mu_*} \ar[d]^{\lm_*} &
     T_{b}X \ar[d]^{\lm^{-1}_*} \\
     X \ar[r]_\mu & X && T_{a}X \ar[r]_{\mu_*} & T_{b}X.
    }
 \]
 From this we see that the eigenvalue of $\lm$ on $T_aX$ is the same
 as the eigenvalue of $\lm^{-1}$ on $T_bX$.  The claim follows easily
 from this.
\end{proof}

\begin{definition}\lbl{defn-e-points}
 We define points $e_0,e_1,e_\infty\in X/\ip{\lm^2}$ as follows.
 First, we let $v_0$ and $v_1$ be as in Proposition~\ref{prop-v-zero},
 and we put $e_0=[v_0]$ and $e_\infty=[v_1]$.  Next, we note that in
 $V^*$ there are precisely two points with stabiliser
 $\ip{\lm^2\mu,\lm^2\nu}$ (namely $3$ and $5$), and that these are
 exchanged by $\lm^2$.  It follows that the same is true in $V$.
 These points therefore form an equivalence class in $X/\ip{\lm^2}$,
 which we call $e_1$.
\end{definition}

We saw in Corollary~\ref{cor-quotient-types} that $X/\ip{\lm^2}$ is
isomorphic to $\C_\infty$.  It is well-known that the conformal
automorphisms of $\C_\infty$ are the M\"obius transformations, and
that these act freely and transitively on the triples of distinct
points in $\C_\infty$.  This validates the following definition:
\begin{definition}\lbl{defn-p}
 We let $p\:X/\ip{\lm^2}\to\C_\infty$ denote the unique conformal
 isomorphism such that $p(e_i)=i$ for $i\in\{0,1,\infty\}$.  We will
 also use the symbol $p$ for the composite
 $X\to X/\ip{\lm^2}\xra{p}\C_\infty$.
\end{definition}

\begin{lemma}\lbl{lem-p-equivariance}
 If we let $G$ act on $\C_\infty$ by
 \[ \lm(z)=-z \hspace{5em}
    \mu(z)=1/z \hspace{5em}
    \nu(z)=\ov{z}
 \]
 then the map $p$ is equivariant.
\end{lemma}
\begin{proof}
 The maps $x\mapsto -p(\lm(x))$ and $x\mapsto 1/p(\mu(x))$ and
 $x\mapsto \ov{p(\nu(x))}$ all have the defining property of $p$.
\end{proof}

\begin{lemma}\lbl{lem-v-eleven}
 There is a unique point $v_{11}\in V$ such that the number
 $a=p(v_{11})$ lies in $(0,1)$.  Moreover, if we define
 $v_{10}=\lm(v_{11})$ and $v_{12}=\lm\mu(v_{11})$ and
 $v_{13}=\mu(v_{11})$, then we have
 \[ p(v_{10}) = -a   \hspace{4em}
    p(v_{11}) =  a   \hspace{4em}
    p(v_{12}) = -1/a \hspace{4em}
    p(v_{13}) = 1/a.
 \]
\end{lemma}
\begin{proof}
 Put $W=\{x\in X\st\stab_G(x)=\ip{\lm^2,\nu}\}$.  Because the group
 $\ip{\lm^2,\nu}$ is normal in $G$, this is a $G$-set.  As $X$ is
 precromulent, it is equivariantly isomorphic to the $G$-set
 \[ W^* = \{i\in V^*\st \stab_G(i)=\ip{\lm^2,\nu}\} =
     \{10,11,12,13\} \simeq G/\ip{\lm^2,\nu}.
 \]
 As $p$ gives an equivariant isomorphism $X/\ip{\lm^2}\to\C_\infty$,
 it must restrict to give an equivariant injection
 \[ W/\ip{\lm^2} \to
     \{z\in\C_\infty\st\stab_G(z)=\ip{\lm^2,\nu}\}.
 \]
 The domain here is just $W$, and the codomain is the set
 \[ U = \R_\infty\sm\{0,1,-1,\infty\} =
     (-\infty,-1) \amalg
     (-1, 0)      \amalg
     ( 0, 1)      \amalg
     (1,\infty).
 \]
 The action of $G$ on $\C_\infty$ permutes the four components of $U$
 transitively, so the preimage under $p$ of each component must
 contain precisely one point of $W$.  The claim is clear from this.
\end{proof}

\begin{proof}[Proof of Theorem~\ref{thm-classify-precromulent}]
 The map $p\:X\to\C_\infty$ is a branched covering, with branch set
 $p(U)$, where $U=\{x\in X\st\lm^2(x)=x\}$.  This is equivariantly
 isomorphic to the $G$-set
 \[ U^*=\{i\in V^*\st\lm^2(i)=i\}=\{0,1,10,11,12,13\}, \]
 and using this we see that $p(U)=\{0,\infty,\pm a,\pm 1/a\}$.  This
 is the same as the branch set for the map $p\:PX(a)\to\C_\infty$
 defined in Remark~\ref{rem-P-quotient}.  Our claim now follows from
 Proposition~\ref{prop-branched-unique}.
\end{proof}

\begin{corollary}\lbl{cor-precromulent-aut}\leavevmode
 \begin{itemize}
  \item[(a)] If $X$ is a precromulent surface, then the group of
   precromulent automorphisms of $X$ is $C_2=\{1,\lm^2\}$.
  \item[(b)] If $X$ and $Y$ are isomorphic precromulent surfaces, then
   there are precisely two precromulent isomorphisms between them.  If
   one of them is $\phi$, then the other is $\phi\lm_X^2=\lm_Y^2\phi$.
 \end{itemize}
\end{corollary}
\begin{proof}
 Clear from the theorem.
\end{proof}

\begin{corollary}\lbl{cor-cromulent-aut}
 If $X$ is a cromulent surface, then the only cromulent automorphism
 is the identity.
\end{corollary}
\begin{proof}
 Any cromulent automorphism is also a precromulent automorphism, and
 so is $1$ or $\lm^2$; but $\lm^2$ does not preserve the labelling.
\end{proof}

\begin{corollary}\lbl{cor-cromulent-iso}
 Let $X$ and $Y$ be isomorphic cromulent surfaces; then there is a
 unique cromulent isomorphism between them. \qed
\end{corollary}

\begin{proposition}\lbl{prop-labellings}
 Let $X$ be a precromulent surface.  Then $X$ has precisely two
 cromulent labellings, which are related by the action of $\lm^2$.
\end{proposition}
\begin{proof}
 By Theorem~\ref{thm-classify-precromulent} we can reduce to the
 case $X=PX(a)$.  In particular, this means that we have a curve
 system, and nets as in Section~\ref{sec-fundamental}.

 The labelling given in Definition~\ref{defn-P} is cromulent, by
 Proposition~\ref{prop-P-fundamental}.  It follows easily that if we
 change the labelling by $\lm^2$, then it remains cromulent.

 Now suppose we have another cromulent labelling, say
 $(v'_i)_{i=0}^{13}$.  This must have the form $v'_i=v_{\phi(i)}$ for
 some $\phi$ in the group $\Aut(V^*)$, which is described by
 Proposition~\ref{prop-aut-V}.  Proposition~\ref{prop-v-zero} shows
 that we must have $\phi(0)=0$ and $\phi(1)=1$.
 Proposition~\ref{prop-aut-V} shows that $\phi(2)\in\{2,4\}$, and
 after replacing $\phi$ by $\phi\lm^2$ if necessary, we may assume
 that $\phi(2)=2$.  Assuming this, we see from
 Proposition~\ref{prop-aut-V} that $\phi(i)=i$ for $i\in\{2,3,4,5\}$,
 and that $\phi(6)\in\{6,8\}$.

 Next, as the new labelling is cromulent, there must be a connected
 component $F'\sse\{x\in X\st\stab_G(x)=1\}$ whose closure contains
 the set $U=\{v_{\phi(0)},v_{\phi(3)},v_{\phi(6)},v_{\phi(11)}\}$.
 From the discussion in Section~\ref{sec-fundamental} we see that $F'$
 must be $\gm(PF'_{16}(a))$ for some $\gm\in G$.  We have seen that
 $U$ contains $v_0$, $v_3$ and either $v_6$ or $v_8$.  By inspecting
 the nets in Section~\ref{sec-fundamental}, we see that this can only
 be consistent if $\gm=1$ and $\phi(6)=6$ and $\phi(11)=11$.  By
 consulting Proposition~\ref{prop-aut-V} again, we conclude that
 $\phi=1$, as required.
\end{proof}

\begin{proof}[Proof of Theorem~\ref{thm-classify-cromulent}]
 Let $X$ be a cromulent surface.  By
 Theorem~\ref{thm-classify-precromulent}, there is a unique $a$ such
 that $X\simeq PX(a)$ as precromulent surfaces.  Choose a precromulent
 isomorphism $\phi\:X\to PX(a)$.  The points $\phi^{-1}(v_i)$ give a
 cromulent labelling of $X$.  By Proposition~\ref{prop-labellings},
 this must either be the given labelling of $X$, or its twist by
 $\lm^2$.  Thus, after replacing $\phi$ by $\phi\lm^2$ if necessary,
 we may assume that $\phi$ is a cromulent isomorphism.  It is unique
 by Corollary~\ref{cor-cromulent-iso}.
\end{proof}

\begin{remark}\lbl{rem-p-extra}
 We now see that the map $p\:X/\ip{\lm^2}\to\C_\infty$ from
 Definition~\ref{defn-p} must factor as a cromulent isomorphism
 $X\to PX(a)$, followed by the map $p\:PX(a)\to\C_\infty$ from
 Remark~\ref{rem-P-quotient}.  This implies that we have the
 following additional properties:
 \begin{align*}
  p(v_0) &= 0 \\
  p(v_1) &= \infty \\
  p(v_2) = p(v_4) &=    -1 \\
  p(v_3) = p(v_5) &= \pp 1 \\
  p(v_6) = p(v_8) &= \pp i \\
  p(v_7) = p(v_9) &=    -i \\
  p(v_{10}) &=      -a \\
  p(v_{11}) &= \pp   a \\
  p(v_{12}) &=    -1/a \\
  p(v_{13}) &= \pp 1/a.
 \end{align*}
\end{remark}
The number $p(v_i)$ is recorded in Maples as \mcode+v_C[i]+.

\begin{remark}\lbl{rem-p-hat}
 For some purposes it is more convenient to work with the round sphere
 $S^2\subset\R^3$ rather than the Riemann sphere $\C_\infty$.  We
 identify them using the stereographic projection map
 \[ \xi(x+iy) =
     \left(
      \frac{2x}{x^2+y^2+1},
      \frac{2y}{x^2+y^2+1},
      \frac{x^2+y^2-1}{x^2+y^2+1}
     \right).
 \]
 This has been normalised so that the unit circle in $\C$ is sent to
 the equator.  It is a standard fact that $\xi$ is conformal.  The
 resulting complex structure on the tangent spaces
 $T_yS^2=\{t\in\R^3\st t.y=0\}$ can be described in terms of the cross
 product of vectors: we have $(a+ib)t=at+bt\tm y$.  This means that an
 ordered basis $(u,v)$ for $T_yS^2$ is oriented iff $\det(y,v,u)>0$.
 Note also that if $\xi(z)=y$ then
 \begin{align*}
  \xi(-z)     &= (  - y_1,\; - y_2,\; \pp y_3) \\
  \xi(1/z)    &= (\pp y_1,\; - y_2,\;   - y_3) \\
  \xi(\ov{z}) &= (\pp y_1,\; - y_2,\; \pp y_3).
 \end{align*}
 It follows that the composite $\hp=\xi p\:X\to S^2$ has properties as
 follows:
 \begin{align*}
  \hp(v_0) &= (\pp 0,\pp 0,   -1) \\
  \hp(v_1) &= (\pp 0,\pp 0,\pp 1) \\
  \hp(v_2) &= \hp(v_4) = (  - 1,\pp 0,\pp 0) \\
  \hp(v_3) &= \hp(v_5) = (\pp 1,\pp 0,\pp 0) \\
  \hp(v_6) &= \hp(v_8) = (\pp 0,\pp 1,\pp 0) \\
  \hp(v_7) &= \hp(v_9) = (\pp 0,  - 1,\pp 0) \\
  \hp(v_{10}) &= (  - 2a,0,a^2-1)/(a^2+1) \\
  \hp(v_{11}) &= (\pp 2a,0,a^2-1)/(a^2+1) \\
  \hp(v_{12}) &= (  - 2a,0,1-a^2)/(a^2+1) \\
  \hp(v_{13}) &= (\pp 2a,0,1-a^2)/(a^2+1)
 \end{align*}

 \begin{align*}
  \hp_1(\lm(x)) &=   -\hp_1(x) & \hp_2(\lm(x)) &=   -\hp_2(x) & \hp_3(\lm(x)) &= \pp\hp_3(x) \\
  \hp_1(\mu(x)) &= \pp\hp_1(x) & \hp_2(\mu(x)) &=   -\hp_2(x) & \hp_3(\mu(x)) &=   -\hp_3(x) \\
  \hp_1(\nu(x)) &= \pp\hp_1(x) & \hp_2(\nu(x)) &=   -\hp_2(x) & \hp_3(\nu(x)) &= \pp\hp_3(x).
 \end{align*}
 The points $\hp(v_i)$ are recorded in Maple as \mcode+v_S2[i]+, and
 the induced action of $g\in G$ on $u\in S^2$ is \mcode+act_S2[g](u)+.
\end{remark}

\section{The hyperbolic family}
\lbl{sec-H}

\subsection{The groups \texorpdfstring{$\Pi$ and $\tPi$}{Pi and Pi tilde}}
\lbl{sec-Pi}

Later we will construct a family of cromulent surfaces as quotients of
the unit disc by different actions of a certain group $\Pi$.  In this
section we define and study $\Pi$, together with a larger group $\tPi$
such that $\tPi/\Pi=G$.  Actions of $\tPi$ and $\Pi$ will be given in
the following section.

\begin{definition}\lbl{defn-Pi}
 Let $\Pi$ be the abstract group generated by symbols $\bt_k$ (for
 $k\in\Z/8$) subject to the following relations:
 \begin{align*}
  \bt_{k+4} &= \bt_k^{-1} \\
  \bt_0\bt_1\bt_2\bt_3\bt_4\bt_5\bt_6\bt_7 &= 1.
 \end{align*}
\end{definition}
We saw in Proposition~\ref{prop-pi-one} that any cromulent surface has
fundamental group isomorphic to $\Pi$.  Note also that the
abelianization $\Pi_{\text{ab}}$ is freely generated (as an abelian
group) by $\bt_0,\dotsc,\bt_3$, and so is isomorphic to $\Z^4$.

We can introduce an alternative set of generators as follows:
\begin{align*}
 \al_0 &= \bt_3^{-1}\bt_2^{-1}\bt_0 &
 \al_1 &= \bt_1\bt_2\bt_3 &
 \al_2 &= \bt_2^{-1} &
 \al_3 &= \bt_3^{-1} \\
 \bt_0 &= \al_2^{-1}\al_3^{-1}\al_0 &
 \bt_1 &= \al_1\al_3\al_2 &
 \bt_2 &= \al_2^{-1} &
 \bt_3 &= \al_3^{-1}.
\end{align*}
In terms of these, the relation $\bt_0\dotsb\bt_7=1$ becomes the
standard relation
\[ [\al_0,\al_1][\al_2,\al_3] =
    \al_0\al_1\al_0^{-1}\al_1^{-1}
    \al_2\al_3\al_2^{-1}\al_3^{-1} = 1
\]
for the fundamental group of a surface of genus $2$.  (However, we
prefer to use the generators $\bt_k$, for reasons of symmetry.)
\begin{checks}
 hyperbolic/Pi_check.mpl: check_Pi_alpha()
\end{checks}

\begin{remark}\lbl{rem-runs}
 Note that the relation $\bt_0\dotsb\bt_7=1$ can be conjugated to give
 \[ \bt_k\bt_{k+1}\dotsb\bt_{k+7} = 1 \]
 for any $k\in\Z/8$.  This in turn gives
 \begin{align*}
  \bt_k\bt_{k+1}\bt_{k+2}\bt_{k+3}\bt_{k+4}\bt_{k+5}\bt_{k+6} &=
   \bt_{k+7}^{-1} &&\hspace{-1em}=
   \bt_{k+3} \\
  \bt_k\bt_{k+1}\bt_{k+2}\bt_{k+3}\bt_{k+4}\bt_{k+5} &=
   (\bt_{k+6}\bt_{k+7})^{-1} &&\hspace{-1em}=
   \bt_{k+3}\bt_{k+2} \\
  \bt_k\bt_{k+1}\bt_{k+2}\bt_{k+3}\bt_{k+4} &=
   (\bt_{k+5}\bt_{k+6}\bt_{k+7})^{-1} &&\hspace{-1em}=
   \bt_{k+3}\bt_{k+2}\bt_{k+1} \\
  \bt_k\bt_{k+1}\bt_{k+2}\bt_{k+3} &=
   (\bt_{k+4}\bt_{k+5}\bt_{k+6}\bt_{k+7})^{-1} &&\hspace{-1em}=
   \bt_{k+3}\bt_{k+2}\bt_{k+1}\bt_k \\
  \bt_k\bt_{k+1}\bt_{k+2} &=
   (\bt_{k+3}\bt_{k+4}\bt_{k+5}\bt_{k+6}\bt_{k+7})^{-1} &&\hspace{-1em}=
   \bt_{k+3}\bt_{k+2}\bt_{k+1}\bt_k\bt_{k-1} \\
  \bt_k\bt_{k+1} &=
   (\bt_{k+2}\bt_{k+3}\bt_{k+4}\bt_{k+5}\bt_{k+6}\bt_{k+7})^{-1} &&\hspace{-1em}=
   \bt_{k+3}\bt_{k+2}\bt_{k+1}\bt_k\bt_{k-1}\bt_{k-2} \\
  \bt_k &=
   (\bt_{k+1}\bt_{k+2}\bt_{k+3}\bt_{k+4}\bt_{k+5}\bt_{k+6}\bt_{k+7})^{-1} &&\hspace{-1em}=
   \bt_{k+3}\bt_{k+2}\bt_{k+1}\bt_k\bt_{k-1}\bt_{k-2}\bt_{k-3}.
 \end{align*}
 Thus:
 \begin{itemize}
  \item Any increasing or decreasing run of length at least $5$ can be
   replaced by a strictly shorter run in the opposite direction.
  \item Any run of length $4$ can be reversed.
 \end{itemize}
\end{remark}

\begin{definition}\lbl{defn-beta-reduced}
 A word in the letters $\{\bt_i\st i\in\Z/8\}$ is \emph{reduced} if
 \begin{itemize}
  \item[(a)] There are no subwords of the form $\bt_i\bt_{i+4}$
  \item[(b)] There are no subwords of the form
   $\bt_i\bt_{i+1}\bt_{i+2}\bt_{i+3}\bt_{i+4}$.
  \item[(c)] There are no subwords of the form
   $\bt_i\bt_{i-1}\bt_{i-2}\bt_{i-3}$.
 \end{itemize}
\end{definition}

\begin{proposition}\lbl{prop-beta-reduced}
 Every element $\pi\in\Pi$ can be expressed in a unique way as a reduced
 word in the letters $\bt_i$.  Moreover, if $\pi$ is represented by a
 word $w$, then the corresponding reduced word $w'$ can be obtained
 from $w$ by repeatedly cancelling pairs of the form $\bt_i\bt_{i+4}$,
 and shortening or reversing runs as in Remark~\ref{rem-runs}.
\end{proposition}
\begin{proof}
 A straightforward argument by induction on the length shows that for
 every word there is a reduced word (of the same length or less) that
 represents the same element of $\Pi$.  Uniqueness is less obvious,
 but follows from Dehn's algorithm and small cancellation theory.  In
 more detail, put $X=\{\bt_0,\bt_1,\bt_2,\bt_3\}$, then put
 \begin{align*}
  \sg_i^+ &= \bt_i\bt_{i+1}\dotsb\bt_{i+7} \\
  \sg_i^- &= \bt_i\bt_{i-1}\dotsb\bt_{i-7},
 \end{align*}
 where each $\bt_j$ is replaced by an element of $X\amalg X^{-1}$ in
 the obvious way.  Then the set
 \[ S=\{\sg_i^+\st i\in\Z/8\} \amalg \{\sg_i^-\st i\in\Z/8\} \]
 consists of reduced words in the free group $FX$, and is closed under
 taking inverses and cyclic permutations.  If $\al$ and $\bt$ are
 distinct elements of $S$, then they have length $8$, but they share at
 most one initial letter.  The shared fraction of $1/8$ is less than
 $1/6$, so the main theorem of~\cite{gr:daw} is applicable, and the
 conclusion follows easily.  (For a textbook treatment, see Chapter~V
 of~\cite{lysc:cgt}.)
\end{proof}
\begin{remark}
 Elements of $\Pi$ are represented in Maple as lists of integers.  For
 example, the list \mcode+[5,3,6]+ represents $\bt_5\bt_3\bt_6$.  The
 function \mcode+is_Pi_reduced(L)+ decides whether a list \mcode+L+
 corresponds to a reduced word.  The function \mcode+Pi_reduce(L)+
 finds the unique reduced word that represents the same element as
 \mcode+L+.  Here \mcode+L+ is allowed to have arbitrary integer
 entries, but the first step in the reduction process is to reduce
 them modulo $8$ so that they lie in $\{0,1,\dotsc,7\}$.
 Multiplication and inversion are implemented by the functions
 \mcode+Pi_mult+ and \mcode+Pi_inv+.  All of this (together with
 various other things) is in the file \fname+hyperbolic/Pi.mpl+.
\end{remark}

\begin{corollary}\lbl{cor-trivial-centre}
 The centre of $\Pi$ is trivial, so the conjugation map
 $\Pi\to\Aut(\Pi)$ is injective.
\end{corollary}
\begin{proof}
 Any nontrivial element $\pi\in\Pi\sm\{1\}$ can be represented by a
 nonempty reduced word $w$.  Let $i$ and $j$ be the
 indices of the first and last letters in $w$, and put
 \[ L = \{i-1,i,i+1,i+4,j-1,j+1,j+4\} \]
 (so $|L|\leq 7$).  Choose $k\in(\Z/8)\sm L$, so that $\bt_kw$ and
 $w\bt_k$ are reduced words.  Their first letters are $\bt_k$ and
 $\bt_i$, so they are different.  It follows that $\pi$ is
 not central.
\end{proof}

\begin{proposition}\lbl{prop-Aut-Pi}
 The group $\Pi$ has automorphisms $\lm_*$, $\mu_*$ and $\nu_*$, which
 act on the generators $\bt_i$ as follows:
 \[ \begin{array}{|c|c|c|c|c|c|c|c|c|} \hline
     & \bt_0 & \bt_1 & \bt_2 & \bt_3 & \bt_4 & \bt_5 & \bt_6 & \bt_7
     \\ \hline
     \lm_* &
     \bt_2 & \bt_3 & \bt_4 & \bt_5 & \bt_6 & \bt_7 & \bt_0 & \bt_1
     \\ \hline
     \mu_* &
     \bt_2\bt_0\bt_1 &
     \bt_5\bt_4\bt_3 &
     \bt_0\bt_7\bt_6 &
     \bt_2\bt_3\bt_1 &
     \bt_5\bt_4\bt_6 &
     \bt_7\bt_0\bt_1 &
     \bt_2\bt_3\bt_4 &
     \bt_5\bt_7\bt_6 \\ \hline
     \nu_* &
     \bt_0 &
     \bt_2\bt_1\bt_2 &
     \bt_6 &
     \bt_0\bt_7\bt_0 &
     \bt_4 &
     \bt_6\bt_5\bt_6 &
     \bt_2 &
     \bt_4\bt_3\bt_4 \\ \hline
    \end{array}
 \]
 Moreover, we have $\lm_*^4=\mu_*^2=\nu_*^2=(\lm_*\nu_*)^2=1$, and the
 automorphisms $(\lm_*\mu_*)^2$ and $(\nu_*\mu_*)^2$ are inner.  More precisely, for
 any $\pi\in\Pi$ we have
 \begin{align*}
  (\lm_*\mu_*)^2(\pi) &=
       (\bt_7\bt_6)\;\pi\;(\bt_7\bt_6)^{-1} \\
  (\nu_*\mu_*)^2(\pi) &=
       (\bt_6\bt_0\bt_7\bt_6)\;\pi\;(\bt_6\bt_0\bt_7\bt_6)^{-1}.
 \end{align*}
\end{proposition}
\begin{proof}
 We will only discuss $\mu_*$; similar arguments cover the other two
 cases.  Put $B=\{\bt_k\:k\in\Z/8\}\subset\Pi$ and define
 $\mu_*\:B\to\Pi$ by the above table.  To show that this extends to an
 endomorphism of $\Pi$, we must check that it respects the relations,
 or in other words that
 \begin{align*}
  \mu_*(\bt_k)\mu_*(\bt_{k+4}) &= 1 \\
  \mu_*(\bt_0)\mu_*(\bt_1)\mu_*(\bt_2)\mu_*(\bt_3)
  \mu_*(\bt_4)\mu_*(\bt_5)\mu_*(\bt_6)\mu_*(\bt_7) &= 1.
 \end{align*}
 The first of these follows easily by inspecting the definition of
 $\mu_*(\bt_k)$ and using the relation $\bt_j\bt_{j+4}=1$ three
 times.  For the second, the left hand side can be grouped as
 \[
     \bt_2(\bt_0\bt_1
     \bt_5\bt_4)\bt_3
     (\bt_0(\bt_7\bt_6
     \bt_2\bt_3)(\bt_1
     \bt_5)\bt_4)(\bt_6
     \bt_7\bt_0\bt_1
     \bt_2\bt_3\bt_4
     \bt_5)\bt_7\bt_6.
 \]
 Working from the inside out, we see that the content of each matched
 pair of parentheses cancels down to $1$.  This leaves
 $\bt_2\bt_3\bt_7\bt_6$, which cancels further down to $1$.  We thus
 have an endomorphism $\mu_*$ as claimed.  It satisfies
 \[ \mu_*^2(\bt_0) = \mu_*(\bt_2\bt_0\bt_1) =
    \mu_*(\bt_2)\mu_*(\bt_0)\mu_*(\bt_1) =
    \bt_0(\bt_7(\bt_6\;\bt_2)(\bt_0(\bt_1\;\bt_5)\bt_4)\bt_3) =
    \bt_0.
 \]
 By similar arguments, it also satisfies $\mu_*^2(\bt_i)=\bt_i$ for
 all other indices $i$, so $\mu_*^2=1$.  In particular, this means
 that $\mu_*$ is an automorphism.  The identities
 $\lm_*^4=\nu_*^2=(\lm_*\nu_*)^2=1$ can be verified in a similar way.
 Moreover, it will also suffice to check the identities for
 $(\lm_*\mu_*)^2(\pi)$ and $(\nu_*\mu_*)^2(\pi)$ when $\pi=\bt_k$ for
 some $k$, and this is again straightforward (but long).
 \begin{checks}
  hyperbolic/Pi_check.mpl: check_Pi_relations()
 \end{checks}
\end{proof}

\begin{definition}\lbl{defn-tPi}
 We let $\tPi$ be the abstract group generated by symbols
 $\lm,\mu,\nu,\bt_k$ (for $k\in\Z/8$) subject to the following relations:
 \begin{align*}
  \bt_{k+4} &= \bt_k^{-1} \\
  \bt_0\bt_1\bt_2\bt_3\bt_4\bt_5\bt_6\bt_7 &= 1 \\
  \lm^4 &= \mu^2 = \nu^2 = (\lm\nu)^2 = 1 \\
  (\lm\mu)^2 &= \bt_7\bt_6 \\
  (\nu\mu)^2 &= \bt_6\bt_0\bt_7\bt_6 \\
  \lm\bt_k\lm^{-1} &= \lm_*(\bt_k) \\
  \mu\bt_k\mu &= \mu_*(\bt_k) \\
  \nu\bt_k\nu &= \nu_*(\bt_k).
 \end{align*}
 (Here $\lm_*(\bt_k)$, $\mu_*(\bt_k)$ and $\nu_*(\bt_k)$ refer to the
 words given in the table in Proposition~\ref{prop-Aut-Pi}.)
\end{definition}

\begin{proposition}\lbl{prop-tPi}
 $\Pi$ can be identified with the subgroup of $\tPi$ generated by
 $\bt_0,\dotsc,\bt_7$.  Moreover, this subgroup is normal, and there
 is a canonical isomorphism $\tPi/\Pi\simeq G$.
\end{proposition}
\begin{proof}
 Let $\Pi'$ be the subgroup of $\tPi$ generated by
 $\bt_0,\dotsc,\bt_7$.  Using the defining relations for
 $\lm\bt_k\lm^{-1}$, $\mu\bt_k\mu$ and $\nu\bt_k\nu$ we see that
 $\Pi'$ is normal.  If we just set all the elements $\bt_k$ to the
 identity in our presentation for $\tPi$, we obtain a presentation for
 $\tPi/\Pi'$; from this it is clear that $\tPi/\Pi'=G$.  There is an
 evident surjective homomorphism $\phi\:\Pi\to\Pi'$, sending $\bt_k$
 to $\bt_k$ for all $k$.  All that is left is to show that this is
 injective.  For this, we let $\gm\:\Pi\to\Inn(\Pi)$ be the usual
 conjugation map, given by $\gm(\al)(\pi)=\al\pi\al^{-1}$.  This is
 surjective by the definition of $\Inn(\Pi)$, and injective by
 Corollary~\ref{cor-trivial-centre}, so it is an isomorphism.  Next,
 we define
 \begin{align*}
  \dl(\lm)   &= \lm_* \in \Aut(\Pi) \\
  \dl(\mu)   &= \mu_* \in \Aut(\Pi) \\
  \dl(\nu)   &= \nu_* \in \Aut(\Pi) \\
  \dl(\bt_k) &= \gm(\bt_k) \in \Inn(\Pi) \lhd \Aut(\Pi).
 \end{align*}
 Using Proposition~\ref{prop-Aut-Pi} we see that this is compatible
 with the defining relations for $\tPi$, so it extends to give a
 homomorphism $\tPi\to\Aut(\Pi)$.  We can now chase the diagram
 \[ \xymatrix{
  \Pi \ar[d]_\gm^\simeq \ar[r]^\phi &
  \tPi \ar[d]_{\dl} \ar@{->>}[r] &
  G \ar[d] \\
  \Inn(\Pi) \ar@{ >->}[r]_{\text{inc}} &
  \Aut(\Pi) \ar@{->>}[r] & \Out(\Pi)
 } \]
 to see that $\phi$ is injective as required.
\end{proof}
\begin{remark}
 Elements of $\tPi$ are represented as pairs \mcode+[T,L]+, with
 \mcode+T+ in $G$ and \mcode+L+ in $\Pi$.  The multiplication rule is
 not obvious; it is implemented by the function \mcode+Pi_tilde_mult+,
 which uses data stored in the table \mcode+G_Pi_cocycle+.  Inversion
 is implemented by the function \mcode+Pi_tilde_inv+.  All this is in
 \fname+hyperbolic/Pi.mpl+.
\end{remark}

We next give an alternate presentation of $\Pi$, which will be helpful
when we want to analyse actions on the unit disc.
\begin{proposition}\lbl{prop-Pi-Sg}
 $\Pi$ is generated by the elements
 \begin{align*}
  \sg_a &= \bt_5\bt_6 &
  \sg_b &= \bt_2      &
  \sg_c &= \bt_7\bt_0 \\
  \sg_d &= \bt_4      &
  \sg_e &= \bt_1\bt_2 &
  \sg_f &= \bt_3\bt_4
 \end{align*}
 subject only to the relations
 \[ \sg_a\sg_c\sg_e\sg_f =
    \sg_b\sg_e^{-1}\sg_b^{-1}\sg_a^{-1} =
    \sg_d\sg_f^{-1}\sg_d^{-1}\sg_c^{-1} = 1.
 \]
\end{proposition}
\begin{proof}
 First, it is straightforward to check that the stated relations hold
 in $\Pi$.  Thus, if we let $\Sg$ denote the abstract group with the
 indicated generators and relations, we have a homomorphism
 $\phi\:\Sg\to\Pi$ given by $\phi(\sg_t)=\sg_t$ for all
 $t\in\{a,b,c,d,e,f\}$.  Next, we define elements
 $\bt_0,\dotsc,\bt_7\in\Sg$ by
 \[ \bt_0 = \sg_d^{-1}      \hspace{3em}
    \bt_1 = \sg_e\sg_b^{-1} \hspace{3em}
    \bt_2 = \sg_b           \hspace{3em}
    \bt_3 = \sg_f\sg_d^{-1} \hspace{3em}
    \bt_{i+4}=\bt_i^{-1}.
 \]
 We claim that $\bt_0\bt_1\dotsb\bt_7=1$ in $\Sg$.  It will suffice to
 prove the conjugate relation
 $\bt_5\bt_6\bt_7\bt_0\bt_1\bt_2\bt_3\bt_4=1$.  We can write out the
 left hand side and group the terms as
 \[ (\sg_b\sg_e^{-1}\sg_b^{-1})(\sg_d\sg_f^{-1}\sg_d^{-1})
    \sg_e(\sg_b^{-1}\sg_b)\sg_f(\sg_d^{-1}\sg_d).
 \]
 The relation $\sg_b\sg_e^{-1}\sg_b^{-1}\sg_a^{-1}=1$ converts the
 first parenthesised term to $\sg_a$, and similarly the second becomes
 $\sg_c$.  The third parenthesised term is clearly the identity, as is
 the fourth.  This leaves $\sg_a\sg_c\sg_e\sg_f$, which is the
 identity by the first defining relation for $\Sg$.  This means that
 we have a homomorphism $\psi\:\Pi\to\Sg$ given by $\psi(\bt_i)=\bt_i$
 for all $i$.  Straightforward calculations now show that
 $\phi\psi(\bt_i)=\bt_i$ for all $i$, and $\psi\phi(\sg_t)=\sg_t$ for
 all $t$, so $\phi\psi=1_\Pi$ and $\psi\phi=1_\Sg$.
 \begin{checks}
  hyperbolic/Pi_check.mpl: check_Pi_sigma
 \end{checks}
\end{proof}

\subsection{Cromulent actions}
\lbl{sec-H-action}

We next describe a family of actions of $\tPi$ on the unit disc
\[ \Dl = \{z\in\C\st |z|<1\}. \]
This has a standard Riemannian metric as follows:
\[ ds^2 = 4|dz|^2 / (1-|z|^2)^2. \]
We recall some standard facts about this, most of which can be found
in~\cite[Section 4.1]{an:hg}, for example.  The Gaussian curvature of
the metric is equal to $-1$.  The conformal automorphisms of $\Dl$
have the form
\[ z \mapsto \lm \frac{z-\al}{1-\ov{\al}z}, \]
with $|\al|<1$ and $|\lm|=1$, and that these all preserve the
metric.  Similarly, the anticonformal automorphisms have the form
\[ z \mapsto \lm \frac{\ov{z}-\al}{1-\ov{\al}\,\ov{z}}. \]
The geodesic distance function is
\[ \dhyp(z,w) =
     2\arctanh\left|\frac{z-w}{1-\ov{z}w}\right|.
\]
\begin{checks}
 hyperbolic/HX_check.mpl: check_hyperbolic_metric()
\end{checks}

Now fix a number $b$ with $0<b<1$, and put $b_+=\sqrt{1+b^2}$ and
$b_-=\sqrt{1-b^2}$.  We define an action of $\tPi$ on the
unit disc $\Dl=\{z\in\C\st |z|<1\}$ as follows:
\begin{align*}
 \lm(z) &= iz &
 \bt_0(z) &= \frac{b_+z+1}{z+b_+} \\
 \mu(z) &= \frac{b_+z-b^2-i}{(b^2-i)z-b_+} &
 \bt_1(z) &= \frac{b_+^3z-(2+i)b^2-i}{((i-2)b^2+i)z+b_+^3} \\
 \nu(z) &= \ov{z} &
 \bt_{2n}(z) &= i^n\bt_0(z/i^n) \\
 &&
 \bt_{2n+1}(z) &= i^n\bt_1(z/i^n).
\end{align*}
One can check by direct calculation that the defining relations for
$\tPi$ are satisfied.
\begin{checks}
 hyperbolic/HX_check.mpl: check_Pi_action()
\end{checks}
\begin{remark}
 The parameters $b$, $b_+$ and $b_-$ are \mcode+a_H+, \mcode+ap_H+ and
 \mcode+am_H+ in the Maple code.  The action of $g\in\tPi$ on $z\in\Dl$
 is given by \mcode+act_Pi_tilde(g,z)+.  The group $G$ can be
 considered as a subset of $\tPi$ which is not a subgroup.  For
 $g\in G$, the alternative notation \mcode+act_H[g](z)+ also works for
 \mcode+act_Pi_tilde(g,z)+.  This notation is potentially misleading,
 because we do not actually have an action of $G$.
\end{remark}
\begin{remark}\lbl{rem-a-Ho}
 The evident analogue of Remark~\ref{rem-a-Po} applies to \mcode+a_H+.
 Most parts of the code treat \mcode+a_H+ as a symbol, but there are
 also global variables \mcode+a_H0+ and \mcode+a_H1+ which holds
 numerical values for \mcode+a_H+.  These should be set using the
 function \mcode+set_a_H0+, which is defined in
 \fname+hyperbolic/HX0.mpl+.  There is also a function
 \mcode+simplify_H+ analogous to the function \mcode+simplify_P+
 mentioned in Remark~\ref{rem-simplify-P}: it applies some
 simplification rules like $\sqrt{1-b^4}=b_+b_-$ which are not always
 used by Maple.
\end{remark}

We now put $HX(b)=\Dl/\Pi$.  Later we will give this the structure of a
cromulent surface.  As a first step, we note that $\Pi$ is normal in
$\tPi$, so there is an induced action of $G=\tPi/\Pi$ on $HX(b)$.

\begin{definition}\lbl{defn-v-H}
 We define points $v_0,\dotsc,v_{13}\in \Dl$ as follows:
 \begin{align*}
  v_0 &= 0 &
   v_6 &= \frac{1+i}{\rt}\;\frac{\rt-b_-}{b_+} &
    v_{10} &= i(b_+-b) \\
  v_1 &= \frac{1+i}{2}b_+ &
   v_7 &= i \, v_6 &
    v_{11} &= b_+-b \\
  v_2 &= \frac{b\,b_--b_+}{i-b^2} &
   v_8 &= -v_6 &
    v_{12} &= (b+b_+)\frac{i+(i+2)b^2}{(b+b_+)^2+b^2} \\
  v_3 &= \frac{b\,b_--b_+}{ib^2-1} &
   v_9 &= -i\, v_6 &
    v_{13} &= i\,\ov{v_{12}} \\
  v_4 &= i \, v_3 \\
  v_5 &= -i \, v_2.
 \end{align*}
\end{definition}

\begin{definition}\lbl{defn-v-H-extra}
 We also consider additional points in the same $\Pi$-orbits:
 \[ \begin{array}{rll}
  v_{ 1.1} &= \pp i\,v_1     &= \bt_2\bt_1(v_1) \\
  v_{ 1.2} &= -      v_1     &= \bt_1\bt_2\bt_3\bt_4(v_1) \\
  v_{ 1.3} &= -i\,   v_1     &= \bt_3\bt_4(v_1) \\
  v_{ 2.1} &= \pp\ov{v_2   } &= \bt_6(v_2) \\
  v_{ 3.1} &= -  \ov{v_3   } &= \bt_4(v_3) \\
  v_{ 4.1} &= \pp\ov{v_4   } &= \bt_6(v_4) \\
  v_{ 5.1} &= -  \ov{v_5   } &= \bt_4(v_5) \\
  v_{10.1} &= \pp\ov{v_{10}} &= \bt_6(v_{10}) \\
  v_{11.1} &= -  \ov{v_{11}} &= \bt_4(v_{11}) \\
  v_{12.1} &= -      v_{12}  &= \bt_1(v_{12}) \\
  v_{12.2} &= -  \ov{v_{12}} &= \bt_2\bt_1(v_{12}) \\
  v_{12.3} &= \pp\ov{v_{12}} &= \bt_6(v_{12}) \\
  v_{13.1} &= -      v_{13}  &= \bt_4\bt_3\bt_4(v_{13}) \\
  v_{13.2} &= \pp\ov{v_{13}} &= \bt_3\bt_4(v_{13}) \\
  v_{13.3} &= -  \ov{v_{13}} &= \bt_4(v_{13}).
 \end{array} \]
\end{definition}
\begin{remark}
 The points $v_i$ are represented as \mcode+v_H[i]+, and there are also
 points \mcode+v_H0[i]+ and \mcode+v_H1[i]+ that are obtained by
 substituting the numerical values \mcode+a_H0+ or \mcode+a_H1+ for
 the symbol \mcode+a_H+.  The entries in the above table have the form
 $v_i=\gm_i(v_j)$ where $i$ is not an integer, $j$ is the integer part
 of $i$, and $\gm_i$ is an element of $\Pi$.  The element $\gm_i$
 is represented in maple as \mcode+v_H_fraction_offset[i]+.
\end{remark}

One can check that $\lm$, $\mu$ and $\nu$ act as follows:
\begin{align*}
 \lm(v_{ 0}) &= v_{ 0} &
 \mu(v_{ 0}) &= \bt_2(\bt_3(\bt_4(v_{ 1}))) &
 \nu(v_{ 0}) &= v_{ 0} \\
 \lm(v_{ 1}) &= \bt_2(\bt_1(v_{ 1})) &
 \mu(v_{ 1}) &= \bt_2(v_{ 0}) &
 \nu(v_{ 1}) &= \bt_3(\bt_4(v_{ 1})) \\
 \lm(v_{ 2}) &= \bt_4(v_{ 3}) &
 \mu(v_{ 2}) &= v_{ 2} &
 \nu(v_{ 2}) &= \bt_6(v_{ 2}) \\
 \lm(v_{ 3}) &= v_{ 4} &
 \mu(v_{ 3}) &= \bt_2(v_{ 5}) &
 \nu(v_{ 3}) &= v_{ 5} \\
 \lm(v_{ 4}) &= \bt_4(v_{ 5}) &
 \mu(v_{ 4}) &= \bt_5(\bt_6(v_{ 4})) &
 \nu(v_{ 4}) &= \bt_6(v_{ 4}) \\
 \lm(v_{ 5}) &= v_{ 2} &
 \mu(v_{ 5}) &= \bt_2(\bt_3(\bt_4(v_{ 3}))) &
 \nu(v_{ 5}) &= v_{ 3} \\
 \lm(v_{ 6}) &= v_{ 7} &
 \mu(v_{ 6}) &= \bt_2(v_{ 9}) &
 \nu(v_{ 6}) &= v_{ 9} \\
 \lm(v_{ 7}) &= v_{ 8} &
 \mu(v_{ 7}) &= \bt_5(v_{ 8}) &
 \nu(v_{ 7}) &= v_{ 8} \\
 \lm(v_{ 8}) &= v_{ 9} &
 \mu(v_{ 8}) &= \bt_5(\bt_4(\bt_3(v_{ 7}))) &
 \nu(v_{ 8}) &= v_{ 7} \\
 \lm(v_{ 9}) &= v_{ 6} &
 \mu(v_{ 9}) &= \bt_2(\bt_3(\bt_4(v_{ 6}))) &
 \nu(v_{ 9}) &= v_{ 6} \\
 \lm(v_{10}) &= \bt_4(v_{11}) &
 \mu(v_{10}) &= v_{12} &
 \nu(v_{10}) &= \bt_6(v_{10}) \\
 \lm(v_{11}) &= v_{10} &
 \mu(v_{11}) &= \bt_2(\bt_3(\bt_4(v_{13}))) &
 \nu(v_{11}) &= v_{11} \\
 \lm(v_{12}) &= \bt_4(v_{13}) &
 \mu(v_{12}) &= v_{10} &
 \nu(v_{12}) &= \bt_6(v_{12}) \\
 \lm(v_{13}) &= \bt_2(\bt_1(v_{12})) &
 \mu(v_{13}) &= \bt_2(v_{11}) &
 \nu(v_{13}) &= \bt_3(\bt_4(v_{13})) \\
\end{align*}
This shows that $G$ permutes the corresponding points in $HX(b)$ in
accordance with Definition~\ref{defn-precromulent}.  The elements of
$\Pi$ appearing here are recorded in the table
\mcode+v_action_witness_H+, which is defined in the file
\fname+hyperbolic/HX.mpl+.
\begin{checks}
 cromulent.mpl: check_precromulent("H")
 hyperbolic/HX_check.mpl: check_v_H()
\end{checks}

\subsection{The curve system}
\lbl{sec-H-curves}

We would now like to construct curves $C_0,\dotsc,C_8$ in $\Dl$ or
$HX(b)$.  These will be the fixed sets of certain anticonformal
involutions of $\Dl$, or the images in $HX(b)$ of those fixed sets.
Such involutions can be classified as follows:
\begin{itemize}
 \item[(a)] Suppose that $m\in\C$ with $|m|>1$.  Put
  \[ \xi_m(z) = \frac{m\,\ov{z}-1}{\ov{z}-\ov{m}}. \]
  This is an anticonformal involution on $\C_\infty$ that preserves
  $\Dl$.  Maple notation (defined in the file \fname+hyperbolic/HX.mpl+) is
  \mcode+xi(m,z)+.  We put
  \[ \Xi_m = \{z\in \Dl \st \xi_m(z)=z\}. \]
  If we put $d=\sqrt{|m|^2-1}$, then the fixed set of $\xi_m$ in $\C$
  is the circle of radius $d$ centred at $m$, and $\Xi_m$ is the
  intersection of this circle with $\Dl$.  This is a geodesic in $\Dl$,
  and the following formula gives an isometric parametrisation
  $\om_m\:\R\to\Xi_m$:
  \[ \om_m(s) =
      \left(\frac{id-1}{\ov{m}}\right)
      \frac{(id+1) - i |m| e^s}{i |m| e^s + (id-1)}.
  \]
  Maple notation is \mcode+xi_curve(m,s)+.
  The endpoints of $\Xi_m$ on the unit circle are
  $(1\pm id)/\ov{m}$.  If we have another involution of the same type
  with parameter $m'$, then $\xi_m\xi_{m'}=\xi_{m'}\xi_m$ iff $\Xi_m$
  and $\Xi_{m'}$ cross at right angles, iff $\text{Re}(\ov{m}m')=1$.
 \item[(b)] Suppose instead that $u\in\C$ with $|u|=1$.  We then have
  an anticonformal involution $z\mapsto u\ov{z}$.  The fixed set in
  $\Dl$ is the straight line joining the two square roots of $u$.  This
  is isometrically parameterised by the map
  \[ s\mapsto \sqrt{u}\frac{e^s - 1}{e^s+1} =
      \sqrt{u}\tanh(s/2).
  \]
 \item[(c)] Every anticonformal involution on $\Dl$ arises in one of the
  above two ways.
\end{itemize}
\begin{checks}
 hyperbolic/HX_check.mpl: check_xi()
\end{checks}

\begin{definition}\lbl{defn-H-curves}
 We define constants $s_0,\dotsc,s_4$ as follows:
 \begin{align*}
  s_0 &= 2\log\left(\frac{\rt b}{b_+-b_-}\right) &
  s_1 &= \frac{1}{2}\log\left(\frac{\rt+b_+}{\rt-b_+}\right) &
  s_2 &= \log\left(\frac{1+b}{b_-}\right) \\
  s_3 &= \frac{1}{2}\log\left(\frac{b+b_++1}{b+b_+-1}\right) &
  s_4 &= \frac{1}{4}\log\left(\frac{b_+^2+2b_++2}{b_+^2-2b_++2}\right).
 \end{align*}
 We then define maps $\tc_k\:\R\to \Dl$ for $0\leq k\leq 8$ as follows:
 \begin{align*}
  \tc_0(t) &= \om_{(1+i)/b_+}((t/\pi-1/4)s_0) \\
  \tc_1(t) &= e^{i\pi/4}\tanh(t\,s_1/\pi) &
  \tc_2(t) &= e^{3i\pi/4}\tanh(t\,s_1/\pi) \\
  \tc_3(t) &= \om_{b_+}(-t\,s_2/\pi) &
  \tc_4(t) &= \om_{ib_+}(-t\,s_2/\pi) \\
  \tc_5(t) &= \tanh(t\,s_3/\pi) &
  \tc_6(t) &= i\,\tanh(t\,s_3/\pi) \\
  \tc_7(t) &= \om_{ib_+/2+1/b_+}( t\,s_3/\pi - s_4) &
  \tc_8(t) &= \om_{ b_+/2+i/b_+}(-t\,s_3/\pi + s_4).
 \end{align*}
 We write $c_k$ for the composite
 \[ \R \xra{\tc_k} \Dl \xra{} \Dl/\Pi = HX(b). \]
 We also put $\tC_k=\tc_k(\R)\subset \Dl$ and $C_k=c_k(\R)\subset \Dl/\Pi$.
\end{definition}
\begin{remark}
 The function $\tc_k(t)$ is \mcode+c_H[k](t)+, and the constant $s_j$
 is \mcode+s_H[j]+.  For $k\in\{0,3,4,7,8\}$ we see that $\tc_k(\R)$ is
 a circular arc.  The centre is recorded as \mcode+c_H_p[k]+, and the
 radius as \mcode+c_H_r[k]+.
\end{remark}

The factors $s_k$ are chosen to make the maps $c_k$ periodic with
period $2\pi$.  In more detail:
\begin{proposition}\lbl{prop-c-H-cycle}
 For all $t\in\R$ we have
 \begin{align*}
  \tc_0(t+2\pi) &= \bt_0\bt_2\bt_4\bt_6(\tc_0(t)) \\
  \tc_1(t+2\pi) &= \bt_0\bt_7\bt_6\bt_5(\tc_1(t)) &
  \tc_2(t+2\pi) &= \bt_2\bt_1\bt_0\bt_7(\tc_2(t)) \\
  \tc_3(t+2\pi) &= \bt_0\bt_7(\tc_3(t)) &
  \tc_4(t+2\pi) &= \bt_2\bt_1(\tc_4(t)) \\
  \tc_5(t+2\pi) &= \bt_0(\tc_5(t)) &
  \tc_6(t+2\pi) &= \bt_2(\tc_6(t)) \\
  \tc_7(t+2\pi) &= \bt_0\bt_2\bt_1(\tc_7(t)) &
  \tc_8(t+2\pi) &= \bt_2\bt_3\bt_4(\tc_8(t)).
 \end{align*}
 Thus, for $0\leq k\leq 8$ we have $c_k(t+2\pi)=c_k(t)$.
\end{proposition}
(The group element for $\tc_k(t)$ is recorded as \mcode+c_H_cycle[k]+.)
\begin{proof}
 Computer calculation.
 \begin{checks}
  hyperbolic/HX_check.mpl: check_c_H_monodromy()
 \end{checks}
\end{proof}

For the equivariance properties of a curve system, it will suffice to check that
the equations below hold in $\Dl$, and this can be done by direct calculation.
\begin{align*}
 \lm(\tc_0(t)) &= \bt_4(\tc_0(t+\pi/2)) &
 \mu(\tc_0(t)) &= \tc_0(-t) &
 \nu(\tc_0(t)) &= \bt_6(\tc_0(-t)) \\
 \lm(\tc_1(t)) &= \tc_2(t) &
 \mu(\tc_1(t)) &= \bt_5\bt_4\bt_3(\tc_2(\pi+t)) &
 \nu(\tc_1(t)) &= \tc_2(-t) \\
 \lm(\tc_2(t)) &= \tc_1(-t) &
 \mu(\tc_2(t)) &= \bt_2\bt_3\bt_4(\tc_1(\pi+t)) &
 \nu(\tc_2(t)) &= \tc_1(-t) \\
 \lm(\tc_3(t)) &= \tc_4(t) &
 \mu(\tc_3(t)) &= \bt_2\bt_3\bt_4(\tc_3(\pi+t)) &
 \nu(\tc_3(t)) &= \tc_3(-t) \\
 \lm(\tc_4(t)) &= \bt_4(\tc_3(-t)) &
 \mu(\tc_4(t)) &= \tc_4(-t-\pi) &
 \nu(\tc_4(t)) &= \bt_6(\tc_4(t)) \\
 \lm(\tc_5(t)) &= \tc_6(t) &
 \mu(\tc_5(t)) &= \bt_2\bt_3\bt_4(\tc_7(t)) &
 \nu(\tc_5(t)) &= \tc_5(t) \\
 \lm(\tc_6(t)) &= \tc_5(-t) &
 \mu(\tc_6(t)) &= \bt_2\bt_3\bt_4(\tc_8(-t)) &
 \nu(\tc_6(t)) &= \tc_6(-t) \\
 \lm(\tc_7(t)) &= \bt_2\bt_1(\tc_8(t)) &
 \mu(\tc_7(t)) &= \bt_2(\tc_5(t)) &
 \nu(\tc_7(t)) &= \bt_3\bt_4(\tc_7(t)) \\
 \lm(\tc_8(t)) &= \bt_2\bt_1(\tc_7(-t)) &
 \mu(\tc_8(t)) &= \bt_2(\tc_6(-t)) &
 \nu(\tc_8(t)) &= \bt_3\bt_4(\tc_8(-t))
\end{align*}
(The elements of $\Pi$ appearing here are recorded in the table
\mcode+c_action_witness_H+.)
\begin{checks}
 cromulent.mpl: check_precromulent("H")
\end{checks}

In the case $b=0.75$, these curves and vertices can be illustrated as follows:
\begin{center}
 \begin{tikzpicture}[scale=5]
  \draw (0,0) circle(1);
  \draw[blue] (-1,0) -- (1,0);
  \draw[blue] (0,-1) -- (0,1);
  \draw[green] (-0.707,-0.707) -- ( 0.707, 0.707);
  \draw[green] (-0.707, 0.707) -- ( 0.707,-0.707);
  \draw[cyan] (0.800,0.800) (0.294,0.956) arc(-197:-73:0.529);
  \draw[cyan,dotted] (-0.800,0.800) (-0.956,0.294) arc(-107:17:0.529);
  \draw[cyan,dotted] (-0.800,-0.800) (-0.294,-0.956) arc(-17:107:0.529);
  \draw[cyan,dotted] (0.800,-0.800) (0.956,-0.294) arc(73:197:0.529);
  \draw[magenta] (1.250,0.000) (0.800,0.600) arc(127:233:0.750);
  \draw[magenta,dotted] (0.000,1.250) (-0.600,0.800) arc(-143:-37:0.750);
  \draw[magenta,dotted] (-1.250,0.000) (-0.800,-0.600) arc(-53:53:0.750);
  \draw[magenta,dotted] (0.000,-1.250) (0.600,-0.800) arc(37:143:0.750);
  \draw[magenta] (0.000,1.250) (-0.600,0.800) arc(-143:-37:0.750);
  \draw[magenta,dotted] (-1.250,0.000) (-0.800,-0.600) arc(-53:53:0.750);
  \draw[magenta,dotted] (0.000,-1.250) (0.600,-0.800) arc(37:143:0.750);
  \draw[magenta,dotted] (1.250,0.000) (0.800,0.600) arc(127:233:0.750);
  \draw[blue] (0.800,0.625) (0.670,0.742) arc(-222:-62:0.175);
  \draw[blue,dotted] (-0.625,0.800) (-0.742,0.670) arc(-132:28:0.175);
  \draw[blue,dotted] (-0.800,-0.625) (-0.670,-0.742) arc(-42:118:0.175);
  \draw[blue,dotted] (0.625,-0.800) (0.742,-0.670) arc(48:208:0.175);
  \draw[blue] (0.625,0.800) (0.471,0.882) arc(-208:-48:0.175);
  \draw[blue,dotted] (-0.800,0.625) (-0.882,0.471) arc(-118:42:0.175);
  \draw[blue,dotted] (-0.625,-0.800) (-0.471,-0.882) arc(-28:132:0.175);
  \draw[blue,dotted] (0.800,-0.625) (0.882,-0.471) arc(62:222:0.175);
  \fill[black](0.000,0.000) circle(0.007);
  \fill[black](0.625,0.625) circle(0.007);
  \fill[black](0.322,0.573) circle(0.007);
  \fill[black](-0.573,0.322) circle(0.007);
  \fill[black](-0.322,-0.573) circle(0.007);
  \fill[black](0.573,-0.322) circle(0.007);
  \fill[black](0.426,0.426) circle(0.007);
  \fill[black](-0.426,0.426) circle(0.007);
  \fill[black](-0.426,-0.426) circle(0.007);
  \fill[black](0.426,-0.426) circle(0.007);
  \fill[black](0.000,0.500) circle(0.007);
  \fill[black](0.500,0.000) circle(0.007);
  \fill[black](-0.500,0.000) circle(0.007);
  \fill[black](-0.493,0.685) circle(0.007);
  \fill[black](0.685,0.493) circle(0.007);
  \fill[black](-0.625,-0.625) circle(0.007);
  \fill[black](-0.322,0.573) circle(0.007);
  \fill[black](-0.625,0.625) circle(0.007);
  \fill[black](0.000,-0.500) circle(0.007);
  \fill[black](0.493,0.685) circle(0.007);
  \fill[black](0.625,-0.625) circle(0.007);
  \fill[black](-0.573,-0.322) circle(0.007);
  \fill[black](0.685,-0.493) circle(0.007);
  \fill[black](0.493,-0.685) circle(0.007);
  \fill[black](-0.685,-0.493) circle(0.007);
  \fill[black](0.322,-0.573) circle(0.007);
  \fill[black](-0.493,-0.685) circle(0.007);
  \fill[black](-0.685,0.493) circle(0.007);
  \fill[black](0.573,0.322) circle(0.007);
  \draw( 0.30, 0.25) node{$c_1$};
  \draw(-0.30, 0.25) node{$c_2$};
  \draw( 0.80,-0.03) node{$c_5$};
  \draw(-0.05, 0.80) node{$c_6$};
  \draw( 0.773, 0.246) node{$c_{0}$};
  \draw( 0.486, 0.179) node{$c_{3}$};
  \draw(-0.229, 0.566) node{$c_{4}$};
  \draw( 0.771, 0.422) node{$c_{7}$};
  \draw( 0.422, 0.771) node{$c_{8}$};
  \draw( 0.060, 0.020) node{$v_{0}$};
  \draw( 0.655, 0.595) node{$v_{1}$};
  \draw( 0.272, 0.583) node{$v_{2}$};
  \draw(-0.623, 0.322) node{$v_{3.1}$};
  \draw(-0.372,-0.573) node{$v_{4.1}$};
  \draw( 0.623,-0.322) node{$v_{5}$};
  \draw( 0.386, 0.426) node{$v_{6}$};
  \draw(-0.386, 0.426) node{$v_{7}$};
  \draw(-0.396,-0.456) node{$v_{8}$};
  \draw( 0.396,-0.456) node{$v_{9}$};
  \draw(-0.040, 0.470) node{$v_{10}$};
  \draw( 0.470,-0.030) node{$v_{11}$};
  \draw(-0.440,-0.030) node{$v_{11.1}$};
  \draw( 0.443, 0.685) node{$v_{12}$};
  \draw( 0.635, 0.493) node{$v_{13}$};
  \draw(-0.655,-0.595) node{$v_{1.2}$};
  \draw(-0.382, 0.573) node{$v_{4}$};
  \draw(-0.675, 0.595) node{$v_{1.1}$};
  \draw(-0.060,-0.530) node{$v_{10.1}$};
  \draw(-0.433, 0.685) node{$v_{12.2}$};
  \draw( 0.655,-0.595) node{$v_{1.3}$};
  \draw(-0.623,-0.322) node{$v_{5.1}$};
  \draw( 0.745,-0.493) node{$v_{13.2}$};
  \draw(-0.433,-0.685) node{$v_{12.1}$};
  \draw(-0.745,-0.493) node{$v_{13.1}$};
  \draw( 0.382,-0.563) node{$v_{2.1}$};
  \draw( 0.433,-0.685) node{$v_{12.3}$};
  \draw(-0.745, 0.493) node{$v_{13.3}$};
  \draw( 0.523, 0.312) node{$v_{3}$};
 \end{tikzpicture}
\end{center}
The dotted curves are images of the undotted curves under the action
of various elements of $\Pi$.  Combinatorially, the whole picture is
essentially the same as the first net described in
Section~\ref{sec-fundamental}.

The following diagrams show how the picture varies as $b$ varies from
$0$ to $1$.
\begin{center}
 \begin{tikzpicture}[scale=3]
  \begin{scope}
   \draw (0,0) (1,0) arc(0:90:1);
   \draw[blue] (0,1) -- (0,0) -- (1,0);
   \draw[green] (0,0) -- ( 0.707, 0.707);
   \draw[cyan] (0.980,0.980) (0.021,1.000) arc(-181:-89:0.960);
   \draw[blue] (0.980,0.510) (0.606,0.795) arc(-217:-88:0.470);
   \draw[blue] (0.510,0.980) (0.041,0.999) arc(-182:-53:0.470);
   \draw[magenta](1.020,0) (0.817,0) arc(180:101:0.203);
   \draw[magenta](0,1.020) (0,0.817) arc(270:349:0.203);
   \draw (0.5,-0.1) node {$b=0.20$};
  \end{scope}
  \begin{scope}[xshift=1.3cm]
   \draw (0,0) (1,0) arc(0:90:1);
   \draw[blue] (0,1) -- (0,0) -- (1,0);
   \draw[green] (0,0) -- ( 0.707, 0.707);
   \draw[cyan] (0.900,0.900) (0.118,0.993) arc(-187:-83:0.787);
   \draw[blue] (0.900,0.556) (0.633,0.774) arc(-219:-77:0.344);
   \draw[blue] (0.556,0.900) (0.220,0.976) arc(-193:-51:0.344);
   \draw[magenta](1.111,0) (0.627,0) arc(180:116:0.484);
   \draw[magenta](0,1.111) (0,0.627) arc(270:334:0.484);
   \draw (0.5,-0.1) node {$b=0.50$};
  \end{scope}
  \begin{scope}[xshift=2.6cm]
   \draw (0,0) (1,0) arc(0:90:1);
   \draw[blue] (0,1) -- (0,0) -- (1,0);
   \draw[green] (0,0) -- ( 0.707, 0.707);
   \draw[cyan] (0.800,0.800) (0.294,0.956) arc(-197:-73:0.529);
   \draw[blue] (0.800,0.625) (0.670,0.742) arc(-222:-62:0.175);
   \draw[blue] (0.625,0.800) (0.471,0.882) arc(-208:-48:0.175);
   \draw[magenta](1.250,0) (0.500,0) arc(180:127:0.750);
   \draw[magenta](0,1.250) (0,0.500) arc(270:323:0.750);
   \draw (0.5,-0.1) node {$b=0.75$};
  \end{scope}
  \begin{scope}[xshift=3.9cm]
   \draw (0,0) (1,0) arc(0:90:1);
   \draw[blue] (0,1) -- (0,0) -- (1,0);
   \draw[green] (0,0) -- ( 0.707, 0.707);
   \draw[cyan] (0.720,0.720) (0.561,0.828) arc(-214:-56:0.192);
   \draw[blue] (0.720,0.694) (0.702,0.712) arc(-225:-47:0.026);
   \draw[blue] (0.694,0.720) (0.676,0.737) arc(-223:-45:0.026);
   \draw[magenta](1.389,0) (0.425,0) arc(180:134:0.964);
   \draw[magenta](0,1.389) (0,0.425) arc(270:316:0.964);
   \draw (0.5,-0.1) node {$b=0.95$};
  \end{scope}
 \end{tikzpicture}
\end{center}
In particular, we see that there is no change in the combinatorial
structure.

We write $HF_1(b)$ for the following region.
\begin{center}
 \begin{tikzpicture}[scale=5]
  \draw[magenta]( 1.250, 0.000) +(139:0.750) arc(139:221:0.750);
  \draw[blue]   ( 0.800, 0.625) +(180:0.175) arc(180:229:0.175);
  \draw[blue]   ( 0.625, 0.800) +(221:0.175) arc(221:270:0.175);
  \draw[magenta]( 0.000, 1.250) +(229:0.750) arc(229:311:0.750);
  \draw[blue]   (-0.625, 0.800) +(270:0.175) arc(270:319:0.175);
  \draw[blue]   (-0.800, 0.625) +(311:0.175) arc(311:360:0.175);
  \draw[magenta](-1.250, 0.000) +(-41:0.750) arc(-41: 41:0.750);
  \draw[blue]   (-0.800,-0.625) +(  0:0.175) arc(  0: 49:0.175);
  \draw[blue]   (-0.625,-0.800) +( 41:0.175) arc( 41: 90:0.175);
  \draw[magenta]( 0.000,-1.250) +( 49:0.750) arc( 49:131:0.750);
  \draw[blue]   ( 0.625,-0.800) +( 90:0.175) arc( 90:139:0.175);
  \draw[blue]   ( 0.800,-0.625) +(131:0.175) arc(131:180:0.175);
  \draw[magenta,->] ( 0.500, 0.000) -- ( 0.500,-0.001);
  \draw[magenta,->] (-0.500, 0.000) -- (-0.500,-0.001);
  \draw[magenta,->] ( 0.000, 0.500) -- (-0.001, 0.500);
  \draw[magenta,->] ( 0.000,-0.500) -- (-0.001,-0.500);
  \draw[blue,->] ( 0.800, 0.625) +(205:0.175) -- +(206:0.175);
  \draw[blue,->] ( 0.625, 0.800) +(246:0.175) -- +(245:0.175);
  \draw[blue,->] (-0.625, 0.800) +(295:0.175) -- +(296:0.175);
  \draw[blue,->] (-0.800, 0.625) +(335:0.175) -- +(334:0.175);
  \draw[blue,->] (-0.800,-0.625) +( 25:0.175) -- +( 26:0.175);
  \draw[blue,->] (-0.625,-0.800) +( 65:0.175) -- +( 64:0.175);
  \draw[blue,->] ( 0.625,-0.800) +(115:0.175) -- +(116:0.175);
  \draw[blue,->] ( 0.800,-0.625) +(155:0.175) -- +(154:0.175);
  \fill( 0.685, 0.493) circle(0.01);
  \fill( 0.625, 0.625) circle(0.01);
  \fill( 0.493, 0.685) circle(0.01);
  \fill(-0.493, 0.685) circle(0.01);
  \fill(-0.625, 0.625) circle(0.01);
  \fill(-0.685, 0.493) circle(0.01);
  \fill(-0.685,-0.493) circle(0.01);
  \fill(-0.625,-0.625) circle(0.01);
  \fill(-0.493,-0.685) circle(0.01);
  \fill( 0.493,-0.685) circle(0.01);
  \fill( 0.625,-0.625) circle(0.01);
  \fill( 0.685,-0.493) circle(0.01);
  \draw( 0.685, 0.493) node[anchor=west]{$v_{13}$};
  \draw( 0.625, 0.625) node[anchor=south west]{$v_{1}$};
  \draw( 0.493, 0.685) node[anchor=south]{$v_{12}$};
  \draw(-0.493, 0.685) node[anchor=south]{$v_{12.2}$};
  \draw(-0.625, 0.625) node[anchor=south east]{$v_{1.1}$};
  \draw(-0.685, 0.493) node[anchor=east]{$v_{13.3}$};
  \draw(-0.685,-0.493) node[anchor=east]{$v_{13.1}$};
  \draw(-0.625,-0.625) node[anchor=north east]{$v_{1.2}$};
  \draw(-0.493,-0.685) node[anchor=north]{$v_{12.1}$};
  \draw( 0.493,-0.685) node[anchor=north]{$v_{12.3}$};
  \draw( 0.625,-0.625) node[anchor=north west]{$v_{1.3}$};
  \draw( 0.685,-0.493) node[anchor=west]{$v_{13.2}$};
  \draw( 0.535, 0.605) node{$a$};
  \draw(-0.530, 0.605) node{$*a$};
  \draw( 0.000, 0.450) node{$b$};
  \draw( 0.000,-0.450) node{$*b$};
  \draw(-0.595, 0.532) node{$c$};
  \draw(-0.590,-0.532) node{$*c$};
  \draw(-0.450, 0.000) node{$d$};
  \draw( 0.445, 0.000) node{$*d$};
  \draw(-0.535,-0.605) node{$e$};
  \draw( 0.539,-0.605) node{$*e$};
  \draw( 0.595,-0.532) node{$f$};
  \draw( 0.590, 0.532) node{$*f$};
 \end{tikzpicture}
\end{center}
Note that we have marked each edge with a direction and a label.
These edges are portions of the curves $\tC_k$, or images of those
curves under the action of elements of $\Pi$, so in particular they
are geodesics.  One can check that when
$b=\sqrt{2/\sqrt{3}-1}\simeq 0.3933$, the fundamental domain is
actually a right angled regular dodecagon.
\begin{checks}
 hyperbolic/HX_check.mpl: check_H_F1()
\end{checks}

\begin{proposition}\lbl{prop-Pi-free}
 The group $\Pi$ acts freely on $\Dl$, and $HF_1(b)$ is a fundamental
 domain for this action.
\end{proposition}
\begin{proof}
 We will use a standard theorem of Poincar\'e (which is presented as
 Theorem VII.1.7 in the textbook~\cite{iv:hg}, for example).
 Recall that in Proposition~\ref{prop-Pi-Sg} we introduced elements
 $\sg_t\in\Pi$ for $t\in\{a,b,c,d,e,f\}$.  We claim that $\sg_t$
 carries the edge labelled $*t$ to the edge labelled $t$, and carries
 the interior of $HF_1(b)$ to the exterior.  As the edges are geodesic,
 this claim can be checked by calculating the effect of $\sg_t$ on the
 ends of the relevant edge, which is straightforward.  Thus, we have a
 system of edge pairings as in the Poincar\'e Theorem.

 For the next ingredient, we need to know the internal angles of the
 hyperbolic polygon $HF_1(b)$.  We claim that they are all $\pi/2$.  For
 example, side $a$ is part of the curve $\tC_8$, and by inspecting the
 formula for $\tc_8(t)$ we see that $\tC_8=\Xi_{m_8}$, where
 $m_8=b_+/2+i/b_+$.  Similarly, side $b$ is part of $\tC_4=\Xi_{m_4}$,
 where $m_4=ib_+$.  As we mentioned previously, geodesics $\Xi_m$ and
 $\Xi_{m'}$ are orthogonal if and only if $\text{Re}(\ov{m}m')=1$.  It
 is clear  that $\text{Re}(\ov{m_4}m_8)=1$, so edges $a$ and $b$ meet
 at right angles.  Similarly, $*f$ is part of $\tC_7=\Xi_{m_7}$, where
 $m_7=ib_+/2+1/b_+$, and this also meets $a$ at right angles, by the same
 test.  It follows by symmetry that the remaining internal angles are
 $\pi/2$.

 We next need to understand the edge cycle map.  We write $\ov{a}$ for
 the edge $a$ considered in the reverse direction, and similarly for
 the other edges.  We write $E$ for the set of directed edges, so
 $|E|=24$.  We define $\al\:E\to E$ by
 \[ \al(t) = *t \qquad
    \al(*t) = t \qquad
    \al(\ov{t}) = \ov{*t} \qquad
    \al(\ov{*t}) = \ov{t}.
 \]
 We also define $\bt(u)$ to be the directed edge different from $u$
 that has the same initial point as $u$.  For example, we have
 $\bt(a)=*f$, $\bt(b)=\ov{a}$ and so on.  The edge cycle map $\gm$ is
 the composite $\bt\al\:E\to E$.  It can be written in disjoint cycle
 notation as
 \[ \gm = (a\;c\;e\;f)
          (*f\;*e\;*c\;*a)
          (b\;\ov{*e}\;\ov{*b}\;\ov{*a})
          (\ov{a}\;\ov{b}\;\ov{e}\;*b)
          (d\;\ov{*f}\;\ov{*d}\;\ov{*c})
          (\ov{c}\;\ov{d}\;\ov{f}\;*d).
 \]
 For each cycle in $\gm$, we can consider the sum of the internal
 angles at the initial points of the corresponding edges.  The key
 hypothesis in the Poincar\'e Theorem is that this sum must be a
 multiple of $2\pi$.  This is clearly satisfied here, as each cycle
 has length $4$ and all internal angles are $\pi/2$.

 The theorem now tells us that the side pairing maps $\sg_t$ generate
 a group $\Sg$ that acts freely on $\Dl$, with $HF_1(b)$ as a fundamental
 domain.  Moreover, $\Sg$ is generated freely by the $\sg_t$ subject
 only to a small family of relations, one for each cycle in $\gm$.
 The relations are constructed in an obvious way from the cycles, with
 a factor $\sg_t$ for each entry $t$ or $*t$, and a factor
 $\sg_t^{-1}$ for each entry $\ov{t}$ or $\ov{*t}$.  Specifically, the
 relations are as follows:
 \begin{align*}
  \sg_a\sg_c\sg_e\sg_f &= 1 &
  \sg_f^{-1}\sg_e^{-1}\sg_c^{-1}\sg_a^{-1} &= 1 \\
  \sg_b\sg_e^{-1}\sg_b^{-1}\sg_a^{-1} &= 1 &
  \sg_a\sg_b\sg_e\sg_b^{-1} &= 1 \\
  \sg_d\sg_f^{-1}\sg_d^{-1}\sg_c^{-1} &= 1 &
  \sg_c\sg_d\sg_f\sg_d^{-1} &= 1.
 \end{align*}
 The relations on the right are equivalent to those on the left and so
 can be ignored.  Proposition~\ref{prop-Pi-Sg} now allows us to
 identify $\Sg$ with $\Pi$.
 \begin{checks}
  hyperbolic/HX_check.mpl: check_H_F1()
 \end{checks}
\end{proof}

\begin{remark}\lbl{rem-move-inwards}
 We can make a more constructive statement as follows.  Any conformal
 automorphism of $\Dl$ has the form $\gm(z)=\lm(z-\al)/(1-\ov{\al}z)$
 for some $\lm,\al$ with $|\lm|=1$ and $|\al|<1$.  Suppose that
 $\al\neq 0$.  One can check that
 \[ |\gm(z)|^2 - |z|^2 =
     \left|\frac{\al}{1-\ov{\al}z}\right|^2
     (1-|z|^2) \left(|z-\ov{\al}^{-1}|^2-(|\al|^{-2}-1)\right).
 \]
 \begin{checks}
  hyperbolic/HX_check.mpl: check_move_inwards()
 \end{checks}
 This means that $|\gm(z)|=|z|$ if and only if
 $|z-\ov{\al}^{-1}|=\sqrt{|\al|^{-2}-1}$, and this locus describes a circle 
 that cuts $\partial\Dl$ at right angles, or in other words, a
 hyperbolic geodesic.  Now put
 \[ B = \{\bt_0,\bt_2,\bt_4,\bt_6,
          \bt_0\bt_7,\bt_1\bt_2,\bt_2\bt_1,\bt_3\bt_4,
          \bt_4\bt_3,\bt_5\bt_6,\bt_6\bt_5,\bt_7\bt_0\}.
 \]
 Using the above analysis one can check that
 \[ HF_1(b) = \{z\in\Dl\st |z|\leq |\gm(z)| \text{ for all } \gm\in B\}.
 \]
 (Note that there is one element of $B$ for each of the twelve sides
 of $HF_1(b)$.)  Now suppose we have a point $z_0\in\Dl$ that lies
 outside $HF_1(b)$.  We can choose $\gm_0\in B$ such that the point
 $z_1=\gm_0(z_0)$ has $|z_1|<|z_0|$.  If $z_1$ still does not lie in
 $HF_1(b)$, we can choose $\gm_1\in B$ such that the point
 $z_2=\gm_1(z_1)$ has $|z_2|<|z_1|$, and so on.  As the orbit
 $\Pi z_0$ is discrete, it can only contain finitely many points of
 absolute value less than or equal to $|z_0|$, so this process will
 terminate after a finite number of steps.  This gives us a point
 $z_n\in HF_1(b)\cap\Pi z_0$, which will be unique unless it lies on
 the boundary of $HF_1(b)$.  This algorithm is implemented by the
 function \mcode+retract_F1_H0_aux+, defined in
 \fname+hyperbolic/HX0.mpl+.
\end{remark}

\begin{corollary}\lbl{cor-bound-i}
 For $\gm\in\Pi\sm\{1\}$ we have $1/\sqrt{2}\leq|\gm(0)|<1$.
\end{corollary}
\begin{proof}
 Because $0$ lies in the interior of $F$ we see that the orbit $\Pi.0$
 is discrete.  We can thus choose $\gm\in\Pi\sm\{1\}$ such that
 $|\gm(0)|$ is minimal.  Because $F$ is a fundamental domain, it is
 clear that $\gm(0)\not\in F$.  Thus, by
 Remark~\ref{rem-move-inwards}, there is an element $\dl\in B$
 such that $|\dl\gm(0)|<|\gm(0)|$.  By our choice of $\gm$, we must
 have $\dl\gm=1$, so $\gm=\dl^{-1}$.  As $B$ is closed under
 inversion, we must have $\gm\in B$.  From the definitions we see
 that when $\gm\in B$ we have
 \[ |\gm(0)|^2 \in \left\{
      \frac{1}{2}\left(1+\frac{1-b^2}{1+b^2}\right),
      \frac{4}{5}\left(1+\frac{b^2(3-b^2)}{5+2b^2+b^4}\right)
     \right\}.
 \]
 From these expressions it is clear that $|\gm(0)|^2\geq 1/2$, as
 required.
 \begin{checks}
  hyperbolic/HX_check.mpl: check_Pi_bound()
 \end{checks}
\end{proof}

\begin{corollary}\lbl{cor-bound-ii}
 Consider an element $\gm\in\Pi\sm\{1\}$, given by
 $\gm(z)=(az+b)/(cz+d)$ with $ad-bc=1$.  Then
 \begin{align*}
  |c| &\geq 1 \\
  1 &< |d/c|\leq \sqrt{2} \\
  |cz+d| &\leq (1+\sqrt{2})|c|.
 \end{align*}
\end{corollary}
\begin{proof}
 As $\gm$ preserves $\Dl$, it can be written in the form
 $\gm(z)=\mu^2(z+\al)/(\ov{\al}z+1)$ with $|\al|<1$ and $|\mu|=1$.
 This can be rewritten as $\gm(z)=(az+b)/(cz+d)$, where
 \begin{align*}
  a &= \mu/\sqrt{1-|\al|^2} &
  b &= \mu\al/\sqrt{1-|\al|^2} \\
  c &= \ov{\mu}\ov{\al}/\sqrt{1-|\al|^2} &
  d &= \ov{\mu}/\sqrt{1-|\al|^2},
 \end{align*}
 and then $ad-bc=1$.

 As $\gm(0)=\mu^2\al$, Corollary~\ref{cor-bound-i} tells us that
 $1/\sqrt{2}\leq|\al|<1$.  Now $|c|=|\al|/\sqrt{1-|\al|^2}$, which is
 a strictly increasing function of $|\al|$, and is equal to $1$ when
 $|\al|=1/\sqrt{2}$.  We deduce that $|c|\geq 1$ as claimed.
 Similarly, we have $|d/c|=1/|\al|\in(1,\sqrt{3}]$.  Finally, we have
 \[ |cz+d|/|c|=|z+d/c| \leq |z|+|d/c| < 1 + \sqrt{2}, \]
 so $|cz+d|\leq (1+\sqrt{2})c$ as claimed.
\end{proof}

\subsection{Fundamental domains}
\lbl{sec-H-fundamental}

\begin{definition}\lbl{defn-HF-sixteen}
 We define $HF_{16}(b)$ to be the region indicated below:
 \begin{center}
  \begin{tikzpicture}[scale=8]
   \draw[blue] (0,0) -- (0.5,0);
   \draw[green] (0,0) -- ( 0.426, 0.426);
   \draw[cyan] (0.800,0.800) +(-135:0.529) arc(-135:-115:0.529);
   \draw[magenta] (1.250,0.000) +(155:0.750) arc(155:180:0.750);
   \fill[black](0.000,0.000) circle(0.004);
   \fill[black](0.426,0.426) circle(0.004);
   \fill[black](0.500,0.000) circle(0.004);
   \fill[black](0.573,0.322) circle(0.004);
   \draw( 0.000,-0.030) node{$v_{0}$};
   \draw( 0.603, 0.312) node{$v_{3}$};
   \draw( 0.386, 0.426) node{$v_{6}$};
   \draw( 0.500,-0.030) node{$v_{11}$};
   \draw( 0.505, 0.385) node{$c_{0}$};
   \draw( 0.200, 0.250) node{$c_{1}$};
   \draw( 0.556, 0.160) node{$c_{3}$};
   \draw( 0.250,-0.030) node{$c_{5}$};
  \end{tikzpicture}
 \end{center}
\end{definition}

\begin{proposition}\lbl{prop-HF-sixteen}
 $HF_{16}(b)$ is a fundamental domain for the action of $\tPi$ on
 $\Dl$.  Similarly, the image in $HX(b)$ is a fundamental domain for
 the action of $G$ on $HX(b)$.
\end{proposition}
\begin{proof}
 First put
 \[ HF_8(b) = \{z\in HF_1(b)\st 0\leq\arg(z)\leq\pi/4\}. \]
 By inspecting the following picture, we see that
 \begin{align*}
  HF_{16}(b)\cup \lm\nu\mu(HF_{16}(b)) &= HF_8(b) \\
  HF_{16}(b)\cap \lm\nu\mu(HF_{16}(b)) &= c_0([\pi/4,\pi/2]).
 \end{align*}
 \begin{center}
  \begin{tikzpicture}[scale=8]
   \draw[blue] (0,0) -- (0.5,0);
   \draw[green] (0,0) -- ( 0.625, 0.625);
   \draw[cyan] (0.800,0.800) +(-135:0.529) arc(-135:-115:0.529);
   \draw[magenta] (1.250,0.000) +(139:0.750) arc(139:180:0.750);
   \draw[blue] (0.800,0.625) +(-180:0.175) arc(-180:-130:0.175);
   \fill[black](0.000,0.000) circle(0.004);
   \fill[black](0.625,0.625) circle(0.004);
   \fill[black](0.426,0.426) circle(0.004);
   \fill[black](0.500,0.000) circle(0.004);
   \fill[black](0.685,0.493) circle(0.004);
   \fill[black](0.573,0.322) circle(0.004);
   \draw( 0.000,-0.030) node{$v_{0}$};
   \draw( 0.625, 0.650) node{$v_{1}$};
   \draw( 0.603, 0.312) node{$v_{3}$};
   \draw( 0.386, 0.426) node{$v_{6}$};
   \draw( 0.500,-0.030) node{$v_{11}$};
   \draw( 0.715, 0.493) node{$v_{13}$};
   \draw( 0.490, 0.340) node{$c_{0}$};
   \draw( 0.200, 0.250) node{$c_{1}$};
   \draw( 0.556, 0.160) node{$c_{3}$};
   \draw( 0.250,-0.030) node{$c_{5}$};
   \draw( 0.670, 0.550) node{$c_{7}$};
  \end{tikzpicture}
 \end{center}
 Now put
 \begin{align*}
  T_8 &= \{ 1,\lm,\lm^2,\lm^3,\nu,\lm\nu,\lm^2\nu,\lm^3\nu \} \\
  T_{16} &= T_8 \amalg \{\tau\lm\mu\nu\st \tau\in T_8\}.
 \end{align*}
 It is easy to check that
 \[ \bigcup_{\tau\in T_{16}}HF_{16}(b) =
    \bigcup_{\tau\in T_8}HF_8(b) = HF_1(b).
 \]
 It is also easy to check that the homomorphism $\pi\:\tPi\to G$
 restricts to give a bijection $T_{16}\to G$, so that
 \[ \tPi=\coprod_{\tau\in T_{16}} \Pi\tau.  \]
 We also know that $HF_1(b)$ is a fundamental domain for $\Pi$, so
 $\Dl=\bigcup_{\phi\in\Pi}\phi(HF_1(b))$.  Putting this together, we
 deduce that $\Dl=\bigcup_{\phi\in\tPi}\phi(HF_{16}(b))$.

 Now consider an element $\phi\in\tPi\sm\{1\}$ and a point
 $z\in HF_{16}(b)$ such that $\phi(z)$ also lies in $HF_{16}(b)$.  We
 can write $\phi=\psi\tau$ with $\psi\in\Pi$ and $\tau\in T_{16}$.
 Now the point $z'=\tau(z)$ lies in $HF_1(b)$, so the point
 $z''=\phi(z)=\psi(z')$ lies in $\psi(HF_1(b))$, but it also lies in
 $HF_{16}(b)$ by assumption, and one can see directly that
 $HF_{16}(b)$ is in the interior of $HF_1(b)$.  As $HF_1(b)$ is a
 fundamental domain for $\Pi$, this can only be consistent if
 $\psi=1$.  This means that
 $\tau(z)\in HF_{16}(b)\cap\tau(HF_{16}(b))$, and the action of
 $T_{16}$ is simple enough that we can just do a check of cases to
 show that $\tau(z)\in\partial(HF_{16}(b))$.  This proves that
 $HF_{16}(b)$ is a fundamental domain as claimed.
\end{proof}

\begin{proposition}\lbl{prop-square-diffeo-H}
 $HF_{16}(b)$ is homeomorphic to the unit square.
\end{proposition}
\begin{proof}
 This is visually obvious, but we will outline a construction of an
 explicit homeomorphism.  Suppose we have two disjoint geodesics in
 $\Dl$, the first with endpoints $a_1,a_2\in S^1$, and the second with
 endpoints $a_3,a_4\in S^1$.  It is easy to produce a M\"obius
 transformation $m$ that sends $\Dl$ to the upper half-plane and
 $\{a_1,a_2\}$ to $\{1,-1\}$.  This means that $m$ sends our first
 geodesic to the upper half of the unit circle, and our second
 geodesic to another semicircle that crosses the real line
 orthogonally at $m(a_3)$ and $m(a_4)$.  As the two geodesics are
 disjoint, the second circle must either be wholly inside or wholly
 outside the unit circle.  By composing $m$ with $z\mapsto -1/z$ if
 necessary, we may assume that it lies outside.  There is in fact a
 one-parameter family of choices of $m$ that have the properties
 mentioned so far, and one can check that there is precisely one
 choice for which $m(a_3)+m(a_4)=0$.  With this condition, we see that
 $m$ sends the second geodesic to a semicircle centred at the origin,
 of radius $r>1$ say.  Now the function $p(z)=\log|m(z)|/\log(r)$
 takes the values $0$ and $1$ on our two geodesics.  In the cases of
 interest it is convenient to adjust this procedure slightly.  Rather
 than explaining the intermediate steps, we just describe the
 outcome.  We put
 \begin{align*}
  \zt_1  &= \sqrt{\frac{i-b}{i+b}} = \frac{1+ib}{b_+} \in S^1 &
  \zt_2  &= \frac{b_++ib_-}{1+i} \in S^1 \\
  r_1    &= \sqrt{\frac{1-b}{1+b}} = \frac{b_-}{1+b} \in\R^+ &
  r_2    &= \frac{b_++b_-}{\rt b} \in\R^+ \\
  m_1(z) &= \frac{\zt_1 - z}{1 - \zt_1 z} &
  m_2(z) &= \frac{i\zt_2 - z}{1 - \zt_2 z} \\
  p_1(z) &= \log|m_1(z)|/\log(r_1) &
  p_2(z) &= \log|m_2(z)|/\log(r_2),
 \end{align*}
 then $p(z)=(1-p_1(z),p_2(z))$.  This defines a map
 $p\:HF_{16}(b)\to [0,1]^2$, with boundary behaviour as discussed in
 Section~\ref{sec-fundamental}.  Recall that M\"obius transformations
 send circles to circles (provided that we interpret straight lines as
 circles of infinite radius).  It follows that for any point
 $t\in[0,1]^2$, the fibres $p_1^{-1}\{1-t_1\}$ and $p_2^{-1}\{t_2\}$
 are circles, with centres $c_1$ and $c_2$ say.  It is not too hard to
 obtain formulae for $c_1$ and $c_2$; the observation that
 $m_1=m_1^{-1}$ and $m_2=m_2^{-1}$ is helpful for this.  The relevant
 circles must cross the unit circle orthogonally, so the radii are
 $\sqrt{|c_1|^2-1}$ and $\sqrt{|c_2|^2-1}$.  The point $p^{-1}(t)$
 lies on the intersection of the two circles, and so can be found by
 an exercise in coordinate geometry.  After some simplification we
 arrive at the following formulae:
 \begin{align*}
  s_1 &= r_1^{1-t_1} &
  s_2 &= r_2^{t_2} \\
  c_1 &= \frac{s_1^2/\zt_1-\zt_1}{s_1^2-1} &
  c_2 &= \frac{s_2^2/\zt_2-i\zt_2}{s_2^2-1} \\
  \al &= \text{Re}\left(\frac{c_1(\ov{c_1}-\ov{c_2})}{|c_1-c_2|^2}\right) &
  \bt &= \sqrt{\frac{|c_1|^2-1}{|c_1-c_2|^2}-\al^2},
 \end{align*}
 then
 \[ p^{-1}(t) = (1-\al+i\bt)c_1 + (\al-i\bt)c_2. \]
 Unfortunately this formula is not meaningful when $t_1=1$ or $t_2=0$,
 but those cases can be handled in an \emph{ad hoc} way.
 \begin{checks}
  hyperbolic/HX_check.mpl: check_square_diffeo_H()
 \end{checks}
\end{proof}

\subsection{The hyperbolic family is universal}
\lbl{sec-H-universal}

\begin{theorem}\lbl{thm-H-universal}
 Let $X$ be a cromulent surface.  Then there is a unique number
 $b\in(0,1)$ such that $X$ is isomorphic to $HX(b)$, and the
 isomorphism is also unique.
\end{theorem}

The rest of this section will constitute the proof.  The threads will
be gathered together in Proposition~\ref{prop-H-universal}.

Theorem~\ref{thm-classify-cromulent} says that every cromulent surface
is isomorphic to $PX(a)$ for some $a$, so we may assume that
$X=PX(a)$.  We therefore have a curve system as in
Definition~\ref{defn-curve-system}, and a fundamental domain
$PF_{16}(a)$ as in Proposition~\ref{prop-P-fundamental}.

\begin{definition}\lbl{defn-lm-nu-Dl}
 As before we define $\lm,\nu\:\Dl\to\Dl$ by $\lm(z)=iz$ and
 $\nu(z)=\ov{z}$.  These maps clearly satisfy
 $\lm^4=\nu^2=(\lm\nu)^2=1$.
\end{definition}

\begin{lemma}\lbl{lem-p-Dl-X}
 There is a unique covering map $p\:\Dl\to X$ such that $p(0)=v_0$ and
 $p'(0)$ is a positive multiple of $c'_5(0)$.  Moreover, this
 satisfies $\lm p=p\lm$ and $\nu p=p\nu$.
\end{lemma}
\begin{proof}
 The uniformization theorem for Riemann surfaces shows that the
 universal cover of $X$ is conformally equivalent to $\Dl$, so we can
 choose a covering map $p_0\:\Dl\to X$.  We can then choose a point
 $a\in\Dl$ with $p_0(a)=v_0$, and put $m(z)=(z+a)/(1+\ov{a}z)$.  Then
 $m$ is an automorphism of $\Dl$ with $m(0)=a$, so the composite
 $p_1=p_0\circ m$ is a covering with $p_1(0)=v_0$.  Now
 $p'_1(0)=re^{i\tht}\,c'_5(0)$ for some $r>0$ and $\tht\in\R$, and the
 map $p(z)=p_1(z/e^{i\tht})$ is a covering with the required
 properties.  If $q$ is another such map, then standard theory of
 coverings gives an automorphism $f\:\Dl\to\Dl$ such that $q=pf$.  Our
 condition on $p$ and $q$ implies that $f(0)=0$ and $f'(0)>0$.
 However, any conformal automorphism of $\Dl$ has the form
 $f(z)=\al(z-\bt)/(1-\ov{\bt}z)$ for some $\al,\bt$ with $|\al|=1$ and
 $|\bt|<1$.  The condition $f(0)=0$ gives $\bt=0$, and then the
 condition $f'(0)>0$ gives $\al=1$, so $f$ is the identity and $p=q$
 as claimed.

 The maps $\lm^{-1}p\lm$ and $\nu^{-1}p\nu$ are easily seen to satisfy
 the defining conditions for $p$, so $\lm p=p\lm$ and $\nu p=p\nu$.
\end{proof}

\begin{proposition}\lbl{prop-classify-a}
 There is a unique system of points $v_i^*\in\Dl$ (for $0\leq i\leq
 13$) and continuous maps $c_j^*\:\R\to\Dl$ (for $0\leq j\leq 8$) such
 that the following hold:
 \begin{itemize}
  \item[(a)] $v_0^*=0$
  \item[(b)] For all $i$ we have $p(v_i^*)=v_i$, and for all $j$ and
   $t$ we have $p(c_j^*(t))=c_j(t)$.
  \item[(c)] Whenever a number $t$ appears in the $i$'th column of the
   $j$'th row of the table below, we have $c^*_j(t)=v_i^*$.
   \[ \begin{array}{|c|c|c|c|c|c|c|c|c|c|c|c|c|c|c|}
    \hline
   &0&1&2&3&4&5&6&7&8&9&10&11&12&13\\ \hline
    0&&&0&\ppi&&&\tfrac{\pi}{4}&&&&&&&\\ \hline
    1&0&\pi&&&&&\ppi&&-\ppi&&&&&\\ \hline
    2&0&&&&&&&\ppi&&-\ppi&&&&\\ \hline
    3&&&&\ppi&&-\ppi&&&&&&0&&\pi\\ \hline
    4&&&-\ppi&&\ppi&&&&&&0&&-\pi&\\ \hline
    5&0&&&&&&&&&&&\pi&&\\ \hline
    6&0&&&&&&&&&&\pi&&&\\ \hline
    7&&0&&&&&&&&&&&&\\ \hline
    8&&0&&&&&&&&&&&&\\ \hline
   \end{array} \]
  \item[(d)] The action of $\lm$ and $\nu$ on the points $v_i^*$ is
   partially given by the following table:
   \[ \begin{array}{|c|c|c|c|c|c|c|c|c|c|c|c|c|c|c|}
    \hline
     &v_0^*&v^*_1&v^*_2&v^*_3&v^*_4&v^*_5&v^*_6&v^*_7&v^*_8&v^*_9&
      v^*_{10}&v^*_{11}&v^*_{12}&v^*_{13}\\ \hline
     \lm & v^*_0 & & & v^*_4 & & v^*_3 &
           v^*_7 & v^*_8 & v^*_9 & v^*_6 & & v^*_{10} & & \\ \hline
     \nu & v^*_0 & & & v^*_5 & & v^*_3 &
           v^*_9 & v^*_8 & v^*_7 & v^*_6 & & v^*_{11} & & \\ \hline
     \lm\nu & v^*_0 & v^*_1 & v^*_3 & v^*_2 & v^*_5 & v^*_4 &
           v^*_6 & v^*_9 & v^*_8 & v^*_7 &
           v^*_{11} & v^*_{10} & v^*_{13} & v^*_{12} \\ \hline
   \end{array} \]
  \item[(e)] The action of $\lm$ and $\nu$ on the curves $c_j^*$ is
   partially given by the following table:
   \[ \begin{array}{|c|c|c|c|c|c|c|c|c|c|}
    \hline
     &c_0^*(t)&c^*_1(t)&c^*_2(t)&c^*_3(t)&c^*_4(t)&c^*_5(t)&c^*_6(t)&c^*_7(t)&c^*_8(t)\\ \hline
     \lm & & c_2^*(t) & c_1^*(-t) & c_4^*(t) & &
           c_6^*(t) & c_5^*(-t) & &  \\ \hline
     \nu & & c_2^*(-t) & c_1^*(-t) & c_3^*(-t) & &
             c_5^*(t) & c_6^*(-t) & & \\ \hline
     \lm\nu & c_0^*(\pi/2-t) & c_1^*(t) & c_2^*(-t) &
              c_4^*(-t) & c_3^*(-t) & c_6^*(t) & c_5^*(t) &
              c_8^*(t) & c_7^*(t) \\ \hline
   \end{array} \]
 \end{itemize}
\end{proposition}
\begin{proof}\leavevmode
 \begin{itemize}
  \item[(0)] We define $v^*_0=0$ and note that this is fixed by $\lm$
   and $\nu$, and that $p(v_0^*)=v_0$ by the definition of $p$.
  \item[(1)] For $j\in\{1,2,5,6\}$ we define $c_j^*$ to be the unique
   continuous lift of $c_j$ that satisfies $c^*_j(0)=v^*_0$.  The
   claimed formulae for $\lm(c^*_j(t))$ and $\nu(c^*_j(t))$ then hold
   by an evident uniqueness argument.  For example, we have seen
   previously that $\lm(c_2(t))=c_1(-t)$ in $PX(a)$, so the maps
   $t\mapsto\lm(c_2^*(t))$ and $t\mapsto c^*_1(-t)$ are both
   continuous lifts of the map $t\mapsto c_1(-t)$.  Both give $v^*_0$
   when $t=0$, so they must agree for all $t$.
  \item[(2)] We define
   \begin{align*}
    v^*_1 &= c_1^*(\pi) &
    v^*_6 &= c_1^*(\pi/2) &
    v^*_8 &= c_1^*(-\pi/2) \\
          & &
    v^*_7 &= c_2^*(\pi/2) &
    v^*_9 &= c_2^*(-\pi/2) \\
    v^*_{10} &= c^*_6(\pi) \\
    v^*_{11} &= c^*_5(\pi).
   \end{align*}
   By taking $t=\pi$ or $t=\pm\pi/2$ in~(1) we see that $\lm$ and
   $\nu$ act on $\{v^*_1,v^*_6,v^*_7,v^*_8,v^*_9,v^*_{10},v^*_{11}\}$
   as indicated in~(d).  Because the maps $c_j\:\R\to PX(a)$ form a
   curve system, we see that $p(v^*_i)=v_i$ in all these cases.
  \item[(3)] For $j\in\{3,4,7,8\}$ we let $c^*_j$ denote the unique
   lift of $c_j(t)$ satisfying the following initial condition:
   \[ c_3^*(0) = v_{11}^* \hspace{4em}
      c_4^*(0) = v_{10}^* \hspace{4em}
      c_7^*(0) = c_8^*(0) = v_1^*.
   \]
   The claimed formulae for $\lm(c^*_j(t))$ and $\nu(c^*_j(t))$ then
   hold by an evident uniqueness argument.
  \item[(4)] We define
   \begin{align*}
    v^*_{13} &= c_3^*(\pi) &
    v^*_3 &= c_3^*(\pi/2) &
    v^*_5 &= c_3^*(-\pi/2) \\
    v^*_{12} &= c_4^*(-\pi) &
    v^*_4 &= c_4^*(\pi/2) &
    v^*_2 &= c_4^*(-\pi/2).
   \end{align*}
   By taking $t=\pi$ or $t=\pm\pi/2$ in~(3) we see that $\lm$ and
   $\nu$ act on $\{v^*_2,v^*_3,v^*_4,v^*_5,v^*_{12},v^*_{13}\}$
   as indicated in~(d).  Because the maps $c_j\:\R\to PX(a)$ form a
   curve system, we see that $p(v^*_i)=v_i$ in all these cases.
  \item[(5)] We let $c^*_0$ denote the unique lift of $c_0$ satisfying
   $c^*_0(\pi/4)=v^*_6$.   By the usual uniqueness argument, so have
   $\lm\nu(c^*_0(t))=c^*_0(\pi/2-t)$.
  \item[(6)] Now all of~(a), (b) and~(c) is true by construction
   except for the identities $c^*_0(0)=v^*_2$ and
   $c^*_0(\pi/2)=v^*_3$.  All parts of~(d) and~(e) have also been
   established.  For the remaining facts, we consider the fundamental
   domain $PF_{16}(a)$ in $PX(a)$.  We have two paths in $PF_{16}(a)$
   from $v_0$ to $v_3$: one given by $c_5([0,\pi])$ followed by
   $c_3([0,\pi/2])$, and the other by $c_1([0,\pi/2])$ followed by
   $c_0([\pi/4,\pi/2])$.  If we lift the first path starting with
   $v^*_0$ then the endpoint is $c^*_3(\pi/2)=v^*_3$, and if we lift
   the second then the endpoint is $c_0^*(\pi/2)$.  Now $PF_{16}(a)$
   is homeomorphic to a square and so is contractible.  In particular,
   our two paths are homotopic relative to the endpoints, and it
   follows that the two lifts have the same endpoints, so
   $c_0^*(\pi/2)=v_3$ as claimed.  By applying the map $\lm\nu$ we
   deduce that $c_0^*(0)=v_2$.
 \end{itemize}
\end{proof}

\begin{proposition}
 For $0\leq j\leq 8$ there is an antiholomorphic involution
 $\tht_j\:\Dl\to\Dl$ such that $c_j^*$ gives a diffeomorphism from
 $\R$ to the geodesic $C^*_j=\{z\in\Dl\st\tht_j(z)=z\}$.
 Specifically, for $j\in\{1,2,5,6\}$ we have
 \begin{align*}
  \tht_1 &= \lm\nu   & C^*_1 &= (-1,1).e^{ i\pi/4} \\
  \tht_2 &= \lm^3\nu & C^*_2 &= (-1,1).e^{-i\pi/4} \\
  \tht_5 &= \nu      & C^*_5 &= (-1,1) \\
  \tht_6 &= \lm^2\nu & C^*_6 &= (-1,1).i.
 \end{align*}
\end{proposition}
\begin{proof}
 Put $C_j=c_j(\R)\subset PX(a)$.  We have seen previously that in each
 case there is an antiholomorphic involution $\rho_j\in G$ such that
 $\rho_j(c_j(t))=c_j(t)$ for all $t\in\R$, and in fact $C_j$ is a
 connected component of the set
 $PX(a)^{\rho_j}=\{z\in PX(a)\st\rho_j(z)=z\}$.  Moreover, $C_j$ is
 diffeomorphic to $S^1$ and the map $c_j\:\R\to C_j$ is a universal
 covering.

 Next, standard covering theory shows that there is a
 unique continuous map $\tht_j\:\Dl\to\Dl$ with $p\tht_j=\rho_jp$ and
 $\tht_j(c^*_j(0))=c_j^*(0)$.  As $p$ is a holomorphic covering, the
 equation $p\tht_j=\rho_jp$ implies that $\tht_j$ is antiholomorphic.
 The map $\tht_j^2$ covers $\rho_j^2=1$ and fixes $c_j^*(0)$; it
 follows that $\tht_j^2=1$.  Now $c^*_j$ and $\rho_j\circ c^*_j$ are
 both lifts of $c_j$ with the same value at $t=0$, so they must be the
 same, so $c_j^*(\R)\sse C_j^*$.  We previously classified the
 antiholomorphic involutions on $\Dl$, and using that classification
 we see that $C_j^*$ is a geodesic in $\Dl$ and is diffeomorphic to
 $\R$.  It follows that $p(C_j^*)$ is a connected subset of
 $PX(a)^{\rho_j}$ containing $p(c^*_j(0))=c_j(0)$, so
 $p(C_j^*)\sse C_j$.  Now $p$ is a proper map with nonzero complex
 derivative everywhere in $\Dl$.  It follows that $p\:C_j^*\to C_j$ is
 also proper with nonzero real derivative everywhere.  This means that
 $p\:C_j^*\to C_j$ is another universal covering, and by the
 uniqueness of universal coverings, we see that $c^*_j\:\R\to C_j^*$
 must be a diffeomorphism.

 For the case $j=5$, we have seen that $\nu(c_5^*(t))=c_5^*(t)$ for
 all $t$ and it follows that $\tht_5=\nu$.  We also have
 $\nu(z)=\ov{z}$ so $C_5^*=(-1,1)$.  The cases $j\in\{1,2,6\}$ are
 similar.
\end{proof}

\begin{lemma}\lbl{lem-Dl-involutions}
 For any point $v\in\Dl$, there is a unique holomorphic involution on
 $\Dl$ that fixes $v$.
\end{lemma}
\begin{proof}
 As $\Aut(\Dl)$ acts transitively on $\Dl$, we may assume that $v=0$.
 In this case the map $z\mapsto -z$ is a holomorphic involution that
 fixes $v$.  Let $\phi$ be any other holomorphic involution that
 fixes $v$.  As $\phi$ is an automorphism of $\Dl$ we have
 $\phi(z)=\lm(z-\al)/(1-\ov{\al}z)$ for some $\lm,\al$ with $|\lm|=1$
 and $|\al|<1$.  As $\phi$ fixes $v=0$ we have $\al=0$, so
 $\phi(z)=\lm z$.  As $\phi$ is an involution we must have $\lm=-1$.
\end{proof}

\begin{proposition}\lbl{prop-Dl-kp}
 Let $\kp$ be the unique holomorphic involution on $\Dl$ that fixes
 $v^*_6$.  Then $p\kp=\lm\mu p$, and the action on the points $v^*_i$
 is partially given by the following table:
 \[ \begin{array}{|c|c|c|c|c|c|c|c|c|c|c|c|c|c|c|}
    \hline
     &v_0^*&v^*_1&v^*_2&v^*_3&v^*_4&v^*_5&v^*_6&v^*_7&v^*_8&v^*_9&
      v^*_{10}&v^*_{11}&v^*_{12}&v^*_{13}\\ \hline
     \kp & v^*_1 & v^*_0 & v^*_3 & v^*_2 & & &
      v^*_6 & & & & v^*_{13} & v^*_{12} & v^*_{11}& v^*_{10}\\ \hline
 \end{array} \]
 Also, the action on the curves $c^*_j$ is partially given by the
 following table:
 \[ \begin{array}{|c|c|c|c|c|c|c|c|c|c|}
    \hline
     &c_0^*(t)&c^*_1(t)&c^*_2(t)&c^*_3(t)&c^*_4(t)&
      c^*_5(t)&c^*_6(t)&c^*_7(t)&c^*_8(t)\\ \hline
     \kp & c_0^*(\pi/2-t) & c_1^*(\pi-t) & &
           c_4^*(t-\pi) & c_3^*(t+\pi) &
           c_8^*(t) & c_7^*(t) & c_6^*(t) & c_5^*(t) \\ \hline
 \end{array} \]
\end{proposition}
\begin{proof}
 Given a point $z\in\Dl$, we let $u$ be any path in $\Dl$ from $v_6^*$
 to $z$ in $\Dl$.  We recall that $\lm\mu(v_6)=v_6$, so the path
 $\lm\mu p u$ starts at $p(v^*_6)$, so there is a unique lifting $u'$
 of $\lm\mu p u$ with $u'(0)=v^*_6$.  We define $\kp(z)=u'(1)$.
 Standard covering theory shows that this is independent of the choice
 of $u$ and gives the unique continuous map $\kp\:\Dl\to\Dl$ with
 $\kp(v^*_6)=v^*_6$ and $p\kp=\lm\mu p$.  As $p$ is a holomorphic
 covering and holomorphy can be checked locally it follows that $\kp$
 is holomorphic.  The map $\kp^2$ covers $(\lm\mu)^2=1$ and fixes
 $v^*_6$; it follows that $\kp^2=1$.  Thus, $\kp$ is the unique
 holomorphic involution on $\Dl$ that fixes $v^*_6$.

 The curves $\kp(c_0^*(t))$ and $c_0^*(\pi/2-t)$ both lift
 $c_0(\pi/2-t)$ and pass through $v^*_6$ at $t=\pi/4$, so they are the
 same.  Taking $t\in\{0,\pi/2\}$ we deduce that $\kp$ exchanges
 $v^*_2$ and $v^*_3$.  Essentially the same argument gives
 $\kp(c^*_1(t))=c^*_1(\pi-t)$, and shows that $\kp$ exchanges $v^*_0$
 and $v^*_1$.  Now that we know the action on the point
 $v^*_3=c^*_3(\pi/2)$ we can see that the curves $\kp(c^*_3(t))$ and
 $c^*_4(t-\pi)$ both lift $c_4(t-\pi)$ and pass through $v^*_2$ at
 $t=\pi/2$ so they are the same.  As $\kp^2=1$ we can also deduce that
 $\kp(c^*_4(t))=c^*_3(t+\pi)$.  By taking $t=0$ or $t=\pm\pi$ we
 deduce that $\kp$ exchanges $v^*_{10}$ and $v^*_{13}$, and also
 exchanges $v^*_{11}$ and $v^*_{12}$.  This just leaves the action on
 $c^*_j(t)$ for $j\in\{5,6,7,8\}$, which can be checked in the same
 way using $v^*_0$ and $v^*_1$ as basepoints.
\end{proof}

\begin{lemma}\lbl{lem-v-eleven-sign}
 $v^*_{11}$ is a positive real number, and $v^*_6$ is a positive
 multiple of $e^{i\pi/4}=(1+i)/\rt$.
\end{lemma}
\begin{proof}
 We have seen that $c^*_5$ gives a diffeomorphism from $\R$ to
 $(-1,1)$.  We also have $p(c^*_5(t))=c_5(t)$ so
 $(c_5^*)'(0)=c'_5(0)/p'(0)$, and this is a positive real number by
 the definition of $p$.  As $c^*_5$ is a diffeomorphism the derivative
 cannot change sign, so it is a strictly increasing map.  It follows
 that the point $v^*_{11}=c^*_5(\pi)$ lies on the positive real axis
 as claimed.

 Next, we also know that $c^*_1$ gives a diffeomorphism from $\R$
 to $(-1,1).e^{i\pi/4}$.  By examining the formula in
 Definition~\ref{defn-P-curves} we see that to first order in $t$ we
 have
 \begin{align*}
  c_1(t) &= [t e^{i\pi/4}/2:1:0:0] \\
  c_5(t) &= [t \sqrt{a}/2:1:0:0],
 \end{align*}
 so $c'_1(0)$ is a positive multiple of $e^{i\pi/4}\,c'_5(0)$.  Using
 this we see that $c^*_5$ must carry $(0,\infty)$ to
 $(0,1).e^{i\pi/4}$.  In particular, the point $v^*_6=c^*_1(\pi/2)$ is
 a positive multiple of $e^{i\pi/4}$ as claimed.
\end{proof}

\begin{lemma}\lbl{lem-right-circles}
 Suppose that two circles in $\R^2$ meet at right angles.  Let $r_1$
 and $r_2$ be the radii, and let $d$ be the distance between the
 centres; then $d^2=r_1^2+r_2^2$.
\end{lemma}
\begin{proof}
 Elementary.
\end{proof}

\begin{proposition}\lbl{prop-b-H}
 There is a unique number $b\in(0,1)$ such that
 \begin{align*}
  v^*_0 &= 0 &
  v^*_1 &= \frac{1+i}{2}b_+ \\
  v^*_2 &= \frac{b\,b_--b_+}{i-b^2} &
  v^*_5 &= -i \, v^*_2 \\
  v^*_3 &= \frac{b\,b_--b_+}{ib^2-1} &
  v^*_4 &= i \, v^*_3 \\
  v^*_6 &= \frac{1+i}{\rt}\;\frac{\rt-b_-}{b_+} &
  v^*_7 &= i \, v^*_6 \\
  v^*_8 &= -v^*_6 &
  v^*_9 &= -i\, v^*_6 \\
  v^*_{10} &= i(b_+-b) &
  v^*_{11} &= b_+-b \\
  v^*_{12} &= (b+b_+)\frac{i+(i+2)b^2}{(b+b_+)^2+b^2} &
  v^*_{13} &= i\,\ov{v^*_{12}}
 \end{align*}
 (where $b_{\pm}=\sqrt{1\pm b^2}$ as before).  Moreover, the map $\kp$
 is given by
 \[ \kp(z) = \frac{b_+i - (1+i)z}{1+i - b_+z}. \]
\end{proposition}
\begin{proof}
 We will use the curves $C_j^*\subset\Dl$ for $j\in\{0,1,3,5\}$.  We
 have already seen that $C_5^*=(-1,1)$ and $C_1^*=(-1,1).e^{i\pi/4}$.
 The set $C_3^*$ is a geodesic in $\Dl$ that does not pass through the
 origin, so it is the intersection of $\Dl$ with a circle centred
 outside $\Dl$ that crosses $\partial\Dl$ at right angles.  (This is a
 standard fact of hyperbolic geometry.)  We let $b$ denote the radius
 of $C_3^*$ (so $b>0$).

 We next claim that
 \begin{itemize}
  \item[(a)] The curves $C_0^*$ and $C_3^*$ cross at right angles at $v_3^*$
  \item[(b)] The curves $C_0^*$ and $C_1^*$ cross at right angles at $v_6^*$
  \item[(c)] The curves $C_3^*$ and $C_5^*$ cross at right angles at $v_{11}^*$.
 \end{itemize}
 Indeed, we see from Proposition~\ref{prop-classify-a} that the
 indicated curves cross at the indicated points, and that in all
 relevant cases we have $p(c_i^*(t))=c_i(t)$ and $p(v_j^*)=v_j$ in
 $PX(a)$.  As $p$ is a holomorphic covering it preserves angles, so
 the claim follows from Lemma~\ref{lem-right-angle}.

 As $C^*_3$ meets the curve $C^*_5=(-1,1)$ at right angles at the
 point $v^*_{11}>0$, we see that the centre of $C^*_3$ must be on the
 positive real axis.  As $C^*_3$ also meets $\partial\Dl$
 orthogonally, Lemma~\ref{lem-right-circles} shows that the centre is
 $\sqrt{1+b^2}=b_+$.  It then follows that $v^*_{11}=b_+-b$.

 Now put $\om=e^{i\pi/4}=(1+i)/\rt$.  By a similar argument,
 there is a number $c>0$ such that $\tC_0$ is a circular arc with
 centre $\sqrt{1+c^2}\om$ and radius $c$, and we have
 $v^*_6=(\sqrt{1+c^2}-c)\om$.

 As $\tC_0$ and $\tC_3$ meet at right angles, we must have
 \[ b^2+c^2 = |\sqrt{1+c^2}\om - \sqrt{1+b^2}|^2 =
     \left(\frac{\sqrt{1+c^2}}{\rt} - \sqrt{1+b^2}\right)^2 +
     \left(\frac{\sqrt{1+c^2}}{\rt}\right)^2.
 \]
 After some manipulation this gives $c^2=(1-b^2)/(1+b^2)$.  This
 ensures that $b<1$, and we can take square roots to get $c=b_-/b_+$.
 We also get $\sqrt{1+c^2}=\rt/b_+$ and so
 \[ v^*_6 = (\sqrt{1+c^2}-c)\om =
     \frac{1+i}{\rt}\;\frac{\rt-b_-}{b_+}
 \]
 as claimed.

 Now put $w=(b_+-b\,b_-)/(1-ib^2)$.  One can check that
 \[ 1 - |w|^2 = 2b\,b_-\,(b_+-b\,b_-)/(1+b^4) > 0, \]
 so $w\in\Dl$.  Long but fairly straightforward calculations also show
 that $|w-b_+|^2=b^2$ and $|w-\sqrt{1+c^2}\om|^2=c^2$, so
 $w\in\tC_3\cap\tC_0$.  It is standard that distinct geodesics in
 $\Dl$ meet in only one place, so we must have $v^*_3=w$.

 Next, if we define
 \[ \kp^*(z) = \frac{b_+i - (1+i)z}{1+i - b_+z}, \]
 a straightforward calculation shows that this is a holomorphic
 involution on $\Dl$ that fixes $v^*_6$.  However,
 Lemma~\ref{lem-Dl-involutions} shows that there is only one such
 involution, so $\kp$ must be the same as $\kp^*$

 Above we have established formulae for $v^*_0$, $v^*_3$, $v^*_6$ and
 $v^*_{11}$ in terms of $b$.  Propositions~\ref{prop-classify-a}
 and~\ref{prop-Dl-kp} also give
 \begin{align*}
  v^*_1 &= \kp(v^*_0) &
  v^*_2 &= \kp(v^*_3) &
  v^*_4 &= \lm(v^*_3) &
  v^*_5 &= \nu(v^*_3) \\
  v^*_7 &= \lm(v^*_6) &
  v^*_8 &= \lm^2(v^*_6) &
  v^*_9 &= \lm^3(v^*_6) \\
  v^*_{10} &= \lm(v_{11}) &
  v^*_{12} &= \kp(v_{11}) &
  v^*_{13} &= \lm\nu(v_{12}).
 \end{align*}
 Using these we can deduce the stated formulae for all points
 $v^*_i$.
\end{proof}

From now on we use the value of $b$ coming from the previous
Proposition.  This gives maps $\lm,\mu,\nu,\bt_0,\dotsc,\bt_7$
generating a group $\Pi$ as in Section~\ref{sec-H}.  The maps
$\lm$ and $\nu$ from that section are of course the same as the ones
we have been using already in this section.

\begin{definition}\lbl{defn-tPhi}
 We say that a conformal or anticonformal automorphism
 $\phi\:\Dl\to\Dl$ is \emph{$G$-compatible} if there is an element
 $\phi_1\in G$ such that the following diagram commutes:
 \[ \xymatrix{
   \Dl \ar[d]_p \ar[r]^\phi & \Dl \ar[d]^p \\
   X \ar[r]_{\phi_1} & X.
 } \]
 We write $\tPhi$ for the group of all $G$-compatible automorphisms,
 and note that the construction $\phi\mapsto\phi_1$ gives a
 homomorphism $\tPhi\to G$.  We write $\Phi$ for the kernel, which is
 just the group of automorphisms $\phi$ satisfying $p\phi=p$, or in
 other words deck transformations.  Standard covering theory shows
 that $p$ induces a conformal isomorphism $\Dl/\Phi\to PX(a)$.
\end{definition}

\begin{proposition}\lbl{prop-tPhi-gens}
 The maps $\lm$, $\mu$, $\nu$ and $\kp$ are elements of $\tPhi$, with
 $\pi(\lm)=\lm$ and $\pi(\mu)=\mu$ and $\pi(\nu)=\nu$ and
 $\pi(\kp)=\lm\mu$.
\end{proposition}
\begin{proof}
 The claims for $\lm$, $\nu$ and $\kp$ are clear by construction.  By
 the same argument that we used in Proposition~\ref{prop-Dl-kp}, if we
 let $\mu'$ denote the unique holomorphic involution on $\Dl$ that
 fixes $v_2$, then $p\mu'=\mu p$, so $\mu'\in\tPhi$ with
 $\pi(\mu')=\mu$.  However, if we define
 \[ \mu(z) = \frac{b_+z-b^2-i}{(b^2-i)z-b_+} \]
 as in Section~\ref{sec-H}, then straightforward calculation
 shows that $\mu$ is a holomorphic involution with $\mu(v^*_2)=v^*_2$,
 so $\mu$ is the same as $\mu'$.
\end{proof}

\begin{proposition}\lbl{prop-Pi-in-Phi}
 We have $\Pi\leq\Phi$ and $\tPi\leq\tPhi$.
\end{proposition}
\begin{proof}
 We first claim that $\bt_0\in\Phi$.  Note that $\lm^2(v_{11})=v_{11}$ in
 $PX(a)$, so the points $v^*_{11}=b_+-b$ and $\lm^2(v^*_{11})=b-b_+$
 have the same image under $p$.  Recall that
 $\bt_0(z)=(b_+z+1)/(z+b_+)$; this implies that
 $\bt_0(\lm^2(v^*_{11}))=v^*_{11}$, and that $\bt_0$ restricts to give
 a strictly increasing automorphism of $(-1,1)$.

 We originally introduced $p$ as the unique holomorphic covering map
 $\Dl\to PX(a)$ such that $p(0)=v_0$ and $p'(0)$ is a positive
 multiple of $c'_5(0)$.  However, the same line of argument shows that
 $p$ is also the unique holomorphic covering map $\Dl\to PX(a)$ such
 that $p(\lm^2(v^*_{11}))=v_{11}$ and $p'(\lm^2(v^*_{11}))$ is a
 positive multiple of $c'_5(-\pi)$.  The composite $p\bt_0$ has these
 properties, so $p\bt_0=p$, so $\bt_0\in\Phi$ as claimed.

 It is also clear that $\Phi$ is normal in $\tPhi$ and $\lm\in\tPhi$
 so the conjugates $\bt_{2k}=\lm^k\bt_0\lm^{-k}$ also lie in $\Phi$.

 We now recall that the elements $\lm,\mu\in\tPi$ lie in $\tPhi$ and
 satisfy $(\lm\mu)^2=\bt_7\bt_6$.  As $(\lm\mu)^2=1$ in $G$ we can
 deduce that $\bt_7\bt_6\in\Phi$.  We saw above that $\bt_6\in\Phi$,
 so $\bt_7\in\Phi$.  Using $\lm^k\bt_7\lm^{-k}=\bt_{7+2k}$ we deduce
 that $\bt_j\in\Phi$ for all $j$, so $\Pi\leq\Phi$.  As
 $\lm,\mu,\nu\in\tPhi$ it also follows that $\tPi\leq\tPhi$.
\end{proof}

\begin{proposition}\lbl{prop-H-universal}
 We have $\Pi=\Phi$ and $\tPi=\tPhi$, and the map $p\:\Dl\to PX(a)$
 induces an isomorphism $\ov{p}\:HX(b)=\Dl/\Pi\to PX(a)$ of cromulent
 surfaces.
\end{proposition}
\begin{proof}
 As $\Pi\leq\Phi$, we can factor the map $p$ as
 \[ \Dl \xra{q} HX(b) = \Dl/\Pi \xra{\ov{p}} \Dl/\Phi \simeq PX(a).
 \]
 As $p$ and $q$ are holomorphic coverings, we see that $\ov{p}$ is
 also a holomorphic covering.  We have also seen that both $HX(b)$ and
 $PX(a)$ are compact, so $\ov{p}$ has degree $d<\infty$ say.  It
 follows (by choosing compatible triangulations, say) that
 $\chi(HX(b))=d\;\chi(PX(a))$ (where $\chi$ denotes the Euler
 characteristic).  However, both $HX(b)$ and $PX(a)$ have genus $g=2$
 and therefore Euler characteristic $2-2g=-2$, so we must have $d=1$,
 so $\ov{p}$ is an isomorphism of Riemann surfaces.  By construction
 it is $G$-equivariant and sends $v_i$ to $v_i$ so it is an
 isomorphism of cromulent surfaces.
\end{proof}

\section{Relating the projective and hyperbolic families}
\lbl{sec-P-H}

Recall that the projective and algebraic families are both universal,
so for each $a\in(0,1)$ there is a unique $b\in (0,1)$ such that
$PX(a)$ is isomorphic (in a unique way) to $HX(b)$, and
\emph{vice-versa}.  In this section we will give two different methods
for calculating this correspondence.  The first method starts with $b$
and calculates $a$, and the second works in the opposite direction.

\subsection{Preliminaries}
\lbl{sec-P-H-prelim}

The space $HX(b)$ is by definition a quotient of $\Dl$.  We have an
isomorphism $HX(b)\to PX(a)$, and an isomorphism
$PX(a)/\ip{\lm^2}\to\C_\infty$ sending $j(w,z)$ to $z$.  Composing
these maps gives a map $p\:\Dl\to\C_\infty$.  Our main task is to
calculate this map.

In Definition~\ref{defn-v-H} we defined certain points $v_i\in\Dl$
(depending on $b$), which become the labelled points in $HX(b)$.  In
this section, we will write $v_{Hi}$ for these points.  Similarly, we
will write $c_{Hj}$ for the curves $\tc_j\:\R\to\Dl$ defined in
Definition~\ref{defn-H-curves}.  Moreover, the images in
$\C_\infty$ of the points $v_i\in PX(a)$ and the curves
$c_j\:\R\to P(a)$ will be denoted by $v_{Ci}$ and $c_{Cj}$.  We are
primarily interested in the corners of the fundamental domain, which
we can tabulate as follows:
\begin{align*}
 v_{H0}  &= 0 &
 v_{C0}  &= 0 \\
 v_{H3}  &= \frac{b\,b_--b_+}{ib^2-1} &
 v_{C3}  &= 1 \\
 v_{H6}  &= \frac{1+i}{\rt}\;\frac{\rt-b_-}{b_+} &
 v_{C6}  &= i \\
 v_{H11} &= b_+-b &
 v_{C11} &= a.
\end{align*}

Because the map $HX(b)\to PX(a)$ is cromulent, we have
$p(v_{Hi})=v_{Ci}$ and $p(c_{Hj}(\R))=c_{Cj}(\R)$.  In particular, we
have $v_{H0}=0$ and $v_{C0}=0$, so $p(0)=0$.  Using equivariance with
respect to $\lm$, we also see that $p(iz)=-p(z)$, so $p(-z)=p(z)$, and
it follows that $p'(0)=0$.  This makes it inconvenient to work with
$p$ itself.  Instead, we will work with a certain map of the form
$p_1=\phi p\psi$, where $\phi\in\Aut(\C_\infty)$ and
$\psi\in\Aut(\Dl)$.  This will be arranged so that $p_1(0)=0$ and
$p'_1(0)>0$.  Details are as follows:
\begin{definition}\lbl{defn-schwarz-phi}
 We define $\phi\in\Aut(\C_\infty)$ and $\psi\in\Aut(\Dl)$ by
 \begin{align*}
  \phi(z) &= \frac{i-z}{i+z} &
  \phi^{-1}(z) &= i \frac{1-z}{1+z} \\
  \psi(z) &= \frac{1+i}{\rt}
    \frac{\rt-b_- - b_+z}{b_+ (b_--\rt)z} &
  \psi^{-1}(z) &= -\frac{\rt-b_--(1-i)b_+z/\rt}{
                          (1-i)(1-b_-/\rt)z-b_+}.
 \end{align*}
 We then define $p_1=\phi\circ p\circ\psi\:\Dl\to\C_\infty$.  We also
 put $v_{HSi}=\psi^{-1}(v_{Hi})\in\Dl$ and
 $v_{PSi}=\phi(v_{Ci})\in\C_\infty$, so that $p_1(v_{HSi})=v_{PSi}$.
 We define curves $c_{HSj}$ and $c_{PSj}$ in the same way.  Finally,
 we put
 \[ t = b^2\frac{\rt b_+ - 2bb_-}{1-b^2+2b^4}
    \hspace{6em}
    s = \frac{(b_+b_--\rt)(b-b^3)}{1-b^2+2b^4}.
 \]
\end{definition}
\begin{remark}
 In Maple, the maps $\phi$ and $\psi$ are \mcode+schwarz_phi+ and
 \mcode+schwarz_psi+, and the inverse maps are \mcode+schwarz_phi_inv+ and
 \mcode+schwarz_psi_inv+.  Maple notation for $v_{HSi}$ and
 $c_{HSj}(t)$ is \mcode+v_HS[i]+ and \mcode+c_HS[j](t)+, and similarly
 for $v_{PSi}$ and $c_{PSj}(t)$.  Maple notation for $t$ and $s$ is
 \mcode+t_schwarz+ and \mcode+s_schwarz+.  All of this is in
 \fname+hyperbolic/schwarz.mpl+.
\end{remark}

By direct calculation, we have
\begin{align*}
 v_{HS0}  &= \frac{\rt-b_-}{b_+} &
 v_{PS0}  &= 1 \\
 v_{HS3}  &= i \frac{b_-}{b_++\rt b} &
 v_{PS3}  &= i \\
 v_{HS6}  &= 0 &
 v_{PS6}  &= 0 \\
 v_{HS11} &= t+is &
 v_{PS11} &= \frac{i-a}{i+a}
\end{align*}
(where $t$ and $s$ are as in Definition~\ref{defn-schwarz-phi}).
\begin{checks}
 hyperbolic/schwarz_check.mpl: check_schwarz()
\end{checks}

\begin{lemma}\lbl{lem-psi-edges}
 $\psi^{-1}$ acts as follows on the edges of $HF_{16}$:
 \begin{itemize}
  \item $\psi^{-1}(C_0)=(-1,1).i$
  \item $\psi^{-1}(C_1)=(-1,1)$
  \item $\psi^{-1}(C_3)$ is the intersection of $\Dl$ with the circle
   of radius $\rt b/b_-$ centred at $ib_+/b_-$
  \item $\psi^{-1}(C_5)$ is the intersection of $\Dl$ with the circle
   of radius $\rt b_-/b_+$ centred at $(\rt+b_-)/b_+$.
 \end{itemize}
\end{lemma}
\begin{proof}
 Let $A_k$ denote the set that is claimed to be equal to
 $\psi^{-1}(C_k)$.  In each case, it is easy to check that $A_k$ is a
 geodesic.  We also know that $C_k$ is a geodesic and $\psi$ is an
 isometry, so $\psi^{-1}(C_k)$ is also a geodesic.  Because of this,
 it will suffice to check that $|A_k\cap\psi^{-1}(C_k)|\geq 2$, which
 we can do by considering the points
 $\{\psi^{-1}(v_i)\st i\in\{0,3,6,11\}\}$.
 \begin{checks}
  hyperbolic/schwarz_check.mpl: check_schwarz()
 \end{checks}
\end{proof}

We can illustrate the maps $p$, $p_1$, $\phi$ and $\psi$ as follows.
\begin{center}
 \begin{tikzpicture}
  \path[use as bounding box] (-0.5,-9) rectangle(16,6);
  \draw[->] ( 2.0,-2.5) -- ( 2.0,-1.0);
  \draw[->] (11.6,-1.0) -- (11.6,-2.5);
  \draw[->] ( 5.0, 1.5) -- ( 7.0, 1.5);
  \draw[->] ( 5.5,-5.5) -- ( 7.5,-5.5);
  \draw ( 2.0,-1.75) node[anchor=east] {$\psi$};
  \draw (11.6,-1.75) node[anchor=west] {$\phi$};
  \draw ( 6.0, 1.50) node[anchor=south] {$p$};
  \draw ( 6.5,-5.55) node[anchor=north] {$p_1$};
  \begin{scope}[scale=6]
   \draw[blue] (0,0) -- (0.500,0);
   \draw[blue] (0,0) -- (0,0.500);
   \draw[green] (0,0) -- ( 0.625, 0.625);
   \draw[cyan] (0.800,0.800) (0.426,0.426) arc(225:245:0.529);
   \draw[cyan] (0.800,0.800) (0.426,0.426) arc(225:205:0.529);
   \draw[magenta] (1.250,0.000) (0.500,0.000) arc(180:139:0.750);
   \draw[magenta] (0.000,1.250) (0.000,0.500) arc(270:311:0.750);
   \draw[blue] (0.800,0.625) (0.625,0.625) arc( 180:230:0.175);
   \draw[blue] (0.625,0.800) (0.625,0.625) arc( 270:220:0.175);
   \fill[black](0.000,0.000) circle(0.006);
   \fill[black](0.625,0.625) circle(0.006);
   \fill[black](0.322,0.573) circle(0.006);
   \fill[black](0.426,0.426) circle(0.006);
   \fill[black](0.000,0.500) circle(0.006);
   \fill[black](0.500,0.000) circle(0.006);
   \fill[black](0.685,0.493) circle(0.006);
   \fill[black](0.493,0.685) circle(0.006);
   \fill[black](0.573,0.322) circle(0.006);
   \draw( 0.000,-0.030) node{$\ss 0$};
   \draw( 0.650, 0.620) node{$\ss 1$};
   \draw( 0.292, 0.585) node{$\ss 2$};
   \draw( 0.613, 0.312) node{$\ss 3$};
   \draw( 0.386, 0.426) node{$\ss 6$};
   \draw(-0.040, 0.480) node{$\ss 10$};
   \draw( 0.493,-0.030) node{$\ss 11$};
   \draw( 0.500, 0.715) node{$\ss 12$};
   \draw( 0.715, 0.493) node{$\ss 13$};
  \end{scope}
  \begin{scope}[scale=4,xshift=0.5cm,yshift=-1.4cm]
   \draw[magenta] (0.000, 1.890) +(-105:1.604) arc(-105: -75:1.604);
   \draw[magenta] (0.000,-1.890) +( 105:1.604) arc( 105:  75:1.604);
   \draw[blue] ( 1.131, 0.529) +(225:0.748) arc(225:195:0.748);
   \draw[blue] ( 1.131,-0.529) +(135:0.748) arc(135:165:0.748);
   \draw[blue] (-1.131, 0.529) +(315:0.748) arc(315:345:0.748);
   \draw[blue] (-1.131,-0.529) +( 45:0.748) arc( 45: 15:0.748);
   \draw[green] (-0.602,0.000) -- ( 0.602,0.000);
   \draw[cyan] (0.000,-0.286) -- (0.000, 0.286);
   \fill[black]( 0.000, 0.000) circle(0.01);
   \fill[black]( 0.000, 0.286) circle(0.01);
   \fill[black]( 0.000,-0.286) circle(0.01);
   \fill[black]( 0.602, 0.000) circle(0.01);
   \fill[black](-0.602, 0.000) circle(0.01);
   \fill[black]( 0.408, 0.339) circle(0.01);
   \fill[black]( 0.408,-0.339) circle(0.01);
   \fill[black](-0.408, 0.339) circle(0.01);
   \fill[black](-0.408,-0.339) circle(0.01);
   \draw( 0.650, 0.000) node{$\ss 0$};
   \draw(-0.650, 0.000) node{$\ss 1$};
   \draw( 0.000,-0.340) node{$\ss 2$};
   \draw( 0.000, 0.340) node{$\ss 3$};
   \draw( 0.030, 0.030) node{$\ss 6$};
   \draw( 0.420,-0.390) node{$\ss 10$};
   \draw( 0.420, 0.390) node{$\ss 11$};
   \draw(-0.420,-0.390) node{$\ss 12$};
   \draw(-0.420, 0.390) node{$\ss 13$};
  \end{scope}
  \begin{scope}[scale=2.5,xshift=4.5cm]
   \draw[magenta] (-1.333,0.000) -- (-0.750,0.000);
   \draw[magenta] ( 1.333,0.000) -- ( 0.750,0.000);
   \draw[blue] (-1.667,0.000) -- (-1.333,0.000);
   \draw[blue] (-0.750,0.000) -- ( 0.750,0.000);
   \draw[blue] ( 1.667,0.000) -- ( 1.333,0.000);
   \draw[green] (0.000,0.000) -- ( 0.000,1.667);
   \draw[cyan] (0.000,0.000) (1.000,0.000) arc(0:180:1.000);
   \fill[black](-1.333,0.000) circle(0.015);
   \fill[black](-1.000,0.000) circle(0.015);
   \fill[black](-0.750,0.000) circle(0.015);
   \fill[black]( 0.000,0.000) circle(0.015);
   \fill[black]( 0.750,0.000) circle(0.015);
   \fill[black]( 1.000,0.000) circle(0.015);
   \fill[black]( 1.333,0.000) circle(0.015);
   \fill[black]( 0.000,1.000) circle(0.015);
   \draw (-1.333,-0.100) node{$\ss 12$};
   \draw (-1.000,-0.100) node{$\ss 2$};
   \draw (-0.750,-0.100) node{$\ss 10$};
   \draw ( 0.000,-0.100) node{$\ss 0$};
   \draw ( 0.750,-0.100) node{$\ss 11$};
   \draw ( 1.000,-0.100) node{$\ss 3$};
   \draw ( 1.333,-0.100) node{$\ss 13$};
   \draw ( 0.100, 1.100) node{$\ss 6$};
  \end{scope}
  \begin{scope}[scale=2,xshift=5.75cm,yshift=-2.7cm]
   \draw[magenta] ( 0.000, 0.000) ( 23:1.000) arc( 23:157:1.000);
   \draw[magenta] ( 0.000, 0.000) (203:1.000) arc(203:337:1.000);
   \draw[blue]    ( 0.000, 0.000) (157:1.000) arc(157:203:1.000);
   \draw[blue]    ( 0.000, 0.000) (337:1.000) arc(337:383:1.000);
   \draw[green]   (-1.000, 0.000) -- ( 1.000, 0.000);
   \draw[cyan]    ( 0.000,-1.000) -- ( 0.000, 1.000);
   \fill[black] ( 0.000, 0.000) circle(0.02);
   \fill[black] ( 0.000, 1.000) circle(0.02);
   \fill[black] ( 0.000,-1.000) circle(0.02);
   \fill[black] (-1.000, 0.000) circle(0.02);
   \fill[black] ( 1.000, 0.000) circle(0.02);
   \fill[black] ( 157:1.000) circle(0.02);
   \fill[black] ( 203:1.000) circle(0.02);
   \fill[black] ( 337:1.000) circle(0.02);
   \fill[black] ( 383:1.000) circle(0.02);
   \draw( 1.090, 0.000) node{$\ss 0$};
   \draw(-1.090, 0.000) node{$\ss 1$};
   \draw( 0.000, 1.090) node{$\ss 3$};
   \draw( 0.000,-1.090) node{$\ss 2$};
   \draw( 0.050, 0.050) node{$\ss 6$};
   \draw( 337:1.120) node{$\ss 10$};
   \draw( 383:1.120) node{$\ss 11$};
   \draw( 203:1.120) node{$\ss 12$};
   \draw( 157:1.120) node{$\ss 13$};
  \end{scope}
 \end{tikzpicture}
\end{center}

\begin{lemma}\lbl{lem-p-one-props}
 $p_1(-z)=-p_1(z)$ and $p_1(\ov{z})=\ov{p_1(z)}$.  Thus, $p_1(z)$ has
 a Taylor series $\sum_ia_iz^{2i+1}$ with $a_i\in\R$.
\end{lemma}
\begin{proof}
 We know that $p$ is $\tPi$-equivariant, so $p\bt_0\lm\mu=\lm\mu p$.
 Thus, if we put $\pi=\psi^{-1}\bt_0\lm\mu\psi\in\Aut(\Dl)$ and
 $\pi'=\phi\lm\mu\phi^{-1}\in\Aut(\C_\infty)$, we have
 $p_1\pi=\pi'p_1$.  Direct calculation
 shows that $\bt_0\lm\mu$ is the holomorphic involution fixing $v_{H6}$, so
 $\pi$ is the holomorphic involution fixing $v_{HS6}=0$, or in other
 words $\pi(z)=-z$.  Direct calculation also gives $\pi'(z)=-z$, so
 $p_1(-z)=-p_1(z)$ as claimed.

 Similarly, we have $p\lm\nu=\lm\nu p$.  Thus, if we put
 $\xi=\psi^{-1}\lm\nu\psi\in\Aut(\Dl)$ and
 $\xi'=\phi\lm\nu\phi^{-1}\in\Aut(\C_\infty)$, we have
 $p_1\xi=\xi'p_1$.  Here $\lm\nu$ is the antiholomorphic involution
 of $\Dl$ that fixes $v_{H0}$ and $v_{H6}$, so $\xi$ is the antiholomorphic
 involution of $\Dl$ that fixes $v_{HS0}$ and $v_{HS6}$, which gives
 $\xi(z)=\ov{z}$.  Direct calculation also gives $\xi'(z)=\ov{z}$, so
 $p_1(\ov{z})=\ov{p_1(z)}$ as claimed.
\end{proof}

\begin{lemma}\lbl{lem-p-one-poles}
 The set of poles of $p_1$ is $\psi^{-1}(\Pi.\{v_{H1},v_{H9}\})$, and
 all these poles are simple.  Moreover, the points $\pm ib_-/b_+$ are
 poles, and the corresponding residues are equal and are real.
\end{lemma}
\begin{proof}
 We are interested in the preimage of $\infty$ under the composite
 \[ \xymatrix{
     \Dl             \ar[r]^{m_1=\psi}_{\simeq} &
     \Dl             \ar[r]^(0.3){m_2}          &
     \Dl/\Pi = HX(b) \ar[r]^(0.6){m_3}_(0.6){\simeq}      &
     PX(a)           \ar[r]^{m_4}               &
     \C_\infty       \ar[r]^{m_5=\phi}_{\simeq} &
     \C_\infty.
 } \]
 First, we have $m_5^{-1}\{\infty\}=\{\phi^{-1}(\infty)\}=\{-i\}$.  The map
 $m_4\:PX(a)\to\C_\infty$ induces a bijection
 \[ PX(a)/\ip{\lm^2}\to\C_\infty, \]
 and it sends both $v_{P7}$ and $v_{P9}$ to $-i$, so
 $m_4^{-1}\{-i\}=\{v_{P7},v_{P9}\}$.  As $m_3$ is a cromulent
 isomorphism, it follows that
 \[ (m_3m_2)^{-1}\{v_{P7},v_{P9}\}=\Pi\{v_{H7},v_{H9}\}. \]
 It follows easily that the set of poles is as claimed.  The points
 $v_7$ and $v_9$ in $PX(a)$ are not fixed by $\lm^2$, so they are not
 branch points for the map $PX(a)\to\C_\infty$; it follows that all
 the poles are simple.  Now put $\al=b_-/b_+\in\R$.  After unwinding
 the definitions and performing some algebraic simplification we find
 that $\psi^{-1}\bt_0(v_{H7})=i\al$ and
 $\psi^{-1}\bt_2(v_{H9})=-i\al$, which shows that the points
 $\pm i\al$ are poles.  This means that there are constants
 $r_1,r_2\in\C$ and a meromorphic function $q(z)$ on $\Dl$ such that
 $q$ is holomorphic at $\pm i\al$ and
 \[ p_1(z) = \frac{r_1}{z-i\al} + \frac{r_2}{z+i\al} + q(z). \]
 Moreover, the triple $(r_1,r_2,q(z))$ is characterised uniquely by
 these properties.  Next, recall that $p_1(z)=-p_1(-z)$.  This shows that
 $(r_1,r_2,q(z))=(r_2,r_1,-q(-z))$.  Similarly, the fact that
 $p_1(z)=\ov{p_1(\ov{z})}$ shows that
 $(r_1,r_2,q(z))=(\ov{r_2},\ov{r_1},\ov{q(\ov{z})})$.  This means that
 $r_1=r_2\in\R$ as claimed.
\end{proof}

In the case $b=0.75$, the poles can be illustrated as shown below.
The inner dotted circle (with radius $0.6$) is the smallest circle
centred at the origin that contains $\psi^{-1}(HF_4)$.  The only two
poles inside this circle are $\pm ib_-/b_+$; these are shown as solid
red dots.  A further $22$ poles are also shown; they are on or outside
the outer dashed circle, which has radius $0.8$.  All remaining poles
are even closer to the unit circle.
\begin{center}
 \begin{tikzpicture}
  \begin{scope}[scale=3]
   \draw[black] circle(1);
   \draw[black,dotted] circle(0.6);
   \draw[black,dashed] circle(0.8);
   \draw[magenta] (0.000, 1.890) +(-105:1.604) arc(-105: -75:1.604);
   \draw[magenta] (0.000,-1.890) +( 105:1.604) arc( 105:  75:1.604);
   \draw[blue] ( 1.131, 0.529) +(225:0.748) arc(225:195:0.748);
   \draw[blue] ( 1.131,-0.529) +(135:0.748) arc(135:165:0.748);
   \draw[blue] (-1.131, 0.529) +(315:0.748) arc(315:345:0.748);
   \draw[blue] (-1.131,-0.529) +( 45:0.748) arc( 45: 15:0.748);
   \draw[green] (-0.602,0.000) -- ( 0.602,0.000);
   \draw[cyan] (0.000,-0.286) -- (0.000, 0.286);
   \fill[red] ( 0.0000, -0.5292) circle(0.015);
   \fill[red] ( 0.0000,  0.5292) circle(0.015);
   \draw[red] ( 0.0000, -0.9433) circle(0.015);
   \draw[red] ( 0.0000,  0.9433) circle(0.015);
   \draw[red] ( 0.7252, -0.3392) circle(0.015);
   \draw[red] ( 0.5222, -0.7734) circle(0.015);
   \draw[red] ( 0.9283,  0.0950) circle(0.015);
   \draw[red] ( 0.6684,  0.6509) circle(0.015);
   \draw[red] ( 0.2616, -0.8956) circle(0.015);
   \draw[red] (-0.2616, -0.8956) circle(0.015);
   \draw[red] (-0.7252, -0.3392) circle(0.015);
   \draw[red] (-0.5222, -0.7734) circle(0.015);
   \draw[red] (-0.6684,  0.6509) circle(0.015);
   \draw[red] (-0.9283,  0.0950) circle(0.015);
   \draw[red] ( 0.7252,  0.3392) circle(0.015);
   \draw[red] ( 0.5222,  0.7734) circle(0.015);
   \draw[red] ( 0.9283, -0.0950) circle(0.015);
   \draw[red] ( 0.2616,  0.8956) circle(0.015);
   \draw[red] (-0.7252,  0.3392) circle(0.015);
   \draw[red] ( 0.6684, -0.6509) circle(0.015);
   \draw[red] (-0.6684, -0.6509) circle(0.015);
   \draw[red] (-0.9283, -0.0945) circle(0.015);
   \draw[red] (-0.2616,  0.8956) circle(0.015);
   \draw[red] (-0.5222,  0.7734) circle(0.015);
  \end{scope}
 \end{tikzpicture}
\end{center}

\begin{lemma}\lbl{lem-p-one-branches}
 $p'_1(v_{HS0})=p'_1(v_{HS11})=0$.
\end{lemma}
\begin{proof}
 This follows from the fact that the map $HX(b)\to PX(a)\to\C_\infty$
 is $\lm^2$-invariant, and $v_0$ and $v_{11}$ are fixed by $\lm^2$ in
 $HX(b)$.
\end{proof}

\subsection{Finding \texorpdfstring{$a$}{a} from \texorpdfstring{$b$}{b}}
\lbl{sec-a-from-b}

In this section, we describe an algorithm to calculate $a$ from $b$.
This algorithm is implemented by the methods of the class
\mcode+H_to_P_map+, which is defined in the file
\fname+hyperbolic/H_to_P.mpl+.  In more detail, if we want to
take $b=0.75$, we can enter
\begin{mcodeblock}
   HP := `new/H_to_P_map`():
   HP["set_a_H",0.75]:
   HP["make_samples"]:
   HP["find_p1"]:
   HP["a_P"];
   HP["p1"](z);
   HP["err"];
\end{mcodeblock}
The \mcode+find_p1+ method takes about 22 seconds on a fairly capable
PC.  The line \mcode+HP["a_P"]+ returns the value of $a$, which is
about $0.1816$.  The line \mcode+HP["p1"](z)+ returns a rational
function of $z$, with poles only at the points $\pm ib_-/b_+$
mentioned in Lemma~\ref{lem-p-one-poles}.  When restricted to
$\psi^{-1}(HF_4)$, this is a good approximation to $p_1(z)$.  The line
\mcode+HP["err"]+ returns a measure of the quality of approximation,
which is about $7\tm 10^{-11}$ in this case.  It could be improved by
increasing the degree of polynomials and the number of sample points
used in the algorithm; the code has options for this.  The code also
has methods to generate various different visualisations of the
behaviour of $p_1$, and to analyse the errors in more detailed ways.

If one wants to perform the above calculation for several different
values of $b$, and to compare the results with those obtained by the
method of Section~\ref{sec-b-from-a}, then we can instead use the
class \mcode+HP_table+, defined in \fname+hyperbolic/HP_table.mpl+.
For example, we can enter the following:
\begin{mcodeblock}
   HPT := `new/HP_table`():
   HPT["add_a_H",0.75];
   HPT["add_a_H",0.76];
   HPT["add_a_H",0.77];
\end{mcodeblock}
This will perform the above calculation for the valuse $b=0.75$,
$b=0.76$ and $b=0.77$.  The object of class \mcode+H_to_P_map+ for
$b=0.75$ can then be retrieved as \mcode+HPT["H_to_P_maps"][0.75]+.
After calcuating a sufficient range of values of $b$, one can enter
\mcode+HPT["set_spline"]+ and then \mcode+HPT["full_plot"]+ to
generate a plot of $a$ against $b$.

Alternatively (as discussed in Section~\ref{sec-build}), one can read
the file \fname+build_data.mpl+ and execute
\begin{mcodeblock}
   build_data["HP_table"](); 
\end{mcodeblock}
to perform all calculations for
$b=0.06$ to $b=0.94$ in steps of $0.02$, and to do various other
related work.  Here one may wish to enter 
\mcode+infolevel[genus2] := 7;+ before starting the calculation; this
will instruct Maple to print various progress reports as it proceeds.

\begin{lemma}
 There is a unique sequence of polynomials
 $p_{10}(z),p_{11}(z),p_{12}(z)$ and $p_{14}(z)$, such that:
 \begin{itemize}
  \item[(a)] All the polynomials are odd, with real coefficients.
  \item[(b)] The polynomials $p_{10}(z)$ to $p_{12}(z)$ have degree
   $13$.
  \item[(c)] The polynomial $p_{14}(z)$ has degree $15$, and has the
   form $p_{14}(z)=z+O(z^3)$.
  \item[(d)] For all $k$ we have
   $p'_{1k}(v_{HS0})=p'_{1k}(v_{HS11})=0$.
  \item[(e)] Values at $v_{HS0}$, $v_{HS3}$ and $v_{HS11}$ are as
   follows:
   \[ \renewcommand{\arraystretch}{1.5}
      \begin{array}{|c|c|c|c|} \hline
              & v_0 & v_3 & v_{11} \\ \hline
       p_{10} & 1   & i   & 0      \\ \hline
       p_{11} & 0   & 0   & 1      \\ \hline
       p_{12} & 0   & 0   & i      \\ \hline
       p_{14} & 0   & 0   & 0      \\ \hline
      \end{array}
   \]
 \end{itemize}
\end{lemma}
\begin{proof}
 Put $\al=v_{HS0}\in\R$ and $\bt=v_{HS3}/i\in\R$ and $\gm=v_{HS11}$.
 Let $F$ denote the space of all odd polynomials of degree at most
 $13$, and note that this has dimension $7$.  Note also that for
 $f\in F$ we automatically have $f(\R)\sse\R$ and $f(i\R)\sse i\R$ and
 $f'(i\R)\sse\R$, and the roots and their multiplicities are invariant
 under the maps $z\mapsto -z$ and $z\mapsto\ov{z}$.  Using this, we
 see that the only possibility for $p_{14}(z)$ is the polynomial
 \[ p_{14}(z) =
     z(1-z^2/\al^2)^2(1+z^2/\bt^2)(1-z^2/\gm^2)^2(1-z^2/\ov{\gm}^2)^2.
 \]
 Next, we can define $\ep\:F\to\R^7$ by
 \[ \ep(f) = (f(\al),\;
              f(i\bt)/i,\;
              \text{Re}(f(\gm)),\;
              \text{Im}(f(\gm)),\;
              f'(\al),
              \text{Re}(f'(\gm)),\;
              \text{Im}(f'(\gm))).
 \]
 If $\ep(f)=0$ then $f$ must be divisible by $p_{14}(z)$, but we can
 then compare degrees to see that $f=0$.  This means that $\ep$ is an
 injective linear map between spaces of dimension $7$, so it is an
 isomorphism.  The polynomials $p_{10}(z)$, $p_{11}(z)$ and
 $p_{12}(z)$ can be obtained by applying $\ep^{-1}$ to suitable
 vectors in $\R^7$.
\end{proof}

\begin{definition}
 We put
 \[ p_{15}(z) =
     (z-ib_-/b_+)(z+ib_-/b_+) = z^2 + \frac{1-b^2}{1+b^2},
 \]
 and then $p_{13}(z)=p_{14}(z)/p_{15}(z)$.  Then, given $a\in\R^d$, we
 put
 \[ P(a)(z) = p_{10}(z) + \sum_{i=1}^3 a_ip_{1i}(z) +
     p_{14}(z)\sum_{i=4}^da_iz^{2(i-4)}.
 \]
\end{definition}

We now choose $\al\in(0,1)$ which we believe is a reasonable
approximation to the required value of $a$.  It would not be too
harmful to just take $\al=1/2$.  Alternatively, the code defines a
polynomial \mcode+f(t)=schwarz_b_approx(t)+ of degree ten, which is a
good approximation to the function $a\mapsto b$; we can thus find
$\al$ by solving $f(\al)=b$ numerically.  We then put
$z_0=(i-\al)/(i+\al)$, which is the value of $v_{PS11}$ corresponding
to $a=\al$.

Now consider the map $p_1(z)$.  Let $a_1$ and $a_2$ be the real and
imaginary parts of $p_1(v_{HS11})$, and let $a_3$ be the unique
real constant such $p_1(z)-a_3p_{13}(z)$ has residue zero at
$ib_-/b_+$.  We then find that the function
$p_1(z)-p_{10}(z)-\sum_{i=1}^3a_ip_{1i}(z)$ is holomorphic on a disc
$\Dl'$ centred at $0$ that includes all of $HF_4$.  Moreover, it is
odd, with real Taylor coefficients.  By considering it order of
vanishing at the various points $v_{HSj}$, we see that is the product
of $p_{14}(z)$ with an even function that is also holomorphic on $\Dl'$.
This means that when $d$ is sufficiently large, $p_1(z)$ can be well
approximated by $P(a)(z)$ for some $a\in\R^d$.  To find $a$, we note
that
\[ p_1(C_{HS3}\cup C_{HS5})=C_{PS3}\cup C_{PS5} = S^1. \]
We therefore choose a reasonably large number $n$ (say $n=200$) and a
list of closely spaced points $s=(s_1,\dotsc,s_n)$ lying in
$\psi^{-1}(HF_4\cap(C_{H3}\cup C_{H5}))$.  We then define
$\eta\:\R^d\to\R^n$ by
\[ \eta(a)_j = |P(a)(s_j)|^2 - 1. \]
It is not hard to see that this has the form
\[ \eta(a)_j = |(Ma+c)_j|^2 - 1 \]
for some matrix $M\in M_{nd}(\C)$ and some vector $c\in\C^n$ which can
be precomputed.  Using this, we get
\[ \frac{\partial}{\partial a_k}\eta(a)_j =
    2\text{Re}(\ov{M_{jk}}(Ma+c)_j).
\]
This makes it easy to minimise $\|\eta(a)\|^2$ by an iterative process.

In the case $b\simeq 0.80053190489236$ which is relevant for
uniformising $EX^*$, we have taken $d=50$ and $n=300$, and have ended
up with errors $|P(a)(s_j)|^2-1$ that are less than $10^{-23}$.

Calculations using the above method give the following graph of $a$ as
a function of $b$ (with the marked point representing $EX^*$):
\begin{center}
 \begin{tikzpicture}[scale=4]
  \draw[black,->] (-0.05,0) -- (1.05,0);
  \draw[black,->] (0,-0.05) -- (0,1.05);
  \draw[black] (1,-0.05) -- (1,0);
  \draw[black] (-0.05,1) -- (0,1);
  \draw ( 0.00,-0.05) node[anchor=north] {$0$};
  \draw ( 1.00,-0.05) node[anchor=north] {$1$};
  \draw (-0.05, 0.00) node[anchor=east ] {$0$};
  \draw (-0.05, 1.00) node[anchor=east ] {$1$};
  \draw ( 1.05, 0.00) node[anchor=west ] {$b$};
  \draw ( 0.00, 1.05) node[anchor=south] {$a$};
 \draw[red] plot[smooth] coordinates{ (0.000,1.000) (0.060,1.000) (0.080,1.000) (0.100,1.000) (0.120,1.000) (0.140,1.000) (0.160,1.000) (0.180,1.000) (0.200,1.000) (0.220,1.000) (0.240,0.999) (0.260,0.997) (0.280,0.994) (0.300,0.990) (0.320,0.983) (0.340,0.974) (0.360,0.961) (0.380,0.944) (0.400,0.923) (0.420,0.898) (0.440,0.869) (0.460,0.835) (0.480,0.798) (0.500,0.757) (0.520,0.713) (0.540,0.667) (0.560,0.620) (0.580,0.571) (0.600,0.521) (0.620,0.472) (0.640,0.423) (0.660,0.375) (0.680,0.328) (0.700,0.283) (0.720,0.241) (0.740,0.201) (0.760,0.163) (0.780,0.129) (0.800,0.099) (0.820,0.073) (0.840,0.050) (0.860,0.032) (0.880,0.019) (0.900,0.009) (0.920,0.004) (0.940,0.001) (1.000,0.000) };
  \fill[black] (0.801,0.098) circle(0.015);
 \end{tikzpicture}
\end{center}
The middle section of this graph can be obtained by the methods of
this section, or those of Section~\ref{sec-b-from-a}.  However, the
methods of Section~\ref{sec-b-from-a} behave poorly when $(a,b)$
approaches $(1,0)$.

An optimist might hope for an explicit formula for the above graph,
perhaps involving the elliptic integrals from
Section~\ref{sec-P}, the constants $s_k$ in
Definition~\ref{defn-H-curves}, or other standard special
functions or quantities found elsewhere in this document.  We have
performed a fairly extensive experimental search for such formulae,
but without success.  Of course one can find polynomials of high
degree fitting the graphs to any desired accuracy, but the answers are
not illuminating.  It might be helpful if we had a meaningful
interpretation of the endpoints $(a,b)=(1,0)$ and $(a,b)=(0,1)$,
perhaps in terms of a Deligne-Mumford compactification of an
appropriate moduli space of stable curves, but we have not
investigated this seriously.

\subsection{Recollections on the Schwarzian derivative}
\lbl{sec-schwarz}

We now start to discuss our second method, where we are given
$a\in(0,1)$ and we try to find $b\in (0,1)$ such that
$PX(a)\simeq HX(b)$.  This method is based on the Schwarzian
derivative, whose definition and properties we now recall.

\begin{definition}
 Let $f$ be a nonconstant meromorphic function on a connected domain
 $U\sse\C$.  The Schwarzian derivative $S(f)$ is defined by
 \[ S(f) = (f''/f')' - \tfrac{1}{2} (f''/f')^2
         = f'''/f' - \tfrac{3}{2} (f''/f')^2.
 \]
\end{definition}

\begin{proposition}\lbl{prop-vanishing-schwarzian}
 We have $S(f)=0$ if and only if there are constants $a,b,c,d\in\C$
 with $ad-bc\neq 0$ and $f(z)=(az+b)/(cz+d)$ on $U$ (or in other
 words, $f$ is a M\"obius function).
\end{proposition}
\begin{proof}
 This is standard.  One can check by direct calculation that
 functions $f(z)=(az+b)/(cz+d)$ have $S(f)=0$.  Conversely, suppose
 that $f$ is meromorphic on $U$ with $S(f)=0$.  Then the function
 $g=f''/f'$ satisfies $g'=g^2/2$.  If $g=0$ then we deduce that
 $f''=0$ so $f$ has the form $f(z)=az+b$ for some $a$ and $b$, as
 required.  Otherwise, we can consider the meromorphic function $1/g$
 and we deduce that $(1/g)'=-g'/g^2=-1/2$, which gives
 $g(z)=-2/(z+\dl)$ for some constant $\dl$, or in other words
 $(z+\dl)f''(z)=-2f'(z)$.  From this it follows that the function
 $(z+\dl)^2f'(z)$ has zero derivative and so is constant.  This gives
 that $f'(z)=\gm/(z+\dl)^2$ and then $f(z)=\bt-\gm/(z+\dl)$ for some
 constants $\bt$ and $\dl$, and this can evidently be rewritten in the
 form $(az+b)/(cz+d)$.
\end{proof}

\begin{proposition}
 Given meromorphic functions $V\xra{g}U\xra{f}\C$ we have
 \[ S(f\circ g) = (S(f)\circ g) \cdot (g')^2 + S(g). \]
 (We call this the \emph{Schwarzian chain rule}.)  In particular, if
 $f$ is a M\"obius function then $S(f\circ g)=S(g)$.
\end{proposition}
\begin{proof}
 This is also standard.  We can use the ordinary chain rule repeatedly to
 express the first three derivatives of $f\circ g$ in terms of those
 of $f$ and $g$, and the rest is pure algebra.
\end{proof}

We next recall the relationship between the Schwarzian derivative and
certain types of second order linear differential equations.  Again,
almost all of this is well-known, but we give a self-contained account
to serve as a convenient basis for discussing some additional points
that are less standard.

\begin{proposition}\lbl{prop-schwarzian-solutions}
 Let $U$ be a simply connected open subset of $\C$, and let $s$ be a
 holomorphic function on $U$.  Put
 \begin{align*}
  F &= \{ \text{ meromorphic functions $f$ on $U$
                  such that $f''+\half sf=0$} \} \\
  G &= \{ \text{ nonconstant meromorphic functions $g$ on $U$
                  such that $S(g)=s$} \}.
 \end{align*}
 Then
 \begin{itemize}
  \item[(a)] Every function in $F$ is actually holomorphic.
  \item[(b)] For any $z_0\in U$ and any $u_0,u_1\in\C$ there is a
   unique function $f\in F$ such that
   $f(z)=u_0+u_1(z-z_0)+O((z-z_0)^2)$.  In particular, $F$ has
   dimension two over $\C$.
  \item[(c)] $G$ is precisely the set of functions of the form
   $f_1/f_0$, where $f_0$ and $f_1$ are linearly independent elements
   of $F$.
  \item[(d)] For any $z_0\in U$ and any $v_0,v_1,v_2\in\C$ with
   $v_1\neq 0$ there is a unique function $g\in G$ such that
   $g(z)=v_0+v_1(z-z_0)+v_2(z-z_0)^2+O((z-z_0)^3)$.
 \end{itemize}
\end{proposition}
\begin{proof}
 \begin{itemize}
  \item[(a)] This is Lemma~\ref{lem-F-holomorphic}.
  \item[(b)] This is Corollary~\ref{cor-F-disc}.
  \item[(c)] Combine Lemma~\ref{lem-quotient-schwarzian} and
   Corollary~\ref{cor-quotient-schwarzian}.
 \end{itemize}
\end{proof}

\begin{lemma}\lbl{lem-F-holomorphic}
 Every function in $F$ is holomorphic.
\end{lemma}
\begin{proof}
 Consider an element $f\in F$ and a point $z_0\in U$.  If $f$ is not
 holomorphic at $z_0$, then it must have a pole of order $d>0$ at
 $z_0$, so $f''$ has a pole of order $d+2$, whereas $\half sf$ has a
 pole of order at most $d$ (or is holomorphic).  This contradicts the
 equation $f''+\half sf=0$.
\end{proof}

\begin{lemma}\lbl{lem-F-disc}
 Let $U$ be the open disc centred at $z_0$ with radius $r>0$, and let
 $s$ be a holomorphic function on $U$.  Define $F$ and $G$ as above.
 Then for any $u_0,u_1\in\C$ there is a unique function $f\in F$ such
 that $f(z)=u_0+u_1(z-z_0)+O((z-z_0)^2)$.  In particular, $F$ has
 dimension two over $\C$.
\end{lemma}
\begin{proof}
 We can expand $s$ as a power series, say $s(z)=\sum_ks_k(z-z_0)^k$.
 Put $a_0=u_0$ and $a_1=u_1$, then define $a_k$ recursively for $k>1$
 by
 \[ a_k = \frac{-1}{2k(k-1)}\sum_{j=0}^{k-2}a_js_{k-2-j}. \] It is
 then easy to see that the formal power series
 $f_0(z)=\sum_ka_k(z-z_0)^k$ is the unique one with $f''_0+\half
 sf_0=0$ and $f_0(z)=u_0+u_1(z-z_0)+O((z-z_0)^2)$.  Moreover, the
 coefficients $a_k$ grow at a rate comparable to that of the
 coefficients $s_k$, so they have the same radius of convergence.
 Thus, the above expression defines a holomorphic function on $U$.
\end{proof}

\begin{corollary}\lbl{cor-F-disc}
 Claim~(b) in Proposition~\ref{prop-schwarzian-solutions} holds for
 an arbitrary simply connected domain $U$.
\end{corollary}
\begin{proof}
 This follows by analytic continuation.  In more detail, for any
 $z\in U$ we can choose a path $\gm$ from $z_0$ to $z$ in $U$.  We can
 then choose closely spaced points $\gm(t_i)$ and radii $r_i>0$ such
 that $t_0=1$ and $t_N=1$ and the disc $U_i$ of radius $r_i$ centred
 at $\gm(t_i)$ is contained in $U$ and contains $\gm(t_{i+1})$.  We
 then let $f_0$ be the unique holomorphic function on $U_0$ with
 $f''_0+\half sf_0=0$ and $f_0(z)=u_0+u_1(z-z_0)+O((z-z_0)^2)$.  Using
 this as a starting point, we let $f_k$ denote the unique holomorphic
 function on $U_k$ that satisfies $f''_k+\half sf_k=0$ and agrees with
 $f_{k-1}$ to second order at $\gm(t_k)$.  We then define
 $f(z)=f_N(z)$.  It is easy to check that this does not depend on the
 precise choice of points $t_i$, nor does it change if we move $\gm$
 by a small homotopy fixing the endpoints.  As $U$ is simply connected
 any two paths from $z_0$ to $z$ are homotopic relative to endpoints,
 and any homotopy can be broken down into small homotopies.  It
 follows that $f(z)$ is independent of all choices, and it defines an
 element of $F$ with the required behaviour near $z_0$.  Any other
 such element will agree with $f$ at least on $U_0$, but then it must
 agree everywhere on $U$ by analytic continuation.
\end{proof}

\begin{lemma}\lbl{lem-quotient-schwarzian}
 Suppose that $f_0$ and $f_1$ are linearly independent elements of
 $F$, so the quotient $g=f_1/f_0$ is a nonconstant meromorphic
 function.  Then $S(g)=s$.
\end{lemma}
\begin{proof}
 Put $W=f_0f'_1-f'_0f_1$ (the Wronskian of $f_0$ and $f_1$).  This is
 easily seen to satisfy $W'=0$, so it is constant.  The quotient
 $f=f_1/f_0$ has $f'=W/f_0^2$, so $f''=-2Wf'_0/f_0^3$, so
 $f''/f'=-2f'_0/f_0$.  From this we find that $S(f)=s$ as claimed.
\end{proof}

\begin{lemma}\lbl{lem-G-transitive}
 If $g,h\in G$ then $g=m\circ h$ for some M\"obius function $m$.
\end{lemma}
\begin{proof}
 As $h$ is nonconstant and meromorphic, we can choose a small disc
 $V\sse U$ such that $h$ is holomorphic on $V$ and $h'$ is nonzero
 everywhere in $V$.  After shrinking $V$ if necessary, we can then
 assume that the map $h\:V\to h(V)$ is a conformal isomorphism.  Put
 $m=g\circ h^{-1}$, which is meromorphic on $h(V)$ and satisfies
 $g=m\circ h$.  The Schwarzian chain rule gives
 $S(g)=(S(m)\circ h)\cdot(h')^2+S(h)$ on $V$.  However, $S(g)=S(h)=s$
 and $h'$ is nowhere zero, so $S(m)\circ h=0$ on $V$, so $S(m)=0$ on
 $h(V)$.  It follows that $m$ is a M\"obius function.  The functions
 $g$ and $m\circ h$ are both meromorphic and they agree on a disc so
 they must agree everywhere in $U$.
\end{proof}

\begin{corollary}\lbl{cor-quotient-schwarzian}
 Every element $g\in G$ can be written as $g=f_1/f_0$ for some
 linearly independent pair of elements $f_0,f_1\in F$.
\end{corollary}
\begin{proof}
 Let $e_0$ and $e_1$ be any basis for $F$, and put $h=e_1/e_0$.
 Lemma~\ref{lem-quotient-schwarzian} tells us that $h\in G$, so
 Lemma~\ref{lem-G-transitive} tells us that $g=(ah+b)/(ch+d)$ for some
 $a,b,c,d$ with $ad-bc\neq 0$.  This means that the functions
 $f_1=ae_1+be_0$ and $f_0=ce_1+de_0$ are linearly independent elements
 of $F$ with $g=f_1/f_0$.
\end{proof}

\begin{lemma}\lbl{lem-mobius-straighten}
 Suppose that $f$ is holomorphic at $z_0$, with $f'(z_0)\neq 0$.  Then
 there is a unique M\"obius function $m$ such that
 $m(f(z))=z-z_0+O((z-z_0)^3)$.  Specifically, if
 \[ f(z) = u_0 + u_1(z-z_0) + u_2(z-z_0)^2 + O((z-z_0)^3) \]
 then
 \[ m(z) = \frac{u_1(z-u_0)}{u_2(z-u_0)+u_1^2}. \]
\end{lemma}
\begin{proof}
 If we define $m(z)$ as above, then it is straightforward to check
 that $m(f(z))=z-z_0+O((z-z_0)^3)$.  If $n$ is another M\"obius
 function with $n(f(z))=z-z_0+O((z-z_0)^3)$ then the function
 $k=n\circ m^{-1}$ must have the form $k(z)=(az+b)/(cz+d)$ for some
 $a,b,c,d$, but also $k(z)=z+O(z^3)$.  In particular, we have
 $k(0)=0$, which gives $b=0$.  We then have $k'(0)=1$, which gives
 $a=d$.  After cancelling we may assume that $a=d=1$.  Finally, we
 have $k''(0)=0$, so $c=0$, so $k$ is the identity as required.
\end{proof}

\begin{corollary}\lbl{cor-mobius-straighten}
 For any $z_0\in U$ and any $v_0,v_1,v_2\in\C$ with $v_1\neq 0$ there
 is a unique function $g\in G$ such that
 $g(z)=v_0+v_1(z-z_0)+v_2(z-z_0)^2+O((z-z_0)^3)$.
\end{corollary}
\begin{proof}
 It will be harmless, and notationally convenient, to assume that
 $z_0=0$.  By Corollary~\ref{cor-F-disc} we can choose
 $f_0,f_1\in F$ with $f_0(z)=1+O(z^2)$ and $f_1(z)=z+O(z^2)$.  The
 function $g_0=f_1/f_0$ now lies in $G$ and has $g'_0(0)\neq 0$, so we
 can find a M\"obius function $m_1$ such that the function
 $g_1(z)=m_1(g_0(z))$ satisfies $g_1(z)=z+O(z^3)$.  We then put
 \[ m_2(z)=(v_0v_1-(v_1^2-v_0v_2)z)/(v_1-v_2z) \]
 and $g(z)=m_2(g_1(z))$.
\end{proof}

The next result refers to circles in $\C_\infty$.  Here we regard
straight lines as circles of infinite radius.  With this convention,
it is well-known that M\"obius functions send circles to circles.
\begin{proposition}\lbl{prop-image-circle}
 Suppose that a real interval $(a,b)$ is contained in $U$, and that
 $f'(t)\neq 0$ for all $t\in(a,b)$, so $S(f)$ is holomorphic on
 $(a,b)$.  Then the following are equivalent:
 \begin{itemize}
  \item[(a)] $f((a,b))$ is contained in some circle
   $C\subset\C_\infty$.
  \item[(b)] $S(f)$ is real on $(a,b)$
 \end{itemize}
\end{proposition}
\begin{proof}
 Suppose that~(a) holds.  We can choose a M\"obius function $m$ such
 that $m(C)=\R$, and put $g=m\circ f$, so $g((a,b))\sse\R$.  From the
 definition of $S(g)$ it is clear that $S(g)$ is real on $(a,b)$, but
 the Schwarzian chain rule shows that $S(f)=S(g)$, so~(b) holds.

 Conversely, suppose that the function $s=S(f)$ is real on $(a,b)$.
 Choose a point $t_0\in(a,b)$.  Lemma~\ref{lem-mobius-straighten}
 gives us a M\"obius function $m$ such that the function $g=m\circ f$
 has $g(t)=(t-t_0)+O((t-t_0)^3)$.  Note also that $S(g)$ is again
 equal to $s$.  Now put $h(z)=\ov{g(\ov{z})}$, and note that this is
 again meromorphic.  Using power series representations we can see
 that $S(h)=r$, where $r(z)=\ov{s(\ov{z})}$.  Now $r$ is also
 holomorphic, and agrees with $s$ on $(a,b)$, so it must agree with
 $s$ everywhere in $U$.  This means that $g$ and $h$ are both elements
 of $G$, so Lemma~\ref{lem-G-transitive} gives us a M\"obius
 function $n$ with $g=n\circ h$.  However, both $g(t)$ and $h(t)$ are
 of the form $(t-t_0)+O((t-t_0)^3)$.  It follows that $n(z)=z+O(z^3)$,
 and thus that $n$ is the identity, so $g=h$.  This means that
 $g((a,b))\sse\R$, so $f((a,b))\sse m^{-1}(\R)$.  Moreover, as $m$ is
 a M\"obius function, the set $m^{-1}(\R)$ is a circle as required.
\end{proof}

\begin{proposition}
 Let $U$ be the disc of radius $r$ centred at $z_0$, and let $s(z)$ be
 a function that is holomorphic on $U\sm\{z_0\}$ and has Laurent
 expansion $\sum_{k\geq -2}s_k(z-z_0)^k$ with $s_{-2}=3/8$.  Let $U'$
 be obtained from $U$ by removing a line segment from $z_0$ to the
 edge, and let $\xi(z)$ be a holomorphic branch of $(z-z_0)^{1/4}$ on
 $U'$.  Let $F'$ be the space of holomorphic solutions of
 $f''+\half sf=0$ on $U'$.  Then
 \begin{itemize}
  \item There is a unique holomorphic function $e_0$ on $U$ such that
   $e_0(0)=1$ and the function $f_0=e_0\xi$ lies in $F'$.
  \item There is a unique holomorphic function $e_1$ on $U$ such that
   $e_1(0)=1$ and the function $f_1=e_1\xi^3$ lies in $F'$.
 \end{itemize}
\end{proposition}
\begin{proof}
 It will be harmless, and notationally convenient, to assume that
 $z_0=0$.

 By assumption we have $\xi(z)^4=z$, so $4\xi^3\xi'=1$, so
 $\xi'(z)=1/(4\xi(z)^3)=\xi(z)/(4z)$.

 Consider a function of the form $f=e\xi$, where
 $e(z)=\sum_{k\geq 0}a_kz^k$ (and we take $a_k=0$ for $k<0$).  We then have
 \[ f'(z)=e'(z)\xi(z)+e(z)\xi'(z)=(e'(z)+\tfrac{1}{4}e(z)z^{-1})\xi(z). \]
 By a similar argument, we have
 \[ f''(z) =
     (e''(z) + \tfrac{1}{2}e'(z)z^{-1} -
     \tfrac{3}{16}e(z)z^{-2})\xi(z).
 \]
 Thus, the equation $f''+\half sf=0$ is equivalent to
 \[ e''(z) + \half z^{-1} e'(z) + \half t(z) e(z) = 0, \]
 where
 \[ t(z) = s(z) - \tfrac{3}{8} z^{-2} = \sum_{k\geq -1}s_kz^k. \]
 Looking at the coefficient of $z^k$, we get
 \[ (k+\tfrac{3}{2})(k+2) a_{k+2} +
      \half\sum_{j=0}^{k+1}s_{k-j}a_j = 0.
 \]
 For $k\leq -2$ this is trivially satisfied, and for $k\geq -1$ it can be
 rearranged to express $a_{k+2}$ in terms of $a_0,\dotsc,a_{k+1}$.  It
 follows that there is a unique power series solution with $a_0=1$.
 The rate of growth of the coefficients $a_k$ is comparable with that
 of the coefficients $s_k$, so the series $\sum_ka_kz^k$ converges to
 give the claimed function $e_0(z)$.

 Now consider instead a function of the form $f=e\xi^3$.  We find in
 the same way that
 \[ f''(z) =
     (e''(z) + \tfrac{3}{2}e'(z)z^{-1} -
     \tfrac{3}{16}e(z)z^{-2})\xi(z)^3,
 \]
 and thus that the equation $f''+\half sf=0$ is equivalent to
 \[ e''(z) + \tfrac{3}{2} z^{-1} e'(z) + \half t(z) e(z) = 0, \]
 or
 \[ (k+\tfrac{5}{2})(k+2) a_{k+2} +
      \half\sum_{j=0}^{k+1}s_{k-j}a_j = 0.
 \]
 Just as in the previous case, there is a unique solution, as
 claimed.  The functions $f_0$ and $f_1$ are linearly independent (as
 we can see by considering their rate of growth near $z=0$), and we
 have seen that $F'$ has dimension two, so they must form a basis.
\end{proof}

\subsection{Application to cromulent surfaces}
\lbl{sec-b-from-a}

For each $a\in(0,1)$ we have shown that there is a unique $b\in(0,1)$
such that $PX(a)\simeq HX(b)$ as cromulent surfaces, and in fact there
is a unique cromulent isomorphism $HX(b)\to PX(a)$.  As in
Section~\ref{sec-P-H-prelim}, we write $p$ for the canonical map
\[ \Dl \to \Dl/\Pi = HX(b) \to PX(a) \to
    PX(a)/\ip{\lm}^2 \to \C_\infty,
\]
and we also consider the map $p_1=\phi p\psi\:\Dl\to\C_\infty$.

In this section, we describe an algorithm to calculate $b$ from $a$.
This algorithm is implemented by the methods of the class
\mcode+P_to_H_map+, which is defined in the file
\fname+hyperbolic/P_to_H.mpl+.  In more detail, if we want to
take $a=0.1$, we can enter
\begin{mcodeblock}
   PH := `new/P_to_H_map`():
   PH["set_a_P",0.1]:
   PH["add_charts"]:
   PH["find_p1_inv"];
   PH["a_H"];
   PH["p1_inv"](z);
   PH["err"];
\end{mcodeblock}
The line \mcode+PH["a_H"]+ returns the value of $b$, which is
about $0.7994$.  The line \mcode+PH["p1_inv"](z)+ returns a polynomial
in $z$, which is a good approximation to $p_1^{-1}(z)$.  The line
\mcode+HP["err"]+ returns a measure of the quality of approximation,
which is about $10^{-21}$ in this case.

If one wants to perform the above calculation for several different
values of $a$, and to compare the results with those obtained by the
method of Section~\ref{sec-a-from-b}, then we can instead use the
class \mcode+HP_table+, defined in \fname+hyperbolic/HP_table.mpl+.
For example, we can enter the following:
\begin{mcodeblock}
   HPT := `new/HP_table`():
   HPT["add_a_P",0.1];
   HPT["add_a_P",0.2];
   HPT["add_a_P",0.3];
\end{mcodeblock}
This will perform the above calculation for the valuse $a=0.1$,
$a=0.2$ and $a=0.3$.  The object of class \mcode+P_to_H_map+ for
$a=0.1$ can then be retrieved as \mcode+HPT["P_to_H_maps"][0.1]+.

The function \mcode+build_data["HP_table"]()+ (defined in
\fname+build_data.mpl+) does the calculation for $a$ from $0.06$ to
$0.94$ in steps of $0.02$ (as well as implementing the method of
Section~\ref{sec-a-from-b} and performing various other work).

We can think of the inverse of $p$ as giving a multivalued function
$p^{-1}\:\C_\infty\to\Dl$.  It is a key point that the Schwarzian
derivative $S(p^{-1})$ is single-valued, and in fact is a rational
function whose properties we can understand quite explicitly.  To
explain this, we will use the following definitions
\begin{definition}\lbl{defn-schwarz-s}
 We put
 \begin{align*}
  B &= \{v_{Ci}\st i\in\{0,1,10,11,12,13\}\} \\
    &= \{0,\infty,a,-a,1/a,-1/a\} \subset\C_\infty \\
  B_0 &= B\sm\{\infty\} = \{0,a,-a,1/a,-1/a\} \\
  r_a(z) &= \prod_{u\in B_0}(z-u) = z^5-(a^2+a^{-2})z^3+z \\
  s_0(z) &= \frac{-3z^3/2}{r_a(z)} +
             \frac{3}{8}\sum_{u\in B_0}\frac{1}{(z-u)^2} \\
  s_1(z) &= \frac{z}{r_a(z)}.
 \end{align*}
\end{definition}
We have seen $r_a(z)$ before; the definition is repeated for ease of
reference and to display the connection with $B$ and $B_0$.  Recall
also that $B$ is the set of critical values of $p$, so $p$
restricts to give a covering map $\Dl\sm p^{-1}(B)\to\C_\infty\sm B$.
For any sufficiently small connected open set
$U\subset\C_\infty\sm B$, we can choose a holomorphic map $f\:U\to\Dl$
with $pf=1_U$ (so $f$ is a local branch of $p^{-1}$); we then define
$S(p^{-1})_U=S(f)$.  This is independent of the choice of $f$, because
any other choice has the form $m\circ f$ for some M\"obius map
$m\in\ip{\Pi,\lm^2}$, and the Schwarzian chain rule gives
$S(m\circ f)=S(f)$.  Given this invariance, it is clear that
$S(p^{-1})_U=S(p^{-1})_V$ on $U\cap V$.  We can therefore patch these
local functions together to get a meromorphic function $S(p^{-1})$ on
$\C_\infty\sm B$.

\begin{proposition}\lbl{prop-S-p-inv}
 There is a constant $d\in\R$ such that $S(p^{-1})=s_0+ds_1$ (where
 $s_0$ and $s_1$ are as in Definition~\ref{defn-schwarz-s}).
\end{proposition}
\begin{proof}
 For convenience, we write $s(z)=S(p^{-1})(z)$.  Put
 \[ d(z)=(s(z)-s_0(z))/s_1(z)=(s(z)-s_0(z))r_a(z)/z. \]
 The claim is that this is a real constant.

 First let $U$ and $f$ be as in our definition of $s(z)$.  The
 equation $pf=1_U$ implies that $f'$ is nonzero everywhere in $U$, so
 the Schwarzian derivative $s=f'''/f'-\tfrac{3}{2}(f''/f')^2$ is
 holomorphic in $U$.  This shows that all singularities of $s$ must
 lie in $B$.

 Now consider a point $u\in B_0$, and choose $\tu\in\Dl$ with
 $p(\tu)=u$.  By a standard argument with branched double covers, we
 have $p(z)=u+c(z-\tu)^2+O((z-\tu)^3)$ for some $c\neq 0$.  It follows
 that on any small open set close to $u$, we have
 $p^{-1}(z)=((z-u)/c)^{1/2}+O((z-u)^{3/2})$.  Computing the
 Schwarzian derivative from this approximation gives
 $s(z)=\tfrac{3}{8}(z-u)^{-2}+O((z-u)^{-1})$, which matches the
 behaviour of $s_0(z)$.  We therefore see that $s(z)-s_0(z)$ has at
 worst a simple pole at $u$.  As this holds for all $u\in B_0$, we see
 that the product $e(z)=(s(z)-s_0(z))r_a(z)=z\,d(z)$ is holomorphic
 everywhere in $\C$, so $d(z)$ has at worst a simple pole at $z=0$.

 Next, recall that there is an element $\mu\in\tPi$ that satisfies
 $p(\mu(z))=1/p(z)$ for all $z\in\Dl$.  In other words, if we define
 $\mu_0\:\C_\infty\to\C_\infty$ by $\mu_0(z)=1/z$, then we have
 $\mu_0\circ p=p\circ\mu$, and so $p^{-1}\circ\mu_0=\mu\circ p^{-1}$.
 Here $\mu$ and $\mu_0$ are M\"obius maps so the Schwarzian chain
 rule gives $(S(p^{-1})\circ\mu_0)\;(\mu_0')^2=S(p^{-1})$, and thus
 $s(z^{-1})=z^4s(z)$.  A similar argument with $\lm$ gives
 $s(-z)=s(z)$.  By direct calculation we also see that
 \begin{align*}
  s_0(z^{-1}) &= z^4s_0(z^{-1}) & s_0(-z) &= s_0(z) \\
  s_1(z^{-1}) &= z^4s_1(z^{-1}) & s_1(-z) &= s_1(z).
 \end{align*}
 It follows that $d(z^{-1})=d(-z)=d(z)$.  As $d$ is even and has at
 worst a simple pole at $0$, it must actually be holomorphic at $0$.
 It is also holomorphic elsewhere in $\C$ and it satisfies
 $d(z^{-1})=d(z)$ so it must be bounded on $\C$ and thus constant.

 Finally, recall that $p$ is equivariant with respect to the map
 $\nu$, or in other words $p(\ov{z})=\ov{p(z)}$.  This means that $p$
 is real on $\R\cap\Dl$, so $p^{-1}$ can be chosen to be real on $\R$,
 so $S(p^{-1})$ is real on $\R$.  From this it is clear that $d\in\R$.
\end{proof}
\begin{remark}
 Maple notation for $d$ and $s(z)$ is \mcode+d_p_inv+ and
 \mcode+S_p_inv+.
\end{remark}
\begin{remark}
 We can regard the parameter $d$ in the proposition as a function of
 $a$ or of $b$.  Numerical calculations (by a method to be described
 below) suggest that $d$ grows like $a^{-2}$ as $a\to 0$, and give the
 following graph of $a^2d$ against $a$:
 \begin{center}
  \begin{tikzpicture}[xscale=10,yscale=20]
   \draw[->] (-0.02,0.45) -- (1.05,0.45);
   \draw[->] (0,0.4) -- (0,0.63);
   \draw (0,0.63) node[anchor=south] {$\ss a^2d$};
   \draw (1.07,0.45) node[anchor=west] {$\ss a$};
   \draw (-0.02,0.45) node[anchor=east] {$\ss 0.45$};
   \draw (-0.02,0.50) node[anchor=east] {$\ss 0.50$};
   \draw (-0.02,0.55) node[anchor=east] {$\ss 0.55$};
   \draw (-0.02,0.60) node[anchor=east] {$\ss 0.60$};
   \draw (0,0.45) -- (-0.02,0.45);
   \draw (0,0.50) -- (-0.02,0.50);
   \draw (0,0.55) -- (-0.02,0.55);
   \draw (0,0.60) -- (-0.02,0.60);
   \draw (0.1,0.45) -- (0.1,0.44);
   \draw (0.2,0.45) -- (0.2,0.44);
   \draw (0.3,0.45) -- (0.3,0.44);
   \draw (0.4,0.45) -- (0.4,0.44);
   \draw (0.5,0.45) -- (0.5,0.44);
   \draw (0.6,0.45) -- (0.6,0.44);
   \draw (0.7,0.45) -- (0.7,0.44);
   \draw (0.8,0.45) -- (0.8,0.44);
   \draw (0.9,0.45) -- (0.9,0.44);
   \draw (1.0,0.45) -- (1.0,0.44);
   \draw (0.1,0.44) node[anchor=north] {$\ss 0.1$};
   \draw (0.2,0.44) node[anchor=north] {$\ss 0.2$};
   \draw (0.3,0.44) node[anchor=north] {$\ss 0.3$};
   \draw (0.4,0.44) node[anchor=north] {$\ss 0.4$};
   \draw (0.5,0.44) node[anchor=north] {$\ss 0.5$};
   \draw (0.6,0.44) node[anchor=north] {$\ss 0.6$};
   \draw (0.7,0.44) node[anchor=north] {$\ss 0.7$};
   \draw (0.8,0.44) node[anchor=north] {$\ss 0.8$};
   \draw (0.9,0.44) node[anchor=north] {$\ss 0.9$};
   \draw[red] (0.030,0.592) -- (0.040,0.589) -- (0.050,0.585) --
   (0.060,0.583) -- (0.070,0.580) -- (0.080,0.577) -- (0.090,0.575) --
   (0.100,0.572) -- (0.110,0.570) -- (0.120,0.568) -- (0.130,0.566) --
   (0.140,0.564) -- (0.150,0.561) -- (0.160,0.559) -- (0.170,0.557) --
   (0.180,0.555) -- (0.190,0.553) -- (0.200,0.551) -- (0.210,0.549) --
   (0.220,0.547) -- (0.230,0.546) -- (0.240,0.544) -- (0.250,0.542) --
   (0.260,0.540) -- (0.270,0.538) -- (0.280,0.536) -- (0.290,0.535) --
   (0.300,0.533) -- (0.310,0.531) -- (0.320,0.529) -- (0.330,0.528) --
   (0.340,0.526) -- (0.350,0.524) -- (0.360,0.523) -- (0.370,0.521) --
   (0.380,0.520) -- (0.390,0.518) -- (0.400,0.517) -- (0.410,0.515) --
   (0.420,0.514) -- (0.430,0.513) -- (0.440,0.511) -- (0.450,0.510) --
   (0.460,0.509) -- (0.470,0.508) -- (0.480,0.507) -- (0.490,0.506) --
   (0.500,0.505) -- (0.510,0.504) -- (0.520,0.503) -- (0.530,0.502) --
   (0.540,0.502) -- (0.550,0.501) -- (0.560,0.501) -- (0.570,0.500) --
   (0.580,0.500) -- (0.590,0.500) -- (0.600,0.500) -- (0.610,0.500) --
   (0.620,0.500) -- (0.630,0.500) -- (0.640,0.500) -- (0.650,0.500) --
   (0.660,0.501) -- (0.670,0.502) -- (0.680,0.502) -- (0.690,0.503) --
   (0.700,0.504) -- (0.710,0.505) -- (0.720,0.506) -- (0.730,0.508) --
   (0.740,0.509) -- (0.750,0.511) -- (0.760,0.513) -- (0.770,0.515) --
   (0.780,0.517) -- (0.790,0.519) -- (0.800,0.521) -- (0.810,0.524) --
   (0.820,0.527) -- (0.830,0.529) -- (0.840,0.532) -- (0.850,0.535) --
   (0.860,0.539) -- (0.870,0.542) -- (0.880,0.545) -- (0.890,0.549) --
   (0.900,0.552) -- (0.910,0.556) -- (0.920,0.557);
  \end{tikzpicture}
 \end{center}
\end{remark}

We can now try to find $p^{-1}$ and $p$ by power series methods.  As
before, it turns out to be convenient to focus on $p_1^{-1}$ rather
than $p^{-1}$, because $p_1^{-1}$ is unbranched at the origin, and the
associated power series has a reasonably large radius of convergence.

\begin{proposition}\lbl{prop-p-one}
 We have $S(p_1^{-1})=s^*_0+ds^*_1$, where $d$ is the same real
 constant as in Proposition~\ref{prop-S-p-inv}, and
 \begin{align*}
  s^*_0(z) &= \frac{192a^4z^2(1+z^2)^2-9(1-a^4)^2(1-z^2)^4}{
                    2(1-z^2)^2((1+a^2)^2(1-z^2)^2+16a^2z^2)^2} &
  s^*_1(z) &= \frac{4a^2}{(1+a^2)^2(1-z^2)^2+16a^2z^2}.
 \end{align*}
 The set of poles of $S(p_1^{-1})$ is
 \[ \psi^{-1}(B) = \{v_{PSi}\st i\in\{0,1,10,11,12,13\}\} =
     \left\{\pm 1,
            \frac{i-a}{i+a},\frac{i+a}{i-a},
            \frac{a-i}{a+i},\frac{a+i}{a-i}\right\}.
 \]
\end{proposition}
\begin{proof}
 Recall that $p_1=\phi p\psi$, so $p_1^{-1}=\psi^{-1}p^{-1}\phi^{-1}$.
 Here $\phi^{-1}$ and $\psi^{-1}$ are both M\"obius maps, so
 $S(\phi^{-1})=0$ and $S(\psi^{-1})=0$.  The Schwarzian chain rule
 therefore gives
 \[ S(p_1^{-1}) =
    S(\psi^{-1}\circ p^{-1}\circ\phi^{-1}) =
    (S(p^{-1})\circ\phi^{-1}) \cdot ((\phi^{-1})')^2.
 \]
 Here $S(p^{-1})$ is given by Proposition~\ref{prop-S-p-inv}, and
 \[ (\phi^{-1})'(z) = -\frac{2i}{(1+z)^2}. \]
 It is now a lengthy but straightforward calculation to show that this
 is equivalent to the claimed formula.  We also see that the poles of
 $S(p_1^{-1})$ are the images under $\phi$ of the poles of
 $S(p^{-1})$, together with the poles of $(\phi^{-1})'$.  We saw
 previously that $S(p^{-1})$ has poles only in the set
 \[ \{0,\pm a,\pm 1/a\}=\{v_{C0},v_{C10},v_{C11},v_{C12},v_{C13}\}.
 \]
 These points $v_{Ci}$ are mapped by $\phi$ to the corresponding
 points $v_{PSi}$, and the only pole of $(\phi^{-1})'$ is at
 $-1=v_{PS1}$, so the poles of $S(p_1^{-1})$ are as claimed.
\end{proof}

\begin{remark}
 Maple notation for $s_0^*(z)$ and $s_1^*(z)$ is \mcode+S0_p1_inv(z)+
 and \mcode+S1_p1_inv(z)+.
\end{remark}
\begin{remark}\lbl{rem-heun}
 We can also consider the map $p_2=\xi\circ p_1$, where
 $\xi(z)=2/(z^2+z^{-2})$.  If we choose domains so that the maps
 $U\xra{p_1}V\xra{\xi}W$ are invertible, then we can use the
 schwarzian chain rule to obtain a formula for $S(p_2^{-1})$ on $W$.
 It turns out that this is a rational function with poles (of various
 orders) at $0$, $1$, $-1$ and $-(1+a^2)/(1-6a^2+a^4)$.  Because there
 are only four poles, the basic solutions to the differential equation
 $f''+S(p_2^{-1})f/2$ can be expressed as products of certain factors
 $(z-\al)^{n/8}$ with suitable instances of the Heun $G$-function.
 Further details are given in the Maple code, but we will not discuss
 them here, because we did not find this representation to be useful.
 However, we believe that this is the closest possible connection with
 special functions that have previously been named and studied.
 \begin{checks}
  hyperbolic/schwarz_check.mpl: check_heun()
 \end{checks}
\end{remark}

\begin{definition}\lbl{defn-star-U}
 We put
 \[ U = \C \sm\{t\,\psi^{-1}(u)\st t\geq 1,\;u\in B\}. \]
\end{definition}
In other words, $U$ is the domain obtained from $\C$ by deleting
rays from the points of $\psi^{-1}(B)$ to $\infty$.  This is
simply connected and contains $i\R\cup\Dl$, and the maps $s^*_0$ and
$s^*_1$ are holomorphic on $U$.

\begin{proposition}\lbl{prop-p-one-section}
 There is a unique holomorphic map $g\:U\to\Dl$ satisfying $g(0)=0$
 and $p_1g=1$.  This satisfies $g(\ov{z})=\ov{g(z)}$ and
 $g(-z)=-g(z)$.  Moreover, $g(z)$ can be written in the form
 $c\,f_1(z)/f_0(z)$, where
 \begin{itemize}
  \item The maps $f_k$ are holomorphic on $U$ and satisfy
   $f''_k+\half(s^*_0+ds^*_1)f_k=0$.
  \item $f_0$ is an even function with $f_0(0)=1$ and
   $f_0(\ov{z})=\ov{f_0(z)}$.
  \item $f_1$ is an odd function with $f'_1(0)=1$ and
   $f_1(\ov{z})=\ov{f_1(z)}$.
  \item $c$ is a positive real number.
 \end{itemize}
\end{proposition}
\begin{proof}
 As $U$ contains none of the critical values of $p_1$, we see that
 the map $p_1\:p_1^{-1}(U)\to U$ is a covering.  We have also seen
 that $p_1(0)=0$.  As $U$ is simply connected, standard covering
 theory tells us that there is a unique section $g\:U\to\Dl$ with
 $p_1g=1$ and $g(0)=0$.  As $p_1$ is holomorphic, it is easy to check
 that the same holds for $g$.  As $p_1(-z)=-p_1(z)$ and
 $p_1((-1,1))\sse\R$ we see that $g(-z)=-g(z)$ and
 $g(\R)\sse(-1,1)$.

 Now put $s^*=S(p_1^{-1})$, which has the form $s^*_0+ds_1^*$ as in
 Proposition~\ref{prop-p-one}.  Let $F$ denote the space of
 holomorphic functions on $U$ that satisfy $e''+\half s^*e=0$.
 Proposition~\ref{prop-schwarzian-solutions} tells us that there are
 unique functions $f_0,f_1\in F$ with $f_0(z)=1+O(z^2)$ and
 $f_1(z)=z+O(z^2)$, and these functions form a basis for $F$.  From
 the definition of $s^*_0$ and $s^*_1$ we see that $s^*(z)=-z$.  It
 follows that the maps $z\mapsto f_0(-z)$ and $z\mapsto -f_1(-z)$
 satisfy the defining conditions of $f_0$ and $f_1$ respectively; so
 $f_0$ is even and $f_1$ is odd.  Similarly, we see that
 $s^*(z)=\ov{s^*(\ov{z})}$, and it follows that
 $f_k(z)=\ov{f_k(\ov{z})}$.  Proposition~\ref{prop-schwarzian-solutions}
 also tells us that $g$ is the quotient of two linearly independent
 elements of $F$, so there are constants $A,B,C$ and $D$ such that
 $g=(Af_0+Bf_1)/(Cf_0+Df_1)$ and $AD-BC\neq 0$.  As $g(0)=0$ we must
 have $A=0$, so $Cf_0+Df_1=Bf_1/g$.  As both $g$ and $f_1$ are odd we
 see that $Cf_0+Df_1$ is even, so $D=0$.  Thus, if we put $c=B/C$ we
 have $g=c\,f_1/f_0$.  By considering derivatives at the origin, we
 see that $c$ is real and positive.
\end{proof}

\subsection{Methods for explicit calculation}
\lbl{sec-schwarz-methods}

The parameters $b$ and $d$ depend on $a$; in this section we will call
them $\bt(a)$ and $\dl(a)$.  While
Proposition~\ref{prop-p-one-section} is very satisfactory, we cannot
immediately use it for computation, because we do not know the
values of $\bt(a)$ and $\dl(a)$.  We will describe some calculations
that we can do with an arbitrary pair $(b,d)$, which will enable us to
test whether $(b,d)\simeq(\bt(a),\dl(a))$, and to improve the degree
of approximation if necessary.  First, we define
$s^*(z)=s_0^*(z)+ds_1^*(z)$.  Next, for any open set $V\subset\C$, we
put
\begin{align*}
 F^*(V) &= \{f\in\text{Hol}(V)\st f''+s^*f/2=0\} \\
 G^*(V) &= \{g\in\text{Mer}(V)\st S(g)=s^*\},
\end{align*}
and note that these are described by
Proposition~\ref{prop-schwarzian-solutions}.  In particular, we can
consider $F^*(U)$ and $G^*(U)$, where $U$ is as in
Definition~\ref{defn-star-U}.  Just as in
Proposition~\ref{prop-p-one-section}, we see that there is a unique
basis $\{f_0,f_1\}$ for $F^*(U)$ such that $f_k(z)=z^k+O(z^2)$.  Both
the power series solution method in Lemma~\ref{lem-F-disc} and the
analytic continuation method in Corollary~\ref{cor-F-disc} are
straightforwardly computable, so we can calculate $f_k(z)$ for any
$z\in U$.  (To calculate $f_k(u)$ when $|u|=1$, we have typically
computed power series solutions $f_{kj}(z)$ centred at $ju/10$ for
$0\leq j\leq 10$, and compared then by evaluating $f_{kj}((j+1)u/10)$
and $f'_{kj}((j+1)u/10)$.  This works provided that $u$ is not too
close to any of the branch points.  If it is very close to a branch
point, then we need to take smaller steps as we approach it.)  We also
put $g_0(z)=f_1(z)/f_0(z)$.  It is not hard to see that this is odd
with real coefficients, so $g_0(i\R)\sse i\R$.

\begin{definition}\lbl{defn-circle-fit}
 Suppose that
 \[ g_0(ie^{it})/i = u_0+iu_1t+u_2t^2+O(t^3). \]
 Using the fact that $g_0(-\ov{z})=-\ov{g(z)}$, we see that the
 coefficients $u_j$ are real.  We put
 \begin{align*}
  b &= \frac{u_1^2}{\sqrt{8u_0u_1(u_0u_2+u_1^2)+u_1^4}} \\
  c &= \sqrt{\frac{u_2}{u_0(u_0u_2+u_1^2)}} \\
  g(z) &= c\,g_0(z).
 \end{align*}
\end{definition}
These formulae are embedded in the Maple function
\mcode+series_circle_fit(u0,u1,u2)+.

\begin{proposition}
 If $d=\dl(a)$ then $b=\bt(a)$ and $g(z)$ is the function described in
 Proposition~\ref{prop-p-one-section}.  Moreover, for any $z_0$ on the
 arc of $S^1$ between $1$ and $(i-a)/(i+a)$, we have
 $\text{Im}(\psi(g(z_0)))=0$.
\end{proposition}
\begin{proof}
 Because $d=\dl(a)$, we see that $f_0$ and $f_1$ are as described in
 Proposition~\ref{prop-p-one-section}.  Put $b^*=\bt(a)$, and let
 $c^*$ and $g^*$ be the number and the function denoted by $c$ and $g$
 in Proposition~\ref{prop-p-one-section}, so $g^*=c^*g_0$; our
 task is to prove that $b^*=b$ and $c^*=c$ and $g^*=g$.

 We saw in Lemma~\ref{lem-psi-edges} that $C_{HS3}$ is part
 of the circle of radius $R=\rt b^*/b^*_-$ centred at the point
 $iA=ib^*_+/b^*_-$.  We also have
 $p_1(C_{HS3})\sse C_{PS3}\sse S^1$, and it follows that
 $g^*(ie^{it})\sse C_{HS3}$ for small $t$, or in other words
 $|A-c^*g_0(ie^{it})/i|^2=R^2$.  This gives
 \[ R^2 = (A-c^*u_0-c^*u_2t^2)^2 + u_1^2t^2 + O(t^3), \]
 and we can expand out and compare coefficients to get
 \begin{align*}
  R^2 &= (A-c^*u_0)^2 \\
  u_1^2 &= 2(A-c^*u_0)c^*u_2.
 \end{align*}
 Note also that the definitions $R=\rt b^*/b^*_-$ and $A=b^*/b^*_-$
 imply $A^2-R^2=1$ (corresponding to the fact that $C_{HS3}$ crosses
 the unit circle orthogonally).  It is now an exercise in algebra to
 solve these equations giving
 \begin{align*}
  c^* &=  \sqrt{\frac{u_2}{u_0(u_0u_2+u_1^2)}}
           = c \\
  A &= \frac{u_0u_2 + \half u_1^2}{\sqrt{u_0u_2(u_0u_2+u_1^2)}} \\
  R &= \frac{\half u_1^2}{\sqrt{u_0u_2(u_0u_2+u_1^2)}}.
 \end{align*}
 From the $c^*=c$ we deduce that $g^*=g$.  We also have
 $A=\rt b^*/\sqrt{1-(b^*)^2}$, which gives $b^*=A/\sqrt{2+A^2}$; after
 some further manipulation this gives $b^*=b$.

 Finally, let $L$ denote the arc between $1$ and $(i-a)/(i+a)$.  This
 is part of $C_{PS5}=p_1(C_{HS5})=p\psi^{-1}(C_{H5})$, but
 $C_{H5}=(-1,1)$; it follows that $\psi(g(L))\sse (-1,1)$, so
 $\text{Im}(\psi(g(L)))=0$.
 \begin{checks}
  hyperbolic/schwarz_check.mpl: check_schwarz()
 \end{checks}
\end{proof}

We thus arrive at the following method.  Given $a$ we fix a point
$z_0$ on $S^1$ between $1$ and $(i-a)/(i+a)$.  We then choose $d$ and
compute $g_0(z)$.  The power series for $g_0(ie^{it})/i$ gives the
coefficients $u_0$, $u_1$ and $u_2$, from which we compute $b$, $c$
and $g(z)$.  We then put $\ep(d)=\text{Im}(\psi(z_0))$.  If $d=\dl(a)$
then we will have $\ep(d)=0$.  If $\ep(d)\neq 0$ then we can adjust
our value of $d$ and try again.  As evaluation of $\ep(d)$ is quite
expensive, it is important to use an efficient search algorithm.  It
is also useful to retain a lot of information generated in the course
of the calculation which we have found to be awkward when using
Maple's built in \mcode+fsolve+ function.  We have therefore used our
own implementation of Brent's method~\cite{br:amw} (closely following
the Matlab implementation by John Burkardt~\cite{bu:br}).  This gives
us a value of $d$ such that $\ep(d)=0$ to high accuracy.  We have not
given a general proof that $\ep(d)=0$ implies $d=\dl(a)$, but in any
given case it is easy to perform additional checks to verify that this
is the case; for example, we can feed our new value of $b$ into the
method of Section~\ref{sec-a-from-b} and check that everything is
consistent.

Maple commands for the above algorithm were given at the beginning of
Section~\ref{sec-b-from-a}.  For more detail, see the comments in the
code.

\subsection{Holomorphic forms}
\lbl{sec-hol-forms}

Suppose we have constructed an isomorphism $f\:HX(b)\to PX(a)$, given
generically by
\[ f(z)=j(q(z),p(z)). \]
Recall that Proposition~\ref{prop-holomorphic-forms} gives a basis
$\{\om_0,\om_1\}$ for $\Om^1(PX(a))$.  Let $m(z)$ be the function on
$\Dl$ given by $f^*(\om_0)=m(z)\,dz$.  Note that this is holomorphic
on $\Dl$ and so is given by a power series that converges everywhere
on $\Dl$, unlike the functions $p(z)$ and $p_1(z)$ which have
infinitely many poles.  In this section we will investigate the
properties of $m(z)$.

\begin{proposition}\lbl{prop-m-props}
 The function $m(z)$ has the following automorphy properties:
 \begin{itemize}
  \item[(a)] For $\gm\in\Pi$ we have $m(z)=m(\gm(z))\,\gm'(z)$.
  \item[(b)] $m(z)=m(iz)$ (so $m(z)$ is a power series in $z^4$).
  \item[(c)] $m(\ov{z})=\ov{m(z)}$ (so the power series for $m(z)$ has
   real coefficients).
  \item[(d)]
   \[ m(\mu(z)) = -\frac{b^2(1-b^2)\,m(z)}{((i-b^2)z+b_+)^2\,p(z)}.
   \]
 \end{itemize}
 Moreover, for a suitable branch of the square root we have
 \[ m(z) = p'(z)\, r_a(p(z))^{-1/2} \]
 (where $r_a(z)=z(z-a)(z+a)(z-1/a)(z+1/a)$ as in
 Definition~\ref{defn-P}).
\end{proposition}

Note that property~(d) means that $p$ can be computed from $m$.

\begin{proof}
 \begin{itemize}
  \item[(a)] For $\gm\in\Pi$ we have $f\gm=f\:\Dl\to PX(a)$ and so
   \[ m(\gm(z))\gm'(z)\,dz = \gm^*(m(z)\,dz) = (f\gm)^*(\om_0) =
       f^*(\om_0) =m(z)\,dz.
   \]
  \item[(b)] We have $f\lm=\lm f\:\Dl\to PX(a)$, and $\lm^*dz=i\,dz$
   on $\Dl$ whereas $\lm^*\om_0=i\,\om_0$ on $PX(a)$ (by
   Proposition~\ref{prop-holomorphic-forms}).  It follows easily from
   this that $m(iz)=m(\lm(z))=m(z)$.
  \item[(c)] This follows in the same way, using the fact that
   $\nu^\#(dz)=dz$ and $\nu^\#(\om_0)=\om_0$.
  \item[(d)] Recall that $\mu^*(\om_0)=\om_1$, and on $PX_0(a)$ we
   have $\om_1=z\,\om_0$, which means that
   \[ f^*(\om_1)=p(z)\,f^*(\om_0)=p(z)m(z)\,dz. \]
   On the other hand, as
   $f\mu=\mu f$ we have
   \[ f^*(\om_1)=\mu^*(\om_0)=m(\mu(z))\mu'(z)\,dz. \]
   A calculation
   gives $\mu'(z)=-b^2(1-b^2)/((i-b^2)z+b_+)^2$.  Putting this
   together gives the claimed equation.
 \end{itemize}

 Finally, $\om_0$ is given on $PX_0(a)$ by $dz/w$, and this gives
 $f^*(\om_0)=p'(z)\,dz/q(z)$.  Also, as $(p(z),q(z))\in PX_0(a)$ we
 have $q(z)=\pm\sqrt{r_a(p(z))}$.
\end{proof}

Using the methods of Sections~\ref{sec-a-from-b}
and~\ref{sec-b-from-a}, we can calculate $p_1(z)$ quite accurately on
a domain including $\psi^{-1}(HF_{16}(b))$, and this lets us calculate
$p(z)$ on $HF_{16}(b)$.  Given an arbitrary point $z\in\Dl$ we can use
the method in Remark~\ref{rem-move-inwards} to find $\gm\in\Pi$ such
that $\gm(z)\in HF_1(b)$, then we can find $\bt\in G$ such that
$\bt\gm(z)\in HF_{16}(b)$.  Using the automorphy properties of $m$ we
can then find $m(z)$ from $m(\bt\gm(z))$.  Given an object \mcode+HP+
of the class \mcode+H_to_P_map+, the method \mcode+HP["m_piecewise",z]+
will calculate $m(z)$ by the above algorithm.

To obtain the power series for $m(z)$, it is best to calculate $m(s)$
for all $s$ in some finite subset $S\subset\C$, and then find a
polynomial $m_0(z)=\sum_{i=0}^da_iz^{4i}$ which minimises
$\sum_s\|m_0(s)-m(s)\|^2$.  As this objective function depends
quadratically on the coefficients $a_i\in\R$, the minimisation problem
reduces easily to a matrix calculation.  We have generally taken
\[ S = \{re^{k\pi i/(2n)}\st 0\leq k<n\} \]
for some radius $r\in(0,1)$ (say $r=0.95$) and some integer $n>0$ (say
$n=400$).  Note that the relation $m(iz)=m(z)$ makes it natural to
consider only sample points in the first quadrant, and the maximum
principle of complex analysis makes it reasonable to consider only
sample points on the boundary of the region where we want our
approximation to be accurate.  This algorithm is implemented by the
method \mcode+HP["find_m_series",r,n,d]+.

The following plot shows the curve $m(0.93e^{it})$ with the parameters
$a$ and $b$ that are relevant for $EX^*$.  (The real axis is drawn
vertically.)
\begin{center}
 \begin{tikzpicture}[scale=.4]
  \draw[gray]  (-13, -7) rectangle (13,13);
  \draw[gray]  (-13,  0) -- (13,0);
  \draw[gray]  (  0, -7) -- (0,13);
  \fill[black] ( 13,  0) circle(0.15);
  \fill[black] (-13,  0) circle(0.15);
  \fill[black] (  0, -7) circle(0.15);
  \fill[black] (  0, 13) circle(0.15);
  \draw        ( 13,  0) node[anchor=west ] {$-13i$};
  \draw        (-13,  0) node[anchor=east ] {$ 13i$};
  \draw        (  0, -8) node[anchor=north] {$   8$};
  \draw        (  0, 13) node[anchor=south] {$  13$};
  \draw[smooth,red] (-0.4839,11.0507) -- (-0.9605,10.9726) -- (-1.4223,10.8445) -- (-1.8627,10.6692) -- (-2.2754,10.4508) -- (-2.6550,10.1940) -- (-2.9970,9.9047) -- (-3.2979,9.5889) -- (-3.5556,9.2533) -- (-3.7687,8.9046) -- (-3.9375,8.5493) -- (-4.0629,8.1938) -- (-4.1471,7.8438) -- (-4.1929,7.5043) -- (-4.2037,7.1796) -- (-4.1836,6.8732) -- (-4.1365,6.5876) -- (-4.0667,6.3245) -- (-3.9780,6.0850) -- (-3.8742,5.8694) -- (-3.7588,5.6775) -- (-3.6348,5.5088) -- (-3.5047,5.3624) -- (-3.3708,5.2373) -- (-3.2349,5.1324) -- (-3.0986,5.0466) -- (-2.9630,4.9788) -- (-2.8292,4.9280) -- (-2.6979,4.8931) -- (-2.5700,4.8735) -- (-2.4459,4.8684) -- (-2.3261,4.8771) -- (-2.2111,4.8991) -- (-2.1013,4.9340) -- (-1.9969,4.9814) -- (-1.8985,5.0411) -- (-1.8063,5.1130) -- (-1.7208,5.1971) -- (-1.6426,5.2935) -- (-1.5724,5.4024) -- (-1.5109,5.5240) -- (-1.4593,5.6585) -- (-1.4187,5.8063) -- (-1.3908,5.9673) -- (-1.3773,6.1417) -- (-1.3803,6.3291) -- (-1.4019,6.5291) -- (-1.4445,6.7406) -- (-1.5108,6.9624) -- (-1.6032,7.1926) -- (-1.7244,7.4287) -- (-1.8764,7.6677) -- (-2.0614,7.9060) -- (-2.2807,8.1394) -- (-2.5351,8.3635) -- (-2.8247,8.5732) -- (-3.1487,8.7635) -- (-3.5052,8.9294) -- (-3.8917,9.0661) -- (-4.3048,9.1693) -- (-4.7403,9.2355) -- (-5.1934,9.2615) -- (-5.6592,9.2455) -- (-6.1325,9.1862) -- (-6.6082,9.0832) -- (-7.0814,8.9369) -- (-7.5475,8.7482) -- (-8.0024,8.5185) -- (-8.4422,8.2497) -- (-8.8635,7.9438) -- (-9.2632,7.6030) -- (-9.6386,7.2297) -- (-9.9870,6.8264) -- (-10.3063,6.3958) -- (-10.5944,5.9408) -- (-10.8495,5.4642) -- (-11.0703,4.9693) -- (-11.2557,4.4593) -- (-11.4051,3.9373) -- (-11.5182,3.4065) -- (-11.5952,2.8698) -- (-11.6365,2.3300) -- (-11.6427,1.7895) -- (-11.6145,1.2503) -- (-11.5524,0.7142) -- (-11.4571,0.1829) -- (-11.3286,-0.3423) -- (-11.1671,-0.8598) -- (-10.9722,-1.3680) -- (-10.7435,-1.8649) -- (-10.4806,-2.3479) -- (-10.1829,-2.8143) -- (-9.8506,-3.2605) -- (-9.4839,-3.6826) -- (-9.0841,-4.0762) -- (-8.6531,-4.4369) -- (-8.1938,-4.7602) -- (-7.7101,-5.0420) -- (-7.2068,-5.2784) -- (-6.6896,-5.4667) -- (-6.1648,-5.6050) -- (-5.6391,-5.6924) -- (-5.1193,-5.7295) -- (-4.6121,-5.7181) -- (-4.1237,-5.6610) -- (-3.6598,-5.5623) -- (-3.2248,-5.4268) -- (-2.8225,-5.2601) -- (-2.4552,-5.0679) -- (-2.1243,-4.8562) -- (-1.8298,-4.6308) -- (-1.5710,-4.3971) -- (-1.3462,-4.1600) -- (-1.1532,-3.9236) -- (-0.9892,-3.6914) -- (-0.8511,-3.4662) -- (-0.7359,-3.2498) -- (-0.6405,-3.0436) -- (-0.5619,-2.8483) -- (-0.4974,-2.6642) -- (-0.4447,-2.4909) -- (-0.4016,-2.3279) -- (-0.3667,-2.1747) -- (-0.3385,-2.0302) -- (-0.3160,-1.8938) -- (-0.2988,-1.7645) -- (-0.2863,-1.6416) -- (-0.2784,-1.5247) -- (-0.2749,-1.4133) -- (-0.2759,-1.3073) -- (-0.2813,-1.2066) -- (-0.2909,-1.1112) -- (-0.3046,-1.0214) -- (-0.3221,-0.9372) -- (-0.3432,-0.8588) -- (-0.3673,-0.7863) -- (-0.3942,-0.7196) -- (-0.4236,-0.6588) -- (-0.4551,-0.6039) -- (-0.4885,-0.5548) -- (-0.5238,-0.5116) -- (-0.5608,-0.4745) -- (-0.5993,-0.4438) -- (-0.6391,-0.4200) -- (-0.6797,-0.4036) -- (-0.7206,-0.3952) -- (-0.7608,-0.3955) -- (-0.7990,-0.4050) -- (-0.8337,-0.4239) -- (-0.8632,-0.4524) -- (-0.8855,-0.4898) -- (-0.8985,-0.5353) -- (-0.9004,-0.5872) -- (-0.8895,-0.6436) -- (-0.8647,-0.7018) -- (-0.8252,-0.7587) -- (-0.7714,-0.8113) -- (-0.7041,-0.8562) -- (-0.6252,-0.8906) -- (-0.5372,-0.9120) -- (-0.4431,-0.9185) -- (-0.3466,-0.9093) -- (-0.2512,-0.8842) -- (-0.1603,-0.8441) -- (-0.0770,-0.7906) -- (-0.0039,-0.7260) -- (0.0573,-0.6530) -- (0.1054,-0.5747) -- (0.1402,-0.4939) -- (0.1620,-0.4136) -- (0.1717,-0.3362) -- (0.1707,-0.2641) -- (0.1607,-0.1988) -- (0.1437,-0.1418) -- (0.1216,-0.0940) -- (0.0966,-0.0557) -- (0.0709,-0.0270) -- (0.0464,-0.0077) -- (0.0251,0.0032) -- (0.0088,0.0066) -- (-0.0012,0.0039) -- (-0.0036,-0.0032) -- (0.0025,-0.0128) -- (0.0173,-0.0229) -- (0.0409,-0.0316) -- (0.0730,-0.0368) -- (0.1127,-0.0367) -- (0.1591,-0.0297) -- (0.2108,-0.0143) -- (0.2663,0.0105) -- (0.3241,0.0458) -- (0.3823,0.0923) -- (0.4392,0.1504) -- (0.4929,0.2206) -- (0.5413,0.3029) -- (0.5823,0.3972) -- (0.6136,0.5032) -- (0.6325,0.6202) -- (0.6366,0.7469) -- (0.6230,0.8816) -- (0.5894,1.0217) -- (0.5333,1.1642) -- (0.4530,1.3050) -- (0.3473,1.4398) -- (0.2163,1.5633) -- (0.0607,1.6704) -- (-0.1172,1.7555) -- (-0.3138,1.8138) -- (-0.5244,1.8408) -- (-0.7433,1.8332) -- (-0.9638,1.7888) -- (-1.1787,1.7071) -- (-1.3806,1.5890) -- (-1.5627,1.4372) -- (-1.7184,1.2559) -- (-1.8425,1.0505) -- (-1.9310,0.8277) -- (-1.9814,0.5947) -- (-1.9931,0.3591) -- (-1.9668,0.1281) -- (-1.9051,-0.0912) -- (-1.8118,-0.2929) -- (-1.6917,-0.4722) -- (-1.5503,-0.6255) -- (-1.3937,-0.7508) -- (-1.2275,-0.8473) -- (-1.0574,-0.9155) -- (-0.8882,-0.9568) -- (-0.7242,-0.9736) -- (-0.5687,-0.9686) -- (-0.4240,-0.9451) -- (-0.2918,-0.9061) -- (-0.1729,-0.8547) -- (-0.0678,-0.7934) -- (0.0236,-0.7247) -- (0.1016,-0.6506) -- (0.1666,-0.5727) -- (0.2190,-0.4923) -- (0.2592,-0.4108) -- (0.2875,-0.3292) -- (0.3039,-0.2486) -- (0.3086,-0.1702) -- (0.3019,-0.0955) -- (0.2840,-0.0259) -- (0.2555,0.0369) -- (0.2175,0.0913) -- (0.1710,0.1356) -- (0.1179,0.1684) -- (0.0602,0.1886) -- (0.0000,0.1954) -- (-0.0602,0.1886) -- (-0.1179,0.1684) -- (-0.1710,0.1356) -- (-0.2175,0.0913) -- (-0.2555,0.0369) -- (-0.2840,-0.0259) -- (-0.3019,-0.0955) -- (-0.3086,-0.1702) -- (-0.3039,-0.2486) -- (-0.2875,-0.3292) -- (-0.2592,-0.4108) -- (-0.2190,-0.4923) -- (-0.1666,-0.5727) -- (-0.1016,-0.6506) -- (-0.0236,-0.7247) -- (0.0678,-0.7934) -- (0.1729,-0.8547) -- (0.2918,-0.9061) -- (0.4240,-0.9451) -- (0.5687,-0.9686) -- (0.7242,-0.9736) -- (0.8882,-0.9568) -- (1.0574,-0.9155) -- (1.2275,-0.8473) -- (1.3937,-0.7508) -- (1.5503,-0.6255) -- (1.6917,-0.4722) -- (1.8118,-0.2929) -- (1.9051,-0.0912) -- (1.9668,0.1281) -- (1.9931,0.3591) -- (1.9814,0.5947) -- (1.9310,0.8277) -- (1.8425,1.0505) -- (1.7184,1.2559) -- (1.5627,1.4372) -- (1.3806,1.5890) -- (1.1787,1.7071) -- (0.9638,1.7888) -- (0.7433,1.8332) -- (0.5244,1.8408) -- (0.3138,1.8138) -- (0.1172,1.7555) -- (-0.0607,1.6704) -- (-0.2163,1.5633) -- (-0.3473,1.4398) -- (-0.4530,1.3050) -- (-0.5333,1.1642) -- (-0.5894,1.0217) -- (-0.6230,0.8816) -- (-0.6366,0.7469) -- (-0.6325,0.6202) -- (-0.6136,0.5032) -- (-0.5823,0.3972) -- (-0.5413,0.3029) -- (-0.4929,0.2206) -- (-0.4392,0.1504) -- (-0.3823,0.0923) -- (-0.3241,0.0458) -- (-0.2663,0.0105) -- (-0.2108,-0.0143) -- (-0.1591,-0.0297) -- (-0.1127,-0.0367) -- (-0.0730,-0.0368) -- (-0.0409,-0.0316) -- (-0.0173,-0.0229) -- (-0.0025,-0.0128) -- (0.0036,-0.0032) -- (0.0012,0.0039) -- (-0.0088,0.0066) -- (-0.0251,0.0032) -- (-0.0464,-0.0077) -- (-0.0709,-0.0270) -- (-0.0966,-0.0557) -- (-0.1216,-0.0940) -- (-0.1437,-0.1418) -- (-0.1607,-0.1988) -- (-0.1707,-0.2641) -- (-0.1717,-0.3362) -- (-0.1620,-0.4136) -- (-0.1402,-0.4939) -- (-0.1054,-0.5747) -- (-0.0573,-0.6530) -- (0.0039,-0.7260) -- (0.0770,-0.7906) -- (0.1603,-0.8441) -- (0.2512,-0.8842) -- (0.3466,-0.9093) -- (0.4431,-0.9185) -- (0.5372,-0.9120) -- (0.6252,-0.8906) -- (0.7041,-0.8562) -- (0.7714,-0.8113) -- (0.8252,-0.7587) -- (0.8647,-0.7018) -- (0.8895,-0.6436) -- (0.9004,-0.5872) -- (0.8985,-0.5353) -- (0.8855,-0.4898) -- (0.8632,-0.4524) -- (0.8337,-0.4239) -- (0.7990,-0.4050) -- (0.7608,-0.3955) -- (0.7206,-0.3952) -- (0.6797,-0.4036) -- (0.6391,-0.4200) -- (0.5993,-0.4438) -- (0.5608,-0.4745) -- (0.5238,-0.5116) -- (0.4885,-0.5548) -- (0.4551,-0.6039) -- (0.4236,-0.6588) -- (0.3942,-0.7196) -- (0.3673,-0.7863) -- (0.3432,-0.8588) -- (0.3221,-0.9372) -- (0.3046,-1.0214) -- (0.2909,-1.1112) -- (0.2813,-1.2066) -- (0.2759,-1.3073) -- (0.2749,-1.4133) -- (0.2784,-1.5247) -- (0.2863,-1.6416) -- (0.2988,-1.7645) -- (0.3160,-1.8938) -- (0.3385,-2.0302) -- (0.3667,-2.1747) -- (0.4016,-2.3279) -- (0.4447,-2.4909) -- (0.4974,-2.6642) -- (0.5619,-2.8483) -- (0.6405,-3.0436) -- (0.7359,-3.2498) -- (0.8511,-3.4662) -- (0.9892,-3.6914) -- (1.1532,-3.9236) -- (1.3462,-4.1600) -- (1.5710,-4.3971) -- (1.8298,-4.6308) -- (2.1243,-4.8562) -- (2.4552,-5.0679) -- (2.8225,-5.2601) -- (3.2248,-5.4268) -- (3.6598,-5.5623) -- (4.1237,-5.6610) -- (4.6121,-5.7181) -- (5.1193,-5.7295) -- (5.6391,-5.6924) -- (6.1648,-5.6050) -- (6.6896,-5.4667) -- (7.2068,-5.2784) -- (7.7101,-5.0420) -- (8.1938,-4.7602) -- (8.6531,-4.4369) -- (9.0841,-4.0762) -- (9.4839,-3.6826) -- (9.8506,-3.2605) -- (10.1829,-2.8143) -- (10.4806,-2.3479) -- (10.7435,-1.8649) -- (10.9722,-1.3680) -- (11.1671,-0.8598) -- (11.3286,-0.3423) -- (11.4571,0.1829) -- (11.5524,0.7142) -- (11.6145,1.2503) -- (11.6427,1.7895) -- (11.6365,2.3300) -- (11.5952,2.8698) -- (11.5182,3.4065) -- (11.4051,3.9373) -- (11.2557,4.4593) -- (11.0703,4.9693) -- (10.8495,5.4642) -- (10.5944,5.9408) -- (10.3063,6.3958) -- (9.9870,6.8264) -- (9.6386,7.2297) -- (9.2632,7.6030) -- (8.8635,7.9438) -- (8.4422,8.2497) -- (8.0024,8.5185) -- (7.5475,8.7482) -- (7.0814,8.9369) -- (6.6082,9.0832) -- (6.1325,9.1862) -- (5.6592,9.2455) -- (5.1934,9.2615) -- (4.7403,9.2355) -- (4.3048,9.1693) -- (3.8917,9.0661) -- (3.5052,8.9294) -- (3.1487,8.7635) -- (2.8247,8.5732) -- (2.5351,8.3635) -- (2.2807,8.1394) -- (2.0614,7.9060) -- (1.8764,7.6677) -- (1.7244,7.4287) -- (1.6032,7.1926) -- (1.5108,6.9624) -- (1.4445,6.7406) -- (1.4019,6.5291) -- (1.3803,6.3291) -- (1.3773,6.1417) -- (1.3908,5.9673) -- (1.4187,5.8063) -- (1.4593,5.6585) -- (1.5109,5.5240) -- (1.5724,5.4024) -- (1.6426,5.2935) -- (1.7208,5.1971) -- (1.8063,5.1130) -- (1.8985,5.0411) -- (1.9969,4.9814) -- (2.1013,4.9340) -- (2.2111,4.8991) -- (2.3261,4.8771) -- (2.4459,4.8684) -- (2.5700,4.8735) -- (2.6979,4.8931) -- (2.8292,4.9280) -- (2.9630,4.9788) -- (3.0986,5.0466) -- (3.2349,5.1324) -- (3.3708,5.2373) -- (3.5047,5.3624) -- (3.6348,5.5088) -- (3.7588,5.6775) -- (3.8742,5.8694) -- (3.9780,6.0850) -- (4.0667,6.3245) -- (4.1365,6.5876) -- (4.1836,6.8732) -- (4.2037,7.1796) -- (4.1929,7.5043) -- (4.1471,7.8438) -- (4.0629,8.1938) -- (3.9375,8.5493) -- (3.7687,8.9046) -- (3.5556,9.2533) -- (3.2979,9.5889) -- (2.9970,9.9047) -- (2.6550,10.1940) -- (2.2754,10.4508) -- (1.8627,10.6692) -- (1.4223,10.8445) -- (0.9605,10.9726) -- (0.4839,11.0507) -- (0.0000,11.0769) --  cycle;
 \end{tikzpicture}
\end{center}

We will mention one other approach to the calculation of $m(z)$, but
we will not go into great detail because we have not found it to be
computationally efficient.

\begin{definition}\lbl{defn-tensor-sections}
 Put
 \[ A_k = \{f\in\Hol(\Dl)\st f(z) = f(\gm(z)) \gm'(z)^k
     \text{ for all } \gm\in\Pi \text{ and } z\in\Dl
    \}.
 \]
 Multiplication by $dz^{\ot k}$ identifies $A_k$ with the space of
 holomorphic sections of the $k$'th tensor power of the cotangent
 bundle of $HX(b)$.  In particular, we have $A_1=\Om^1(HX(b))$.
\end{definition}

\begin{lemma}\lbl{lem-riemann-roch}
 \[ \dim_\C(A_k) = \begin{cases}
     0 & \text{ if } k < 0 \\
     1 & \text{ if } k = 0 \\
     2 & \text{ if } k = 1 \\
     2k-1 & \text{ if } k > 1.
    \end{cases}
 \]
\end{lemma}
\begin{proof}
 This is a standard consequence of the Riemann-Roch theorem
 (see~\cite[Section IV.1]{ha:ag}, for example).  In more
 detail, that theorem tells us that for any line bundle $\CL$ over a
 compact Riemann surface $Z$ of genus $g$, we have
 \[ \dim(H^0(Z;\CL)) - \dim(H^0(Z;\Om^1\ot\CL^*)) =
     \deg(\CL) + 1 - g.
 \]
 Now take $Z=HX(b)$ (so $g=2$) and put
 $f(n)=\dim(H^0(HX(b);(\Om^1)^{\ot n}))$ for $n\in\Z$.  We have seen
 that $\{\om_0,\om_1\}$ is a basis for $H^0(HX(b);\Om^1)$ over $\C$,
 so $f(1)=2$.  On any compact Riemann surface, the only holomorphic
 $\C$-valued functions are constant, so $f(0)=1$.  The Riemann-Roch
 theorem gives $f(n)-f(1-n)=n\,\deg(\Om^1)-1$.  Taking $n=1$ gives
 $\deg(\Om^1)=2$, so we get $f(n)-f(1-n)=2n-1$.  Moreover, as
 $\deg(\Om^1)>0$ we have $f(n)=0$ for $n<0$.  Thus, when $n>1$ we have
 $f(1-n)=0$ and so $f(n)=2n-1$ as required.
\end{proof}

\begin{definition}\lbl{defn-automorphic-series}
 For $j\geq 0$ and $k\geq 2$ we put
 \[ p_{jk}(z) = \sum_{\gm\in\Pi} \gm(z)^j\gm'(z)^k. \]
\end{definition}

It is a standard theorem that the above series is absolutely uniformly
convergent on compact subsets of $\Dl$.  We will give a proof that
includes explicit bounds in the case of interest.

\begin{definition}\lbl{defn-automorphic-area}
 For any subset $A\sse\Pi$ we put
 \[ a(A) =
     \text{area}\left(\bigcup_{\gm\in A}\gm(HF_1(b))\right) =
     \sum_{\gm\in A}\text{area}(\gm(HF_1(b))).
 \]
 (Here areas are defined in terms of the Euclidean metric on $\Dl$,
 not the hyperbolic metric.)  We also write
 $a(\gm)=a(\{\gm\})=\text{area}(\gm(HF_1(b)))$.  In particular,
 $a(1)$ is just the area of $HF_1(b)$.
\end{definition}

\begin{remark}\lbl{rem-automorphic-area}
 We can find $a(1)$ as follows.  It is easy to see that $a(1)$ is $8$
 times the area of $HF_8(b)$.  The boundary of this region consists
 of straight lines joining $v_0=0$ to $v_1$ and $v_{11}$, together
 with an arc of the circle $C_3$ joining $v_{11}$ to $v_{13}$, and an
 arc of the circle $C_7$ joining $v_1$ to $v_{13}$.  Recall that $C_3$
 has centre $a_3=b_+$ and radius $r_3=\sqrt{|a_3|^2-1}=b$, whereas
 $C_7$ has centre $a_7=b_+^{-1}+\half ib_+$ and radius
 $r_7=\sqrt{|a_7|^2-1}=\half b_-^2b_+^{-1}$.  One can check that
 $v_{11}=a_3-r_3$ and $v_1=a_7-r_7$ and
 \[ v_{13} = a_3 - r_3 e^{-i\tht} = a_7 - i\,r_7 e^{-i\tht} \]
 where $\tht=2\arctan(b(b_+-b))$.  Using this, it is easy to
 parameterise the boundary of $HF_8(b)$ and use Green's theorem (in the
 form $\text{area}(D)=-\oint_{\partial D}y\,dx$) to calculate the
 area.  We eventually arrive at the following formula:
 \[ a(1) =
  3+b^2
  - \frac{2bb_-^2}{b_+} - \frac{b_-^4}{b_+^2}\frac{\pi}{2}
  - \frac{3b^4+6b^2-1}{1+b^2}\tht.
 \]
 The graph is as follows:
 \begin{center}
  \begin{tikzpicture}[scale=6]
   \draw[->] (-0.03,0) -- (1.05,0);
   \draw[->] (0,-0.03) -- (0,0.55);
   \draw[smooth,red] (0.000,0.455) -- (0.100,0.464) -- (0.200,0.480) -- (0.300,0.493) -- (0.400,0.498) -- (0.500,0.491) -- (0.600,0.472) -- (0.700,0.440) -- (0.800,0.396) -- (0.900,0.340) -- (1.000,0.273);
   \draw (0.2,0) -- (0.2,-0.02);
   \draw (0.4,0) -- (0.4,-0.02);
   \draw (0.6,0) -- (0.6,-0.02);
   \draw (0.8,0) -- (0.8,-0.02);
   \draw (1.0,0) -- (1.0,-0.03);
   \draw (0,0.1) -- (-0.02,0.1);
   \draw (0,0.2) -- (-0.02,0.2);
   \draw (0,0.3) -- (-0.02,0.3);
   \draw (0,0.4) -- (-0.02,0.4);
   \draw (0,0.5) -- (-0.03,0.5);
   \draw (1.05,0) node[anchor=west] {$b$};
   \draw (0.8,0.47) node {$a(1)$};
   \draw (0,-0.03) node[anchor=north] {$0$};
   \draw (1,-0.03) node[anchor=north] {$1$};
   \draw (-0.03,0.0) node[anchor=east] {$0$};
   \draw (-0.03,0.5) node[anchor=east] {$\pi/2$};
  \end{tikzpicture}
 \end{center}
 \begin{checks}
  hyperbolic/HX_check.mpl: check_F1_area()
 \end{checks}
\end{remark}

\begin{proposition}\lbl{prop-automorphic-conv}
 Fix $k\geq 2$ and $\dl>0$, and let $K$ be the closed disc of radius
 $1-\dl$ centred at the origin.  Then the series defining $p_{jk}(z)$
 converges absolutely and uniformly on $K$.  More precisely, for any
 subset $A\sse\Pi$ and any $z\in K$ we have
 \[ \left|p_{jk}(z)-\sum_{\gm\in A}\gm(z)^j\gm'(z)^k\right|
     \leq \sum_{\gm\in\Pi\sm A} |\gm'(z)|^k
     \leq\dl^{-2k}a(\Pi\sm A)/a(1) = \dl^{-2k}(\pi - a(A))/a(1).
 \]
\end{proposition}
\begin{proof}
 Put $F=HF_1(b)$.  Note that the images $\{\gm(F)\st\gm\in\Pi\}$
 cover $\Dl$, and any two of these images have intersection of measure
 zero, so
 \[ \sum_{\gm\in\Pi}a(\gm)=\text{area}(\Dl)=\pi. \]
 Next, if we write $\gm(z)$ in the form $(az+b)/(cz+d)$ with
 $ad-bc=1$, then we find that $\gm'(z)=(cz+d)^{-2}$, so the Jacobian
 determinant is $|cz+d|^{-4}$.  It follows that
 $a(\gm)=\iint_{F}|cz+d|^{-4}$.   Corollary~\ref{cor-bound-ii} tells
 us that $|cz+d|\leq(\sqrt{2}+1)|c|$, so
 $|c|^{-4}\leq(\sqrt{2}+1)^4|cz+d|^{-4}$, so $|c|^{-4}a(1)\leq a(\gm)$.

 Now let $z$ be any point in $K$.  Clearly $|\gm(z)^j|\leq 1$.  We have
 $|cz+d|^{-2k}=|c|^{-2k}|z+d/c|^{-2k}$.  Corollary~\ref{cor-bound-ii}
 also gives $|d/c|\geq 1$ so for $z\in K$ we have
 $|z+d/c|^{-2k}\leq\dl^{-2k}$.  The same result also gives $|c|\geq 1$
 and $k\geq 2$ by assumption so
 $|c|^{-2k}\leq|c|^{-4}\leq a(\gm)/a(1)$.  Putting this together gives
 $|\gm(z)^j\gm'(z)^k|\leq\dl^{-2k}a(\gm)/a(1)$.  Taking the sum over $\gm$
 gives
 \[ \sum_{\gm\not\in A}|\gm'(z)|^{-k}\leq\dl^{-2k}a(\Pi\sm A)/a(1). \]
\end{proof}

\begin{corollary}\lbl{cor-automorphic-series}
 $p_{jk}\in A_k$ for all $j$ and $k$.
\end{corollary}
\begin{proof}
 The proposition shows that the series for $p_{jk}(z)$ is absolutely
 uniformly convergent on compact subsets of $\Dl$.  This validates the
 following manipulation:
 \begin{align*}
  p_{jk}(\dl(z))\dl'(z)^k &=
   \sum_{\gm\in\Pi} (\gm\dl)(z)^j \gm'(\dl(z))^k\dl'(z)^k \\
   &= \sum_{\gm\in\Pi} (\gm\dl)(z)^j (\gm\dl)'(z)^k \\
   &= \sum_{\ep\in\Pi} \ep(z)^j\ep'(z)^k = p_{jk}(z).
 \end{align*}
\end{proof}

Next, the isomorphism $A_2=\Gamma(HX(b);\Om^{\ot 2})$ gives rise to an
action of $G$ on $A_2$; we will write $\gm^\bullet f$ for the action
of $\gm$ on $f$.  Note that the functions $m(z)$ and
\[ n(z) = (\mu^\bullet m)(z) = m(\mu(z))\mu'(z) \]
give a basis for $A_1$.  It follows that the functions $m^2$, $n^2$
and $mn$ are linearly independent in $A_2$, and $A_2$ has dimension
$3$ by Lemma~\ref{lem-riemann-roch}, so the indicated elements must
give a basis.  In particular, we have $p_{02}=a_0m^2+a_1n^2+a_2mn$ for
some constants $a_i$.  One can check that $p_{02}$, $m^2$ and $n^2$
are fixed by $\lm$ whereas $mn$ is negated, so $a_2=0$.  Using the
automorphy properties of $m$ together with the relation
$\bt_6\mu(v_1)=v_0$ one can check that $m(v_1)=n(v_0)=0$ and
\[ n(v_1)/m(v_0) = (\bt_6\mu)'(v_1) = 2i/(1-b^2). \]
From this we obtain $a_1=a^*a_0$, where
\[ a^* = -\tfrac{1}{4}(1-b^2) p_{02}(v_1)/p_{02}(v_0). \]
From this it follows that $p_{02}-a^*\mu^\bullet p_{02}$ is a constant
multiple of $m^2$.  Thus, if we can calculate $p_{02}$ effectively,
then we can recover the function $m$ up to a constant multiplier,
without using the methods in Sections~\ref{sec-a-from-b}
and~\ref{sec-b-from-a}.  We can then find the map $p$ from the
relation $p=(\mu^\bullet m)/m$, noting that the unknown constant
cancels out.

However, in practice we need an extremely large number of terms to
calculate $p_{02}$ accurately, so this is not an efficient approach.
Various tactics are available to streamline the calculation, but they
are not sufficient to change the conclusion.

The above algorithm is implemented by the class
\mcode+automorphy_system+, which is declared in the file
\fname+hyperbolic/automorphic.mpl+.  In more detail, we can enter the
following:
\begin{mcodeblock}
   AS := `new/automorphy_system`();
   AS["a_H"] := a_H0;
   AS["poly_deg"] := 100;
   AS["band_number"] := 4;
   AS["set_p0_series",2]:
   AS["set_m_series"]:
   AS["m_series"](z);
\end{mcodeblock}
The last line will give a polynomial $m^*(z)$ of degree $100$ which
approximates $m(z)/m(0)$.  It is based on a calculation of $p_{02}(z)$
obtained by summing over a certain subset of $\Pi$.  More
specifically, we can put $B_0=\{1\}$ and
\[ B_1 = \{\gm\in\Pi \st \gm(HF_1(b))\cap HF_1(b)\neq\emptyset\}, \]
then we can define $B_n=B_1.B_{n-1}$ recursively.  There are $25$
elements in $B_1$, and it is not hard to list them explicitly, and
then to give an algorithm which enumerates $B_n$ for all $n$.  The
line
\begin{mcodeblock}
   AS["band_number"] := 4;
\end{mcodeblock}
specifies that sums should be taken over the set $B_4$, which has
$156772$ elements.  We find that $|m(0)m^*(z)-m(z)|\leq 10^{-5}$ for
$|z|\leq 0.5$, but the error grows to about $10^{-3}$ when $|z|=0.65$,
and becomes very large when $|z|>0.8$.

\section{The embedded family}
\lbl{sec-E}

\subsection{Geometry behind the definition}
\lbl{sec-E-geometry}

\begin{definition}\lbl{defn-X}
 Fix $a\in (0,1)$.  For $x=(x_1,x_2,x_3,x_4)\in\R^4$ we put
 \begin{align*}
  \rho(x) &= x_1^2+x_2^2+x_3^2+x_4^2 = \sum_i x_i^2 \\
  f_1(x) &= 2 x_2^2 + (x_4-1-x_3/a)^2 \\
  f_2(x) &= 2 x_1^2 + (x_4-1+x_3/a)^2 \\
  f(x)   &= f_1(x)f_2(x) \\
  EX(a)  &= \{x\in\R^4\st \rho(x)=1 \text{ and } f(x)=f(-x) \}.
 \end{align*}
\end{definition}
Straightforward algebra shows that this is the same as the definition
in the Introduction.

Now put
\[ \Om^+_1 = \{x\in S^3 \st f_1(x)=0\}. \]
Recall that $f_1(x)=2 x_2^2 + (x_4-1-x_3/a)^2$, and a sum of two real
squares can only be zero if the individual terms are zero.  It follows that
\begin{align*}
 \Om^+_1 &= \{x\in S^3 \st x_2=0 \text{ and } x_4=1+x_3/a\} \\
  &= \{(x_1,0,x_3,1+x_3/a)\in\R^4 \st x_1^2+x_3^2+(1+x_3/a)^2=1\}.
\end{align*}
This is the intersection of $S^3$ with a two-dimensional affine
subspace of $\R^4$, so it is a circle.  This circle passes through
the point $(0,0,0,1)$, which corresponds to infinity under the
stereographic projection map $s\:S^3\to\R^3\cup\{\infty\}$ that we
defined in the Introduction.  This means that the image $s(\Om^+_1)$
is a ``circle through $\infty$'', or in other
words a straight line.  In fact, one can check that
\[ s(\Om^+_1) = \{(x,y,z)\in\R^3\st y=0,z=-a\}. \]
Similarly, the set
\[ \Om^+_2 = \{x\in S^3 \st f_2(x)=0\} \]
is another circle in $S^3$, with stereographic projection
\[ s(\Om^+_2) = \{(x,y,z)\in\R^3\st x=0,z=+a\}. \]
\begin{checks}
 embedded/geometry_check.mpl: check_omega()
\end{checks}
Note that the lines $s(\Om^+_1)$ and $s(\Om^+_2)$ are at right angles
to each other, but they do not touch except at $\infty$.  We can also
put
\[ \Om^+ = \{x\in S^3\st f(x)=0\} = \{x\in S^3\st f_1(x)f_2(x)=0\}
    = \Om^+_1\cup\Om^+_2.
\]

Now recall that the definition of $X$ also involves $f(-x)$, so we
should study the sets $\Om^-_i=\{x\in S^3\st f_i(-x)=0\}$ and
$\Om^-=\{x\in S^3\st f(-x)=0\}=\Om^-_1\cup\Om^-_2$.  The sets
$\Om_i^-$ are again circles in $S^3$, but they do not pass through
$(0,0,0,1)$ so their stereographic projections are circles rather than
straight lines.  In fact one can check that
\begin{align*}
 s(\Om^-_1) &= \{(x,0,z)\in\R^3\st x^2+(z-1/(2a))^2=1/(2a)^2\} \\
  &= \text{ the circle of radius $1/(2a)$ in the $(x,z)$-plane centred at $(0,0,1/(2a))$} \\
 s(\Om^-_2) &= \{(0,y,z)\in\R^3\st y^2+(z+1/(2a))^2=1/(2a)^2\} \\
  &= \text{ the circle of radius $1/(2a)$ in the $(y,z)$-plane centred at $(0,0,-1/(2a))$}.
\end{align*}
\begin{checks}
 embedded/geometry_check.mpl: check_omega()
\end{checks}
The picture for $a=1/\rt$ is as follows:
\[ \includegraphics[scale=0.25,clip=true,
     trim=6cm 6cm 6cm 6cm]{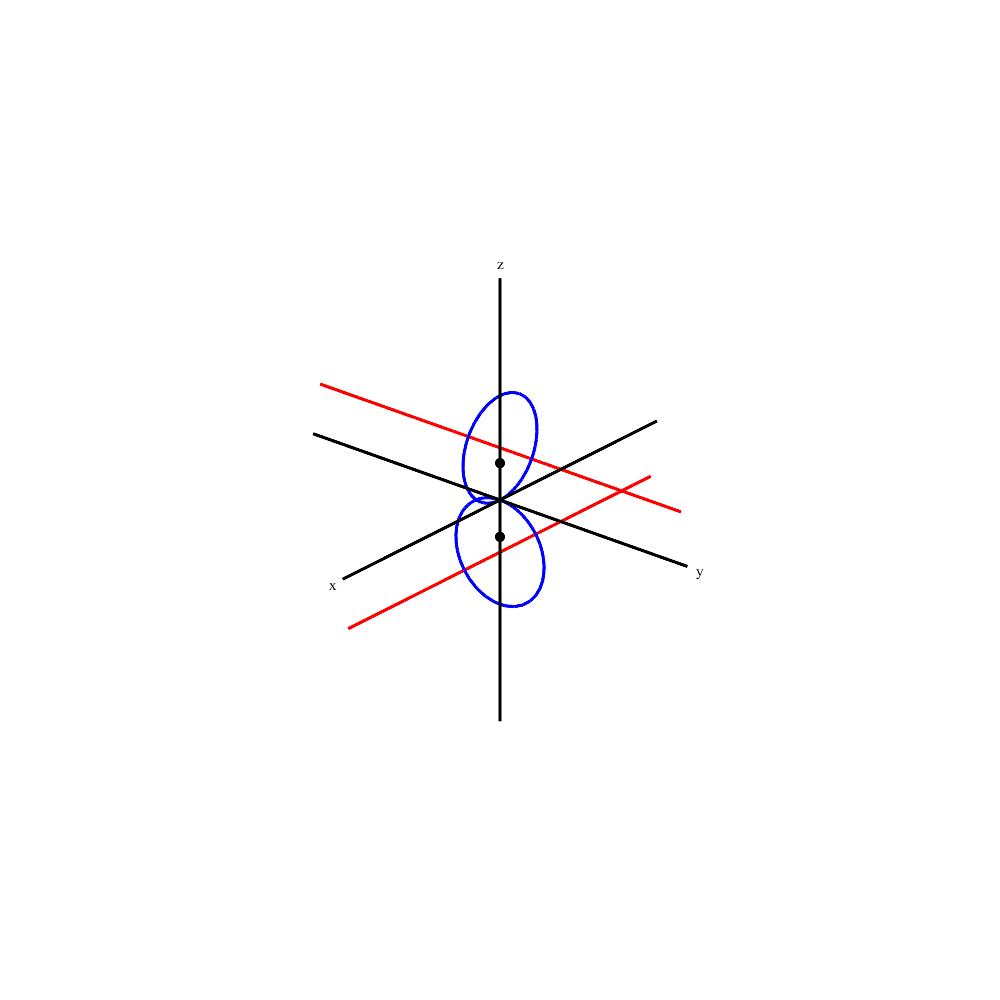} \]
(The set $s(\Omega^+)$ is shown in red, and the set $s(\Omega^-)$ is
shown in blue.)

On $\Om^+$ we have $f(x)=0$ and $f(-x)>0$, whereas on $\Om^-$ we have
$f(x)>0$ and $f(-x)=0$.  We defined $EX(a)=\{x\in S^3\st f(x)=f(-x)\}$, and
we now see that this fits between $\Om^+$ and $\Om^-$.  This can be
displayed as follows:
\[ \includegraphics[scale=0.25,clip=true,
       trim=6cm 6cm 6cm 6cm]{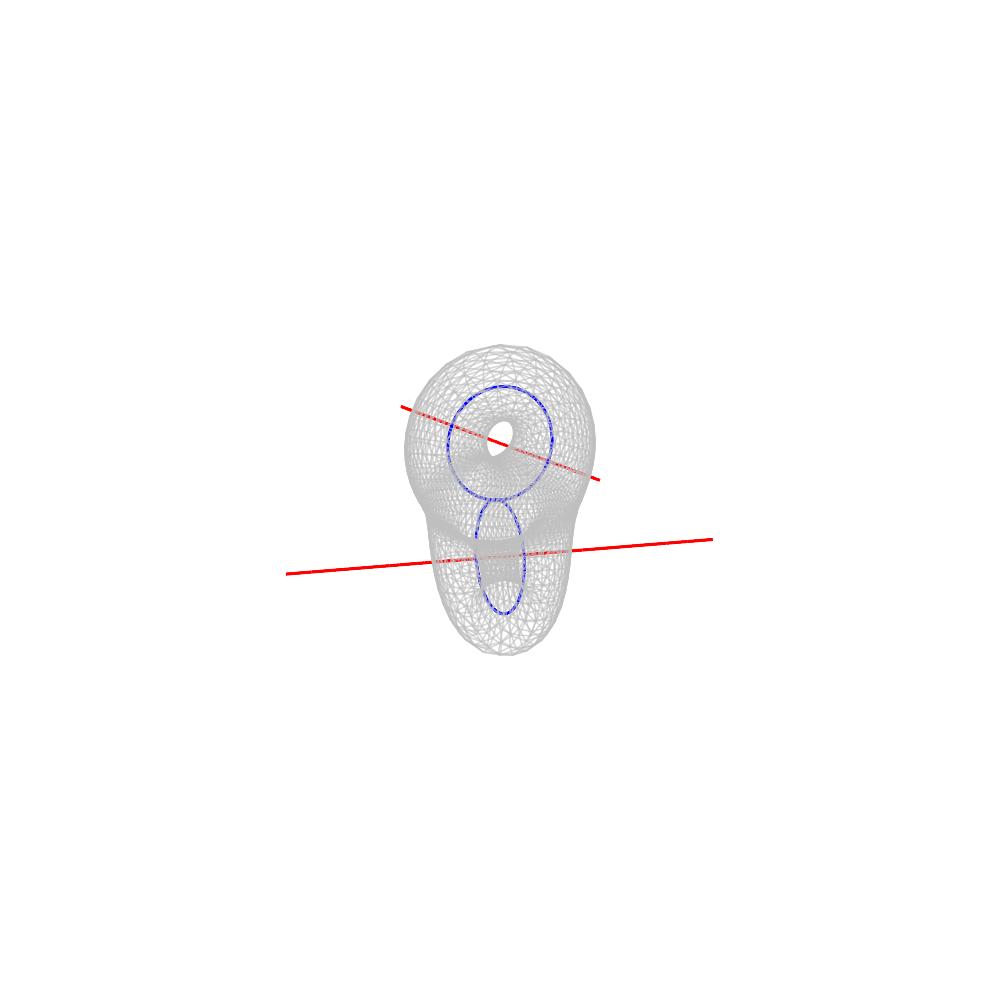}
\]

It will be convenient to describe $EX(a)$ using the functions $g$ and
$g_0$ given below.
\begin{definition}\lbl{defn-g}
 We put
 \begin{align*}
  g_0(x) &= (f(x)-f(-x))/8 + x_4(\rho(x)-1) \\
         &= ((1+a^{-2})x_3^2-2)x_4+a^{-1}(x_1^2-x_2^2)x_3 \\
  g(x) &= (f(x)-f(-x))/8 - x_4(\rho(x)-1) \\
       &= (a^{-2}-1)x_3^2x_4-2(x_1^2+x_2^2)x_4-2x_4^3+a^{-1}(x_1^2-x_2^2)x_3.
 \end{align*}
 (The advantage of $g_0(x)$ is that it has few terms, and the advantage
 of $g(x)$ is that it is a homogeneous cubic.)
 \begin{checks}
  embedded/geometry_check.mpl: check_g()
 \end{checks}

 It is now straightforward to check that
 \[ EX(a) = \{x\in S^3 \st g_0(x)=0\} = \{x\in S^3\st g(x)=0\}. \]
 We put
 \begin{align*}
  \tA &= \R[x_1,x_2,x_3,x_4] \\
  A   &= \CO_{EX(a)} = \tA/(\rho(x)-1,g(x)),
 \end{align*}
 so $A$ is the ring of polynomial functions on $EX(a)$.
\end{definition}
\begin{remark}
 Maple notation for the parameter $a$ is $a_E$.  The global variable
 \mcode+a_E0+ is set to $1/\rt$, and \mcode+a_E1+ is a 100 digit
 approximation to that.  Elements of $\R^4$ are represented in Maple
 as lists of length $4$.  The functions $\rho(x)$, $f_1(x)$, $f_2(x)$,
 $f(x)$, $g_0(x)$ and $g(x)$ are \mcode+rho(x)+, \mcode+f_1(x)+,
 \mcode+f_2(x)+, \mcode+f(x)+, \mcode+g_0(x)+ and \mcode+g(x)+.  The
 functions obtained from these by setting $a=1/\rt$ are \mcode+f_10(x)+,
 \mcode+f_20(x)+, \mcode+f0(x)+, \mcode+g_00(x)+ and \mcode+g0(x)+.
 Note in particular the difference between \mcode+g_0(x)+ and
 \mcode+g0(x)+.
\end{remark}

\begin{proposition}\lbl{prop-X-smooth}
 The space $EX(a)$ is a compact smooth oriented embedded submanifold in
 $S^3$.
\end{proposition}

We will use the following notation:
\begin{definition}\lbl{defn-nx}
 We put
 \begin{align*}
  n(x) &= (\nabla g)_x =
           (\partial g/\partial x_1,\dotsc,\partial g/\partial x_4) \\
   &= ( 2x_1(x_3/a-2x_4),\;
       -2x_2(x_3/a+2x_4), \\
   &\qquad
       2(a^{-2}-1)x_3x_4+a^{-1}(x_1^2-x_2^2),\;
       (a^{-2}-1)x_3^2-6x_4^2-2(x_1^2+x_2^2)).
 \end{align*}
 (This is \mcode+dg(x)+ in Maple.)
\end{definition}
\begin{proof}
 Define $p\:\R^4\to\R^2$ by $p(x)=(\rho(x)-1,g(x))$, so
 $EX(a)=p^{-1}\{0\}$.  We must first show that $0$ is a regular value
 of $p$, or equivalently that the gradients of $\rho-1$ and $g$ are
 linearly independent at every point in $EX(a)$.  Note here that the
 gradient of $\rho-1$ at $x$ is just $2x$, which is certainly nonzero
 at all points in $EX(a)$.  Moreover, as $g$ is a homogeneous cubic
 function we have $(2x).n(x)=6g(x)$, which is zero on $EX(a)$, so the
 two gradients are orthogonal.  It will thus be enough to show that
 $n(x)$ is nonzero everywhere on $EX(a)$.  By direct expansion one can
 check that
 \[ \frac{1-a^2}{16}(n(x)_1^2+n(x)_2^2) +
     \frac{a}{2}(x_1^2-x_2^2)n(x)_3 -
     \frac{1}{4}(x_1^2+x_2^2)n(x)_4 =
      x_1^4+x_2^4+(\tfrac{5}{2}-a^2)(x_1^2+x_2^2)x_4^2.
 \]
 \begin{checks}
  embedded/geometry_check.mpl: check_g()
  embedded/EX_check.mpl: check_smoothness()
 \end{checks}
 If $n(x)=0$ then the left hand side vanishes so the right hand side
 must also vanish.  After noting that $a\in (0,1)$ so
 $\frac{5}{2}-a^2>0$, it follows that $x_1=x_2=0$.  After substituting
 this back into the equations for $n(x)$, we see that $x_3x_4=0$ and
 $(a^{-2}-1)x_3^2-6x_4^2=0$, which easily implies that $x_3=x_4=0$.
 Thus, the only place in $\R^4$ where $n(x)=0$ is the origin, so in
 particular there are no points in $EX(a)$ with this property.

 This completes the proof that $0$ is a regular value of $p$, which
 implies in a standard way that $EX(a)$ is a smooth closed submanifold
 of $\R^4$.  It is compact because it is closed in $S^3$.  Now let
 $\om_k$ denote the standard volume form in $\Lm^k(\R^k)$.  Standard
 exterior algebra now tells us that for each $x\in EX(a)$ there is a
 unique element $\al_x\in\Lm^2(T_xEX(a))<\Lm^2(\R^4)$ such that
 $\al_x\wedge p^*(\om_2)=\om_4$.  These forms give a smooth, nowhere
 vanishing section of $\Lm^2(T)$ over $EX(a)$, and thus an orientation
 of $EX(a)$.
\end{proof}

\begin{remark}\lbl{rem-move-to-X}
 The above considerations also give an efficient practical method to
 compute points on $X$ numerically.  For any $y\in\R^4$, we define
 \[ \sg(y) =
     \frac{1}{\|y\|}\left(y - \frac{g(y)}{\|n(y)\|^2}n(y)\right).
 \]
 One can show that if the distance from $y$ to $X$ is of order
 $\ep\ll 1$, then the distance from $\sg(y)$ to $X$ is of order
 $\ep^2$.  This implies that the sequence $(\sg^k(x))_{k\geq 0}$
 converges rapidly to a point $\sg^\infty(x)\in X$.  This is
 implemented by the Maple function \mcode+move_to_X(x)+.
\end{remark}

Next, we use the metric and the orientation to give an almost complex
structure on $X$.  Explicitly, for each $x\in X$ there is a unique
isometric linear map $J_x\:T_xX\to T_xX$ such that $\ip{u,J_xu}=0$ for
all $u$, and $u\wedge J_xu$ is a positive multiple of $\al_x$ for all
$u\neq 0$.  This satisfies $J_x^2=-1$ and so gives a complex structure
on $T_xX$.  All this can be verified easily after choosing an oriented
orthonormal basis.  As mentioned in the introduction, this can be
integrated to give a complex structure on $X$.  In more detail, if
$U\sse X$ is open and $q\:U\to\C$ is a smooth map, we say that $q$ is
\emph{holomorphic} (or \emph{conformal}) if for each $x\in U$, the
derivative $Dq_x\:T_xX\to\C$ is $\C$-linear.  It is a nontrivial fact
that for each $x\in X$ there is an open neighbourhood $U$ of $X$ and
an injective holomorphic function $q\:U\to\C$ such that
$Dq_y\:T_yX\to\C$ is an isomorphism for all $y\in U$.  This has been
known at least since the work of Korn and Lichtenstein in~1916.  A
modern proof is given in~\cite{deka:srt}.  As we mentioned in the
introduction, it can also be seen as a special case of the
Newlander-Nirenberg theorem.  It is clear that the transition map
between any pair of holomorphic charts as above, is a function on an
open domain in $\C$ that is holomorphic in the traditional sense.  We
can thus use these charts to regard $X$ as a Riemann surface.

\subsection{The group action}
\lbl{sec-E-G}

We define linear isometric maps $\lm,\mu,\nu:\R^4\to\R^4$ as follows:
\begin{align*}
 \lm(x_1,x_2,x_3,x_4) &= (   -x_2,\pp x_1,\pp x_3,   -x_4) \\
 \mu(x_1,x_2,x_3,x_4) &= (\pp x_1,   -x_2,   -x_3,   -x_4) \\
 \nu(x_1,x_2,x_3,x_4) &= (\pp x_1,   -x_2,\pp x_3,\pp x_4).
\end{align*}
Maple notation for $g.x$ is \mcode+act_R4[g](x)+.

It is straightforward to check that
\[ \lm^4=\mu^2=\nu^2=(\mu\nu)^2=(\mu\lm)^2=(\nu\lm)^2=1, \]
so we have an isometric action of our group $G$ on $\R^4$.
\begin{checks}
 embedded/EX_check.mpl: check_R4_action()
\end{checks}
All three of the generators have determinant $-1$ and so reverse the
orientation of $\R^4$.  Note also that $\lm^2\mu\nu(x)=-x$.

Directly from the definitions we see that
\begin{align*}
 \rho(\lm(x)) &= \rho(\mu(x))=\rho(\nu(x))=\rho(x) \\
 g(\lm(x))    &= g(\mu(x))=-g(x) \\
 g(\nu(x))    &= g(x)
\end{align*}
This means that $\lm$, $\mu$ and $\nu$ send $X$ to itself.
\begin{checks}
 embedded/EX_check.mpl: check_symmetry()
\end{checks}

\begin{remark}\lbl{rem-heegard}
 If we put $D_+=\{x\in S^3\st g(x)\geq 0\}$ and
 $D_-=\{x\in S^3\st g(x)\leq 0\}$ then we have $D_+\cup D_-=S^3$ and
 $D_+\cap D_-=EX(a)$, and the map $\mu$ gives a homeomorphism between
 $D_+$ and $D_-$.  Moreover, $D_+$ and $D_-$ are handlebodies, with
 the figure eight curve $\Om^-$ as a deformation retract of $D_+$, and
 $\Om^+$ as a deformation retract of $D_-$.  This gives an unusually
 explicit Heegaard splitting of $S^3$ along $EX(a)$.
\end{remark}

Recall that the orientation form $\al_x\in\Lm^2(T_xX)$ is
characterised by the equation $\al_x\wedge (\rho-1,g)^*(\om_2)=\om_4$.
As $\lm$ preserves $\rho-1$ and changes the sign of $g$ and $\om_4$,
we deduce that $\lm^*\al=\al$, so $\lm$ preserves orientation.  As the
complex structure on $EX(a)$ is determined by the metric and the
orientation, it follows that $\lm$ is a conformal automorphism of
$EX(a)$.  By the same kind of logic, the map $\mu$ preserves
orientation and is conformal, whereas $\nu$ reverses orientation and
is anticonformal.

\subsection{Isotropy}
\lbl{sec-E-isotropy}

We now define points $v_i\in\R^4$ as follows:
\begin{align*}
 v_{ 0} &= (\pp 0,\pp 0,\pp 1,\pp 0) &
 v_{ 6} &= (\pp 1,\pp 1,\pp 0,\pp 0)/\rt \\
 v_{ 1} &= (\pp 0,\pp 0,   -1,\pp 0) &
 v_{ 7} &= (   -1,\pp 1,\pp 0,\pp 0)/\rt \\
 v_{ 2} &= (\pp 1,\pp 0,\pp 0,\pp 0) &
 v_{ 8} &= (   -1,   -1,\pp 0,\pp 0)/\rt \\
 v_{ 3} &= (\pp 0,\pp 1,\pp 0,\pp 0) &
 v_{ 9} &= (\pp 1,   -1,\pp 0,\pp 0)/\rt \\
 v_{ 4} &= (   -1,\pp 0,\pp 0,\pp 0) &
 v_{10} &= (0,0,\pp\sqrt{\frac{2a^2}{1+a^2}},\pp\sqrt{\frac{1-a^2}{1+a^2}}) \\
 v_{ 5} &= (\pp 0,-1,\pp 0,\pp 0) &
 v_{11} &= (0,0,\pp\sqrt{\frac{2a^2}{1+a^2}},  -\sqrt{\frac{1-a^2}{1+a^2}}) \\
 &&
 v_{12} &= (0,0,  -\sqrt{\frac{2a^2}{1+a^2}},  -\sqrt{\frac{1-a^2}{1+a^2}}) \\
 &&
 v_{13} &= (0,0,  -\sqrt{\frac{2a^2}{1+a^2}},\pp\sqrt{\frac{1-a^2}{1+a^2}})
\end{align*}
These are \mcode+v_E[i]+ in Maple.

Routine calculation shows that these points all lie in $X$, and that
the action of $G$ permutes them, according to the following rules:
\begin{align*}
 \lm &\mapsto (2\;3\;4\;5)\;(6\;7\;8\;9)\;(10\;11)\;(12\;13) \\
 \mu &\mapsto (0\;1)\;(3\;5)\;(6\;9)\;(7\;8)\;(10\;12)\;(11\;13) \\
 \nu &\mapsto (3\;5)\;(6\;9)\;(7\;8).
\end{align*}
(For example, the cycle $(2\;3\;4\;5)$ in $\lm$ means that
$\lm(v_2)=v_3$, $\lm(v_3)=v_4$, $\lm(v_4)=v_5$ and $\lm(v_5)=v_2$.)
This almost shows that we have a precromulent surface, except that we
need to check that there are no further points with nontrivial
stabiliser in $D_8$.

\begin{proposition}\lbl{prop-no-more-isotropy}
 If $x\in EX(a)\sm\{v_0,\dotsc,v_{13}\}$ then $\stab_{D_8}(x)=1$.
\end{proposition}
\begin{proof}
 Consider for example a point $x\in EX(a)$ with $\lm^2(x)=x$, or in other
 words $(-x_1,-x_2,x_3,x_4)=(x_1,x_2,x_3,x_4)$ or $x_1=x_2=0$.  The
 equations for $EX(a)$ become $x_3^2+x_4^2=1$ and
 $x_4((a^{-2}-1)x_3^2-2x_4^3)=0$.  If $x_4=0$ we must have
 $x=(0,0,\pm 1,0)$, so $x=v_0$ or $x=v_1$.  Otherwise we must have
 $(a^{-2}-1)x_3^2=2x_4^2$.  In conjunction with $x_3^2+x_4^2=1$ this gives
 $x_3^2=2a^2/(1+a^2)$ and $x_4^2=(1-a^2)/(1+a^2)$ so
 $x\in\{v_{10},v_{11},v_{12},v_{13}\}$.  Note that if $x$ is fixed by
 $\lm$ or $\lm^3=\lm^{-1}$ then it is certainly fixed by $\lm^2$, so
 again it is one of the $v_i$.

 Similarly:
 \begin{itemize}
  \item We have $\mu(x)=x$ iff $x_2=x_3=x_4=0$, and $\lm^2\mu(x)=x$
   iff $x_1=x_3=x_4=0$.  In these cases it is clear that
   $x\in\{v_2,v_3,v_4,v_5\}$.
  \item We have $\lm\mu(x)=x$ iff $x_1=x_2$ and $x_3=0$.  In this
   context we have $g_0(x)=-2x_4$, so we must also have $x_4=0$, which
   makes it clear that $x\in\{v_6,v_8\}$.
  \item A similar argument shows that if $\lm^3\mu(x)=x$ then
   $x\in\{v_7,v_9\}$.
 \end{itemize}
 \begin{checks}
  embedded/EX_check.mpl: check_fixed_points()
 \end{checks}
\end{proof}

\begin{remark}\lbl{rem-tangent}
 It will be useful to understand the tangent space $T_{v_0}X$ more
 explicitly.  The formula in Definition~\ref{defn-nx} shows that the
 gradient of $g$ at $v_0$ is
 \[ n(v_0) = n(e_3) = (0,0,0,a^{-2}-1) = (a^{-2}-1)e_4. \]
 On the other hand, we have $\nabla(\rho-1)_{v_0}=2v_0=2e_3$, so
 \[ (\rho-1,g)^*(\om_2) =
     \nabla(\rho-1)_{v_0}\wedge\nabla(g)_{v_0} =
      2(a^{-2}-1)e_3\wedge e_4.
 \]
 The tangent space is the orthogonal complement to $v_0$ and $n(v_0)$,
 so it is spanned by the basis vectors $e_1$ and $e_2$ (which are
 orthonormal).  The orientation form $\al_{v_0}$ must therefore be
 $(e_1\wedge e_2)/(2(a^{-2}-1))$, which is a positive multiple of
 $e_1\wedge e_2$, so $J_{v_0}(e_1)=e_2$.  In other words, the map
 $(x+iy)\mapsto(x,y,0,0)$ gives an isometric $\C$-linear isomorphism
 $\C\to T_{v_0}X$.  It is also clear from this that $\lm$ acts on
 $T_{v_0}EX(a)$ as multiplication by $i$.
\end{remark}

\subsection{Associated complex varieties}
\lbl{sec-E-complex}

For $x\in\C^4$ we again put $\rho(x)=\sum_jx_j^2$ (not
$\sum_j|x_j|^2$).  We then put
\begin{align*}
 CEX(a)  &= \{x\in\C^4 \st g(x)=0,\;\rho(x)=1\} \\
 PEX(a)  &= \{[x]\in\C P^3\st g(x)=0\} \\
 PEX'(a) &= \{[x]\in\C P^3\st g(x)=0,\;\rho(x)\neq 0\}.
\end{align*}
It is clear that $CEX(a)$ is an affine variety, and $PEX(a)$ is a
projective variety, and $PEX'(a)$ is a quasiprojective open subvariety
of $PEX(a)$.  The map $x\mapsto [x]$ gives a double covering
$CEX(a)\to PEX'(a)$.  We can identify $EX(a)$ with the set of real
points in $CEX(a)$.

\begin{proposition}
 If $a\neq 1/\rt$, then $CEX(a)$, $PEX(a)$ and $PEX'(a)$ are all
 smooth.
\end{proposition}
\begin{proof}
 As $PEX'(a)$ is open in $PEX(a)$ and $CEX(a)$ is a double cover of
 $PEX'(a)$, it will suffice to treat the case of $PEX(a)$.  The
 partial derivatives of $g$ are
 \begin{align*}
  n_1 = \partial g/\partial x_1 &= \pp 2x_1(x_3/a - 2x_4) \\
  n_2 = \partial g/\partial x_2 &=    -2x_2(x_3/a + 2x_4) \\
  n_3 = \partial g/\partial x_3 &= 2(a^{-2}-1)x_3x_4 + (x_1^2-x_2^2)/a \\
  n_4 = \partial g/\partial x_4 &= (a^{-2}-1)x_3^2 - 2x_1^2 - 2x_2^2 - 6x_4^2.
 \end{align*}
 Put
 \[ U_i = \{x\in\C^4\st x_i=1,\; n_1=n_2=n_3=n_4=0\}. \]
 By well-known arguments, it will suffice to check that the sets $U_i$
 are all empty.  Consider a point $x\in U_1$.  The equation $n_1=0$
 gives $x_3=2ax_4$, and we can substitute this in $n_2=0$ to get
 $x_2x_4=0$.  If $x_2\neq 0$ then this gives $x_4=0$ so $x_3=0$, and
 the relation $n_3=0$ gives a contradiction.  We must therefore have
 $x_2=0$.  Putting $x_1=1$ and $x_2=0$ and $x_3=4ax_4$ in $n_3=n_4=0$
 we get $2a^2=1$ and $2x_4^2=-1$.  We are assuming explicitly that
 $a\neq 1/\rt$ and implicitly that $a\in(0,1)$, so this is
 impossible.  We conclude that $U_1=\emptyset$, and a similar argument
 gives $U_2=\emptyset$.  Now consider $x\in U_3$.  If we had
 $x_1\neq 0$ then $x/x_1$ would be in $U_1$, which is impossible.
 Thus $x_1=0$, and similarly $x_2=0$, and $x_3=1$ by assumption.
 Substituting this into $n_3=n_4=0$ gives a contradiction.  A similar
 argument works for $U_4$.
 \begin{checks}
  embedded/EX_check.mpl: check_PEX_smoothness()
 \end{checks}
\end{proof}

\begin{proposition}
 None of the surfaces $CEX(1/\rt)$, $PEX(1/\rt)$ or $PEX'(1/\rt)$ is
 smooth.  Moreover, $PEX(1/\rt)$ is isomorphic to the singular Cayley
 cubic with equation
 \[ X_1X_2X_3 + X_1X_2X_4 + X_1X_3X_4 + X_2X_3X_4 = 0. \]
\end{proposition}
\begin{remark}
 Although the isomorphism with the Cayley cubic is interesting, it
 interacts poorly with the underlying real structure, and so is not
 too helpful for studying the cromulent surface $EX(1/\rt)$.
\end{remark}
\begin{proof}
 At the point $w=(\sqrt{-2},0,\rt,1)$ we find that $\rho(w)=1$ and
 $g(w)=0$ and all partial derivatives of $g$ also vanish.  It follows
 that $w$ is a singular point in $CEX(1/\rt)$ and that $[w]$ is a
 singular point in $PEX'(1/\rt)\subset PEX(1/\rt)$.  The same holds
 for all points in the $G$-orbit of $w$.  (One can check that
 $\nu(w)=w$ and $\lm^2\mu\nu(w)=-w$ and $|G.w|=8$ and $|G.[w]|=4$.)

 Now put
 \begin{align*}
  X_1 &= x_4 + x_3/\rt + x_1\sqrt{-2} &
  X_2 &= x_4 + x_3/\rt - x_1\sqrt{-2} \\
  X_3 &= x_4 - x_3/\rt + x_2\sqrt{-2} &
  X_4 &= x_4 - x_3/\rt - x_2\sqrt{-2}.
 \end{align*}
 We find that
 \[ X_1X_2X_3 + X_1X_2X_4 + X_1X_3X_4 + X_2X_3X_4 = -2g(x), \]
 so this gives an isomorphism with the Cayley cubic.
 \begin{checks}
  embedded/roothalf/cayley_surface_check.mpl: check_cayley_surface();
 \end{checks}
\end{proof}

\subsection{The ring of functions}
\lbl{sec-E-functions}

\begin{definition}\lbl{defn-yzu}
 We put
 \begin{align*}
  y_1 &= x_3   & y_2 &= (x_2^2 - x_1^2 - (a^{-1}+a)x_3x_4)/(2a) \\
  z_1 &= y_1^2 & z_2 &= y_2^2
 \end{align*}
 and
 \begin{align*}
  u_1 &= (1-2ay_2)/2 - \half(y_2-a)(y_2-a^{-1})y_1^2 \\
  u_2 &= (1+2ay_2)/2 - \half(y_2+a)(y_2+a^{-1})y_1^2 \\
  u_3 &= 4u_1u_2
       = (1-z_1-z_1z_2)^2 - z_2((a+a^{-1})z_1-2a)^2 \\
  u_4 &= u_1+u_2 = 1 - z_1 - z_1z_2
 \end{align*}
\end{definition}
\begin{checks}
 embedded/invariants_check.mpl: check_invariants()
\end{checks}

\begin{proposition}\lbl{prop-OX-basis}
 The ring $A$ of polynomial functions on $EX(a)$ can be described as
 \[ A = \R[y_1,y_2][x_1,x_2]/(x_1^2-u_1,x_2^2-u_2), \]
 with $x_3=y_1$ and $x_4=-y_1y_2$.  The set
 \[ M = \{x_1^ix_2^jy_1^ky_2^l\st i,j,k,l\in\{0,1\}\} \]
 is a basis for $A$ over the subring $\R[z_1,z_2]$.
\end{proposition}
\begin{proof}
 We have $x_3=y_1$ by definition, and the relation $g_0(x)=0$ can
 easily be rearranged as $x_4=-y_1y_2$.  We also have
 \begin{align*}
  x_1^2+x_2^2 &= 1-x_3^2-x_4^2 = 1-y_1^2-y_1^2y_2^2 \\
  x_1^2-x_2^2 &= -2ay_2-(a+a^{-1})x_3x_4
               = -2ay_2+(a+a^{-1})y_1^2y_2.
 \end{align*}
 Adding these equations gives $x_1^2=u_1$, and subtracting them gives
 $x_2^2=u_2$.  We now see that $A=\R[y_1,y_2]/(x_i^2-u_i)$.  It is
 clear from this that $\{1,x_1,x_2,x_1x_2\}$ is a basis for $A$ over
 $\R[y_1,y_2]$, and $\{1,y_1,y_2,y_1y_2\}$ is a basis for
 $\R[y_1,y_2]$ over $\R[z_1,z_2]$, so $M$ is a basis for $A$ over
 $\R[z_1,z_2]$.
\end{proof}
\begin{remark}
 The global symbols \mcode+x+, \mcode+y+, \mcode+z+ (and various
 others) are protected (by a call to the \mcode+protect+ command in
 the file \fname+Rn.mpl+).  This prevents the user from assigning
 values to these symbols, which is necessary in for our use of
 Gr\"obner bases to work properly.  Variables such as \mcode+x0+ or
 \mcode+X+ can be used instead of \mcode+x+ for storing points in
 $\R^4$.  The variable \mcode+xx+ is set equal to
 \mcode+[x[1],x[2],x[3],x[4]]+ (a list of length $4$, whose entries
 are unassigned symbols).  Similarly, \mcode+yy+ is
 \mcode+[y[1],y[2]]+, and \mcode+zz+ is \mcode+[z[1],z[2]]+.  Note
 although $z_1=x_3^2$, the symbol \mcode+z[1]+ does not have the value
 \mcode+x[3]^2+.  Instead, there is another variable \mcode+zx[1]+
 with the value \mcode+x[3]^2+.  If we have an expression \mcode+m+
 involving \mcode+z[1]+ and \mcode+z[2]+, we can convert it to an
 expression in \mcode+x[1]+ to \mcode+x[4]+ using the syntax
 \mcode+subs({z[1]=zx[1],z[2]=zx[2]},m)+ or
 \mcode+eval(subs({z=zx},m))+.  Some esoteric features of Maple mean
 that the second form will not work correctly without \mcode+eval()+.
 Similarly, there are variables \mcode+yx[i]+ which contain
 expressions for $y_i$ in terms of $x_j$, and variables \mcode+uy[1]+,
 \mcode+uy[2]+, \mcode+uz[3]+ and \mcode+uz[4]+ which contain
 expressions for $u_i$ in terms of $y_j$ or $z_j$.  We can also regard
 the rule $x\mapsto y$ as giving a function $\R^4\to\R^2$.  Maple
 notation for this function is \mcode+y_proj(x)+, and
 \mcode+z_proj(x)+ is similar.
\end{remark}
\begin{remark}
 The functions \mcode+NF_x+, \mcode+NF_y+ and \mcode+NF_z+ can be used
 to simplify elements of the ring $A$, by reducing them modulo a
 suitable Gr\"obner basis.  The function \mcode+NF_x+ will convert any
 expression to one that involves only the variables $x_i$, whereas
 \mcode+NF_y+ converts $x$'s to $y$'s as far as possible, and
 similarly for \mcode+NF_z+.
\end{remark}

It is straightforward to check that $G$ acts on $A$ as follows:
\begin{align*}
 \lm^*(x_1) &=    -x_2 & \mu^*(x_1) &= \pp x_1 & \nu^*(x_1) &= \pp x_1 \\
 \lm^*(x_2) &= \pp x_1 & \mu^*(x_2) &=    -x_2 & \nu^*(x_2) &=    -x_2 \\
 \lm^*(x_3) &= \pp x_3 & \mu^*(x_3) &=    -x_3 & \nu^*(x_3) &= \pp x_3 \\
 \lm^*(x_4) &=    -x_4 & \mu^*(x_4) &=    -x_4 & \nu^*(x_4) &= \pp x_4 \\
 \lm^*(y_1) &= \pp y_1 & \mu^*(y_1) &=    -y_1 & \nu^*(y_1) &= \pp y_1 \\
 \lm^*(y_2) &=    -y_2 & \mu^*(y_2) &= \pp y_2 & \nu^*(y_2) &= \pp y_2 \\
 \lm^*(z_1) &= \pp z_1 & \mu^*(z_1) &= \pp z_1 & \nu^*(z_1) &= \pp z_1 \\
 \lm^*(z_2) &= \pp z_2 & \mu^*(z_2) &= \pp z_2 & \nu^*(z_2) &= \pp z_2.
\end{align*}
In particular, the group acts as the identity on $\R[z_1,z_2]$, and
permutes the set $M\cup(-M)$.  Some of this can also be expressed in
terms of the characters listed in Proposition~\ref{prop-characters}:
\begin{align*}
 \gm^*(x_3) &= \chi_2(\gm) x_3 &
 \gm^*(x_4) &= \chi_3(\gm) x_4 \\
 \gm^*(y_1) &= \chi_2(\gm) y_1 &
 \gm^*(y_2) &= \chi_1(\gm) y_2 \\
 \gm^*(x_1x_2) &= \chi_4(\gm)x_1x_2.
\end{align*}

This makes it easy to analyse the invariants for various subgroups of
$G$.  The most important cases are as follows:
\begin{proposition}\lbl{prop-invariants}
 $A^G=\R[z_1,z_2]$ and $A^{\ip{\lm^2,\nu}}=\R[y_1,y_2]$.
\end{proposition}
\begin{proof}
 Any element $a\in A$ can be written uniquely as
 $a=a_0+a_1x_1+a_2x_2+a_3x_1x_2$ with $a_0,\dotsc,a_3\in\R[y_1,y_2]$.
 We then find that
 \begin{align*}
  (\lm^2)^*(a) &= a_0 - a_1x_1 - a_2x_2 + a_3x_1x_2 \\
  \nu^*(a)     &= a_0 + a_1x_1 - a_2x_2 - a_3x_1x_2,
 \end{align*}
 so $a$ is invariant under $\ip{\lm^2,\nu}$ if and only if
 $a_1=a_2=a_3=0$ and $a=a_0\in\R[y_1,y_2]$.  If this holds, we can
 write $a$ uniquely as $b_0+b_1y_1+b_2y_2+b_3y_1y_2$ with
 $b_0,\dotsc,b_3\in\R[z_1,z_2]$.  We then find that
 \begin{align*}
  \lm^*(a) &= b_0 + b_1y_1 - b_2y_2 - b_3y_1y_2 \\
  \mu^*(a) &= b_0 - b_1y_1 + b_2y_2 - b_3y_1y_2,
 \end{align*}
 so $a$ is invariant under all of $G$ if and only if $b_1=b_2=b_3=0$
 and $a=b_0\in\R[z_1,z_2]$.
\end{proof}

\subsection{The curve system}
\lbl{sec-E-curves}

In this section we will construct a curve system for $EX(a)$.

\begin{definition}\lbl{defn-slices}
 We put $X_k=\{x\in EX(a)\st x_k=0\}$.
\end{definition}

\begin{proposition}\lbl{prop-slices-a}
 The fixed points of antiholomorphic elements of $G$ are as follows:
 \begin{align*}
  EX(a)^{\mu\nu}   &= \{x\in EX(a)\st x_3=x_4=0\} = X_3\cap X_4 \\
  EX(a)^{\lm\nu}   &= \{x\in EX(a)\st x_1=x_2,\;x_4=0\} \sse X_4 \\
  EX(a)^{\lm^3\nu} &= \{x\in EX(a)\st x_1=-x_2,\;x_4=0\} \sse X_4 \\
  EX(a)^{\lm^2\nu} &= \{x\in EX(a)\st x_1=0\} = X_1 \\
  EX(a)^\nu        &= \{x\in EX(a)\st x_2=0\} = X_2.
 \end{align*}
\end{proposition}
\begin{proof}
 Immediate from formulae for the action of the relevant group elements
 on $\R^4$.
\end{proof}

\begin{lemma}\lbl{lem-a-order}
 Put $a^*=\sqrt{(a^{-2}-1)/2}$, so $2at^2+a-a^{-1}=0$ iff $t=\pm a^*$.
 Then:
 \begin{itemize}
  \item If $0<a<1/\rt$ we have
   \[ -a^* < 0 < a < \frac{1}{2a} < a^* < \frac{1}{a}.  \]
  \item If $a=1/\rt$ we have
   \[ -a^* < 0 < a = \frac{1}{2a} = a^* < \frac{1}{a}. \]
  \item If $1/\rt<a<1$ we have
   \[ -a^* < 0 < a^* < \frac{1}{2a} < a < \frac{1}{a}. \]
 \end{itemize}
\end{lemma}
\begin{proof}
 Straightforward.
\end{proof}

\begin{definition}\lbl{defn-T-alg}
 For all $a\in(0,1)$ we put
 \begin{align*}
  T_{\alg}^- &= [-a^*,0] \\
  T_{\alg}^+ &= [\min(1/(2a),a^*),\max(1/(2a),a^*)]
    = \begin{cases}
       [1/(2a),a^*] & \text{ if } 0 < a \leq 1/\rt \\
       [a^*,1/(2a)] & \text{ if } 1/\rt \leq a < 1.
      \end{cases} \\
  T_{\alg} &= T_{\alg}^- \cup T_{\alg}^+
 \end{align*}
 Note that $1/a$ is never in $T_{\alg}$, and $a$ is in $T_{\alg}$ if
 and only if $a=1/\rt$.  Note also that $T_{\alg}^+$ and
 $T_{\alg}^-$ are nonempty disjoint closed sets.
\end{definition}

\begin{proposition}\lbl{prop-T-alg}
 Consider a point $x\in X_1$ (so $x_1=0$), and define
 $y_1=x_3$ and $y_2=(x_2^2-(a+a^{-1})x_3x_4)/(2a)$ as in
 Definition~\ref{defn-yzu}.  Then
 \begin{align*}
  y_1^2(y_2-a)(y_2-a^{-1}) &= -2a(y_2-1/(2a)) \tag{A} \\
  x_2^2(y_2-a)(y_2-a^{-1}) &= 2ay_2(y_2^2-(a^*)^2) \tag{B} \\
  \left(y_1^2-\frac{2}{a^{-2}+1}\right)(y_2-a)(y_2-a^{-1})
   &= \frac{2}{a^{-2}+1}((a^*)^2-y_2^2) \tag{C}
 \end{align*}
 Thus $y_2\in T_{\alg}^+\amalg T_{\alg}^-$, and if
 $y_2\in T_{\alg}^-$ then
 \[ y_1 \in \left[-1,-\sqrt{\frac{2}{a^{-2}+1}}\right] \amalg
            \left[\sqrt{\frac{2}{a^{-2}+1}},1\right].
 \]
\end{proposition}
\begin{proof}
 Recall from Proposition~\ref{prop-OX-basis} that $x_1^2=u_1$, where
 \[ u_1 = (1-2ay_2)/2 - \half(y_2-a)(y_2-a^{-1})y_1^2. \]
 Here $x_1=0$ so $u_1=0$, and this can be rearranged to give
 equation~(A).  Next, as $x_1=0$ and $x_3=y_1$ and $x_4=-y_1y_2$
 we have $x_2^2=1-x_3^2-x_4^2=1-(1+y_2^2)y_1^2$.  We can thus take the
 equation
 \[ (y_2-a)(y_2-a^{-1})=y_2^2-(a+a^{-1})y_2+1 \tag{D} \]
 and subtract $(1+y_2^2)$ times equation~(A) to get equation~(B).
 Alternatively, we can subtract $2/(a^{-2}+1)$ times~(D) from~(A) and
 rearrange slightly to get~(C).  Now consider the signs of the left
 and right hand sides of~(A), bearing in mind Lemma~\ref{lem-a-order};
 it follows that we must have $y_2\in T_{\alg}$.  Suppose in fact that
 $y_2\in T_{\alg}^-$, so $-a^*\leq y_2\leq 0$.  On the left hand side
 of~(C) the factors $y_2-a^{\pm 1}$ are both strictly negative, and
 on the right hand side $(a^*)^2-y_2^2\geq 0$.  It therefore follows
 from~(D) that $y_1^2\geq 2/(a^{-2}-1)$.  On the other hand, we also
 have $y_1^2=x_3^2=1-x_2^2-x_4^2\leq 1$.  The relation
 \[ y_1 \in \left[-1,-\sqrt{\frac{2}{a^{-2}+1}}\right] \amalg
            \left[\sqrt{\frac{2}{a^{-2}+1}},1\right]
 \]
 is now clear.
\end{proof}
\begin{remark}\lbl{rem-T-alg}
 We have $\lm(X_1)=\lm^{-1}(X_1)=X_2$ and $\lm^*(y_1)=y_1$ and
 $\lm^*(y_2)=-y_2$.  Using this we deduce that when $x\in X_2$ we have
 $-y_2\in T_{\alg}$, and if $-y_2\in T_{\alg}^-$ we again have
 \[ y_1 \in \left[-1,-\sqrt{\frac{2}{a^{-2}+1}}\right] \amalg
            \left[\sqrt{\frac{2}{a^{-2}+1}},1\right].
 \]
\end{remark}

\begin{definition}\lbl{defn-C-star}
 We put
 \begin{align*}
  C_0^* &= \{(\cos(t),\;\sin(t),\;0,\;0)\st t\in\R\} \\
  C_1^* &= \{(\pp\sin(t)/\rt,\;\sin(t)/\rt,\;\cos(t),\;0)\st t\in\R\} \\
  C_2^* &= \{(  -\sin(t)/\rt,\;\sin(t)/\rt,\;\cos(t),\;0)\st t\in\R\} \\
  C_3^* &= \{x\in X_1\;|\hspace{0.75em} \pp y_2\in T_{\alg}^+\} \\
  C_4^* &= \{x\in X_2\;|\hspace{0.30em}    -y_2\in T_{\alg}^+\} \\
  C_5^* &= \{x\in X_2\;|\hspace{0.30em}    -y_2\in T_{\alg}^-,\; y_1\geq 0\} \\
  C_6^* &= \{x\in X_1\;|\hspace{0.75em} \pp y_2\in T_{\alg}^-,\; y_1\geq 0\} \\
  C_7^* &= \{x\in X_2\;|\hspace{0.30em}    -y_2\in T_{\alg}^-,\; y_1\leq 0\} \\
  C_8^* &= \{x\in X_1\;|\hspace{0.75em} \pp y_2\in T_{\alg}^-,\; y_1\leq 0\}.
 \end{align*}
\end{definition}
It is straightforward to check that $C_k=C_k^*$ for $k\in\{0,1,2\}$.
Using Proposition~\ref{prop-T-alg} and Remark~\ref{rem-T-alg} we see
that
\begin{align*}
 X^{\lm^2\nu} &= X_1 = C_3^* \amalg C_6^* \amalg C_8^* \\
 X^\nu        &= X_2 = C_4^* \amalg C_5^* \amalg C_7^*.
\end{align*}
Next, recall that $C_3$ is the component of $X^{\lm^2\nu}=X_1$
containing $v_{11}$.  One can check that $y_2=a^*$ at $v_{11}$, so
$v_{11}\in C_3^*$.  By connectivity, it follows that
$C_3\sse C_3^*$.  The same line of argument shows that
$C_k\sse C_k^*$ for all $k\in\{3,\dotsc,8\}$, and $C_k$ will be the
same as $C_k^*$ if $C_k^*$ is connected.  To prove that this holds
we need to parameterise $C_k^*$.  We will do this in two different
ways.

\begin{definition}\lbl{defn-c-alg}
 We put
 \[ c_{\alg}(t) =
   \left(0,
         \sqrt{\frac{(2at^2+a-a^{-1})t}{(t-a)(t-a^{-1})}},
         \sqrt{\frac{1-2at}{(t-a)(t-a^{-1})}},
         -t\sqrt{\frac{1-2at}{(t-a)(t-a^{-1})}}
   \right).
 \]
 One can check using Lemma~\ref{lem-a-order} that this defines a
 continuous map $c_{\alg}\:T_{\alg}\to\R^4$ except in
 the case $a=1/\rt$, when the domain of $c_{\alg}$ is
 $T_{\alg}\sm\{a\}=T_{\alg}^-$.
\end{definition}
The map $c_{\alg}(t)$ is defined in \fname+embedded/extra_curves.mpl+
as \mcode+c_algebraic(t)+.

We will show that the map $c_{\alg}$ is essentially inverse to
$y_2\:X_1\to T_{\alg}$.  The formula is forced by the identities in
Proposition~\ref{prop-T-alg}.  A precise statement of the key property
is as follows.

\begin{proposition}\lbl{prop-c-alg}
 The image of $c_{\alg}$ is contained in $X_1$, and we have
 $y_2(c_{\alg}(t))=t$.  Conversely, consider a point $x\in X_1$, so
 $y_2\in T_{\alg}$.  Except in the case where $y_2=a=1/\rt$,
 there is an element $\gm\in\{1,\mu,\nu,\mu\nu\}$ such that
 $x=\gm(c_{\alg}(y_2))$.  If $x_2,x_3\geq 0$ then we can take $\gm=1$.
\end{proposition}
\begin{proof}
 First, it is straightforward to check that $\rho(c_{\alg}(t))=1$ and
 $g_0(c_{\alg}(t))=0$ so the image of $c_{\alg}$ is contained in
 $X_1$.  When $x_3\neq 0$ we have $y_2=-x_4/x_3$, and using this we
 see that $y_2(c_{\alg}(t))=t$ except possibly when $t=1/(2a)$, but
 that case can be recovered by continuity.

 For the converse, consider a point $x\in X_1$, and suppose we are not
 in the exceptional case where $y_2=a=1/\rt$, so
 $(y_2-a)(y_2-a^{-1})\neq 0$ and the identities
 in Proposition~\ref{prop-T-alg} can be rearranged to give formulae
 for $x_2^2$ and $x_3^2$ in terms of $y_2$.  Recall also
 that $x_4=-x_3y_2$.  Using this, we see that the point
 $x'=c_{\alg}(y_2)$ satisfies $x_2=\pm x'_2$ and $x_3=\pm x'_3$.
 Recall also that
 \begin{align*}
  \mu(0,x_2,x_3,x_4) &= (0,   -x_2,   -x_3,   -x_4) \\
  \nu(0,x_2,x_3,x_4) &= (0,   -x_2,\pp x_3,\pp x_4),
 \end{align*}
 and that $y_2$ is invariant under $\mu$ and $\nu$.  The claim now
 follows easily.
\end{proof}

It would be possible to prove $C_k=C_k^*$ using only $c_{\alg}$, with
a slight digression to cover the case $a=1/\rt$, but we prefer to
use a full curve system instead.

\begin{definition}\lbl{defn-E-curves}
 We define maps $c_k\:\R\to\R^4$ as follows.
 \begin{align*}
  c_0(t) &= (\cos(t),\sin(t),0,0) \\
  c_1(t) &= (\sin(t)/\rt,\sin(t)/\rt,\cos(t),0) \\
  c_2(t) &= \lm(c_1(t)) \\
  p_3(t) &= (1+a^2)\sin(t)^2 +
            \sqrt{(1+a^2)(1-a^2+2a^2\sin(t)^2)+(1-a^2)^2\cos(t)^4} \\
  c_3(t) &= \left(0,
             \sqrt{\frac{2(1-a^2)+4a^2\sin(t)^2}{p_3(t)}}\sin(t),\right.\\
         &\qquad \left.\sqrt{\frac{2}{1+a^{-2}}}\cos(t),
             \sqrt{\frac{2}{1+a^{-2}}}\frac{a^{-1}-a+2a\sin(t)^2}{p_3(t)}\cos(t)\right)\\
  c_4(t) &= \lm(c_3(t)) \\
  \tau_5(t) &= -\sqrt{(a^{-2}-1)/2}\sin(t/2)^2 \\
  p_5(t) &= (\tau_5(t)-a)(\tau_5(t)-a^{-1}) \\
  c_5(t) &= \left(
         \frac{(a^{-2}-1)^{3/4}}{2^{5/4}}\sqrt{\frac{a(1+\sin(t/2)^2)}{p_5(t)}}\sin(t),
         0, \sqrt{\frac{1-2a\tau_5(t)}{p_5(t)}},
         -\tau_5(t)\sqrt{\frac{1-2a\tau_5(t)}{p_5(t)}}
           \right) \\
  c_6(t) &= \lm(c_5(t)); \quad
  c_7(t)  = \mu(c_5(t)); \quad
  c_8(t)  = \lm(\mu(c_5(t))).
 \end{align*}
 Maple notation for these is \mcode+c_E[k](t)+.
\end{definition}
\begin{remark}
 We have written $p_3(t)$ in a form which makes it clear that it is
 always strictly positive, and using this it is not hard to see
 that $c_3(t)$ is well-defined.  It is also clear that
 $\tau_5(t)\leq 0$ and so $p_5(t)\geq 1$, which in turn implies that
 $c_5(t)$ is well-defined.  The map $c_5$ is essentially
 $c_{\alg}\circ\tau_5$ except that $c_{\alg}$ involves nonnegative
 square roots of certain quantities, whereas $c_5$ uses different
 branches of these roots that are sometimes negative.  A similar
 approach would be possible for $c_3$ but would run into problems as
 $a$ passes through $1/\rt$.
\end{remark}

\begin{proposition}\lbl{prop-E-curves}
 The images of the above maps lie in $EX(a)$, and they give a curve
 system for $EX(a)$.  Moreover, we have $c_k(\R)=C_k=C_k^*$ for all
 $k$.
\end{proposition}

\begin{proof}
 We first need to check that $c_k(\R)\sse EX(a)$, or equivalently
 $\rho(c_k(t))=1$ and $g_0(c_k(t))=0$.  This is easy for
 $k\in\{0,1,2\}$.  The algebra is harder for $k\in\{3,5\}$ but there
 is no conceptual difficulty, it is just a lengthy exercise in
 trigonometric simplification.  The remaining cases $k\in\{4,6,7,8\}$
 follow from the cases $k\in\{3,5\}$ using the group action.

 Next, it is elementary to check all the group transformation
 equations in axiom~(c) of Definition~\ref{defn-curve-system}, and to
 check all the identities $c_i(\tht)=v_j$ corresponding to the
 nonempty boxes in axiom~(b).

 We now address axiom~(a).  The cases $k\in\{0,1,2\}$ are again easy,
 and the cases $k\in\{4,6,7,8\}$ will follow from $k\in\{3,5\}$ by
 symmetry.  Suppose that $c_3(t)=c_3(u)$.  By looking at the third
 component, we see that $\cos(t)=\cos(u)$, and thus that
 $\sin(t)=\pm\sin(u)$, and thus that $p_3(t)=p_3(u)$.  After recalling
 that $p_3>0$ and inspecting the second component we deduce that
 $\sin(t)=\sin(u)$, so $t-u\in 2\pi\Z$.  Similarly, if $c'_3(t)=0$
 then by looking at the third component we see that $\sin(t)=0$, so
 $t=n\pi$ for some $n\in\Z$.  This means that the functions $\cos$ and
 $\sin^2$ are both constant to first order near $t$, so the same is
 true of $p_3$.  By inspecting the second component of $c_3$ we deduce
 that $\sin$ must also be constant to first order, so $\cos(t)=0$,
 which is impossible.  This proves all claims for $c_3$.

 For $c_5$, we first observe that the map $y_2\:EX(a)\to\R$ is
 invariant under $\mu$ and $\nu$ and satisfies
 $y_2(c_5(t))=-\tau_5(t)$.  Thus, if $c_5(t)=c_5(u)$ then
 $\tau_5(t)=\tau_5(u)$, which easily gives $\cos(t)=\cos(u)$.  Given
 this, the rest of the argument is essentially the same as for $c_3$.
 This completes the proof of axiom~(a).

 Now note that $c_3(\R)$ is contained in $X^{\lm^2\nu}=X_1$ and is
 connected and contains $c_3(0)=v_{11}$, but $C_3$ is defined to be
 the component of $v_{11}$ in $X^{\lm^2\nu}$, so
 $c_3(\R)\sse C_3\sse C_3^*$.  Moreover, $y_2(c_3(\R))$ is a connected
 subset of $T_{\alg}^+$ containing both of the endpoints
 $a^*=y_2(c_3(0))$ and $1/(2a)=y_2(c_3(\pi/2))$, so
 $y_2(c_3(\R))=T_{\alg}^+$.  Thus, if $x\in C_3^*$ then there exists
 $t\in\R$ with $y_2(c_3(t))=y_2(x)$, and it follows that
 $x=\gm(c_3(t))$ for some $\gm\in\{1,\mu,\nu,\mu\nu\}$.  As
 $\mu(c_3(t))=c_3(t+\pi)$ and $\nu(c_3(t))=c_3(-t)$ we deduce that
 $x\in c_3(\R)$.  In conclusion, we have $c_3(\R)=C_3=C_3^*$.  Similar
 arguments give $c_k(\R)=C_k=C_k^*$ for all $k\in\{3,\dotsc,8\}$.
 Recall also that the sets $C_3^*$, $C_6^*$ and $C_8^*$ are disjoint
 (immediately from the definitions).
 Proposition~\ref{prop-empty-boxes} therefore guarantees that we have
 a curve system.
\end{proof}

\begin{proposition}\lbl{prop-slices}
 $X_3=C_0$ and $X_4=C_0\cup C_1\cup C_2$.
\end{proposition}
\begin{proof}
 First, it is straightforward to check that $X_3\supseteq C_0$ and
 $X_4\supseteq C_0\cup C_1\cup C_2$.  For the converse, suppose that
 $x\in X_3$.  Then the relation $x_4=-y_1y_2=-x_3y_2$ shows that
 $x_4$ vanishes as well as $x_3$, and we also have $\|x\|=1$ so
 $x\in C_0$.  Suppose instead that $x\in X_4\sm X_3$.  As
 $x_4=0\neq x_3$ the relation $g_0(x)=0$ becomes
 $2(x_1-x_2)(x_1+x_2)x_3=0$ and we can divide by $x_3$ to get
 $x_1=\pm x_2$.  We also have $\|x\|=1$ and it follows easily that
 $x\in C_1\cup C_2$.
\end{proof}

\begin{remark}\lbl{rem-curves-roothalf}
 The formulae for $c_3$ and $c_5$ simplify significantly in the case $a=1/\rt$:
 \begin{align*}
  c_3(t) &= \left(0,\sin(t),\sqrt{2/3}\cos(t),-\sqrt{1/3}\cos(t)\right) \\
  c_5(t) &= \left(-\sin(t),0,2^{3/2},\cos(t)-1\right)/\sqrt{10-2\cos(t)}.
 \end{align*}
 The following picture shows all the curves $c_k(t)$ in that case.
 \[ \includegraphics[scale=0.25,clip=true,
      trim=6cm 6cm 6cm 6cm]{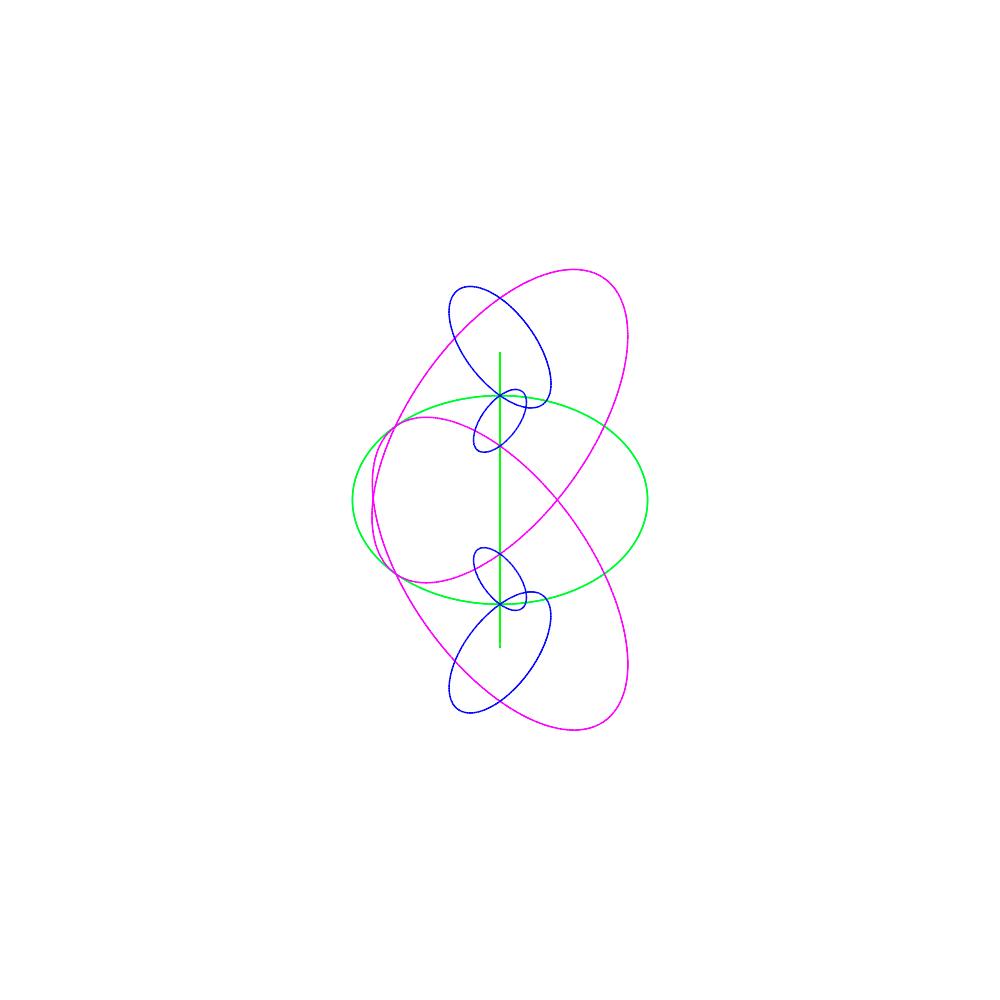}
 \]
\end{remark}

\subsection{Fundamental domains}
\lbl{sec-E-fundamental}

In this section we define and study retractive fundamental domains in
$EX(a)$ for certain subgroups of $G$.  It is convenient to start with
an easy case:
\begin{definition}\lbl{defn-F-two}
 We put $H_2=\{1,\lm^2\nu\}$ and recall that $C_2=\{1,\lm^2\}$, noting
 also that
 \begin{align*}
  \lm^2(x)    &= (-x_1,\;-x_2,\;x_3,\;x_4) \\
  \lm^2\nu(x) &= (-x_1,\;\pp x_2,\;x_3,\;x_4).
 \end{align*}
 We then put $F_2=\{x\in EX(a)\st x_1\geq 0\}$, and define
 $r_2\:EX(a)\to F_2$ by
 \[ r_2(x) = (|x_1|,\;x_2,\;x_3,\;x_4). \]
\end{definition}
\begin{proposition}\lbl{prop-F-two}
 $F_2$ is a retractive fundamental domain for $H_2$, with retraction
 $r_2$.  Moreover, $F_2$ is also a non-retractive fundamental domain
 for $C_2$.
\end{proposition}
\begin{proof}
 Clear.
\end{proof}

We will see later that $F_2$ is homeomorphic to a disc with two holes.

\begin{definition}\lbl{defn-F-four}
 We put $H_4=\{1,\lm^2,\nu,\lm^2\nu\}$, recalling that
 \begin{align*}
  \lm^2(x)    &= (-x_1,\;-x_2,\;x_3,\;x_4) \\
  \nu(x)      &= (\pp x_1,\;-x_2,\;x_3,\;x_4) \\
  \lm^2\nu(x) &= (-x_1,\;\pp x_2,\;x_3,\;x_4).
 \end{align*}
 We then put $F_4=\{x\in EX(a)\st x_1,x_2\geq 0\}$, and define
 $r_4\:EX(a)\to F_4$ by
 \[ r_4(x) = (|x_1|,\;|x_2|,\;x_3,x_4). \]
 We also put
 \[ F_4^* = \{y\in\R^2\st u_1,u_2\geq 0\}, \]
 where $u_1$ and $u_2$ are defined in terms of $y_1$ and $y_2$ as in
 Definition~\ref{defn-yzu}:
 \begin{align*}
  u_1 &= (1-2ay_2)/2 - \half(y_2-a)(y_2-a^{-1})y_1^2 \\
  u_2 &= (1+2ay_2)/2 - \half(y_2+a)(y_2+a^{-1})y_1^2.
 \end{align*}
\end{definition}

\begin{proposition}\lbl{prop-F-four}
 $F_4$ is a retractive fundamental domain for $H_4$, with retraction
 $r_4$.  Moreover, there is a map $p_4\:EX(a)\to F_4^*$ given by
 \[ p_4(x) = (x_3,(x_2^2 - x_1^2 - (a^{-1}+a)x_3x_4)/(2a))=(y_1,y_2), \]
 and a map $s_4\:F_4^*\to F_4$ given by
 \[ s_4(y) = (\sqrt{u_1},\sqrt{u_2},y_1,-y_1y_2), \]
 and these satisfy $p_4s_4=1$ and $s_4p_4=r_4$.  Thus, $p_4$ restricts
 to give a homeomorphism $F_4\to F_4^*$ with inverse $s_4$.
\end{proposition}
Maple notation for $p_4(x)$ and $s_4(t)$ is \mcode+y_proj(x)+ and
\mcode+y_lift(t)+.
\begin{proof}
 It is clear that $F_4$ is a retractive fundamental domain for $H_4$, with
 retraction $r_4$.  Recall that the ring of functions on $EX(a)$ is
 generated by $y_1$, $y_2$, $x_1$ and $x_2$, with $x_i^2=u_i$ for
 $i=1,2$ and $x_3=y_1$ and $x_4=-y_1y_2$.  It follows that $u_1$ and
 $u_2$ are nonnegative as functions on $EX(a)$, or equivalently that
 $p_4(EX(a))\sse F_4^*$.  It also follows that the stated formula gives a
 well-defined map $s_4\:F_4^*\to F_4$, and it is straightforward to
 check that $p_4s_4=1$ and $s_4p_4=r_4$.
 \begin{checks}
  embedded/EX_check.mpl: check_E_F4()
 \end{checks}
\end{proof}

The following picture shows $F_4$ together with the curves
$c_3,\dotsc,c_8$.
\[ \includegraphics[scale=0.35,clip=true,
     trim=6cm 6cm 6cm 7cm]{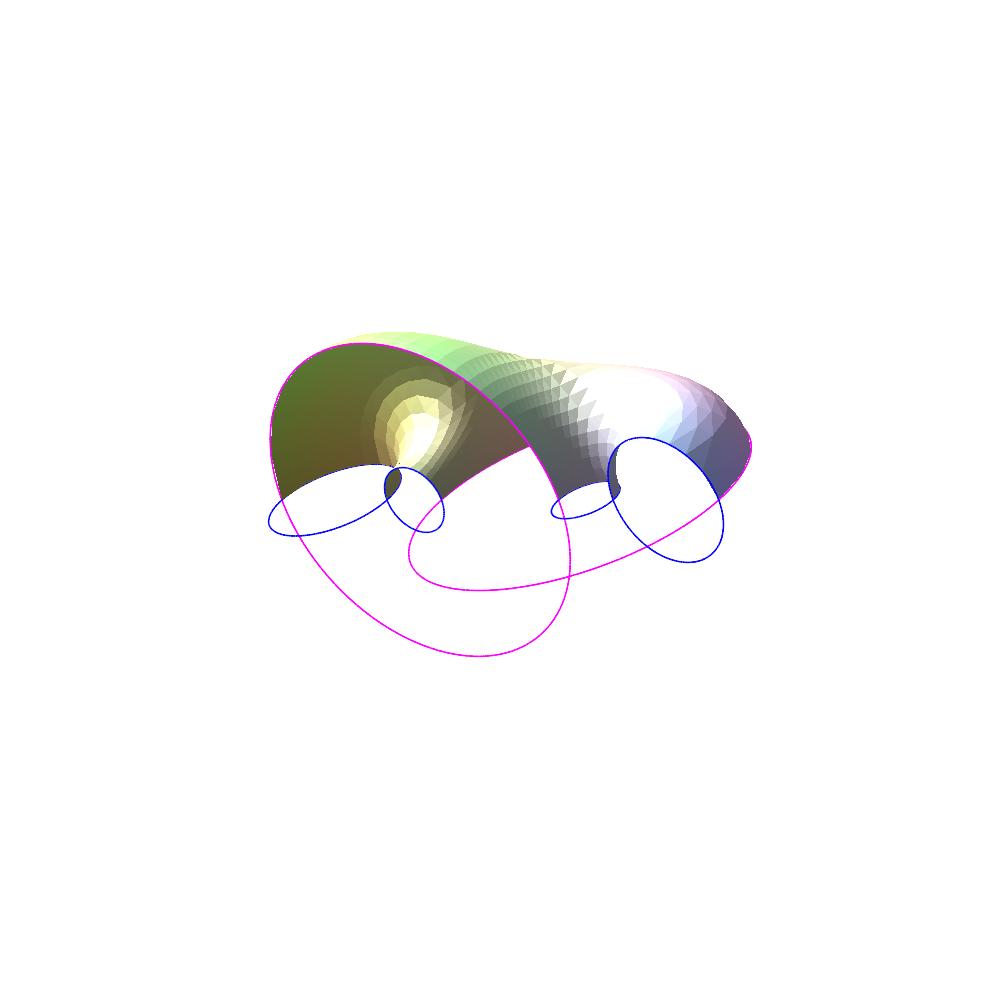}
\]

Formulae for the action of $p_4$ on some of the curves $c_i$ are as follows:
\begin{align*}
 p_4(c_{\alg}(t)) &=
  \left(\sqrt{\frac{1-2at}{(t-a)(t-a^{-1})}},t\right) \\
 p_4(c_0(t)) &= (0,-\cos(2t)/(2a)) \\
 p_4(c_1(t)) = p_4(c_2(t)) &= (\cos(t),0).
\end{align*}
Formulae for the remaining curves are not illuminating.

The following picture shows the set $F_4^*$ for three different values
of the parameter $a$, together with the images of the points $v_i$
under the map $p_4$.
\begin{center}
 \begin{tikzpicture}[scale=1.5]
  \begin{scope}
 \draw[cyan] plot[smooth] coordinates{ (0.000,-1.000) (0.000,-0.309) (0.000,0.809) (0.000,0.809) (0.000,-0.309) (0.000,-1.000) (0.000,-0.309) (0.000,0.809) (0.000,0.809) (0.000,-0.309) (0.000,-1.000) };
 \draw[green] plot[smooth] coordinates{ (1.000,0.000) (0.809,0.000) (0.309,0.000) (-0.309,0.000) (-0.809,0.000) (-1.000,0.000) (-0.809,0.000) (-0.309,0.000) (0.309,0.000) (0.809,0.000) (1.000,0.000) };
 \draw[green] plot[smooth] coordinates{ (1.000,0.000) (0.809,0.000) (0.309,0.000) (-0.309,0.000) (-0.809,0.000) (-1.000,0.000) (-0.809,0.000) (-0.309,0.000) (0.309,0.000) (0.809,0.000) (1.000,0.000) };
 \draw[magenta] plot[smooth] coordinates{ (0.632,1.225) (0.512,1.144) (0.195,1.019) (-0.195,1.019) (-0.512,1.144) (-0.632,1.225) (-0.512,1.144) (-0.195,1.019) (0.195,1.019) (0.512,1.144) (0.632,1.225) };
 \draw[magenta] plot[smooth] coordinates{ (0.632,-1.225) (0.512,-1.144) (0.195,-1.019) (-0.195,-1.019) (-0.512,-1.144) (-0.632,-1.225) (-0.512,-1.144) (-0.195,-1.019) (0.195,-1.019) (0.512,-1.144) (0.632,-1.225) };
 \draw[blue] plot[smooth] coordinates{ (1.000,0.000) (0.925,0.117) (0.798,0.423) (0.703,0.802) (0.649,1.108) (0.632,1.225) (0.649,1.108) (0.703,0.802) (0.798,0.423) (0.925,0.117) (1.000,0.000) };
 \draw[blue] plot[smooth] coordinates{ (1.000,0.000) (0.925,-0.117) (0.798,-0.423) (0.703,-0.802) (0.649,-1.108) (0.632,-1.225) (0.649,-1.108) (0.703,-0.802) (0.798,-0.423) (0.925,-0.117) (1.000,-0.000) };
 \draw[blue] plot[smooth] coordinates{ (-1.000,0.000) (-0.925,0.117) (-0.798,0.423) (-0.703,0.802) (-0.649,1.108) (-0.632,1.225) (-0.649,1.108) (-0.703,0.802) (-0.798,0.423) (-0.925,0.117) (-1.000,0.000) };
 \draw[blue] plot[smooth] coordinates{ (-1.000,0.000) (-0.925,-0.117) (-0.798,-0.423) (-0.703,-0.802) (-0.649,-1.108) (-0.632,-1.225) (-0.649,-1.108) (-0.703,-0.802) (-0.798,-0.423) (-0.925,-0.117) (-1.000,-0.000) };
 \fill (1.000,0.000) circle(0.015);
 \fill (-1.000,0.000) circle(0.015);
 \fill (0.000,-1.000) circle(0.015);
 \fill (0.000,1.000) circle(0.015);
 \fill (0.000,-1.000) circle(0.015);
 \fill (0.000,1.000) circle(0.015);
 \fill (0.000,0.000) circle(0.015);
 \fill (0.000,0.000) circle(0.015);
 \fill (0.000,0.000) circle(0.015);
 \fill (0.000,0.000) circle(0.015);
 \fill (0.632,-1.225) circle(0.015);
 \fill (0.632,1.225) circle(0.015);
 \fill (-0.632,-1.225) circle(0.015);
 \fill (-0.632,1.225) circle(0.015);
 \draw (0.000,0.000) node[anchor=north] {$v_{6},v_{7},v_{8},v_{9}$};
 \draw (-1.000,0.000) node[anchor=east] {$v_{1}$};
 \draw (0.000,-1.000) node[anchor=north] {$v_{2},v_{4}$};
 \draw (1.000,0.000) node[anchor=west] {$v_{0}$};
 \draw (-0.632,1.225) node[anchor=south] {$v_{13}$};
 \draw (0.632,1.225) node[anchor=south] {$v_{11}$};
 \draw (0.000,1.000) node[anchor=south] {$v_{3},v_{5}$};
 \draw (-0.632,-1.225) node[anchor=north] {$v_{12}$};
 \draw (0.632,-1.225) node[anchor=north] {$v_{10}$};
   \draw (0,-1.7) node {$a=1/2$};
  \end{scope}
  \begin{scope}[xshift=3.5cm]
 \draw[cyan] plot[smooth] coordinates{ (0.000,-0.707) (0.000,-0.219) (0.000,0.572) (0.000,0.572) (0.000,-0.219) (0.000,-0.707) (0.000,-0.219) (0.000,0.572) (0.000,0.572) (0.000,-0.219) (0.000,-0.707) };
 \draw[green] plot[smooth] coordinates{ (1.000,0.000) (0.809,0.000) (0.309,0.000) (-0.309,0.000) (-0.809,0.000) (-1.000,0.000) (-0.809,0.000) (-0.309,0.000) (0.309,0.000) (0.809,0.000) (1.000,0.000) };
 \draw[green] plot[smooth] coordinates{ (1.000,0.000) (0.809,0.000) (0.309,0.000) (-0.309,0.000) (-0.809,0.000) (-1.000,0.000) (-0.809,0.000) (-0.309,0.000) (0.309,0.000) (0.809,0.000) (1.000,0.000) };
 \draw[magenta] plot[smooth] coordinates{ (0.816,0.707) (0.661,0.707) (0.252,0.707) (-0.252,0.707) (-0.661,0.707) (-0.816,0.707) (-0.661,0.707) (-0.252,0.707) (0.252,0.707) (0.661,0.707) (0.816,0.707) };
 \draw[magenta] plot[smooth] coordinates{ (0.816,-0.707) (0.661,-0.707) (0.252,-0.707) (-0.252,-0.707) (-0.661,-0.707) (-0.816,-0.707) (-0.661,-0.707) (-0.252,-0.707) (0.252,-0.707) (0.661,-0.707) (0.816,-0.707) };
 \draw[blue] plot[smooth] coordinates{ (1.000,0.000) (0.977,0.068) (0.923,0.244) (0.868,0.463) (0.830,0.640) (0.816,0.707) (0.830,0.640) (0.868,0.463) (0.923,0.244) (0.977,0.068) (1.000,0.000) };
 \draw[blue] plot[smooth] coordinates{ (1.000,0.000) (0.977,-0.068) (0.923,-0.244) (0.868,-0.463) (0.830,-0.640) (0.816,-0.707) (0.830,-0.640) (0.868,-0.463) (0.923,-0.244) (0.977,-0.068) (1.000,-0.000) };
 \draw[blue] plot[smooth] coordinates{ (-1.000,0.000) (-0.977,0.068) (-0.923,0.244) (-0.868,0.463) (-0.830,0.640) (-0.816,0.707) (-0.830,0.640) (-0.868,0.463) (-0.923,0.244) (-0.977,0.068) (-1.000,0.000) };
 \draw[blue] plot[smooth] coordinates{ (-1.000,0.000) (-0.977,-0.068) (-0.923,-0.244) (-0.868,-0.463) (-0.830,-0.640) (-0.816,-0.707) (-0.830,-0.640) (-0.868,-0.463) (-0.923,-0.244) (-0.977,-0.068) (-1.000,-0.000) };
 \fill (1.000,0.000) circle(0.015);
 \fill (-1.000,0.000) circle(0.015);
 \fill (0.000,-0.707) circle(0.015);
 \fill (0.000,0.707) circle(0.015);
 \fill (0.000,-0.707) circle(0.015);
 \fill (0.000,0.707) circle(0.015);
 \fill (0.000,0.000) circle(0.015);
 \fill (0.000,0.000) circle(0.015);
 \fill (0.000,0.000) circle(0.015);
 \fill (0.000,0.000) circle(0.015);
 \fill (0.816,-0.707) circle(0.015);
 \fill (0.816,0.707) circle(0.015);
 \fill (-0.816,-0.707) circle(0.015);
 \fill (-0.816,0.707) circle(0.015);
 \draw (0.000,0.000) node[anchor=north] {$v_{6},v_{7},v_{8},v_{9}$};
 \draw (-0.816,-0.707) node[anchor=north] {$v_{12}$};
 \draw (-1.000,0.000) node[anchor=east] {$v_{1}$};
 \draw (0.816,-0.707) node[anchor=north] {$v_{10}$};
 \draw (1.000,0.000) node[anchor=west] {$v_{0}$};
 \draw (0.000,-0.707) node[anchor=north] {$v_{2},v_{4}$};
 \draw (0.816,0.707) node[anchor=south] {$v_{11}$};
 \draw (0.000,0.707) node[anchor=south] {$v_{3},v_{5}$};
 \draw (-0.816,0.707) node[anchor=south] {$v_{13}$};
   \draw (0,-1.7) node {$a=1/\rt$};
  \end{scope}
  \begin{scope}[xshift=7.0cm]
 \draw[cyan] plot[smooth] coordinates{ (0.000,-0.625) (0.000,-0.193) (0.000,0.506) (0.000,0.506) (0.000,-0.193) (0.000,-0.625) (0.000,-0.193) (0.000,0.506) (0.000,0.506) (0.000,-0.193) (0.000,-0.625) };
 \draw[green] plot[smooth] coordinates{ (1.000,0.000) (0.809,0.000) (0.309,0.000) (-0.309,0.000) (-0.809,0.000) (-1.000,0.000) (-0.809,0.000) (-0.309,0.000) (0.309,0.000) (0.809,0.000) (1.000,0.000) };
 \draw[green] plot[smooth] coordinates{ (1.000,0.000) (0.809,0.000) (0.309,0.000) (-0.309,0.000) (-0.809,0.000) (-1.000,0.000) (-0.809,0.000) (-0.309,0.000) (0.309,0.000) (0.809,0.000) (1.000,0.000) };
 \draw[magenta] plot[smooth] coordinates{ (0.883,0.530) (0.715,0.577) (0.273,0.620) (-0.273,0.620) (-0.715,0.577) (-0.883,0.530) (-0.715,0.577) (-0.273,0.620) (0.273,0.620) (0.715,0.577) (0.883,0.530) };
 \draw[magenta] plot[smooth] coordinates{ (0.883,-0.530) (0.715,-0.577) (0.273,-0.620) (-0.273,-0.620) (-0.715,-0.577) (-0.883,-0.530) (-0.715,-0.577) (-0.273,-0.620) (0.273,-0.620) (0.715,-0.577) (0.883,-0.530) };
 \draw[blue] plot[smooth] coordinates{ (1.000,0.000) (0.988,0.051) (0.958,0.183) (0.921,0.347) (0.894,0.480) (0.883,0.530) (0.894,0.480) (0.921,0.347) (0.958,0.183) (0.988,0.051) (1.000,0.000) };
 \draw[blue] plot[smooth] coordinates{ (1.000,0.000) (0.988,-0.051) (0.958,-0.183) (0.921,-0.347) (0.894,-0.480) (0.883,-0.530) (0.894,-0.480) (0.921,-0.347) (0.958,-0.183) (0.988,-0.051) (1.000,-0.000) };
 \draw[blue] plot[smooth] coordinates{ (-1.000,0.000) (-0.988,0.051) (-0.958,0.183) (-0.921,0.347) (-0.894,0.480) (-0.883,0.530) (-0.894,0.480) (-0.921,0.347) (-0.958,0.183) (-0.988,0.051) (-1.000,0.000) };
 \draw[blue] plot[smooth] coordinates{ (-1.000,0.000) (-0.988,-0.051) (-0.958,-0.183) (-0.921,-0.347) (-0.894,-0.480) (-0.883,-0.530) (-0.894,-0.480) (-0.921,-0.347) (-0.958,-0.183) (-0.988,-0.051) (-1.000,-0.000) };
 \fill (1.000,0.000) circle(0.015);
 \fill (-1.000,0.000) circle(0.015);
 \fill (0.000,-0.625) circle(0.015);
 \fill (0.000,0.625) circle(0.015);
 \fill (0.000,-0.625) circle(0.015);
 \fill (0.000,0.625) circle(0.015);
 \fill (0.000,0.000) circle(0.015);
 \fill (0.000,0.000) circle(0.015);
 \fill (0.000,0.000) circle(0.015);
 \fill (0.000,0.000) circle(0.015);
 \fill (0.883,-0.530) circle(0.015);
 \fill (0.883,0.530) circle(0.015);
 \fill (-0.883,-0.530) circle(0.015);
 \fill (-0.883,0.530) circle(0.015);
 \draw (0.000,0.625) node[anchor=south] {$v_{3},v_{5}$};
 \draw (0.883,0.530) node[anchor=south] {$v_{11}$};
 \draw (0.000,0.000) node[anchor=north] {$v_{6},v_{7},v_{8},v_{9}$};
 \draw (0.000,-0.625) node[anchor=north] {$v_{2},v_{4}$};
 \draw (-0.883,-0.530) node[anchor=north] {$v_{12}$};
 \draw (0.883,-0.530) node[anchor=north] {$v_{10}$};
 \draw (1.000,0.000) node[anchor=west] {$v_{0}$};
 \draw (-1.000,0.000) node[anchor=east] {$v_{1}$};
 \draw (-0.883,0.530) node[anchor=south] {$v_{13}$};
   \draw (0,-1.7) node {$a=4/5$};
  \end{scope}
 \end{tikzpicture}
\end{center}

Here is a more detailed version for the case $a=1/\rt$:
\begin{center}
\begin{tikzpicture}[scale=4]
 \draw[cyan] plot[smooth] coordinates{ (0.000,-0.707) (0.000,-0.572) (0.000,-0.219) (0.000,0.219) (0.000,0.572) (0.000,0.707) (0.000,0.572) (0.000,0.219) (0.000,-0.219) (0.000,-0.572) (0.000,-0.707) (0.000,-0.572) (0.000,-0.219) (0.000,0.219) (0.000,0.572) (0.000,0.707) (0.000,0.572) (0.000,0.219) (0.000,-0.219) (0.000,-0.572) (0.000,-0.707) };
 \draw[green] plot[smooth] coordinates{ (1.000,0.000) (0.951,0.000) (0.809,0.000) (0.588,0.000) (0.309,0.000) (-0.000,0.000) (-0.309,0.000) (-0.588,0.000) (-0.809,0.000) (-0.951,0.000) (-1.000,0.000) (-0.951,0.000) (-0.809,0.000) (-0.588,0.000) (-0.309,0.000) (0.000,0.000) (0.309,0.000) (0.588,0.000) (0.809,0.000) (0.951,0.000) (1.000,0.000) };
 \draw[green] plot[smooth] coordinates{ (1.000,0.000) (0.951,0.000) (0.809,0.000) (0.588,0.000) (0.309,0.000) (-0.000,0.000) (-0.309,0.000) (-0.588,0.000) (-0.809,0.000) (-0.951,0.000) (-1.000,0.000) (-0.951,0.000) (-0.809,0.000) (-0.588,0.000) (-0.309,0.000) (0.000,0.000) (0.309,0.000) (0.588,0.000) (0.809,0.000) (0.951,0.000) (1.000,0.000) };
 \draw[magenta] plot[smooth] coordinates{ (0.816,0.707) (0.777,0.707) (0.661,0.707) (0.480,0.707) (0.252,0.707) (-0.000,0.707) (-0.252,0.707) (-0.480,0.707) (-0.661,0.707) (-0.777,0.707) (-0.816,0.707) (-0.777,0.707) (-0.661,0.707) (-0.480,0.707) (-0.252,0.707) (0.000,0.707) (0.252,0.707) (0.480,0.707) (0.661,0.707) (0.777,0.707) (0.816,0.707) };
 \draw[magenta] plot[smooth] coordinates{ (0.816,-0.707) (0.777,-0.707) (0.661,-0.707) (0.480,-0.707) (0.252,-0.707) (-0.000,-0.707) (-0.252,-0.707) (-0.480,-0.707) (-0.661,-0.707) (-0.777,-0.707) (-0.816,-0.707) (-0.777,-0.707) (-0.661,-0.707) (-0.480,-0.707) (-0.252,-0.707) (0.000,-0.707) (0.252,-0.707) (0.480,-0.707) (0.661,-0.707) (0.777,-0.707) (0.816,-0.707) };
 \draw[blue] plot[smooth] coordinates{ (1.000,0.000) (0.994,0.017) (0.977,0.068) (0.952,0.146) (0.923,0.244) (0.894,0.354) (0.868,0.463) (0.846,0.561) (0.830,0.640) (0.820,0.690) (0.816,0.707) (0.820,0.690) (0.830,0.640) (0.846,0.561) (0.868,0.463) (0.894,0.354) (0.923,0.244) (0.952,0.146) (0.977,0.068) (0.994,0.017) (1.000,0.000) };
 \draw[blue] plot[smooth] coordinates{ (1.000,0.000) (0.994,-0.017) (0.977,-0.068) (0.952,-0.146) (0.923,-0.244) (0.894,-0.354) (0.868,-0.463) (0.846,-0.561) (0.830,-0.640) (0.820,-0.690) (0.816,-0.707) (0.820,-0.690) (0.830,-0.640) (0.846,-0.561) (0.868,-0.463) (0.894,-0.354) (0.923,-0.244) (0.952,-0.146) (0.977,-0.068) (0.994,-0.017) (1.000,-0.000) };
 \draw[blue] plot[smooth] coordinates{ (-1.000,0.000) (-0.994,0.017) (-0.977,0.068) (-0.952,0.146) (-0.923,0.244) (-0.894,0.354) (-0.868,0.463) (-0.846,0.561) (-0.830,0.640) (-0.820,0.690) (-0.816,0.707) (-0.820,0.690) (-0.830,0.640) (-0.846,0.561) (-0.868,0.463) (-0.894,0.354) (-0.923,0.244) (-0.952,0.146) (-0.977,0.068) (-0.994,0.017) (-1.000,0.000) };
 \draw[blue] plot[smooth] coordinates{ (-1.000,0.000) (-0.994,-0.017) (-0.977,-0.068) (-0.952,-0.146) (-0.923,-0.244) (-0.894,-0.354) (-0.868,-0.463) (-0.846,-0.561) (-0.830,-0.640) (-0.820,-0.690) (-0.816,-0.707) (-0.820,-0.690) (-0.830,-0.640) (-0.846,-0.561) (-0.868,-0.463) (-0.894,-0.354) (-0.923,-0.244) (-0.952,-0.146) (-0.977,-0.068) (-0.994,-0.017) (-1.000,-0.000) };
 \fill (1.000,0.000) circle(0.010);
 \fill (-1.000,0.000) circle(0.010);
 \fill (0.000,-0.707) circle(0.010);
 \fill (0.000,0.707) circle(0.010);
 \fill (0.000,-0.707) circle(0.010);
 \fill (0.000,0.707) circle(0.010);
 \fill (0.000,0.000) circle(0.010);
 \fill (0.000,0.000) circle(0.010);
 \fill (0.000,0.000) circle(0.010);
 \fill (0.000,0.000) circle(0.010);
 \fill (0.816,-0.707) circle(0.010);
 \fill (0.816,0.707) circle(0.010);
 \fill (-0.816,-0.707) circle(0.010);
 \fill (-0.816,0.707) circle(0.010);
 \draw (0.000,-0.707) node[anchor=north] {$v_{2},v_{4}$};
 \draw (-0.816,0.707) node[anchor=south] {$v_{13}$};
 \draw (0.000,0.000) node[anchor=north] {$v_{6},v_{7},v_{8},v_{9}$};
 \draw (-0.816,-0.707) node[anchor=north] {$v_{12}$};
 \draw (0.000,0.707) node[anchor=south] {$v_{3},v_{5}$};
 \draw (1.000,0.000) node[anchor=west] {$v_{0}$};
 \draw (-1.000,0.000) node[anchor=east] {$v_{1}$};
 \draw (0.816,0.707) node[anchor=south] {$v_{11}$};
 \draw (0.816,-0.707) node[anchor=north] {$v_{10}$};
 \draw (0.000,-0.354) node[anchor=east] {$c_{0}$};
 \draw (-0.666,0.000) node[anchor=north] {$c_{1}$};
 \draw (-0.505,0.000) node[anchor=north] {$c_{2}$};
 \draw (-0.412,0.707) node[anchor=south] {$c_{3}$};
 \draw (-0.412,-0.707) node[anchor=north] {$c_{4}$};
 \draw (0.901,0.329) node[anchor=west] {$c_{5}$};
 \draw (0.901,-0.329) node[anchor=west] {$c_{6}$};
 \draw (-0.901,0.329) node[anchor=east] {$c_{7}$};
 \draw (-0.901,-0.329) node[anchor=east] {$c_{8}$};
\end{tikzpicture}
\end{center}
The above picture shows the images of all the points $v_i$, but it is
also useful to restrict attention to those that lie in $F_4$:
\begin{center}
\begin{tikzpicture}[scale=4]
 \draw[cyan] plot[smooth] coordinates{ (0.000,-0.707) (0.000,-0.672) (0.000,-0.572) (0.000,-0.416) (0.000,-0.219) (0.000,0.000) (0.000,0.219) (0.000,0.416) (0.000,0.572) (0.000,0.672) (0.000,0.707) };
 \draw[green] plot[smooth] coordinates{ (1.000,0.000) (0.951,0.000) (0.809,0.000) (0.588,0.000) (0.309,0.000) (-0.000,0.000) (-0.309,0.000) (-0.588,0.000) (-0.809,0.000) (-0.951,0.000) (-1.000,0.000) };
 \draw[magenta] plot[smooth] coordinates{ (0.816,0.707) (0.777,0.707) (0.661,0.707) (0.480,0.707) (0.252,0.707) (-0.000,0.707) (-0.252,0.707) (-0.480,0.707) (-0.661,0.707) (-0.777,0.707) (-0.816,0.707) };
 \draw[magenta] plot[smooth] coordinates{ (-0.816,-0.707) (-0.777,-0.707) (-0.661,-0.707) (-0.480,-0.707) (-0.252,-0.707) (0.000,-0.707) (0.252,-0.707) (0.480,-0.707) (0.661,-0.707) (0.777,-0.707) (0.816,-0.707) };
 \draw[blue] plot[smooth] coordinates{ (1.000,0.000) (0.994,0.017) (0.977,0.068) (0.952,0.146) (0.923,0.244) (0.894,0.354) (0.868,0.463) (0.846,0.561) (0.830,0.640) (0.820,0.690) (0.816,0.707) };
 \draw[blue] plot[smooth] coordinates{ (1.000,0.000) (0.994,-0.017) (0.977,-0.068) (0.952,-0.146) (0.923,-0.244) (0.894,-0.354) (0.868,-0.463) (0.846,-0.561) (0.830,-0.640) (0.820,-0.690) (0.816,-0.707) };
 \draw[blue] plot[smooth] coordinates{ (-1.000,0.000) (-0.994,0.017) (-0.977,0.068) (-0.952,0.146) (-0.923,0.244) (-0.894,0.354) (-0.868,0.463) (-0.846,0.561) (-0.830,0.640) (-0.820,0.690) (-0.816,0.707) };
 \draw[blue] plot[smooth] coordinates{ (-1.000,0.000) (-0.994,-0.017) (-0.977,-0.068) (-0.952,-0.146) (-0.923,-0.244) (-0.894,-0.354) (-0.868,-0.463) (-0.846,-0.561) (-0.830,-0.640) (-0.820,-0.690) (-0.816,-0.707) };
 \draw[cyan,arrows={-angle 90}] (0.0000,-0.5000) -- (0.0000,-0.4899);
 \draw[cyan,arrows={-angle 90}] (0.0000,0.5000) -- (0.0000,0.5099);
 \draw[green,arrows={-angle 90}] (0.7071,0.0000) -- (0.7000,0.0000);
 \draw[green,arrows={-angle 90}] (-0.7071,0.0000) -- (-0.7141,0.0000);
 \draw[magenta,arrows={-angle 90}] (0.5774,0.7071) -- (0.5715,0.7071);
 \draw[magenta,arrows={-angle 90}] (-0.5774,0.7071) -- (-0.5831,0.7071);
 \draw[magenta,arrows={-angle 90}] (-0.5774,-0.7071) -- (-0.5715,-0.7071);
 \draw[magenta,arrows={-angle 90}] (0.5774,-0.7071) -- (0.5831,-0.7071);
 \draw[blue,arrows={-angle 90}] (0.8944,0.3536) -- (0.8935,0.3571);
 \draw[blue,arrows={-angle 90}] (0.8944,-0.3536) -- (0.8935,-0.3571);
 \draw[blue,arrows={-angle 90}] (-0.8944,0.3536) -- (-0.8935,0.3571);
 \draw[blue,arrows={-angle 90}] (-0.8944,-0.3536) -- (-0.8935,-0.3571);
 \fill (1.000,0.000) circle(0.010);
 \fill (-1.000,0.000) circle(0.010);
 \fill (0.000,-0.707) circle(0.010);
 \fill (0.000,0.707) circle(0.010);
 \fill (0.000,0.000) circle(0.010);
 \fill (0.816,-0.707) circle(0.010);
 \fill (0.816,0.707) circle(0.010);
 \fill (-0.816,-0.707) circle(0.010);
 \fill (-0.816,0.707) circle(0.010);
 \draw (0.000,-0.707) node[anchor=north] {$v_{2}$};
 \draw (-0.816,0.707) node[anchor=south] {$v_{13}$};
 \draw (0.000,0.000) node[anchor=north] {$v_{6}$};
 \draw (-0.816,-0.707) node[anchor=north] {$v_{12}$};
 \draw (0.000,0.707) node[anchor=south] {$v_{3}$};
 \draw (1.000,0.000) node[anchor=west] {$v_{0}$};
 \draw (-1.000,0.000) node[anchor=east] {$v_{1}$};
 \draw (0.816,0.707) node[anchor=south] {$v_{11}$};
 \draw (0.816,-0.707) node[anchor=north] {$v_{10}$};
 \draw (0.000,-0.354) node[anchor=east] {$c_{0}(0\dotsb\ppi)$};
 \draw (-0.589,0.000) node[anchor=north] {$c_{1}(0\dotsb\pi)$};
 \draw (-0.412,0.707) node[anchor=south] {$c_{3}(0\dotsb\pi)$};
 \draw (-0.412,-0.707) node[anchor=north] {$c_{4}(\pi\dotsb2\pi)$};
 \draw (0.901,0.329) node[anchor=west] {$c_{5}(0\dotsb\pi)$};
 \draw (0.901,-0.329) node[anchor=west] {$c_{6}(0\dotsb\pi)$};
 \draw (-0.901,0.329) node[anchor=east] {$c_{7}(0\dotsb\pi)$};
 \draw (-0.901,-0.329) node[anchor=east] {$c_{8}(0\dotsb\pi)$};
\end{tikzpicture}
\end{center}
The annotation $c_1(0\dotsb\pi)$ indicates that $c_1$ sends the
interval $[0,\pi]$ to $F_4$.  For values $t\in(\pi,2\pi)$, the
point $c_1(t)$ lies outside $F_4$, but it has the same $H_4$-orbit as
some point $c_1(t')$ with $t'\in[0,\pi]$, so $p_4c_1(t)$ will still
lie on the middle horizontal line in the above diagram.  The other
annotations should be interpreted in the same way.

We now define a retractive fundamental domain for the full group $G$.
\begin{definition}\lbl{defn-F-sixteen}
 We put
 \begin{align*}
  F_{16} &= \{x\in EX(a) \st x_1,x_2,y_1,y_2\geq 0\} \\
  r_{16}(x) &=
   \begin{cases}
    (|x_1|,|x_2|,|x_3|,-|x_4|) & \text{ if } y_2\geq 0 \\
    (|x_2|,|x_1|,|x_3|,-|x_4|) & \text{ if } y_2\leq 0.
   \end{cases} \\
  F_{16}^* &= \{(z_1,z_2)\in\R^2\st z_1,z_2,u_3,u_4\geq 0\}
 \end{align*}
 Here $y_1,y_2,z_1,z_2,u_3$ and $u_4$ are as in
 Definition~\ref{defn-yzu}.  Note that if $y_2=0$ then
 $x_4=-y_1y_2=0$ so the relation $y_2=0$ becomes $x_2^2=x_1^2$, so
 $|x_1|=|x_2|$; this shows that $r_{16}$ is well-defined.
\end{definition}
\begin{proposition}\lbl{prop-F-sixteen}
 $F_{16}$ is a retractive fundamental domain for $G$, with retraction
 $r_{16}$.  Moreover, there is a map $p_{16}\:EX(a)\to F_{16}^*$ given by
 \[ p_{16}(x) = (x_3^2,(x_2^2 - x_1^2 - (a^{-1}+a)x_3x_4)^2/(4a^2))
              = (z_1,z_2),
 \]
 and a map $s_{16}\:F_{16}^*\to F_{16}$ given by
 \[ s_{16}(z_1,z_2) = s_4(\sqrt{z_1},\sqrt{z_2}) \]
 and these satisfy $p_{16}s_{16}=1$ and $s_{16}p_{16}=r_{16}$.  Thus,
 $p_{16}$ restricts to give a homeomorphism $F_{16}\to F_{16}^*$ with
 inverse $s_{16}$.
\end{proposition}
Maple notation for $p_{16}(x)$ and $s_{16}(t)$ is \mcode+z_proj(x)+ and
\mcode+z_lift(t)+.
\begin{proof}
 First, a straightforward check of cases shows that
 $r_{16}(EX(a))\sse F_{16}$ and that $r_{16}$ is the identity on $F_{16}$,
 so $r_{16}$ is a retraction.  We also claim that
 $r_{16}(\gm(x))=r_{16}(x)$ for all $x\in EX(a)$ and $\gm\in G$.  If
 $\gm\in\ip{\lm^2,\mu,\nu}$ then $|\gm(x)_i|=|x_i|$ for all $i$ and
 $y_2(\gm(x))=y_2(x)$ so everything is easy.  This just leaves the
 case $\gm=\lm$.  Here $\gm^*$ exchanges $|x_1|$ and $|x_2|$, and
 changes the sign of $y_2$, so we again have
 $r_{16}(\gm(x))=r_{16}(x)$.  It follows that $r_{16}$ induces a
 surjective map $EX(a)/G\to F_{16}$.

 We now show that any point $x\in EX(a)$ can be moved into $F_{16}$
 by the action of $G$.  First, after applying an element of $H_4$ we may
 assume that $x\in F_4$, so $x_1,x_2\geq 0$.  Now note that
 \begin{align*}
  \lm\mu(x) &= (x_2,x_1,   -x_3,\pp x_4) & y_2(\lm\mu(x)) &= -y_2(x) \\
  \lm\nu(x) &= (x_2,x_1,\pp x_3,   -x_4) & y_2(\lm\nu(x)) &= -y_2(x) \\
  \mu\nu(x) &= (x_1,x_2,   -x_3,   -x_4) & y_2(\mu\nu(x)) &= \pp y_2(x).
 \end{align*}
 We can thus apply one of the maps $1,\lm\mu,\lm\nu,\mu\nu$ to move
 into $F_{16}$, as required.  Note also that if $\gm(x)=a\in F_{16}$
 then we can apply $r_{16}$ to deduce that $a=r_{16}(x)$, so
 $x=\gm^{-1}(r_{16}(x))$.  It follows that $r_{16}(x)=r_{16}(x')$ iff
 $Gx=Gx'$, so the induced map $EX(a)/G\to F_{16}$ is a bijective
 retraction and therefore a homeomorphism.

 Now put
 \[ E = \{x\in F_{16}\st x_1=0 \text{ or } x_2 = 0 \text{ or }
          y_1 = 0 \text{ or } y_2 = 0\}.
 \]
 The set $F_{16}\sm E$ is defined by strict inequalities and so is
 contained in the interior of $F_{16}$.  To understand the structure
 of $E$, it is helpful to recall that $c_k(\R)=C_k=C_k^*$ for all $k$,
 where $C_k^*$ was defined in Definition~\ref{defn-C-star}.  Using
 this, we see that
 \[ E =
    c_0([\tfrac{\pi}{4},\ppi]) \cup
    c_1([0,\ppi]) \cup
    c_3([0,\ppi]) \cup
    c_5([0,\pi]).
 \]
 Using this, we can show that arbitrarily close to every point in $E$
 there are points outside $F_{16}$.  For example, if
 $x=c_0(t)=(\cos(t),\sin(t),0,0)$ with $\pi/4\leq t\leq\pi/2$ then the
 vector $(0,0,-1,-\cos(2t))$ is tangent to $EX(a)$ at $x$, and after
 moving a small distance in this direction we reach the region where
 $y_1=x_3<0$.  Similar arguments work for the other cases, so we see
 that $E$ is precisely the boundary of $F_{16}$.

 Now consider a point $x\in F_{16}$ and an element $\gm\in G\sm\{1\}$
 such that $\gm(x)$ also lies in $F_{16}$.  By applying $r_{16}$, we
 see that $\gm(x)=x$.  In Section~\ref{sec-E-curves} we
 discussed the fixed sets $EX(a)^\gm$ for all orientation-reversing
 elements of $G$, and in particular we saw that $EX(a)^\gm$ is always
 contained in $\bigcup_{i=0}^8C_i$.  Thus, if $\gm$ reverses
 orientation then $x\in E$.  On the other hand, if $\gm$ preserves
 orientation then Proposition~\ref{prop-no-more-isotropy} tells us
 that $x=v_i$ for some $i$ with $0\leq i\leq 13$.  The only points of
 this type lying in $F_{16}$ are $v_0,v_3,v_6$ and $v_{11}$, but we
 have
 \begin{align*}
  v_0 &= c_1(0) = c_5(0) \\
  v_3 &= c_0(\pi/2) = c_3(\pi/2) \\
  v_6 &= c_0(\pi/4) = c_1(\pi/2) \\
  v_{11} &= c_3(\pi/2) = c_5(\pi)
 \end{align*}
 so these points are also in $E$.  This proves that
 $\text{int}(F_{16})\cap\gm(F_{16})=\emptyset$.  We conclude that
 $F_{16}$ is a retractive fundamental domain, as claimed.

 We now need to show that $p_{16}(EX(a))\sse F_{16}^*$, or equivalently
 that $z_1,z_2,u_3,u_4\geq 0$ as functions on $EX(a)$.  This is clear from
 the identities $z_i=y_i^2$ and $u_3=4u_1u_2=4x_1^2x_2^2$ and
 $u_4=u_1+u_2=x_1^2+x_2^2$.

 Now suppose we start with a point $z\in F_{16}^*$, and define
 $y_i=\sqrt{z_i}$ for $i=1,2$, and then
 \begin{align*}
  u_1 &= \half(1-y_2) - \half(y_2-a)(y_2-a^{-1})y_1^2 \\
  u_2 &= \half(1+y_2) - \half(y_2+a)(y_2+a^{-1})y_1^2.
 \end{align*}
 We find that $u_1+u_2=1-z_1-z_1z_2$ and
 $4u_1u_2=(1-z_1-z_1z_2)^2 - z_2((a+a^{-1})z_1-2a)^2$; these are the
 quantities $u_3$ and $u_4$ that are assumed to be nonnegative by the
 definition of $F_{16}^*$.  It follows from this by a check of cases
 that both $u_1$ and $u_2$ are nonnegative, so $y\in F_4^*$ and
 the point $x=s_4(y)=(\sqrt{u_1},\sqrt{u_2},y_1,-y_1y_2)$ is a
 well-defined element of $F_4$.  It is clear by construction that
 in fact $x\in F_{16}$, so we have a well-defined map
 $s_{16}\:F_{16}^*\to F_{16}$ with the claimed properties.
 \begin{checks}
  embedded/EX_check.mpl: check_E_F16()
 \end{checks}
\end{proof}

The following picture shows the set $F_{16}^*$ (where $a=1/\rt$)
together with the images under $p_{16}$ of the curves $c_i(t)$ (for
$0\leq i\leq 8$) and the points $v_j$ (for $0\leq j\leq 13$).
\begin{center}
\begin{tikzpicture}[scale=8]
 \draw[cyan] plot[smooth] coordinates{ (0.000,0.500) (0.000,0.327) (0.000,0.048) (0.000,0.048) (0.000,0.327) (0.000,0.500) (0.000,0.327) (0.000,0.048) (0.000,0.048) (0.000,0.327) (0.000,0.500) (0.000,0.327) (0.000,0.048) (0.000,0.048) (0.000,0.327) (0.000,0.500) (0.000,0.327) (0.000,0.048) (0.000,0.048) (0.000,0.327) (0.000,0.500) };
 \draw[green] plot[smooth] coordinates{ (1.000,0.000) (0.905,0.000) (0.655,0.000) (0.345,0.000) (0.095,0.000) (0.000,0.000) (0.095,0.000) (0.345,0.000) (0.655,0.000) (0.905,0.000) (1.000,0.000) (0.905,0.000) (0.655,0.000) (0.345,0.000) (0.095,0.000) (0.000,0.000) (0.095,0.000) (0.345,0.000) (0.655,0.000) (0.905,0.000) (1.000,0.000) };
 \draw[green] plot[smooth] coordinates{ (1.000,0.000) (0.905,0.000) (0.655,0.000) (0.345,0.000) (0.095,0.000) (0.000,0.000) (0.095,0.000) (0.345,0.000) (0.655,0.000) (0.905,0.000) (1.000,0.000) (0.905,0.000) (0.655,0.000) (0.345,0.000) (0.095,0.000) (0.000,0.000) (0.095,0.000) (0.345,0.000) (0.655,0.000) (0.905,0.000) (1.000,0.000) };
 \draw[magenta] plot[smooth] coordinates{ (0.667,0.500) (0.603,0.500) (0.436,0.500) (0.230,0.500) (0.064,0.500) (0.000,0.500) (0.064,0.500) (0.230,0.500) (0.436,0.500) (0.603,0.500) (0.667,0.500) (0.603,0.500) (0.436,0.500) (0.230,0.500) (0.064,0.500) (0.000,0.500) (0.064,0.500) (0.230,0.500) (0.436,0.500) (0.603,0.500) (0.667,0.500) };
 \draw[magenta] plot[smooth] coordinates{ (0.667,0.500) (0.603,0.500) (0.436,0.500) (0.230,0.500) (0.064,0.500) (0.000,0.500) (0.064,0.500) (0.230,0.500) (0.436,0.500) (0.603,0.500) (0.667,0.500) (0.603,0.500) (0.436,0.500) (0.230,0.500) (0.064,0.500) (0.000,0.500) (0.064,0.500) (0.230,0.500) (0.436,0.500) (0.603,0.500) (0.667,0.500) };
 \draw[blue] plot[smooth] coordinates{ (1.000,0.000) (0.988,0.000) (0.954,0.005) (0.907,0.021) (0.853,0.060) (0.800,0.125) (0.753,0.214) (0.716,0.315) (0.689,0.409) (0.672,0.476) (0.667,0.500) (0.672,0.476) (0.689,0.409) (0.716,0.315) (0.753,0.214) (0.800,0.125) (0.853,0.060) (0.907,0.021) (0.954,0.005) (0.988,0.000) (1.000,0.000) };
 \draw[blue] plot[smooth] coordinates{ (1.000,0.000) (0.988,0.000) (0.954,0.005) (0.907,0.021) (0.853,0.060) (0.800,0.125) (0.753,0.214) (0.716,0.315) (0.689,0.409) (0.672,0.476) (0.667,0.500) (0.672,0.476) (0.689,0.409) (0.716,0.315) (0.753,0.214) (0.800,0.125) (0.853,0.060) (0.907,0.021) (0.954,0.005) (0.988,0.000) (1.000,0.000) };
 \draw[blue] plot[smooth] coordinates{ (1.000,0.000) (0.988,0.000) (0.954,0.005) (0.907,0.021) (0.853,0.060) (0.800,0.125) (0.753,0.214) (0.716,0.315) (0.689,0.409) (0.672,0.476) (0.667,0.500) (0.672,0.476) (0.689,0.409) (0.716,0.315) (0.753,0.214) (0.800,0.125) (0.853,0.060) (0.907,0.021) (0.954,0.005) (0.988,0.000) (1.000,0.000) };
 \draw[blue] plot[smooth] coordinates{ (1.000,0.000) (0.988,0.000) (0.954,0.005) (0.907,0.021) (0.853,0.060) (0.800,0.125) (0.753,0.214) (0.716,0.315) (0.689,0.409) (0.672,0.476) (0.667,0.500) (0.672,0.476) (0.689,0.409) (0.716,0.315) (0.753,0.214) (0.800,0.125) (0.853,0.060) (0.907,0.021) (0.954,0.005) (0.988,0.000) (1.000,0.000) };
 \fill (1.000,0.000) circle(0.005);
 \fill (1.000,0.000) circle(0.005);
 \fill (0.000,0.500) circle(0.005);
 \fill (0.000,0.500) circle(0.005);
 \fill (0.000,0.500) circle(0.005);
 \fill (0.000,0.500) circle(0.005);
 \fill (0.000,0.000) circle(0.005);
 \fill (0.000,0.000) circle(0.005);
 \fill (0.000,0.000) circle(0.005);
 \fill (0.000,0.000) circle(0.005);
 \fill (0.667,0.500) circle(0.005);
 \fill (0.667,0.500) circle(0.005);
 \fill (0.667,0.500) circle(0.005);
 \fill (0.667,0.500) circle(0.005);
 \draw (0.000,0.000) node[anchor=north] {$v_{6},v_{7},v_{8},v_{9}$};
 \draw (0.000,0.500) node[anchor=south] {$v_{2},v_{3},v_{4},v_{5}$};
 \draw (1.000,0.000) node[anchor=north] {$v_{0},v_{1}$};
 \draw (0.667,0.500) node[anchor=south] {$v_{10},v_{11},v_{12},v_{13}$};
 \draw (0.000,0.272) node[anchor=east] {$c_{0}$};
 \draw (0.444,0.000) node[anchor=north] {$c_{1}$};
 \draw (0.544,0.000) node[anchor=north] {$c_{2}$};
 \draw (0.296,0.500) node[anchor=south east] {$c_{3}$};
 \draw (0.362,0.500) node[anchor=south east] {$c_{4}$};
 \draw (0.765,0.188) node[anchor=west] {$c_{5}$};
 \draw (0.745,0.235) node[anchor=west] {$c_{6}$};
 \draw (0.727,0.283) node[anchor=west] {$c_{7}$};
 \draw (0.711,0.331) node[anchor=west] {$c_{8}$};
\end{tikzpicture}
\end{center}

If we show only the curves and vertices lying in $F_{16}$, we obtain the
following picture:
\begin{center}
\begin{tikzpicture}[scale=8]
 \draw[cyan] plot[smooth] coordinates{ (0.000,0.000) (0.000,0.012) (0.000,0.048) (0.000,0.103) (0.000,0.173) (0.000,0.250) (0.000,0.327) (0.000,0.397) (0.000,0.452) (0.000,0.488) (0.000,0.500) };
 \draw[green] plot[smooth] coordinates{ (1.000,0.000) (0.976,0.000) (0.905,0.000) (0.794,0.000) (0.655,0.000) (0.500,0.000) (0.345,0.000) (0.206,0.000) (0.095,0.000) (0.024,0.000) (0.000,0.000) };
 \draw[magenta] plot[smooth] coordinates{ (0.667,0.500) (0.650,0.500) (0.603,0.500) (0.529,0.500) (0.436,0.500) (0.333,0.500) (0.230,0.500) (0.137,0.500) (0.064,0.500) (0.016,0.500) (0.000,0.500) };
 \draw[blue] plot[smooth] coordinates{ (1.000,0.000) (0.988,0.000) (0.954,0.005) (0.907,0.021) (0.853,0.060) (0.800,0.125) (0.753,0.214) (0.716,0.315) (0.689,0.409) (0.672,0.476) (0.667,0.500) };
 \draw[cyan,arrows={-angle 90}] (0.0000,0.2500) -- (0.0000,0.2600);
 \draw[green,arrows={-angle 90}] (0.5000,0.0000) -- (0.4900,0.0000);
 \draw[magenta,arrows={-angle 90}] (0.3333,0.5000) -- (0.3267,0.5000);
 \draw[blue,arrows={-angle 90}] (0.7385,0.2507) -- (0.7373,0.2539);
 \fill (1.000,0.000) circle(0.005);
 \fill (0.000,0.500) circle(0.005);
 \fill (0.000,0.000) circle(0.005);
 \fill (0.667,0.500) circle(0.005);
 \draw (0.000,0.000) node[anchor=north] {$v_{6}$};
 \draw (0.000,0.500) node[anchor=south] {$v_{3}$};
 \draw (1.000,0.000) node[anchor=north west] {$v_{0}$};
 \draw (0.667,0.500) node[anchor=south] {$v_{11}$};
 \draw (0.000,0.250) node[anchor=east] {$c_{0}(\qpi\dotsb\ppi)$};
 \draw (0.500,0.000) node[anchor=north] {$c_{1}(0\dotsb\ppi)$};
 \draw (0.333,0.500) node[anchor=south] {$c_{3}(0\dotsb\ppi)$};
 \draw (0.739,0.251) node[anchor=west] {$c_{5}(0\dotsb\pi)$};
\end{tikzpicture}
\end{center}

It is clear from these pictures that $F_{16}^*$ is homeomorphic to the
unit square.  It is convenient to have an explicit homeomorphism,
which is provided by the following result.  (Later we will give other
homeomorphisms that are specific to the case $a=1/\rt$, and have
better properties.)

\begin{proposition}\lbl{prop-square}
 Put
 \begin{align*}
  w_1 &= \frac{z_1(1+(a+a^{-1})\sqrt{z_2}+z_2)}{1+2a\sqrt{z_2}} &
  w_2 &=
   \sqrt{z_2}\frac{2a(1-z_1)+z_1\sqrt{z_2}}{1-z_1+(a^{-1}-a)z_1\sqrt{z_2}},
 \end{align*}
 with the convention that $w_2=0$ at points where
 $1-z_1+(a^{-1}-a)z_1\sqrt{z_2}=0$.  Then the map
 $q\:(z_1,z_2)\mapsto(w_1,w_2)$ gives a homeomorphism $F_{16}^*\to[0,1]^2$.
\end{proposition}
\begin{proof}
 In view of Proposition~\ref{prop-F-sixteen} we can identify
 $F_{16}^*$ with $F_{16}\sse EX(a)$, so we have nonnegative functions
 $y_1,y_2,u_1,u_2$ as discussed previously and $\sqrt{z_2}=y_2$.

 It is clear that $w_1$ is continuous and nonnegative.  Next,
 note that the functions $1-z_1=u_4+z_1z_2$ and $z_1\sqrt{z_2}$ are
 nonnegative on $F_{16}^*$, and the only place where they both vanish
 is $(1,0)$.  Away from that point we deduce that $w_2$ is continuous
 with $0\leq w_2\leq\sqrt{z_2}$, and these inequalities imply that
 $w_2$ is continuous at the exceptional point as well.

 One can also check that $1-w_1=2u_2/(1+2ay_2)\geq 0$ and
 $1-w_2=2u_1/(1-y_1^2+(a^{-1}-a)y_1^2y_2)\geq 0$ so
 $w_1,w_2\leq 1$.  Thus, we have a well-defined and continuous map
 $q\:F_{16}^*\to [0,1]^2$.

 Now suppose we start with a point $w\in[0,1]^2$.  We will assume that
 $w_1<1$; the case $w_1=1$ requires only minor modifications and is
 left to the reader.  Consider the functions
 \begin{align*}
  p_0(s) &= s((1+w_1)s^2+(a^{-1}+a+((2a)^{-1}-2a)w_1)s+(1-w_1)) \\
  p_1(s) &= ((2a)^{-1}+(a^{-1}-a)w_1)s^2+
             \half((a^{-2}-1)(w_1+1)+2(1-w_1))s+(2a)^{-1}(1-w_1) \\
  p(s) &= p_0(s)/p_1(s).
 \end{align*}
 All coefficients of powers of $s$ in $p_0(s)$ and $p_1(s)$
 are strictly positive, and $p_0(s)$ is cubic whereas $p_1(s)$ is
 quadratic.  It follows that $p$ defines a continuous function
 $[0,\infty)\to[0,\infty)$, with $p(0)=0$ and $p(s)\to\infty$ as
 $s\to\infty$.  We claim that this is strictly increasing.  Indeed, by
 standard algebra one can check that
 $p'(s)=p_1(s)^{-2}\sum_{i=0}^4m_is^i$, where
 \begin{align*}
  m_0 &= \half a^{-1}(1-w_1)^2 \\
  m_1 &= (1-w_1)((\tfrac{3}{2}a^{-2}-1)w_1+(a^{-2}+1)(1-w_1)) \\
  m_2 &= \half a^{-3}(1-a^2)(1-w_1) +
         a^{-3}(1-a^2)(\tfrac{3}{2}-a^2)w_1^2 + \\ & \qquad
         \tfrac{1}{4}a^{-3}(1+2a^2)(5-a^2)w_1(1-w_1) +
         \half a^{-1}(5+a^2)(1-w_1)^2 \\
  m_3 &= a^{-2}(1+w_1)(2(1-a^2)w_1+(1+a^2)(1-w_1)) \\
  m_4 &= \half a^{-1}(1+w_1)(1+2(1-a^2)w_1).
 \end{align*}
 We have written these coefficients in a form that makes it clear that
 they are positive.  It follows that $p'(s)>0$ for $s\geq 0$, as
 claimed.  It follows that there is a unique number $y_2\geq 0$ with
 $p(y_2)=w_2$.  We put
 \[ y_1 = \sqrt{\frac{w_1(1+2ay_2)}{(y_2+a)(y_2+a^{-1})}} \]
 and $y=(y_1,y_2)\in [0,\infty)^2$.  We then define $u_1$ and $u_2$ in
 terms of $y$ in the usual way.  If we substitute the above value for
 $y_1$, then straightforward algebra gives
 \begin{align*}
  u_1 &= (1-p(y_2))
         \frac{((1-a^2)w_1+\half)y_2^2+(\half(1-w_1)(a+a^{-1})+w_1(a+a^{-1}))y_2+(1-w_1)/2}{(y_2+a)(y_2+a^{-1})} \\
  u_2 &= (1+2ay_2)(1-w_1)/2.
 \end{align*}
 After recalling that $y_2\geq 0$ and $p(y_2)=w_2\in [0,1]$ it follows that
 $u_1,u_2\geq 0$, so $y$ lies in $F_4^*$ and the point
 $z=(y_1^2,y_2^2)$ lies in $F_{16}^*$.  Now note that $y_2\geq 0$ and put
 \begin{align*}
  w'_1 &= \frac{z_1(1+(a+a^{-1})\sqrt{z_2}+z_2)}{1+2a\sqrt{z_2}} &
  w'_2 &=
   \sqrt{z_2}\frac{2a(1-z_1)+z_1\sqrt{z_2}}{1-z_1+(a^{-1}-a)z_1\sqrt{z_2}},
 \end{align*}
 so $q(z)=(w'_1,w'_2)$.  If we substitute in our definition of $y_1$
 and simplify we get $w'_1=w_1$ and $w'_2=p(y_2)$, but $p(y_2)=w_2$,
 so $q(z)=w$.  We leave it to the reader to check that all steps in this
 construction are forced, so $z$ is the unique point in $F_{16}^*$ with
 $q(z)=w$.  This means that $q\:F_{16}^*\to [0,1]^2$ is a continuous
 bijection between compact Hausdorff spaces, so it is a
 homeomorphism.
 \begin{checks}
  embedded/invariants_check.mpl: check_invariants()
 \end{checks}
\end{proof}

The following picture shows the images under $q\circ p_{16}$ of the
curves $c_i(t)$ and the points $v_j$.
\begin{center}
\begin{tikzpicture}[scale=6]
 \draw[cyan] plot[smooth] coordinates{ (0.000,0.000) (0.000,0.156) (0.000,0.309) (0.000,0.454) (0.000,0.588) (0.000,0.707) (0.000,0.809) (0.000,0.891) (0.000,0.951) (0.000,0.988) (0.000,1.000) };
 \draw[green] plot[smooth] coordinates{ (1.000,0.000) (0.976,0.000) (0.905,0.000) (0.794,0.000) (0.655,0.000) (0.500,0.000) (0.345,0.000) (0.206,0.000) (0.095,0.000) (0.024,0.000) (0.000,0.000) };
 \draw[magenta] plot[smooth] coordinates{ (1.000,1.000) (0.977,1.000) (0.911,1.000) (0.805,1.000) (0.668,1.000) (0.514,1.000) (0.357,1.000) (0.214,1.000) (0.099,1.000) (0.026,1.000) (0.000,1.000) };
 \draw[blue] plot[smooth] coordinates{ (1.000,0.000) (1.000,0.025) (1.000,0.097) (1.000,0.208) (1.000,0.348) (1.000,0.502) (1.000,0.656) (1.000,0.795) (1.000,0.905) (1.000,0.976) (1.000,1.000) };
 \draw[cyan,arrows={-angle 90}] (0.0000,0.5403) -- (0.0000,0.5234);
 \draw[green,arrows={-angle 90}] (0.5000,0.0000) -- (0.4900,0.0000);
 \draw[magenta,arrows={-angle 90}] (0.5138,1.0000) -- (0.5037,1.0000);
 \draw[blue,arrows={-angle 90}] (1.0000,0.5665) -- (1.0000,0.5715);
 \fill (1.000,0.000) circle(0.005);
 \fill (0.000,1.000) circle(0.005);
 \fill (0.000,0.000) circle(0.005);
 \fill (1.000,1.000) circle(0.005);
 \draw (1.000,1.000) node[anchor=south] {$v_{11}$};
 \draw (0.000,1.000) node[anchor=south] {$v_{3}$};
 \draw (1.000,0.000) node[anchor=north west] {$v_{0}$};
 \draw (0.000,0.000) node[anchor=north] {$v_{6}$};
 \draw (0.000,0.540) node[anchor=east] {$c_{0}(\qpi\dotsb\ppi)$};
 \draw (0.500,0.000) node[anchor=north] {$c_{1}(0\dotsb\ppi)$};
 \draw (0.514,1.000) node[anchor=south] {$c_{3}(0\dotsb\ppi)$};
 \draw (1.000,0.567) node[anchor=west] {$c_{5}(0\dotsb\pi)$};
\end{tikzpicture}
\end{center}

Some relevant formulae are as follows:
\begin{align*}
 qp_{16}(v_i) &= (1,0)      && \text{ for } 0\leq i \leq 1 \\
 qp_{16}(v_i) &= (0,1)      && \text{ for } 2\leq i \leq 5 \\
 qp_{16}(v_i) &= (0,0)      && \text{ for } 6\leq i \leq 9 \\
 qp_{16}(v_i) &= (1,1)      && \text{ for } 10\leq i \leq 13
\end{align*}
\begin{align*}
 qp_{16}(c_0(t)) &=\! (0,\;|\cos(2t)|) \\
 qp_{16}(c_1(t))=qp_{16}(c_2(t)) &=\! (\cos(t)^2,\;0) \\
 qp_{16}(c_{\alg}(t)) &=\! \begin{cases}
    \left(a^{-1}\frac{(1-2at)(t+a)(t+a^{-1})}{(1+2at)(t-a)(t-a^{-1})},\;1\right)
     & \text{ if } t\in T_{\alg}^+ \\
    \left(1,\;-at\frac{3-2a^2-4at}{2-2a^2-(3a+2a^3)t}\right)
     & \text{ if } t\in T_{\alg}^-.
   \end{cases}
\end{align*}

\begin{proposition}\lbl{prop-E-cromulent}
 $EX(a)$ is a cromulent surface.
\end{proposition}
\begin{proof}
 Condition~(a) in Definition~\ref{defn-precromulent} is proved in
 Section~\ref{sec-E-G}, and conditions~(b) and~(c) are
 proved in Section~\ref{sec-E-isotropy}.  For condition~(d),
 we can take $F'$ to be the interior of $F_{16}$.
\end{proof}

\subsection{Additional points and curves}
\lbl{sec-extra-curves}

The surface $EX(a)$ has some additional points and curves that are
not part of the precromulent structure but are nevertheless useful for
some purposes (such as our analysis of torus quotients of $EX^*$ in
Section~\ref{sec-torus-quotients}).  All claims in this section are checked as
follows:
\begin{checks}
 embedded/extra_vertices_check.mpl: check_extra_vertices()
 embedded/extra_curves_check.mpl: check_extra_curves()
\end{checks}

\begin{definition}\lbl{defn-c-nine}
 We put
 \[ c_9(t) = \left(
     \sqrt{\frac{1-a^2}{2(1+a^2)}}\sin(t),\;
     \sqrt{\frac{1-a^2}{2(1+a^2)}}\sin(t),\;
     \sqrt{\frac{2a^2}{1+a^2}},\;
     -\sqrt{\frac{1-a^2}{1+a^2}}\cos(t)
    \right),
 \]
 then
 \[ c_{10}(t) = \lm(c_9(t)) \qquad
    c_{11}(t) = \mu(c_9(t)) \qquad
    c_{12}(t) = \lm\mu(c_9(t)).
 \]
\end{definition}

It is straightforward to check that $c_9$ lands in $EX(a)$, so the
same is true for $c_{10}$, $c_{11}$ and $c_{12}$.  One can also check
that $\lm^2c_9(t)=c_9(-t)$ and $\lm\nu(c_9(t))=c_9(\pi-t)$.  It
follows that we cannot get anything interestingly new by applying
further group elements to the above curves.  In the case $a=1/\rt$, we
have
\[ c_9(t) = \left(
             \frac{\sin(t)}{\sqrt{6}},\;
             \frac{\sin(t)}{\sqrt{6}},\;
             \sqrt{\frac{2}{3}},\;
             -\frac{\cos(t)}{\sqrt{3}}
            \right).
\]

\begin{definition}\lbl{defn-c-thirteen}
 We put
 \[\begin{split}
  c_{13}(t) = &\left(
     \frac{\cos(t)}{\rt}\left(1-\frac{\sin(t)}{\sqrt{2/(1-a^2)-\cos(t)^2}}\right),\quad
     \frac{\cos(t)}{\rt}\left(1+\frac{\sin(t)}{\sqrt{2/(1-a^2)-\cos(t)^2}}\right),
     \right.\\ & \qquad\left.
     \frac{\rt a}{\sqrt{1+a^2}}\,\sin(t),\quad
     \frac{-\rt\sin(t)^2}{\sqrt{1+a^2}\sqrt{2/(1-a^2)-\cos(t)^2}}
    \right),
 \end{split}\]
 then
 \[
    c_{14}(t) = \lm(c_{13}(t)) \qquad
    c_{15}(t) = \nu(c_{13}(t)) \qquad
    c_{16}(t) = \lm\nu(c_{13}(t)).
 \]
\end{definition}
In the case $a=1/\rt$ this becomes
\[ c_{13}(t) = \left(
     \frac{\cos(t)}{\rt}\left(1-\frac{\sin(t)}{\sqrt{4-\cos(t)^2}}\right),
     \frac{\cos(t)}{\rt}\left(1+\frac{\sin(t)}{\sqrt{4-\cos(t)^2}}\right),
     \sqrt{2/3}\,\sin(t),
     \frac{-2\sin(t)^2}{\sqrt{3}\sqrt{4-\cos(t)^2}}
    \right).
\]
It is straightforward to check that $c_{13}$ lands in $EX(a)$, so the
same is true for $c_{14}$, $c_{15}$ and $c_{16}$.  One can also check
that $\lm\mu(c_{13}(t))=c_{13}(-t)$ and
$\lm^2(c_{13}(t))=c_{13}(\pi-t)$.  It again follows that we cannot get
anything interestingly new by applying further group elements.

\begin{definition}\lbl{defn-c-seventeen}
 If $a\leq 1/\rt$, we put
 \[ c_{17}(t) = \left(
     \sqrt{\frac{1-2a^2}{3(1+2a^2)}}\sin(t),\;
     -\cos(t),\;
     -\frac{4a}{\sqrt{6(1+2a^2)}}\sin(t),\;
     \frac{2}{\sqrt{6(1+2a^2)}}\sin(t)
    \right),
 \]
 then
 \[ c_{18}(t) = \lm(c_{13}(t)) \qquad
    c_{19}(t) = \lm^2(c_{13}(t)) \qquad
    c_{20}(t) = \lm^3(c_{13}(t)).
 \]
\end{definition}
One can again check that this produces curves in $EX(a)$ satisfying
$\nu c_{17}(t)=c_{17}(\pi-t)$ and $\lm^2\mu(c_{17}(t))=c_{17}(-t)$.
In the case $a=1/\rt$, we find that
$c_{17}(t)=c_3(-\tfrac{\pi}{2}-t)$, and similarly $c_{18}$, $c_{19}$
and $c_{20}$ are just reparametrisations of the lower numbered curves.

Note that for $k\in\{0,1,2,17,\dotsc,20\}$, the image $C_k=c_k(\R)$ is
a great circle.  The homogeneous polynomial $g(x)$ defines a cubic
surface in the projective space $\mathbb{P}^3$, which is smooth except
when $a=1/\rt$.  A famous theorem of Cayley and Salmon says that any
smooth cubic surface contains precisely 27 linearly embedded copies of
$\mathbb{P}^1$ (when counted with appropriate multiplicities);
see~\cite[Theorem 9.1.13]{do:cag} for a modern treatment.  In our
case, the above great circles give seven copies of $\mathbb{P}^1$.
One can check that the remaining copies come in ten complex conjugate
pairs, and so do not correspond to great circles in the real variety
$EX(a)$.  In the case $a=1/\rt$ everything degenerates and we have
only five great circles and two additional conjugate pairs of
$\mathbb{P}^1$'s.  Some or all of these must have multiplicity greater
than one, but we have not investigated this.
\begin{checks}
 embedded/cayley_check.mpl: check_cayley()
\end{checks}

For $k\in\{9,\dotsc,12\}$ the image is again the intersection of $S^3$
with a two-dimensional subspace of $\R^4$, but in these cases it is an
affine subspace rather than a vector subspace.  We suspect that again
there are no more curves of this type contained in $EX(a)$, but we
have not proved this.

\begin{remark}
 The curves $c_i(t)$ for $0\leq i\leq 16$ are represented in Maple as
 \mcode+c_E[i](t)+.  However, for $17\leq i\leq 20$, the curves
 $c_i(t)$ are not defined when $a>1/\rt$, and this makes it
 inconvenient to use the same framework.  Instead, these curves are
 represented in Maple by the functions \mcode+c_cayley[j](t)+ for
 $1\leq j\leq 4$.
\end{remark}

We next introduce some additional points $v_i$ for $14\leq i\leq 45$.
For this, it is convenient to enumerate the elements of $G$ as
follows:
\begin{align*}
 \gm_0    &= 1           & \gm_1    &= \lm &
 \gm_2    &= \lm^2       & \gm_3    &= \lm^3 \\
 \gm_4    &= \mu         & \gm_5    &= \lm\mu &
 \gm_6    &= \lm^2\mu    & \gm_7    &= \lm^3\mu \\
 \gm_8    &= \nu         & \gm_9    &= \lm\nu &
 \gm_{10} &= \lm^2\nu    & \gm_{11} &= \lm^3\nu \\
 \gm_{12} &= \mu\nu      & \gm_{13} &= \lm\mu\nu &
 \gm_{14} &= \lm^2\mu\nu & \gm_{15} &= \lm^3\mu\nu.
\end{align*}

\begin{definition}
 We put
 \[
  v_{14}  = \left(\sqrt{\frac{1-a^2}{2(1+a^2)}},\;
                  \sqrt{\frac{1-a^2}{2(1+a^2)}},\;
                  \sqrt{\frac{2a^2}{1+a^2}},\;
                  0\right),
 \]
 then $v_{14+i}=\gm_i(v_{14})$ for $0\leq i<8$.  If
 $a\leq 1/\rt$ we also put
 \begin{align*}
  v_{22} &= \left(
             \sqrt{\frac{1-2a^2}{3(1+2a^2)}},\;
             0,\;
             \sqrt{\frac{8a^2}{3(1+2a^2)}},\;
             -\sqrt{\frac{2}{3(1+2a^2)}}
            \right) \\
  v_{30} &= \left(
             \sqrt{\frac{1-2a^2}{4(1+a^2)}},\;
             \sqrt{\frac{1-2a^2}{4(1+a^2)}},\;
             \sqrt{\frac{2a^2}{1+a^2}},\;
             \sqrt{\frac{1}{2(1+a^2)}}
            \right),
 \end{align*}
 then
 \begin{align*}
  v_{22+i} &= \gm_i(v_{22}) && (0\leq i<8) \\
  v_{30+i} &= \gm_i(v_{30}) && (0\leq i<16).
 \end{align*}
\end{definition}

Straightforward calculations show that these points lie in $EX(a)$.
We have $\lm\nu(v_{14})=v_{14}$ and $\nu(v_{22})=v_{22}$, which
implies that $\{v_0,\dotsc,v_{45}\}$ is closed under the action of
$G$.  One can check that
\begin{align*}
 C_1\cap C_9    &= \{v_{14},v_{16}\} \\
 C_5\cap C_{19} &= \{v_{22}\} \\
 C_9\cap C_{20} &= \{v_{30}\}.
\end{align*}
This is the main justification for considering these extra points.

\subsection{Charts}
\lbl{sec-E-charts}

In this section we discuss three different kinds of charts for $EX(a)$.

Our first construction is simple and works at every point of $EX(a)$.  It
only gives an approximate chart, but that is sufficient for many
purposes.

\begin{definition}\lbl{defn-quadratic-chart}
 Let $x$ be a point in $EX(a)$, and let $T$ be the tangent space to
 $EX(a)$ at $x$.  We define $\phi\:T\to\R^4$ by
 \[ \phi(t) = \left(1 - \frac{\|t\|^2}{2}\right)x + t
               - \frac{n(t).x}{\|n(x)\|^2}n(x).
 \]
 (Here $n(x)$ is the gradient of $g$, as in Definition~\ref{defn-nx}.)
 We call this the \emph{quadratic approximate chart} at $x$.  Maple
 notation for $\phi(t)$ is \mcode+quadratic_chart(x,t)+ (defined in
 \fname+embedded/geometry.mpl+).
\end{definition}

\begin{proposition}\lbl{prop-quadratic-chart}
 We have
 \begin{align*}
  \phi(t) &= x + t + O(\|t\|^2) \\
  g(\phi(t)) &= O(\|t\|^3) \\
  \rho(\phi(t)) &= 1 + O(\|t\|^4).
 \end{align*}
\end{proposition}
\begin{proof}
 As the map $n\:\R^4\to\R^4$ is homogeneous quadratic, we see that
 $\phi(t)=x+t+O(\|t\|^2)$.  Next, as $x$, $n(x)$ and $t$ are mutually
 orthogonal we have
 \begin{align*}
  \rho(\phi(t)) &=
   \left(1 - \frac{\|t\|^2}{2}\right)^2 + \|t\|^2
               + \left(\frac{n(t).x}{\|n(x)\|^2}\right)^2\|n(x)\|^2 \\
   &= 1 + \frac{\|t\|^4}{4} + \frac{(n(t).x)^2}{\|n(x)\|^2}.
 \end{align*}
 Using the fact that $n$ is homogeneous quadratic again, we see that
 this is $1+O(\|t\|^4)$.

 Next, using the fact that $g$ is homogeneous cubic we find that
 \[ g(a+b) = g(a)+n(a).b + n(b).a + g(b). \]
 We apply this with $a=\left(1 - \frac{\|t\|^2}{2}\right)x$ and
 $b=t-\frac{n(t).x}{\|n(x)\|^2}n(x)$, neglecting terms of order
 $\|t\|^3$ everywhere.  As $g(x)=0$ we have $g(a)=0$.  Moreover, $b$
 is $O(\|t\|)$, so $g(b)$ is negligible, and when calculating $n(a).b$
 we can neglect terms in $n(a)$ that are quadratic in $t$.  This
 leaves $n(a)\simeq n(x)$, and $n(x)$ is normal to $t$, so
 \[ n(x).b \simeq -\frac{n(t).x}{\|n(x)\|^2} n(x).n(x) =
     -n(t).x.
 \]
 Similarly, as $n$ is quadratic, we can neglect terms in $b$ that are
 $O(\|t\|^2)$ when calculating $n(b)$.  This gives $n(b)\simeq n(t)$
 and so $n(b).a\simeq n(t).a\simeq n(t).x$.  Altogether, we have
 \[ g(\phi(t)) = g(a) + n(a).b + n(b).a + g(b) \simeq
     0 -n(t).x + n(t).x + 0 = 0.
 \]
 \begin{checks}
  embedded/geometry_check.mpl: check_quadratic_chart()
 \end{checks}
\end{proof}

Next, recall from Section~\ref{sec-holomorphic-curves} that each of
the maps $c_k\:\R\to EX(a)$ can be extended in a canonical way to give
a holomorphic map $\tc_k$ defined on a neighbourhood of $\R$ in $\C$.

The file \fname+embedded/annular_charts.mpl+ gives a formula for
$\tc_0(t+iu)$ modulo $u^4$, and formulae for $\tc_1(t+iu)$ and
$\tc_2(t+iu)$ modulo $u^3$, but we will not reproduce them here.
Given a fixed value of $a$ and $t_0\in\R$ it is also not hard to
compute power series for $\tc_k(t_0+t+iu)$ to reasonably high order,
and similar methods can be used to produce series for conformal charts
centred at points that do not lie on any of the curves $C_k$.  We
postpone a more detailed discussion to
Section~\ref{sec-roothalf-charts}, where we focus on the case
$a=1/\rt$.

We next describe a different class of charts that will be useful for
triangulating $EX(a)$.  It is inspired by the definition of barycentric
coordinates for spherical triangles described in~\cite{labese:sbc}.
Related code is in the file \fname+embedded/barycentric.mpl+.

\begin{definition}\lbl{defn-barycentric}
 Let $a_0$, $a_1$ and $a_2$ be distinct points on $EX(a)$.  For any
 $x\in EX(a)$ we let $n(x)$ denote the gradient of $g$ at $EX(a)$, and we put
 \[ \widetilde{p}(x) =
     (\det(x,n(x),a_1,a_2),\;
      \det(x,n(x),a_2,a_0),\;
      \det(x,n(x),a_0,a_1)) \in \R^3.
 \]
 If $\sum_i\widetilde{p}(x)_i\neq 0$, we put
 \[ p(x) = \widetilde{p}(x)/\sum_i\widetilde{p}(x)_i. \]
 This clearly lies in the set $\R^3_1=\{t\in\R^3\st\sum_it_i=1\}$,
 which contains the simplex $\Dl_2$.  If we need to emphasise the
 dependence on the points $a_i$, we will write $p_{a_0,a_1,a_2}(x)$
 rather than $p(x)$.  We call the components of $p(x)$ the
 \emph{barycentric coordinates} of $x$ (with respect to the $a_i$).
\end{definition}

\begin{definition}
 For $x\in EX(a)$ we put $T'_xEX(a)=x+T_xEX(a)\subset\R^4$, and call this the
 \emph{affine tangent space} to $EX(a)$ at $x$.  In a small neighbourhood
 of $x$, this is of course a good approximation to $EX(a)$ itself.  We
 write $\pi'_x(y)$ for the closest point in $T'_xEX(a)$ to $y$.  This can
 be computed as
 \[ \pi'_x(y) = x+y-\ip{y,x}x-\ip{y,n(x)}n(x)/\|n(x)\|^2. \]
\end{definition}

The next result motivates the term ``barycentric coordinates''.
\begin{lemma}
 If $p(x)$ is defined, then it is the unique element $t\in\R^3_1$ such
 that $x=\sum_it_i\pi'_x(a_i)$.
\end{lemma}
\begin{proof}
 First, we write $\tu=\widetilde{p}(x)$, so
 \begin{align*}
  \tu_0 &= \det(x,n(x),a_1,a_2) \\
  \tu_1 &= \det(x,n(x),a_2,a_0) \\
  \tu_2 &= \det(x,n(x),a_0,a_1).
 \end{align*}
 As $p(x)$ is assumed to be defined, we must have
 $\tu_0+\tu_1+\tu_2\neq 0$, and in particular
 $(\tu_0,\tu_1,\tu_2)\neq(0,0,0)$.

 Next, the list $a_0,a_1,a_2,x,n(x)$ consists of five vectors in
 $\R^4$, so there must be a linear relation
 \[ s_0a_0 + s_1a_1 + s_2a_2 + s_3x + s_4n(x) = 0 \]
 where not all of the coefficients $s_i$ are zero.  As $x$ and $n(x)$
 are nonzero and orthogonal, it follows easily that
 $(s_0,s_1,s_2)\neq(0,0,0)$.

 Now apply the map $y\mapsto\det(x,n(x),a_0,y)$ to our linear
 relation.  Only the second and third terms contribute anything, and
 we deduce that $s_1\tu_2-s_2\tu_1=0$.  Similarly, we can apply the
 map $y\mapsto\det(x,n(x),a_1,y)$ to see that $s_0\tu_2-s_2\tu_0=0$,
 and we can apply the map $y\mapsto\det(x,n(x),a_2,y)$ to see that
 $s_0\tu_1-s_1\tu_0=0$.  We have already seen that the vectors
 $(\tu_0,\tu_1,\tu_2)$ and $(s_0,s_1,s_2)$ are nonzero, and the above
 relations imply that they are unit multiples of each other.  It
 follows that the vector $u=p(x)=\tu/\sum_i\tu_i$ is also a unit
 multiple of $(s_0,s_1,s_2)$.  The definition of the coefficients
 $s_i$ implies that
 \[ s_0a_0+s_1a_1+s_2a_2\in\text{span}(x,n(x)) = (T_xEX(a))^\perp, \]
 and we now see that $\sum_iu_ia_i\in(T_xEX(a))^\perp$ as well.  From this
 it follows easily that $x=\sum_iu_i\pi'_x(a_i)$ as claimed.

 All that is left is to check that the numbers $u_i$ are uniquely
 characterised by the above property.  Suppose there is another
 vector $u'\in\R^3_1$ with $x=\sum_iu'_i\pi'_x(a_i)$.  It
 follows that the vector $r=u'-u$ has $\sum_ir_i=0$ and
 $\sum_ir_ia_i\in(T_xEX(a))^\perp=\text{span}(x,n(x))$, so there exist
 scalars $r_3,r_4$ such that
 \[ r_0a_0 + r_1a_1 + r_2a_2 + r_3x + r_4n(x) = 0. \]
 Just as before we deduce that $r_i\tu_j-r_j\tu_i=0$ for all
 $i,j\in\{0,1,2\}$, so $(r_0,r_1,r_2)$ is a multiple of
 $(\tu_0,\tu_1,\tu_2)$.  As $\sum_ir_i=0$ but $\sum_i\tu_i\neq 0$, we
 see that the multiplier must be zero, so $u'=u$ as claimed.
\end{proof}

From the above characterisation, we can show that barycentric
coordinates for adjacent triangles are equal on the shared edge, in
the following sense.

\begin{corollary}\lbl{cor-edge-barycentric}
 Suppose we have points $a_0,a_1,b,b',x\in EX(a)$ such that the vectors
 $u=p_{a_0,a_1,b}(x)$ and $u'=p_{a_0,a_1,b'}(x)$ are both defined, and
 that $u_2=0$.  Then $u'=u$ (and in particular, $u'_2$ is also zero).
\end{corollary}
\begin{proof}
 We can apply the lemma to see that
 $x=u_0\pi'_x(a_0)+u_1\pi'_x(a_1)$.  On the other hand, $u'$ is
 uniquely characterised by the fact that
 $x=u'_0\pi'_x(a_0)+u'_1\pi'_x(a_1)+u'_2\pi'_x(b')$.  The claim is
 clear from this.
\end{proof}

\begin{remark}
 A similar uniqueness argument shows that if $p_{a_0,a_1,a_2}(a_i)$ is
 defined, then it must be equal to the $i$'th standard basis vector
 $e_i$.  More precisely, it is clear from the definitions that
 $\widetilde{p}_{a_0,a_1,a_2}(a_i)_j=0$ for all $j\neq i$, so either
 $\widetilde{p}_{a_0,a_1,a_2}(a_i)_i\neq 0$ (and
 $p_{a_0,a_1,a_2}(a_i)=e_i$) or $\widetilde{p}_{a_0,a_1,a_2}(a_i)_i=0$
 (and $p_{a_0,a_1,a_2}(a_i)$ is undefined).
\end{remark}

\begin{definition}
 We put
 \[ T_0(a_0,a_1,a_2) = \{x\in EX(a)\st p_{a_0,a_1,a_2}(x)
     \text{ is defined and lies in } \Dl_2 \}.
 \]
 Typically, if the $a_i$ are close together and in general position,
 there will be a single connected component
 $T(a_0,a_1,a_2)\sse T_0(a_0,a_1,a_2)$ that contains all the points
 $a_i$, and the map $p_{a_0,a_1,a_2}$ will restrict to give a
 diffeomorphism $T(a_0,a_1,a_2)\to\Dl_2$.  We can use maps of this
 form to give a triangulation of $EX(a)$, with compatibility along edges
 given by Corollary~\ref{cor-edge-barycentric}.
\end{definition}

\begin{remark}\lbl{rem-barycentric-inverse}
 The inverse of the map $p\:T(a_0,a_1,a_2)\to\Dl_2$ can be computed
 efficiently by a kind of Newton-Raphson method.  If $x$ is reasonably
 close to $p^{-1}(t)$ then the first order Taylor approximation at $x$
 to the map $p\:EX(a)\to\R^3_1$ will be an affine isomorphism
 $p_*\:T'_xEX(a)\to\R^3_1$, so we can define
 $\kp(x)=\sg^{\infty}(p_*^{-1}(t))$, where $\sg^\infty$ is as in
 Remark~\ref{rem-move-to-X}.  As an initial approximation we can take
 $x_0=\sg^{\infty}(\sum_it_ia_i)$, and then the sequence
 $(\kp^n(x_0))_{n\geq 0}$ converges rapidly to $p^{-1}(t)$.  If we
 need to do this for a large number of different points $t$, we can
 also speed up the method by precomputing various coefficients that
 depend only on the points $a_i$.
 \begin{checks}
  embedded/barycentric_check.mpl: check_barycentric()
 \end{checks}
\end{remark}

One could attempt to triangulate $EX^*$ using the points $v_i$ as
vertices, but it turns out that the resulting simplices are too large
for barycentric coordinates to work properly.  We have tried several
different triangulations with smaller simplices.  In one of these, we
introduce a new set of points $a_{ij}$ for $0\leq i\leq 6$ and $0\leq
j\leq 4$, and use them as vertices for a triangulation of $F_{16}$
with $48$ triangles. We can then use the group action to obtain a
triangulation of all of $EX^*$, with $768$ triangles.  It is
convenient to use points with good rationality properties, as
described in Section~\ref{sec-rational}.  This ensures that at least
the first steps of our calculations can be specified exactly.  It is
also convenient to choose points such that the various edges have
lengths that do not differ by too large a factor.  We use the
following corner points:
{\tiny \[ \hspace{-6em}
 \renewcommand{\arraystretch}{1.5} \begin{array}{llll}
  \left(\tfrac{1}{2}\sqrt{2},\tfrac{1}{2}\sqrt{2},0,0\right) &
  \left(0,1,0,0\right) & \left(0,0,1,0\right) &
  \left(0,0,\tfrac{1}{3}\sqrt{6},-\tfrac{1}{3}\sqrt{3}\right)
\end{array} \]}

Plus the following additional edge points:
{\tiny \[ \hspace{-6em} \renewcommand{\arraystretch}{1.5} \begin{array}{lll}
 \left(\tfrac{407}{745},\tfrac{624}{745},0,0\right)  &
 \left(\tfrac{5}{13},\tfrac{12}{13},0,0\right)  &
 \left(\tfrac{9}{41},\tfrac{40}{41},0,0\right)  \\
 \left(\tfrac{1}{60}\sqrt{195},0,\tfrac{1}{4}\sqrt{15},-\tfrac{1}{60}\sqrt{30}\right)  &
 \left(\tfrac{3}{70}\sqrt{55},0,\tfrac{2}{7}\sqrt{10},-\tfrac{9}{70}\sqrt{5}\right)  &
 \left(\tfrac{1}{68}\sqrt{238},0,\tfrac{1}{12}\sqrt{102},-\tfrac{7}{102}\sqrt{51}\right)
\end{array} \]}

{\tiny \[ \hspace{-6em} \renewcommand{\arraystretch}{1.5} \begin{array}{lllll}
 \left(\tfrac{23}{34},\tfrac{23}{34},\tfrac{7}{34}\sqrt{2},0\right)  &
 \left(\tfrac{79}{130},\tfrac{79}{130},\tfrac{47}{130}\sqrt{2},0\right)  &
 \left(\tfrac{1}{2},\tfrac{1}{2},\tfrac{1}{2}\sqrt{2},0\right)  &
 \left(\tfrac{79}{202},\tfrac{79}{202},\tfrac{119}{202}\sqrt{2},0\right)  &
 \left(\tfrac{7}{34},\tfrac{7}{34},\tfrac{23}{34}\sqrt{2},0\right)  \\
 \left(0,\tfrac{59}{62},\tfrac{11}{62}\sqrt{2},-\tfrac{11}{62}\right) &
 \left(0,\tfrac{11}{13},\tfrac{4}{13}\sqrt{2},-\tfrac{4}{13}\right)  &
 \left(0,\tfrac{61}{86},\tfrac{35}{86}\sqrt{2},-\tfrac{35}{86}\right) &
 \left(0,\tfrac{1}{2},\tfrac{1}{2}\sqrt{2},-\tfrac{1}{2}\right)  &
 \left(0,\tfrac{11}{38},\tfrac{21}{38}\sqrt{2},-\tfrac{21}{38}\right)
\end{array} \]}
\medskip

Plus the following points in the interior:
{\tiny \[ \hspace{-6em} \renewcommand{\arraystretch}{1.5} \begin{array}{lll}
 \left(\tfrac{92124}{152207},\tfrac{232883}{304414},\tfrac{31}{202}\sqrt{2},-\tfrac{11005}{304414}\right)  &
 \left(\tfrac{15}{38},\tfrac{86}{95},\tfrac{1}{10}\sqrt{2},-\tfrac{13}{190}\right)  &
 \left(\tfrac{4793}{21846},\tfrac{72098}{76461},\tfrac{23}{154}\sqrt{2},-\tfrac{20585}{152922}\right)  \\
 \left(\tfrac{1287}{2425},\tfrac{344}{485},\tfrac{8}{25}\sqrt{2},-\tfrac{248}{2425}\right)  &
 \left(\tfrac{214099}{533478},\tfrac{217940}{266739},\tfrac{73}{274}\sqrt{2},-\tfrac{91469}{533478}\right)  &
 \left(\tfrac{1121}{4510},\tfrac{1864}{2255},\tfrac{25}{82}\sqrt{2},-\tfrac{237}{902}\right)  \\
 \left(\tfrac{7267}{15458},\tfrac{8565}{15458},\tfrac{125}{262}\sqrt{2},-\tfrac{1000}{7729}\right)  &
 \left(\tfrac{37323}{100798},\tfrac{29390}{50399},\tfrac{95}{202}\sqrt{2},-\tfrac{28595}{100798}\right)  &
 \left(\tfrac{22828}{111879},\tfrac{1572959}{2461338},\tfrac{907}{2046}\sqrt{2},-\tfrac{975025}{2461338}\right)  \\
 \left(\tfrac{92619}{240218},\tfrac{224276}{600545},\tfrac{359}{610}\sqrt{2},-\tfrac{166217}{1201090}\right)  &
 \left(\tfrac{5389}{14982},\tfrac{27671}{74910},\tfrac{1301}{2270}\sqrt{2},-\tfrac{10408}{37455}\right)  &
 \left(\tfrac{342221}{1435238},\tfrac{4932765}{10046666},\tfrac{1549}{3038}\sqrt{2},-\tfrac{2143816}{5023333}\right)  \\
 \left(\tfrac{143}{486},\tfrac{298}{1701},\tfrac{83}{126}\sqrt{2},-\tfrac{415}{3402}\right)  &
 \left(\tfrac{42732}{137557},\tfrac{20435}{137557},\tfrac{280}{457}\sqrt{2},-\tfrac{49720}{137557}\right)  &
 \left(\tfrac{23}{119},\tfrac{4}{17},\tfrac{4}{7}\sqrt{2},-\tfrac{60}{119}\right)
\end{array} \]}

To create a triangulation of $F_{16}$, we link the above points according to the
following combinatorial scheme:
\begin{center}
 \begin{tikzpicture}
  \draw[cyan]    (0,0) -- (0,4);
  \draw[green]   (0,0) -- (6,0);
  \draw[blue]    (6,0) -- (6,4);
  \draw[magenta] (0,4) -- (6,4);
  \draw (1,0) -- (1,4);
  \draw (2,0) -- (2,4);
  \draw (3,0) -- (3,4);
  \draw (4,0) -- (4,4);
  \draw (5,0) -- (5,4);
  \draw (0,1) -- (6,1);
  \draw (0,2) -- (6,2);
  \draw (0,3) -- (6,3);
  \draw (0,1) -- (1,2);
  \draw (0,0) -- (2,2);
  \draw (1,0) -- (3,2);
  \draw (2,0) -- (3,1);
  \draw (0,3) -- (1,2);
  \draw (0,4) -- (2,2);
  \draw (1,4) -- (3,2);
  \draw (2,4) -- (3,3);
  \draw (6,1) -- (5,2);
  \draw (6,0) -- (4,2);
  \draw (5,0) -- (3,2);
  \draw (4,0) -- (3,1);
  \draw (6,3) -- (5,2);
  \draw (6,4) -- (4,2);
  \draw (5,4) -- (3,2);
  \draw (4,4) -- (3,3);
 \end{tikzpicture}
\end{center}
This has been arranged to ensure that no edge in the triangulation has
endpoints on two different sides of $F_{16}$, which would cause
trouble in certain numerical algorithms.  We have also used a finer
triangulation obtained by subdividing each of the above triangles in
the following pattern:
\begin{center}
 \begin{tikzpicture}[scale=2]
  \begin{scope}
   \fill (-0.50,0.00) circle(0.02);
   \fill ( 0.50,0.00) circle(0.02);
   \fill ( 0.00,0.86) circle(0.02);
   \draw (-0.50,0.00) -- (0.5,0) -- (0,0.86) -- cycle;
  \end{scope}
  \begin{scope}[xshift=1.5cm]
   \draw (0,0.43) -- (1,0.43);
   \draw (0.9,0.53) -- (1,0.43) -- (0.9,0.33);
  \end{scope}
  \begin{scope}[xshift=4cm]
   \fill (-0.50,0.00) circle(0.02);
   \fill ( 0.50,0.00) circle(0.02);
   \fill ( 0.00,0.86) circle(0.02);
   \fill (-0.25,0.43) circle(0.02);
   \fill ( 0.25,0.43) circle(0.02);
   \fill ( 0.00,0.00) circle(0.02);
   \draw (-0.50,0.00) -- (0.5,0) -- (0,0.86) -- cycle;
   \draw (-0.25,0.43) -- (0.25,0.43) -- (0,0) -- cycle;
  \end{scope}
 \end{tikzpicture}
\end{center}
This gives a grid with 192 faces in $F_{16}$.

Information about this kind of triangulation can be encoded in an
object of the class \mcode+E_grid+, which is declared in the file
\fname+embedded/E_domain.mpl+.  This class extends the \mcode+grid+
class, which is declared in \fname+domain/grid.mpl+.  (There used to
be parallel classes \mcode+H_grid+ and \mcode+P_grid+, but we found that
various algorithms based on triangulations were not very effective, so
we have not maintained that code.)  In particular, the 192 face
triangulation described above is encoded in this form and stored in
the file
\fname+embedded/roothalf/split_rational_grid_wx_30.mpl+ in the
\mcode+data+ directory.  One can
thus enter
\begin{mcodeblock}
   read(cat(data_dir,"/embedded/roothalf/split_rational_grid_wx_30.mpl"));
   G := eval(split_rational_grid_wx_30):
   G["num_points"];
\end{mcodeblock}
This will print $117$, indicating that the triangulation has $117$
vertices lying in $F_{16}$.  As well as the obvious data about the
vertices, edges and faces of the triangulation, the object \mcode+G+
also contains extensive information about $175$ sample points in each of
the $192$ faces.  This can be used for computing integrals over
$EX^*$, as will be explained in Section~\ref{sec-integration}.  All of
this information is computed to $100$ decimal places.  Because of
this, the file is rather large (about 53MB).

One can also regenerate the object \mcode+G+ using the function
\begin{mcodeblock}
   build_data["grid"]();
\end{mcodeblock}
defined in \fname+build_data.mpl+.  See
Section~\ref{sec-build} for more discussion of this framework. 

\subsection{Curvature and the Laplacian}
\lbl{sec-E-curvature}

We now discuss the Gaussian curvature of $EX(a)$.  Most treatments of
this invariant are formulated in terms of local coordinates on the
manifold, but for us it is more useful to have a formula in terms of
the coordinates $x_i$ for the ambient space $\R^4$.  We have not been
able to find a reference for the formula given below, although it
would be surprising if it did not appear somewhere.  Most of our work
will be valid for $EX(a)$ for all $a$, but we will focus on the case
$a=1/\rt$ for simplicity.

Let $n(x)\in\R^4$ be the gradient of $g$ at $x$, and let
$m(x)\in M_4(\R)$ be the Hessian, so
\begin{align*}
 n(x)_i &= \partial g(x)/\partial x_i \\
 m(x)_{ij} &= \partial^2 g(x)/\partial x_i\,\partial x_j.
\end{align*}
Explicitly, for $a=1/\rt$ we have

\begin{align*}
 n(x) &= \left[\begin{array}{c}
           -4 x_1 x_4+2\rt x_1 x_3 \\
           -4 x_2 x_4-2\rt x_2 x_3 \\
           2 x_3 x_4+\rt x_1^2-\rt x_2^2 \\
           -2 x_1^2-2 x_2^2+x_3^2-6 x_4^2
         \end{array}\right] \\
 m(x) &=  \left[ \begin {array}{cccc}
            -4 x_4+2\rt x_3&0&2\rt x_1&-4 x_1 \\
            0&-4 x_4-2\rt x_3&-2\rt x_2&-4 x_2\\
            2\rt x_1&-2\rt x_2&2 x_4&2 x_3 \\
            -4 x_1&-4 x_2&2 x_3&-12 x_4
          \end {array} \right] .
\end{align*}

Next, we let $\xi$ be the usual isomorphism $\Lm^3(\R^4)\to\R^4$,
given by
\begin{align*}
 \xi(e_1\wedge e_2\wedge e_3) &= \pp e_4 \\
 \xi(e_1\wedge e_2\wedge e_4) &=    -e_3 \\
 \xi(e_1\wedge e_3\wedge e_4) &= \pp e_2 \\
 \xi(e_2\wedge e_3\wedge e_4) &=    -e_1,
\end{align*}
so
\[ \xi^{-1}(u)\wedge v =
     \ip{u,v} e_1\wedge e_2\wedge e_3\wedge e_4.
\]
We then let $p(x)\in M_4(\R)$ denote the unique matrix such that
\[ p(x)y = \xi(x\wedge n(x)\wedge m(x)y) \]
for all $y\in\R^4$.

\begin{theorem}\lbl{thm-curvature}
 The Gaussian curvature of $EX^*$ at $x$ is
 $1+\text{trace}(p(x)^2)/\|n(x)\|^2$.
\end{theorem}

The proof will follow after some preliminaries.

\begin{definition}
 We put
 \[ g_{ijk} = \frac{1}{6}
     \frac{\partial^3g(x)}{\partial x_i\partial x_j\partial x_k}.
 \]
 It is clear that $g_{ijk}$ is invariant under permutations of the
 three indices.  As $g$ is a homogeneous cubic, we see that $g_{ijk}$
 is constant, and that $g(x)=\sum_{ijk}g_{ijk}x_ix_jx_k$.  By
 differentiating this we obtain $n(x)_p=3\sum_{j,k}g_{pjk}x_jx_k$ and
 $m(x)_{pq}=6\sum_kg_{pqk}x_k$.
\end{definition}

\begin{remark}
 Some of our calculations in this section are most easily understood
 in terms of Penrose diagrams.  The paper~\cite{jost:gtci} is a good
 reference for these from a mathematical perspective.  The above
 expressions for $g(x)$, $n(x)$ and $m(x)$ can be expressed
 graphically as follows:
 \begin{center}
  \begin{tikzpicture}[scale=3]
   \begin{scope}
    \draw[red] (-30:0.2) -- (90:0.2) -- (210:0.2) -- cycle;
    \draw (0,0) node {$g$};
    \draw (-30:0.2) -- (-30:0.6);
    \draw ( 90:0.2) -- ( 90:0.6);
    \draw (210:0.2) -- (210:0.6);
    \draw[blue] (-30:0.7) circle(0.1);
    \draw[blue] ( 90:0.7) circle(0.1);
    \draw[blue] (210:0.7) circle(0.1);
    \draw (-30:0.7) node {$x$};
    \draw ( 90:0.7) node {$x$};
    \draw (210:0.7) node {$x$};
    \draw (270:0.7) node {$g(x)$};
   \end{scope}
   \begin{scope}[xshift=2cm]
    \draw[red] (-30:0.2) -- (90:0.2) -- (210:0.2) -- cycle;
    \draw (0,0) node {$g$};
    \draw (-30:0.2) -- (-30:0.6);
    \draw ( 90:0.2) -- ( 90:0.6);
    \draw (210:0.2) -- (210:0.6);
    \draw[blue] (-30:0.7) circle(0.1);
    \draw[blue] ( 90:0.7) circle(0.1);
    \draw[blue] (210:0.7) circle(0.1);
    \draw (-30:0.7) node {$x$};
    \draw ( 90:0.7) node {$u$};
    \draw (210:0.7) node {$x$};
    \draw (270:0.7) node {$n(x).u/3$};
   \end{scope}
   \begin{scope}[xshift=4cm]
    \draw[red] (-30:0.2) -- (90:0.2) -- (210:0.2) -- cycle;
    \draw (0,0) node {$g$};
    \draw (-30:0.2) -- (-30:0.6);
    \draw ( 90:0.2) -- ( 90:0.6);
    \draw (210:0.2) -- (210:0.6);
    \draw[blue] (-30:0.7) circle(0.1);
    \draw[blue] ( 90:0.7) circle(0.1);
    \draw[blue] (210:0.7) circle(0.1);
    \draw (-30:0.7) node {$x$};
    \draw ( 90:0.7) node {$u$};
    \draw (210:0.7) node {$v$};
    \draw (270:0.7) node {$(m(x)u).v/6$};
   \end{scope}
  \end{tikzpicture}
 \end{center}
\end{remark}

We next need a little exterior algebra.  To keep everything straight,
we need to spell out some conventions.
\begin{definition}\lbl{defn-exterior}
 For any vector space $V$, we will identify $\lm^k(V)$ with a subspace
 of $V^{\ot k}$ in such a way that $v_1\wedge\dotsb\wedge v_k$ becomes
 \[ \sum_{\sg\in\Sg_k}
     \ep(\sg) v_{\sg(1)}\ot\dotsb\ot v_{\sg(k)}.
 \]
 If $V$ has an inner product, we give $V^{\ot k}$ the inner product
 such that
 \[ \ip{v_1\ot\dotsb\ot v_k,\;w_1\ot\dotsb\ot w_k} =
     \prod_i \ip{v_i,w_i}.
 \]
 However, we give the subspace $\lm^k(V)$ the alternative inner
 product $\ip{\al,\bt}'=\ip{\al,\bt}/k!$; this has the property that
 \[ \ip{v_1\wedge\dotsb\wedge v_k,\;w_1\wedge\dotsb\wedge w_k}' \]
 is the determinant of the matrix of inner products $\ip{v_i,w_j}$.

 Now suppose that $V$ has dimension $d$ and we have a given volume
 form $\om\in\lm^d(V)$ with $\ip{\om,\om}'=1$.  We define the Hodge
 operator $*\:\lm^k(V)\to\lm^{d-k}(V)$ by the property
 \[ \al\wedge *\bt = \ip{\al,\bt}'\om. \]
\end{definition}

\begin{lemma}\lbl{lem-hodge}
 Suppose that $x\in EX(a)$, and let $(u,v)$ be any oriented
 orthonormal basis for $T_xEX^*$.  Let $\ep$ be the usual totally
 antisymmetric tensor:
 \[ \ep_{ijkl} = \begin{cases}
     +1 & \text{ if $(i,j,k,l)$ is an even permutation of $(1,2,3,4)$} \\
     -1 & \text{ if $(i,j,k,l)$ is an odd  permutation of $(1,2,3,4)$} \\
      0 & \text{ if $(i,j,k,l)$ is not a   permutation of $(1,2,3,4)$}.
    \end{cases}
 \]
 Then
 \[ u\wedge v = u\ot v - v\ot u =
     \frac{1}{2\|n\|}\sum_{ijkl}\ep_{ijkl}x_in(x)_je_k\wedge e_l.
 \]
\end{lemma}
\begin{proof}
 This is standard, except that we need to check that our conventions
 give the indicated factor of two.

 The vectors $x$ and $w=n(x)/\|n(x)\|$ form an orthonormal basis for
 $T_xEX(a)^\perp$, and it follows that the map
 $\phi\:\al\mapsto x\wedge w\wedge\al$ gives an isomorphism
 $\lm^2(T_xEX(a))\to\lm^4(\R^4)=\R.\om_4$.  As $(u,v)$ is an oriented
 orthonormal basis for $T_xEX(a)$, we have $\phi(u\wedge v)=\om_4$.
 Now put
 \[ \tht =
     \frac{1}{2\|n\|}\sum_{ijkl}\ep_{ijkl}x_in(x)_je_k\wedge e_l =
     \frac{1}{2}\sum_{ijkl}\ep_{ijkl}x_iw_je_k\wedge e_l.
 \]
 From the definitions we have
 \[ x\wedge w\wedge e_k\wedge e_l =
     \sum_{p,q} \ep_{pqkl} x_pw_q\om_4.
 \]
 This gives
 \[ \phi(\tht) = x\wedge w\wedge\tht =
     \frac{1}{2}\sum_{ijklpq}
      \ep_{ijkl}\ep_{pqkl} x_ix_pw_jw_q\om_4.
 \]
 One can also check that
 \[ \sum_{kl} \ep_{ijkl}\ep_{pqkl} =
     2\dl_{ip}\dl_{jq} - 2\dl_{iq}\dl_{jp}.
 \]
 (For example, if $i=p=1$ and $j=q=2$ then the terms for $(k,l)=(3,4)$
 and $(k,l)=(4,3)$ both contribute $+1$, and all other terms are zero
 so the sum is $+2$.  On the other hand, if $i=q=1$ and $j=p=2$ then
 the terms for $(k,l)=(3,4)$ and $(k,l)=(4,3)$ both contribute $-1$,
 and all other terms are zero so the sum is $-2$.)  This gives
 \begin{align*}
  \phi(\tht) &=
    \sum_{ijpq}
      (\dl_{ip}\dl_{jq} - \dl_{iq}\dl_{jp}) x_ix_pw_jw_q\om_4 \\
  &= \left(\sum_{ij} x_i^2w_j^2 - \sum_{ij} x_iw_ix_jw_j\right)\om_4 \\
  &= (\|x\|^2\|w\|^2 - \ip{x,w}^2)\om_4 = \om_4.
 \end{align*}
 As $\phi$ is an isomorphism, this means that $u\wedge v=\tht$.
\end{proof}

\begin{proof}[Proof of Theorem~\ref{thm-curvature}]
 Put $r=\|n(x)\|$.  Let $N$ be the span of $x$ and $n(x)$, so the
 tangent space to $EX^*$ at $x$ is the space $T=N^\perp$.  Define a
 quadratic map $\psi_0\:T\to N$ by
 \[ \psi_0(t) = (1 - \|t\|^2/2)x - (n(t).x)n(x)/r^2, \]
 and put
 \[ X' = \text{ graph of } \psi_0 = \{t+\psi_0(t)\st t\in T\}. \]
 Proposition~\ref{prop-quadratic-chart} tells us that $X'$ agrees with
 $EX^*$ to second order near $x$.  As curvature depends only on second
 derivatives, the curvature of $EX^*$ at $x$ is the same as that of
 $X'$.

 Now choose an orthonormal basis $(u,v)$ for $T$, oriented so that
 $x\wedge n(x)\wedge u\wedge v$ is a positive multiple of
 $e_1\wedge e_2\wedge e_3\wedge e_4$.  Define $\psi\:\R^2\to N$ by
 $\psi(a,b)=\psi_0(au+bv)$.  It will be enough to calculate the
 curvature of the graph of $\psi$ at the point where $a=b=0$.

 There is a well-known formula for the curvature of the graph of a
 function $\psi(a,b)$ from $\R^2$ to $\R$ (rather than $N$): we put
 \begin{align*}
  E &= 1 + \frac{\partial \psi}{\partial a} . \frac{\partial \psi}{\partial a} &
  F &= \frac{\partial \psi}{\partial a} . \frac{\partial \psi}{\partial b} &
  G &= 1 + \frac{\partial \psi}{\partial b} . \frac{\partial \psi}{\partial b} \\
  L &= \frac{\partial^2\psi}{\partial^2 a} &
  M &= \frac{\partial^2\psi}{\partial a\;\partial b} &
  N &= \frac{\partial^2\psi}{\partial^2 b},
 \end{align*}
 where everything is evaluated at $(0,0)$.  The curvature is then
 \[ K = \frac{LN-M^2}{EG-F^2}. \]
 Essentially the same argument works in our case, except that now $L$,
 $M$ and $N$ are vectors, and the formula is
 \[ K = \frac{L.N-M.M}{EG-F^2}. \]
 In our case $\psi$ is constant plus quadratic and so the first order
 derivatives vanish at $(a,b)=(0,0)$, which gives $E=G=1$ and $F=0$,
 so $K=L.N-M.M$.

 Next, using the fact that $\|t\|^2$ and $n(t)$ are homogeneous
 quadratic functions of $t$, we see that
 \begin{align*}
  L &= -x - 2(x.n(u))n(x)/r^2 \\
  N &= -x - 2(x.n(v))n(x)/r^2 \\
  M &= -(x.m(u)v)n(x)/r^2.
 \end{align*}
 Here $x.n(x)=3g(x)=0$, and $n(u)$ can be written as $m(u)u/2$, and
 similarly for $v$.  This gives
 \begin{align*}
  K &= L.N - M.M \\
    &= x.x +
       \frac{n(x).n(x)}{r^4}
       \left(4\;x.n(u)\;x.n(v)-(x.m(u)v)^2\right) \\
    &= 1 + \frac{1}{r^2}
           \left((x.m(u)u)(x.m(v)v)-(x.m(u)v)^2\right).
 \end{align*}
 We put $P=(x.m(u)u)(x.m(v)v)-(x.m(u)v)^2$ so that $K=1+P/r^2$.  Note
 that $P/36$ can be expressed as the following difference of Penrose
 diagrams:
 \begin{center}
  \begin{tikzpicture}[scale=2]
   \draw (-0.95,-0.7) -- (-1,-0.7) -- (-1,1) -- (-0.95,1);
   \begin{scope}
    \draw[red] (-30:0.2) -- (90:0.2) -- (210:0.2) -- cycle;
    \draw (0,0) node {$g$};
    \draw (-30:0.2) -- (-30:0.6);
    \draw ( 90:0.2) -- ( 90:0.6);
    \draw (210:0.2) -- (210:0.6);
    \draw[blue] (-30:0.7) circle(0.1);
    \draw[blue] ( 90:0.7) circle(0.1);
    \draw[blue] (210:0.7) circle(0.1);
    \draw ( 90:0.7) node {$x$};
    \draw (-30:0.7) node {$u$};
    \draw (210:0.7) node {$u$};
   \end{scope}
   \begin{scope}[xshift=1.8cm]
    \draw[red] (-30:0.2) -- (90:0.2) -- (210:0.2) -- cycle;
    \draw (0,0) node {$g$};
    \draw (-30:0.2) -- (-30:0.6);
    \draw ( 90:0.2) -- ( 90:0.6);
    \draw (210:0.2) -- (210:0.6);
    \draw[blue] (-30:0.7) circle(0.1);
    \draw[blue] ( 90:0.7) circle(0.1);
    \draw[blue] (210:0.7) circle(0.1);
    \draw ( 90:0.7) node {$x$};
    \draw (-30:0.7) node {$v$};
    \draw (210:0.7) node {$v$};
   \end{scope}
   \draw (2.75,-0.7) -- (2.8,-0.7) -- (2.8,1) -- (2.75,1);
   \draw (3.1,0) node{$-$};
   \draw (3.55,-0.7) -- (3.5,-0.7) -- (3.5,1) -- (3.55,1);
   \begin{scope}[xshift=4.4cm]
    \draw[red] (-30:0.2) -- (90:0.2) -- (210:0.2) -- cycle;
    \draw (0,0) node {$g$};
    \draw (-30:0.2) -- (-30:0.6);
    \draw ( 90:0.2) -- ( 90:0.6);
    \draw (210:0.2) -- (210:0.6);
    \draw[blue] (-30:0.7) circle(0.1);
    \draw[blue] ( 90:0.7) circle(0.1);
    \draw[blue] (210:0.7) circle(0.1);
    \draw ( 90:0.7) node {$x$};
    \draw (210:0.7) node {$u$};
    \draw (-30:0.7) node {$v$};
   \end{scope}
   \begin{scope}[xshift=6.2cm]
    \draw[red] (-30:0.2) -- (90:0.2) -- (210:0.2) -- cycle;
    \draw (0,0) node {$g$};
    \draw (-30:0.2) -- (-30:0.6);
    \draw ( 90:0.2) -- ( 90:0.6);
    \draw (210:0.2) -- (210:0.6);
    \draw[blue] (-30:0.7) circle(0.1);
    \draw[blue] ( 90:0.7) circle(0.1);
    \draw[blue] (210:0.7) circle(0.1);
    \draw ( 90:0.7) node {$x$};
    \draw (210:0.7) node {$u$};
    \draw (-30:0.7) node {$v$};
   \end{scope}
   \draw (7.15,-0.7) -- (7.2,-0.7) -- (7.2,1) -- (7.15,1);
  \end{tikzpicture}
 \end{center}

 We now consider various elements of the space $(\R^4)^{\ot 4}$.  We
 give this the obvious inner product so that elements of the form
 $e_i\ot e_j\ot e_k\ot e_l$ give an orthonormal basis.  Any
 permutation $\sg\in\Sg_4$ gives an automorphism $\al_\sg$ of
 $(\R^4)^{\ot 4}$ by permuting the tensor factors; for example, we have
 \[ \al_{(2;3)}.(u_1\ot u_2\ot u_3\ot u_4) =
         u_1\ot u_3\ot u_2\ot u_4.
 \]
 We put
 \begin{align*}
  A &= \sum_ix_ig_{ijk}e_j\ot e_k\in(\R^4)^{\ot 2} \\
  B &= u\ot v - v\ot u\in(\R^4)^{\ot 2} \\
  C &= u\ot u\ot v\ot v-u\ot v\ot u\ot v \in (\R^4)^{\ot 4}.
 \end{align*}
 It is now not hard to see that $P=36\ip{A\ot A,C}$.  We claim that
 also
 \[ P = 18\ip{A\ot A,\al_{(2\;3)}(B\ot B)}
      = 18\ip{\al_{(2\;3)}(A\ot A),B\ot B}.
 \]
 Indeed, we have
 \begin{align*}
  \al_{(2\;3)}(B\ot B) &=
   \al_{(2\;3)}(u\ot v\ot u\ot v -
                u\ot v\ot v\ot u -
                v\ot u\ot u\ot v +
                v\ot u\ot v\ot u) \\
   &= u\ot u\ot v\ot v -
      u\ot v\ot v\ot u -
      v\ot u\ot u\ot v +
      v\ot v\ot u\ot u.
 \end{align*}
 Now $A\ot A$ is invariant under the permutations $(1\;2)$, $(3\;4)$
 and $(1\;3)(2\;4)$.  Using this we see that
 \begin{align*}
  \ip{A\ot A,u\ot v\ot v\ot u} &= \ip{A\ot A,u\ot v\ot u\ot v} \\
  \ip{A\ot A,v\ot u\ot u\ot v} &= \ip{A\ot A,u\ot v\ot u\ot v} \\
  \ip{A\ot A,v\ot v\ot u\ot u} &= \ip{A\ot A,u\ot u\ot v\ot v}.
 \end{align*}
 The claim follows easily from this.  It can be rewritten as
 \[ P = 18\sum_{i,j,k,l} A_{ij}A_{kl}B_{ik}B_{jl}. \]
 We now define a matrix $Q$ with $Q_{jk}=\sum_iA_{ij}B_{ik}$.  After
 noting that $A_{kl}=A_{lk}$ and $B_{jl}=-B_{lj}$, the
 above expression can be rewritten again as
 \[ P=-18\sum_{j,k}Q_{jk}Q_{kj} = -18\text{trace}(Q^2). \]
 Some of the above can be represented graphically as follows:
 \begin{center}
  \begin{tikzpicture}[scale=2]
   \begin{scope}
    \draw[green] (-1.1, 0.8) circle(0.13);
    \draw[green] ( 1.1, 0.8) circle(0.13);
    \draw[magenta] (-1.25,-0.65) -- (-1.25,-0.95) -- (-0.95,-0.8) -- cycle;
    \draw[magenta] ( 0.95,-0.65) -- ( 0.95,-0.95) -- ( 1.25,-0.8) -- cycle;
    \draw[rounded corners] (-1.23,0.8) -- (-1.5,0.8) -- (-1.5,-0.8) -- (-1.25,-0.8);
    \draw[rounded corners] ( 1.23,0.8) -- ( 1.5,0.8) -- ( 1.5,-0.8) -- ( 1.25,-0.8);
    \draw[rounded corners] (-0.97, 0.8) -- (-0.8,0.8) -- (0.8,-0.8) -- (0.95,-0.8);
    \draw[rounded corners] (-0.95,-0.8) -- (-0.8,-0.8) -- (-0.1,-0.1);
    \draw[rounded corners] (0.1,0.1) -- (0.8,0.8) -- (0.97,0.8);
    \draw(-1.10, 0.8) node {$A$};
    \draw( 1.10, 0.8) node {$A$};
    \draw(-1.16,-0.8) node {$B$};
    \draw( 1.04,-0.8) node {$B$};
    \draw(0,-1.1) node{$P/18$};
   \end{scope}
   \begin{scope}[xshift=5cm]
    \draw (-0.85,0) -- (-0.53,0);
    \draw[green] (-0.4, 0.0) circle(0.13);
    \draw (-0.27,0) -- (0.25,0);
    \draw[magenta] ( 0.25,0.15) -- ( 0.25,-0.15) -- (0.55,0.0) -- cycle;
    \draw (0.55,0) -- (0.85,0);
    \draw(-0.40, 0.0) node {$A$};
    \draw( 0.34, 0.0) node {$B$};
    \draw(0,-1.1) node{$Q$};
   \end{scope}
  \end{tikzpicture}
 \end{center}
 Note that our choice of symbols reflects the fact that $A$ is
 symmetric and $B$ is not.

 Now put
 \[ D = x\wedge n(x) = x\ot n(x) - n(x)\ot x, \]
 so
 \[ D_{ij} = 3\sum_{k,l} (g_{jkl} x_ix_kx_l - g_{ikl} x_jx_kx_l).
 \]
 Lemma~\ref{lem-hodge} tells us that
 \[ B_{ik} = \sum_{l,m}\ep_{iklm}D_{lm}/(2r). \]

 Now
 \[ Q_{jk} = \sum_iA_{ij}B_{ik}
     = \frac{1}{2r}\sum_{i,l,m} \ep_{iklm}A_{ij}D_{lm}
     = \frac{3}{2r}\sum_{h,i,l,m,n,p}
        \ep_{iklm} g_{hij}x_h(g_{lnp}x_mx_nx_p-g_{mnp}x_lx_nx_p).
 \]
 Note that the two terms in brackets are essentially the same except
 that $l$ and $m$ are exchanged, but $\ep_{iklm}$ is also antisymmetric
 in $l$ and $m$.  We can thus drop one of the terms and introduce a
 factor of $2$ giving
 \begin{align*}
  Q_{jk}
   &= \frac{3}{r}\sum_{h,i,l,m,n,p}
       \ep_{iklm} g_{hij}g_{lnp}x_hx_mx_nx_p \\
   &= \frac{1}{6r}\sum_{i,l,m}\ep_{iklm}n(x)_l\;m(x)_{ij}
    = \pm p(x)_{jk}/(6r).
 \end{align*}
 Equivalently, we have the following Penrose diagram for $rQ/3$:
 \begin{center}
  \begin{tikzpicture}[scale=2]
   \draw[rounded corners] (-0.9,0) -- (0,0) -- (0,0.4);
   \draw[red] (0,0.4) -- (0,0.7) -- (-0.27,0.55) -- cycle;
   \draw (-0.09,0.55) node {$g$};
   \draw (-0.27,0.55) -- (-0.57,0.55);
   \draw[blue] (-0.67,0.55) circle(0.1);
   \draw (-0.67,0.55) node {$x$};
   \draw(0,0.7) -- (0,1);
   \draw(-0.2,1) -- (2.0,1) -- (2.0,1.5) -- (-0.2,1.5) -- cycle;
   \draw (0.9,1.25) node{$\epsilon$};
   \draw[rounded corners] (2.6,0) -- (0.5,0) -- (0.5,1);
   \draw[blue] (1.0,0.3) circle(0.1);
   \draw (1.0,0.3) node {$x$};
   \draw[blue] (1.3,0.3) circle(0.1);
   \draw (1.3,0.3) node {$x$};
   \draw[blue] (1.8,0.3) circle(0.1);
   \draw (1.8,0.3) node {$x$};
   \draw[red] (1.0,0.6) -- (1.3,0.6) -- (1.15,0.87) -- cycle;
   \draw (1.15,0.7) node {$g$};
   \draw (1.0,0.4) -- (1.0,0.6);
   \draw (1.3,0.4) -- (1.3,0.6);
   \draw (1.15,0.87) -- (1.15,1);
   \draw (1.8,0.4) -- (1.8,1);
  \end{tikzpicture}
 \end{center}
 We now have
 \[ P=-18\text{trace}(Q^2)=-\text{trace}(p(x)^2)/(2r^2), \]
 and so
 \[ K = 1+P/r^2 = 1 - \text{trace}(p(x)^2)/(2r^4). \]
 \begin{checks}
  embedded/curvature_check.mpl: check_EX_curvature()
 \end{checks}
\end{proof}

\begin{remark}\lbl{rem-curvature-z}
 The full formula for $K$ in terms of the variables $x_i$ is too large
 to be given here.  However, $K$ is invariant under the group action,
 and so can be expressed in terms of the functions $z_1$ and $z_2$
 from Section~\ref{sec-E-functions}.  Even that expression is somewhat
 unwieldy for general $a$, but when $a=1/\rt$ one can check that
 the formula is as follows:
 \[ K = 1+ 8 \frac{2z_2-1}{(2-z_1)^2(1+z_2)^2}. \]
\end{remark}

It will also be useful for us to have an expression for the
Laplace-Beltrami operator $\Dl$ on $EX^*$ in terms of the ambient
coordinates $x_i$.

\begin{definition}
 As before, we write $r=\|n(x)\|$, so
 \[ r^2 = n(x).n(x) = 9\sum_{ijklm}g_{ijm}g_{klm}x_ix_jx_kx_l. \]
 We also define
 \begin{align*}
  r' &= \text{trace}(m(x)) = 6\sum_{i,j} g_{ijj}x_i  \\
  r'' &= n(x)^T m(x) n(x)
       = 54 \sum_{ijklmnp} g_{ijk}g_{klm}g_{mnp}x_ix_jx_lx_nx_p
 \end{align*}
\end{definition}

\begin{definition}\lbl{defn-Delta-prime}
 Let $U'$ be an open subset of $\R^4\sm\{0\}$.  We define a
 differential operator $\Dl'\:C^\infty(U')\to C^\infty(U')$ by
 \[ \Dl'(p) =
    \sum_i \frac{\partial^2p}{\partial x_i^2}
    - \sum_{i,j} x_ix_j \frac{\partial^2p}{\partial x_i\partial x_j}
    - \frac{1}{r^2} \sum_{i,j} n(x)_i n(x)_j
         \frac{\partial^2p}{\partial x_i\partial x_j}
    - 2\sum_i x_i\frac{\partial p}{\partial x_i}
    + \left(\frac{r''}{r^4}-\frac{r'}{r^2}\right)
       \sum_i n(x)_i\frac{\partial p}{\partial x_i}
 \]
\end{definition}

\begin{proposition}\lbl{prop-Delta-prime}
 If we put $U=U'\cap EX^*$, then the following diagram commutes:
 \[ \xymatrix{
  C^\infty(U') \ar[r]^{\Dl'} \ar[d]_{\text{res}} &
  C^\infty(U') \ar[d]^{\text{res}} \\
  C^\infty(U) \ar[r]_\Dl &
  C^\infty(U).
 } \]
\end{proposition}
\begin{proof}
 Consider a smooth function $p\in C^\infty(U')$ and a point $x\in EX^*$.
 Choose a chart $\phi\:\R^2\to EX^*$ with $\phi(0,0)=x$.  We will write
 $a$ and $b$ for coordinates on $\R^2$.  The chart gives a Riemannian
 metric on $\R^2$, corresponding to the matrix $M=\bbm E&F\\ F&G\ebm$,
 where
 \begin{align*}
  E &= \frac{\partial \phi}{\partial a} . \frac{\partial \phi}{\partial a} &
  F &= \frac{\partial \phi}{\partial a} . \frac{\partial \phi}{\partial b} &
  G &= \frac{\partial \phi}{\partial b} . \frac{\partial \phi}{\partial b}.
 \end{align*}
 We will use the standard formula
 \[ (\Delta p)\circ\phi =
     \det(M)^{-1/2}
      \text{div}\left(\det(M)^{1/2} M^{-1}\text{grad}(p\circ\phi)\right).
 \]
 We will only be using this formula at the point $(a,b)=(0,0)$, so it
 will be harmless to replace $M$ by an approximation involving only
 terms that are constant or linear in $a$ and $b$.  Similarly, $\phi$
 need not be an exact chart, so long as it is quadratically close to
 $EX^*$.  We can thus follow Proposition~\ref{prop-quadratic-chart} and
 define $\phi$ as below:
 \begin{align*}
  \psi_0(t) &= (1-\|t\|^2/2)x - (n(t).x)n(x)/r^2 &
  \phi_0(t) &= t + \psi_0(t) \\
  \psi(a,b) &= \psi_0(au+bv) &
  \phi(a,b) &= \phi_0(au+bv).
 \end{align*}
 By routine calculation we have
 \begin{align*}
  \frac{\partial\phi}{\partial a} &=
   u - ax - (Aa+Cb) n(x) \\
  \frac{\partial\phi}{\partial b} &=
   v - bx - (Bb+Ca) n(x)
 \end{align*}
 where
 \begin{align*}
  A &= \frac{6}{r^2} \sum_{ijk} g_{ijk}x_iu_ju_k \\
  B &= \frac{6}{r^2} \sum_{ijk} g_{ijk}x_iv_jv_k \\
  C &= \frac{6}{r^2} \sum_{ijk} g_{ijk}x_iu_jv_k.
 \end{align*}
 Because $x$, $n(x)$, $u$ and $v$ are orthogonal, it follows that we
 have $E=G=1$ and $F=0$ up to quadratic corrections.  We may thus take
 $M$ to be the identity matrix, and deduce that $(\Delta\ p)(x)$ is
 the value at $(0,0)$ of
 $\left(\frac{\partial^2}{\partial a^2}+
        \frac{\partial^2}{\partial b^2}\right)(p\circ\phi)$.  After
 using the chain rule twice, we see that this is the same as $P+Q$,
 where
 \begin{align*}
  P &= \sum_{i,j} \frac{\partial^2p}{\partial x_i\,\partial x_j}
     \left(\frac{\partial\phi_i}{\partial a}
           \frac{\partial\phi_j}{\partial a}+
           \frac{\partial\phi_i}{\partial b}
           \frac{\partial\phi_j}{\partial b}\right) \\
  Q &=  \sum_{i} \frac{\partial p}{\partial x_i}
     \left(\frac{\partial^2\phi_i}{\partial a^2}+
           \frac{\partial^2\phi_i}{\partial b^2}\right).
 \end{align*}
 Previously we recorded formulae for
 $\partial\phi/\partial a$ and $\partial\phi/\partial b$; when $a=b=0$
 they just give $u$ and $v$.  Thus, we have
 \[ P = \sum_{i,j} \frac{\partial^2\phi}{\partial x_i\,\partial x_j}
     \left(u_iu_j+v_iv_j\right).
 \]
 Now, the numbers $u_iu_j+v_iv_j$ are the matrix entries for the
 orthogonal projection onto the tangent space $T$, which is the
 identity minus the projection onto the normal space
 $N=T^\perp=\text{span}(x,n(x))$.  We thus have
 \[ u_iu_j+v_iv_j=\dl_{ij} - x_ix_j - n(x)_in(x)_j/r^2. \]
 Using this we obtain
 \[ P = \sum_i \frac{\partial^2p}{\partial x_i^2}
    - \sum_{i,j} x_ix_j \frac{\partial^2p}{\partial x_i\partial x_j}
    - \frac{1}{r^2} \sum_{i,j} n(x)_i n(x)_j
         \frac{\partial^2p}{\partial x_i\partial x_j}.
 \]
 This accounts for the first three terms in $\Dl'(p)$.

 We now turn to $Q$.  We have
 \begin{align*}
  \frac{\partial^2\phi}{\partial a^2} &= - x - A\,n(x) \\
  \frac{\partial^2\phi}{\partial b^2} &= - x - B\,n(x)
 \end{align*}
 and it follows that
 \[ Q = (-2x-(A+B)n(x)) . \nabla(p). \]
 Now
 \[ A+B = \frac{6}{r^2} \sum_{i,j,k} g_{ijk}x_i(u_ju_k+v_jv_k), \]
 and we can again use the relation
 $u_ju_k+v_jv_k=\dl_{jk}-x_jx_k-n(x)_jn(x)_k/r^2$ to eliminate $u$ and
 $v$, giving
 \[ A+B = \frac{6}{r^2} \left(
     \sum_{i,j} g_{ijj}x_i -
     \sum_{i,j,k} g_{ijk}x_ix_jx_k -
     \frac{1}{r^2} \sum_{i,j,k} g_{ijk}x_in(x)_jn(x)_k
    \right).
 \]
 Here $\sum_{i,j,k} g_{ijk}x_ix_jx_k$ is $g(x)$, which is zero
 because $x\in EX^*$.  We also have
 \begin{align*}
  \sum_{i,j} g_{ijj}x_i &= r'/6 \\
  \sum_{i,j,k} g_{ijk}x_in(x)_jn(x)_k &= r''/6.
 \end{align*}
 Putting this together gives
 \[ A+B = \frac{r'}{r^2} - \frac{r''}{r^4}, \]
 and this identifies $Q$ with the remaining terms in $\Dl'(p)$.
 \begin{checks}
  embedded/geometry_check.mpl: check_laplacian_a()
  embedded/geometry_check.mpl: check_laplacian_b()
  embedded/geometry_check.mpl: check_laplacian_z()
 \end{checks}
\end{proof}

\begin{remark}
 The uniformization theorem for Riemann surfaces says that $EX^*$ is
 conformally equivalent to the quotient of the open unit disc by a
 discrete group of automorphisms that preserve the hyperbolic metric.
 Using this, we deduce that the original metric on $EX^*$ can be
 conformally rescaled so that the new metric has constant curvature
 $-1$.  Let $g$ denote the original metric, with curvature $K$, and
 consider a rescaled metric $g^*=e^{2f}g$.  It is then known that the
 corresponding curvature is $K^*=(K-\Dl(f))/e^{2f}$.  We therefore
 want to find a $G$-invariant function $f$ such that
 $K=\Dl(f)-e^{2f}$.  There is a lot of freedom to do this locally, but
 not globally.  Thus, the most natural approach is to try to minimize
 $\int_{EX^*}(1+(K-\Dl(f))/e^{2f})^2$ as $f$ ranges over some
 finite-dimensional space of invariant functions.  To carry this
 forward, we need some theory of integration on $EX^*$, which will be
 treated in the next section.
\end{remark}

\section{The surface \texorpdfstring{$EX^*$}{EX*}}
\lbl{sec-roothalf}

We now focus on the surface $EX^*=EX(1/\rt)$.  We start by recording
explicitly a number of formulae that are obtained by substituting
$a=1/\rt$ in the results of Section~\ref{sec-E}.  Maple
notation for all these things is obtained by appending the character
\mcode+0+ to the corresponding notation in Section~\ref{sec-E}:
the points $v_i$ are \mcode+v_E0[i]+, the curves $c_j(t)$ are
\mcode+c_E0[j](t)+ and so on.

We have
\[ EX^* = \{x\in \R^4\st \rho(x)=1,\;g(x)=0\}
        = \{x\in \R^4\st \rho(x)=1,\;g_0(x)=0\},
\]
where
\begin{align*}
 g(x)   &= x_3^2x_4-2x_4^3-(2(x_1^2+x_2^2))x_4+\rt(x_1^2-x_2^2)x_3 \\
 g_0(x) &= (3x_3^2-2)x_4+\rt(x_1^2-x_2^2)x_3.
\end{align*}
The gradient of $g(x)$ is
\[ n(x) = \left(
     2x_1(\rt x_3-2x_4),\;
    -2x_2(\rt x_3+2x_4),\;
     2x_3x_4+\rt(x_1^2-x_2^2),\;
     -2x_1^2-2x_2^2+x_3^2-6x_4^2
   \right).
\]

We put $y_1=x_3$ and $y_2=(x_2^2-x_1^2)/\rt-\tfrac{3}{2}x_3x_4$ and
$z_i=y_i^2$.  We find that
\begin{align*}
 x_1^2 &= u_1= \half(1-\rt y_2)(1-y_1^2(1-y_2/\rt)) \\
 x_2^2 &= u_2= \half(1+\rt y_2)(1-y_1^2(1+y_2/\rt)) \\
 4x_1^2x_2^2 &= u_3 = (1-2z_2)((1-z_1)^2-z_1^2z_2/2) \\
 x_1^2+x_2^2 &= u_4 = 1-z_1-z_1z_2.
\end{align*}

The ring of $G$-invariant polynomial functions on $EX^*$ is $\R[z_1,z_2]$.
In particular, we have
\[ \|n(x)\|^2 = 4 (1-z_1/2)^2(1+z_2), \]
and the curvature is
\[ K = 1 - 2 \frac{1-2z_2}{(1-z_1/2)^2(1+z_2)^2}. \]
We can reduce polynomials to normal form using the functions
\mcode+NF_x0+, \mcode+NF_y0+ and \mcode+NF_z0+.  There are also
functions \mcode+FNF_y0+ and \mcode+FNF_z0+ which deal intelligently
with rational functions as well as polynomials.

The isotropy points are
\begin{align*}
 v_{ 0} &= (\pp 0,\pp 0,\pp 1,\pp 0) &
 v_{ 6} &= (\pp 1,\pp 1,\pp 0,\pp 0)/\rt \\
 v_{ 1} &= (\pp 0,\pp 0,   -1,\pp 0) &
 v_{ 7} &= (   -1,\pp 1,\pp 0,\pp 0)/\rt \\
 v_{ 2} &= (\pp 1,\pp 0,\pp 0,\pp 0) &
 v_{ 8} &= (   -1,   -1,\pp 0,\pp 0)/\rt \\
 v_{ 3} &= (\pp 0,\pp 1,\pp 0,\pp 0) &
 v_{ 9} &= (\pp 1,   -1,\pp 0,\pp 0)/\rt \\
 v_{ 4} &= (   -1,\pp 0,\pp 0,\pp 0) &
 v_{10} &= (0,0,\pp\sqrt{2/3},\pp\sqrt{1/3}) \\
 v_{ 5} &= (\pp 0,-1,\pp 0,\pp 0) &
 v_{11} &= (0,0,\pp\sqrt{2/3},  -\sqrt{1/3}) \\
 &&
 v_{12} &= (0,0,  -\sqrt{2/3},  -\sqrt{1/3}) \\
 &&
 v_{13} &= (0,0,  -\sqrt{2/3},\pp\sqrt{1/3}).
\end{align*}
The curve system is as follows:
\begin{align*}
 c_0(t) &= \left(\cos(t),\sin(t),0,0\right) \\
 c_1(t) &= \left(\sin(t)/\rt,\sin(t)/\rt,\cos(t),0\right) \\
 c_2(t) &= \lm(c_1(t)) \\
 c_3(t) &= \left(0,\sin(t),\sqrt{2/3}\cos(t),-\sqrt{1/3}\cos(t)\right) \\
 c_4(t) &= \lm(c_3(t)) \\
 c_5(t) &= \left(-\sin(t),0,2\rt,\cos(t)-1\right)/\sqrt{10-2\cos(t)} \\
 c_6(t) &= \lm(c_5(t)) \\
 c_7(t) &= \mu(c_5(t)) \\
 c_8(t) &= \lm\mu(c_5(t)).
\end{align*}

In this case, the curves $c_5,\dotsc,c_8$ have some additional
properties.  If we put
\[ h(x) = x_3x_4/\rt+x_1^2+x_4^2, \]
we find that $h=0$ on $C_5\cup C_7$.  Note here that $h(x)$ is a
homogeneous quadratic.  It can be diagonalised as
$h(x)=\sum_{i=1}^4m_i\ip{u_i,x}^2$, where
\begin{align*}
 m_1 &= 1 &
 u_1 &= (1,0,0,0) \\
 m_2 &= 0 &
 u_2 &= (0,1,0,0) \\
 m_3 &= \frac{2-\sqrt{6}}{4} \simeq -0.11 &
 u_3 &= \left(0,0,\pp\sqrt{1/2+1/\sqrt{6}},\quad -\sqrt{1/2-1/\sqrt{6}}\right) \\
 m_4 &= \frac{2-\sqrt{6}}{4} \simeq -0.11 &
 u_4 &= \left(0,0,\pp\sqrt{1/2-1/\sqrt{6}},\quad\pp\sqrt{1/2+1/\sqrt{6}}\right).
\end{align*}
The vectors $u_i$ here form an oriented orthonormal basis for $\R^4$.
\begin{checks}
 embedded/roothalf/E_roothalf_check.mpl: check_oval()
\end{checks}
One can also check that there is an alternative parametrisation of
$C_5$ as follows:
\[
 c_5^{\text{alt}}(t) =
 \left(\frac{\sin(t)}{\bt},\;0,\;
 \frac{1+\bt^2}{12}\left(\sqrt{1-\sin(t)^2/\bt^4} +  \cos(t)/\bt^2\right),\;
 \frac{\cos(t)-\sqrt{1-\sin(t)^2/\bt^4}}{2\sqrt{3}}\right),
\]
where $\bt=\sqrt{2}+\sqrt{3}\simeq 3.15$.  Note here that
$\bt^4\simeq 98$, so we have a good approximation
\[
 c_5^{\text{alt}}(t) \simeq
 \left(\frac{\sin(t)}{\bt},\;0,\;
 \frac{1+\bt^2}{12}\left(1 +  \cos(t)/\bt^2\right),\;
 \frac{\cos(t)-1}{2\sqrt{3}}\right),
\]

showing that $C_5$ is close to an ellipse.  If we put
\[ c_6^{\text{alt}}(t) = \lm(c_5^{\text{alt}}(t)) \qquad
   c_7^{\text{alt}}(t) = \mu(c_5^{\text{alt}}(t)) \qquad
   c_8^{\text{alt}}(t) = \lm\mu(c_5^{\text{alt}}(t)),
\]
and $c_k^{\text{alt}}=c_k$ for $0\leq k\leq 4$, then one can check
that this gives an alternative curve system.
\begin{checks}
 embedded/roothalf/E_roothalf_check.mpl: check_c_alt()
\end{checks}

\begin{proposition}\lbl{prop-roothalf-fundamental}
 In the case $a=1/\rt$, we have
 \begin{align*}
  F_4^* &= \{y\in\R^2\st |y_2|\leq 1/\rt,\;
             |y_1|\leq (1+|y_2|/\rt)^{-1/2}\} \\
  F_{16}^* &= \{z\in \R^2\st 0\leq z_2\leq 1/2,\;
                0\leq z_1\leq (1+\sqrt{z_2/2})^{-1}\}.
 \end{align*}
\end{proposition}
\begin{proof}
 From the definitions it is clear that
 \[ F_4^*=\{y\in\R^2 \st (y_1^2,y_2^2)\in F_{16}^*\}; \]
 using this, we can reduce the first claim to the second one.  Put
 \[ F'_{16} = \{z\in \R^2\st 0\leq z_2\leq 1/2,\;
                0\leq z_1\leq (1+\sqrt{z_2/2})^{-1}\},
 \]
 so we need to show that $F^*_{16}=F'_{16}$.

 Recall that $F_{16}^*=\{z\in\R^2\st z_1,z_2,u_3,u_4\geq 0\}$, and
 that in the present case $a=1/\rt$ we have
 \begin{align*}
  u_3 &= (1-2z_2)((1-z_1)^2-z_1^2z_2/2) \\
  u_4 &= 1-z_1-z_1z_2.
 \end{align*}
 We assume implicitly throughout that $z_1,z_2\geq 0$.  Put
 $w_1=1-2z_2$ and $w_2=(1-z_1)^2-z_1^2z_2/2$, so $u_3=w_1w_2$.  We
 will leave to the reader all the cases where any of the quantities
 $z_1,z_2,w_1,w_2,u_3$ or $u_4$ are zero.

 First suppose that $z\in F^*_{16}$, so $u_3,u_4>0$.  As
 $u_3=w_1w_2$ we see that $w_1$ and $w_2$ have the same sign.
 One can check that
 \[ \half z_1^2z_2w_1 + w_2 = u_4(u_4+2z_1z_2) > 0, \]
 so $w_1$ and $w_2$ must both be positive.  As $w_1>0$ we have
 $0<z_2<1/2$.  As $u_4>0$ we have $1-z_1>z_1z_2>0$, and as $w_2>0$ we
 have $(1-z_1)^2>z_1^2z_2/2$; it follows that $1-z_1>z_1\sqrt{z_2/2}$,
 and thus that $0<z_1<1/(1+\sqrt{z_2/2})$ as claimed.

 Conversely, suppose that $0<z_1<1/(1+\sqrt{z_2/2})$ (so in particular
 $z_1<1$) and $0<z_2<1/2$.  We can reverse the above arguments to see
 that $w_1,w_2>0$ and so $u_3=w_1w_2>0$.  Next, one can check that
 \[ ((1-\half z_2)(1+z_2)(1-z_1)+\half z_2(4+z_2))u_4 =
     \half z_2 w_1 + (1+z_2)^2 w_2.
 \]
 The coefficient of $u_4$ is strictly positive, as are all the terms
 on the right hand side, so we deduce that $u_4>0$, which means that
 $z\in F^*_{16}$.
\end{proof}

\subsection{Linear projections}
\lbl{sec-disc}

In this section we study the images of our points and curves under
three different orthogonal projections $\R^4\to\R^2$.  These do not
have any great theoretical significance, but they provide some insight
into the geometry.  Many additional details are given in the Maple
code, especially the files \fname+embedded/disc_proj.mpl+ and
\fname+embedded/roothalf/zeta.mpl+ and
\fname+embedded/roothalf/crease.mpl+.
The projections that we consider are
\begin{align*}
 \pi(x) &= (x_1,x_2) \\
 \dl(x) &= ((x_1-x_2)/\rt,x_3) \\
 \zt(x) &= ((x_3-x_4)/\rt,x_2).
\end{align*}

The picture for $\pi$ is as follows:
\begin{center}
\begin{tikzpicture}[scale=4]
 \draw[cyan] plot[smooth] coordinates{ (1.000,0.000) (0.988,0.156) (0.951,0.309) (0.891,0.454) (0.809,0.588) (0.707,0.707) (0.588,0.809) (0.454,0.891) (0.309,0.951) (0.156,0.988) (-0.000,1.000) (-0.156,0.988) (-0.309,0.951) (-0.454,0.891) (-0.588,0.809) (-0.707,0.707) (-0.809,0.588) (-0.891,0.454) (-0.951,0.309) (-0.988,0.156) (-1.000,-0.000) (-0.988,-0.156) (-0.951,-0.309) (-0.891,-0.454) (-0.809,-0.588) (-0.707,-0.707) (-0.588,-0.809) (-0.454,-0.891) (-0.309,-0.951) (-0.156,-0.988) (0.000,-1.000) (0.156,-0.988) (0.309,-0.951) (0.454,-0.891) (0.588,-0.809) (0.707,-0.707) (0.809,-0.588) (0.891,-0.454) (0.951,-0.309) (0.988,-0.156) (1.000,0.000) };
 \draw[green] plot[smooth] coordinates{ (0.000,0.000) (0.111,0.111) (0.219,0.219) (0.321,0.321) (0.416,0.416) (0.500,0.500) (0.572,0.572) (0.630,0.630) (0.672,0.672) (0.698,0.698) (0.707,0.707) (0.698,0.698) (0.672,0.672) (0.630,0.630) (0.572,0.572) (0.500,0.500) (0.416,0.416) (0.321,0.321) (0.219,0.219) (0.111,0.111) (-0.000,-0.000) (-0.111,-0.111) (-0.219,-0.219) (-0.321,-0.321) (-0.416,-0.416) (-0.500,-0.500) (-0.572,-0.572) (-0.630,-0.630) (-0.672,-0.672) (-0.698,-0.698) (-0.707,-0.707) (-0.698,-0.698) (-0.672,-0.672) (-0.630,-0.630) (-0.572,-0.572) (-0.500,-0.500) (-0.416,-0.416) (-0.321,-0.321) (-0.219,-0.219) (-0.111,-0.111) (0.000,0.000) };
 \draw[green] plot[smooth] coordinates{ (0.000,0.000) (-0.111,0.111) (-0.219,0.219) (-0.321,0.321) (-0.416,0.416) (-0.500,0.500) (-0.572,0.572) (-0.630,0.630) (-0.672,0.672) (-0.698,0.698) (-0.707,0.707) (-0.698,0.698) (-0.672,0.672) (-0.630,0.630) (-0.572,0.572) (-0.500,0.500) (-0.416,0.416) (-0.321,0.321) (-0.219,0.219) (-0.111,0.111) (0.000,-0.000) (0.111,-0.111) (0.219,-0.219) (0.321,-0.321) (0.416,-0.416) (0.500,-0.500) (0.572,-0.572) (0.630,-0.630) (0.672,-0.672) (0.698,-0.698) (0.707,-0.707) (0.698,-0.698) (0.672,-0.672) (0.630,-0.630) (0.572,-0.572) (0.500,-0.500) (0.416,-0.416) (0.321,-0.321) (0.219,-0.219) (0.111,-0.111) (-0.000,0.000) };
 \draw[magenta] plot[smooth] coordinates{ (0.000,0.000) (0.000,0.156) (0.000,0.309) (0.000,0.454) (0.000,0.588) (0.000,0.707) (0.000,0.809) (0.000,0.891) (0.000,0.951) (0.000,0.988) (0.000,1.000) (0.000,0.988) (0.000,0.951) (0.000,0.891) (0.000,0.809) (0.000,0.707) (0.000,0.588) (0.000,0.454) (0.000,0.309) (0.000,0.156) (0.000,-0.000) (0.000,-0.156) (0.000,-0.309) (0.000,-0.454) (0.000,-0.588) (0.000,-0.707) (0.000,-0.809) (0.000,-0.891) (0.000,-0.951) (0.000,-0.988) (0.000,-1.000) (0.000,-0.988) (0.000,-0.951) (0.000,-0.891) (0.000,-0.809) (0.000,-0.707) (0.000,-0.588) (0.000,-0.454) (0.000,-0.309) (0.000,-0.156) (0.000,0.000) };
 \draw[magenta] plot[smooth] coordinates{ (0.000,0.000) (-0.156,0.000) (-0.309,0.000) (-0.454,0.000) (-0.588,0.000) (-0.707,0.000) (-0.809,0.000) (-0.891,0.000) (-0.951,0.000) (-0.988,0.000) (-1.000,0.000) (-0.988,0.000) (-0.951,0.000) (-0.891,0.000) (-0.809,0.000) (-0.707,0.000) (-0.588,0.000) (-0.454,0.000) (-0.309,0.000) (-0.156,0.000) (0.000,0.000) (0.156,0.000) (0.309,0.000) (0.454,0.000) (0.588,0.000) (0.707,0.000) (0.809,0.000) (0.891,0.000) (0.951,0.000) (0.988,0.000) (1.000,0.000) (0.988,0.000) (0.951,0.000) (0.891,0.000) (0.809,0.000) (0.707,0.000) (0.588,0.000) (0.454,0.000) (0.309,0.000) (0.156,0.000) (-0.000,0.000) };
 \draw[blue] plot[smooth] coordinates{ (0.000,0.000) (0.055,0.000) (0.109,0.000) (0.158,0.000) (0.203,0.000) (0.241,0.000) (0.272,0.000) (0.295,0.000) (0.310,0.000) (0.317,0.000) (0.316,0.000) (0.308,0.000) (0.292,0.000) (0.270,0.000) (0.242,0.000) (0.209,0.000) (0.172,0.000) (0.132,0.000) (0.090,0.000) (0.045,0.000) (-0.000,0.000) (-0.045,0.000) (-0.090,0.000) (-0.132,0.000) (-0.172,0.000) (-0.209,0.000) (-0.242,0.000) (-0.270,0.000) (-0.292,0.000) (-0.308,0.000) (-0.316,0.000) (-0.317,0.000) (-0.310,0.000) (-0.295,0.000) (-0.272,0.000) (-0.241,0.000) (-0.203,0.000) (-0.158,0.000) (-0.109,0.000) (-0.055,0.000) (0.000,0.000) };
 \draw[blue] plot[smooth] coordinates{ (0.000,0.000) (0.000,0.055) (0.000,0.109) (0.000,0.158) (0.000,0.203) (0.000,0.241) (0.000,0.272) (0.000,0.295) (0.000,0.310) (0.000,0.317) (0.000,0.316) (0.000,0.308) (0.000,0.292) (0.000,0.270) (0.000,0.242) (0.000,0.209) (0.000,0.172) (0.000,0.132) (0.000,0.090) (0.000,0.045) (0.000,-0.000) (0.000,-0.045) (0.000,-0.090) (0.000,-0.132) (0.000,-0.172) (0.000,-0.209) (0.000,-0.242) (0.000,-0.270) (0.000,-0.292) (0.000,-0.308) (0.000,-0.316) (0.000,-0.317) (0.000,-0.310) (0.000,-0.295) (0.000,-0.272) (0.000,-0.241) (0.000,-0.203) (0.000,-0.158) (0.000,-0.109) (0.000,-0.055) (0.000,0.000) };
 \draw[blue] plot[smooth] coordinates{ (0.000,0.000) (0.055,0.000) (0.109,0.000) (0.158,0.000) (0.203,0.000) (0.241,0.000) (0.272,0.000) (0.295,0.000) (0.310,0.000) (0.317,0.000) (0.316,0.000) (0.308,0.000) (0.292,0.000) (0.270,0.000) (0.242,0.000) (0.209,0.000) (0.172,0.000) (0.132,0.000) (0.090,0.000) (0.045,0.000) (-0.000,0.000) (-0.045,0.000) (-0.090,0.000) (-0.132,0.000) (-0.172,0.000) (-0.209,0.000) (-0.242,0.000) (-0.270,0.000) (-0.292,0.000) (-0.308,0.000) (-0.316,0.000) (-0.317,0.000) (-0.310,0.000) (-0.295,0.000) (-0.272,0.000) (-0.241,0.000) (-0.203,0.000) (-0.158,0.000) (-0.109,0.000) (-0.055,0.000) (0.000,0.000) };
 \draw[blue] plot[smooth] coordinates{ (0.000,0.000) (0.000,0.055) (0.000,0.109) (0.000,0.158) (0.000,0.203) (0.000,0.241) (0.000,0.272) (0.000,0.295) (0.000,0.310) (0.000,0.317) (0.000,0.316) (0.000,0.308) (0.000,0.292) (0.000,0.270) (0.000,0.242) (0.000,0.209) (0.000,0.172) (0.000,0.132) (0.000,0.090) (0.000,0.045) (0.000,-0.000) (0.000,-0.045) (0.000,-0.090) (0.000,-0.132) (0.000,-0.172) (0.000,-0.209) (0.000,-0.242) (0.000,-0.270) (0.000,-0.292) (0.000,-0.308) (0.000,-0.316) (0.000,-0.317) (0.000,-0.310) (0.000,-0.295) (0.000,-0.272) (0.000,-0.241) (0.000,-0.203) (0.000,-0.158) (0.000,-0.109) (0.000,-0.055) (0.000,0.000) };
 \fill (0.000,0.000) circle(0.010);
 \fill (0.000,0.000) circle(0.010);
 \fill (1.000,0.000) circle(0.010);
 \fill (0.000,1.000) circle(0.010);
 \fill (-1.000,0.000) circle(0.010);
 \fill (0.000,-1.000) circle(0.010);
 \fill (0.707,0.707) circle(0.010);
 \fill (-0.707,0.707) circle(0.010);
 \fill (-0.707,-0.707) circle(0.010);
 \fill (0.707,-0.707) circle(0.010);
 \fill (0.000,0.000) circle(0.010);
 \fill (0.000,0.000) circle(0.010);
 \fill (0.000,0.000) circle(0.010);
 \fill (0.000,0.000) circle(0.010);
 \draw (0.000,0.000) node[anchor=south] {$v_{0},v_{1},v_{10},v_{11},v_{12},v_{13}$};
 \draw (-1.000,0.000) node[anchor=east] {$v_{4}$};
 \draw (0.000,-1.000) node[anchor=north] {$v_{5}$};
 \draw (1.000,0.000) node[anchor=west] {$v_{2}$};
 \draw (0.707,0.707) node[anchor=south west] {$v_{6}$};
 \draw (0.000,1.000) node[anchor=south] {$v_{3}$};
 \draw (-0.707,0.707) node[anchor=south east] {$v_{7}$};
 \draw (-0.707,-0.707) node[anchor=north east] {$v_{8}$};
 \draw (0.707,-0.707) node[anchor=north west] {$v_{9}$};
 \draw (0.924,0.383) node[anchor=north] {$c_{0}$};
 \draw (0.595,0.595) node[anchor=north west] {$c_{1}$};
 \draw (-0.500,0.500) node[anchor=north east] {$c_{2}$};
 \draw (0.000,0.707) node[anchor=west] {$c_{3}$};
 \draw (-0.707,0.000) node[anchor=north] {$c_{4}$};
 \draw (0.276,0.000) node[anchor=north] {$c_{5}$};
 \draw (0.000,0.276) node[anchor=west] {$c_{6}$};
 \draw (-0.300,0.000) node[anchor=north] {$c_{7}$};
 \draw (0.000,-0.300) node[anchor=west] {$c_{8}$};
 \draw[dotted] plot[smooth] coordinates{ (0.000,0.318) (-0.182,0.331) (-0.247,0.343) (-0.291,0.354) (-0.324,0.365) (-0.348,0.374) (-0.367,0.383) (-0.381,0.390) (-0.392,0.396) (-0.399,0.401) (-0.405,0.405) (-0.407,0.407) (-0.408,0.408) (-0.407,0.407) (-0.405,0.405) (-0.401,0.399) (-0.396,0.392) (-0.390,0.381) (-0.383,0.367) (-0.374,0.348) (-0.365,0.324) (-0.354,0.291) (-0.343,0.247) (-0.331,0.182) (-0.318,0.000) };
 \draw[dotted] plot[smooth] coordinates{ (0.000,0.318) (0.182,0.331) (0.247,0.343) (0.291,0.354) (0.324,0.365) (0.348,0.374) (0.367,0.383) (0.381,0.390) (0.392,0.396) (0.399,0.401) (0.405,0.405) (0.407,0.407) (0.408,0.408) (0.407,0.407) (0.405,0.405) (0.401,0.399) (0.396,0.392) (0.390,0.381) (0.383,0.367) (0.374,0.348) (0.365,0.324) (0.354,0.291) (0.343,0.247) (0.331,0.182) (0.318,0.000) };
 \draw[dotted] plot[smooth] coordinates{ (0.000,-0.318) (-0.182,-0.331) (-0.247,-0.343) (-0.291,-0.354) (-0.324,-0.365) (-0.348,-0.374) (-0.367,-0.383) (-0.381,-0.390) (-0.392,-0.396) (-0.399,-0.401) (-0.405,-0.405) (-0.407,-0.407) (-0.408,-0.408) (-0.407,-0.407) (-0.405,-0.405) (-0.401,-0.399) (-0.396,-0.392) (-0.390,-0.381) (-0.383,-0.367) (-0.374,-0.348) (-0.365,-0.324) (-0.354,-0.291) (-0.343,-0.247) (-0.331,-0.182) (-0.318,0.000) };
 \draw[dotted] plot[smooth] coordinates{ (0.000,-0.318) (0.182,-0.331) (0.247,-0.343) (0.291,-0.354) (0.324,-0.365) (0.348,-0.374) (0.367,-0.383) (0.381,-0.390) (0.392,-0.396) (0.399,-0.401) (0.405,-0.405) (0.407,-0.407) (0.408,-0.408) (0.407,-0.407) (0.405,-0.405) (0.401,-0.399) (0.396,-0.392) (0.390,-0.381) (0.383,-0.367) (0.374,-0.348) (0.365,-0.324) (0.354,-0.291) (0.343,-0.247) (0.331,-0.182) (0.318,0.000) };
\end{tikzpicture}
\end{center}

The set of singular values of $\pi$ is the union of the unit circle
with the dotted curve.  Points inside the dotted curve have six
preimages in $EX^*$, whereas those between the dotted curve and the
unit circle have two preimages, but there is no tidy formula for
these.

The effect of $\dl$ on the curves $c_i$ and vertices $v_j$ can be
displayed as follows:
\begin{center}
\begin{tikzpicture}[scale=4]
 \draw[cyan] plot[smooth] coordinates{ (0.707,0.000) (0.588,0.000) (0.454,0.000) (0.309,0.000) (0.156,0.000) (0.000,0.000) (-0.156,0.000) (-0.309,0.000) (-0.454,0.000) (-0.588,0.000) (-0.707,0.000) (-0.809,0.000) (-0.891,0.000) (-0.951,0.000) (-0.988,0.000) (-1.000,0.000) (-0.988,0.000) (-0.951,0.000) (-0.891,0.000) (-0.809,0.000) (-0.707,0.000) (-0.588,0.000) (-0.454,0.000) (-0.309,0.000) (-0.156,0.000) (0.000,0.000) (0.156,0.000) (0.309,0.000) (0.454,0.000) (0.588,0.000) (0.707,0.000) (0.809,0.000) (0.891,0.000) (0.951,0.000) (0.988,0.000) (1.000,0.000) (0.988,0.000) (0.951,0.000) (0.891,0.000) (0.809,0.000) (0.707,0.000) };
 \draw[green] plot[smooth] coordinates{ (0.000,1.000) (0.000,0.988) (0.000,0.951) (0.000,0.891) (0.000,0.809) (0.000,0.707) (0.000,0.588) (0.000,0.454) (0.000,0.309) (0.000,0.156) (0.000,-0.000) (0.000,-0.156) (0.000,-0.309) (0.000,-0.454) (0.000,-0.588) (0.000,-0.707) (0.000,-0.809) (0.000,-0.891) (0.000,-0.951) (0.000,-0.988) (0.000,-1.000) (0.000,-0.988) (0.000,-0.951) (0.000,-0.891) (0.000,-0.809) (0.000,-0.707) (0.000,-0.588) (0.000,-0.454) (0.000,-0.309) (0.000,-0.156) (0.000,0.000) (0.000,0.156) (0.000,0.309) (0.000,0.454) (0.000,0.588) (0.000,0.707) (0.000,0.809) (0.000,0.891) (0.000,0.951) (0.000,0.988) (0.000,1.000) };
 \draw[green] plot[smooth] coordinates{ (0.000,1.000) (-0.156,0.988) (-0.309,0.951) (-0.454,0.891) (-0.588,0.809) (-0.707,0.707) (-0.809,0.588) (-0.891,0.454) (-0.951,0.309) (-0.988,0.156) (-1.000,-0.000) (-0.988,-0.156) (-0.951,-0.309) (-0.891,-0.454) (-0.809,-0.588) (-0.707,-0.707) (-0.588,-0.809) (-0.454,-0.891) (-0.309,-0.951) (-0.156,-0.988) (0.000,-1.000) (0.156,-0.988) (0.309,-0.951) (0.454,-0.891) (0.588,-0.809) (0.707,-0.707) (0.809,-0.588) (0.891,-0.454) (0.951,-0.309) (0.988,-0.156) (1.000,0.000) (0.988,0.156) (0.951,0.309) (0.891,0.454) (0.809,0.588) (0.707,0.707) (0.588,0.809) (0.454,0.891) (0.309,0.951) (0.156,0.988) (-0.000,1.000) };
 \draw[magenta] plot[smooth] coordinates{ (0.000,0.883) (-0.102,0.873) (-0.204,0.840) (-0.304,0.787) (-0.399,0.715) (-0.486,0.625) (-0.562,0.519) (-0.624,0.401) (-0.670,0.273) (-0.698,0.138) (-0.707,-0.000) (-0.698,-0.138) (-0.670,-0.273) (-0.624,-0.401) (-0.562,-0.519) (-0.486,-0.625) (-0.399,-0.715) (-0.304,-0.787) (-0.204,-0.840) (-0.102,-0.873) (0.000,-0.883) (0.102,-0.873) (0.204,-0.840) (0.304,-0.787) (0.399,-0.715) (0.486,-0.625) (0.562,-0.519) (0.624,-0.401) (0.670,-0.273) (0.698,-0.138) (0.707,0.000) (0.698,0.138) (0.670,0.273) (0.624,0.401) (0.562,0.519) (0.486,0.625) (0.399,0.715) (0.304,0.787) (0.204,0.840) (0.102,0.873) (-0.000,0.883) };
 \draw[magenta] plot[smooth] coordinates{ (0.000,0.883) (-0.102,0.873) (-0.204,0.840) (-0.304,0.787) (-0.399,0.715) (-0.486,0.625) (-0.562,0.519) (-0.624,0.401) (-0.670,0.273) (-0.698,0.138) (-0.707,-0.000) (-0.698,-0.138) (-0.670,-0.273) (-0.624,-0.401) (-0.562,-0.519) (-0.486,-0.625) (-0.399,-0.715) (-0.304,-0.787) (-0.204,-0.840) (-0.102,-0.873) (0.000,-0.883) (0.102,-0.873) (0.204,-0.840) (0.304,-0.787) (0.399,-0.715) (0.486,-0.625) (0.562,-0.519) (0.624,-0.401) (0.670,-0.273) (0.698,-0.138) (0.707,0.000) (0.698,0.138) (0.670,0.273) (0.624,0.401) (0.562,0.519) (0.486,0.625) (0.399,0.715) (0.304,0.787) (0.204,0.840) (0.102,0.873) (-0.000,0.883) };
 \draw[blue] plot[smooth] coordinates{ (0.000,1.000) (0.027,0.999) (0.053,0.997) (0.078,0.993) (0.101,0.988) (0.121,0.982) (0.138,0.975) (0.151,0.967) (0.161,0.958) (0.166,0.949) (0.167,0.939) (0.163,0.930) (0.156,0.921) (0.145,0.913) (0.131,0.906) (0.114,0.899) (0.094,0.894) (0.072,0.889) (0.049,0.886) (0.025,0.884) (-0.000,0.883) (-0.025,0.884) (-0.049,0.886) (-0.072,0.889) (-0.094,0.894) (-0.114,0.899) (-0.131,0.906) (-0.145,0.913) (-0.156,0.921) (-0.163,0.930) (-0.167,0.939) (-0.166,0.949) (-0.161,0.958) (-0.151,0.967) (-0.138,0.975) (-0.121,0.982) (-0.101,0.988) (-0.078,0.993) (-0.053,0.997) (-0.027,0.999) (0.000,1.000) };
 \draw[blue] plot[smooth] coordinates{ (0.000,1.000) (-0.027,0.999) (-0.053,0.997) (-0.078,0.993) (-0.101,0.988) (-0.121,0.982) (-0.138,0.975) (-0.151,0.967) (-0.161,0.958) (-0.166,0.949) (-0.167,0.939) (-0.163,0.930) (-0.156,0.921) (-0.145,0.913) (-0.131,0.906) (-0.114,0.899) (-0.094,0.894) (-0.072,0.889) (-0.049,0.886) (-0.025,0.884) (0.000,0.883) (0.025,0.884) (0.049,0.886) (0.072,0.889) (0.094,0.894) (0.114,0.899) (0.131,0.906) (0.145,0.913) (0.156,0.921) (0.163,0.930) (0.167,0.939) (0.166,0.949) (0.161,0.958) (0.151,0.967) (0.138,0.975) (0.121,0.982) (0.101,0.988) (0.078,0.993) (0.053,0.997) (0.027,0.999) (-0.000,1.000) };
 \draw[blue] plot[smooth] coordinates{ (0.000,-1.000) (0.027,-0.999) (0.053,-0.997) (0.078,-0.993) (0.101,-0.988) (0.121,-0.982) (0.138,-0.975) (0.151,-0.967) (0.161,-0.958) (0.166,-0.949) (0.167,-0.939) (0.163,-0.930) (0.156,-0.921) (0.145,-0.913) (0.131,-0.906) (0.114,-0.899) (0.094,-0.894) (0.072,-0.889) (0.049,-0.886) (0.025,-0.884) (-0.000,-0.883) (-0.025,-0.884) (-0.049,-0.886) (-0.072,-0.889) (-0.094,-0.894) (-0.114,-0.899) (-0.131,-0.906) (-0.145,-0.913) (-0.156,-0.921) (-0.163,-0.930) (-0.167,-0.939) (-0.166,-0.949) (-0.161,-0.958) (-0.151,-0.967) (-0.138,-0.975) (-0.121,-0.982) (-0.101,-0.988) (-0.078,-0.993) (-0.053,-0.997) (-0.027,-0.999) (0.000,-1.000) };
 \draw[blue] plot[smooth] coordinates{ (0.000,-1.000) (-0.027,-0.999) (-0.053,-0.997) (-0.078,-0.993) (-0.101,-0.988) (-0.121,-0.982) (-0.138,-0.975) (-0.151,-0.967) (-0.161,-0.958) (-0.166,-0.949) (-0.167,-0.939) (-0.163,-0.930) (-0.156,-0.921) (-0.145,-0.913) (-0.131,-0.906) (-0.114,-0.899) (-0.094,-0.894) (-0.072,-0.889) (-0.049,-0.886) (-0.025,-0.884) (0.000,-0.883) (0.025,-0.884) (0.049,-0.886) (0.072,-0.889) (0.094,-0.894) (0.114,-0.899) (0.131,-0.906) (0.145,-0.913) (0.156,-0.921) (0.163,-0.930) (0.167,-0.939) (0.166,-0.949) (0.161,-0.958) (0.151,-0.967) (0.138,-0.975) (0.121,-0.982) (0.101,-0.988) (0.078,-0.993) (0.053,-0.997) (0.027,-0.999) (-0.000,-1.000) };
 \fill (0.000,1.000) circle(0.010);
 \fill (0.000,-1.000) circle(0.010);
 \fill (0.707,0.000) circle(0.010);
 \fill (-0.707,0.000) circle(0.010);
 \fill (-0.707,0.000) circle(0.010);
 \fill (0.707,0.000) circle(0.010);
 \fill (0.000,0.000) circle(0.010);
 \fill (-1.000,0.000) circle(0.010);
 \fill (0.000,0.000) circle(0.010);
 \fill (1.000,0.000) circle(0.010);
 \fill (0.000,0.883) circle(0.010);
 \fill (0.000,0.883) circle(0.010);
 \fill (0.000,-0.883) circle(0.010);
 \fill (0.000,-0.883) circle(0.010);
 \draw (0.000,1.000) node[anchor=south] {$v_{0}$};
 \draw (0.000,-1.000) node[anchor=north] {$v_{1}$};
 \draw (0.000,0.000) node[anchor=north] {$v_{6},v_{8}$};
 \draw (0.000,0.883) node[anchor=north] {$v_{10},v_{11}$};
 \draw (-0.707,0.000) node[anchor=north] {$v_{3},v_{4}$};
 \draw (1.000,0.000) node[anchor=west] {$v_{9}$};
 \draw (0.707,0.000) node[anchor=north] {$v_{2},v_{5}$};
 \draw (-1.000,0.000) node[anchor=east] {$v_{7}$};
 \draw (0.000,-0.883) node[anchor=south] {$v_{12},v_{13}$};
 \draw (0.383,0.000) node[anchor=north] {$c_{0}$};
 \draw (0.000,0.540) node[anchor=west] {$c_{1}$};
 \draw (-0.707,0.707) node[anchor=south east] {$c_{2}$};
 \draw (-0.486,0.625) node[anchor=north west] {$c_{3}$};
 \draw (-0.581,0.488) node[anchor=north west] {$c_{4}$};
 \draw (0.148,0.915) node[anchor=west] {$c_{5}$};
 \draw (-0.148,0.915) node[anchor=east] {$c_{6}$};
 \draw (0.148,-0.915) node[anchor=west] {$c_{7}$};
 \draw (-0.148,-0.915) node[anchor=east] {$c_{8}$};
\end{tikzpicture}
\end{center}
Here again the image of $\dl$ is the full unit disc.  For most points
in the disc, the preimage consists of two points that are exchanged by
the action of $\lm^{-1}\nu$, and one can give a nice formula for these
points.

For the map $\zt$, we have the following picture:
\begin{center}
\begin{tikzpicture}[scale=4]
 \draw[cyan] plot[smooth] coordinates{ (0.000,0.000) (0.000,0.156) (0.000,0.309) (0.000,0.454) (0.000,0.588) (0.000,0.707) (0.000,0.809) (0.000,0.891) (0.000,0.951) (0.000,0.988) (0.000,1.000) (0.000,0.988) (0.000,0.951) (0.000,0.891) (0.000,0.809) (0.000,0.707) (0.000,0.588) (0.000,0.454) (0.000,0.309) (0.000,0.156) (0.000,-0.000) (0.000,-0.156) (0.000,-0.309) (0.000,-0.454) (0.000,-0.588) (0.000,-0.707) (0.000,-0.809) (0.000,-0.891) (0.000,-0.951) (0.000,-0.988) (0.000,-1.000) (0.000,-0.988) (0.000,-0.951) (0.000,-0.891) (0.000,-0.809) (0.000,-0.707) (0.000,-0.588) (0.000,-0.454) (0.000,-0.309) (0.000,-0.156) (0.000,0.000) };
 \draw[green] plot[smooth] coordinates{ (0.707,0.000) (0.698,0.111) (0.672,0.219) (0.630,0.321) (0.572,0.416) (0.500,0.500) (0.416,0.572) (0.321,0.630) (0.219,0.672) (0.111,0.698) (-0.000,0.707) (-0.111,0.698) (-0.219,0.672) (-0.321,0.630) (-0.416,0.572) (-0.500,0.500) (-0.572,0.416) (-0.630,0.321) (-0.672,0.219) (-0.698,0.111) (-0.707,-0.000) (-0.698,-0.111) (-0.672,-0.219) (-0.630,-0.321) (-0.572,-0.416) (-0.500,-0.500) (-0.416,-0.572) (-0.321,-0.630) (-0.219,-0.672) (-0.111,-0.698) (0.000,-0.707) (0.111,-0.698) (0.219,-0.672) (0.321,-0.630) (0.416,-0.572) (0.500,-0.500) (0.572,-0.416) (0.630,-0.321) (0.672,-0.219) (0.698,-0.111) (0.707,0.000) };
 \draw[green] plot[smooth] coordinates{ (0.707,0.000) (0.698,0.111) (0.672,0.219) (0.630,0.321) (0.572,0.416) (0.500,0.500) (0.416,0.572) (0.321,0.630) (0.219,0.672) (0.111,0.698) (-0.000,0.707) (-0.111,0.698) (-0.219,0.672) (-0.321,0.630) (-0.416,0.572) (-0.500,0.500) (-0.572,0.416) (-0.630,0.321) (-0.672,0.219) (-0.698,0.111) (-0.707,-0.000) (-0.698,-0.111) (-0.672,-0.219) (-0.630,-0.321) (-0.572,-0.416) (-0.500,-0.500) (-0.416,-0.572) (-0.321,-0.630) (-0.219,-0.672) (-0.111,-0.698) (0.000,-0.707) (0.111,-0.698) (0.219,-0.672) (0.321,-0.630) (0.416,-0.572) (0.500,-0.500) (0.572,-0.416) (0.630,-0.321) (0.672,-0.219) (0.698,-0.111) (0.707,0.000) };
 \draw[magenta] plot[smooth] coordinates{ (0.986,0.000) (0.973,0.156) (0.937,0.309) (0.878,0.454) (0.797,0.588) (0.697,0.707) (0.579,0.809) (0.447,0.891) (0.305,0.951) (0.154,0.988) (-0.000,1.000) (-0.154,0.988) (-0.305,0.951) (-0.447,0.891) (-0.579,0.809) (-0.697,0.707) (-0.797,0.588) (-0.878,0.454) (-0.937,0.309) (-0.973,0.156) (-0.986,-0.000) (-0.973,-0.156) (-0.937,-0.309) (-0.878,-0.454) (-0.797,-0.588) (-0.697,-0.707) (-0.579,-0.809) (-0.447,-0.891) (-0.305,-0.951) (-0.154,-0.988) (0.000,-1.000) (0.154,-0.988) (0.305,-0.951) (0.447,-0.891) (0.579,-0.809) (0.697,-0.707) (0.797,-0.588) (0.878,-0.454) (0.937,-0.309) (0.973,-0.156) (0.986,0.000) };
 \draw[magenta] plot[smooth] coordinates{ (0.169,0.000) (0.167,0.000) (0.161,0.000) (0.151,0.000) (0.137,0.000) (0.120,0.000) (0.099,0.000) (0.077,0.000) (0.052,0.000) (0.026,0.000) (-0.000,0.000) (-0.026,0.000) (-0.052,0.000) (-0.077,0.000) (-0.099,0.000) (-0.120,0.000) (-0.137,0.000) (-0.151,0.000) (-0.161,0.000) (-0.167,0.000) (-0.169,0.000) (-0.167,0.000) (-0.161,0.000) (-0.151,0.000) (-0.137,0.000) (-0.120,0.000) (-0.099,0.000) (-0.077,0.000) (-0.052,0.000) (-0.026,0.000) (0.000,0.000) (0.026,0.000) (0.052,0.000) (0.077,0.000) (0.099,0.000) (0.120,0.000) (0.137,0.000) (0.151,0.000) (0.161,0.000) (0.167,0.000) (0.169,0.000) };
  \filldraw[draw=blue,fill=gray!20] plot[smooth] coordinates{ (0.707,0.000) (0.709,0.000) (0.715,0.000) (0.725,0.000) (0.737,0.000) (0.753,0.000) (0.771,0.000) (0.791,0.000) (0.812,0.000) (0.834,0.000) (0.856,0.000) (0.877,0.000) (0.898,0.000) (0.917,0.000) (0.934,0.000) (0.949,0.000) (0.962,0.000) (0.972,0.000) (0.980,0.000) (0.984,0.000) (0.986,0.000) (0.984,0.000) (0.980,0.000) (0.972,0.000) (0.962,0.000) (0.949,0.000) (0.934,0.000) (0.917,0.000) (0.898,0.000) (0.877,0.000) (0.856,0.000) (0.834,0.000) (0.812,0.000) (0.791,0.000) (0.771,0.000) (0.753,0.000) (0.737,0.000) (0.725,0.000) (0.715,0.000) (0.709,0.000) (0.707,0.000) };
  \filldraw[draw=blue,fill=gray!20] plot[smooth] coordinates{ (0.707,0.000) (0.703,0.055) (0.691,0.109) (0.671,0.158) (0.644,0.203) (0.612,0.241) (0.575,0.272) (0.535,0.295) (0.493,0.310) (0.451,0.317) (0.409,0.316) (0.368,0.308) (0.330,0.292) (0.294,0.270) (0.262,0.242) (0.235,0.209) (0.211,0.172) (0.193,0.132) (0.180,0.090) (0.172,0.045) (0.169,-0.000) (0.172,-0.045) (0.180,-0.090) (0.193,-0.132) (0.211,-0.172) (0.235,-0.209) (0.262,-0.242) (0.294,-0.270) (0.330,-0.292) (0.368,-0.308) (0.409,-0.316) (0.451,-0.317) (0.493,-0.310) (0.535,-0.295) (0.575,-0.272) (0.612,-0.241) (0.644,-0.203) (0.671,-0.158) (0.691,-0.109) (0.703,-0.055) (0.707,0.000) };
  \filldraw[draw=blue,fill=gray!20] plot[smooth] coordinates{ (-0.707,0.000) (-0.709,0.000) (-0.715,0.000) (-0.725,0.000) (-0.737,0.000) (-0.753,0.000) (-0.771,0.000) (-0.791,0.000) (-0.812,0.000) (-0.834,0.000) (-0.856,0.000) (-0.877,0.000) (-0.898,0.000) (-0.917,0.000) (-0.934,0.000) (-0.949,0.000) (-0.962,0.000) (-0.972,0.000) (-0.980,0.000) (-0.984,0.000) (-0.986,0.000) (-0.984,0.000) (-0.980,0.000) (-0.972,0.000) (-0.962,0.000) (-0.949,0.000) (-0.934,0.000) (-0.917,0.000) (-0.898,0.000) (-0.877,0.000) (-0.856,0.000) (-0.834,0.000) (-0.812,0.000) (-0.791,0.000) (-0.771,0.000) (-0.753,0.000) (-0.737,0.000) (-0.725,0.000) (-0.715,0.000) (-0.709,0.000) (-0.707,0.000) };
  \filldraw[draw=blue,fill=gray!20] plot[smooth] coordinates{ (-0.707,0.000) (-0.703,0.055) (-0.691,0.109) (-0.671,0.158) (-0.644,0.203) (-0.612,0.241) (-0.575,0.272) (-0.535,0.295) (-0.493,0.310) (-0.451,0.317) (-0.409,0.316) (-0.368,0.308) (-0.330,0.292) (-0.294,0.270) (-0.262,0.242) (-0.235,0.209) (-0.211,0.172) (-0.193,0.132) (-0.180,0.090) (-0.172,0.045) (-0.169,-0.000) (-0.172,-0.045) (-0.180,-0.090) (-0.193,-0.132) (-0.211,-0.172) (-0.235,-0.209) (-0.262,-0.242) (-0.294,-0.270) (-0.330,-0.292) (-0.368,-0.308) (-0.409,-0.316) (-0.451,-0.317) (-0.493,-0.310) (-0.535,-0.295) (-0.575,-0.272) (-0.612,-0.241) (-0.644,-0.203) (-0.671,-0.158) (-0.691,-0.109) (-0.703,-0.055) (-0.707,0.000) };
 \fill (0.707,0.000) circle(0.010);
 \fill (-0.707,0.000) circle(0.010);
 \fill (0.000,0.000) circle(0.010);
 \fill (0.000,1.000) circle(0.010);
 \fill (0.000,0.000) circle(0.010);
 \fill (0.000,-1.000) circle(0.010);
 \fill (0.000,0.707) circle(0.010);
 \fill (0.000,0.707) circle(0.010);
 \fill (0.000,-0.707) circle(0.010);
 \fill (0.000,-0.707) circle(0.010);
 \fill (0.169,0.000) circle(0.010);
 \fill (0.986,0.000) circle(0.010);
 \fill (-0.169,0.000) circle(0.010);
 \fill (-0.986,0.000) circle(0.010);
 \draw (0.000,0.707) node[anchor=south] {$v_{6},v_{7}$};
 \draw (-0.169,0.000) node[anchor=east] {$v_{12}$};
 \draw (0.000,0.000) node[anchor=south] {$v_{2},v_{4}$};
 \draw (0.707,0.000) node[anchor=east] {$v_{0}$};
 \draw (0.986,0.000) node[anchor=west] {$v_{11}$};
 \draw (-0.707,0.000) node[anchor=west] {$v_{1}$};
 \draw (0.000,-0.707) node[anchor=north] {$v_{8},v_{9}$};
 \draw (0.000,-1.000) node[anchor=north] {$v_{5}$};
 \draw (-0.986,0.000) node[anchor=east] {$v_{13}$};
 \draw (0.169,0.000) node[anchor=west] {$v_{10}$};
 \draw (0.000,1.000) node[anchor=south] {$v_{3}$};
 \draw (0.000,0.383) node[anchor=west] {$c_{0}$};
 \draw (0.500,0.500) node[anchor=south] {$c_{1}$};
 \draw (0.500,0.500) node[anchor=east] {$c_{2}$};
 \draw (0.697,0.707) node[anchor=west] {$c_{3}$};
 \draw (0.120,0.000) node[anchor=north] {$c_{4}$};
 \draw (0.856,0.000) node[anchor=north] {$c_{5}$};
 \draw (0.409,0.316) node[anchor=north] {$c_{6}$};
 \draw (-0.856,0.000) node[anchor=north] {$c_{7}$};
 \draw (-0.409,0.316) node[anchor=north] {$c_{8}$};
\end{tikzpicture}
\end{center}
In this case, the shaded regions are not part of the image, so the
image is homeomorphic to a disc with two holes.  It can be identified
with $EX^*/\ip{\lm^2\nu}$, and is combinatorially equivalent to the
space $\Net_4^+$ discussed in Section~\ref{sec-fundamental}.

\subsection{Homeomorphisms with the square}
\lbl{sec-roothalf-square}

\begin{proposition}
 There is a homeomorphism $\tau\:EX^*/G\to [0,1]^2$ given by
 \[ \tau(x) = (2z_1-z_1^2+\half z_1^2z_2,\;2z_2). \]
 Moreover, we can define a homeomorphism $\tau^*\:[0,1]^2\to F_{16}$ by
 \[ \tau^*(t_1,t_2) = \left(
     \sqrt{\frac{(1-\sqrt{t_2})(\sqrt{t_3}+\sqrt{t_2})}{2(2+\sqrt{t_2})}},\;
     \sqrt{\frac{(1+\sqrt{t_2})(\sqrt{t_3}-\sqrt{t_2})}{2(2-\sqrt{t_2})}},\;
     \rt\sqrt{\frac{2-\sqrt{t_3}}{4-t_2}},\;
     -\sqrt{t_2}\sqrt{\frac{2-\sqrt{t_3}}{4-t_2}}
    \right),
 \]
 where $t_3=t_1t_2+4(1-t_1)$.  This is inverse to $\tau|_{F_{16}}$.
\end{proposition}
Maple notation for $\tau(x)$ and $\tau^*(t)$ is \mcode+t_proj(x)+ and
\mcode+t_lift(t)+.
\begin{proof}
 First, recall that the map $p_{16}\:x\mapsto z$ gives a homeomorphism
 from $F_{16}\simeq EX^*/G$ to the set
 \[ F^*_{16} = \{z\in \R^2\st 0\leq z_2\leq 1/2,\;
                0\leq z_1\leq (1+\sqrt{z_2/2})^{-1}\}.
 \]
 We have $\tau=\sg\circ p_{16}$, where
 \[ \sg(z) = (2z_1-z_1^2+\half z_1^2z_2,\;2z_2). \]
 An elementary exercise shows that for fixed $z_2\in[0,1/2]$, the map
 $z_1\mapsto 2z_1-z_1^2+\half z_1^2z_2$ gives a bijection from the
 interval $[0,(1+\sqrt{z_2/2})^{-1}]$ to $[0,1]$.  This implies that
 $\sg$ gives a continuous bijection from $F^*_{16}$ to $[0,1]^2$.  It
 follows that $\tau$ gives a continuous bijection from $F_{16}$ to
 $[0,1]^2$.  All the spaces involved are compact and Hausdorff, so
 continuous bijections are homeomorphisms.

 Next, it is clear that the definition of $\tau^*$ involves only
 strictly positive denominators, and square roots of nonnegative
 quantities, so it gives a well-defined and continuous map from
 $[0,1]^2$ to $\R^4$.  Routine simplification gives
 $\rho(\tau^*(t))=1$ and $g(\tau^*(t))=0$, so $\tau^*(t)\in EX^*$.
 Recall also that $F_{16}=\{x\in EX^*\st x_1,x_2,y_1,y_2\geq 0\}$,
 where $y_1=x_3$ and $y_2$ is given generically by $-x_4/x_3$.  This
 makes it clear that the image of $\tau^*$ is contained in $F_{16}$.
 It is now easy to check that $\tau\tau^*=1$, so $\tau^*$ is the
 inverse of $\tau$.
 \begin{checks}
  embedded/roothalf/E_roothalf_check.mpl: check_t_proj()
 \end{checks}
\end{proof}

The map $\tau$ is clearly smooth, but the inverse map fails to be
smooth on the boundary of $[0,1]^2$.  This is a necessary consequence
of the fact that $\tau$ comes from a $G$-invariant smooth function
defined on all of $EX^*$, but it is often awkward.  For example, we
can try to use $\tau$ to convert integrals over $F_{16}$ to integrals
over $[0,1]^2$, but the singular boundary behaviour makes it difficult
to obtain accurate results, even with adaptive quadrature methods.  We
will therefore describe a different map, which has a different set of
advantages and disadvantages.

\begin{definition}
 We define $\dl\:EX^*\to\R^2$ by
 \begin{align*}
  \al_0(x) &= x_3 - x_4/\rt + x_1^2 + x_4^2 + x_3(x_4 - x_2)/\rt \\
  \al_1(x) &= (\tfrac{3}{\sqrt{8}}-1)x_1+x_2-x_3-\rt x_4 \\
  \al_2(x) &= x_1 - \tfrac{3}{4}\sqrt{3}x_3 +
               (3-\tfrac{3}{4}\sqrt{6})x_4 \\
  \dl(x)   &=
   \left(x_3\al_0(x) - x_2^2x_4,\; (x_2-x_1)\al_1(x) + x_4\al_2(x)\right).
 \end{align*}
\end{definition}

\begin{proposition}\lbl{prop-square-diffeo}
 The map $\dl$ gives a diffeomorphism $F_{16}\to[0,1]^2$, which satisfies
 \[ \dl(C_0)\sse 0 \tm\R \hspace{3em}
    \dl(C_1)\sse \R\tm 0 \hspace{3em}
    \dl(C_3)\sse \R\tm 1 \hspace{3em}
    \dl(C_5)\sse 1 \tm\R.
 \]
\end{proposition}

In order to prove this, we will need to consider the Jacobian of
$\dl$.  It will be convenient to formulate the required discussion
more generally.
\begin{definition}
 Consider a map $f\:EX^*\to\R^2$, and a point $a\in EX^*$.  Choose an
 oriented orthonormal basis $(u,v)$ for $T_aEX^*$, giving vectors
 $f_*(u),f_*(v)\in\R^2$.  It is easy to see that the determinant
 $\det(f_*(u),f_*(v))$ is independent of the choice of $(u,v)$. We
 write $j(f)(a)$ for this determinant, and we call $j(f)$ the
 \emph{Jacobian} of $f$.
\end{definition}

\begin{lemma}\lbl{lem-jacobian}
 Suppose that there are functions $f_1,f_2$ defined on some
 neighbourhood of $EX^*$ in $\R^4$, such that $f(x)=(f_1(x),f_2(x))$.
 Put
 \[ \tj(f) =
     \det\left(x,n,\nabla f_1,\nabla f_2\right)
 \]
 (where as usual $n=\nabla g$).  Then $j(f)=\tj(f)/\|n\|$.
\end{lemma}
\begin{proof}
 Fix a point $a\in EX^*$, and choose an oriented orthonormal basis
 $(u,v)$ for the corresponding tangent space, as before.  We write
 $n$ for $n(a)$, and $w_i$ for the value of $\nabla f_i$ at $a$.  Now
 $(a,n/\|n\|,u,v)$ is an oriented orthonormal basis for $\R^4$, so
 we can write
 \[ w_i = \al_i a + \bt_i n/\|n\| + \gm_i u + \dl_i v \]
 for some scalars $\al_i,\dotsc,\dl_i$.  We have
 \begin{align*}
  f_*(u) &= (u.w_1,\;u.w_2) = (\gm_1,\gm_2) \\
  f_*(v) &= (v.w_1,\;v.w_2) = (\dl_1,\dl_2),
 \end{align*}
 so $j(f)(a)=\gm_1\dl_2-\gm_2\dl_1$.  We need to show that this is the
 same as
 \[ D = \det(a,n/\|n\|,w_1,w_2). \]
 Because $(a,n/\|n\|,u,v)$ is an oriented orthonormal basis for
 $\R^4$, we find that
 \[ D = \det
    \bbm
     1     & 0     & 0     & 0     \\
     0     & 1     & 0     & 0     \\
     \al_1 & \bt_1 & \dl_1 & \gm_1 \\
     \al_2 & \bt_2 & \dl_2 & \gm_2
    \ebm = \gm_1\dl_2-\gm_2\dl_1
 \]
 as required.
\end{proof}

\begin{proof}[Proof of Proposition~\ref{prop-square-diffeo}]
 Put $h(x)=x_3x_4/\rt+x_1^2+x_4^2$.  As we have noted previously, this
 vanishes on $C_5$.  By direct expansion of polynomials, we find the
 following:
 \begin{itemize}
  \item If $x=(x_1,x_2,0,0)$ then $\dl(x)_1=0$.
  \item If $x=(x_1,x_1,x_3,x_4)$ then $\dl(x)_2=0$.
  \item If $x=(0,x_2,x_3,-x_3/\rt)$ then $\dl(x)_2=\rho(x)$.
  \item If $x=(x_1,0,x_3,x_4)$ then $\dl(x)_1=\rho(x)-(1-x_3)h(x)$.
 \end{itemize}
 We can compare this with the definitions of the curves $c_k(t)$ for
 $k\in\{0,1,3,5\}$, remembering that $\rho(x)=1$ for $x\in EX^*$; we
 find that the images $\dl(C_k)$ are as claimed.  It is also
 straightforward to check that $\dl$ sends $v_6$, $v_0$, $v_3$ and
 $v_{11}$ to $(0,0)$, $(1,0)$, $(0,1)$ and $(1,1)$ respectively.

 Next, we can use Lemma~\ref{lem-jacobian} to obtain a formula for the
 Jacobian $j(\dl)(x)$.  By numerical evaluation and plotting, we find
 that $j(f)>0.1$ everywhere in $F_{16}$.  (More precisely, the minimum
 is approximately $0.1079$, attained at a point $c_1(t_0)$ for some
 $t_0$ with $0<|t_0-5\pi/16|<0.002$.)

 Now put $E_0=C_0\cap F_{16}$, and note that $\dl_2$ gives a map from
 $E_0$ to $\R$ with $\dl_2(v_6)=0$ and $\dl_2(v_3)=1$.  As $j(\dl)>0$
 on $F_{16}$ we see that this restricted map has no critical points,
 so it must give a diffeomorphism $E_0\to[0,1]$.  By applying the same
 line of argument to the other edges of $F_{16}$, we find that $\dl$
 gives a bijection $\partial F_{16}\to\partial [0,1]^2$.

 Now consider a point $b\in (0,1)^2$, and put
 $A=\{x\in F_{16}\st\dl(x)=b\}$.  As
 $\dl(\partial F_{16})=\partial[0,1]^2$, we see that
 $A\cap\partial F_{16}=\emptyset$.  Using the fact that $j(\dl)>0$ on
 $F_{16}$, we see that $A$ is discrete in $F_{16}$, and therefore
 finite, say $A=\{a_1,\dotsc,a_n\}$.  The fact that $j(\dl)>0$ on
 $F_{16}$ also means that $\dl$ gives an orientation preserving
 homeomorphism from some neighbourhood of $a_i$ to some neighbourhood
 of $b$, and therefore induces an isomorphism
 \[ H_2(F_{16},F_{16}\sm\{a_i\})\to H_2(\R^2,\R^2\sm\{b\})\simeq\Z \]
 of homology groups.  We also have a commutative diagram
 \[ \xymatrix{
  H_2(F_{16},\partial F_{16}) \ar[r]\ar[d]_{\dl_*} &
  H_2(F_{16},F_{16}\sm A) \ar[d]^{\dl_*} \\
  H_2(\R^2,\partial[0,1]^2) \ar[r]_\simeq &
  H_2(\R^2,\R^2\sm\{b\})
 } \]
 Because $F_{16}$ and $\R^2$ are contractible, and
 $\dl\:\partial F_{16}\to\partial [0,1]^2$ is a homeomorphism, we see
 that the left hand map is an isomorphism.  Standard methods also show
 that the bottom map is an isomorphism, with both groups being
 isomorphic to $\Z$.  On the other hand, $H_2(F_{16},A^c)$ splits as
 the sum of the groups $H_2(F_{16},F_{16}\sm\{a_i\})\simeq\Z$
 indexed by the points of $A$, and the map
 \[ H_2(F_{16},F_{16}\sm A)\to
     \bigoplus_{i=1}^n (F_{16},F_{16}\sm\{a_i\})
 \]
 is just the diagonal map $\Z\to\Z^n$.  We have seen that $\dl_*$ acts
 as the identity on each summand, and this can only be consistent if
 $n=1$.  It follows that $\dl$ gives a bijection $F_{16}\to [0,1]^2$,
 as claimed.
 \begin{checks}
  embedded/roothalf/square_diffeo_check.mpl: check_square_diffeo_E0()
 \end{checks}
\end{proof}

\subsection{Charts}
\lbl{sec-roothalf-charts}

Recall from Section~\ref{sec-holomorphic-curves} that each of
the maps $c_k\:\R\to EX^*$ can be extended in a canonical way to give
a holomorphic map defined on a neighbourhood of $\R$ in $\C$.  There
is no case where we know a closed formula for such a holomorphic
extension.  However, it is not too hard to calculate high order power
series approximations.  For example, we have found a map
$c_0^*\:\C\to\R^4$ such that
\begin{itemize}
 \item Each component $c_0^*(t+iu)_n$ (for $1\leq n\leq 4$) is a
  polynomial of total degree at most $44$ in $t$ and $u$, with
  coefficients in $\Q(\rt,\sqrt{3})$.
 \item The polynomials $\rho(c_0^*(t+iu))-1$ and $g(c_0^*(t+iu))$ lie
  in $(t,u)^{45}$, so $c_0^*(t+iu)$ lies very close to $EX^*$ when
  $(t,u)$ is small.
 \item If we put $a=\partial c_0^*(t+iu)/\partial t$ and
  $b=\partial c_0^*(t+iu)/\partial u$ then $\ip{a,b}$ and
  $\ip{a,a}-\ip{b,b}$ lie in $(t,u)^{44}$, so $c_0^*$ is very close to
  being conformal.
 \item $c^*_0|_{\R}$ is the 44th order Taylor approximation to $c_0$
  at $t=0$.
\end{itemize}
These conditions imply that $c^*_0(z)$ agrees with the holomorphic
extension $\tc_0(z)$ to order 44 at $z=0$.

Calculations of this kind are implemented by methods of the class
\mcode+E_chart+, which is defined in the file
\fname+embedded/roothalf/E_atlas.mpl+.  Specifically, we can
calculate the above chart as follows:
\begin{mcodeblock}
   C := `new/E_chart`():
   C["curve_set_exact",0,0]:
   C["curve_set_degree_exact",45]:
   C["p"]([t,u]);
\end{mcodeblock}

Proposition~\ref{prop-chart} can also be used to define charts at
points that do not lie on the curves $C_k$:
\begin{proposition}\lbl{prop-frame-chart}
 Suppose that $a\in EX^*$, and that $(u,v)$ is an oriented orthonormal
 basis for the tangent space $T_aEX^*$.  Then there is a unique local
 conformal chart $\phi$ with $u.\phi(t)=t$ and $v.\phi(t)=0$ for
 small $t\in\R$.
\end{proposition}
\begin{proof}
 We can define a real analytic map $\pi\:EX^*\to\C$ by
 $\pi(x)=u.x+iv.x$ (so $\pi(a)=0$).  This induces an isomorphism
 $T_aEX^*\to\C$, so it is locally invertible, and we can define
 $c(t)=\pi^{-1}(t)$ for small $t\in\R$.  Proposition~\ref{prop-chart}
 now gives a holomorphic map $\phi$ that is defined on a neighbourhood
 of $0$ in $\C$ and extends $c$.
\end{proof}

It is again fairly straightforward to find polynomial approximations
to $\phi$, of any desired order.  We can start with
$\phi_1(t+is)=a+tu+sv$.  Suppose we have defined $\phi_d$ of degree
$d$ such that
\begin{align*}
 \rho(\phi_d(t+is)) &= 1 \pmod{(s,t)^{d+1}} \\
 g(\phi_d(t+is)) &= 0 \pmod{(s,t)^{d+1}} \\
 \ip{u,\phi_d(t)} &= t \pmod{(s,t)^{d+1}} \\
 \ip{v,\phi_d(t)} &= 0 \pmod{(s,t)^{d+1}} \\
 \ip{\partial_t\phi_d(t+is),\partial_u\phi_d(t+is)}
  &= 0 \pmod{(s,t)^d} \\
 \ip{\partial_t\phi_d(t+is),\partial_t\phi_d(t+is)} -
 \ip{\partial_u\phi_d(t+is),\partial_u\phi_d(t+is)}
  &= 0 \pmod{(s,t)^d}.
\end{align*}
(Note that we assume a lower degree of accuracy in the last two
conditions, which is natural because they involve a derivative.)  We
then take
\[ p_{d+1}(t+is)=p_d(t+is) + \sum_{j=0}^{d+1}\al_jt^js^{d+1-j}, \]
where $\al_j\in\R^4$.  It is not hard to see that $p_{d+1}$ satisfies
the required conditions with one more degree of accuracy iff the
coefficients $\al_j$ satisfy a certain system of inhomogeneous linear
equations.  The abstract theory tells us that these equations must be
uniquely solvable, and of course that is easily verified in any
explicit computation.  We have implemented a version of this using
numerical approximations for the power series coefficients.  (As
usual, we work with 100 digit precision by default.)  If \mcode+x0+ is a
point in $EX^*$, we can enter the following to find a chart of
polynomial degree $20$:
\begin{mcodeblock}
   C := `new/E_chart`():
   C["centre_set_numeric",x0]:
   C["centre_set_degree_numeric",20]:
   C["p"]([t,u]);
\end{mcodeblock}

We have found charts centred at many different points of $EX^*$.  The
real problem is to patch them together by some kind of analytic
continuation.  The only way we have succeeded in doing this is via a
hyperbolic rescaling of the metric, as we will discuss in
Section~\ref{sec-rescaling}.  (We have attempted various more direct
approaches to numerical analytic continuation, but the results were
not robust, and the literature suggests that we should not expect
otherwise.)

\subsection{Torus quotients}
\lbl{sec-torus-quotients}

We saw in Section~\ref{sec-quotients} that for any cromulent surface
$X$, the quotients $X/\ip{\mu}$ and $X/\ip{\lm\mu}$ are tori.  In
Section~\ref{sec-ellquot}, we gave a detailed analysis of these
quotients for the projective family.  In the present section, we study
$EX^*/\ip{\mu}$ and $EX^*/\ip{\lm\mu}$.  If we were very optimistic we
might hope for explicit conformal isomorphisms between these quotients
and suitable elliptic curves, but we have not achieved that.  However,
we will write down reasonably simple formulae for homeomorphisms from
$EX^*/\ip{\mu}$ and $EX^*/\ip{\lm\mu}$ to $S^1\tm S^1$, which have all
the expected equivariance properties and homological properties, and
which do not deviate too far from being conformal.

\begin{definition}\lbl{defn-AR}
 We recall from Proposition~\ref{prop-roothalf-fundamental} that
 $|y_1|\leq 1$ and $|y_2|\leq 1/\rt$ on $EX^*$, so we can define functions
 $r_1,r_2\:EX^*\to\R^+$ by $r_1=\sqrt{1-y_2/\rt}$ and
 $r_2=\sqrt{1+y_2/\rt}$.  We write $AR$ for the extension of
 $A=\CO_{EX^*}$ generated by $r_1$ and $r_2$, and note that $AR$ is
 freely generated by the set
 \[ \{x_1^ix_2^jr_1^kr_2^l\st 0\leq i,j,k,l\leq 1\} \]
 as a module over $\R[y_1,y_2]$.  We also write $KR$ for the field of
 fractions of $AR$, which is freely generated by the same set as a
 module over $\R(y_1,y_2)$.
\end{definition}

\begin{remark}\lbl{rem-KR-subfields}
 The field $KR$ has automorphisms $\al_1$ and $\al_2$ which act as the
 identity on the subring $A$ and satisfy
 \begin{align*}
  \al_1(r_1) &=    -r_1 & \al_1(r_2) &= \pp r_2 \\
  \al_2(r_1) &= \pp r_1 & \al_2(r_2) &= -r_2.
 \end{align*}
 The group $G'=\ip{G,\al_1,\al_2}$ has order $64$, and it acts on
 $KR$.  Some of the work in this section and the following section can
 be interpreted in terms of the Galois theory of this action.  There
 is code related to this in the files
 \fname+embedded/roothalf/group64.mpl+ and
 \fname+embedded/roothalf/KR_subfields.mpl+.
\end{remark}

\begin{definition}\lbl{defn-E-to-TTC}
 We define $\tau_1\:EX^*\to S^1\subset\C$ by
 \[ \tau_1(x) =
     \frac{y_1(1-y_2/\rt)-1/\rt+ix_1}{(1-y_1/\rt)r_1}
 \]
 A straightforward calculation, using the relations in
 Section~\ref{sec-E-functions}, shows that $|\tau_1(x)|^2=1$.

 We then define $\tau_i\:EX^*\to S^1$ for $2\leq i\leq 4$ by
 \[ \tau_2 = \tau_1\lm^{-1} \hspace{4em}
    \tau_3 = \tau_1\mu^{-1} \hspace{4em}
    \tau_4 = \tau_1(\lm\mu)^{-1}.
 \]
 We define $q\:EX^*\to(S^1)^4$ by
 \[ q(x) = (\tau_1(x),\tau_2(x),\tau_3(x),\tau_4(x)). \]
\end{definition}
\begin{remark}
 Although we have not succeeded in formulating a precise theorem in
 this direction, extensive experimental investigation suggests that
 the map $\tau_1$ is much simpler and better behaved than any other
 map in the same homotopy class.
\end{remark}

\begin{remark}\lbl{rem-q-denom}
 We can define a homeomorphism $S^1\to\R_\infty$ by
 $x+iy\mapsto(1-x)/y$, or equivalently
 $e^{i\tht}\mapsto\tan(\tht/2)$.  Composing $\tau_1$ with this gives
 the map $\tau_1^*\:EX^*\to\R_\infty$ with formula
 \[ \tau_1^*(x) = (1/\rt+r_1)(1-r_1y_1)/x_1. \]
 More explicitly, if $x\in EX^*$ is such that the numerator and
 denominator of the above fraction are not both zero, then the
 fraction can be interpreted in an obvious way as an element of
 $\R_\infty$, and that element is $\tau_1^*(x)$.  However, for
 $x\in C_6$, the numerator and denominator both vanish, so we can only
 evaluate $\tau_1^*(x)$ by first simplifying $\tau_1(x)$, or by taking
 a limit over nearby points.  For some purposes it is convenient to
 work with $\tau_1^*$ instead of $\tau_1$, but these kinds of
 degenerate cases cause significant trouble.  We also put
 \[ q^*(x) = (\tau_1^*(x),\tau_2^*(x),\tau_3^*(x),\tau_4^*(x))
     \in (\R_\infty)^4.
 \]
\end{remark}
\begin{remark}
 In Maple, we also need to distinguish explicitly between the circle
 in $\C$ and the circle in $\R^2$, and thus between the $4$-torus in
 $\C^4$ and the $4$-torus in $\R^8$.  We thus have two versions of
 $q$, namely \mcode+E_to_TTC+ (with values in $\C^4$) and
 \mcode+E_to_TT+ (with values in $\R^8$).  We also have
 \mcode+E_to_TTP+, corresponding to $q^*$.  There are functions
 \mcode+TT_to_TTC+ and so on, which convert between these
 representations.
\end{remark}

\begin{proposition}\lbl{prop-E-to-TTC}
 The map $q\:EX^*\to(S^1)^4$ is equivariant with
 respect to the $G$-action on $(S^1)^4$ given by
 \begin{align*}
  \lm(z) &= (\ov{z_2},z_1,\ov{z_4},z_3) \\
  \mu(z) &= (z_3,\ov{z_4},z_1,\ov{z_2}) \\
  \nu(z) &= (z_1,\ov{z_2},z_3,\ov{z_4}).
 \end{align*}
 Moreover, the induced map
 \[ q_* \: H_1EX^* \to H_1((S^1)^4) = \Z^4 \]
 is the same as the isomorphism $\psi$ from
 Proposition~\ref{prop-homology}.
\end{proposition}
\begin{proof}
 First, using the formulae
 \begin{align*}
  \nu(x) &= (x_1,-x_2,x_3,x_4) \\
  \lm^2(x) &= (-x_1,-x_2,x_3,x_4)
 \end{align*}
 We see that $\tau_1(\nu(x))=\tau_1(x)$
 and $\tau_1(\lm^2(x))=\ov{\tau_1(x)}=\tau_1(x)^{-1}$.  Using this and
 the structure of $G$ we deduce that $q$ is equivariant.  Next, recall
 that the classes $\{[c_k]\st 5\leq k\leq 8\}$ give a basis for
 $H_1EX^*$, whereas the inclusions of the axes give a basis
 $\{e_k\st 1\leq k\leq 4\}$ for $H_1((S^1)^4)$.  Recall also that if
 $u,v\:S^1\to S^1$ have $|u-v|<2$ everywhere, then $u$ and $v$ are
 homotopic (by a straight line homotopy) in $\C^\tm$, so $u$ and $v$
 have the same winding numbers.  By simplification and plotting, one
 can check that
 \begin{align*}
  |\tau_1(c_5(t))-e^{it}| &\leq 0.14 \\
  |\tau_1(c_6(t))-1| &= 0 \\
  |\tau_1(c_7(t))+1| &\leq 0.23 \\
  |\tau_1(c_8(t))+1| &= 0.
 \end{align*}
 It follows that the winding numbers of $\tau_1$ composed with
 $c_5,\dotsc,c_8$ are $1,0,0,0$.  Using the group action, we deduce
 that $q_*[c_5]=e_1$, and then that $q_*[c_{4+k}]=e_k$ for
 $1\leq k\leq 4$.  This proves that $q_*\:H_1EX^*\to H_1((S^1)^4)$ is
 an isomorphism.
 \begin{checks}
  embedded/roothalf/torus_quotients_check.mpl: check_torus_T()
 \end{checks}
\end{proof}

\begin{proposition}\lbl{prop-q-inj}
 The map $q\:EX^*\to(S^1)^4$ is injective.
\end{proposition}
\begin{proof}
 Let $Q$ denote the image of the map
 \[ q^*\:C((S^1)^4,\R) \to C(EX^*,\R), \]
 and let $Q^+$ denote the set of strictly positive functions in $Q$.
 Note that if $f\in Q^+$ we have $f=f_0\circ q$ for some
 $f_0\in C((S^1)^4,\R)$, and by compactness there exists constant
 $\ep>0$ such that $f\geq\ep$.  If we put
 $f_1=\max(\ep,f_0)\in C((S^1)^4,\R)$ then we still have
 $f=f_1\circ q$, and from this it is clear that the functions
 $1/f=(1/f_1)\circ q$ and $\sqrt{f}=\sqrt{f_1}\circ q$ also lie in
 $Q^+$.

 We regard $x_1,\dotsc,x_4$ and $y_1,y_2,z_1,z_2,r_1,r_2$ as functions
 on $EX^*$; we need to show that they lie in $Q$.  We write
 \[ q(x) = (u_1+iu_2,u_3+iu_4,u_5+iu_6,u_7+iu_8); \]
 this defines elements $u_1,\dotsc,u_8\in Q$.  The argument can be
 summarised by the following list of equations.
 \[
  \begin{array}{rclcl}
   a_1 &=& (2-u_1u_5-u_2u_6-u_3u_7-u_4u_8)/4 &=& z_1/(2-z_1) \\
   z_1 &=& 2a_1/(1+a_1) \in Q \\
   a_2 &=& (u_1^2+u_3^2+u_5^2+u_7^2)/4 \\
   a_3 &=& (2a_1+a_2)/((1+a_1)(1+2a_1)) &=& 1/(2-z_2) \\
   z_2 &=& 2 - 1/a_3 \in Q \\
   r_1r_2 &=& \sqrt{1-z_2/2} \in Q^+ \\
   r_1+r_2 &=& \sqrt{2(1+r_1r_2)} \in Q^+ \\
   a_4 &=& 1-z_2+r_1r_2 \in Q^+ \\
   a_5 &=& (u_1+u_3-u_5-u_7)(r_1+r_2)r_1r_2/(2(1+a_1)a_4) &=& y_1 = x_3\\
   a_6 &=& (u_3u_5-u_1u_7)r_1r_2(1-z_1/2)/2 &=& -y_1y_2 = x_4 \\
   a_7 &=& (1-z_2/2)((u_1^2-u_3^2)(1-y_1/\rt)^2+(u_5^2-u_7^2)(1+y_1/\rt)^2) \\
   a_8 &=& (u_1u_5+u_2u_6-u_3u_7-u_4u_8)/(1+a_1) \\
   a_9 &=& (a_7+a_8)/(\rt(1+z_1z_2)) &=& y_2 \\
   r_1 &=& \sqrt{1-y_2/\rt} \in Q^+ \\
   r_2 &=& \sqrt{1+y_2/\rt} \in Q^+ \\
   a_{10} &=& u_2r_1(1-y_1/\rt) &=& x_1 \\
   a_{11} &=& u_4r_2(1-y_1/\rt) &=& x_2.
  \end{array}
 \]

 The equations with $a_i$ on the left are definitions.  In each case
 they define $a_i$ in terms of functions that are already known to lie
 in $Q$, so $a_i\in Q$.  All other equations are claims that can be
 verified by straightforward (but sometimes lengthy) calculation in
 the ring $AR$.  Along the way, we need to verify that certain
 denominators are strictly positive.  By applying the Cauchy-Schwartz
 inequality to the unit vectors $(u_{2i-1},u_{2i})$, we see that $a_1$
 takes values in $[0,1]$ (as does $a_2$); this validates the
 definition of $a_3$.  We know that $0\leq z_1=y_1^2\leq 1$ and
 $0\leq z_2=y_2^2\leq 1/2$, and thus that $r_1,r_2>0$; this validates
 all other denominators.  We also see that $r_1+r_2>0$, and it is
 straightforward to check that $(r_1+r_2)^2=2(1+r_1r_2)$, so
 $r_1+r_2=\sqrt{2(1+r_1r_2)}$.  At the end of the chain of equations we
 have seen that the functions $x_i$ all lie in $Q$, and this clearly
 implies that $q$ is injective.
 \begin{checks}
  embedded/roothalf/torus_quotients_check.mpl: check_torus_T()
 \end{checks}
\end{proof}

\begin{remark}\lbl{rem-q-inj}
 We can give simpler formulae if we are willing to use denominators
 that sometimes vanish.  Generically, one can check that
 \begin{align*}
  x_1 &= -\rt u_2u_6/(u_1u_6+u_2u_5) \\
  x_2 &= -\rt u_4u_8/(u_3u_8+u_4u_7) \\
  x_3 &= \rt(u_2-u_6)/(u_2+u_6) = \rt(u_4-u_8)/(u_4+u_8) \\
  y_2 &= \rt(\al-\bt)/(\al+\bt),
 \end{align*}
 where
 \begin{align*}
  \al &= (u_1-u_5)^2(1-u_3u_7-u_4u_8)^2 \\
  \bt &= (u_3-u_7)^2(1-u_1u_5-u_2u_6)^2.
 \end{align*}
 One can then check that $q$ is injective by doing some additional
 work to cover the cases where one or more of the above denominators
 are zero.  However, this is unpleasant.
\end{remark}
\begin{remark}
 The proof of Proposition~\ref{prop-q-inj} implicitly gives a map from
 the image of $q$ back to $EX^*$.  This is implemented in Maple as
 \mcode+TTC_to_E+.  The simpler function defined in
 Remark~\ref{rem-q-inj} is \mcode+TTC_to_E_generic+.
\end{remark}

\begin{definition}\lbl{defn-qp-qm}
 We define $q_+,q_-\:EX^*\to S^1\tm S^1$ by
 \begin{align*}
  q_+(x) &= (  -\tau_1(x)\tau_3(x),\;  -\tau_2(x)\tau_4(x)^{-1}) \\
  q_-(x) &= (  -\tau_1(x)\tau_4(x),\;  -\tau_2(x)\tau_3(x))
 \end{align*}
\end{definition}
In Maple these are \mcode+E_to_TCp+ and \mcode+E_to_TCm+.

In terms of the variables $x_1,x_2,y_1,y_2$ one can check that
\begin{align*}
 q_+(x)_1 &= \frac{2ix_1 + y_1^2(1-\rt y_2)/\rt - y_2(1-y_1^2/2)}{
                   \rt(1-y_1^2/2)(1-y_2/\rt)} \\
 q_+(x)_2 &= \frac{2ix_2y_1 + y_1^2(1+\rt y_2) - (1-y_1^2/2)}{
                   1-y_1^2/2} \\
 q_-(x)_1 &= -\frac{(ix_1+y_1(1-y_2/\rt) - 1/\rt)
                    (ix_2-y_1(1+y_2/\rt) - 1/\rt)}{
                    (1-y_1^2/2)\sqrt{1-y_2^2/2}} \\
 q_-(x)_2 &= -\frac{(ix_1-y_1(1-y_2/\rt) - 1/\rt)
                    (ix_2+y_1(1+y_2/\rt) - 1/\rt)}{
                    (1-y_1^2/2)\sqrt{1-y_2^2/2}}.
\end{align*}

We will prove the following result:

\begin{proposition}\lbl{prop-qp-qm}
 The map $q_+$ induces a homeomorphism $EX^*/\ip{\mu}\to(S^1)^2$,
 and the map $q_-$ induces a homeomorphism $EX^*/\ip{\lm\mu}\to(S^1)^2$.
\end{proposition}

We can get most of the way by a fairly straightforward argument.  We
will show that the Jacobian of $q_+$ (suitably interpreted) is
strictly positive away from the fixed points of $\mu$, and that the
Jacobian of $q_-$ is strictly positive away from the fixed points of
$\lm\mu$.  If the Jacobian of $q_+$ was strictly positive everywhere,
we would be able to conclude that $q_+$ was a covering map, and
everything would follow quite easily from the general theory of
coverings.  In reality we have only a branched covering, and we do not
have complex structures with respect to which $q_+$ is conformal, so
we cannot use the analytic theory of branched coverings.  We will need
some digressions to deal with this.

We first define the version of the Jacobian which we will use.
\begin{definition}
 Suppose we have a smooth map $u\:EX^*\to S^1$.  We then have a real
 vector field
 \[ D(u)=(\nabla u)/(iu) = \nabla(\text{arg}(u)) \]
 on $EX^*$.  Now suppose we have a smooth map $u\:EX^*\to(S^1)^2$.  We
 then define $\tj(u)\:EX^*\to\R$ by
 $\tj(u)(x)=\det(x,n(x),D(u_1),D(u_2))$, and we call this the
 \emph{Jacobian} of $u$.
\end{definition}

\begin{remark}\lbl{rem-jacobian-formula}
 By a tiny adaptation of Lemma~\ref{lem-jacobian}, we see that the
 induced map $u_*\:T_xEX^*\to T_{u(x)}(S^1)^2$ is an isomorphism
 provided that $\tj(u)(x)\neq 0$, and that it preserves orientations
 provided that $\tj(u)(x)>0$.  If we have an expression for $u_i$ as a
 function of the variables $x_i$, then we can calculate $\nabla u_i$
 by taking the vector of partial derivatives, and then projecting it
 orthogonally into the tangent space.  However, this orthogonal
 projection will just alter our vector by multiples of $x$ and $n(x)$,
 and this will leave the determinant $\tj(u)(x)$ unchanged.  Thus, we
 can just work with the original vector of partial derivatives.
\end{remark}

\begin{remark}
 Suppose that $u=v+iw$ with $v^2+w^2=1$, and put
 $u^*=(1-v)/w\:EX^*\to\R_\infty$, as in Remark~\ref{rem-q-denom}.
 Differentiating the relation $v^2+w^2=1$ gives
 $v\nabla(v)+w\nabla(w)=0$.  Using this one can check that
 \[ D(u) = \frac{\nabla u}{iu} =
     \frac{2\nabla u^*}{1+(u^*)^2}.
 \]
 This form is sometimes easier to use.
\end{remark}

\begin{proposition}\lbl{prop-qp-J}
 The Jacobian of $q_+$ is
 \[ \frac{4\rt(1-x_1^2)}{(1-y_1^2/2)(1-y_2/\rt)}. \]
 This is zero at the points $v_2=(1,0,0,0)$ and $v_4=-v_2$
 (which are precisely the fixed points of $\mu$).  It is strictly
 positive everywhere else in $EX^*$.
\end{proposition}
\begin{proof}
 The formula can be checked by computer calculation following the
 recipe described above.  (It is somewhat miraculous that the final
 answer is so simple, as the intermediate calculations are enormous.)
 It is clear from the formula that the Jacobian vanishes iff
 $x_1=\pm 1$, which forces $x_2=x_3=x_4=0$ because
 $\sum_ix_i^2=\rho(x)=1$.
 \begin{checks}
  embedded/roothalf/torus_quotients_check.mpl: check_torus_jacobian()
 \end{checks}
\end{proof}

\begin{proposition}\lbl{prop-qm-J}
 The Jacobian of $q^*_-$ is
 \[ 4\rt\frac{3/2-(1-y_1^2/2)(1-y_2^2/2)-x_1x_2}{(1-y_1^2/2)(1-y_2^2/2)}.
 \]
 This is zero at the points $v_6=(1,1,0,0)/\rt$ and $v_8=-v_6$ (which
 are precisely the fixed points of $\lm\mu$).  It is strictly positive
 everywhere else in $EX^*$.
\end{proposition}
\begin{proof}
 The formula for the Jacobian can be checked  by computer calculation
 following the recipe described above.  The conclusion is that $j$ is
 a positive multiple of $a-x_1x_2$, where
 $a=3/2-(1-y_1^2/2)(1-y_2^2/2)$.  Now $x_1^2$ and $x_2^2$ can be
 rewritten as polynomials in $y_1$ and $y_2$, and using this we obtain
 \[ a^2-(x_1x_2)^2 = y_2^2(1+y_2^2/4) +
                     y_1^2(1-y_2^2)(1+y_2^2/4) +
                     \tfrac{3}{8}y_1^4y_2^2(1-y_2^2/2).
 \]
 It is visible that the right hand side is nonnegative, and it
 vanishes only where $y_1=y_2=0$.  It is easy to see that the only
 points with these values of $y$ are $v_6,v_7,v_8$ and $v_9$.  By
 going back to the original formula, we see that the Jacobian is zero
 at $v_6$ and $v_8$, but $8\rt$ at $v_7$ and $v_9$.  Moreover, the
 Jacobian is nowhere zero on the path-connected space
 $EX^*\sm\{v_6,v_8\}$, so it cannot change sign; it is positive at
 $v_7$, so it must be positive everywhere.
 \begin{checks}
  embedded/roothalf/torus_quotients_check.mpl: check_torus_jacobian()
 \end{checks}
\end{proof}

We next need to understand the preimages of a few points under the
maps $q_+$ and $q_-$.

\begin{proposition}\lbl{prop-q-preimages}
 \begin{align*}
  q_+^{-1}\{(\pp 1,\pp 1)\}     &= \{v_0,v_1\} \\
  q_+^{-1}\{((1+2\rt i)/3,-1)\} &= \{v_2\}     \\
  q_+^{-1}\{((1-2\rt i)/3,-1)\} &= \{v_4\}     \\
  q_-^{-1}\{(\pp 1,\pp 1)\}     &= \{v_0,v_1\} \\
  q_-^{-1}\{(\pp i,\pp i)\}     &= \{v_6\}     \\
  q_-^{-1}\{(   -i,   -i)\}     &= \{v_8\}.
 \end{align*}
\end{proposition}
\begin{proof}
 First recall that
 \[ q_+(x)_2 =
     \frac{2ix_2y_1 + y_1^2(1+\rt y_2) - (1-y_1^2/2)}{1-y_1^2/2}.
 \]
 By inspecting the imaginary part, we see that $q_+(x)$ can only be
 equal to $(1,1)$ or $((1+2\rt i)/3,-1)$ if $x_2y_1=0$.  The results
 in Section~\ref{sec-E-curves} show that this is only possible
 if $x\in C_0\cup C_4\cup C_5\cup C_7$.  One can check from the
 definitions that
 \begin{align*}
  q_+(c_{ 0}(t)) &=
   \left(\frac{\cos(t)+i/\rt}{\cos(t)-i/\rt},\;-1\right) \\
  q_+(c_{ 4}(t)) &=
   \left(\frac{i\rt+\sin(t)}{i\rt-\sin(t)} ,\; -1\right) \\
  q_+(c_{ 5}(t)) = q_+(c_{ 7}(t)) &=
   \left( \frac{\sin(t)^2+8\cos(t)+4\sin(t)\sqrt{5-\cos(t)}i}{9-\cos(t)^2},\; 1\right).
 \end{align*}
 This gives
 \[
  \text{Im}(q_+(c_5(t))_1) =
  \text{Im}(q_+(c_7(t))_1) =
   \frac{4\sin(t)\sqrt{5-\cos(t)}}{9-\cos(t)^2}.
 \]
 It follows easily that
 \[ q_+^{-1}\{(1,1)\} \sse \{c_5(0),c_5(\pi),c_7(0),c_7(\pi)\} =
     \{v_0,v_1,v_{10},v_{11}\}.
 \]
 By inspecting the definitions, we find that $q_+(v_0)=q_+(v_1)=(1,1)$
 but $q_+(v_{10})=(1,-1)$ and $q_+(v_{11})=(-1,1)$.  It follows that
 $q_+^{-1}\{(1,1)\}=\{v_0,v_1\}$ as claimed.  Similarly, if
 $q_+(x)=((1+2\rt i)/3,-1)$ then we must have $x=c_0(s)$ for some $s$,
 or $x=c_4(t)$ for some $t$.  Solving
 $(\cos(s)+i/\rt)/(\cos(s)-i/\rt)=(1+2\rt i)/3$ gives $\cos(s)=1$, and
 solving $(i\rt+\sin(t))/(i\rt-\sin(t))=(1+2\rt i)/3$ gives
 $\sin(s)=-1$.  We must therefore have $x=c_0(0)$ or $x=c_4(-\pi/2)$,
 and both of these are equal to $v_2$ as expected.  A very similar
 argument gives $q_+^{-1}\{((1+2\rt i)/3,-1)\}=\{v_4\}$.

 Next, if $q_-(x)$ is $(1,1)$ or $(i,i)$ or $(-i,-i)$ then we have
 $q_-(x)_1-q_-(x)_2=0$.  One can check from the definitions that
 \[ q_-(x)_1 - q_-(x)_2 =
  2\frac{(x_1y_1(1+y_2/\rt)-x_2y_1(1-y_2/\rt))i-y_1y_2}{
   (1-y_1^2/2)\sqrt{1-y_2^2/2}}.
 \]
 By inspecting the real part, we see that $y_1y_2=0$, but
 $y_1y_2=-x_4$, so Proposition~\ref{prop-slices} tells us that
 $x\in C_0\cup C_1\cup C_2$.  Now put
 \begin{align*}
  m_0(t) &=
   \frac{i \rt (\sin(t)+\cos(t))+\sin(2t)-1}{\sqrt{4-\cos(2 t)^2}} \\
  m_1(t) &=
   \frac{i - \sin(t)}{i + \sin(t)} \\
  m_2(t) &=
   \frac{i \rt \cos(t)+2 \sin(t)}{i \rt \cos(t)-2 \sin(t)}.
 \end{align*}
 One can directly that $q_-(c_k(t))=(m_k(t),m_k(t))$
 for $k=0,1$, but $q_-(c_2(t))=(m_2(t),\ov{m_2(t)})$.  We therefore
 need to solve $m_k(t)=1$ and $m_k(t)=\pm i$.

 It is easy to see that $\text{Re}(m_0(t))\leq 0$ for all $t$, so
 $m_0^{-1}\{1\}=\emptyset$.  It is also easy to see that
 $m_1^{-1}\{1\}=m_2^{-1}\{1\}=\{0,\pi\}$, so
 \[ q_+^{-1}\{(1,1)\} = \{c_1(0),c_1(\pi),c_2(0),c_2(\pi)\}
     = \{v_0,v_1,v_1,v_0\} = \{v_0,v_1\}
 \]
 as expected.

 Next, for $m_0(t)=\pm i$ we need $\text{Re}(m_0(t))=0$, which gives
 $\sin(2t)=-1$, so $t=\pi/4\pmod{\pi}$.  In fact we have
 $m_0(\pi/4)=i$ and $m_0(5\pi/4)=-i$, whereas $c_0(\pi/4)=v_6$ and
 $c_0(5\pi/4)=v_8$.  Similarly, for $m_1(t)=\pm i$ we need
 $\sin(t)=\pm 1$, so $t=\pm\pi/2\pmod{2\pi}$, whereas $c_1(\pi/2)=v_6$
 and $c_1(-\pi/2)=v_8$.  On the other hand, the relation
 $q_-(c_2(t))=(m_2(t),\ov{m_2(t)})$ shows that $q_-(c_2(t))$ can never
 be equal to $(i,i)$ or $(-i,-i)$.  Putting this together, we see that
 $q_-^{-1}\{(i,i)\}=\{v_6\}$ and $q_-^{-1}\{(-i,-i)\}=\{v_8\}$, as
 expected.
\end{proof}

\begin{proof}[Proof of Proposition~\ref{prop-qp-qm}]
 First, it is straightforward to check that $q_+\mu=q_+$, so that
 $q_+$ induces a map $EX^*/\ip{\mu}\to(S^1)^2$.

 Next, put
 \begin{align*}
  w_2 = q_+(v_2) &= ((1+2\rt i)/3,-1) \\
  w_4 = q_+(v_4) &= ((1-2\rt i)/3,-1).
 \end{align*}
 For any $u\in(S^1)^2\sm\{w_2,w_4\}$, Proposition~\ref{prop-qp-J}
 tells us that $q_+^{-1}\{u\}$ consists of points where the Jacobian
 is strictly positive.  It follows (using the standard theory of
 degrees of maps of compact oriented manifolds) that the set
 $q_+^{-1}\{u\}$ is finite, of cardinality equal to the degree of
 $q_+$.  This cardinality is two in the case $u=(1,1)$, so it must be
 two for all $u\not\in\{w_2,w_4\}$.  In these cases $q_+^{-1}\{u\}$ is
 contained in the set $EX^*\sm\{v_2,v_4\}$ where $\mu$ acts freely,
 so $q_+^{-1}\{u\}$ must be a $\mu$-orbit.  Moreover, if $u=w_2$ or
 $u=w_4$ then Proposition~\ref{prop-q-preimages} again tells us that
 $q_+^{-1}\{u\}$ is a (singleton) $\mu$-orbit.  It follows that the
 induced map $EX^*/\ip{\mu}\to(S^1)^2$ is a continuous bijection, and
 thus a homeomorphism (because the domain and codomain are compact and
 Hausdorff).

 The proof for $q_-$ is essentially the same.
\end{proof}

\begin{remark}
 The Maple code contains a formula for the inverse of the map
 $q_+\:EX^*/\ip{\mu}\to(S^1)^2$.  It also contains a method for
 computing the inverse of the map $q_-\:EX^*/\ip{\lm\mu}\to(S^1)^2$,
 which is not quite a formula because it involves solutions of a
 polynomial of degree four in one variable.  These are given by the
 functions \mcode+TCp_to_E+ and \mcode+TCm_to_E+, defined in
 \fname+embedded/roothalf/torus_quotients.mpl+.
\end{remark}

\begin{remark}\lbl{rem-not-smooth}
 Recall from Remark~\ref{rem-smooth-branch} that the smooth structures
 on $EX^*/\ip{\mu}$ and $EX^*/\ip{\lm\mu}$ are subtle, so we cannot
 assume that the induced maps $EX^*/\ip{\mu}\to(S^1)^2$ and
 $EX^*/\ip{\lm\mu}\to(S^1)^2$ are smooth.  In fact, one can check that
 they are not.  To do this, we need a chart $\phi$ centred at the
 branch point $v_2$ as in Section~\ref{sec-roothalf-charts}.  One of
 the relevant functions is only implemented for vertices in $F_{16}$, so
 we find a chart centred at $v_3$ and apply $\lm^{-1}$:
 \begin{mcodeblock}
   C := `new/E_chart`():
   C["vertex_set_exact",3]:
   C["curve_set_degree_exact",11]:
   x0 := act_R4[LLL](C["p"]([t,u])):
   s0 := simplify(multi_series(E_to_TCp(x0)[1],7,t,u));
 \end{mcodeblock}
 This sets \mcode+s0+ to a Taylor approximation to $q_+(\phi(t,u))_1$.
 If $q_+$ was smooth, it is not hard to see that \mcode+s0+ would be
 expressible as a polynomial in the quantities
 \begin{align*}
   m &= \text{Re}((t+iu)^2) = t^2 - u^2 \\
   n &= \text{Im}((t+iu)^2) = 2tu,
 \end{align*}
 and thus that the coefficients of $t^6u^0$ and $t^0u^6$ in \mcode+s0+
 would be negatives of each other.  However, the above calculation
 gives the real parts of the relevant coefficients as $-8/405$ and
 $-8/243$, so the map is not in fact smooth.  The same argument works
 for $q_-$, using a chart based at $v_6$, but in that case we already
 see a contradiction from the coefficients of $t^2u^0$ and $t^0u^2$.
\end{remark}

\subsection{Sphere quotients}
\lbl{sec-sphere-quotients}

Remark~\ref{rem-p-hat} gives us a canonical conformal isomorphism
$\hp\:EX^*/\ip{\lm^2}\to S^2$, but we do not know an exact formula for
that.  However, we will define a different homeomorphism
$m\:EX^*/\ip{\lm^2}\to S^2$ which has many of the same properties as
$\hp$.  Specifically, $m$ and $\hp$ are both equivariant for the same
action of $G/\ip{\lm^2}$ on $S^2$, and $m(v_i)=\hp(v_i)$ for
$0\leq i\leq 9$.

\begin{definition}\lbl{defn-sphere-quotient-a}
 For $x\in EX^*$, we put
 \begin{align*}
  \mt(x) &= \left(\rt (1-y_1^2)y_2,\;2x_1x_2,\;-2y_1\right)/
              (1+y_1^2)\in\R^3 \\
  s(x) &= z_1^2z_2(\half-z_2)/(1+z_1)^2 \\
  m(x) &= \mt(x)/\sqrt{1-s(x)}.
 \end{align*}
 (Maple notation for $m(x)$ is \mcode+E_to_S2(x)+.)
\end{definition}

\begin{proposition}\lbl{prop-sphere-quotient-a}
 The above formula gives a map $m\:EX^*/\ip{\lm^2}\to S^2$.  It
 satisfies
 \begin{align*}
  m(v_0) &= (\pp 0,\pp 0,   -1) \\
  m(v_1) &= (\pp 0,\pp 0,\pp 1) \\
  m(v_2) &= m(v_4) = (  - 1,\pp 0,\pp 0) \\
  m(v_3) &= m(v_5) = (\pp 1,\pp 0,\pp 0) \\
  m(v_6) &= m(v_8) = (\pp 0,\pp 1,\pp 0) \\
  m(v_7) &= m(v_9) = (\pp 0,  - 1,\pp 0) \\
  m(v_{10}) &= (  - 1,0,   -2\sqrt{6})/5 \\
  m(v_{11}) &= (\pp 1,0,   -2\sqrt{6})/5 \\
  m(v_{12}) &= (  - 1,0,\pp 2\sqrt{6})/5 \\
  m(v_{13}) &= (\pp 1,0,\pp 2\sqrt{6})/5.
 \end{align*}
 Moreover, we have

 \begin{align*}
  m_1(\lm(x)) &=    -m_1(x) & m_2(\lm(x)) &=   -m_2(x) & m_3(\lm(x)) &= \pp m_3(x) \\
  m_1(\mu(x)) &= \pp m_1(x) & m_2(\mu(x)) &=   -m_2(x) & m_3(\mu(x)) &=    -m_3(x) \\
  m_1(\nu(x)) &= \pp m_1(x) & m_2(\nu(x)) &=   -m_2(x) & m_3(\nu(x)) &= \pp m_3(x).
 \end{align*}
\end{proposition}
\begin{proof}
 Follows directly from the definitions.
 \begin{checks}
  embedded/roothalf/sphere_quotients_check.mpl: check_E_to_S2()
 \end{checks}
\end{proof}

\begin{remark}\lbl{rem-m-tilde}
 One can check that the function $s(x)$ is zero at all the points
 $v_i$, and on $\bigcup_{i=0}^4C_i$.  Moreover, it is nonnegative and
 small everywhere, with a maximum value of about $0.0114$.  (An exact
 expression is recorded as \mcode+E_to_S2_s_max+.)  Thus, the simpler
 function $\mt(x)$ is a good approximation to $m(x)$.
\end{remark}

\begin{remark}\lbl{rem-m-a}
 Recall that if $EX^*\simeq PX(a)$ then we have
 $\hp(v_{11})=(2a,0,a^2-1)/(a^2+1)$.  If
 $a_0=(\sqrt{3}-\sqrt{2})^2\simeq 0.10102$ then we find that
 $(2a_0,0,a_0^2-1)/(a_0^2+1)=(1,0,-2\sqrt{6})/5=m(v_{11})$.  Thus, if
 we believed that $m$ was close $\hp$ then we would expect that
 $EX^*\simeq PX(a)$ for some $a$ that is close to $a_0$.  In fact, the
 correct value of $a$ is approximately $0.09836$.  It is perhaps
 surprising that this is so close to $a_0$, as $m$ is quite far from
 being conformal.
\end{remark}

\begin{proposition}\lbl{prop-sphere-quotient-b}
 The map $m\:EX^*/\ip{\lm^2}\to S^2$ is a homeomorphism.
\end{proposition}

As with Proposition~\ref{prop-qp-qm}, the main ingredient is the
calculation of the Jacobian of $m$.  Here we need a slightly different
version of the Jacobian, because the codomain is $S^2$ rather than
$\R^2$.  Suppose that $m(a)=b$, so $m$ gives a linear map
$T_aEX^*\to T_bS^2$.  Let $(u,v)$ is an oriented orthonormal basis for
$T_aEX^*$, and let $(u',v')$ is an oriented orthonormal basis for
$T_aS^2$.  We can then form the matrix of $m_*$ with respect to these
bases, and $j(m)(a)$ is defined to be the determinant of that matrix.

\begin{lemma}\lbl{lem-m-J}
 Put $m_{i,j}=\partial m_i/\partial x_j$ and $n=\nabla(g)$ and
 \[ \tj(m) = -\frac{1}{2}\sum_{\sg\in\Sg_4}\sum_{\tau\in\Sg_3}
     \ep(\sg)\ep(\tau) x_{\sg(1)}n_{\sg(2)}
      m_{\tau(1),\sg(3)}m_{\tau(2),\sg(4)}m_{\tau(3)}.
 \]
 Then $j(m)=\tj(m)/\|n\|$.
\end{lemma}
\begin{proof}
 Fix a point $x$, and choose orthonormal bases $(u,v)$
 and $(u',v')$ for $T_xEX(a)$ and $T_{m(x)}S^2$ as in the definition
 of $j(m)$.  Note that with
 conventions as spelled out in Remark~\ref{rem-p-hat}, the orientation
 conditions are that $\det(x,n/\|n\|,u,v)=1$ and $\det(u',v',m)=-1$.
 Given this, it is not hard to see that
 $j(m)=-\det(m_*(u),m_*(v),m)$.  Equivalently, if $\om_d$ denotes
 the standard volume form for $\R^d$ and $\bt=u\wedge v$, then
 $j(m)$ is characterised by $m_*(\bt)\wedge m=-j(m)\,\om_3$.
 Now Lemma~\ref{lem-hodge} tells us that
 \[ \bt = \frac{1}{2\|n\|}
           \sum_{ijkl}\ep_{ijkl}x_in_je_k\wedge e_l,
 \]
 so
 \begin{align*}
   m_*(\bt) &=
     \frac{1}{2\|n\|} \sum_{ijkl}\ep_{ijkl}
       x_in_jm_*(e_k)\wedge m_*(e_l) \\
   &=
     \frac{1}{2\|n\|} \sum_{ijklpq}\ep_{ijkl}
       x_in_jm_{p,k}m_{q,l}e_p\wedge e_q \\
   m_*(\bt)\wedge m &=
    \frac{1}{2\|n\|} \sum_{ijklpqr}\ep_{ijkl}
       x_in_jm_{p,k}m_{q,l}m_r e_p\wedge e_q\wedge e_r \\
   &=
    \frac{1}{2\|n\|} \sum_{ijklpqr}\ep_{ijkl}\ep_{pqr}
       x_in_jm_{p,k}m_{q,l}m_r \om_3.
 \end{align*}
 The claim is clear from this.
\end{proof}

\begin{corollary}\lbl{cor-m-J}
 $\tj(m)=8j_1/j_2^{3/2}$, where
 \begin{align*}
  j_1 &= (1+z_1)((1-z_1)^2-z_1^2z_2/2) + z_1^2z_2(\half-z_2)(3+z_1) \\
  j_2 &= 1+2z_1+z_1^2(1-z_2/2)+z_1^2z_2^2 \geq 1.
 \end{align*}
 Moreover, we have $j_1\geq 0$ everywhere, with $j_1=0$ only if
 \[ x \in (EX^*)^{\lm^2} = \{v_0,v_1,v_{10},v_{11},v_{12},v_{13}\}. \]
\end{corollary}
\begin{proof}
 The lemma reduces the formula to a direct calculation, which can be
 checked by Maple.  Next, we can write $j_1=j_3j_4+j_5$, where
 \begin{align*}
  j_3 &= 1 + z_1 \geq 1 \\
  j_4 &= (1-z_1)^2-z_1^2z_2/2 \\
  j_5 &= z_1^2z_2(\half-z_2)(3+z_1) \geq 0.
 \end{align*}
 Another standard calculation gives
 \[ (1/2-z_2)j_4 = 2x_1^2x_2^2 \geq 0. \]
 Recall that $1/2-z_2\geq 0$ everywhere, and $1/2-z_2>0$ on a dense
 subset of $EX^*$; it follows that $j_4\geq 0$ everywhere.  It follows
 that $j_1\geq 0$ everywhere, and that if $j_1=0$ then we must have
 $j_4=j_5=0$.  Note that $j_4=1$ when $z_1=0$, so we must have
 $z_1>0$ (and also $z_1\leq 1$ as always).  This lets us rearrange
 $j_4=0$ to give $z_2=2(z_1^{-1}-1)^2$.  After substituting this in
 the relation $j_5=0$ we see that $z_1\in\{2/3,1\}$ and so
 $z\in\{(2/3,1/2),(1,0)\}$.  Each point in the $z$-plane corresponds
 to a $G$-orbit in $EX^*$, and it is straightforward to check that the
 relevant $G$-orbits are $\{v_{10},v_{11},v_{12},v_{13}\}$ and
 $\{v_0,v_1\}$.
 \begin{checks}
  embedded/roothalf/sphere_quotients_check.mpl: check_E_to_S2()
 \end{checks}
\end{proof}

\begin{lemma}\lbl{lem-sphere-quotient-preimages}
 For all $i$ we have $m^{-1}\{m(v_i)\}=\{v_i,\lm^2(v_i)\}$.  In
 particular, for $i\in\{0,1,10,11,12,13\}$, we have
 $\lm^2(v_i)=v_i$ and $m^{-1}\{m(v_i)\}=\{v_i\}$.
\end{lemma}
\begin{proof}
 Recall that formulae for $m(v_i)$ were given in
 Proposition~\ref{prop-sphere-quotient-a}.

 Now suppose that $x\in EX^*$ with $m(x)=u$.
 \begin{itemize}
  \item[(a)] We have $u_1=0$ iff $(1-y_1^2)y_2=0$ iff $x_3=y_1=\pm 1$
   or $y_2=0$.  Note that if $x_3=\pm 1$ we must have
   $x\in\{e_3,-e_3\}=\{v_0,v_1\}$.  On the other hand, we have $y_2=0$
   iff $x\in C_1\cup C_2$.
  \item[(b)] We have $u_2=0$ iff $x_1=0$ or $x_2=0$, which means that
   $x\in\bigcup_{i=3}^8C_i$.
  \item[(c)] We have $u_3=0$ iff $x_3=0$ iff $x\in C_0$.
 \end{itemize}
 If $u=m(v_i)$ for some $i<10$ then two of the coordinates $u_p$ are
 zero, and it follows that $x\in C_r\cap C_s$ for some $r\neq s$, so
 $x=v_j$ for some $j$.  A check of cases then shows that
 $x\in\{v_i,\lm^2(v_i)\}$.  Suppose instead that
 \[ m(x) = m(v_{11})=(1,0,-2\sqrt{6})/5. \]
 As $m(x)_2=0$ we have $x_1x_2=0$ and so $x\in\bigcup_{i=3}^8C_i$.
 From the form of $m(x)_1$ and $m(x)_3$ it is also clear that
 $y_1,y_2>0$.  Also, we have
 \[ \frac{\rt(1-y_1^2)y_2}{2y_1} = -\frac{m(x)_1}{m(x)_3} =
     \frac{1}{2\sqrt{6}},
 \]
 so $y_2=(2\sqrt{3}(y_1^{-1}-y_1))^{-1}$.  Substituting this in the
 relation $m(x)_1^2=1/25$ and factoring leads to a relation
 \[ (2y_1^2-3)(3y_1^2-2)^2(8y_1^6-27y_1^4+30y_1^2-12) = 0. \]
 One can check that the only root in the required range $0<y_1\leq 1$
 is $y_1=\sqrt{2/3}$, and this in turn gives
 $y_2=(2\sqrt{3}(y_1^{-1}-y_1))^{-1}=1/\rt$.  This gives
 $(x_3,x_4)=(y_1,-y_1y_2)=(\sqrt{2/3},-\sqrt{1/3})$, and as
 $x_3^2+x_4^2=1$ we must have $x_1=x_2=0$, so $x=v_{11}$ as required.
 The remaining cases $i\in\{10,12,13\}$ now follow using the group
 action.
\end{proof}

\begin{proof}[Proof of Proposition~\ref{prop-sphere-quotient-b}]
 Put
 \[ U=S^2\sm\{m(v_i)\st i\in\{0,1,10,11,12,13\}\}. \]
 We now see that for $u\in U$, the preimage $m^{-1}\{u\}$ contains
 only points where the Jacobian of $m$ is strictly positive, so the
 number of points is equal to the degree of $m$.  Taking $u=m(v_2)$ we
 see that $m^{-1}\{u\}=\{v_2,\lm^2(v_2)\}=\{v_2,v_4\}$, so the degree
 is two.  It follows that for all $u\in U$, the preimage consists of
 two points and is closed under the action of $\lm^2$.  The only
 points fixed by $\lm^2$ are
 $\{v_0,v_1,v_{10},v_{11},v_{12},v_{13}\}$, and these cannot lie in
 $m^{-1}\{u\}$, so $m^{-1}\{u\}$ must consist of a single
 $\lm^2$-orbit.  The same holds by
 Lemma~\ref{lem-sphere-quotient-preimages} in the exceptional cases
 where $u\in S^2\sm U$.  It follows that the induced map
 $EX^*/\ip{\lm^2}\to S^2$ is a continuous bijection of compact
 Hausdorff spaces, so it is a homeomorphism.
\end{proof}

\begin{remark}
 Again put
 \[ U=S^2\sm\{m(v_i)\st i\in\{0,1,10,11,12,13\}\}. \]
 In Section~\ref{sec-energy} we will explain how to define a function
 $u=C_+(Dm)/(C_+(Dm)+C_-(Dm))\:U\to [0,1]$ which is zero at points
 where $m$ is conformal, and one at points where $m$ is anticonformal.
 We have not found a formula for $u$, but it is not hard to
 evaluate it at any given point.  We find that $0\leq u<1/2$
 everywhere in $U$, with $u(x)\to 1/2$ as $x\to v_0$ or $x\to v_1$.
 For $i\in\{10,11,12,13\}$ we find that $u$ oscillates between about
 $0.01$ and $0.08$ on any small circle surrounding $v_i$, so it does
 not extend continuously to $v_i$.  On about $90\%$ of the area of
 $EX^*$ we have $u < 0.075$.
\end{remark}

There are two other maps $EX^*\to S^2$ that have natural geometric
descriptions, so it is reasonable to ask whether they are related to
$\hp$.  The first is the Hopf fibration $\eta\:S^3\to S^2$, which we
can restrict to $EX^*$.  (In Maple this is \mcode+hopf_map+, which
is defined in \fname+Rn.mpl+.)  Note that $H_2(EX^*)\simeq
H_2(S^2)\simeq\Z$, but $H_2(S^3)=0$, which implies that the map
$\eta_*\:H_2(EX^*)\to H_2(S^2)$ is zero.  In other words, the
restricted map $\eta\:EX^*\to S^2$ has degree zero, whereas $\hp$ has
degree $2$, and any other nonconstant conformal map has strictly
positive degree.  This means that $\eta$ cannot be closely related to
$\hp$.
\begin{checks}
 Rn_check.mpl: check_hopf_map()
\end{checks}

The second possibility is a variant of the Gauss map.  We can identify
$\R^4$ with the algebra $\H$ of quaternions, and $\R^3$ with the
subspace $\H_0$ of purely imaginary quaternions.  This allows us to
interpret conjugation and multiplication of elements of $\R^4$.  For
$x\in X$ we recall that $n(x)$ is orthogonal to $x$, so the element
$\gm(x)=n(x)\ov{x}/\|n(x)\|$ is a purely imaginary quaternion of norm
one.  This defines a map $\gm\:X\to S^2$ (which is \mcode+gauss_map+
in Maple).  It depends on our conventions for identifying $\R^4$ with
$\H$, but the dependence is easy to analyse.  It is known that every
special orthogonal automorphism of $\H$ has the form
$x\mapsto ux\ov{v}$ for some $u,v\in S^3$, and every special
orthogonal automorphism of $\H_0$ has the form $x\mapsto wx\ov{w}$.
Changing conventions will replace $\gm$ by a map of the form
$\gm'(x)=w\gm(ux\ov{v})\ov{w}$.  By choosing paths in $S^3$ from $u$,
$v$ and $w$ to the identity, we can produce a homotopy between $\gm$
and $\gm'$, so they at least have the same degree.  One can calculate
the degree by counting preimages of a regular value, with signs
determined by the orientation behaviour.  One can check that
$\gm^{-1}\{(1,1,0)/\rt\}=\{v_6\}$, and that $\gm$ gives an
orientation-reversing isomorphism of tangent spaces at this point.  It
follows that $\gm$ has degree $-1$, and cannot be homotopic to $\hp$.
\begin{checks}
 embedded/geometry_check.mpl: check_gauss_map();
\end{checks}

\subsection{Rational points}
\lbl{sec-rational}

In this section we study points in $EX^*$ where the coordinates are
rational or lie in some small extension of $\Q$.  As well as being
interesting for its own sake, it is useful to have a supply of points
where we can easily do exact calculations rather than relying on
numerical approximation.

\begin{proposition}\lbl{prop-rational}
 The set $EX^*(\Q)=EX^*\cap\Q^4$ is as follows:
 \begin{itemize}
  \item[(a)] For every $t\in\Q$, the point $(t^2-1,2t,0,0)/(1+t^2)$ lies in
   $EX^*(\Q)\cap C_0$.
  \item[(b)] The point $v_2=(1,0,0,0)$ (corresponding to $t=\infty$)
   also lies in $EX^*(\Q)\cap C_0$.
  \item[(c)] For any pair $(s,t)\in\Q^2$ with $2s^2+t^2=1$, the point
   $(s,s,t,0)$ lies in $EX^*(\Q)\cap C_1$, and the point $(s,-s,t,0)$ lies
   in $EX^*(\Q)\cap C_2$.
  \item[(d)] All points in $EX^*(\Q)$ are of type~(a), (b) or~(c).
 \end{itemize}
\end{proposition}

\begin{proof}
 Recall that the cubic defining equation is
 \[ (\tfrac{3}{2}x_3^2-2)x_4+(x_1^2-x_2^2)x_3/\rt = 0. \]
 Note here that $1$ and $1/\sqrt{2}$ are linearly independent over
 $\Q$, and also that $\tfrac{3}{2}x_3^2-2\neq 0$ for $x_3\in\Q$.
 It follows that all rational solutions have $x_4=0$ and (either
 $x_3=0$ or $x_1=\pm x_2$).  If $x_3=x_4=0$ then $x_1^2+x_2^2=1$.  If
 $x_1=1$ then we have case~(a), otherwise we have case~(b) with
 $t=x_2/(1-x_1)$.  If $x_4=0$ and $x_1=\pm x_2$ then we have
 case~(c).
\end{proof}

The pairs $(s,t)$ as in~(c) are well-understood in terms of the
arithmetic of the field $\Q(\sqrt{-2})$, as we now recall briefly.
For any element $x=t+s\sqrt{-2}\in\Q(\sqrt{-2})$, we put
\[ N(x) = |x|^2 = 2s^2+t^2 \in \Q. \]
This gives a homomorphism $\Q(\sqrt{-2})^\tm\to\Q^\tm$, and
$EX^*(\Q)\cap C_1$ bijects with $\ker(N)$.  If $p$ is a prime
congruent to $1$ or $3$ mod $8$, then it is well-known that there is a
unique pair of positive integers $(a,b)$ such that $2a^2+b^2=p$.  We
put $\pi_p=b+a\sqrt{-2}$ and
\[ u_p=\frac{\pi_p}{\ov{\pi_p}} = ((b^2-2a^2)+2ab\sqrt{-2})/p. \]
Standard methods of algebraic number theory show that $\ker(N)$ is the
product of $\{\pm 1\}$ with the free abelian group generated by these
elements $u_p$.  Moreover, the denominator of a product
$\prod_pu_p^{n_p}$ is $\prod_pp^{|n_p|}$.  Using this, we can
enumerate all the solutions for which the denominator is less than
some specified bound, and thus produce rational points that are
closely spaced around $C_1$ and $C_2$.

\begin{remark}
 The map in part~(a) of Proposition~\ref{prop-rational} is
 \mcode+c_rational[0](t)+, and the map in part~(c) is
 \mcode+c_rational[1]([s,t])+.  The function \mcode+two_circle(n)+
 returns the list of all pairs $(s,t)\in\Q^2$ where $2s^2+t^2=1$ and
 the denominators of $s$ and $t$ are less than or equal to $n$.
\end{remark}

As the rational points do not cover much of $EX^*$, we instead
consider the set
\[ Q = \{x\in EX^*\st x_1,x_2,x_4,\;x_3/\rt\in\Q\}. \]
This set is again arithmetically simple, but has a richer structure
than $EX^*(\Q)$.  We call the points in $Q$ \emph{quasirational}.

We first consider quasirational points which have nontrivial isotropy
(and so lie in $\bigcup_{i=0}^8C_i$).
\begin{proposition}\lbl{prop-quasirational-isotropy}
 \begin{itemize}
  \item[(a)] For every $t\in\Q\cup\{\infty\}$ we have a point
   \[ \left(t^2-1,2t,0,0\right)/(1+t^2) \in Q\cap C_0 = EX^*(\Q)\cap C_0. \]
  \item[(b)] For every $t\in\Q\cup\{\infty\}$ we have points
   \begin{align*}
    \left(1-t^2-2t,\;1-t^2-2t,\;\rt(1-t^2+2t),0\right)/(2(1+t^2))
       & \in Q\cap C_1 \\
    \left(-(1-t^2-2t),\;1-t^2-2t,\;\rt(1-t^2+2t),0\right)/(2(1+t^2))
       & \in Q\cap C_2.
   \end{align*}
  \item[(c)] For each $(s,t)\in\Q^2$ with $3s^2+t^2=1$, we have points
   \begin{align*}
    \left(0,t,\rt s,-s\right) &\in Q\cap C_3 \\
    \left(-t,0,\rt s,s\right) &\in Q\cap C_4.
   \end{align*}
  \item[(d)] All  quasirational points with nontrivial isotropy are
   accounted for by~(a) to~(c).  In particular, we have
   $Q\cap\bigcup_{i=5}^8C_i=\emptyset$.
 \end{itemize}
\end{proposition}
\begin{proof}
 First, it is straightforward to check that the constructions in~(a)
 to~(c) do in fact give quasirational points on the indicated curves.

 Next, points in $C_0$ have $x_3=0$ so they are quasirational if and
 only if they are rational.  Thus, we have seen already that all
 quasirational points on $C_0$ are as in~(a).

 Now let $x$ be a quasirational point in $C_1$.  We then have
 $x_1=x_2=m$ and $x_3=\rt n$ and $x_4=0$ for some rational numbers $m$
 and $n$.  Put $p=n-m\in\Q$ and $q=n+m\in\Q$ so
 \[ x = \left((q-p)/2,\;(q-p)/2,\;(q+p)/\rt,\;0\right). \]
 For such points we have $g(x)=0$ automatically, and
 $\rho(x)=p^2+q^2$.  It follows that $(p,q)=(2t,1-t^2)/(1+t^2)$ for
 some $t\in\Q\cup\{\infty\}$, and we can use this to get the first
 formula in~(b).  The $C_2$ case follows from the $C_1$ case by the
 group action.

 Next, one can check that (in the current case where $a=1/\rt$) we
 have
 \[ g(x_1,0,x_3,x_4) = \rt(x_3-\rt x_4)(x_1^2+x_4^2+x_3x_4/\rt). \]
 The first factor vanishes on $C_4$.  The claim about $C_4$ in~(c)
 follows easily from this, and the claim about $C_3$ can be deduced
 using the group action.

 Now consider a point $x\in C_5$, so $x_2=0$ and the functions
 $r_0=x_1^2+x_3^2+x_4^2-1$ and $r_1=x_1^2+x_4^2+x_3x_4/\rt$ also
 vanish at $x$.  (Here $r_1$ is the second term in the above
 factorisation of $g$.)  Now put $t=x_1^2$ and
 \[ u = x_1(x_4^2-5x_3x_4/\rt-1). \]
 We claim that $u^2=t^3-10t^2+t$.  In fact, one can check by direct
 expansion that
 \[ u^2 - (t^3-10t^2+t) = a_0r_0 + a_1r_1 + a_2r_0r_1, \]
 where
 \begin{align*}
  a_0 &= -6 \rt (x_3+\rt x_4) x_4^3 \\
  a_1 &= (1-x_3^2-x_4^2) (x_3^2-10 x_4^2+9+x_3 x_4/\rt) \\
  a_2 &= x_3^2+2 x_4^2+x_3 x_4/\rt-x_1^2+9.
 \end{align*}
 As $r_0=r_1=0$, the claim follows.  Note also that if $x$ is
 quasirational then $t$ and $u$ will be rational.  Now, the equation
 $u^2=t^3-10t^2+t$ describes an elliptic curve $E$ over $\Q$, and
 algorithms to determine rational points on such curves are built in
 to the symbolic mathematics system Sage.  Our curve can be
 described in extended Weierstrass form as
 \[ u^2 + \al_1ut + \al_3u = t^3 + \al_2t^2 + \al_4t + \al_6, \]
 where
 \[ (\al_1,\al_2,\al_3,\al_4,\al_6) = (0,-10,0,1,0). \]
 We can thus enter the following in Sage:
\begin{mcodeblock}
   E = EllipticCurve([0,-10,0,1,0])
   E.cremona_label()
   E.rank()
   E.torsion_points()
\end{mcodeblock}
 We learn that $E$ is isomorphic to the curve labelled 96b2 in
 Cremona's database of elliptic curves~\cite{cr:ecd}, and that the
 rank is zero, so the rational points are all torsion points.  We also
 learn that the only rational torsion points in the projective closure
 are $[0:0:1]$ and $[0:1:0]$, so the only rational point on the
 original affine curve is $(0,0)$.  Thus, if $x$ is quasirational then
 we must have $t=u=0$, but that gives $x_1=0$.  Substituting $x_1=0$
 in the relation $r_1=0$ gives $x_4(x_4+x_3/\rt)=0$, so $x_4=0$ or
 $x_4=-x_3/\rt$.  Putting $x_1=x_4=0$ in $r_0$ gives $x_3^2=1$;
 putting $x_1=0$ and $x_4=-x_3/\rt$ instead gives $x_3^2=2/3$.
 Neither of these is possible with $x_3/\rt\in\Q$, so we see that
 there are no quasirational points in $C_5$.  It follows using the
 group action that there are no quasirational points in
 $\bigcup_{i=5}^8C_i$.
 \begin{checks}
  embedded/roothalf/rational_check.mpl: check_rational_elliptic()
 \end{checks}
\end{proof}

We can understand rational solutions to $3s^2+t^2=1$ in terms of the
arithmetic of the field $\Q(\sqrt{-3})$, and thus produce
quasirational points that are closely spaced around $C_3$ and $C_4$.
The story similar to that for $\Q(\sqrt{-2})$, and we will not give
the details here.

\begin{remark}
 The map in part~(a) of Proposition~\ref{prop-quasirational-isotropy}
 is \mcode+c_rational[0]+, and the maps in~(b) and~(c) are
 \mcode+c_quasirational[i]+ for $i\in\{1,2,3,4\}$.  The function
 \mcode+three_circle(n)+ returns the list of all pairs $(s,t)\in\Q^2$
 where $3s^2+t^2=1$ and the denominators of $s$ and $t$ are less than
 or equal to $n$.
\end{remark}

We now consider quasirational points with trivial isotropy.  We do not
have a very good theory for these, but we have a reasonably efficient
method for exhaustive search, as we now explain.

\begin{proposition}\lbl{prop-quasirational-lift}
 For any $x\in Q\cap F_{16}$, the numbers $s_1=y_1/\rt=x_3/\rt$ and
 \[ s_2 = \rt y_2 = x_2^2-x_1^2-3x_3x_4/\rt \]
 lie in $\Q\cap [0,1]$.  Conversely, suppose that
 $s\in(\Q\cap [0,1])^2$, and put
 \[ t_1 = s_1^2 \hspace{4em}
    t_2 = 6s_1^2 - 2 \hspace{4em}
    t_3 = -2s_2^2-4
 \]
 \[ p_1 = 2+t_1t_3 \hspace{4em}
    p_2 = s_2t_2 \hspace{4em}
    p_3 = p_1+p_2 \hspace{4em}
    p_4 = p_1-p_2.
 \]
 Then $s$ arises from a quasirational point $x\in F_{16}$ if and only
 if $p_3$ and $p_4$ are perfect squares (and therefore nonnegative).
 If so, then
 \[ x = (\sqrt{p_3}/2,\;\sqrt{p_4}/2,\;\rt s_1,\;-s_1s_2). \]
\end{proposition}
\begin{proof}
 This is essentially a reformulation of Proposition~\ref{prop-F-four},
 in the special case $a=1/\rt$.  The formulae have been reorganised
 slightly to make rationality questions more visible, and to allow for
 efficient calculation as discussed below.
\end{proof}

We wrote code in C to look for solutions using the above proposition.
We ran the program on a cluster of machines with 64 bit processors.
The native 64 bit integers are not large enough for our intermediate
calculations, but fortunately the GCC compiler provides built in
support for 128 bit integers encoded as pairs of native integers, and
similarly for floating point numbers.  (If we needed larger integers
than 128 bits we would need to use an arbitrary precision library,
which would come with a significant performance penalty.)  We took
$N=2^{12}=4096$ and enumerated the rational numbers of denominator at
most $N$ as a Farey sequence (of length $5100021\simeq 5\tm 10^6$).
Note that the numbers $t_i$ in
Proposition~\ref{prop-quasirational-lift} depend only on a single
rational number in the sequence, so we precomputed them in a single
pass.  We then looped through all possible pairs $(s_1,s_2)$ and
computed the numbers $p_i$.  One can check that if $p_1<0$ at
$(s_1,s_2)$, then all pairs $(s'_1,s'_2)$ with $s'_1\geq s_1$ and
$s'_2\geq s_2$ can be disregarded.  Using this we can cut down the
number of pairs to be considered, but the order of magnitude is still
$10^{13}$, so the remaining steps must be highly optimised.

To check whether a rational number $p=a/b$ (possibly not in lowest
terms) is a perfect square, we first test whether the $2$-adic
valuations of $a$ and $b$ are the same mod $2$.  Here the $2$-adic
valuation is the number of trailing zeros in the binary
representation; this can be calculated in a single processor
instruction for $64$ bit integers, and is only slightly slower for
$128$ bit pairs.  If this first test is passed, we divide $a$ and $b$
by appropriate powers of $2$ to make them odd.  This can be done as a
bitwise shift rather than a division, so again it is very fast.  Next,
with the new $a$ and $b$, it is not hard to check that $a/b$ can only
be a square if $a=b\pmod{8}$.  This can again be checked by bit
masking rather than division, so it is very fast.  If these fast tests
are passed for both $p_3$ and $p_4$ (which is already relatively
rare), we then start using some slower tests.  We calculate the gcd of
$a$ and $b$ and then divide by it to make $a$ and $b$ coprime.  It
seems that no algorithm is known that is usefully more efficient
than the obvious one, so a significant number of divisions may be
required.  Now $a/b$ will be a perfect square if and only if $a$ and
$b$ are individually perfect squares.  They are still odd, so we first
test that they are congruent to $1$ mod $8$.  If so, we reduce them
modulo $3\tm 5\tm 7\tm 11\tm 13\tm 17=255255$, and look up in a
precomputed table whether the result is a quadratic residue.  Only
about $0.63\%$ of numbers pass these modular tests.  If we get to this
stage, we calculate the square root as a $128$ bit floating point
number, take the nearest integer, square it, and check whether the
result is the same as the number we first thought of.  If so, our
number is obviously a perfect square.  The converse is less obvious
because of the possibility of rounding errors.  However, one can check
that our numerators and denominators are not much bigger than
$N^7=2^{84}$ so $128$ bit accuracy should be ample to avoid problems.

We found precisely $130$ pairs $s$ corresponding to quasirational
points in the interior of $F_{16}$:
\[ \renewcommand{\arraystretch}{1.5}  \setlength{\arraycolsep}{4pt}
 \begin{array}{cccccccc}
  \left(\frac{ 103}{ 154},\frac{  65}{ 449}\right) &
  \left(\frac{ 103}{ 554},\frac{2945}{3041}\right) &
  \left(\frac{ 104}{1073},\frac{ 105}{ 233}\right) &
  \left(\frac{   1}{  10},\frac{  13}{  19}\right) &
  \left(\frac{   1}{  10},\frac{  16}{  17}\right) &
  \left(\frac{   1}{  10},\frac{ 245}{ 267}\right) &
  \left(\frac{  11}{ 118},\frac{ 819}{1331}\right) &
  \left(\frac{  11}{1190},\frac{1664}{2057}\right)\\
  \left(\frac{1127}{2194},\frac{ 581}{ 731}\right) &
  \left(\frac{1135}{2114},\frac{1088}{1137}\right) &
  \left(\frac{1147}{2390},\frac{  97}{ 961}\right) &
  \left(\frac{1147}{2830},\frac{ 640}{1369}\right) &
  \left(\frac{1151}{3722},\frac{  11}{ 139}\right) &
  \left(\frac{1223}{2330},\frac{  16}{  17}\right) &
  \left(\frac{  12}{  29},\frac{ 247}{ 297}\right) &
  \left(\frac{ 125}{ 262},\frac{  16}{  59}\right)\\
  \left(\frac{1255}{2146},\frac{1547}{1675}\right) &
  \left(\frac{1260}{3599},\frac{ 559}{1009}\right) &
  \left(\frac{ 127}{ 450},\frac{ 416}{ 417}\right) &
  \left(\frac{1301}{2270},\frac{  16}{  33}\right) &
  \left(\frac{  13}{  30},\frac{ 145}{1521}\right) &
  \left(\frac{  13}{  30},\frac{  16}{  33}\right) &
  \left(\frac{1389}{3038},\frac{ 176}{1049}\right) &
  \left(\frac{1424}{3913},\frac{1571}{2179}\right)\\
  \left(\frac{ 148}{ 215},\frac{  29}{ 323}\right) &
  \left(\frac{ 151}{ 338},\frac{ 279}{1129}\right) &
  \left(\frac{1549}{3038},\frac{2768}{3307}\right) &
  \left(\frac{1603}{2350},\frac{  19}{ 147}\right) &
  \left(\frac{1611}{3878},\frac{ 355}{ 643}\right) &
  \left(\frac{1648}{2875},\frac{2295}{3113}\right) &
  \left(\frac{1672}{3329},\frac{1257}{2057}\right) &
  \left(\frac{ 169}{ 322},\frac{ 475}{ 507}\right)\\
  \left(\frac{  17}{ 106},\frac{1760}{2601}\right) &
  \left(\frac{ 172}{ 511},\frac{  15}{  17}\right) &
  \left(\frac{ 175}{ 298},\frac{2464}{2825}\right) &
  \left(\frac{ 175}{ 362},\frac{ 112}{ 163}\right) &
  \left(\frac{1792}{3635},\frac{2623}{3649}\right) &
  \left(\frac{1809}{3050},\frac{  32}{  41}\right) &
  \left(\frac{1829}{2750},\frac{  97}{ 961}\right) &
  \left(\frac{1864}{2843},\frac{ 145}{2993}\right)\\
  \left(\frac{ 191}{3490},\frac{1211}{1243}\right) &
  \left(\frac{1939}{3350},\frac{  19}{ 147}\right) &
  \left(\frac{ 196}{ 335},\frac{2719}{3553}\right) &
  \left(\frac{1964}{3085},\frac{  31}{  97}\right) &
  \left(\frac{  19}{  70},\frac{  16}{  33}\right) &
  \left(\frac{2093}{3534},\frac{ 115}{ 147}\right) &
  \left(\frac{ 209}{ 338},\frac{ 819}{1331}\right) &
  \left(\frac{ 211}{ 350},\frac{  80}{ 107}\right)\\
  \left(\frac{  23}{ 154},\frac{ 895}{ 993}\right) &
  \left(\frac{  23}{  34},\frac{ 481}{2881}\right) &
  \left(\frac{  23}{  50},\frac{  32}{  41}\right) &
  \left(\frac{2435}{3854},\frac{  16}{  59}\right) &
  \left(\frac{ 248}{ 401},\frac{  95}{ 449}\right) &
  \left(\frac{ 248}{ 665},\frac{  29}{ 323}\right) &
  \left(\frac{ 256}{ 377},\frac{  35}{ 387}\right) &
  \left(\frac{  25}{  82},\frac{ 237}{ 275}\right)\\
  \left(\frac{2636}{3887},\frac{ 445}{4003}\right) &
  \left(\frac{ 280}{ 457},\frac{1243}{2107}\right) &
  \left(\frac{  28}{  45},\frac{ 559}{1617}\right) &
  \left(\frac{ 307}{ 830},\frac{ 267}{ 779}\right) &
  \left(\frac{  31}{ 202},\frac{ 355}{1507}\right) &
  \left(\frac{ 313}{ 466},\frac{ 560}{2651}\right) &
  \left(\frac{   3}{  14},\frac{   5}{  27}\right) &
  \left(\frac{  31}{  50},\frac{ 355}{1507}\right)\\
  \left(\frac{ 316}{ 487},\frac{ 835}{2243}\right) &
  \left(\frac{  32}{  49},\frac{ 767}{2433}\right) &
  \left(\frac{ 341}{1070},\frac{  97}{ 961}\right) &
  \left(\frac{ 359}{ 610},\frac{ 463}{1969}\right) &
  \left(\frac{ 364}{ 687},\frac{ 409}{ 441}\right) &
  \left(\frac{  36}{  77},\frac{2575}{2673}\right) &
  \left(\frac{ 367}{ 986},\frac{2176}{3401}\right) &
  \left(\frac{  37}{ 190},\frac{ 688}{1369}\right)\\
  \left(\frac{ 373}{ 630},\frac{  16}{  33}\right) &
  \left(\frac{  37}{  70},\frac{   5}{  27}\right) &
  \left(\frac{ 401}{ 650},\frac{  16}{1067}\right) &
  \left(\frac{ 403}{ 590},\frac{  97}{ 961}\right) &
  \left(\frac{ 404}{ 685},\frac{  31}{  97}\right) &
  \left(\frac{ 415}{ 706},\frac{ 341}{ 459}\right) &
  \left(\frac{ 427}{ 694},\frac{ 128}{ 803}\right) &
  \left(\frac{  43}{ 166},\frac{ 224}{ 251}\right)\\
  \left(\frac{ 432}{1681},\frac{1055}{3553}\right) &
  \left(\frac{  43}{  94},\frac{ 480}{1849}\right) &
  \left(\frac{   4}{  47},\frac{  65}{ 193}\right) &
  \left(\frac{ 455}{ 778},\frac{ 973}{1075}\right) &
  \left(\frac{ 463}{1970},\frac{  16}{  17}\right) &
  \left(\frac{ 468}{ 775},\frac{ 385}{ 673}\right) &
  \left(\frac{  47}{ 106},\frac{  80}{  97}\right) &
  \left(\frac{   4}{   7},\frac{  15}{  17}\right)\\
  \left(\frac{  48}{ 235},\frac{ 107}{ 459}\right) &
  \left(\frac{ 501}{ 742},\frac{   5}{  27}\right) &
  \left(\frac{  52}{ 205},\frac{ 137}{ 169}\right) &
  \left(\frac{ 555}{ 974},\frac{ 163}{ 675}\right) &
  \left(\frac{ 560}{3163},\frac{ 951}{1225}\right) &
  \left(\frac{  56}{ 107},\frac{ 605}{ 931}\right) &
  \left(\frac{ 569}{1042},\frac{2735}{2993}\right) &
  \left(\frac{  57}{ 130},\frac{  32}{  41}\right)\\
  \left(\frac{ 577}{1546},\frac{1680}{1873}\right) &
  \left(\frac{ 600}{ 913},\frac{   7}{  25}\right) &
  \left(\frac{  61}{ 718},\frac{ 581}{ 731}\right) &
  \left(\frac{  65}{ 122},\frac{  32}{ 507}\right) &
  \left(\frac{  65}{ 122},\frac{ 973}{1075}\right) &
  \left(\frac{  65}{ 274},\frac{2777}{2873}\right) &
  \left(\frac{ 700}{1303},\frac{ 931}{3075}\right) &
  \left(\frac{ 704}{1163},\frac{1435}{2299}\right)\\
  \left(\frac{   7}{ 106},\frac{ 245}{ 267}\right) &
  \left(\frac{  71}{ 130},\frac{  13}{  19}\right) &
  \left(\frac{   7}{ 194},\frac{1855}{3777}\right) &
  \left(\frac{  72}{ 115},\frac{ 217}{ 729}\right) &
  \left(\frac{  73}{ 274},\frac{1253}{1947}\right) &
  \left(\frac{ 735}{1594},\frac{1024}{2401}\right) &
  \left(\frac{ 739}{1318},\frac{ 249}{ 601}\right) &
  \left(\frac{   7}{ 466},\frac{ 656}{ 931}\right)\\
  \left(\frac{  75}{ 134},\frac{ 341}{ 459}\right) &
  \left(\frac{  75}{ 134},\frac{ 656}{ 675}\right) &
  \left(\frac{  76}{ 143},\frac{ 327}{ 473}\right) &
  \left(\frac{ 767}{1834},\frac{ 235}{ 779}\right) &
  \left(\frac{ 777}{2050},\frac{ 224}{ 513}\right) &
  \left(\frac{   7}{  90},\frac{  19}{ 147}\right) &
  \left(\frac{  80}{ 187},\frac{ 799}{2049}\right) &
  \left(\frac{   8}{  25},\frac{1007}{1041}\right)\\
  \left(\frac{   8}{  25},\frac{  31}{  97}\right) &
  \left(\frac{ 827}{2590},\frac{1963}{3531}\right) &
  \left(\frac{  83}{ 126},\frac{   5}{  27}\right) &
  \left(\frac{ 839}{2170},\frac{ 736}{ 857}\right) &
  \left(\frac{ 881}{1538},\frac{ 160}{ 209}\right) &
  \left(\frac{ 907}{2046},\frac{1075}{1203}\right) &
  \left(\frac{  92}{ 381},\frac{1127}{2073}\right) &
  \left(\frac{  95}{ 202},\frac{1264}{1411}\right)\\
  \left(\frac{  95}{ 202},\frac{ 301}{ 499}\right) &
  \left(\frac{ 965}{2086},\frac{1072}{1075}\right) &
 \end{array}
\]
\begin{checks}
 embedded/roothalf/rational_check.mpl: check_rational()
\end{checks}

There do not appear to be any discernable patterns in the prime
factorisations of these rational numbers.  The corresponding points in
$EX^*$ are stored in the variable
\mcode+inner_quasirational_points+.

The corresponding points in
$F_4^*$ can be displayed as follows:

\begin{center}
 \begin{tikzpicture}[scale=8]
  \draw[cyan   ] (0.000,0.000) -- (0.000,0.707);
  \draw[green  ] (0.000,0.000) -- (1.000,0.000);
  \draw[magenta] (0.000,0.707) -- (0.816,0.707);
  \draw[blue,smooth] (1.000,0.000) -- (0.996,0.012) -- (0.984,0.047) -- (0.965,0.104) -- (0.943,0.177) -- (0.919,0.262) -- (0.894,0.354) -- (0.872,0.445) -- (0.853,0.530) -- (0.837,0.604) -- (0.826,0.660) -- (0.819,0.695) -- (0.816,0.707);
  \fill(0.946,0.102) circle(0.004);
  \fill(0.263,0.685) circle(0.004);
  \fill(0.137,0.319) circle(0.004);
  \fill(0.141,0.484) circle(0.004);
  \fill(0.141,0.666) circle(0.004);
  \fill(0.141,0.649) circle(0.004);
  \fill(0.132,0.435) circle(0.004);
  \fill(0.013,0.572) circle(0.004);
  \fill(0.726,0.562) circle(0.004);
  \fill(0.759,0.677) circle(0.004);
  \fill(0.679,0.071) circle(0.004);
  \fill(0.573,0.331) circle(0.004);
  \fill(0.437,0.056) circle(0.004);
  \fill(0.742,0.666) circle(0.004);
  \fill(0.585,0.588) circle(0.004);
  \fill(0.675,0.192) circle(0.004);
  \fill(0.827,0.653) circle(0.004);
  \fill(0.495,0.392) circle(0.004);
  \fill(0.399,0.705) circle(0.004);
  \fill(0.811,0.343) circle(0.004);
  \fill(0.613,0.067) circle(0.004);
  \fill(0.613,0.343) circle(0.004);
  \fill(0.647,0.119) circle(0.004);
  \fill(0.515,0.510) circle(0.004);
  \fill(0.974,0.063) circle(0.004);
  \fill(0.632,0.175) circle(0.004);
  \fill(0.721,0.592) circle(0.004);
  \fill(0.965,0.091) circle(0.004);
  \fill(0.587,0.390) circle(0.004);
  \fill(0.811,0.521) circle(0.004);
  \fill(0.710,0.432) circle(0.004);
  \fill(0.742,0.662) circle(0.004);
  \fill(0.227,0.478) circle(0.004);
  \fill(0.476,0.624) circle(0.004);
  \fill(0.830,0.617) circle(0.004);
  \fill(0.684,0.486) circle(0.004);
  \fill(0.697,0.508) circle(0.004);
  \fill(0.839,0.552) circle(0.004);
  \fill(0.941,0.071) circle(0.004);
  \fill(0.927,0.034) circle(0.004);
  \fill(0.077,0.689) circle(0.004);
  \fill(0.819,0.091) circle(0.004);
  \fill(0.827,0.541) circle(0.004);
  \fill(0.900,0.226) circle(0.004);
  \fill(0.384,0.343) circle(0.004);
  \fill(0.838,0.553) circle(0.004);
  \fill(0.874,0.435) circle(0.004);
  \fill(0.853,0.529) circle(0.004);
  \fill(0.211,0.637) circle(0.004);
  \fill(0.957,0.118) circle(0.004);
  \fill(0.651,0.552) circle(0.004);
  \fill(0.894,0.192) circle(0.004);
  \fill(0.875,0.150) circle(0.004);
  \fill(0.527,0.063) circle(0.004);
  \fill(0.960,0.064) circle(0.004);
  \fill(0.431,0.609) circle(0.004);
  \fill(0.959,0.079) circle(0.004);
  \fill(0.866,0.417) circle(0.004);
  \fill(0.880,0.244) circle(0.004);
  \fill(0.523,0.242) circle(0.004);
  \fill(0.217,0.167) circle(0.004);
  \fill(0.950,0.149) circle(0.004);
  \fill(0.303,0.131) circle(0.004);
  \fill(0.877,0.167) circle(0.004);
  \fill(0.918,0.263) circle(0.004);
  \fill(0.924,0.223) circle(0.004);
  \fill(0.451,0.071) circle(0.004);
  \fill(0.832,0.166) circle(0.004);
  \fill(0.749,0.656) circle(0.004);
  \fill(0.661,0.681) circle(0.004);
  \fill(0.526,0.452) circle(0.004);
  \fill(0.275,0.355) circle(0.004);
  \fill(0.837,0.343) circle(0.004);
  \fill(0.748,0.131) circle(0.004);
  \fill(0.872,0.011) circle(0.004);
  \fill(0.966,0.071) circle(0.004);
  \fill(0.834,0.226) circle(0.004);
  \fill(0.831,0.525) circle(0.004);
  \fill(0.870,0.113) circle(0.004);
  \fill(0.366,0.631) circle(0.004);
  \fill(0.363,0.210) circle(0.004);
  \fill(0.647,0.184) circle(0.004);
  \fill(0.120,0.238) circle(0.004);
  \fill(0.827,0.640) circle(0.004);
  \fill(0.332,0.666) circle(0.004);
  \fill(0.854,0.405) circle(0.004);
  \fill(0.627,0.583) circle(0.004);
  \fill(0.808,0.624) circle(0.004);
  \fill(0.289,0.165) circle(0.004);
  \fill(0.955,0.131) circle(0.004);
  \fill(0.359,0.573) circle(0.004);
  \fill(0.806,0.171) circle(0.004);
  \fill(0.250,0.549) circle(0.004);
  \fill(0.740,0.460) circle(0.004);
  \fill(0.772,0.646) circle(0.004);
  \fill(0.620,0.552) circle(0.004);
  \fill(0.528,0.634) circle(0.004);
  \fill(0.929,0.198) circle(0.004);
  \fill(0.120,0.562) circle(0.004);
  \fill(0.753,0.045) circle(0.004);
  \fill(0.753,0.640) circle(0.004);
  \fill(0.335,0.683) circle(0.004);
  \fill(0.760,0.214) circle(0.004);
  \fill(0.856,0.441) circle(0.004);
  \fill(0.093,0.649) circle(0.004);
  \fill(0.772,0.484) circle(0.004);
  \fill(0.051,0.347) circle(0.004);
  \fill(0.885,0.210) circle(0.004);
  \fill(0.377,0.455) circle(0.004);
  \fill(0.652,0.302) circle(0.004);
  \fill(0.793,0.293) circle(0.004);
  \fill(0.021,0.498) circle(0.004);
  \fill(0.792,0.525) circle(0.004);
  \fill(0.792,0.687) circle(0.004);
  \fill(0.752,0.489) circle(0.004);
  \fill(0.591,0.213) circle(0.004);
  \fill(0.536,0.309) circle(0.004);
  \fill(0.110,0.091) circle(0.004);
  \fill(0.605,0.276) circle(0.004);
  \fill(0.453,0.684) circle(0.004);
  \fill(0.453,0.226) circle(0.004);
  \fill(0.452,0.393) circle(0.004);
  \fill(0.932,0.131) circle(0.004);
  \fill(0.547,0.607) circle(0.004);
  \fill(0.810,0.541) circle(0.004);
  \fill(0.627,0.632) circle(0.004);
  \fill(0.341,0.384) circle(0.004);
  \fill(0.665,0.633) circle(0.004);
  \fill(0.665,0.427) circle(0.004);
  \fill(0.654,0.705) circle(0.004);
 \end{tikzpicture}
\end{center}

(It may appear that some of the points lie on the boundary, but in
fact they are just inside, by a distance of only about $5\tm 10^{-4}$ in
some cases.)

\subsection{Integration}
\lbl{sec-integration}

Later on we will need to integrate various functions on $EX^*$.
Integration is most naturally defined for differential $2$-forms.
However, the metric and orientation on $EX^*$ gives a volume form $\om$
in a standard way, and we define the integral of a function $f$ to be
the integral of the form $f\om$.  In most cases of interest, $f$ will
be $G$-invariant so we can just integrate over $F_{16}$ and multiply
by $16$.

Unfortunately, it seems to be difficult to compute such integrals
accurately.  Given an explicit smooth embedding
$\phi\:[0,1]^2\to EX^*$, we can compute the Jacobian and then use
standard numerical integration techniques to evaluate integrals over
the image of $\phi$.  However, we have not succeeded in finding a
family of such maps $\phi$ for which the Jacobians are explicitly
computable and the images cover $EX^*$ without overlap.  We do know
several different homeomorphisms $[0,1]^2\to F_{16}$ that are
diffeomorphisms away from the boundary, but it seems that the singular
boundary behaviour destroys any possibility of using these maps for
accurate integration.  The best that we can do along these lines is to
construct a barycentric triangulation of $EX^*$ as in
Section~\ref{sec-E-charts}.  This gives us a decomposition of $EX^*$
into triangles $T$ and diffeomorphisms $\phi\:T\to\Dl_2$ where $\phi$
and its Jacobian are simple and explicit, but $\phi^{-1}$ is not.
Nonetheless, we can compute $\phi^{-1}(a)$ and associated quantities
for a large number of points $a\in T$.  This leads to an approximate
integration rule of the form $I(f)=\sum_{i=1}^nw_i\,f(a_i)$ for some
points $a_1,\dotsc,a_n\in EX^*$ and weights $w_i\in\R$.  We have
carried out this process and obtained a rule for which we believe that
$|I(f)-\int_{EX^*}f|/\|f\|_2$ is at most $10^{-25}$ or so for typical
functions that we need to consider.  The basis for this estimate will
be explained after we have discussed some more theoretical ideas.
Unfortunately, for this rule the number $n=33600$ of sample points is
very large, so we cannot compute integrals quickly.  We will also
discuss a method that gives less accurate integrals much more easily.

\subsubsection{A characterisation of the integration functional}
\lbl{sec-int-props}

\begin{definition}\lbl{defn-stokes}
 For any smooth one-form $\al=\sum_{j=1}^4u_j\,dx_j$ on $\R^4$, we
 define a function $D(\al)\:EX^*\to\R$ by
 \[ D(\al) =
     \|n\|^{-1} \sum_{ijkl} \ep_{ijkl}
      \frac{\partial u_j}{\partial x_i} x_k n_l.
 \]
 Here $n$ is the gradient of the function $g$ in the definition of
 $EX^*$, and $\ep$ is the totally antisymmetric tensor.  The operator
 $D$ is called \mcode+stokes+ in Maple; the definition is in
 \fname+embedded/roothalf/forms.mpl+.
\end{definition}

\begin{lemma}\lbl{lem-stokes}
 For any smooth one-form $\al$ on $\R^4$, we have
 \[ d(\al|_{EX^*})=(d\al)|_{EX^*}=D(\al)\om. \]
 It follows that $D(\al)$ depends only on $\al|_{EX^*}$, and that
 $\int_{EX^*}D(\al)\om=0$.
\end{lemma}
\begin{proof}
 We will freely use the metric to identify one-forms with vectors, so
 $dx_i$ becomes the $i$'th basis vector $e_i$.  We then have
 $d\al=\sum_{i,j}\frac{\partial u_j}{\partial x_i}e_i\wedge e_j$.

 Now consider a point $x\in EX^*$.  The vectors $x$ and $n/\|n\|$ form
 an orthonormal basis for $(T_xEX^*)^\perp$, so multiplication by the
 form $\bt_x=x\wedge(n/\|n\|)$ gives an isometric isomorphism from
 $\Lm^2T_xEX^*$ to $\Lm^4\R^4$.  Thus, if we put
 $\ep=e_1\wedge\dotsb\wedge e_4$, we must have
 $\om_x\wedge\bt_x=\pm\ep$, and a glance at our orientation
 conventions shows that the sign is positive.  However, it is clear
 from the definitions that $d\al\wedge\bt_x=D(\al)\ep$, so we must
 have $(d\al)|_{EX^*}=D(\al)\om$ as claimed.  It is standard that
 $(d\al)|_{EX^*}=d(\al|_{EX^*})$ so we can apply Stokes's Theorem to
 $\al|_{EX^*}$ to see that $\int_{EX^*}D(\al)\om=0$.
\end{proof}

\begin{definition}\lbl{defn-antiinvariant}
 We say that a one-form $\al$ is \emph{antiinvariant} if
 $\lm^*(\al)=\mu^*(\al)=\al$ and $\nu^*(\al)=-\al$.  Equivalently, we
 must have $\gm^*(\al)=\chi_7(\gm)\al$ for all $\gm\in G$, where
 $\chi_7$ is as in Proposition~\ref{prop-characters}, so
 $\gm^*(\al)=\al$ whenever $\gm$ preserves orientation, and
 $\gm^*(\al)=-\al$ whenever $\gm$ reverses orientation.
\end{definition}

\begin{proposition}\lbl{prop-integration}
 The integration functional $I\:C^\infty(EX^*)\to\R$ is the unique
 $\R$-linear map with the following properties.
 \begin{itemize}
  \item[(a)] For all $f\in C^\infty(EX^*)$ and all $\gm\in G$ we
   have $I(\gm^*f)=I(f)$.
  \item[(b)] For the curvature map $K$ we have $I(K)=-4\pi$.
  \item[(c)] For all antiinvariant one-forms $\al$ we have
   $I(D(\al))=0$.
 \end{itemize}
\end{proposition}
\begin{proof}
 First consider the following property:
 \begin{itemize}
  \item[(d)] For all one-forms $\al$ we have $I(D(\al))=0$.
 \end{itemize}
 This clearly implies~(c), and in fact~(a) and~(c) imply~(d).  To see
 this, note that the operator $D$ involves division by $\om$; because
 of this, it satisfies $D(\gm^*(\al))=\chi_7(\gm)\gm^*(D(\al))$.
 Thus, for any one-form $\al$, the form
 $\al'=|G|^{-1}\sum_\gm\chi_7(\gm)\gm^*(\al)$ is antiinvariant, and
 if~(a) holds then $I(D(\al))=I(D(\al'))$, so if~(c) holds then
 $I(D(\al))=0$.

 For the integration functional property~(a) is clear, and~(b) is the
 Gauss-Bonnet theorem, and~(d) follows from Lemma~\ref{lem-stokes}.

 Now let $I'$ be another functional with properties~(a), (b) and (c),
 and therefore also~(d).  The de Rham Theorem tells us that the space
 of two-forms on $EX^*$ modulo the image of $d$ is isomorphic to the
 cohomology group $H^2(EX^*;\R)$, which has dimension one.  Integration
 gives a well-defined and nontrivial map from this quotient to $\R$,
 which is therefore an isomorphism.  It is clear from this that $I'$
 must be equal to $I$.
\end{proof}

To use the above proposition, we need to have a good understanding of
the antiinvariant forms.
\begin{definition}
 We write $\Xi$ for the set of antiinvariant polynomial $1$-forms on
 $EX^*$.  This is a module over the ring $B=A^G=\R[z_1,z_2]$.  We will
 also consider the ring $B'=B[(2-z_1)^{-1}]$ and the module
 $\Xi'=B'\ot_B\Xi$.
\end{definition}

Recall that $0\leq z_1\leq 1$ on $EX^*$, so $2-z_1$ is everywhere
positive.  It follows that $B'$ can still be regarded as a ring of
real analytic functions on $EX^*$.

\begin{proposition}\lbl{prop-alpha-basis}
 $\Xi'$ is freely generated over $B'$ by the following forms:
 \begin{align*}
  \al_1 &= y_1(x_2\,dx_1-x_1\,dx_2) \\
  \al_2 &= z_1(1+z_2)y_1y_2(x_2\,dx_1+x_1\,dx_2)
            -2z_1y_2x_1x_2\,dx_3 + 2(1+z_1z_2)x_1x_2\,dx_4.
 \end{align*}
\end{proposition}
\begin{proof}
 First, let $\Om^*$ denote the free module over $A$ on generators
 $dx_1,\dotsc,dx_4$.  Let $\Tht$ be the submodule generated by the
 elements
 \begin{align*}
  \tht_1 &= dg_0 \\
  \tht_2 &= \half d(\rho-1) = \sum_ix_i\,dx_i;
 \end{align*}
 then $\Om=\Om^*/\Tht$.  There is an evident action of $G$ on $\Om^*$
 that preserves $\Tht$ and is compatible with the standard action on
 $\Om$.

 Next, for any $\R[G]$-module $V$ we define $\Pi\:V\to V$ by
 \[ \Pi(v)=|G|^{-1}\sum_{\gm\in G} \chi_7(\gm)\gm^*(v). \]
 It is standard that $\Pi$ is $B$-linear with $\Pi^2=\Pi$, and the
 image of $\Pi$ is the subspace
 \[ V[\chi_7] = \{v\in V\st \gm^*(v)  = \chi_7(\gm)v
                   \text{ for all } \gm\in G\}.
 \]
 We thus have $\Xi=\Om[\chi_7]=\Pi(\Om^*)/\Pi(\Tht)$.

 Now put
 \[ M = \{x_1^ix_2^jy_1^ky_2^l\st 0\leq i,j,k,l \leq 1\}, \]
 and recall that this is a basis for $A=\CO_{EX^*}$ over $B$.  It follows
 that the group $\Om^*[\chi_7]=\Pi(\Om^*)$ is generated by elements of the
 form $\Pi(m\,dx_i)$ with $m\in M$ and $1\leq i\leq 4$.  A computer
 calculation shows that every nonzero element of this form is a
 constant multiple of one of the following forms:
 \begin{align*}
  \bt_1 &= y_1(x_2\,dx_1 - x_1\,dx_2) \\
  \bt_2 &= y_1y_2(x_2\,dx_1 + x_1\,dx_2) \\
  \bt_3 &= x_1x_2y_2\,dx_3 \\
  \bt_4 &= x_1x_2\,dx_4.
 \end{align*}
 Thus, these form a basis for $\Om^*[\chi_7]$ over $B$.

 Similarly, $\Tht[\chi_7]$ is generated by elements of the form
 $\Pi(m\,\tht_i)$ with $m\in M$ and $i\in\{1,2\}$.  Another computer
 calculation shows that every nonzero element of this form is a
 constant multiple of one of the following forms:
 \begin{align*}
  \al_3 &= x_1x_2\tht_1 \\
  \al_4 &= x_1x_2y_1y_2\tht_2.
 \end{align*}
 Thus, these give a basis for $\Tht[\chi_7]$ over $B$.

 Now consider the matrix
 \[ P = \bbm
     1                  & 0                & 0     & 0            \\
     0                  & z_1(1+z_2)       & -2z_1 & 2(1+z_1z_2)  \\
     \rt(1-z_1)(1-2z_2) & z_1(1-2z_2)      & z_1-2 & -2+3z_1-4z_1z_2 \\
     (3z_1-2)z_2/(2\rt) & (1-z_1-z_1z_2)/2 & z_1   & -z_1z_2
    \ebm.
 \]
 By straightforward calculation in $A$, we have
 $\al_i=\sum_{j=1}^4P_{ij}\bt_j$ for $i=1,\dotsc,4$.  We also have
 $\det(P)=2-z_1$, which means that $\al_1,\dotsc,\al_4$ give a basis
 for $A'\ot_A\Om^*[\chi_7]$.  It follows that $\al_1$ and $\al_2$ give a
 basis for $A'\ot_A\Om[\chi_7]$, as claimed.
 \begin{checks}
  embedded/roothalf/forms_check.mpl: check_forms()
 \end{checks}
\end{proof}
The forms $\al_i$ and $\bt_j$ are \mcode+alpha_form[i]+ and
\mcode+beta_form[j]+ in Maple.  The matrix $P$ is \mcode+alpha_beta+.

\begin{proposition}\lbl{prop-d-alpha}
 The forms $\al_i$ satisfy
 \begin{align*}
  d\al_1 &= (9z_1^2 z_2-9 z_1^2+2z_1z_2+9z_1-2)\|n\|^{-1}\om \\
  d\al_2 &= (45 z_1^3 z_2^2-45 z_1^3 z_2+78 z_1^2 z_2-20 z_1 z_2^2
                -12 z_1^2-28 z_1 z_2+12 z_1-8 z_2)/\rt \|n\|^{-1}\om\\
  dz_1\wedge\al_1 &= 2 z_1 (3 z_1-2) (z_1 z_2-z_1+1)\|n\|^{-1}\om \\
  dz_1\wedge\al_2 &= \rt z_1
                        (9 z_1^3 z_2^2-9 z_1^3 z_2-12 z_1^2 z_2^2+
                         26 z_1^2 z_2+4 z_1 z_2^2-4 z_1^2-24 z_1 z_2+
                         8 z_1+8 z_2-4)\|n\|^{-1}\om \\
  dz_2\wedge\al_1 &= 4 z_1 z_2 (2 z_2-1)\|n\|^{-1}\om \\
  dz_2\wedge\al_2 &= 2 \rt (3 z_1^2 z_2-2 z_1 z_2+2 z_1-2) z_2 (2 z_2-1)\|n\|^{-1}\om,
 \end{align*}
 where
 \[ \|n\|=\left(\sum_i(\partial g/\partial x_i)^2\right)^{1/2} =
      (2-z_1)\sqrt{1+z_2}.
 \]
\end{proposition}
\begin{proof}
 The ring $A=\CO_{EX^*}$ is an integral domain, with field of fractions
 \[ K = \R(y_1,y_2)[x_1,x_2]/(x_1^2-u_1,x_2^2-u_2). \]
 It will suffice to verify the above identities in $K\ot_A\Om^2$,
 which is the exterior square over $K$ of the space $K\ot_A\Om^1$.
 From the above description of $K$, we have
 \[ dx_i = \frac{1}{2x_i} du_i = \frac{x_i}{2u_i} du_i \in
      K\{dy_1,\;dy_2\}.
 \]
 Using this, it is not hard to see that $K\ot_A\Om^1$ is freely
 generated over $K$ by $dy_1$ and $dy_2$.  After rewriting everything
 in terms of this basis, all the above equations become
 straightforward.
 \begin{checks}
  embedded/roothalf/forms_check.mpl: check_forms()
 \end{checks}
\end{proof}
Equations for rewriting forms in terms of $dy_1$ and $dy_2$ are in the
list \mcode+forms_to_y+.  The functions $(d\al_i)/\om$ are
\mcode+D_alpha[i]+, and the functions $(dz_i\wedge\al_j)/\om$ are
\mcode+dz_cross_alpha[i,j]+.

Using Proposition~\ref{prop-d-alpha}, we can calculate
$D(f_1\al_1+f_2\al_2)$ for any invariant smooth functions $f_1$ and
$f_2$.  This is implemented in Maple as
\mcode+stokes_alpha([f1,f2])+, and it gives us a supply of functions $f$ with
$\int_{EX^*}f=0$.  Given an approximate integration functional $I$, we
can test the accuracy of $I$ by evaluating $I(f)/\sqrt{I(f^2)}$ for
these functions $f$.

\subsubsection{Integration over triangles}

We next discuss integration over $2$-simplices.  Given any continuous
function $f$ on $\Dl_2$, we write
\[ \int_{\Dl_2} f =
    2\int_{t_1=0}^1\int_{t_2=0}^{1-t_1} f(t_1,t_2,1-t_1-t_2)
      \,dt_2\,dt_1.
\]
In other words, this is the integral with respect to the ordinary
Lebesgue measure on $\Dl_2$ normalised in such a way that the total
area of $\Dl_2$ is one.  A standard exercise shows that
\[ \int_{\Dl_2} t_1^it_2^jt_3^k = 2\frac{i!\,j!\,k!}{(i+j+k+2)!}. \]
By an \emph{$n$'th order quadrature rule} we mean a pair $Q=(a,w)$
with $a\in(\Dl_2)^n$ and $w\in\R^n$.  Given any function $f$ on
$\Dl_2$ we put $I_Q(f)=\sum_i w_i\,f(a_i)$.  Given $n$ distinct points
$a_i$ in $\Dl_2$, and an $n$-dimensional space of functions $P$, there
will typically be a unique vector $w$ of weights such that
$I_{(a,w)}(f)=\int_{\Dl_2}f$ for all $f\in P$, which can be found by
solving a system of linear equations.  It may or may not be
the case that $I_{(a,w)}(f)$ is close to $\int_{\Dl_2}f$ for functions
$f$ not lying in $P$.  In particular, if we take the points $a_i$ to
form a regularly spaced grid, then $I_{(a,w)}(f)$ is a rather poor
approximation to $\int_{\Dl_2}f$ for general $f$.  This was a surprise
to the author, but is apparently well-known to numerical analysts.  It
is better to allow the points $a_i$ to vary as well as the weights
$w_i$.  A naive dimension count then suggests that for any
$3n$-dimensional space $P$ there should be a unique $n$'th order
quadrature rule $Q$ that agrees with integration on $P$, but one has
to solve a complex system of nonlinear equations to find $Q$.  Things
become somewhat simpler if $P$ is preserved by the action of the
symmetric group $\Sg_3$ on $P$.  Dunavant~\cite{du:hde} developed an
elegant theory for this case, which made it tractable to solve the
relevant equations by computer.  He found a quadrature rule of order
$73$ that integrates all polynomials of degree at most $20$ exactly.
It appears that rules obtained in this way have much better behaviour
outside the space $P$ on which they are exact by construction.

Later, Wandzurat and Xiao~\cite{waxi:sqr} gave a rule of order $175$
that is exact to degree 30, and Xiao and Gimbutas gave a rule that is
exact to degree 50.  However, we have not been able to get correct
answers from this last rule, so either we are misunderstanding the
conventions or there is some kind of transcription error in the tables
in the paper.  We have therefore used the Wandzurat-Xiao rule
instead.

Rules as above are represented in Maple by instances of the class
\mcode+triangle_quadrature_rule+, which is declared in the file
\fname+quadrature/quadrature.mpl+.  In the \fname+quadrature+
subdirectory of the \mcode+data+ directory there is a file
\mcode+wandzurat_xiao_30.mpl+.  Reading this file creates an object
representing the Wandzurat-Xiao rule, and assigns it to the variable
\mcode+wandzurat_xiao_30+.  There is also another file
\mcode+dunavant_19.mpl+ which implements the Dunavant rule (which is
less accurate but faster).

Next, recall that in Section~\ref{sec-E-charts} we discussed a
triangulation of the fundamental domain $F_{16}$ using certain
barycentric coordinate maps $p:T\xra{\simeq}\Dl_2$ for certain subsets
$T\sse X$.  For each point $a_i$ in our quadrature rule, we can use
Remark~\ref{rem-barycentric-inverse} to find $a'_i=p^{-1}(a_i)\in T$.
The components of $p$ are rational functions in the
coordinates $x_i$, so it is straightforward to differentiate them and
calculate the Jacobian of $p$ at $a'_i$, say $u_i$.  If $w_i$ is the
$i$'th weight of our quadrature rule, then $\sum_iw_iu_i^{-1}f(a_i)$
is a good approximation to $\int_Tf$, and we can take the sum over all
simplices to get a functional $J(f)$ approximating $\int_{F_{16}}f$.

The Gauss-Bonnet theorem says that $J$ of the curvature function $K$
should be $(-4\pi)/|G|=-\pi/4$.  With our $192$-simplex triangulation
we in fact have $|J(K)+\pi/4|<10^{-27}$.  Similarly, for a function
$f$ of the form $D(z_1^iz_2^j\al_k)$ we would ideally have $J(f)=0$,
and in fact we have $|J(f)|\leq 3\tm 10^{-27}\sqrt{J(f^2)}$ provided
that $i+j\leq 10$.  Subdivision makes a big difference here; with the
original $48$-simplex triangulation, errors are larger by a factor of
$10^8$ or so.

\subsubsection{Faster integration on \texorpdfstring{$EX^*$}{EX*}}

After constructing a quadrature rule for $EX^*$, we can tabulate the
integrals of monomials $z_1^iz_2^j$, which then gives a fast way of
integrating arbitrary invariant polynomials.  It is also useful to
extend this slightly and include monomials $z_1^iz_2^j\|n\|^{-k}$ for
$0\leq k\leq 4$ say.  This is enough for some purposes, but in other
cases we need to integrate more general functions (such as
exponentials of polynomials) which cannot easily be expressed as
linear combinations of some standard basis.  It is therefore desirable
to have an approximate integration functional of the form
$I(f)=\sum_{i=1}^n w_if(a_i)$ where $n$ is not too large, but the
approximation is reasonably accurate.

Suppose we fix $n$, and choose an $n$-dimensional subspace $P$ of
smooth invariant functions on $EX^*$.  Let $p_1,\dotsc,p_n$ be a basis
for $P$.  For any $n$-tuple of points $a_i\in EX^*$, we let
$\dl(\un{a})$ denote the determinant of the matrix with entries
$p_i(a_j)$.  The literature on other quadrature problems suggests that
we should aim to choose $\un{a}$ so that $|\dl(\un{a})|$ is as large
as possible.  Note that this problem is independent of the choice of
basis $\{p_i\}$, because a different choice would just change $\dl$ by
a constant factor.  With the obvious kind of monomial bases, it works
out that $\dl(\un{a})$ is always extremely small, but it can be
increased by many orders of magnitude if we choose the points $a_i$
appropriately.

In our largest calculation of this kind, we took $P$ to be the span of
$256$ monomials in $z_1$ and $z_2$, including all monomials of degree
at most $21$, plus some monomials of degree $22$.  For a randomly chosen
set of $256$ points it is typical that $\log_{10}|\dl|<-2800$ or
so, but by numerical optimization we found a set with
$\log_{10}|\dl|\simeq -2539$.

We next want to choose the weights $w_i$.  One option is to set these
weights such that $I(p_j)=\int_{EX^*}p_j$ for all $j$, which can be
done by solving a system of linear equations.  We then find that some
of the weights are negative.  This is an undesirable feature, leading
to reduced accuracy when integrating functions outside the space $P$.
We therefore extended our list of monomials to include all $300$
monomials of degree at most $23$, and then chose the weights $w_i$ to
minimise $\sum_j(I(p_j)-\int p_j)^2$ subject to the constraints
$w_i\geq 0$.  (This was done using Maple's \texttt{LSSolve()}
command.)  It turns out that there are 18 indices $i$ such that
$w_i=0$, so we really only use 238 sample points.  With these weights
we have $|I(K)+\pi/4|\simeq 10^{-17.1}$, and if $f=D(z_1^iz_2^j\al_k)$
with $i+j\leq 10$ then $|I(f)|\leq 10^{-14.5}\sqrt{I(f^2)}$.

Quadrature rules as above are represented by instances of the class
\mcode+E_quadrature_rule+, which is declared in the file
\fname+embedded/E_quadrature.mpl+.  The specific rule described above
is stored in the file \fname+quadrature_frobenius_256a.m+ in the
directory \fname+data/embedded/roothalf+.  After reading that file,
one can enter the following to integrate $1/(1+z_1)$ (for example):
\begin{mcodeblock}
 Q := eval(quadrature_frobenius_256a):
 Q["int_z",1/(1+z[1])];
\end{mcodeblock}
One can test the accuracy of \mcode+Q+ (as described above) using the
methods \mcode+Q["curvature_error"]+ and
\mcode+Q["stokes_error",10]+.  Various other methods are documented in
the code.

One can regenerate the object \mcode+Q+ using the function
\mcode+build_data["E_quadrature_rule"]()+ defined in the file
\fname+build_data.mpl+.  However, there is not a very compelling
reason to do this, as we can check that \mcode+Q+ has the desired
properties without needing to regenerate it.  Also, the process is
very slow, and may take several days to run. 

\section{Classifying \texorpdfstring{$EX^*$}{EX*}}
\lbl{sec-classify-roothalf}

Theorems~\ref{thm-classify-cromulent} and~\ref{thm-H-universal} tell
us that there are cromulent isomorphisms
$HX(b)\xra{q}EX^*\xra{r}PX(a)$ for suitable values $a,b\in(0,1)$.
In this section, we discuss numerical methods that enable us to
calculate approximations to $a$, $b$, $q$ and $r$.  Note that the
methods of Section~\ref{sec-P-H} allow us to compute $a$ and $r$ from
$b$ and $q$, or \emph{vice versa}.  We have tried several different
approaches.  The most successful will be described first, in
Section~\ref{sec-rescaling}.  We will then outline one other approach
in Section~\ref{sec-energy}.

Our current estimates are $a\simeq 0.0983562$ and $b\simeq 0.8005319$.
We have some reason to hope that all the quoted digits are accurate,
but we have not performed a rigorous error analysis.

\subsection{Hyperbolic rescaling}
\lbl{sec-rescaling}

Theorem~\ref{thm-H-universal} tells us that there is a conformal
covering map $q\:\Dl\to EX^*$, which induces an isomorphism
$HX(b)\to EX^*$ for some $b$.  Let $g$ denote the Riemannian metric
that $EX^*$ inherits from $\R^4$, and let $g_{\hyp}$ denote the
standard hyperbolic metric $ds^2=4|dz|^2/(1-|z|^2)^2$ on $\Dl$.  Recall
that Remark~\ref{rem-curvature-z} gives a formula for the Gaussian
curvature $K(g)$, and it is a standard fact that
$K(g_{\hyp})=-1$.

\begin{proposition}\lbl{prop-rescaling}
 There is a unique real analytic function $f$ on $EX^*$ such that
 $K(e^{2f}g)=-1$.  Moreover, this function is $G$-invariant, and it
 satisfies $q^*(e^{2f}g)=g_{\hyp}$.  The curves
 $C_0,\dotsc,C_8\subset EX^*$ are geodesics with respect to the metric
 $e^{2f}g$.
\end{proposition}

\begin{remark}
 Note here that when we multiply the metric tensor $g$ by $e^{2f}$,
 this multiplies lengths by $e^f$, and areas by $e^{2f}$.
\end{remark}

The proof depends on the following formula:
\begin{lemma}\lbl{lem-rescaled-curvature}
 Let $Z$ be a smooth oriented surface equipped with a Riemannian
 metric $g$.  Let $\Delta=\Dl_g$ denote the associated Laplacian
 operator, and let $K(g)$ denote the Gaussian curvature.  Then for any
 smooth function $f$ on $Z$, we have
 \[ K(e^{2f}g) = (K(g) - \Dl(f))/e^{2f}. \]
\end{lemma}
\begin{proof}
 See Chapter~V of~\cite{scya:ldg}, for example.
\end{proof}

\begin{proof}[Proof of Proposition~\ref{prop-rescaling}]
 Because $q$ is a conformal covering, we see that $q^*(g)$ is a
 positive multiple of $g_{\hyp}$, say
 $q^*(g)=e^{-2\tf}g_{\hyp}$ for some real analytic function
 $\tf$ on $\Dl$.  Note that $\tPi$ acts isometrically on $\Dl$, and
 also (via $\tP/\Pi=G$) on $EX^*$, and $q$ is equivariant for these
 actions.  It follows that $\tf$ is invariant under $\tPi$, so it has
 the form $\tf=q^*(f)$ for some $G$-invariant real analytic function
 $f$ on $EX^*$.  Now put $g_1=e^{2f}g$.  We have
 $q^*K(g_1)=K(e^{2\tf}q^*(g))=K(g_{\hyp})=-1$, but $q$ is
 surjective so $K(g_1)=-1$ as required.  Now $q$ is a local isometry
 from $(\Dl,g_{\hyp})$ to $(EX^*,g_1)$, and it carries the
 geodesics $C_i\subset\Dl$ to the curves $C_i\subset EX^*$, so the
 latter must also be geodesic.

 Now suppose we have a function $u$ with $K(e^{2u}g_1)=-1$; we
 claim that $u=0$.  In the proof we will use the gradient operator
 $\nabla$, the Laplacian operator $\Dl$, and the integration operator
 $\int_{EX^*}$; these are all defined using the metric $g_1$.  Using
 Lemma~\ref{lem-rescaled-curvature} we obtain $e^{2u}-1-\Dl(u)=0$.
 It is a standard fact that for any functions $a,b\in C^\infty(EX^*)$
 we have
 \[ \int_{EX^*} a\Dl(b) = -\int_{EX^*} \ip{\nabla(a),\nabla(b)}. \]
 Using this, we get
 \[ \int_{EX^*}\left(u(e^{2u}-1)+\|\nabla(u)\|^2\right)
    = \int_{EX^*}\left(u(e^{2u}-1+\Delta(u))\right) = 0.
 \]
 By considering the cases $u\geq 0$ and $u\leq 0$ separately, we see
 that $u(e^{2u}-1)\geq 0$, with equality only where $u=0$.  In view of
 this, the above integral formula shows that $u=0$ everywhere.  This
 shows that $f$ is uniquely characterised by the fact that
 $K(e^{2f}g)=-1$.
\end{proof}

To go further, we will need to discuss the curves $c_k\:\R\to\Dl$ as
well as the curves $c_k\:\R\to EX^*$.  We will distinguish between
them by writing $c_{Hk}$ for the former, and $c_{Ek}$ for the latter.
Similarly, we will write $v_{Hj}$ and $v_{Ej}$ for the usual points in
$\Dl$ and $EX^*$.

\begin{corollary}\lbl{cor-side-lengths}
 Let $f$ be the unique function such that $K(e^{2f}g)=-1$, and put
 \begin{align*}
  L_0 &= \int_{\pi/4}^{\pi/2} e^{f(c_{E0}(t))}\|c'_{E0}(t)\|\,dt \\
  L_1 &= \int_{0}^{\pi/2} e^{f(c_{E1}(t))}\|c'_{E1}(t)\|\,dt \\
  L_3 &= \int_{0}^{\pi/2} e^{f(c_{E3}(t))}\|c'_{E3}(t)\|\,dt \\
  L_5 &= \int_{0}^{\pi} e^{f(c_{E5}(t))}\|c'_{E5}(t)\|\,dt.
 \end{align*}
 If $b$ is the parameter such that $EX^*\simeq HX(b)$, and the points
 $v_{Hi}\in\Dl$ are defined in terms of $b$ as in
 Definition~\ref{defn-v-H}, then we have
 \begin{align*}
  L_0 &= \dhyp(v_{H3},v_{H6})  = 2\arctanh((b_+-\rt b)/b_-) \\
  L_1 &= \dhyp(v_{H0},v_{H6})  = 2\arctanh((\rt-b_-)/b_+) \\
  L_3 &= \dhyp(v_{H3},v_{H11}) = 2\arctanh((1-b_-)/b) \\
  L_5 &= \dhyp(v_{H0},v_{H11}) = 2\arctanh(b_+-b).
 \end{align*}
\end{corollary}
\begin{proof}
 The domain $HF_{16}(b)\subset\Dl$ is a hyperbolic polygon, with
 geodesic edges, and vertices $v_{H0},v_{H3},v_{H6}$ and $v_{H11}$.
 The map $q$ is a local isometry, which sends $v_{Hi}$ to $v_{Ei}$.
 It follows that the geodesic distance from $v_{Hi}$ to $v_{Hj}$ in
 $\Dl$ is the same as the geodesic distance from $v_{Ei}$ to $v_{Ej}$
 in $EX^*$.  As the curves $C_i\subset EX^*$ are geodesic, the
 distances in $EX^*$ are given by the indicated integrals.

 The distances in $\Dl$ are given by the standard formula
 $\dhyp(z,w)=2\arctanh(m(z,w))$, where 
 \[ m(z,w) = \left|\frac{z-w}{1-\ov{z}w}\right|. \]
 We therefore need to show that $m(v_{H3},v_{H6})=(b_+-\rt b)/b_-$ and
 so on.  It is not hard to see that $(b_+-\rt b)/b_-\geq 0$, so it
 will suffice to show that $m(v_{H3},v_{H6})^2=((b_+-\rt b)/b_-)^2$,
 and this can be done by straightforward algebraic manipulation.  The
 same method works for the other three cases.
 \begin{checks}
  hyperbolic/HX_check.mpl: check_side_lengths()
 \end{checks}
\end{proof}

We now describe an algorithm based on the above results.  The full
algorithm can be carried out by executing the function
\mcode+build_data["EH_atlas",0]()+ defined in \mcode+build_data.mpl+,
which in turn invokes the functions \mcode+build_data["EH_atlas",i]()+
for $i$ from $1$ to $4$, each of which carries out a subset of the
steps described below.  All of these steps are implemented by methods
of the class \mcode+EH_atlas+ (which is declared in the file
\fname+embedded/roothalf/EH_atlas.mpl+) and some other related
classes.  We can enter 
\begin{mcodeblock}
 EHA := `new/EH_atlas`():
\end{mcodeblock}
to create a new object of the required type.

First, we need a quadrature rule $Q(f)=\sum_iw_if(a_i)$,
which is intended to approximate $\int_{EX^*}f$ when $f$ is
$G$-invariant.  It is not important for this approximation to be
accurate, we just want the seminorm $\|f\|_Q=\sqrt{Q(f^2)}$ to be a
reasonable measure of the size of $f$, at least for the functions $f$
that arise in our calculations.  After constructing a suitable
instance \mcode+Q+ of the class \mcode+E_quadrature_rule+, we can
enter
\begin{mcodeblock}
 EHA["quadrature_rule"] := eval(Q):
\end{mcodeblock}
to attach the quadrature rule to the atlas.

Next, we need to choose a finite-dimensional subspace $F$ of
invariant, real analytic functions on $EX^*$, in which we will search
for an approximation to the rescaling function $f$.  Two possibilities
are as follows:
\begin{align*}
 F_{\text{poly}}(d) &=
  \{\text{ polynomials } p(z_1,z_2)
     \text{ of total degree at most } d\} \\
 F_{\text{pade}}(d) &=
  \left\{\text{ rational functions } \frac{p(z_1,z_2)}{q(z_1,z_2)},\;
     \max(\deg(p),\deg(q)) \leq d,\; q(0,0)=1\right\}.
\end{align*}
We have found that $F_{\text{pade}}$ works better than
$F_{\text{poly}}$ even if the degrees are chosen so that the total
number of parameters is the same.  We do not fully understand why,
although it is reminiscent of the situation with Pad\'e approximations
in one variable.

One can search through the above spaces using one of the following
methods:
\begin{mcodeblock}
 EHA["find_rescale_poly",d];
 EHA["find_rescale_poly_alt",d];
 EHA["find_rescale_pade",d];
 EHA["find_rescale_pade_alt",d];
\end{mcodeblock}
It is best to call these methods repeatedly starting with $d=2$.  A
measure of success is stored in the field
\mcode+EHA["rescaling_error"]+, and this should decrease on each
iteration.  When this measure has settled down, one can increase $d$
by one.

The first two of the above methods use $F_{\text{poly}}(d)$, and the
second two use $F_{\text{pade}}(d)$.  The alternative methods
\mcode+find_rescale_poly+ and \mcode+find_rescale_pade+ are coded
using a very direct translation of the mathematical problem that we
seek to solve, so it is easy to check their correctness.  The methods
\mcode+find_rescale_poly_alt+ and \mcode+find_rescale_pade_alt+ are
harder to understand but much more efficient: they precompute various
vectors and matrices, and thereby convert the problem to linear
algebra, as far as possible.  The polynomial case works as follows:
\begin{itemize}
 \item We enumerate the monomials of degree at most $d$ as
  $m_1,\dotsc,m_r$.
 \item The sample points for the quadrature rule are $u_1,\dotsc,u_n$,
  with weights $w_i\geq 0$.  We set \mcode+rweights+ to be the vector
  with entries $\sqrt{w_i}$, and we set \mcode+kpoints+ to be the
  vector whose entries are the values of the curvature at the points
  $u_i$ (computed using Remark~\ref{rem-curvature-z}).  We also set
  \mcode+ones+ to be the vector of length $n$ whose entries are all
  one.
 \item Similarly, we set \mcode+mpoints+ and \mcode+lpoints+ to be
  matrices with entries $m_i(u_j)$ and $\Delta(m_i)(u_j)$ (computed
  using Proposition~\ref{prop-Delta-prime}).
 \item Now if $f=\sum_ja_jm_j$, then the values of $f$ and
  $K(g)-\Delta(f)$ are given by \mcode+mpoints.a+ and
  \mcode+kpoints-lpoints.a+.  This makes it easy to compute the
  objective function $FT(a)=Q(((K(g)-\Delta(f))/e^{2f}+1)^2)$.
 \item More precisely, $FT(a)$ is $\sum_iF_i(a)^2$, where $F_i(a)$ is
  $\sqrt{w_i}$ times the value of $(K(g)-\Delta(f))/e^{2f}+1$ at
  $u_i$.  This is useful, because there are special algorithms (used
  by Maple's \mcode+LSSolve+ command) for minimising a sum of
  squares.
 \item The above framework gives an efficient method for calculating
  the vector $F(a)$ and also the matrix of derivatives
  $\partial F_i(a)/\partial a_j$.  These are the ingredients that we
  need in order to use the \mcode+LSSolve+ command with the
  \mcode+objectivejacobian+ option.
\end{itemize}

The rational case is more complicated.  Here we have $f=g/h$, and the
values of $f$ and $\Delta(f)$ do not depend linearly on the
coefficients of $g$ and $h$.  However, one can construct linear
differential operators $P_i$ and $Q_j$ such that
\[ \Delta(f) =
    \frac{P_1(g)}{h} +
    \frac{P_2(g)Q_2(h)+P_3(g)Q_3(h)+P_4(g)Q_4(h)}{h^2} +
    P_5(g)\frac{Q_2(h)Q_5(h)+Q_3(h)Q_6(h)}{h^3}.
\]
(This is just a version of the quotient rule for second derivatives.)
One can again precompute the values $P_i(m_k)(u_l)$ and
$Q_j(m_k)(u_l)$ and thereby streamline the calculation of the
objective function, and of its derivatives with respect to the
coefficients of $g$ and $h$.

After using these methods to find the rescaling function
$f$, we can enter \mcode+EHA["log_rescale_z"](z)+ to see $f$ as an
expression in $z_1$ and $z_2$, or \mcode+EHA["log_rescale_x"](x)+ to
see it as an expression in $x_1,\dotsc,x_4$.

Once we have found $f$, we define lengths $L_i$ by the integrals
specified in Corollary~\ref{cor-side-lengths}.  We then use numerical
methods to find $b_5$ such that
$L_5=2\arctanh(2b_5^2+1-2b_5\sqrt{1+b_5^2})$, and similarly for $b_3$,
$b_1$ and $b_0$, using the formulae in
Corollary~\ref{cor-side-lengths}.  If our approximations are good,
then $b_0,b_1,b_3$ and $b_5$ should all be close to the parameter $b$
such that $EX^*\simeq HX(b)$.  We can thus get an imperfect measure of
the accuracy of our approximations from the differences $|b_i-b_j|$;
these are at most $10^{-7.4}$ in our best attempt.

The above algorithm is implemented by the method
\mcode+EHA["find_a_H"]+.  The length $L_k$ is stored as
\mcode+EHA["curve_lengths"][k]+, and $b_k$ is stored as
\mcode+EHA["curve_a_H_estimates"][k]+.  The average of these is
\mcode+EHA["a_H"]+, and the maximum discrepancy between them is
\mcode+EHA["a_H_discrepancy"]+.

Having found $b$, we can construct an isomorphism $HX(b)\to PX(a)$ by
the methods of Sections~\ref{sec-a-from-b} and~\ref{sec-b-from-a}.
This gives objects of class \mcode+H_to_P_map+ and
\mcode+P_to_H_map+.  These can be assigned to the fields
\mcode+EHA["H_to_P_map"]+ and \mcode+EHA["P_to_H_map"]+, in order to
keep everything packaged together in a single object.  These steps are
not included in the function \mcode+build_data["EH_atlas",0]()+, but
are instead in the functions \mcode+build_data["H_to_P_map"]()+ and
\mcode+build_data["P_to_H_map"]()+.

We next want to approximate the map $q\:\Dl\to EX^*$.
\begin{remark}\lbl{rem-H-to-E-method}
 The broad outline of our method is as follows:
 \begin{itemize}
  \item[(a)] It is not hard to see that $q(c_{kH}(t))=c_{kE}(u_k(t))$
   for a certain function $u_k(t)$, and to find Fourier series
   approximations to $u_k(t)-t$ by numerical integration of arc lengths.
  \item[(b)] Given a point $a\in EX^*$, it is not too hard to find a
   polynomial map $p_a\:\Dl\to\R^4$ which satisfies $p_a(0)=a$ and is a
   good approximation to an isometry $\Dl\to EX^*$, at least if we
   consider points close to the origin in $\Dl$.  (We will call these
   maps \emph{hyperbolic charts}.)  If $a\in C_k$ for some $k$ then we
   can exploit information from~(a) to find $p_a$; otherwise we use a
   more general method with power series.  Experiment suggests that
   our value of $p_a(z)$ can only be trusted for fairly small values
   of $|z|$, perhaps $|z|<0.1$ or so.
  \item[(c)] We then show that there exists a M\"obius map
   $m_a\in\Aut(\Dl)$ such that $p_a(z)\simeq q(m_a(z))$ for small $z$.
   Equivalently, for $w$ close to $m_a(0)$ we have
   $q(w)\simeq p_a(m_a^{-1}(w))$.  Thus, to find $q$, we should try to
   find $p_a$ and $m_a$ for a reasonable supply of points
   $a\in F_{16}$.  Methods for $p_a$ were discussed above, but we still
   need to consider $m_a$.
  \item[(d)] Given nearby points $a,b\in EX^*$, we can find
   $p_a^{-1}(b)$ and $p_b^{-1}(a)$ numerically, and the values
   $\dhyp(0,p_a^{-1}(b))$ and $\dhyp(p_b^{-1}(a),0)$ give two different
   estimates for the geodesic distance between $a$ and $b$.  (The
   discrepancy between them gives a check on the accuracy of our
   methods.)
  \item[(e)] We now choose a reasonably fine grid of points
   $a_1,\dotsc,a_N$ in $F_{16}$, and try to find the points
   $\bt_i=q^{-1}(a_i)\in HF_{16}(b)$.  For any points $a_i$ that lie in
   $\partial F_{16}$, we can do this using~(a).  For the remaining
   points, we note that when $a_i$ is a neighbour of $a_j$, we can
   estimate the geodesic distance between them as in~(d), and we should
   then have $\dhyp(\bt_i,\bt_j)=d(a_i,a_j)$.  We therefore choose the
   points $\bt_i$ to minimize a suitable measure of the overall
   discrepancy between the lengths $\dhyp(\bt_i,\bt_j)$ and $d(a_i,a_j)$.
  \item[(f)] We now need the M\"obius maps $m_i$ such that
   $p_i=qm_i$.  It is not hard to see that these must have the form
   \[ m_i(z)=\lm_i\frac{z+\ov{\lm_i}\bt_i}{1-\lm_i\ov{\bt_i}z}
      \hspace{4em}
      m_i^{-1}(w) = \ov{\lm_i}\frac{w-\bt_i}{1-\ov{\bt_i}w}.
   \]
   If $a_j$ is adjacent to $a_i$ then we find that $m_i^{-1}(\bt_j)$
   should be equal to $p_i^{-1}(a_j)$.  We can again calculate
   $p_i^{-1}(a_j)$ numerically, and this gives
   \[ \lm_i = (\bt_j-\bt_i)/(1 - \ov{\bt_i}\bt_j)/p_i^{-1}(a_j). \]
   We can perform this calculation for every $j$ such that $a_j$ is
   adjacent to $a_i$, and then take a kind of average to get a final
   estimate for $\lm_i$.  (Of course, in the averaging process we
   impose the constraint $|\lm_i|=1$.)
  \item[(g)] Now given a point $z\in\Dl$, we can approximate $q(z)$ as
   follows: we find $\gm\in\tPi$ such that $\gm(z)\in F_{16}$, then
   find $i$ such that $\gm(z)$ is as close as possible to $\bt_i$,
   then take $q(z)=\gm^{-1}(p_i(m_i^{-1}(\gm(z))))$ (using the action
   of $\tPi$ on $EX^*$ via $\tPi/\Pi=G$).  We can use this method to
   calculate $q(z)$ for a large sample of points $z\in\Dl$, and then
   use numerical techniques to find an approximation to $q(x+iy)$ using
   rational functions in $x$ and $y$.
 \end{itemize}
\end{remark}

We now discuss the above points~(a) to~(e) in more detail.

First, as $q$ gives a cromulent isomorphism $HX(b)\to EX^*$, we must
have $q(v_{Hi})=v_{Ei}$ for all $i$, and $q(c_{Hk}(\R))=C_{Ek}$.  As
$q\circ c_{Hk}\:\R\to C_{Ek}$ and $c_{Ek}\:\R\to C_{Ek}$ are both
$2\pi$-periodic universal coverings, it is not hard to see
that we must have $q(c_{Hk}(t))=c_{Ek}(u_k(t))$ for some strictly
increasing diffeomorphism $u_k\:\R\to\R$ with
$u_k(t+2\pi)=u_k(t)+2\pi$.  This in turn means that the function
$u_k(t)-t$ is $2\pi$-periodic, so it can be represented by a
Fourier series.  The maps $c_k\:\R\to\Dl$ were defined so as to have
constant speed with respect to the hyperbolic metric; let that speed
be $s_k$.  As $q$ is locally isometric, we can differentiate the
relation $q(c_{Hk}(u_k^{-1}(t)))=c_{Ek}(t)$ to get
\[ \frac{d}{dt}u_k^{-1}(t) =
     s_k^{-1}\|c'_{Ek}(t)\|e^{f(c_{Ek}(t))}.
\]
We can integrate this numerically to find a Fourier series
approximation to $u_k^{-1}(t)-t$.  From this we can obtain a Fourier
approximation to $u_k(t)-t$, and thus approximate formulae for
$q(c_{Hk}(t))$.  To give an idea of the size of the dominant terms, we
have
\begin{align*}
 u_0(t) - t &\simeq  0.017 \sin(4t) \\
 u_1(t) - t &\simeq -0.169 \sin(2t) + 0.010 \sin(4t) - 0.001 \sin(6t) \\
 u_3(t) - t &\simeq -0.074 \sin(2t) + 0.002 \sin(4t) \\
 u_5(t) - t &\simeq -0.362 \sin(t) + 0.026\sin(2t) - 0.001\sin(3t).
\end{align*}
These are calculated by the method \mcode+EHA["find_u",d]+ (where $d$
controls the number of terms in the various Fourier series).  After
invoking this method, one can calculate $u_k(t)$ and $u_k^{-1}(t)$ as
\mcode+EHA["u"][k](t)+ and \mcode+EHA["u_inv"][k](t)+.

We next discuss point~(b) in Remark~\ref{rem-H-to-E-method}.

\begin{proposition}\lbl{prop-local-isometry}
 Let $Z$ be an oriented surface with a Riemannian metric of curvature
 $-1$.  Let $a$ be a point in $Z$, and let $v$ be a nonzero tangent
 vector at $a$.  Then there is a unique germ of an oriented local
 isometry $p\:\Dl\to Z$ such that $p(0)=a$ and $p'(0)\in\R^+v$.
\end{proposition}
\begin{proof}
 Local existence of $p$ is a theorem of Riemann; a convenient
 reference is~\cite[Theorem 2.4.11]{wo:scc}.  For uniqueness, it will
 suffice to prove the following: if $u$ is a germ of an oriented local
 isometry $\Dl\to\Dl$ with $u(0)=0$ and $u'(0)>0$, then $u$ is the
 identity.  It is clear that $u$ must act as a rotation on the tangent
 space $T_0\Dl$, so the condition $u'(0)>0$ forces $u'(0)=1$.  The
 exponential map $\exp\:T_0\Dl\to\Dl$ is characterised by its metric
 properties, so it commutes with $u$, and $T_0u=1$ so $u=1$.
\end{proof}

\begin{corollary}\lbl{cor-local-isometry}
 Let $p\:U\to EX^*$ be a hyperbolic chart (where $U$ is a disc around
 $0$ in $\Dl$).  Then there is a M\"obius map $m\in\Aut(\Dl)$ such
 that $p=qm$ on $U$.
\end{corollary}
\begin{proof}
 As $q\:\Dl\to EX^*$ is a covering map and $U$ is simply connected, we
 can choose $m\:U\to\Dl$ such that $p=qm$.  As both $p$ and $q$ are
 orientation-preserving isometries, the same is true of $m$.  Now put
 $\al=m(0)\in\Dl$, and let $\lm$ denote the unit complex number such
 that $m'(0)$ is a positive multiple of $\ov{\lm}$.  Put
 $m_1(z)=(z+\lm\al)/(\lm+\ov{\al}z)$ (so
 $m_1^{-1}(z)=(z-\al)/(1-\ov{\al}z)$).  We find that $m$ and $m_1$ are
 both orientation-preserving isometries of $U$ into $\Dl$ such that
 $m'(0)$ and $m'_1(0)$ are positive multiples of each other.  It
 follows (by the uniqueness clause in the Proposition) that $m=m_1$.
\end{proof}

Charts $p$ as above are represented by instances of the class
\mcode+EH_chart+, which is declared in the file
\fname+embedded/roothalf/EH_atlas.mpl+.  It extends the class
\mcode+E_chart+, which was discussed in
Section~\ref{sec-roothalf-charts}.  The algorithm to make a chart
isometric is actually coded in the \mcode+isometrize+ method of the
\mcode+E_chart+ class; this is invoked automatically by the methods
that initialize instances of the \mcode+EH_chart+ class.  In more
detail, the algorithm is as follows.  We start with an approximate
polynomial conformal chart $p_0\:\C\to EX^*$ as in
Proposition~\ref{prop-frame-chart}, which can be constructed by
methods of the \mcode+E_chart+ class.  It is then not hard to show
that there are unique numbers $a_1\in\R^+$ and $a_2\in\C$ such that
the map $p_2(z)=p_0(a_1z+a_2z^2)$ is isometric to first order.  We can
then try to find $a_3$ such that the map $p_3(z)=p_2(z+a_3z^3)$ is
isometric to second order.  This involves solving a system of
inhomogeneous linear equations for the real and imaginary parts of
$a_3$.  As the curvature of $g_1$ is not exactly equal to $-1$, these
equations will not usually be solvable.  However, we can choose $a_3$
to minimize the mean square error in these equations, and then proceed
to find coefficients $a_4$, $a_5$ and so in in the same way.

As mentioned previously, there a different method that is available
for charts centred on one of the curves $C_k$.  Suppose that
$a=c_{Ek}(t_0)$, and that we have found a good approximation to the
conformal chart $p_0$ with $p_0(t)=c_{Ek}(t_0+t))$ for small $t\in\R$
(as discussed in Section~\ref{sec-roothalf-charts}).  Put
$t_1=u_k^{-1}(t_0)$.  The function $u_k\:\R\to\R$ is real analytic, so
it can be extended (using power series, for example) to give a
holomorphic function on a neighbourhood of the point
$t_1=u_k^{-1}(t_0)$.  Now put
\begin{align*}
 \lm &= c'_{Hk}(t_1)/|c'_{Hk}(t_1)|\in S^1 \\
 \al &= -c_{Hk}(t_1)/\lm \\
 \bt &= -\lm\al \\
 m(z) &= \lm(z-\al)/(1-\ov{\al}z),
\end{align*}
so $m\in\Aut(\Dl)$ with $m(0)=c_{Hk}(t_1)$ and
$m'(0)\in\R^+.c'_{Hk}(t_1)$.  Finally, recall that $s_k$ denotes the
(constant) speed of the map $c_{Hk}\:\R\to\Dl$.

\begin{proposition}\lbl{prop-curve-chart}
 With notation as above, the map
 \[ p(z) = p_0(u_k(t_1+2s_k^{-1}\arctanh(z))-t_0) \]
 is a hyperbolic chart at $a$.  More specifically, we have
 $p(z)=q(m(z))$.  In particular, we have $q^{-1}p(0)=m(0)=-\lm\al=\bt$.
\end{proposition}
\begin{proof}
 Note that $m^{-1}(c_{Hk}(\R))$ is a geodesic in $\Dl$ which is
 tangent to $\R$ at $0$; it follows that
 $m^{-1}(c_{Hk}(\R))=(-1,1)$.  The standard speed one parametrisation
 of $(-1,1)$ is $t\mapsto\tanh(t/2)$.  It follows that
 $c_{Hk}(t_1+t)=m(\tanh(s_kt/2))$.

 We now claim that $p(z)=q(m(z))$.  Both $p(z)$ and $q(m(z))$ are
 holomorphic, so it will suffice to prove this for small real values
 of $z$.  When $z$ is real we see that
 $u_k(t_1+2s_k^{-1}\arctanh(z))-t_0$ is also real, so
 \begin{align*}
  p(z) &= c_{Ek}(u_k(t_1+2s_k^{-1}\arctanh(z))) \\
       &= q(c_{Hk}(t_1+2s_k^{-1}\arctanh(z))) = q(m(t))
 \end{align*}
 as required.
\end{proof}

For $j\in\{0,3,6,11\}$ one can add an isometric chart centred at the
point $v_j$ to the atlas using the method
\mcode+EHA["add_vertex_chart",j]+.  Now suppose that
$k\in\{0,1,3,5\}$, so the set $c_k^{-1}(F_{16})$ is a closed interval
$[a,b]$ for some $a$ and $b$.  Then the method
\mcode+EHA["add_curve_chart",k,t]+ adds an isometric chart centred at
$a+t(b-a)$ (so the natural domain for $t$ is $[0,1]$).  Finally, for a
point $x_0$ in the interior of $F_{16}$, we can use the method
\mcode+EHA["add_centre_chart",x0]+ to add a chart centred at $x_0$.
The function \mcode+build_data["EH_atlas",2]()+ adds a total of 119
charts to the atlas created by \mcode+build_data["EH_atlas",1]()+.
They are chosen so that the centres form an approximately equilateral
triangular grid with respect to the rescaled hyperbolic metric.  (To
make everything fit, some triangles on the edge of $F_{16}$ have to
deviate strongly from being equilateral, but the ones in the interior
are quite regular.)  The corresponding points in $HF_{16}(b)$ can be
displayed as follows: 
\begin{center}
 \begin{tikzpicture}[scale=12]
  \draw[maplegreen] (0.000,0.000) -- (0.020,0.020);
  \draw[maplegreen] (0.000,0.000) -- (0.032,0.032);
  \draw[mapleblue] (0.000,0.000) -- (0.034,0.000);
  \draw[maplecyan] (0.568,0.364) -- (0.547,0.375);
  \draw[maplemagenta] (0.568,0.364) -- (0.557,0.341);
  \draw[maplegrey] (0.568,0.364) -- (0.550,0.346);
  \draw[maplecyan] (0.450,0.450) -- (0.467,0.434);
  \draw[maplegreen] (0.450,0.450) -- (0.422,0.422);
  \draw[maplegrey] (0.450,0.450) -- (0.434,0.418);
  \draw[maplemagenta] (0.480,-0.000) -- (0.481,0.023);
  \draw[mapleblue] (0.480,-0.000) -- (0.470,0.000);
  \draw[maplecyan] (0.467,0.434) -- (0.493,0.412);
  \draw[maplegrey] (0.467,0.434) -- (0.434,0.418);
  \draw[maplegrey] (0.467,0.434) -- (0.463,0.398);
  \draw[maplecyan] (0.493,0.412) -- (0.519,0.393);
  \draw[maplegrey] (0.493,0.412) -- (0.463,0.398);
  \draw[maplegrey] (0.493,0.412) -- (0.492,0.380);
  \draw[maplecyan] (0.519,0.393) -- (0.547,0.375);
  \draw[maplegrey] (0.519,0.393) -- (0.492,0.380);
  \draw[maplegrey] (0.519,0.393) -- (0.521,0.362);
  \draw[maplegrey] (0.547,0.375) -- (0.521,0.362);
  \draw[maplegrey] (0.547,0.375) -- (0.550,0.346);
  \draw[maplegreen] (0.422,0.422) -- (0.403,0.403);
  \draw[maplegrey] (0.422,0.422) -- (0.434,0.418);
  \draw[maplegreen] (0.352,0.352) -- (0.378,0.378);
  \draw[maplegreen] (0.352,0.352) -- (0.331,0.331);
  \draw[maplegrey] (0.352,0.352) -- (0.365,0.350);
  \draw[maplegreen] (0.378,0.378) -- (0.403,0.403);
  \draw[maplegrey] (0.378,0.378) -- (0.365,0.350);
  \draw[maplegrey] (0.378,0.378) -- (0.399,0.367);
  \draw[maplegrey] (0.403,0.403) -- (0.434,0.418);
  \draw[maplegrey] (0.403,0.403) -- (0.399,0.367);
  \draw[maplegrey] (0.403,0.403) -- (0.432,0.383);
  \draw[maplegreen] (0.275,0.275) -- (0.304,0.304);
  \draw[maplegreen] (0.275,0.275) -- (0.253,0.253);
  \draw[maplegrey] (0.275,0.275) -- (0.291,0.273);
  \draw[maplegreen] (0.304,0.304) -- (0.331,0.331);
  \draw[maplegrey] (0.304,0.304) -- (0.291,0.273);
  \draw[maplegrey] (0.304,0.304) -- (0.327,0.293);
  \draw[maplegrey] (0.331,0.331) -- (0.365,0.350);
  \draw[maplegrey] (0.331,0.331) -- (0.327,0.293);
  \draw[maplegrey] (0.331,0.331) -- (0.363,0.311);
  \draw[maplegreen] (0.191,0.191) -- (0.223,0.223);
  \draw[maplegreen] (0.191,0.191) -- (0.166,0.166);
  \draw[maplegrey] (0.191,0.191) -- (0.211,0.186);
  \draw[maplegreen] (0.223,0.223) -- (0.253,0.253);
  \draw[maplegrey] (0.223,0.223) -- (0.211,0.186);
  \draw[maplegrey] (0.223,0.223) -- (0.251,0.209);
  \draw[maplegrey] (0.253,0.253) -- (0.291,0.273);
  \draw[maplegrey] (0.253,0.253) -- (0.251,0.209);
  \draw[maplegrey] (0.253,0.253) -- (0.289,0.231);
  \draw[maplegreen] (0.020,0.020) -- (0.059,0.059);
  \draw[maplegreen] (0.020,0.020) -- (0.032,0.032);
  \draw[maplegreen] (0.059,0.059) -- (0.097,0.097);
  \draw[maplegreen] (0.059,0.059) -- (0.032,0.032);
  \draw[maplegrey] (0.059,0.059) -- (0.080,0.060);
  \draw[maplegreen] (0.097,0.097) -- (0.132,0.132);
  \draw[maplegrey] (0.097,0.097) -- (0.080,0.060);
  \draw[maplegrey] (0.097,0.097) -- (0.126,0.088);
  \draw[maplegreen] (0.132,0.132) -- (0.166,0.166);
  \draw[maplegrey] (0.132,0.132) -- (0.126,0.088);
  \draw[maplegrey] (0.132,0.132) -- (0.169,0.114);
  \draw[maplegrey] (0.166,0.166) -- (0.211,0.186);
  \draw[maplegrey] (0.166,0.166) -- (0.169,0.114);
  \draw[maplegrey] (0.166,0.166) -- (0.211,0.139);
  \draw[maplegrey] (0.032,0.032) -- (0.034,0.000);
  \draw[maplegrey] (0.032,0.032) -- (0.082,0.000);
  \draw[maplegrey] (0.032,0.032) -- (0.080,0.060);
  \draw[maplemagenta] (0.545,0.316) -- (0.557,0.341);
  \draw[maplemagenta] (0.545,0.316) -- (0.533,0.287);
  \draw[maplegrey] (0.545,0.316) -- (0.521,0.329);
  \draw[maplegrey] (0.545,0.316) -- (0.550,0.346);
  \draw[maplegrey] (0.545,0.316) -- (0.523,0.294);
  \draw[maplegrey] (0.557,0.341) -- (0.550,0.346);
  \draw[maplemagenta] (0.524,0.260) -- (0.533,0.287);
  \draw[maplemagenta] (0.524,0.260) -- (0.514,0.228);
  \draw[maplegrey] (0.524,0.260) -- (0.493,0.275);
  \draw[maplegrey] (0.524,0.260) -- (0.523,0.294);
  \draw[maplegrey] (0.524,0.260) -- (0.496,0.239);
  \draw[maplegrey] (0.533,0.287) -- (0.523,0.294);
  \draw[maplemagenta] (0.506,0.200) -- (0.514,0.228);
  \draw[maplemagenta] (0.506,0.200) -- (0.498,0.166);
  \draw[maplegrey] (0.506,0.200) -- (0.466,0.217);
  \draw[maplegrey] (0.506,0.200) -- (0.496,0.239);
  \draw[maplegrey] (0.506,0.200) -- (0.470,0.179);
  \draw[maplegrey] (0.514,0.228) -- (0.496,0.239);
  \draw[maplemagenta] (0.498,0.166) -- (0.491,0.130);
  \draw[maplegrey] (0.498,0.166) -- (0.470,0.179);
  \draw[maplegrey] (0.498,0.166) -- (0.475,0.141);
  \draw[maplemagenta] (0.487,0.101) -- (0.491,0.130);
  \draw[maplemagenta] (0.487,0.101) -- (0.483,0.063);
  \draw[maplegrey] (0.487,0.101) -- (0.444,0.115);
  \draw[maplegrey] (0.487,0.101) -- (0.475,0.141);
  \draw[maplegrey] (0.487,0.101) -- (0.451,0.075);
  \draw[maplegrey] (0.491,0.130) -- (0.475,0.141);
  \draw[maplemagenta] (0.483,0.063) -- (0.481,0.023);
  \draw[maplegrey] (0.483,0.063) -- (0.451,0.075);
  \draw[maplegrey] (0.483,0.063) -- (0.460,0.034);
  \draw[maplegrey] (0.481,0.023) -- (0.470,0.000);
  \draw[maplegrey] (0.481,0.023) -- (0.460,0.034);
  \draw[mapleblue] (0.034,0.000) -- (0.082,0.000);
  \draw[mapleblue] (0.082,0.000) -- (0.131,0.000);
  \draw[maplegrey] (0.082,0.000) -- (0.080,0.060);
  \draw[maplegrey] (0.082,0.000) -- (0.128,0.035);
  \draw[mapleblue] (0.131,0.000) -- (0.175,0.000);
  \draw[maplegrey] (0.131,0.000) -- (0.128,0.035);
  \draw[mapleblue] (0.175,0.000) -- (0.222,0.000);
  \draw[maplegrey] (0.175,0.000) -- (0.128,0.035);
  \draw[maplegrey] (0.175,0.000) -- (0.171,0.063);
  \draw[maplegrey] (0.175,0.000) -- (0.217,0.042);
  \draw[mapleblue] (0.222,0.000) -- (0.269,0.000);
  \draw[maplegrey] (0.222,0.000) -- (0.217,0.042);
  \draw[maplegrey] (0.222,0.000) -- (0.262,0.022);
  \draw[mapleblue] (0.269,0.000) -- (0.307,0.000);
  \draw[maplegrey] (0.269,0.000) -- (0.262,0.022);
  \draw[mapleblue] (0.307,0.000) -- (0.351,0.000);
  \draw[maplegrey] (0.307,0.000) -- (0.262,0.022);
  \draw[maplegrey] (0.307,0.000) -- (0.300,0.051);
  \draw[maplegrey] (0.307,0.000) -- (0.343,0.034);
  \draw[mapleblue] (0.351,0.000) -- (0.395,0.000);
  \draw[maplegrey] (0.351,0.000) -- (0.343,0.034);
  \draw[maplegrey] (0.351,0.000) -- (0.386,0.018);
  \draw[mapleblue] (0.395,0.000) -- (0.428,0.000);
  \draw[maplegrey] (0.395,0.000) -- (0.386,0.018);
  \draw[mapleblue] (0.428,0.000) -- (0.470,0.000);
  \draw[maplegrey] (0.428,0.000) -- (0.386,0.018);
  \draw[maplegrey] (0.428,0.000) -- (0.419,0.047);
  \draw[maplegrey] (0.428,0.000) -- (0.460,0.034);
  \draw[maplegrey] (0.470,0.000) -- (0.460,0.034);
  \draw[maplegrey] (0.434,0.418) -- (0.432,0.383);
  \draw[maplegrey] (0.434,0.418) -- (0.463,0.398);
  \draw[maplegrey] (0.365,0.350) -- (0.399,0.367);
  \draw[maplegrey] (0.365,0.350) -- (0.363,0.311);
  \draw[maplegrey] (0.365,0.350) -- (0.397,0.329);
  \draw[maplegrey] (0.399,0.367) -- (0.432,0.383);
  \draw[maplegrey] (0.399,0.367) -- (0.397,0.329);
  \draw[maplegrey] (0.399,0.367) -- (0.430,0.347);
  \draw[maplegrey] (0.432,0.383) -- (0.463,0.398);
  \draw[maplegrey] (0.432,0.383) -- (0.430,0.347);
  \draw[maplegrey] (0.432,0.383) -- (0.461,0.363);
  \draw[maplegrey] (0.463,0.398) -- (0.461,0.363);
  \draw[maplegrey] (0.463,0.398) -- (0.492,0.380);
  \draw[maplegrey] (0.291,0.273) -- (0.327,0.293);
  \draw[maplegrey] (0.291,0.273) -- (0.289,0.231);
  \draw[maplegrey] (0.291,0.273) -- (0.326,0.251);
  \draw[maplegrey] (0.327,0.293) -- (0.363,0.311);
  \draw[maplegrey] (0.327,0.293) -- (0.326,0.251);
  \draw[maplegrey] (0.327,0.293) -- (0.362,0.272);
  \draw[maplegrey] (0.363,0.311) -- (0.397,0.329);
  \draw[maplegrey] (0.363,0.311) -- (0.362,0.272);
  \draw[maplegrey] (0.363,0.311) -- (0.396,0.291);
  \draw[maplegrey] (0.397,0.329) -- (0.430,0.347);
  \draw[maplegrey] (0.397,0.329) -- (0.396,0.291);
  \draw[maplegrey] (0.397,0.329) -- (0.429,0.310);
  \draw[maplegrey] (0.430,0.347) -- (0.461,0.363);
  \draw[maplegrey] (0.430,0.347) -- (0.429,0.310);
  \draw[maplegrey] (0.430,0.347) -- (0.461,0.328);
  \draw[maplegrey] (0.461,0.363) -- (0.492,0.380);
  \draw[maplegrey] (0.461,0.363) -- (0.461,0.328);
  \draw[maplegrey] (0.461,0.363) -- (0.491,0.345);
  \draw[maplegrey] (0.492,0.380) -- (0.491,0.345);
  \draw[maplegrey] (0.492,0.380) -- (0.521,0.362);
  \draw[maplegrey] (0.211,0.186) -- (0.251,0.209);
  \draw[maplegrey] (0.211,0.186) -- (0.211,0.139);
  \draw[maplegrey] (0.211,0.186) -- (0.251,0.163);
  \draw[maplegrey] (0.251,0.209) -- (0.289,0.231);
  \draw[maplegrey] (0.251,0.209) -- (0.251,0.163);
  \draw[maplegrey] (0.251,0.209) -- (0.290,0.187);
  \draw[maplegrey] (0.289,0.231) -- (0.326,0.251);
  \draw[maplegrey] (0.289,0.231) -- (0.290,0.187);
  \draw[maplegrey] (0.289,0.231) -- (0.327,0.209);
  \draw[maplegrey] (0.326,0.251) -- (0.362,0.272);
  \draw[maplegrey] (0.326,0.251) -- (0.327,0.209);
  \draw[maplegrey] (0.326,0.251) -- (0.362,0.231);
  \draw[maplegrey] (0.362,0.272) -- (0.396,0.291);
  \draw[maplegrey] (0.362,0.272) -- (0.362,0.231);
  \draw[maplegrey] (0.362,0.272) -- (0.397,0.252);
  \draw[maplegrey] (0.396,0.291) -- (0.429,0.310);
  \draw[maplegrey] (0.396,0.291) -- (0.397,0.252);
  \draw[maplegrey] (0.396,0.291) -- (0.429,0.272);
  \draw[maplegrey] (0.429,0.310) -- (0.461,0.328);
  \draw[maplegrey] (0.429,0.310) -- (0.429,0.272);
  \draw[maplegrey] (0.429,0.310) -- (0.461,0.292);
  \draw[maplegrey] (0.461,0.328) -- (0.491,0.345);
  \draw[maplegrey] (0.461,0.328) -- (0.461,0.292);
  \draw[maplegrey] (0.461,0.328) -- (0.492,0.310);
  \draw[maplegrey] (0.491,0.345) -- (0.521,0.362);
  \draw[maplegrey] (0.491,0.345) -- (0.492,0.310);
  \draw[maplegrey] (0.491,0.345) -- (0.521,0.329);
  \draw[maplegrey] (0.521,0.362) -- (0.521,0.329);
  \draw[maplegrey] (0.521,0.362) -- (0.550,0.346);
  \draw[maplegrey] (0.080,0.060) -- (0.126,0.088);
  \draw[maplegrey] (0.080,0.060) -- (0.128,0.035);
  \draw[maplegrey] (0.126,0.088) -- (0.169,0.114);
  \draw[maplegrey] (0.126,0.088) -- (0.128,0.035);
  \draw[maplegrey] (0.126,0.088) -- (0.171,0.063);
  \draw[maplegrey] (0.169,0.114) -- (0.211,0.139);
  \draw[maplegrey] (0.169,0.114) -- (0.171,0.063);
  \draw[maplegrey] (0.169,0.114) -- (0.213,0.091);
  \draw[maplegrey] (0.211,0.139) -- (0.251,0.163);
  \draw[maplegrey] (0.211,0.139) -- (0.213,0.091);
  \draw[maplegrey] (0.211,0.139) -- (0.253,0.117);
  \draw[maplegrey] (0.251,0.163) -- (0.290,0.187);
  \draw[maplegrey] (0.251,0.163) -- (0.253,0.117);
  \draw[maplegrey] (0.251,0.163) -- (0.292,0.142);
  \draw[maplegrey] (0.290,0.187) -- (0.327,0.209);
  \draw[maplegrey] (0.290,0.187) -- (0.292,0.142);
  \draw[maplegrey] (0.290,0.187) -- (0.328,0.167);
  \draw[maplegrey] (0.327,0.209) -- (0.362,0.231);
  \draw[maplegrey] (0.327,0.209) -- (0.328,0.167);
  \draw[maplegrey] (0.327,0.209) -- (0.364,0.190);
  \draw[maplegrey] (0.362,0.231) -- (0.397,0.252);
  \draw[maplegrey] (0.362,0.231) -- (0.364,0.190);
  \draw[maplegrey] (0.362,0.231) -- (0.398,0.212);
  \draw[maplegrey] (0.397,0.252) -- (0.429,0.272);
  \draw[maplegrey] (0.397,0.252) -- (0.398,0.212);
  \draw[maplegrey] (0.397,0.252) -- (0.431,0.234);
  \draw[maplegrey] (0.429,0.272) -- (0.461,0.292);
  \draw[maplegrey] (0.429,0.272) -- (0.431,0.234);
  \draw[maplegrey] (0.429,0.272) -- (0.463,0.255);
  \draw[maplegrey] (0.461,0.292) -- (0.492,0.310);
  \draw[maplegrey] (0.461,0.292) -- (0.463,0.255);
  \draw[maplegrey] (0.461,0.292) -- (0.493,0.275);
  \draw[maplegrey] (0.492,0.310) -- (0.521,0.329);
  \draw[maplegrey] (0.492,0.310) -- (0.493,0.275);
  \draw[maplegrey] (0.492,0.310) -- (0.523,0.294);
  \draw[maplegrey] (0.521,0.329) -- (0.550,0.346);
  \draw[maplegrey] (0.521,0.329) -- (0.523,0.294);
  \draw[maplegrey] (0.128,0.035) -- (0.171,0.063);
  \draw[maplegrey] (0.171,0.063) -- (0.213,0.091);
  \draw[maplegrey] (0.171,0.063) -- (0.217,0.042);
  \draw[maplegrey] (0.213,0.091) -- (0.253,0.117);
  \draw[maplegrey] (0.213,0.091) -- (0.217,0.042);
  \draw[maplegrey] (0.213,0.091) -- (0.257,0.070);
  \draw[maplegrey] (0.253,0.117) -- (0.292,0.142);
  \draw[maplegrey] (0.253,0.117) -- (0.257,0.070);
  \draw[maplegrey] (0.253,0.117) -- (0.295,0.097);
  \draw[maplegrey] (0.292,0.142) -- (0.328,0.167);
  \draw[maplegrey] (0.292,0.142) -- (0.295,0.097);
  \draw[maplegrey] (0.292,0.142) -- (0.332,0.123);
  \draw[maplegrey] (0.328,0.167) -- (0.364,0.190);
  \draw[maplegrey] (0.328,0.167) -- (0.332,0.123);
  \draw[maplegrey] (0.328,0.167) -- (0.367,0.148);
  \draw[maplegrey] (0.364,0.190) -- (0.398,0.212);
  \draw[maplegrey] (0.364,0.190) -- (0.367,0.148);
  \draw[maplegrey] (0.364,0.190) -- (0.401,0.172);
  \draw[maplegrey] (0.398,0.212) -- (0.431,0.234);
  \draw[maplegrey] (0.398,0.212) -- (0.401,0.172);
  \draw[maplegrey] (0.398,0.212) -- (0.434,0.195);
  \draw[maplegrey] (0.431,0.234) -- (0.463,0.255);
  \draw[maplegrey] (0.431,0.234) -- (0.434,0.195);
  \draw[maplegrey] (0.431,0.234) -- (0.466,0.217);
  \draw[maplegrey] (0.463,0.255) -- (0.493,0.275);
  \draw[maplegrey] (0.463,0.255) -- (0.466,0.217);
  \draw[maplegrey] (0.463,0.255) -- (0.496,0.239);
  \draw[maplegrey] (0.493,0.275) -- (0.523,0.294);
  \draw[maplegrey] (0.493,0.275) -- (0.496,0.239);
  \draw[maplegrey] (0.217,0.042) -- (0.257,0.070);
  \draw[maplegrey] (0.217,0.042) -- (0.262,0.022);
  \draw[maplegrey] (0.257,0.070) -- (0.295,0.097);
  \draw[maplegrey] (0.257,0.070) -- (0.262,0.022);
  \draw[maplegrey] (0.257,0.070) -- (0.300,0.051);
  \draw[maplegrey] (0.295,0.097) -- (0.332,0.123);
  \draw[maplegrey] (0.295,0.097) -- (0.300,0.051);
  \draw[maplegrey] (0.295,0.097) -- (0.337,0.079);
  \draw[maplegrey] (0.332,0.123) -- (0.367,0.148);
  \draw[maplegrey] (0.332,0.123) -- (0.337,0.079);
  \draw[maplegrey] (0.332,0.123) -- (0.372,0.105);
  \draw[maplegrey] (0.367,0.148) -- (0.401,0.172);
  \draw[maplegrey] (0.367,0.148) -- (0.372,0.105);
  \draw[maplegrey] (0.367,0.148) -- (0.406,0.131);
  \draw[maplegrey] (0.401,0.172) -- (0.434,0.195);
  \draw[maplegrey] (0.401,0.172) -- (0.406,0.131);
  \draw[maplegrey] (0.401,0.172) -- (0.438,0.156);
  \draw[maplegrey] (0.434,0.195) -- (0.466,0.217);
  \draw[maplegrey] (0.434,0.195) -- (0.438,0.156);
  \draw[maplegrey] (0.434,0.195) -- (0.470,0.179);
  \draw[maplegrey] (0.466,0.217) -- (0.496,0.239);
  \draw[maplegrey] (0.466,0.217) -- (0.470,0.179);
  \draw[maplegrey] (0.262,0.022) -- (0.300,0.051);
  \draw[maplegrey] (0.300,0.051) -- (0.337,0.079);
  \draw[maplegrey] (0.300,0.051) -- (0.343,0.034);
  \draw[maplegrey] (0.337,0.079) -- (0.372,0.105);
  \draw[maplegrey] (0.337,0.079) -- (0.343,0.034);
  \draw[maplegrey] (0.337,0.079) -- (0.378,0.062);
  \draw[maplegrey] (0.372,0.105) -- (0.406,0.131);
  \draw[maplegrey] (0.372,0.105) -- (0.378,0.062);
  \draw[maplegrey] (0.372,0.105) -- (0.411,0.089);
  \draw[maplegrey] (0.406,0.131) -- (0.438,0.156);
  \draw[maplegrey] (0.406,0.131) -- (0.411,0.089);
  \draw[maplegrey] (0.406,0.131) -- (0.444,0.115);
  \draw[maplegrey] (0.438,0.156) -- (0.470,0.179);
  \draw[maplegrey] (0.438,0.156) -- (0.444,0.115);
  \draw[maplegrey] (0.438,0.156) -- (0.475,0.141);
  \draw[maplegrey] (0.470,0.179) -- (0.475,0.141);
  \draw[maplegrey] (0.343,0.034) -- (0.378,0.062);
  \draw[maplegrey] (0.343,0.034) -- (0.386,0.018);
  \draw[maplegrey] (0.378,0.062) -- (0.411,0.089);
  \draw[maplegrey] (0.378,0.062) -- (0.386,0.018);
  \draw[maplegrey] (0.378,0.062) -- (0.419,0.047);
  \draw[maplegrey] (0.411,0.089) -- (0.444,0.115);
  \draw[maplegrey] (0.411,0.089) -- (0.419,0.047);
  \draw[maplegrey] (0.411,0.089) -- (0.451,0.075);
  \draw[maplegrey] (0.444,0.115) -- (0.475,0.141);
  \draw[maplegrey] (0.444,0.115) -- (0.451,0.075);
  \draw[maplegrey] (0.386,0.018) -- (0.419,0.047);
  \draw[maplegrey] (0.419,0.047) -- (0.451,0.075);
  \draw[maplegrey] (0.419,0.047) -- (0.460,0.034);
  \draw[maplegrey] (0.451,0.075) -- (0.460,0.034);
 \end{tikzpicture}
\end{center}

All the charts are based on polynomials of degree $16$.  We find that,
on the disc of radius $0.1$, the entries in $p^*(g_1)-g_{\hyp}$ are of
absolute value less than $10^{-6}$.

We next need to record the combinatorial structure of the above grid.
If there is an edge joining the centre of chart $i$ to the centre of
chart $j$, we need to invoke the method \mcode+EHA["add_edge",i,j]+.
(Charts are numbered from $0$, in the order that they were added to
the atlas.)  This is also done by the function
\mcode+build_data["EH_atlas",2]()+.

Next, for each chart $p_i$, we need to find the point
$\bt_i=q^{-1}p_i(0)\in\Dl$.  For charts that are centred on one of the
curves $C_k$, this is given by Proposition~\ref{prop-curve-chart}.
For the remaining charts, it is useful to start with a crude
approximation, obtained by applying an essentially arbitrary
diffeomorphism between the fundamental domains for $EX^*$ and
$HX(b)$.  This is done using the method
\mcode+EHA["set_beta_approx"]+.  We then invoke the method
\mcode+EHA["set_edge_lengths"]+.  This calculates various quantities
for each edge $(i,j)$.  In particular, it calculates the average of
$d_{\hyp}(0,q_i^{-1}(q_j(0)))$ and $d_{\hyp}(0,q_j^{-1}(q_i(0)))$,
which is an estimate of the hyperbolic distance in $EX^*$ between the
centres $q_i(0)$ and $q_j(0)$.  Each edge is actually represented by
an object \mcode+E+ of class \mcode+EH_atlas_edge+, and this distance
is stored as \mcode+E["EH_length"]+.  On the other hand,
\mcode+E["H_length"]+ is set equal to $d_{\hyp}(\bt_i,\bt_j)$, using the
approximate values of $\bt_i$ and $\bt_j$, which may be quite
inaccurate at this stage.  One can then invoke the method
\mcode+EHA["optimize_beta"]+ to adjust the values of $\bt_i$ so as
optimize the match between the edge lengths measured in $\Dl$ and in
$EX^*$.  The same method also calculates appropriate values for the
parameters $\lm_i$, and thus also for the M\"obius maps $m_i$.

Now all the maps $p_im_i^{-1}$ are approximations to $q$, and it is
useful to test how well they agree with each other.  The method
\mcode+EHA["make_H_samples",N]+ sets \mcode+EHA["H_samples"]+ to be
the list of all numbers $z=(s+it)/N$ (with $s,t\in\Z$) that lie in
$HF_{16}(b)$.  The method \mcode+EHA["max_patching_error",r]+ then
does the following.  For each point $z_i$ in \mcode+EHA["H_samples"]+,
it looks for charts $p_j$ where $|z_i-\bt_j|<r$.  Let $k_i$ be the
number of such charts.  For each such chart, the method
calculates $x_{ij}=p_jm_j^{-1}(z_i)\in EX^*$.  These points should
all be the same, so we let $d_i$ denote the maximum euclidean
distance between any two of them.  The return value of the method is
a triple $(z_i,m_i,d_i)$, where $d_i$ is maximal.  If we take
$r=0.12$, we find that $m_i\geq 3$ and $d_i<10^{-10.4}$ for all $i$.
Thus, for an arbitrary point $z\in HF_{16}(b)$, it is safe to
calculate $q(z)$ as $p_jm_j^{-1}(z)$, where $j$ is chosen to minimize
$d_{\hyp}(z,\bt_j)$.  We can then extend this over all of $\Dl$ by
using the group action, as discussed earlier.  This is implemented by
the methods \mcode+EHA["q",[x,y]]+ or \mcode-EHA["q_c",x+I*y]-.

We now want to find a function given by a single formula which is a
good approximation to $q$ on a reasonably large part of $\Dl$, such as
the disc of radius $0.9$ centred at the origin.  An obvious approach
would be to approximate $q(x+iy)_k$ (for $1\leq k\leq 4$) by a
polynomial or rational function in $x$ and $y$.  This is implemented
by the methods \mcode+set_q_approx_poly+ and \mcode+set_q_approx_pade+
of the class \mcode+EH_atlas+.  However, results from this
approach are poor.  The approximating polynomials have extremely large
coefficients (of different signs), even though $|q(x+iy)_k|\leq 1$,
and the errors are fairly large even if we use polynomials or rational
functions of high degree.  It is better to consider the Fourier series
on circles of fixed radius.  To understand how this works, we first
recall that $q$ is equivariant with respect to $\ip{\lm,\nu}$, which
gives
\begin{align*}
 q_1(re^{i(\tht+\pi/2)}) &=    -q_2(re^{i\tht}) &
 q_1(re^{-i\tht})        &= \pp q_1(re^{i\tht}) \\
 q_2(re^{i(\tht+\pi/2)}) &= \pp q_1(re^{i\tht}) &
 q_2(re^{-i\tht})        &=    -q_2(re^{i\tht}) \\
 q_3(re^{i(\tht+\pi/2)}) &= \pp q_3(re^{i\tht}) &
 q_3(re^{-i\tht})        &= \pp q_3(re^{i\tht}) \\
 q_4(re^{i(\tht+\pi/2)}) &=    -q_4(re^{i\tht}) &
 q_4(re^{-i\tht})        &= \pp q_4(re^{i\tht}).
\end{align*}
From this it follows that there are functions $a_{k,m}(r)$ such that
\begin{align*}
 q_1(r e^{i\tht}) &= \sum_m a_{1,m}(r) \cos((2m+1)\tht) \\
 q_2(r e^{i\tht}) &= \sum_m (-1)^m a_{1,m}(r) \sin((2m+1)\tht) \\
 q_3(r e^{i\tht}) &= \sum_m a_{3,m}(r) \cos(4m\tht) \\
 q_4(r e^{i\tht}) &= \sum_m a_{4,m}(r) \cos((4m+2)\tht).
\end{align*}
We can find approximations to the coefficients $a_{j,l}(r)$ by fixing
$k\geq 0$, then calculating $q(re^{2\pi ij/2^k})$ for $0\leq j<2^k$, then
taking a discrete Fourier transform.  This algorithm is implemented by the
method \mcode+EHA["set_q_approx_fourier",r_max,m,k]+, which calculates
Fourier coefficients for \mcode+m+ different radii, the largest being
\mcode+r_max+.

It seems experimentally that for $r\leq 0.9$ we have $|a_{j,m}(r)|\leq
2$ for all $j$ and $m$, and that $|a_{j,m}(r)|$ decreases quite
rapidly with $m$.  Thus, for a fixed value of $r$, the above
representation is quite satisfactory.  Now fix $k$ and $m$, and
consider $a_{k,m}(r)$ as a function of $r$.  The following picture
shows a typical sample of these functions, for $0\leq r\leq 0.9$.
\begin{center}
 \begin{tikzpicture}[scale=8]
  \draw[->] (-0.05, 0.00) -- (0.95,0.00);
  \draw[->] ( 0.00,-0.30) -- (0.00,0.30);
 \draw[smooth,red]  (0.028,0.000) -- (0.057,0.000) -- (0.085,0.000) -- (0.113,0.000) -- (0.141,0.000) -- (0.169,0.000) -- (0.196,0.000) -- (0.224,0.000) -- (0.251,0.000) -- (0.278,0.000) -- (0.305,0.000) -- (0.331,0.000) -- (0.357,0.000) -- (0.383,-0.000) -- (0.409,-0.000) -- (0.434,-0.000) -- (0.458,0.001) -- (0.482,0.003) -- (0.506,0.007) -- (0.529,0.012) -- (0.552,0.021) -- (0.574,0.032) -- (0.595,0.046) -- (0.616,0.063) -- (0.636,0.081) -- (0.656,0.102) -- (0.675,0.124) -- (0.693,0.145) -- (0.711,0.166) -- (0.728,0.184) -- (0.744,0.199) -- (0.760,0.211) -- (0.775,0.217) -- (0.789,0.218) -- (0.802,0.214) -- (0.814,0.205) -- (0.826,0.192) -- (0.837,0.176) -- (0.847,0.159) -- (0.856,0.141) -- (0.864,0.123) -- (0.872,0.107) -- (0.878,0.093) -- (0.884,0.080) -- (0.889,0.069) -- (0.893,0.061) -- (0.896,0.054) -- (0.898,0.049) -- (0.900,0.046) -- (0.900,0.046) ;
 \draw[smooth,red]  (0.028,-0.000) -- (0.057,-0.000) -- (0.085,-0.000) -- (0.113,-0.000) -- (0.141,-0.000) -- (0.169,-0.000) -- (0.196,-0.000) -- (0.224,-0.000) -- (0.251,-0.000) -- (0.278,-0.000) -- (0.305,-0.000) -- (0.331,-0.000) -- (0.357,-0.000) -- (0.383,-0.000) -- (0.409,-0.000) -- (0.434,-0.001) -- (0.458,-0.001) -- (0.482,-0.001) -- (0.506,-0.001) -- (0.529,-0.001) -- (0.552,0.000) -- (0.574,0.002) -- (0.595,0.005) -- (0.616,0.009) -- (0.636,0.015) -- (0.656,0.024) -- (0.675,0.034) -- (0.693,0.048) -- (0.711,0.064) -- (0.728,0.083) -- (0.744,0.103) -- (0.760,0.125) -- (0.775,0.146) -- (0.789,0.166) -- (0.802,0.184) -- (0.814,0.198) -- (0.826,0.208) -- (0.837,0.214) -- (0.847,0.217) -- (0.856,0.218) -- (0.864,0.216) -- (0.872,0.214) -- (0.878,0.211) -- (0.884,0.208) -- (0.889,0.205) -- (0.893,0.203) -- (0.896,0.201) -- (0.898,0.200) -- (0.900,0.199) -- (0.900,0.198) ;
 \draw[smooth,red]  (0.028,0.000) -- (0.057,0.000) -- (0.085,-0.000) -- (0.113,-0.000) -- (0.141,0.000) -- (0.169,0.000) -- (0.196,0.000) -- (0.224,0.000) -- (0.251,0.000) -- (0.278,-0.000) -- (0.305,-0.000) -- (0.331,-0.000) -- (0.357,-0.000) -- (0.383,-0.000) -- (0.409,-0.000) -- (0.434,-0.000) -- (0.458,-0.000) -- (0.482,-0.000) -- (0.506,-0.001) -- (0.529,-0.001) -- (0.552,-0.002) -- (0.574,-0.003) -- (0.595,-0.004) -- (0.616,-0.006) -- (0.636,-0.009) -- (0.656,-0.012) -- (0.675,-0.017) -- (0.693,-0.022) -- (0.711,-0.028) -- (0.728,-0.036) -- (0.744,-0.044) -- (0.760,-0.055) -- (0.775,-0.068) -- (0.789,-0.082) -- (0.802,-0.099) -- (0.814,-0.117) -- (0.826,-0.136) -- (0.837,-0.156) -- (0.847,-0.174) -- (0.856,-0.191) -- (0.864,-0.207) -- (0.872,-0.220) -- (0.878,-0.232) -- (0.884,-0.242) -- (0.889,-0.250) -- (0.893,-0.256) -- (0.896,-0.261) -- (0.898,-0.265) -- (0.900,-0.267) -- (0.900,-0.267) ;
 \draw[smooth,red]  (0.028,-0.000) -- (0.057,-0.000) -- (0.085,0.000) -- (0.113,-0.000) -- (0.141,-0.000) -- (0.169,0.000) -- (0.196,-0.000) -- (0.224,-0.000) -- (0.251,-0.000) -- (0.278,-0.000) -- (0.305,-0.000) -- (0.331,-0.000) -- (0.357,-0.000) -- (0.383,0.000) -- (0.409,0.000) -- (0.434,0.000) -- (0.458,0.000) -- (0.482,0.000) -- (0.506,0.000) -- (0.529,0.000) -- (0.552,0.000) -- (0.574,0.000) -- (0.595,0.000) -- (0.616,0.001) -- (0.636,0.001) -- (0.656,0.002) -- (0.675,0.002) -- (0.693,0.003) -- (0.711,0.003) -- (0.728,0.003) -- (0.744,0.002) -- (0.760,0.000) -- (0.775,-0.003) -- (0.789,-0.009) -- (0.802,-0.016) -- (0.814,-0.023) -- (0.826,-0.032) -- (0.837,-0.041) -- (0.847,-0.049) -- (0.856,-0.057) -- (0.864,-0.064) -- (0.872,-0.070) -- (0.878,-0.075) -- (0.884,-0.081) -- (0.889,-0.085) -- (0.893,-0.089) -- (0.896,-0.093) -- (0.898,-0.095) -- (0.900,-0.097) -- (0.900,-0.098) ;
 \draw ( 0.900, 0.046) node[anchor=west] {$a_{1,3}$};
 \draw ( 0.900, 0.198) node[anchor=west] {$a_{1,5}$};
 \draw ( 0.900,-0.267) node[anchor=west] {$a_{1,8}$};
 \draw ( 0.900,-0.098) node[anchor=west] {$a_{1,10}$};
 \draw (-0.020, 0.200) -- ( 0.000, 0.200);
 \draw (-0.020, 0.200) node[anchor=east] {$0.2$};
 \draw ( 0.900, 0.000) -- ( 0.900,-0.020);
 \end{tikzpicture}
\end{center}
These functions are extremely flat for small values of $r$, and grow
to modest size as $r$ approaches $0.9$.  They cannot easily be
approximated by polynomials or rational functions in $r$.  We have
tried various transformations, such as using the variable
$s=1/\log(1/r)$ instead of $r$, but none of these have yielded
compelling results.  It would be of interest to have a better
theoretical understanding of the asymptotics of the functions
$a_{k,m}(r)$, but so far we have not achieved that.  However,
approximation by cubic splines (which are also calculated by the
\mcode+set_q_approx_fourier+ method) is quite effective.  An
approximation to $q(x+iy)$ using these splines can be calculated by
invoking \mcode+EHA["q_fourier",[x,y]]+ or
\mcode-EHA["q_fourier_c",x+I*y]-.

The following pictures show the images of $q(0.8e^{it})$ under the
linear projections $\pi,\dl,\zt\:EX^*\to\R^2$ that were discussed in
Section~\ref{sec-disc}:
\[ \includegraphics[scale=0.12]{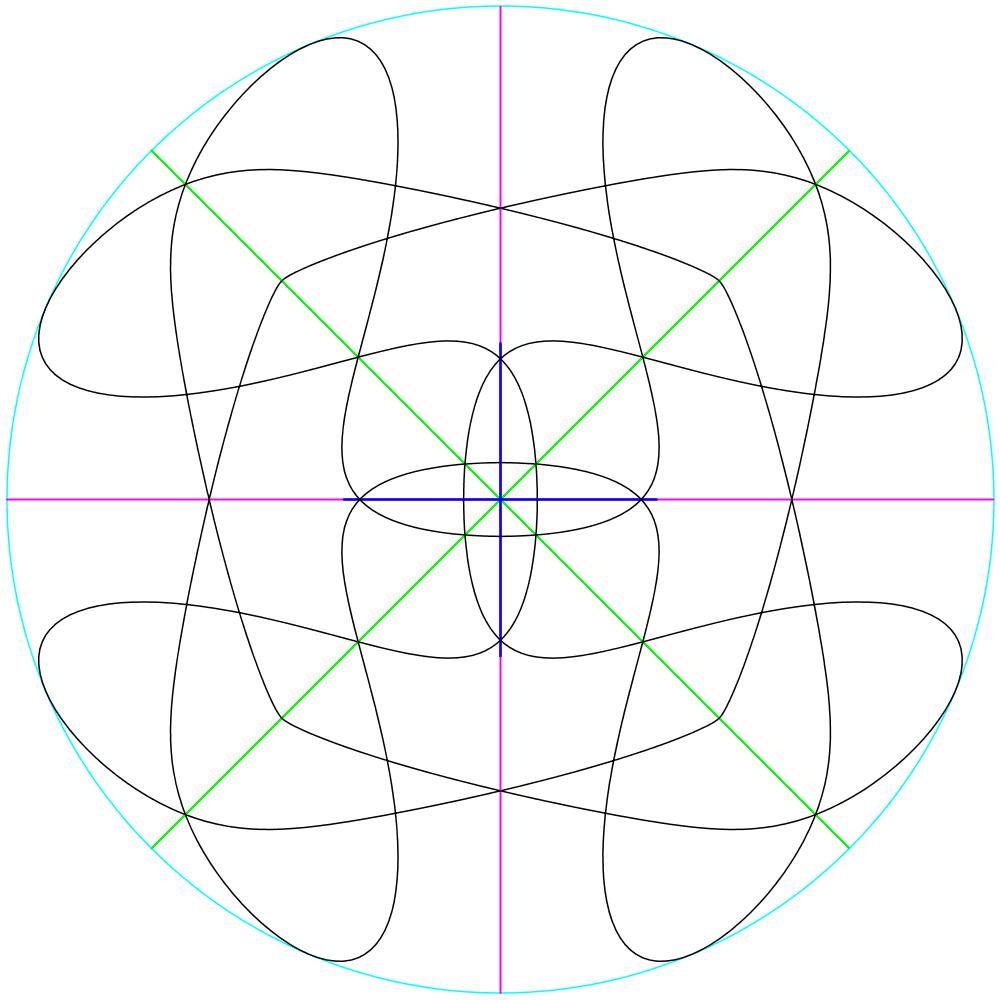} \hspace{4em}
   \includegraphics[scale=0.12]{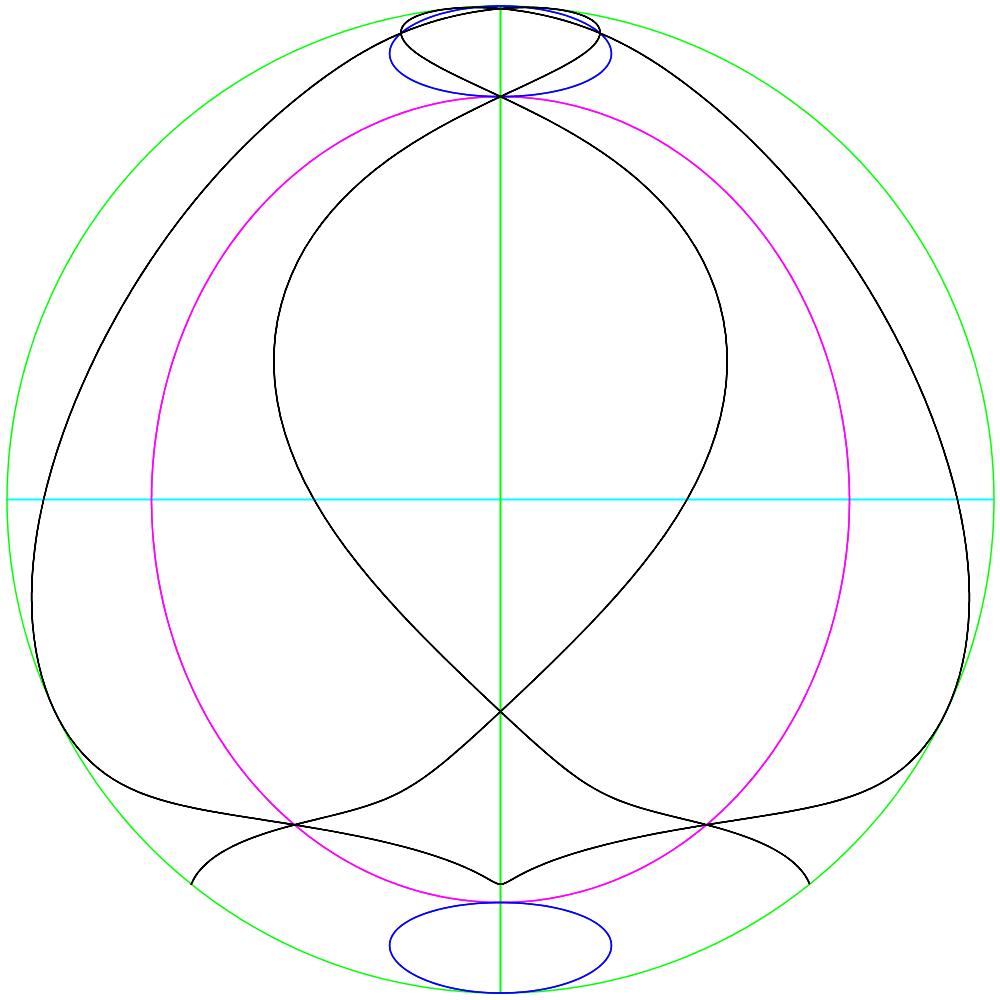}    \hspace{4em}
   \includegraphics[scale=0.12]{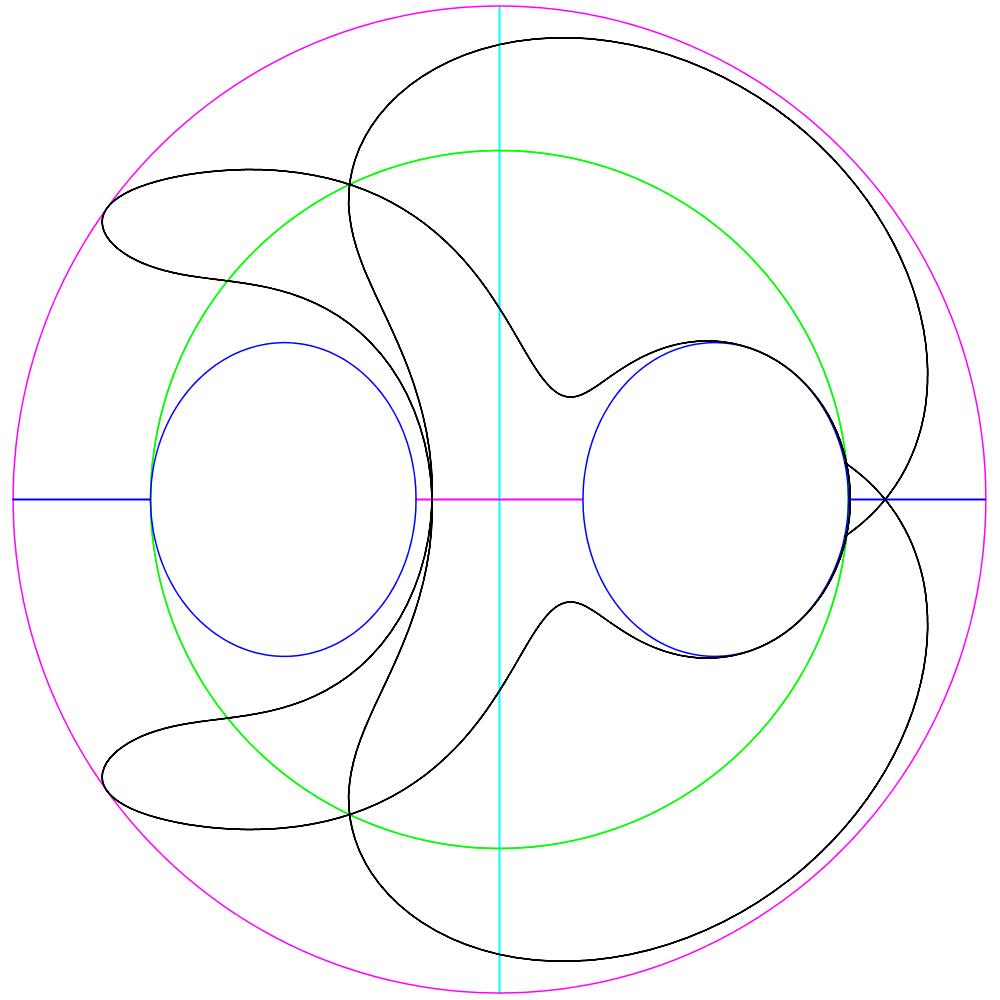}
\]
The following picture shows the image under the map $p_4\:x\mapsto y$
from Proposition~\ref{prop-F-four}:
\[ \includegraphics[scale=0.2]{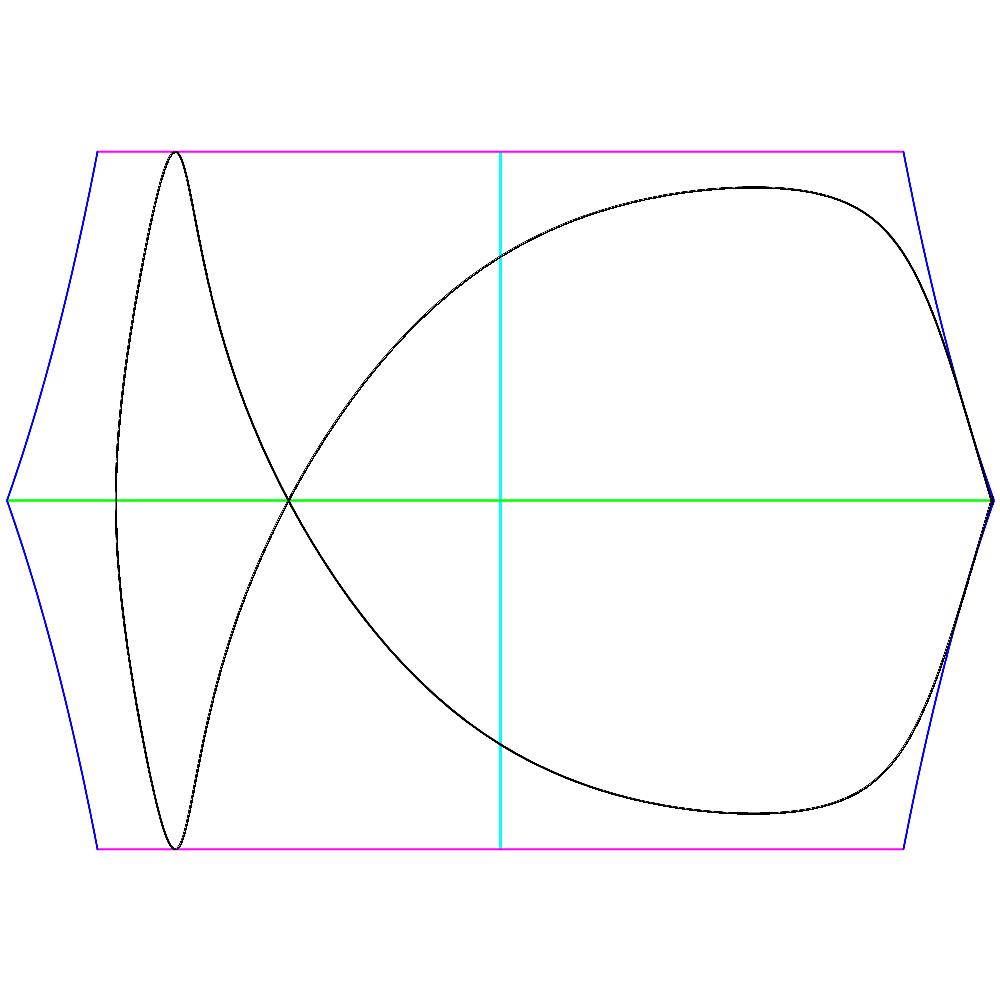} \]

Finally, we discuss the canonical conformal map $\hp\:EX^*\to S^2$.
The transformation properties of $\hp$ were discussed in
Remark~\ref{rem-p-hat}.  By comparing these with the transformation
properties of the basis in Proposition~\ref{prop-OX-basis}, we see
that $\hp(x)_i$ must be given by some $G$-invariant function
$u_i(z_1,z_2)$ multiplied by $p^*(x)_i$, where 
\[ p^*(x) = \left(\rt \,y_2,\; 2x_1x_2,\; -x_3) \right). \]
Remark~\ref{rem-p-hat} also records the values of $\hp(v_i)$ for
$0\leq i\leq 9$; these are equivalent to the conditions
\[ u_1(0,1/2) = u_2(0,0) = u_3(1,0) = 1. \] 
It turns out that the functions $u_i$ can be approximated effectively
by rational functions, using the method that we will now explain.  We
have already seen how to produce a list of points $a^\Dl_i$ in
$HF_{16}(b)$ and calculate the images $a^E_i=q(a^\Dl_i)$ under the map
$q\:\Dl\to EX^*$.  By applying the map $x\mapsto z$ to these, we
obtain points $a^Z_i\in F^*_{16}\subset\R^2$.  We have also seen how
to calculate the isomorphism $HX(b)\to PX(a)$.  By combining this with
the projection $PX(a)\to\C_\infty$ and the stereographic projection
map $\C_\infty\to S^2$ we obtain points $a_i^S\in S^2$.  We must have
$\hp(a^E_i)=a^S_i$, and using this, we can find the values
$a^U_{ij}=u_j(a^Z_i)=\hp(a^E_i)_j/p^*(a^E_i)_j$ for $1\leq j\leq 3$.
Our problem is thus to find rational functions $u_j^R$ such that
$u_j^R(a^Z_i)$ is close to $a^U_{ij}$.  In principle this should hold
for all $i$, but we have found it best to discard the cases where
$|p^*(a^E_i)_j|<10^{-3}$, in order to avoid numerical instability. 

We have found the following general approach to be effective.
\begin{method}
 Suppose we have a finite set $S=\{s_1,\dotsc,s_n\}$, and a function
 $f\:S\to\R$.  We want to find an approximation $f\simeq g_1/g_2$,
 where $g_1$ and $g_2$ lie in some vector space $V\leq\Map(S,\R)$.  If
 $n$ is large then it is not very tractable to minimize
 $\|f-g_1/g_2\|^2$, because this is nonlinear in the coefficients of
 $g_2$.  However, it will often be adequate to minimize $\|fg_2-g_1\|^2$
 subject to a positive definite quadratic constraint on the size of
 $g_2$.  To do this, let $\{v_1,\dotsc,v_m\}$ be a basis for $V$.  Let
 $M_1$ be the matrix of values $v_j(s_i)$, and let $M_2$ be the matrix
 of values $f(s_i)v_j(s_i)$.  We find the QR decompositions
 $M_i=Q_iR_i$, and put $N=Q_1^TQ_2$.  We let $a_2$ denote a singular
 vector for $N$ of maximal singular value with $\|a_2\|=1$, and put
 $a_1=Na_2$.  We then put $b_i=R_i^{-1}a_i$ and
 $g_i=\sum_jb_{ij}v_j$.  The function $g_1/g_2$ is then the desired
 approximation.
\end{method}

The above algorithm is implemented by the \mcode+find_p+ method of the
class \mcode+E_to_S_map+, which is declared in
\fname+embedded/roothalf/E_to_S.mpl+.  This method must be passed an
object of class \mcode+EH_atlas+, which encodes information about the
points $a^\Dl_i$ and the maps $\Dl\to EX^*$ and $\Dl\to\C_\infty$.  We
have used this method to find rational approximations $u_j^R$ where the
numerator and denominator have total degree eight in $z_1$ and $z_2$.
It turns out that these functions are quite tame, as shown in the
graphs below.
\[ \includegraphics[scale=0.16]{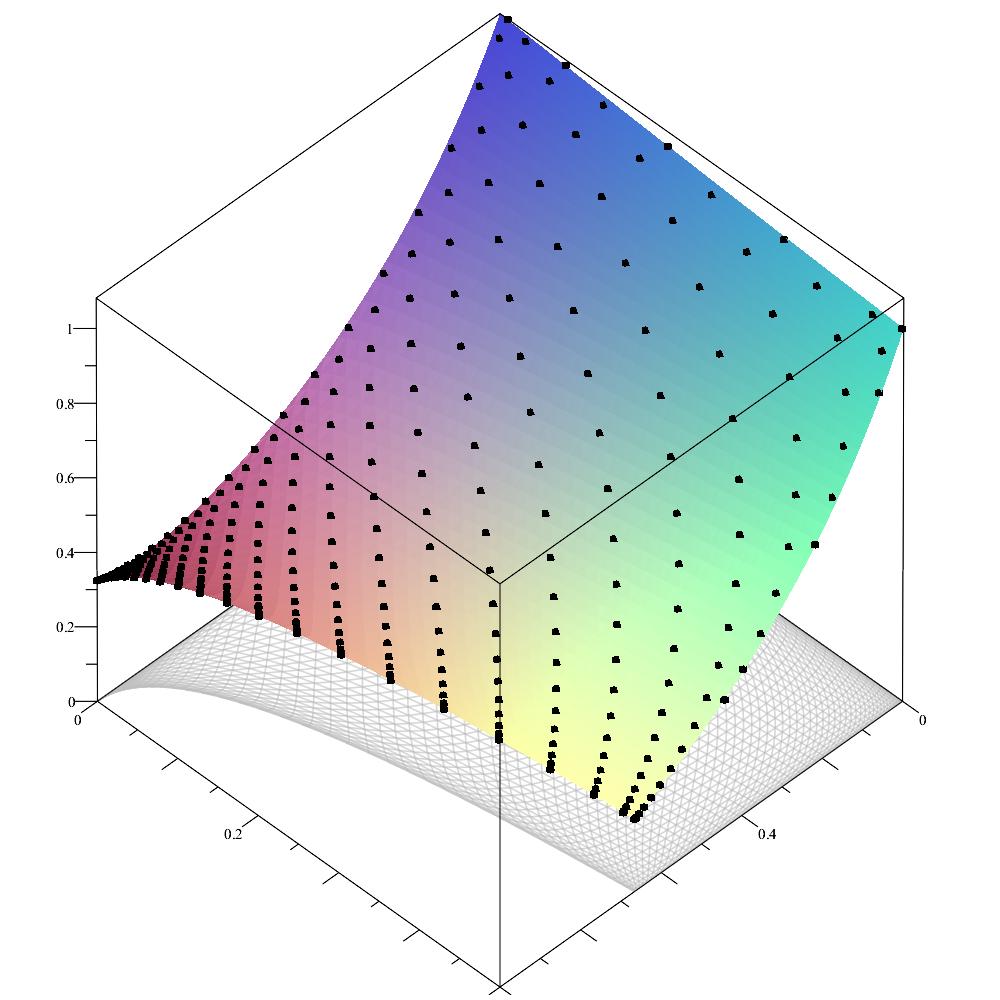}
   \includegraphics[scale=0.16]{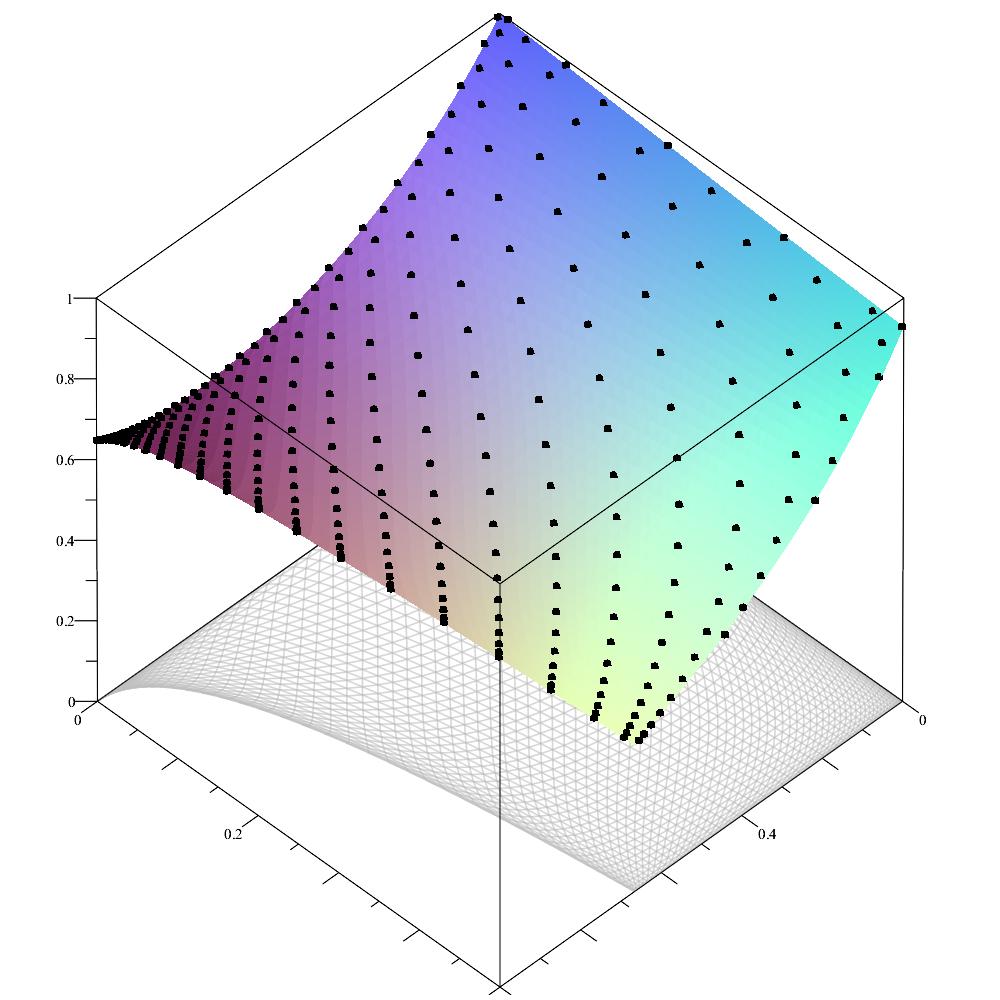}
   \includegraphics[scale=0.16]{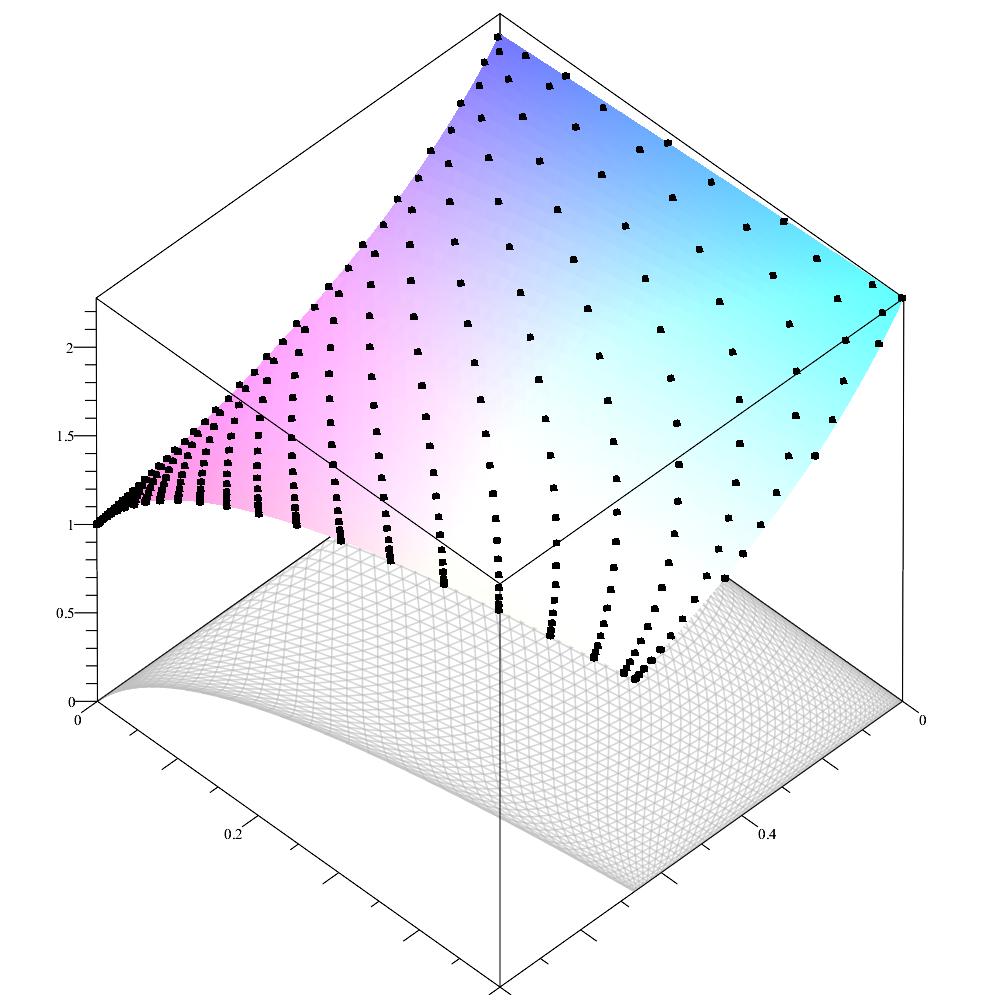} 
\]
All values lie between about $0.3$ and $2.3$.  All coefficients in the
numerators and denominators have absolute value at most one, and they
appear to decrease quite rapidly with the total degree of the
corresponding monomials.  The full calculation can be carried out using
the function \mcode+build_data["E_to_S_map"]()+, which is defined in
\fname+build_data.mpl+. 

\subsection{Energy minimisation}
\lbl{sec-energy}

As explained in Remark~\ref{rem-p-hat}, there is a canonical conformal
map $\hp\:EX^*\to S^2$.  If we can find $\hp$, then all other
information can easily be derived from that.  One approach is to start
with the map $\hp\:EX^*\to S^2$ from
Definition~\ref{defn-sphere-quotient-a}.  This has the right
equivariance properties and the right homotopy class, but it is not
conformal.  We can hope to adjust it by a numerical minimisation
algorithm to make it conformal.  For this, we need to recall the
theory of Dirichlet energy.

\begin{definition}
 Given matrices $P,Q\in M_2(\R)$ we put
 \[ \ip{P,Q} = \sum_{i,j=1}^2 P_{ij}Q_{ij} = \text{trace}(P^TQ). \]
 This is an inner product, with associated norm
 $\|P\|^2=\sum_{i,j}P_{ij}^2$.  We also put
 \begin{align*}
  C_+(P) &= (P_{11}-P_{22})^2 + (P_{12}+P_{21})^2 \\
  C_-(P) &= (P_{11}+P_{22})^2 + (P_{12}-P_{21})^2.
 \end{align*}
\end{definition}
\begin{remark}
 It is clear that $C_+(P)=0$ iff $P=\bbm a&b\\ -b&a\ebm$ for some
 $a,b\in\R$, or in other words $P$ is a conformal matrix.  Similarly,
 we have $C_-(P)=0$ iff $P$ is anticonformal.
\end{remark}
\begin{lemma}
 For all $P\in M_2(\R)$ we have
 \begin{align*}
  \|P\|^2 &= C_+(P) + 2\det(P) \geq 2\det(P) \\
  \|P\|^2 &= C_-(P) - 2\det(P) \geq -2\det(P),
 \end{align*}
 so $\|P\|^2\geq 2|\det(P)|$.
\end{lemma}
\begin{proof}
 The equalities are direct calculations, and it is clear that
 $C_+(P)\geq 0$ and $C_-(P)\geq 0$.
\end{proof}

\begin{corollary}\lbl{cor-invariants}
 For $A,B\in SO(2)$ we have $\|APB\|^2=\|P\|^2$ and
 $C_{\pm}(APB)=C_{\pm}(P)$ and $\det(APB)=\det(P)$.
\end{corollary}
\begin{proof}
 We have $\det(A)=\det(B)=1$ so
 $\det(APB)=\det(A)\det(P)\det(B)=\det(P)$.  We also have
 $A^TA=B^TB=1$ and $\text{trace}(XY)=\text{trace}(YX)$ so
 \[ \|APB\|^2 = \text{trace}(B^TP^TA^TAPB)
     = \text{trace}(B^TP^TPB)
     = \text{trace}(BB^TP^TP)
     = \text{trace}(P^TP) = \|P\|^2.
 \]
 We can now use $C_{\pm}(P)=\|P\|^2\mp 2\det(P)$ to deduce that
 $C_{\pm}(APB)=C_{\pm}(P)$.
\end{proof}

\begin{definition}
 Let $V$ and $W$ be oriented two-dimensional inner product spaces over
 $\R$, and let $\phi\:V\to W$ be a linear map.  We then choose
 oriented orthonormal bases $v_1,v_2$ for $V$ and $w_1,w_2$ for $W$,
 and let $P$ be the matrix such that
 \begin{align*}
  \phi(v_1) &= P_{11}w_1 + P_{21}w_2 \\
  \phi(v_2) &= P_{12}w_1 + P_{22}w_2.
 \end{align*}
 We then put $\|\phi\|^2=\|P\|^2$ and $\det(\phi)=\det(P)$ and
 $C_{\pm}(\phi)=C_{\pm}(P)$.  This is independent of the choice of
 bases, by Corollary~\ref{cor-invariants}.
\end{definition}

\begin{definition}\lbl{defn-energy}
 Let $X$ and $Y$ be connected oriented smooth closed surfaces with
 given Riemannian metrics.  We give them the measures derived from the
 metric in the usual way.  Consider a smooth map $f\:X\to Y$.  For
 each $x\in X$ we have a linear map $D_xf\:T_xX\to T_{f(x)}Y$ between
 oriented two-dimensional inner product spaces, so we can define
 $\|D_xf\|^2$ and $\det(D_xf)$ and $C_{\pm}(D_xf)$.  The
 \emph{Dirichlet energy} of $f$ is the integral over $X$ of the
 scalar-valued function $x\mapsto\half\|D_xf\|^2$.  We write this as
 $E(f)=\int_X\half\|Df\|^2$.  We also define the \emph{area} of $f$ to
 be $A(f)=\int_X\det(Df)$.
\end{definition}

\begin{remark}
 In terms of differential forms, we can let $\om_X$ and $\om_Y$ denote
 the volume forms for $X$ and $Y$, and then
 $f^*(\om_Y)=\det(Df)\om_X$.  We can regard $\om_X$ as a generator of
 the de Rham cohomology group $H^2(X)$, and similarly for $Y$.
 Integration gives an isomorphism from each of these cohomology groups
 to the reals.  From this point of view it is clear that $A(f)$
 depends only on the homotopy class of $f$.  Moreover, if $f$ is an
 orientation-preserving diffeomorphism then the standard
 change-of-variables formula shows that $A(f)$ is just the area of
 $Y$.
\end{remark}

\begin{proposition}\lbl{prop-dirichlet}
 For any $f\:X\to Y$ as above, we have $E(f)\geq A(f)$, with equality
 iff $f$ is conformal.
\end{proposition}
\begin{proof}
 This is clear from the identity $\|Df\|^2=2\det(Df)+C_+(Df)$.
\end{proof}

One way to exploit this is via a discretised version.  We triangulate
the fundamental domain $F_{16}\subset EX^*$, then use the group action
to obtain an equivariant triangulation of $EX^*$, with vertex set
$K_0$ say.  Taking the convex hulls of vertices of simplices gives a
piecewise linear surface $X'\subset\R^4$ that lies close to $EX^*$.
If we have a map $f\:K_0\to S^2$, then we can extend it linearly to
give a map $f'\:X'\to\R^3$, which will usually land in $\R^3\sm\{0\}$.
There is an obvious retraction of $\R^3\sm\{0\}$ onto $S^2$, and one
can also construct a map from $EX^*$ to $X'$ that is close to the
identity.  After composing with these we get a map $f''\:EX^*\to S^2$.
One could attempt to minimise $E(f'')$, but that is analytically
intractable.  However, a slight modification of
Definition~\ref{defn-energy} defines a quantity $E'(f')$ that is
analogous to $E(f)$, and we can attempt to minimise that instead.  We
find that the rate of convergence is slow, and the resulting
approximation is inaccurate close to the points $v_{10},\dotsc,v_{13}$
where the equivariance properties force the derivative of $\hp$ to be
zero.  One can improve the accuracy of the method by subdividing the
triangulation, but this makes everything much slower.  One could also
use an approximation scheme that is better than linear interpolation,
perhaps based on the various types of splines that are popular in
computer graphics.  However, we did not find an approach of this type
that worked well for our purposes.

We could also avoid discretisation, and instead attempt to minimise
the energy over some finite-dimensional space $M$ of maps
$EX^*\to S^2$.  This should ideally be chosen so that it is easy to
calculate $\|Df\|^2$ for $f\in M$, together with the derivatives of
$\|Df\|^2$ with respect to suitable coordinates on $M$.  We have not
found a space $M$ for which this works nicely.

\section{Overview of the Maple code}
\lbl{sec-maple}

\subsection{Directory structure}

The main directory for this project has subdirectories as follows:
\begin{itemize}
 \item \fname+latex:+ \LaTeX code for this document.  The subdirectory
  \fname+tikz_includes+ contains some files that were generated by
  Maple and are included in the main \LaTeX document by \mcode+\input+
  commands.  The code in the file \fname+maple/plots.mpl+ is relevant
  here.
 \item \fname+images:+ Image files (with extensions \mcode+.png+ or
  \mcode+.jpg+), all generated by Maple.  The code in the file
  \fname+maple/plots.mpl+ is relevant here.
 \item \fname+plots:+ Image files (with extensions \mcode+.m+) in
  Maple's internal format for plots.  The code in the file
  \fname+maple/plots.mpl+ is relevant here.
 \item \fname+maple:+ This contains Maple code in various
  subdirectories, which will be described in more detail below.  The
  files contain plain text, and have extension \fname+.mpl+.  Many
  files occur in pairs like \fname+projective/ellquot.mpl+ (which
  defines various functions related to elliptic curve quotients of
  $PX(a)$) an \fname+projective/ellquot_check.mpl+ (which defines
  procedures to check various assertions about those functions).
 \item \fname+doc:+ This contains various kinds of documentation of the 
  Maple code, in HTML format.  Some of the Maple code is object
  oriented (as will be discussed in Section~\ref{sec-oo-maple} below)
  and some is not.  For the object oriented code, there is
  automatically generated documentation of classes, fields and
  methods, similar to the standard javadoc framework for Java code.
  For the remaining code, there is an index of definitions of all
  defined symbols, with links to the defining files.
 \item \fname+worksheets:+ This contains Maple worksheets, with
  extension \fname+.mw+.  They all start with the following block:
  \begin{mcodeblock}
   restart;
   interface(quiet=true):
   olddir := currentdir("../maple"):
   read("genus2.mpl"):
   currentdir(olddir):
   interface(quiet=false):
  \end{mcodeblock}
  Executing this block will read in the file \fname+genus2.mpl+, which
  will in turn read in many other files from the \fname+maple+
  directory and its subdirectories.  We have mostly used worksheets
  for development, and have moved code to the \fname+.mpl+ files when
  it has become stable.
 \item \fname+data:+ This contains files generated by Maple recording
  the results of certain complex calculations.  There is a hierarchy
  of subdirectories parallel to those in the \fname+maple+ directory.
  Some are plain text files with extension \fname+.mpl+, but most are
  in Maple's internal format and have extension \fname+.m+.  See
  Section~\ref{sec-build} for information about how to recalculate
  these results.
\end{itemize}

The subdirectories of the \fname+maple+ directory are as follows:
\begin{itemize}
 \item The top level directory contains some code about the general
  theory of precromulent surfaces (not tied to any of the three
  families), and some general utility code.
 \item \fname+projective+: for code related to the projective family.
 \item \fname+hyperbolic+: for code related to the hyperbolic family,
  and code for isomorphisms between hyperbolic and projective surfaces.
 \item \fname+embedded+: for code related to the embedded family.
 \item \fname+embedded/roothalf+: for code related to the special case
  $EX^*=EX(1/\rt)$.
 \item \fname+domain:+ for object oriented code dealing with
  triangulations of cromulent surfaces; this can be specialised to
  each of the three families.  (At an earlier stage, we planned to do
  various substantial calculations using triangulations, but we
  eventually switched to different methods.)
 \item \fname+quadrature:+ object oriented code for quadrature rules
  on triangles.
\end{itemize}

\subsection{Checks}
\lbl{sec-checks}

Maple code that checks the correctness of various assertions is
contained in files whose names end with \fname+_check+.  Unlike the
other Maple files, these are not loaded automatically by the standard
block at the top of the worksheets.  One can read an individual file
by entering a command like
\begin{mcodeblock}
   read("../maple/projective/ellquot_check.mpl"):
\end{mcodeblock}
Alternatively, one can enter
\begin{mcodeblock}
   read("../maple/check_all.mpl"):
\end{mcodeblock}
to load and run all possible checks (which takes a long time).  The
global variable \mcode+checklist+ contains a list of all the checking
functions that have been loaded.  One can execute all of them by
invoking the \mcode+check_all()+ function.  These functions will
usually stop running if they encounter an assertion that fails, but
this can be prevented by setting the global variable
\mcode+assert_stop_on_fail+ to \mcode+false+.  Each checking function
will print its name when it starts to run.  Most functions will check
a large number of assertions; if the global variable
\mcode+assert_verbosely+ is set to \mcode+true+, then a brief identifier
will be printed for each assertion.  (This happens after Maple has
done the work of checking the assertion, but before it prints an error
message if the assertion has failed.) The general framework for all
this is set up by the files \fname+util.mpl+ and \fname+checks.mpl+ in
the top Maple directory.  The basic claim that the embedded,
projective and hyperbolic families are precromulent has some special
features; see the function \mcode+check_precromulent()+ and associated
comments in the file \fname+cromulent.mpl+.

One might ask about the reliability, rigour and completeness of these
checks.

First, we should explain that almost all checks cover assertions that
are claimed to be exact.  For the parts of the code that involve
numerical approximation, we have also performed many checks, but we
have not encapsulated them in a systematic framework.

\begin{itemize}
 \item[(a)] Some claims are of the form $u=0$, where $u$ is a constant
  expression.  Usually $u$ will be built from rational numbers by
  algebraic operations and by extraction of square roots of positive
  quantities.  In some cases we evaluate trigonometric functions and
  rational multiples of $\pi$, and there are a few examples involving
  roots of polynomials of degree greater than two.  In many cases, we
  can just use the command \mcode+simplify(u)+ and the result will be
  zero.  In some cases we need to use a more complicated command like
\begin{mcodeblock}
   simplify(factor(expand(rationalize(u))))
\end{mcodeblock}
  We have been willing to assume that if a procedure like this returns
  zero, then $u$ is genuinely equal to zero.  One can check this by
  numerical evaluation of $u$.  We have used 100 digit precision by
  default, and have not found any examples where the symbolic
  simplification functions seem to be incorrect.
 \item[(b)] Some other claims are of the form $u=0$, where $u$ is an
  expression involving several constants and variables, which may be
  subject to certain constraints.  The constants are of the type
  discussed in~(a).  Variables may be constrained to be real (using
  a command like \mcode+assume(t::real)+) or to lie in the unit
  interval (using a command like \mcode+assume(a_H>0 and a_H<1)+).
  Most expressions are built using algebraic operations, square roots
  and logarithms of quantities that can be shown to be positive,
  trigonometric functions and exponentials, and extraction of real
  parts of complex numbers.  There are also some derivatives and
  integrals.  There are obvious algorithms to deal with most of these
  things, and it seems unlikely that any bugs would have escaped
  detection.  However, there are some expressions where Maple
  implicitly uses a significant amount of logic to determine that
  various terms are positive, and uses this to justify manipulations
  with roots and logarithms.  We do not know what algorithms are used
  for this, but we have not detected any problems.  All the relevant
  symbolic simplifications can be tested by numerical evaluation,
  either by plotting or by setting parameters to randomly chosen values.
 \item[(c)] There is another kind of constraint that we did not
  mention under~(b): variables can be subject to certain polynomial
  relations, which can be encoded using a Gr\"obner basis for the
  corresponding ideal.  We only have examples where the polynomials
  have coefficients in $\Q(\rt)$, which is easy to handle.  Often,
  we only need to check expressions of the form $u=0$, where $u$ is a
  rational function of the constrained variables.  There are
  very standard algorithms for working with Gr\"obner bases, and it is
  highly unlikely that there could be any problems with expressions of
  this type.  It is also common to have expressions $u$ that involve
  square roots of polynomials that can be shown to be positive, and
  the roots are sometimes nested.  The algorithms for this case are
  not quite as standard, but again we have detected no problems.
 \item[(d)] A few expressions involve more sophisticated functions
  such as elliptic integrals and the Weierstrass $\wp$ function.  It
  is here that we encountered the only significant bug that we have
  seen: Maple's numerical evaluation of $\wp'(z)$ is incorrect for
  certain ranges of arguments.  On the other hand, when we first
  started working with $\wp'(z)$, it immediately became clear that
  something was causing inconsistent results, although it took time to
  locate the precise source of trouble.  This raised our confidence
  that other bugs would also quickly become visible.
 \item[(e)] There are a few cases where we have an indirect reason to
  know that an expression $u$ should be zero, but we have not been
  able to persuade Maple to simplify $u$ to zero.  In these cases, we
  have written checking functions that rely completely on numerical
  methods, by evaluating $u$ to 100 decimal places at various points
  in the parameter space.  Alternatively, if $u$ is a function on the
  surface $EX^*$, we can evaluate $u$ exactly at all the quasirational
  points of $EX^*$ (as discussed in Section~\ref{sec-rational}).
 \item[(f)] As well as the kinds of claims discussed in~(a) to~(e), we
  have various claims about more combinatorial structures, such as the
  groups $G$, $\Pi$ and $\tPi$.  For these we have mostly written our
  own code, both to implement the definitions and to check the claimed
  properties.  Thus, very little is hidden in the internals of Maple,
  and the sceptical reader can inspect all the relevant code.
\end{itemize}

\subsection{Object oriented Maple}
\lbl{sec-oo-maple}

For some of our work, it is natural to use an object oriented style of
programming.  For example, it is natural to have a class whose objects
represent conformal charts on $EX^*$, and another class for atlases,
and a class for quadrature rules, and so on.  We have described a
complex algorithm for calculating the canonical covering
$\Dl\to EX^*$, and it is natural to implement the steps in this
algorithm as methods of various classes.  Maple does not natively
support object oriented programming, but we have implemented our own
framework using Maple's system of tables with user-defined indexing
functions.  Our framework was in fact developed some years ago for a
rather different project, and adapted slightly for our current
purposes.  The relevant code is in the file \fname+class.mpl+ in the
top Maple directory.  Typical syntax is as follows:
\begin{itemize}
\item \mcode+Q := `new/E_quadrature_rule`();+ sets $Q$ to be a new
 quadrature rule on $F_{16}\subset EX^*$.
\item \mcode+Q["int_z",z[1]];+ given a quadrature rule $Q$,
 returns the estimated integral of $z_1$ over $F_{16}$.
\item \mcode+Q["curvature_error"];+ given a quadrature rule $Q$,
 returns the difference between the estimated integral of the
 curvature and the correct value of $-\pi/4$.
\item \mcode+A["num_charts"];+ given an atlas $A$ on $EX^*$, returns
 the number of charts.
\end{itemize}
Classes are declared using the \mcode+`Class/Declare`+ function.

There is one notable place where we have chosen not to use the above
framework.  We have a lot of parallel structures for our three
families of cromulent surfaces, for example the functions
\mcode+c_E[k](t)+, \mcode+c_H[k](t)+ and \mcode+c_P[k](t)+ which encode
the three curve systems.  In some respects it would be natural to
encapsulate these using a system of classes.  However, that would lead
to unwieldy notation for objects that we need to use extremely
frequently, so we chose to avoid it.

\subsection{Building the data}
\lbl{sec-build}

This project involves some numerical computations, the results of
which are stored in the \fname+data+ directory and its
subdirectories.  The file \fname+build_data.mpl+ (in the top Maple
directory) defines various functions that can be used to perform these
calculations.  For example, the function
\mcode+build_data["HP_table"]()+ can be used to perform all the
calculations described in Section~\ref{sec-P-H}, relating the
projective and hyperbolic families.  The result is encoded as an
object of the class \mcode+HP_table+ (declared in
\fname+hyperbolic/HP_table.mpl+).  It can be saved in an appropriate
place using the function \mcode+save_data["HP_table"]()+, and then
reloaded later using the function \mcode+load_data["HP_table"]()+.

The functions \mcode+build_data["all"]()+, \mcode+save_data["all"]()+
and \mcode+load_data["all"]()+ work in the obvious way, building,
saving or loading all of the required data.  A full build of all data
will take several days of computer time, at least.  However, one can
enter \mcode+set_toy_version(true)+ before invoking
\mcode+build_data["all"]+.  This will cause Maple to do all
calculations to lower accuracy, and finish in an hour or two.  In this
context, results will be saved to or loaded from the
\fname+data_toy+ directory instead of the \fname+data+ directory.

The build process will generate a fairly large number of messages
about the progress of the calculation.  One can reduce the volume by
setting \mcode+infolevel[genus2]+ to a number less than the default
value of $7$, before invoking \mcode+build_data["all"]+.

\end{document}